\documentclass[11pt]{aims}
\usepackage{amsmath, amssymb, mathrsfs, accents}
  \usepackage{paralist}
  \usepackage{graphics} 
  \usepackage{epsfig} 
\usepackage{graphicx}  \usepackage{epstopdf}
 \usepackage[colorlinks=true]{hyperref}
 \hypersetup{urlcolor=blue, citecolor=red}
 \usepackage[toc,page]{appendix}
\usepackage{chngcntr}
\usepackage{bbm}

\usepackage{titlesec}
\titleformat{\section}{\Large\bfseries}{\thesection.}{4pt}{}
\titleformat{\subsection}{\large\bfseries}{\thesection.\arabic{subsection}.}{4pt}{}
\titleformat{\subsubsection}{\bfseries}{\thesection.\arabic{subsection}.\arabic{subsubsection}.}{4pt}{}
\titleformat*{\paragraph}{\bfseries}
\titleformat*{\subparagraph}{\bfseries}
  \def\currentyear{}
   \def\currentmonth{}
    \def\ppages{}
     \def\DOI{}

\newtheorem{theorem}{Theorem}[section]

\newtheorem{lemma}[theorem]{Lemma}
\newtheorem{proposition}[theorem]{Proposition}
\theoremstyle{definition}
\newtheorem{definition}[theorem]{Definition}
\newtheorem{remark}[theorem]{Remark}

\newcommand{\vep}{\varepsilon}

\newcommand{\Mod}{\textup{Mod}}

\newcommand{\Rb}{\mathbb{R}}

\newcommand{\Cc}{\mathcal{C}}

\newcommand{\Ec}{\mathcal{E}}
\newcommand{\Fc}{\mathcal{F}}

\newcommand{\Oc}{\mathcal{O}}

\newcommand{\Sc}{\mathcal{S}}

\newcommand{\pa}{\partial}
\newcommand{\pr}{\partial_r}

\newcommand{\inn}{\textup{in}}
\newcommand{\out}{\textup{ex}}
\newcommand{\bou}{\textup{bo}}
\newcommand{\Ls}{\mathscr{L}}
\newcommand{\Ms}{\mathscr{M}}
\newcommand{\Hs}{\mathscr{H}}
\newcommand{\As}{\mathscr{A}}

\newcommand{\Lst}{\tilde{\mathscr{L}}}

\newcommand{\Uk}{\mathfrak{U}}

\numberwithin{equation}{section}

 \title[Collision of collapsing solitons]{SINGULARITY FORMED BY THE COLLISION OF TWO COLLAPSING SOLITONS in interaction FOR THE 2D KELLER-SEGEL SYSTEM}

 \keywords{Keller-Segel system, blowup, multiple collapse, singularity formation}
\author[C. Collot]{Charles Collot}
\address{Charles Collot, Laboratoire AGM, CY Cergy Paris Universit\'e, 2 avenue Adolphe Chauvin, 95300 Pontoise, France}
\email{ccollot@cyu.fr}

\author[T.-E. Ghoul]{Tej-Eddine Ghoul}
\address{Tej-Eddine Ghoul, NYUAD Research Institute, New York University Abu Dhabi, Abu Dhabi, UAE}
\email{teg6@nyu.edu}

\author[N. Masmoudi]{Nader Masmoudi}
\address{Nader Masmoudi, NYUAD Research Institute, New York University Abu Dhabi, Abu Dhabi, UAE, and Courant Institute of Mathematical Sciences, New York University, 251 Mercer St, New York, NY 10012, USA}
\email{nm30@nyu.edu }

\author[V. T. Nguyen]{Van Tien Nguyen}
\address{Van Tien Nguyen, Department of Mathematics, Institute of Applied Mathematical Sciences, National Taiwan University, Taiwan}
\email{vtnguyen@ntu.edu.tw}

\thanks{\today}
\begin{document}
\maketitle


\begin{abstract} It is well-known that the two-dimensional parabolic-elliptic Keller-Segel system admits finite-time blowup solutions, which is the case if the initial density has total mass greater than $8\pi$. Several constructive examples of such solutions have been given, where for all of them a perturbed stationary state undergoes scale instability and collapses at a point, resulting in an $8\pi$-mass concentration. It was conjectured that singular solutions concentrating more than one soliton simultaneously could exist. We construct rigorously such a new blowup mechanism, where two stationary states are simultaneously collapsing and colliding, resulting in a $16\pi$-mass concentration at a single blowup point, and with a new blowup rate which corresponds to the formal prediction by Seki, Sugiyama and Vel\'azquez. We develop, for the first time, a robust framework to rigorously construct blowup solutions that simultaneously involve the non-radial collision and concentration of several solitons, which we expect to have applications to other evolution problems.
\end{abstract}

\section{Introduction}

\subsection{The two dimensional Keller-Segel system and the main result}

We consider the parabolic-elliptic Keller-Segel system introduced by Nanjundiah in \cite{Njtb73},
\begin{equation}\label{eq:KS2d}
\left\{ \begin{array}{r l} \pa_tu &= \Delta u - \nabla \cdot\big(u \nabla \Phi_u\big) ,\\
0 &= \Delta \Phi_u + u,\\
u(x,0)&=u_0(x)\geq 0,
\end{array}\right. \quad (x,t) \in \Rb^2 \times (0,T).
\end{equation}
Above,  the solution to the Poisson equation is chosen as
\begin{equation}\label{def:Phi_u}
\Phi_u=-\frac{1}{2\pi} \log |x|*u.
\end{equation}
The model \eqref{eq:KS2d} was first derived by Keller-Segel in \cite{KSjtb70} to describe \textit{chemotaxis} in biology (see \cite{PATbmb53} for an earlier derivation and \cite{KSjtb71a}, \cite{KSjtb71b}, \cite{KELbook80}  for various assessments). We refer to \cite{wolansky1992steady,Chavanis2008} (see also \cite{SCPRSN02, CSPRE11}) for related models that appear in physical and biological contexts. Connections with other evolution equations will be discussed in the next Subsection \ref{subsec:comment-novelties}.

The problem \eqref{eq:KS2d} enjoys a scaling invariance in that for all $\lambda>0$ and $a\in \mathbb R^2$, the function
\begin{equation}\label{def:scalingIntro}
u_{\lambda, a}(x,t) = \frac{1}{\lambda^2}u\big(\frac{x - a}{\lambda}, \frac{t}{\lambda^2}\big), 
\end{equation}
also solves \eqref{eq:KS2d}. In addition, the mass of the solution is conserved as
$$
\int_{\mathbb R^2} u(x,t)dx=\int_{\mathbb R^2} u_0(x)dx=:M(u).
$$
The problem \eqref{eq:KS2d} is termed as mass critical in two dimensions since
$$\| u_{\lambda,a}(0)\|_{L^1(\Rb^2)} = \| u(0)\|_{L^1(\Rb^2)}. $$
It is well known that the Cauchy problem for \eqref{eq:KS2d} is well-posed in $L^1$ \cite{wei2018global} and in rougher scale invariant spaces under suitable smallness assumptions \cite{BMARMA14,lemarie2013small}. Such solutions belong to $L^\infty$ instantaneously after the initial time by parabolic regularization, and, denoting by $T(u_0)$ their maximal time of existence, there holds the blow-up criterion
$$
T<\infty \quad \Leftrightarrow \quad \lim_{t\to T} \| u(t)\|_{L^\infty(\mathbb R^2)}=\infty .
$$
The blow-up set $\mathcal S\subset \mathbb R^2$ is then defined as the set of points $x$ for which there exists $t_n\uparrow T$ and $x_n\rightarrow x$ such that $u(x_n,t_n)\to \infty$. Positive solutions with mass below the $8\pi$ threshold $M(u)<8\pi$ produce global solutions which for large times converge in $L^1$ towards the unique forward self-similar solution with the same mass $t^{-1}\Psi_{M(u)}(x/\sqrt{t})$, see \cite{BDPjde06,nagai2011convergence,HYnon24}. If the mass is above $M(u)>8\pi$, then blowup in finite time occurs, see \cite{JLtams92,CPZmjm04,DNRjde98,BDPjde06,biler20068pi,BCCjfa12,wei2018global}. The $8\pi$ threshold corresponds to the mass of the stationary state of the equation
$$
U(x) = \frac{8}{(1 + |x|^2)^2}, \qquad \int_{\Rb^2} U(x) dx = 8\pi.
$$
It was known that solutions at the mass threshold $M(u)=8\pi$ with finite second momentum concentrate in infinite time \cite{BCMcpam08}, with precise examples with a logarithmic spatial scale given in \cite{GMcpam18}, \cite{DPDMWarma24}. Recently it was proved that it is the unique universal dynamics \cite{buseghin2026determination,buseghin2026classification}. Without finite second momentum, under finite relative entropy solutions converge towards a fixed stationary state \cite{blanchet2012functional,carlen2013stability,carlen2025stability}. Infinite oscillations between scales are also possible \cite{lopez2014non}.

Several concrete examples of finite time blowup solutions were constructed in \cite{HVma96,RSma14,CGMNcpam21,CGMNapde22}. In these works, the solution blows up with $8\pi$ mass concentration and admits the asymptotic dynamics
\begin{equation}\label{def:singleblowup}
u(x,t) = \frac{1}{\lambda^2(t)}U\Big( \frac{x}{\lambda(t)}\Big)+\tilde u(x,t),
\end{equation}
where, for $C_0=2e^{-\frac{2+\gamma}{2}}$ for $\gamma $ the Euler constant,
\begin{equation}\label{def:lambda1bubble}
\lambda(t) = C_0 \sqrt{T-t} e^{-\sqrt{ |\log(T-t)|/2}}(1+o_{t\uparrow T}(1)) \quad \mbox{and} \quad \lambda^2(t)\| \tilde u(\cdot,t)\|_{L^\infty(\mathbb R^2)}\underset{t\to T}{\longrightarrow} 0.
\end{equation}
The stability of this solution was studied in \cite{Vsiam02} at a formal level, then showed in \cite{RSma14} for the radial case, and in \cite{CGMNcpam21,CGMNapde22} for the full nonradial case with the sharp blowup rate. Other unstable blowup rates for sign-changing solutions were also derived in \cite{CGMNcpam21}. Later, the authors of \cite{BDPMaPDE26} showed the existence of singular solutions blowing up simultaneously at several points where near each one it blows up according to the same dynamics \eqref{def:singleblowup}-\eqref{def:lambda1bubble}, carrying a mass of $8\pi$.
 
Recently, in \cite{Sbook24} (Theorem 3.3), Suzuki showed that if the solution $u$ to \eqref{eq:KS2d} blows up in finite time then the blow-up set is finite and there holds mass quantization,
\begin{equation}\label{def:massclassification}
u(x,T) \rightharpoonup \sum_{x_0 \in \mathcal{S}} m(x_0)\delta_{x_0}dx + f(x) dx, \quad m(x_0) \in 8\pi \mathbb{N},
\end{equation}
where $f \in L^1(\Rb^2)  \cap \mathcal{C}(\Rb^2 \setminus \mathcal{S})$. The asymptotic dynamics \eqref{def:singleblowup}-\eqref{def:lambda1bubble} provides an example of the $8\pi$ mass concentration at a single blowup point: $\mathcal S=\{x_0\}$ and $m(x_0)=8\pi$. Solutions concentrating $8\pi$ mass simultaneously at several distinct points, i.e. $\mathcal S=\{x_0,...,x_k\}$ for some $k\geq 1$ with $m(x_i)=8\pi$ for $0\leq i \leq k$, were constructed in \cite{BDPMaPDE26}. 

 Existence of blowup solutions with a multiple of $8\pi$ mass concentration at a single blowup point remained an open question. A formal approximate solution was given by Seki, Sugiyama and Vel\'azquez in \cite{SSVnon13}, by means of formal matched asymptotic expansions. In this work, we provide the first rigorous example of blowup solutions to \eqref{eq:KS2d} with the $16\pi$ mass concentration at a single point ($\mathcal S=\{x_0\}$ and $m(x_0)=16\pi$), formed by the nonradial collision of two collapsing solitons in interaction. Introducing the function space
\begin{equation}\label{def:Espace}
\Ec = \Big \{u: \Rb^2 \to \Rb, \quad \| u\|_\Ec^2 = \sum_{k=0}^2 \int_{\Rb^2} \langle x \rangle^{\frac{3}{2} + 2k} |\nabla^k u|^2 dx < \infty \Big\},
\end{equation}
the following theorem is the main result of this present paper.

\begin{theorem}[Single point multiple-collapse blowup solutions to \eqref{eq:KS2d}] \label{thm:main} There exists a smooth initial data $u_0 \in L^1\cap \Ec$ with $u_0\geq 0$ such that the corresponding solution $u$ to \eqref{eq:KS2d} blows up in finite time $T = T(u_0) > 0$ and admits the decomposition 
\begin{equation}\label{def:massclassification}
u(x,t) = \frac{1}{\lambda^2(t)} U\Big( \frac{x - x_0(t)}{\lambda(t)}\Big) + \frac{1}{\lambda^2(t)}  U\Big( \frac{x + x_0(t)}{\lambda(t)}\Big) + \tilde u(x,t) , 
\end{equation}
where as $t \to T$, 
\begin{equation}\label{def:lambda2bubbles}
\lambda(t) = C_1(t) \sqrt{T-t} e^{-\sqrt{\gamma_1 |\ln (T-t)|} }, \quad \frac{x_0(t)}{\sqrt{T-t}} \to \big(2, 0 \big), \quad \lambda^2(t) \| \tilde{u}(t)\|_{L^\infty} \to 0,
\end{equation}
for $\gamma_1 > 0$ defined in \eqref{eigenfunction:id:tildelambda1} and $C_1(t)$ that satisfies $c_*\leq C_1(t)\leq c^*$ for all $t\in [0,T)$ for some constants $c^*>c_*>0$. Moreover, we have
$$
\lim_{\delta \downarrow 0}\lim_{t\uparrow T} \int_{|x|<\delta} u(t,x)dx=16\pi.
$$

\end{theorem}

\subsection{Comments and novelties} \label{subsec:comment-novelties}

\paragraph{A new blowup mechanism.} The mechanism of the singularity formation in Theorem \ref{thm:main} is the \emph{nonradial collision at a single point of two collapsing solitons in strong interaction}. This mechanism is brand new for the Keller-Segel system. In particular, the blowup rate in $L^\infty$ is $\| u(t)\|_{L^\infty}\sim (T-t)^{-1}e^{2\sqrt{\gamma_1 |\ln (T-t)|} }C_1^{-2}(t)$. It is of type II because $\lim_{t\uparrow T}(T-t)\| u(t)\|_{L^\infty}= \infty$, and it differs from the previously known solution \eqref{def:singleblowup}-\eqref{def:lambda1bubble}, because the constant $\gamma_1$ is no longer equal to $1/2$. This new blow-up rate confirms that which was formally predicted in \cite{SSVnon13}. The change of blowup rate reflects a strong interaction between the two rescaled stationary states that attract each other in the parabolic neighboorhood from the singularity $\{ |x|\lesssim \sqrt{T-t}\}$, as well as a strong interaction with the remainder $\tilde u$ (that we shall call the radiation term). Our proof links this latter interaction to the change of the blowup law, via a computation of a degenerate eigenvalue of the associated linearized operator around the multisoliton (see \eqref{eq:eigenfunctions_phi10}). A related mechanism was, to our knowledge, only previously addressed in two breakthrough works:
\begin{itemize}
    \item First, C\^ote-Zaag \cite{CZcpam13} (see also \cite{MZcpam18}) constructed a singular solution of the one-dimensional semilinear wave equation, in which several backward self-similar solutions interact in the backward light cone from the singularity. The key difference with respect to the present work is that a backward self-similar solution alone is an exact solution which concentrates to smaller scales, while a stationary state alone remains steady. As a result, the singularity in \cite{CZcpam13} is driven by the sole interaction of the backward self-similar solutions, and the radiation term does not play a role. In our main result however, \emph{the singularity is due both to the interactions between the stationary states and to their interactions with the radiation term.} 
    \item Second, Martel-Rapha\"el \cite{MRens18} constructed a singular solution of the mass critical Schr\"odinger equation formed by the collision of two collapsing traveling waves. The singularity formation mechanism is similar to the one here, but the techniques of \cite{MRens18} rely heavily on specific features of the mass critical Schr\"odinger equation, and could not address the present equation. Indeed, they show first the existence of a global in time solution formed by the sum of traveling waves concentrating at distinct points, by a backward in time construction relying on time reversal symmetry. Then they apply the pseudo-conformal symmetry to generate the desired finite time blow-up solution. These two symmetrie are absent in the parabolic problem \eqref{eq:KS2d}; we develop here a new strategy for a direct construction and stability of finite-time blowup formed by the collision of collapsing solitons.
    \end{itemize}

\paragraph{Interaction of solitons.} The concept of solitons is one of the most important developments in the nonlinear world, and the word "soliton" has been now widely used to designate nonlinear coherent structures of partial differential equations \cite{tao2009solitons,soffer2006soliton,DPbook06}. In all the references we shall mention, these solitons are either stationary states, backward self-similar solutions, or traveling waves, all of which are solutions that are left invariant by the action of a one-dimensional subgroup of the group of symmetries of the equation, and are thus "self-similar" solutions in a broad sense. Fundamental properties of a single soliton including its stability, its interaction with radiation has been intensively studied in the past.

All previously known examples of singularity formation for the Keller-Segel system \eqref{eq:KS2d} involved \emph{the collapse of a single soliton}, see \cite{HVma96,HMVjcam98,BCKSVnon99,mizoguchi2007type,RSma14,CGMNcpam21,CGMNjfa23,glogic2024stable,NNZarXiv23,buseghin2023existence,collot2024stability}. The vast majority of the known singularity formation mechanisms for other equations are also the ones that involve the collapse of a single soliton. This is the case for example for the collapse of a bubble for critical equations \cite{MMjams02,MRam05,KSTim08,MRRim13,RRihes12,RSam10,DPWim19}, related to the aforementioned works \cite{HVma96,RSma14,CGMNcpam21}. Singularities formed by interacting solitons remain however poorly understood.

When it comes to the soliton-soliton interaction (or a wider class of interactions between solutions with self-similarity), the linear superposition does not hold and nonlinear interaction takes place. The study of interaction between solitons is a fascinating topic nowadays, \cite{Schapterbook09}. So far, except in some integrable PDEs where explicit formulas for multiple-soliton solutions are known (see for example \cite{Lcpam68}, \cite{Msiam76}), the question of the collision of solitons is much less understood. Most of the studies have been done for global-in-time solutions. We have some mathematical results for the case where the collision has a negligible effect (perturbative regime), see for example \cite{HSzamp98,Msiam03,Pcpde04,RSSrmp03}. There are examples of global solutions formed by multiple solitary waves with strongly affected traveling speeds studied by \cite{KRMcpam09,Ngcrm19}, and other important works on the study of the collision of solitons in the non-integrable case include \cite{MMam11,MMim18,GLPRapde18}. We refer to \cite{martel2018interaction} for a review of global-in-time multisoliton dynamics. In the radial case, there are examples of superposition of solitons with strongly affected scales in the works by \cite{PMWapde21,kim2024classification,JLim18}. The only works involving the nonradial collision of solitons in strong interactions were \cite{CZcpam13,MRens18} that we mentioned in the previous paragraph. Our present work aims at providing a robust framework to study the interaction of solitons in singularity formation for other equations, via the spectral analysis of the linearized operator around a multisoliton, and a novel concept of matched scalar product to derive Lyapunov functionals to control the flow. \\

\paragraph{Codimension-one stability under symmetry:} The solution we construct satisfies the symmetry
\begin{equation} \label{id:symmetry-even}
u(x,y)=u(-x,y)=u(x,-y)=u(-x,-y),
\end{equation}
which is preserved by the flow of \eqref{eq:KS2d}. We associate to the initial datum of the solution of Theorem \ref{thm:main} the initial values $\lambda_0=\lambda(0)$ and $x_0=x_0(0)$ of the parameters.  Modulating the position $x_0$ of the two stationary states produces at the linear level a variation along the direction
$$
\partial_{|x_0|}\left(\frac{1}{\lambda^2_0} U\Big( \frac{x - x_0}{\lambda_0}\Big) +\frac{1}{\lambda^2_0}  U\Big( \frac{x + x_0}{\lambda_0}\Big)\right)=-\frac{1}{\lambda^3_0}\partial_{x_1} U\Big( \frac{x - x_0}{\lambda_0}\Big) +\frac{1}{\lambda^3_0} \partial_{x_1} U\Big( \frac{x + x_0}{\lambda_0}\Big).
$$
We expect the constructed solution described in Theorem \ref{thm:main} enjoys a codimension-one stability under symmetry in the following sense: Let $u_0$ denote the initial datum of the solution $u$ of Theorem \ref{thm:main}. There exists a function $\tilde v_0\mapsto c(\tilde v_0)$ defined for all $\tilde v_0\in \mathcal E$ small enough satisfying the symmetry \eqref{id:symmetry-even} and the orthogonality
$$
\int_{\mathbb R^2} \tilde v_0 \partial_{x_1} U\Big( \frac{x - x_0}{\lambda_0}\Big) dx=\int_{\mathbb R^2} \tilde v_0 \partial_{x_1} U\Big( \frac{x + x_0}{\lambda_0}\Big) dx=0
$$
(note that the first two quantities are equal by the symmetry \eqref{id:symmetry-even}) such that the initial datum
$$
u_0'(x)= u_0(x)+\tilde v_0(x)+c(\tilde v_0)\left(-\frac{1}{\lambda^3_0}\partial_{x_1} U\Big( \frac{x - x_0}{\lambda_0}\Big) +\frac{1}{\lambda^3_0} \partial_{x_1} U\Big( \frac{x + x_0}{\lambda_0}\Big)\right)
$$
produces a solution $u'$ of \eqref{eq:KS2d} with $u'(0,x)=u'_0(x)$ that blows up at some time $T'>0$ with the same behaviour as $u$. Namely, it satisfies all the conclusions of Theorem \ref{thm:main} for some time-dependent parameters $\lambda'$, $x'_0$ and function $\tilde u'$. Moreover, the function $c$ is Lipschitz in the sense that $|c(\tilde v_0) - c(\tilde w_0)|\leq K \| \tilde v_0-\tilde w_0\|_{\mathcal E}$ for some $K>0$. We refer to \cite{CollotZhang2024} for an example proof of the stable manifold of initial data leading to a chosen self-similar blow-up profile in higher-dimensional cases with a finite Lipschitz codimension.

We believe that the set of initial data is a threshold between two blowup behaviors. Namely, keeping the same hypotheses and notations, for $c''$ close to $c(\tilde v_0)$ an initial datum of the form 
$$
u_0''(x)= u_0(x)+\tilde v_0(x)+c''\left(-\frac{1}{\lambda^3_0}\partial_{x_1} U\Big( \frac{x - x_0}{\lambda_0}\Big) +\frac{1}{\lambda^3_0} \partial_{x_1} U\Big( \frac{x + x_0}{\lambda_0}\Big)\right)
$$
will produce a solution $u''$ of \eqref{eq:KS2d} with $u''(0,x)=u''_0(x)$ which will blow-up in finite time in the following way. If $c''>c(\tilde v_0)$ the two stationary states collapse before they collide, corresponding to two single $8\pi$ blow-ups located away from the origin. If $c''<c(\tilde v_0)$ the two stationary states collide before they collapse, which could result in a merging into a single $8\pi$ blowup at the origin. This unveils an interesting new threshold behavior to study with respect to previous ones separating finite time blow-up and global existence, see \cite{MMam11, CMRcmp17} for example.

The codimension stability is also expected without the symmetry assumption \eqref{id:symmetry-even}. One would have to record separately the positions and scales of the two stationary states. In fact, non-symmetrical positions would not produce a new instability thanks to the invariance of \eqref{eq:KS2d} by spatial isometries. However, non-symmetrical scales would enable a new instability: one of the two stationary states could collapse before the other. We thus expect codimension-two stability to hold for general perturbations, which is worth being addressed in a separate work.\\

\paragraph{Classification of singularities.} In \cite{Vicm06}, Vel\'azquez formulated an open question considered as an ambitious problem concerning a complete classification of all possible blowup dynamics for \eqref{eq:KS2d}. He predicted initially that all the solutions that blow up in finite time behave near the singularity as indicated in \eqref{def:singleblowup} with the blowup rate \eqref{def:lambda1bubble}, but later conjectured other blow-up mechanisms to exist \cite{SSVnon13}. Mizoguchi \cite{Mcpam21} proved that \eqref{def:singleblowup}-\eqref{def:lambda1bubble} is the only blow-up mechanism for the class of radial positive solutions. The proof of \cite{Mcpam21} however relies strongly on the parabolic comparison principle, and on the fact that the number of intersections (with respect to the radial variable) between two
solutions are monotone decreasing in time. These techniques, originating in \cite{MMjfa11}, \cite{Mma07} can only be applied to radial and parabolic equations. In the non-radial case, the best available information is the result of \cite{Sbook24} stated in \eqref{def:massclassification} asserting that a multiple Dirac mass of $8\pi$ can form. New non-radial blowup mechanisms are however expected to happen, and this is the case of blowup solutions described in Theorem \ref{thm:main} with a $16 \pi$ mass concentration at a single point with a different blowup rate \eqref{def:lambda2bubbles}. Even in the radial case but for sign-changing solutions, the present authors constructed in \cite{CGMNcpam21,CGMNapde22} countably many other blow-up rates
\begin{equation}\label{eq:solutiondecomposition}
u(x,t) = \frac{1}{\lambda(t)^2}U\Big(\frac{x}{\lambda(t)}\Big) + \tilde u(x,t),
\end{equation}
where
\begin{equation}
\lambda(t)\sim (T-t)^\frac{\ell}{2} |\log(T-t)|^{-\frac{\ell + 1}{2(\ell - 1)}}, \;\; \ell \in \mathbb{N},\; \ell \geq 2,
\end{equation}
and $\lambda(t)^2\| \tilde u(t)\|_{L^\infty}\to 0 $. Thus, the blow-up mechanisms in the general sign-changing and non-radial case seem to be quite wild and new ones could yet be found.\\

\paragraph{Soliton resolution.} We believe that for the Keller-Segel system \eqref{eq:KS2d}, the soliton resolution conjecture would hold true. In the case of finite time singularities, we expect the solution to decompose as the sum of modulated solitons plus a lower order radiation term. In the non-radial case, such results have been obtained along a sequence of times for the energy critical semilinear wave equation \cite{duyckaerts2017soliton}, and for the energy critical harmonic heat flow \cite{struwe1985evolution,qing1997bubbling,topping2004winding} (and see references therein). While these results state that general solutions decompose into solitons in interaction, it is important to supplement them with the construction of examples to show which interaction configurations are actually possible. This is the general motivation behind the present work.\\

\paragraph{Aspects of the new strategy:} Several methods for blowup constructions have been developed in the last three decades, among which we have modulation technique robustly applied to various problem settings, see for example  \cite{MRam05}, \cite{RRihes12}, \cite{HRapde12}, \cite{Sjfa12}, \cite{RSma14}, \cite{RScpam13, RSapde2014}, \cite{MRRcjm15}, \cite{Capde17, Cmams18}, \cite{GINapde18} and references therein; the matching asymptotic expansions in the works of \cite{HVma96}, \cite{Vsiam02}, \cite{FHVslps00}; the inner-outer gluing method  used in \cite{PMWjfa21}, \cite{PMWams19}, \cite{PMWZdcds20}, \cite{PMWats19}, \cite{DPWim19}, an iterative scheme \cite{KSTam09} \cite{KSTduke09}, \cite{KSTim08}, spectral analysis of multiple-scale problems developed in \cite{HRjems19}, \cite{CMRjams19} and  \cite{CGMNcpam21}. These mentioned methods consist of two common parts: 1) construction of an approximate blowup solution and 2) analysis of the system formed by the finitely many modulation parameters coupled with the infinite dimensional remainder. We push forward the strategy developed in our previous works \cite{CGMNcpam21,CGMNapde22} based on a spectral approach of the underlying multiple-scale problem, with the following novelties in this present work: 
\begin{itemize}
\item[-] (i) \emph{Spectral analysis around two solitons in the backward parabola from the singularity.} We relate the scaling law \eqref{def:lambda2bubbles} to a low eigenvalue of the linearized operator $\Ls^z$ around two concentrated stationary sates in self-similar variables (see \eqref{def:LszNot} for the definition). Our approximate blow-up solution is built using a corresponding approximate solution of the eigenproblem. This point of view provides a natural parametrisation of the approximate solution, and is different from that of the formal construction in \cite{SSVnon13}.

For the previously known single soliton blowup solution \eqref{def:singleblowup}-\eqref{def:lambda1bubble} the dynamics is governed by the radial component to leading order. In the radial sector thanks to the so-called partial mass transformation, we could reduce in \cite{CGMNcpam21} the spectral study of the underlying linearized nonlocal operator to a local Schr\"odinger type operator. This latter operator (a singularly perturbed operator involving 2 scales) can be analyzed in details through a rigorous matched asymptotic expansion procedure and the classical Sturm-Liouville theory. This resulted in the exact spectral properties showed in \cite{CGMNapde22}, which were formally obtained in \cite{DphysD12,DELSaa13}.

The present work initiates the study of the eigenproblem around two solitons, which is much harder since genuinely non-local and non-radial. We only solve this eigenproblem approximately, since anyway it will be used to construct an approximate solution. The main outcome is Proposition \ref{prop:eigen} where we use a matched asymptotic expansion to compute the two approximate eigenfunctions and eigenvalues responsible for the blowup law \eqref{def:lambda2bubbles}. Solving exactly the eigenproblem goes beyond the scope of this article; it is an interesting open problem. Other eigenvalues might allow to construct new blowup rates for sign-changing solutions.\\

\item[-] (ii) \emph{Matched scalar product and coercivity.} For the construction of the single soliton blowup solution \eqref{def:singleblowup}-\eqref{def:lambda1bubble} in \cite{CGMNcpam21}, in the radial sector and in the partial mass framework the linearized operator is a self-adjoint Schr\"odinger-type operator. Thus the knowledge of exact eigenfunctions implies a spectral gap estimate in a weighted $L^2$ space via Sturm-Liouville theory. This naturally defines a Lyapunov functional to control the subleading part of the remainder $\tilde u$. This is not the case for the multiple-collapse blowup treated in this work due to the non-radial, non-local structure of the linearized operator $\Ls^z$. There is, to our knowledge, no Hilbert space for which it is self-adjoint.

To overcome this issue, we develop a new concept of matched scalar product to produce a globally defined Lyapunov function. It is the Hilbert space counterpart to matched asymptotic expansions. Namely, the linearized operator $\Ls^z$ involves a singular limit as the scale of the stationary states collapse. Two limiting operators appear at the scale of the stationary states $\Ls_0$, and at the parabolic self-similar scale $\mathcal H_\beta$. $\Ls_0$ is self-adjoint for a nonlocal scalar product derived from the Hessian of the free entropy functional at the ground state minimizer $U$, and $\mathcal H_\beta$ is self-adjoint in a suitably weighted $L^2$ space as it is a Focker-Planck type operator. We are able to match these two scalar products into a single one \eqref{def:scalar12_intro} for which $\Ls^z$ is to leading order symmetric, and satisfies a dissipative coercivity estimate provided certain orthogonality conditions are satisfied, see Proposition \ref{Prop:coercive-Lz-global}. \\

\item[-] (iii) \emph{Refinement of the solution for nonlinear and modulational analysis}. In order to construct a solution of the full nonlinear problem, The two new ideas for the linearized analysis (i) and (ii) above need to be supplemented by nonlinear corrections to the approximate solution, the determination of  suitable nonlinear evolution equations for its modulation parameters, along with a suitable decomposition and design of norms to show decay estimates for the subleading part of the remainder $\tilde u$. There are two key technical novelties in the present context.

\emph{Correctors for the multisoliton}. First, we need to add suitable corrector terms to the two stationary states. This idea was already used for global-in-time multi-solitons and finite-time blow-up with a single soliton, so the novelty is to implement it here for a non-radial multi-solitons blowup. The purpose is twofold. First, there is a strong long-range interaction between the two via the Poisson field they generate $\nabla \Phi_U(x)=- 4x (1+|x|)^{-2}$. The error generated by their nonlinear interaction is then reduced by introducing suitable correctors. Second, we need to be able to relate the parameter derivatives (modulation directions) of the corrected multisoliton, to the eigenfunctions found in solving the aforementioned spectral problem. This will enable us to compute how the radiation remainder will affect the scale and positions of the stationary states. The main outcome is Proposition \ref{prop:E0} where we compute a corrected multisoliton solution to \eqref{id:stationary-equation-parabolic-variables-normalized} and extract the leading order term of the generated error term to the first eigenfunction of $\Ls^z$.

\emph{Different decompositions for modulation and decay estimates}. The two approximate eigenfunctions and eigenvalues permit to construct the approximate blowup solution (part 1 in the general strategy) and to derive the blowup law through a simple orthogonal projection onto the eigenmodes (see Lemma \ref{lemm:mod}), but the coercivity estimate for $\Ls^z$ is obtained with a different orthogonality condition to the kernel of the linearized operator around a single soliton, see \eqref{eq:orthog}. We thus start by decomposing our solution requiring the former orthogonality conditions for the remainder; this cancels its linear contribution to the modulation dynamics, but when performing an energy estimate based on the matched scalar product to bound the remainder this creates uncontrollable crossed terms (see Proposition \ref{prop:E} and the decomposition \eqref{def:EztauMod}). We then consider a further decomposition \eqref{id:coercivity:decomposition-phi0} and prove that the new component in this decomposition enjoys the same coercivity estimate under the matched scalar product \eqref{def:scalar12_intro}, see Proposition \ref{prop:decomp_vephat}. We emphasize once again on the importance of the adapted scalar product \eqref{def:scalar12_intro} in both the coercivity estimate for $\Ls^z$ and the norm-equivalence  \eqref{bd:coercivity-control-adapted-norm-hatu-2} that are crucial in the linear analysis. 
\end{itemize}

\noindent \paragraph{Extension to other problems:}  We expect the framework developed in this work can be carried out to other different settings such as semilinear parabolic or dispersive equations. In particular, the four novelties described above would have natural adaptations, providing a general framework to guess how the interactions between solitons and their surrounding radiation remainder would generate a precise blow-up dynamics, to construct an approximate solution, and to construct a full solution by the study of its nonlinear stability.

\subsection{Strategy of the proof}
The construction of the single point multiple-collapse blowup solution of Theorem \ref{thm:main} proceeds in several steps in the spirit of the constructive framework developed by the authors in \cite{CGMNcpam21,CGMNapde22}, supplemented by novelties described above in Subsection \ref{subsec:comment-novelties}.\\

\noindent - \underline{Step 1} (Reformulation of the problem in the self-similar variables): We introduce the self-similar renormalized variables associated to the scaling symmetry \eqref{def:scalingIntro},
\begin{equation}
u(x,t) = \frac{1}{T-t}w(z, \tau),  \quad   z = \frac{x}{\sqrt{T-t}}, \quad \tau = -\ln(T-t).
\end{equation}
The new unknown $w$ solves
\begin{equation}\label{eq:wztauIntro}
\pa_\tau w = \nabla \cdot (\nabla w - w \nabla \Phi_w) - \frac 12 \Lambda w.
\end{equation}
We aim at constructing for \eqref{eq:wztauIntro} a global in-time solution that blows up in infinite of the form 
\begin{equation}
w(z, \tau) = \frac{1}{\nu^2}U(\frac{z - a}{\nu}) + \frac{1}{\nu^2}U(\frac{z+a}{\nu}) + \tilde{w}(z,\tau),
\end{equation}
where $\nu \in \Cc^1(\Rb_+, \Rb_+)$ and $a \in \Cc^1(\Rb_+, \Rb^2)$ are the main unknown parameters to be determined.  \\

\noindent -  \underline{Step 2} (Approximate eigenfunctions): The linearized equation for $\tilde w$ is of the form 
\begin{equation} \label{eq:wtilde}
\pa_\tau \tilde w = \Ls^z \tilde{w} + NL(\tilde w) + G,
\end{equation} 
where $\Ls^z$ is the linearized operator around the sum of the two rescaled single solitons defined in \eqref{def:LszNot}, $NL$ is the quadratic nonlinear term expected to be small and $G$ is the error generated by the two stationary states. Extracting leading order dynamics of $\tilde w$ is essential in determining the blowup law; its main part should naturally be located on eigenmodes corresponding to lowest eigenvalues, leading to the study of the associated eigenproblem. Unlike the single blowup in \cite{CGMNcpam21,CGMNapde22} where this eigenproblem can be recast as a singularly perturbed radial elliptic equation, here it is a singularly perturbed non-radial system (see \eqref{eigen:id:Ri}, \eqref{eigenfunctions-interior:id:eigenfunction-system} and \eqref{eigenfunctions-exterior:id:eigenfunction-system}). We construct using matched asymptotics two approximate eigenfunctions $\phi_0$ and $\phi_1$ that solve
\begin{equation}
\Ls^z \phi_i = 2\beta\Big(1 - i + \frac{\gamma_i}{\ln \nu} \Big)\phi_i + R_i,
\end{equation}
where $R_i$ is a lower order term and $\gamma_i \in \Rb$ with $\gamma_1 > 0$, see Proposition \ref{prop:eigen}. \\

\noindent -  \underline{Step 3} (Coercivity of the linearized operator for a matched scalar product): Since performing a rigorous spectral analysis for $\Ls^z$ is too involved (the operator being nonlocal, nonradial, non self-adjoint and depending on a scale parameter in a singular limit $\nu \to 0$), we do not use an approach by spectral decomposition as in \cite{CGMNcpam21}. Instead, we rely on an energy estimate method where deriving a coercive dissipation identity for $\Ls^z$ is essential. We introduce a \emph{matched scalar product}
\begin{equation}\label{def:scalar12_intro}
\langle u,v\rangle_* = \int u v \omega_\nu dz - c_*\nu^2 \int \chi^* u \, \Phi_{\chi^*v} dz, \qquad c_*=2 |a|^4 e^{-|a|^2/4},
\end{equation}
where $\omega_\nu$ is defined in \eqref{def:omega12} and the truncation of the Poison field is to eliminate its effect from the far away zone. The bilinear form \eqref{def:scalar12_intro} matches two bilinear forms: 1) in the outer zone $\cap_{\pm}\{|z\pm a| \gtrsim 1\}$,  the nonlocal term $\nabla \Phi_\vep$ (see \eqref{def:LszNot}) is seen as a small perturbation and $\Ls^z$ without this nonlocal term is self-adjoint in $L^2_{\omega_\nu}$; 2) in the inner zones $\cup_{\pm}\{|z \pm a| \lesssim \nu\}$ the drift term $z\cdot \nabla \vep$ and the term from the far away bubble $\nabla \cdot (\varepsilon \nabla \Phi_{U_\nu (z\pm a)}+U_{\nu}(z\pm a)\nabla \Phi_{\vep})$ are seen as small perturbations, and $\Ls^z$ without these two terms is self-adjoint under the scalar product (see Proposition \ref{pr:scalar-product-interior1})
\begin{equation*}
\langle f, g \rangle_\flat = \int_{\Rb^2} f \Ms^z g dz, \quad \Ms^z \vep =  \vep  \gamma_\nu - \nu^2\Phi_\vep,
\end{equation*}
where $\gamma_\nu$ is defined in \eqref{def:gamma12}. Matching these two bilinear forms yields to \eqref{def:scalar12_intro}. It is indeed a matched scalar product on a set with finite co-dimension defined by suitable local orthogonality conditions related to the sharp logarithmic Hardy-Littlewood-Sobolev inequality, Propositions \ref{pr:scalar-product-interior},
\begin{equation}\label{est:coerIntro2}
\langle  \vep, \vep\rangle_* \sim \|\vep\|^2_{L^2_{\omega_\nu}}.
\end{equation}
We show the coercivity of $\Ls^z$ up to local orthogonality conditions in Proposition \ref{Prop:coercive-Lz-global},
\begin{equation}\label{est:coerIntro}
\langle \tilde \Ls^z \vep, \vep\rangle_* \leq -\delta \Big(\| \nabla \vep\|^2_{L^2_{\omega_\nu}} +\sum_\pm \big\|  \frac{\vep}{(\nu + |z \pm a|)} \big\|_{L^2_{\omega_\nu}}^2 + \| \langle z \rangle \vep\|^2_{L^2_{\omega_\nu}} \Big).
\end{equation}

\noindent - \underline{Step 4} (Improved stationary solution): The sum of the two stationary states is no longer an exact stationary solution and generates nonlinear interaction terms, and we need to relate its parameter derivatives to the approximate eigenfunctions of Step 2. We improve it in the self-similar variables setting by adding a small corrector function $\Xi$, leading to the corrected multi-soliton
$$\mathfrak{U} =  \frac{1}{\nu^2}U(\frac{z - a}{\nu}) + \frac{1}{\nu^2}U(\frac{z+a}{\nu})+ \Xi,$$
as a better approximate solution to the stationary problem of \eqref{eq:wztauIntro} (i.e., $\pa_\tau w = 0$). The crucial task in this step is to extract the leading order term of the generated error and to link it to the first eigenfunction $\phi_0$, from which the remaining error term is admissibly small (see Proposition \ref{prop:E0}). \\

\noindent - \underline{Step 5} (Approximate blowup solution and the law of blowup): The full approximate blowup solution to \eqref{eq:wztauIntro} consists of the improved stationary state and the approximate eigenfunctions, 
$$w = W + \vep, \quad W = \mathfrak{U} + \alpha(\phi_1 - \phi_0).$$
The error generated by $W$ is decomposed into the approximate eigenmodes (see \eqref{def:EztauMod}) and a small remainder term (see Proposition \ref{prop:E}). Using the local orthogonality condition for the coercivity of $\Ls^z \vep$ and projections onto these directions yields the dynamical systems for the law of $\nu, a$ and $\alpha$ (see Lemma \ref{lemm:mod}): 
\begin{align*}
&\Big|-16 \nu^2\Big( \frac{\nu_\tau}{\nu} - \frac{\lambda_0}{2} \Big) + \alpha_\tau -  \alpha \lambda_0 \Big| \lesssim \|q\|_{H^1_\inn} + \|\vep\|_{L^2_{\omega_\nu}} + \frac{\nu^2}{|\ln \nu|}, \\
&\Big|-\alpha_\tau + \alpha \lambda_1 \Big| \lesssim \frac{1}{|\ln \nu|}\Big(\|q\|_{H^1_\inn} + \|\vep\|_{L^2_{\omega_\nu}} + \frac{\nu^2}{|\ln \nu|}\Big), \\
& \Big|a_\tau - \frac{a}{2} + \frac{2 a}{|a|^2} \Big| \lesssim \frac{1}{\nu}\Big(\|q\|_{H^1_\inn} + \|\vep\|_{L^2_{\omega_\nu}} + \frac{\nu^2}{|\ln \nu|}\Big). 
\end{align*}
where $\|q\|_{H^1_\inn}$ and $\|\vep\|_{L^2_{\omega_\nu}}$ measure the smallness of the remainder in some weighted $H^1$ and $L^2$ spaces. They are of the typical size $\nu^2|\ln \nu|^{-1}$, for which we can solve this system asymptotically to obtain the law
$$\alpha \sim \nu^2, \quad a \sim a_\infty, \quad \frac{\nu_\tau}{\nu} \sim \frac{1}{2} \frac{\gamma_1}{\ln \nu}, \quad \nu(\tau) \sim e^{-\sqrt{\gamma_1 \tau}} \quad \textup{as} \;\; \tau \to \infty. $$

\noindent - \underline{Step 6} (Energy estimates): Because of the coercivity \eqref{est:coerIntro}, the main quantity to control the subleading part of the radiation $\vep$ is the matched scalar product \eqref{def:scalar12_intro}, which corresponds to a decay estimate in $L^2_{\omega_\nu}$ thanks to the equivalence \eqref{est:coerIntro2}. Lemma \ref{lemm:L2energy} gives a detailed derivation of the monotonicity of $\langle \vep,\vep \rangle_* $. We perform this energy estimate after a second decomposition,
$$\vep  = \hat\vep + a_0 \phi_0, \quad \langle \hat \vep, \phi_0 \rangle_\ast = 0. $$
We prove in Proposition \ref{prop:decomp_vephat} that there is a norm-equivalence under the adapted scalar product \eqref{def:scalar12_intro}, 
$$\|\vep\|^2_{L^2_{\omega_\nu}} \sim \langle \vep, \vep\rangle_\ast \sim \langle \hat \vep, \hat \vep\rangle_\ast. $$
This second decomposition is introduced so as to produce a cancellation with uncontrollable crossed terms between $\vep$ and the approximate solution (more precisely, a cancellation with the $\textup{Mod}_0$ term in the decomposition of the generated error \eqref{def:EztauMod}).

The sole control of the solution at the $L^2_{\omega_\nu}$ level is not enough to estimate nonlinear terms, so we control $\vep$ in additional higher regularity norms. We perform a $H^1$-energy estimate in the soliton scale $\cup_{\pm}\{|z\pm a| \lesssim \nu\}$, a $H^2$-energy estimate in the parabolic scale $\cap_{\pm}\{|z \pm a| \gtrsim 1\}$ and a weighted $L^\infty$ estimate in the outer zone $\{|z| \gg 1\}$ which allows to close the nonlinear estimates. The $H^1$ energy estimate (see \eqref{def:H1innnorm} for a definition of the norm) is delicate and requires a careful treatment of the scaling term thanks to the localization effect near the soliton-scale (by considering $|z \pm a| \ll 1$) and the dissipative structure of the problem, see Lemma \ref{lemm:H1energy}. The $H^2$ estimate in $\cap_{\pm}\{|z \pm a|\gtrsim 1\}$ is a standard application of parabolic regularization (see Lemma \ref{lemm:H2energy}) and the weighted $L^\infty$ estimate in the outer zone is a standard application of parabolic comparison principle. 

\subsection{Organization of the article}

The notation is introduced in Section \ref{sec:notation}. We start by constructing approximate eigenfunctions in Section \ref{sec:eigenproblem}. The eigenproblem is solved separately in the inner zones in Subsection \ref{sec:inner-eigenfunction} and in the outer zone in Subsection \ref{sec:matched-asymptotic}; the matching then follows in Subsection \ref{sec:matched-asymptotic}. Then, in section \ref{sec:coercivity}, the matched scalar product is studied in Subsection \ref{subsec:matched-scalar-product} and the coercivity of the linearized operator is shown in Subsection \ref{subsec:coercivity} and extended to a second decomposition in Subsection \ref{subsec:coercivity-second-decomposition}. The corrected multisoliton is constructed in Section \ref{sec:corrected-multisoliton}, the approximate nonlinear system it solves is introduced in Subsection \ref{subsec:corrected-multisoliton} and correctors are devised successively in Subsections \ref{subsec:first-ansatz-multisoliton}, \ref{subsec:second-corrections-multisolitons} and \ref{subsec:final-corrections-multisolitons}. Section \ref{sec:bootstrap} describes the bootstrap regime for the full solution; the approximate solution is given in Subsection \ref{subsec:full-approx-sol}, the set of a priori estimates in Subsection \ref{subsec:bootstrap-description} and the modulation equations are derived in Subsection \ref{subsec:modulation-equations}. The analysis in the bootstrap regime is done in Section \ref{sec:energy-estimates} where several energy estimates are performed (the energy estimate based on the matched scalar in Subsection \ref{subsec:energy-matched-scalar}, the inner $H^1$ energy estimate in Subsection \ref{subsec:H1-energy-estimate} and the outer $H^2$ estimate in Subsection \ref{subsec:H2-energy-estimate}) and pointwise bounds are obtained in the outer zone in Subsection \ref{subsec:control-mid-outer}. The main result is then proved as a consequence of a global solution in the bootstrap regime in Subsection \ref{subsec:proof-main-theorem}. The appendix describes the resolvent of the linearization around the stationary state in Section \ref{sec:linear-system}, establishes some bounds for the Poisson field in Section \ref{ap:3} and some estimates in weighted $L^2$ spaces in Section \ref{sec:functional-analysis}.

\subsection{Acknowledgments}

Funded by the European Union. Views and opinions expressed are however those of the author(s) only and do not necessarily reflect those of the European Union or the European Research Council Executive Agency. Neither the European Union nor the granting authority can be held responsible for them.

This work is supported by ERC grant (FloWAS, No. 101117820, DOI 10.3030/101117820). C. Collot is supported by the CY Initiative of Excellence Grant "Investissements d'Avenir" ANR-16-IDEX-0008 via Labex MME-DII,and  the ANR grant "Chaire Professeur Junior" ANR-22-CPJ2-0018-01. T. Ghoul and N. Masmoudi are supported by Tamkeen under the NYU Abu Dhabi Research Institute grant CG002.  V.T. Nguyen is supported by the National Science and Technology Council of Taiwan.

\section{Notation} \label{sec:notation}

\noindent \textbf{Analysis on $\mathbb R^2$}. We use the bracket notation
$$
\langle x \rangle = \sqrt{1 + |x|^2}.
$$
We introduce a smooth cut-off function $\chi$, satisfying $\chi(x) = 1$ if $|x| \leq 1$ and $\chi(x) = 0$ if $|x| \geq 2$. For any $K > 0$ and $a \in \Rb^2$, we denote
\begin{equation}\label{def:chiK}
\chi_{_{a, K}}(x) = \chi\Big( \frac{x - a}{K} \Big), \quad \chi_{_K}(x) = \chi\Big( \frac{x}{K} \Big).
\end{equation}
The reader should not be confused with the notation $f_{\nu, a}$ below, as we always use $\chi$ for the cut-off function throughout this paper. We let $\zeta^*=10$, introduce a fixed constant $\zeta_*\ll 1$ and let
\begin{equation}\label{def:chi-star}
\quad  \chi^*(x) =\chi_{\zeta^*}(x), \quad \chi_{_*}(x) = \chi_{\frac{\zeta_*}{\nu}}(x).
\end{equation}

\noindent \textbf{Symmetries and stationary states}. For $\nu > 0$ and $a \in \Rb^2$ and for any function $f$, we write
$$
f_{\nu,a}(x) = \frac{1}{\nu^2} f \Big( \frac{x - a}{\nu}\Big), \quad f_\nu(x) =  \frac{1}{\nu^2} f \Big( \frac{x}{\nu}\Big).
$$
The associated scaling generator is
\begin{equation}\label{def:LambdaIntro}
\Lambda w (x) = \nabla \cdot(x w) = 2w(x) + x \cdot \nabla w(x).
\end{equation}
We introduce the following notation for the two stationary states,
\begin{equation*}
U_\nu(z) = \frac{8 \nu^2}{(\nu^2 + |z|^2)^2}, \quad U_{1, \nu} = U_\nu(z - a), \quad U_{2, \nu}(z) = U_\nu(z + a).
\end{equation*}
The Poison field they generate is
\begin{equation*}
\nabla \Phi_{U_\nu}(z) = -\frac{4 z}{\nu^2  + |z|^2}, \quad \nabla \Phi_{U_{1, \nu}}(z) = \nabla \Phi_{U_\nu}(z - a), \quad \nabla \Phi_{U_{2, \nu}}(z) = \nabla \Phi_{U_\nu}(z + a).
\end{equation*}
To shorten the writing, we introduce
\begin{equation} \label{def:fnuU12}
U_{12, \nu}(z) = U_{1, \nu}(z) U_{2, \nu}(z), \quad U_{1+2, \nu}(z) = U_{1, \nu}(z) + U_{2, \nu}(z).
\end{equation}
The asymptotic position of the stationary states is
$$
a_\infty=(2,0).
$$

\noindent \textbf{Self-similar variables}. Given $T>0$ the parabolic self-similar variables are
\begin{equation}\label{def:wztau_vars-1}
z = \frac{x}{\sqrt{T-t}} \qquad \mbox{and} \qquad \tau = -\log(T-t).
\end{equation}
The corresponding renormalization of the unknown is
\begin{equation}\label{def:wztau_vars-2}
u(x,t) = \frac{1}{T-t}w(z, \tau) \qquad \mbox{and} \qquad \nabla_x \Phi_u(x,t) = \frac{1}{\sqrt{T-t}}\nabla_z \Phi_w(z, \tau).
\end{equation}
Then $u$ solves the Keller-Segel system \eqref{eq:KS2d} if and only if $w$ solves the self-similar equation
\begin{equation}\label{eq:wztau}
w_\tau = \nabla \cdot (\nabla w - w \nabla \Phi_w) - \beta \Lambda w, \qquad  \beta = \frac{1}{2},
\end{equation}

\noindent \textbf{Linearized operators.} Linearizing $w$ around $U_{1,\nu} + U_{2,\nu}$ leads to the linearized operator in parabolic variables
\begin{equation}\label{def:LszNot}
\Ls^z \vep = \Delta \vep - \nabla \vep \cdot \nabla \Phi_{U_{1+2, \nu}} - \nabla U_{1+2, \nu} \cdot \nabla \Phi_{\vep } + 2 U_{1+2, \nu} \vep - \beta \Lambda \vep, \quad \beta = \frac 12,
\end{equation}
The linearized operator around one bubble $U(y)$ is denoted by $\Ls_0$, 
\begin{equation}\label{def:Ls0Ms0}
\Ls_0 u = \nabla \cdot\big( \nabla u - u\nabla \Phi_U - U\nabla \Phi_u\big) = \nabla \cdot \big( U \nabla \Ms_0 u\big),
\end{equation}
where  
\begin{equation}\label{def:Ms0}
\Ms_0 u = \frac{u}{U} - \Phi_u. 
\end{equation}
The linearized operator with truncated Poisson field is
\begin{equation}\label{def:LszNottilde}
\tilde \Ls^z \vep = \Delta \vep - \nabla \vep \cdot \nabla \Phi_{U_{1+2, \nu}} - \nabla U_{1+2, \nu} \cdot \nabla \Phi_{\chi^*\vep } + 2 U_{1+2, \nu} \vep - \beta \Lambda \vep.
\end{equation}

\noindent \textbf{Inequalities}. We write $a\lesssim b$ if there exists $C>0$ such that $a\leq Cb$ uniformly in the parameters' range at stake, and $a\sim b$ if simultaneously $a\lesssim b$ and $b\lesssim a$.\\

\noindent \textbf{Weights}. We denote by $L^2_g(\Rb^2)$ the weighted $L^2$ space with general weight function $g$, equipped with the scalar product 
$$\langle u, v \rangle_g = \int_{\Rb^2} u v \;g dz, \quad \|u\|_g = \sqrt{\langle u, u \rangle_g}.$$ 
When $g \equiv 1$, we simply write 
$$
 \langle u, v \rangle \equiv \langle u, v \rangle_1 \quad \quad   \|u\|\equiv \|u\|_1  .$$
We use the weight functions 
\begin{align} \label{def:gamma12}
& \gamma_\nu(z) = \frac{\nu^2}{U_{1+2, \nu}(z)},\\
 \label{def:omega12}
& \omega_{\nu}(z) = \frac{\nu^4}{U_{12,\nu}(z)}e^{-\beta z^2/2}, \qquad \nabla \ln \omega_{\nu} = -\nabla \Phi_{U_{1+2, \nu}} - \beta z,\\
& \label{def:rho} \rho_{\nu}(z) = \nu^2 \frac{U_{1+2, \nu}(z)}{U_{12,\nu}(z)} e^{-\frac{\beta z^2}{2}} = \nu^{-2} U_{1+2, \nu} \omega_{\nu}(z), \quad \rho(z) = e^{-\frac{\beta |z|^2}{2}}.
\end{align}
The function $\gamma_\nu$ is the weight function for the inner zone (near the soliton scale) $|z \pm a| \ll 1$ where the linear part $\Lambda$ is considered as a small perturbation, the function $\omega_\nu$ is the exact weight function for the outer zone $|z \pm a| \gtrsim 1$ where the Poisson field $\nabla\Phi_{\vep}$ is negligible, and $\rho_\nu$ is a global weight function.

\section{The two approximate eigenfunctions} \label{sec:eigenproblem}

In this section we find suitable approximate eigenfunctions of the linearized operator in parabolic self-similar variables around two concentrated stationary states $\tilde \Ls^z$ defined by \eqref{def:LszNottilde}. The reason why we consider $\tilde \Ls^z$ and not the full operator $\Ls^z$ is because the eigenfunctions will grow at spatial infinity what will prevent their Poisson field to be well-defined, which is why we localize the source term in the Poisson field in $\tilde \Ls^z$. The error terms created by this localization will be estimated later on. We will find two couples $(\lambda_0,\phi_0)$ and $(\lambda_1,\phi_1)$ such that for $i=0,1$, the function $R_i$ defined as the error in the eigenfunction equation
\begin{equation} \label{eigen:id:Ri}
R_i=\tilde \Ls^z \phi_i -\lambda_i \phi_i 
\end{equation}
is small. For $\beta>0$, the equilibrium position for the stationary states is denoted by $a_\infty=\left(\sqrt{\frac 2\beta},0\right)$. We introduce the parameter $\alpha=\frac{2}{|a|^2}$, so that $|\alpha-\beta|\lesssim |a-a_\infty|$ for $a$ close to $a_{\infty}$.

\begin{proposition}[Two approximate eigenfunctions] \label{prop:eigen} There exists $\tilde \epsilon>0$  and $\zeta_*>0$ small enough such that the following holds true for all $1/4\leq \beta \leq 1$, $0 < \nu \ll 1$ small enough and $a=(|a|,0)\in \mathbb R^2$ with $|a-a_\infty|\lesssim \nu$. There exist two smooth functions $\phi_0[\nu,a]$ and $\phi_1[\nu,a]$ and two approximate eigenvalues
\begin{align}
\label{eigenfunction:id:tildelambda0}&\lambda_0=2\beta \Big(1 +\tilde \lambda_0 \Big), \qquad \tilde \lambda_0(\nu)=\frac{\gamma_0}{\ln \nu}+\Oc(|\ln \nu|^{-2}), \quad \gamma_0\in \mathbb R, \\
\label{eigenfunction:id:tildelambda1}& \lambda_1= 2\beta \tilde \lambda_1, \qquad \tilde \lambda_1(\nu)=\frac{\gamma_1}{\ln \nu}+\Oc(|\ln \nu|^{-2}) , \quad  \gamma_1 = \frac{3}{16} - 16A,
\end{align}
where $A < 0$ is given in Proposition \ref{pr:phi1out}, and where the exact expression of $\tilde \lambda_i$ is given by \eqref{id:value-tilde-lambdai-tech}, that satisfy the following.

\noindent (i) \textup{(Approximate eigenfunctions)} We have the decomposition
\begin{align}
\label{def:phii-2} \phi_{i}(z) &= -\frac{1}{16\nu^{4}}\sum_\pm \phi_{i,\pm}^\inn\Big(\frac{z\mp a}{\nu}\Big)  \chi_{\pm a,\nu^{\tilde \epsilon}}(z) + \phi_{i}^\out(z) \big(1 - \sum_\pm \chi_{\pm a, \nu^{\tilde \epsilon}}\big),  \\
& = \sum_\pm \left(-\frac{1}{16 \nu^4} \Lambda U \Big( \frac{z \mp a}{\nu} \Big)\pm \frac{L_i}{\nu^3}\partial_{x_1}U\Big( \frac{z \mp a}{\nu} \Big)\right)  \chi_{\pm a, \zeta_*}(z)  + \tilde \phi_i(z),\label{def:phii}
\end{align} 
where $\phi_{i,\pm}^\inn$, $\phi^\out_{i}$ are given explicitly in Proposition \ref{prop:phi1inn} and Proposition \ref{pr:phi1out} with the constants being given by \eqref{id:value-tilde-lambdai-tech}, \eqref{id:value-parameters-tech} and \eqref{eigenfunctions:id:estimatetildeLtildeM} and where we introduced $L_i=-\frac{\alpha^{1/2}L_i^\inn}{16}$ that satisfies $|L_0|\lesssim 1$ and $|L_1|\lesssim |\ln \nu|$. In particular, we have for $k \in \mathbb{N}$,
\begin{align}\label{est:pointwise_phii}
& |\nabla^k \phi_i(z)|  + |\nabla^k \nu \partial_\nu\phi_i(z)|  \lesssim  \sum \frac{\langle z \rangle^{C_k}}{(\nu + |z \pm a|)^{4+k}}, \qquad \mbox{for }k\in \mathbb N,\\
& \label{est:pointwise_Phi_phii}
|\nabla^k \Phi_{\chi^* \phi_i}(z)| \lesssim  \sum \frac{1+i|\ln \nu|}{ (\nu + |z \pm a|)^{2+k} }\langle z \rangle^2 , \qquad  \mbox{for }k\geq 1.
\end{align}
\begin{align}\label{est:pointwise_phiitil}
 |\nabla^k \tilde \phi_i(z)|   \lesssim  \sum \frac{\langle z \rangle^{C_k}}{(\nu + |z \pm a|)^{2+k}}, \quad  |\nabla^k \nu \partial_\nu\tilde \phi_i(z)|\lesssim  \sum \frac{\langle z \rangle^{C_k}}{(\nu + |z \pm a|)^{2+k}}\left(\frac{\nu^{\delta_0}}{(\nu+|z\mp a|)^{\delta_0}}+\frac{1}{|\ln \nu|^2}\right)
\end{align}
\begin{align}\label{est:pointwise_Phiphiitil}
|\nabla^k \Phi_{\chi^* \tilde \phi_i}(z)| \lesssim  \sum \frac{1+i|\ln \nu|}{ (\nu + |z \pm a|)^{k} } , \qquad  \mbox{for }k\geq 1.
\end{align}
\begin{align}\label{est:pointwise_phiitil_da}
|\nabla^k \partial_a \phi_i| \lesssim \sum \frac{\langle z \rangle^{C_k}}{(\nu + |z \pm a|)^{5+k}}, \quad |\nabla^k \partial_a \tilde \phi_i| \lesssim \sum \frac{\langle z \rangle^{C_k}}{(\nu + |z \pm a|)^{3+k}}.
\end{align}

\noindent (iii) \textup{(Improved estimates near $\pm a$)} For $k\in \mathbb{N}$, we have 
\begin{align}\label{est:pointwise_phi1m0}
|\nabla ^k (\phi_1 - \phi_0)| \lesssim \sum_{\pm} \frac{|\ln \nu|\langle z \rangle^{C_k}}{(\nu + |z \pm a|)^{3 + k}}, \quad |\nabla \Phi_{\phi_1 - \phi_0}(z)| \lesssim  \sum_{\pm} \frac{|\ln \nu | \langle z \rangle^C}{(\nu + |z \pm a|)}. 
\end{align}
and
\begin{align}\label{est:pointwise_phi1m0_dnu}
|\nabla^k \nu \partial_\nu (\phi_1 - \phi_0)| \lesssim \sum \frac{|\ln \nu|}{ (\nu + |z \pm a|)^{3+ k}} \langle z \rangle^{C_k},
\end{align}


\noindent (iv) \textup{(Pointwise estimates of $R_i$)}. The errors in the approximate eigenfunction equation
\begin{equation} \label{eq:eigenfunctions_phi10}
R_0=\tilde \Ls^z \phi_0 -\lambda_0 \phi_0 \qquad \mbox{and} \qquad R_1= \tilde \Ls^z \phi_1 -\lambda_1 \phi_1
\end{equation}
satisfy
\begin{equation}\label{est:pointwise_Ri}
\sum_{\pm} | \langle R_{i} , \mathbbm{1}_{|z \pm a| < 1} \rangle | \lesssim \frac{1}{|\ln \nu|} , \quad |\nabla^k R_i (z)| \lesssim \frac{1}{|\ln \nu|^2} \sum_{\pm} \frac{1  }{(\nu + |z \pm a|)^{2+k}} \langle z \rangle^{C_k}.
\end{equation}
\end{proposition}

\begin{remark}

Our approximate eigenvalues are precise up to the order $|\ln \nu|^{-1}$. The exact identity \eqref{id:value-tilde-lambdai-tech} we take for $\tilde \lambda_i$ does not give the correct next order $|\ln \nu|^{-2}$ term. This is because this identity is useful in our construction of the approximate eigenfunctions, but that this approximate eigenfunctions are only correct up to an order $|\ln \nu|^{-2}$.

Note that in the approximate eigenfunction equation \eqref{eq:eigenfunctions_phi10} the estimate \eqref{est:pointwise_Ri} quantify how much $R_i$ is negligible compared with $\lambda_i\phi_i$. We have roughly $|\phi_i|\approx (\nu+|z-a|)^{-4} (\nu+|z+a|)^{-4}\langle z \rangle^C $ by Proposition \ref{prop:phi1inn} and Proposition \ref{pr:phi1out}, so that \eqref{est:pointwise_Ri} shows that $R_i$ gains a $|\log \nu|^{i-2}$ factor, as well as a $(\nu+|z\mp a|)^2$ factor near $\pm a$.

\end{remark}


\medskip

\noindent The proof of Proposition \ref{prop:eigen} consists of several parts represented in different subsections: 
\begin{itemize}
    \item In subsection \ref{sec:inner-eigenfunction},  we construct an inner approximate eigenfunction located on $|z\pm a| \ll 1$, where the linearized operator behaves like the one generated by linearization around one bubble. The Poison field generated by other bubble can be approximated by a truncated Taylor expansion, from which we can compute its effect up to any accuracy order. 
\item In subsection \ref{sec:outer-eigenfunction}, we construct an outer approximate eigenfunction located on $|z \pm a| \sim 1$, where we eliminate the effect of the nonlocal term, for which the linearized operator behaves like the one of Fokker-Planck type operator with a singular potential.
\item In subsection \ref{sec:matched-asymptotic}, we perform a matched asymptotic expansion on the common region that gives approximate eigenvalues. We also derive various estimate concerning the approximate eigenfunctions and the generated error. 
\end{itemize}

\subsection{Inner expansion of the eigenfunctions} \label{sec:inner-eigenfunction}


In this subsection we obtain approximate eigenfunctions $\phi^{\inn,z}_i$ near $a$ and $-a$ for $i=0,1$. We replace in \eqref{eigen:id:Ri} $\Phi_{\chi^* \phi^{\inn,z}_i}$ by a solution $V^{\inn,z}_i$ of $-\Delta V^{\inn,z}_i=\phi_i^{\inn,z}$. This yields the following inner approximate eigenfunction system
\begin{equation} \label{eigenfunctions-interior:id:eigenfunction-system}
\left\{\begin{array}{l l l}
& R_i^{\inn,z} = \Ls^z(\phi_i^{\inn,z},V_i^{\inn,z} )-\lambda_i \phi_i^{\inn,z}, \\
&-\Delta V_i^{\inn,z} = \phi_i^{\inn,z},
\end{array}
\right.
\end{equation}
that will be studied in this subsection, where
$$
\Ls^z(\phi,V)=\Delta \phi - \nabla \phi \cdot \nabla \Phi_{U_{1+2, \nu}} - \nabla U_{1+2, \nu} \cdot \nabla V + 2 U_{1+2, \nu} \phi - \beta \Lambda \phi.
$$

\subsubsection{Renormalised error in the eigenfunction equation}

Consider $i=0$ or $i=1$ and define the new variables $z_{\pm}$ centred at $\pm a$ and its rescaled version
$$
 z_{\pm}=z \mp a, \qquad y_{\pm}=\frac{z_{\pm}}{\nu},
$$
where $\nu>0$. In the polar coordinates, we write 
$$
\theta_\pm=\frac{z_\pm}{|z_{\pm}|}=\frac{y_\pm}{|y_\pm|}, \qquad r_\pm=|y_\pm|, \qquad \zeta_\pm =|z_\pm|
$$
so that
$$
z_\pm =\zeta_\pm (\cos \theta_\pm , \sin \theta_\pm), \qquad y_\pm = r_\pm (\cos \theta_\pm , \sin \theta_\pm).
$$
We define the renormalized inner eigenfunction $\phi^\inn_{i, \pm}$, Poisson field $V^\inn_{i,\pm}$ and error $R^\inn_{i, \pm}$ as
$$
\phi^\inn_{i, \pm}(y_\pm)= \nu^4 \phi_i^{\inn,z} (z), \qquad V^\inn_{i, \pm}(y_\pm)= \nu^2 V_i^{\inn,z} (z), \qquad R^\inn_{i, \pm}(y_\pm)= \nu^6 R_i^{\inn,z} (z).
$$
By writing
\begin{align*}
U_{1+2,\nu}(z)=\frac{1}{\nu^2}\Big(U(y\pm)+U\big(y_\pm \pm \frac{2a}{\nu}\big)\Big), \qquad \Phi_{U_{1+2,\nu}}(z)=\Phi_U (y_\pm)+\Phi_U\big(y_\pm\pm\frac{2a}{\nu}\big).
\end{align*}
this transforms \eqref{eigenfunctions-interior:id:eigenfunction-system} into the renormalized inner approximate eigenfunction system
\begin{equation}  \label{eigen:id:Riinn1}
\left\{\begin{array}{l l l}
& R_{i,\pm}^{\inn} = \Ls_0( \phi^\inn_{i, \pm},V_{i,\pm}^{\inn}) - \nu^2 \big(\beta \Lambda   + \lambda_i\big) \phi_{i, \pm}^\inn  - \nabla . \Big( \phi_{i, \pm}^\inn \big[ \nabla \Phi_U(y_\pm  \pm \frac{2a}{\nu}) \pm \nu \beta a   \big] \Big) \\
 & \qquad \qquad  - \nabla. \Big( U(y_\pm \pm 2a/\nu) \nabla V^\inn_{i,\pm}\Big)\\
&-\Delta V_{i,\pm}^{\inn} = \phi_{i,\pm}^{\inn}
\end{array}
\right.
\end{equation}
where
$$
 \Ls_0 (\phi^\inn_{i, \pm},V^\inn_{i,\pm})  =\nabla.\Big(\nabla  \phi^\inn_{i, \pm} - \phi^\inn_{i, \pm} \nabla \Phi_U-U\nabla V^\inn_{i,\pm}\Big).
$$

\subsubsection{Leading order system and decomposition of the error $R_{i, \pm}^\inn$}

To extract the leading order terms in the right-hand side of \eqref{eigen:id:Riinn1} we perform the Taylor expansion of $ \nabla \Phi_U(y_\pm  \pm \frac{2a}{\nu}) $.  For sake of simplicity, we write 
$$y = y_\pm, \quad r = r_\pm, \quad \theta=\theta_\pm,$$
wherever there is no need to specify the sign. 
\begin{lemma} \label{lem:taylorpoissonfield}

For $a=(|a|,0)$ with $|a|\approx 1$, there holds
\begin{equation} \label{id:expansion:PhiU-other-bubble}
\nabla \Phi_{U}(y \pm 2a/\nu) =\mp \alpha a \nu +\sum_{i = 1}^3 \alpha^{\frac{i+1}{2}} \nu^{i+1} G_{i,\pm} (y) +  \tilde G_\pm(y),
\end{equation}
where $|\nabla^k \tilde G_\pm(y)| = \Oc(\nu^5\langle y\rangle^{4-k})$ for $|y|\ll \nu^{-1}$, and $G_{i,\pm}(y)$ are homogeneous in $y$ of degree $i$. In particular, we have 
\begin{align} \label{eigenfunctions:id:G1}
G_{1,\pm}(y)  &= G_1(y) =  \frac{1}{2}(y_1, - y_2) = \frac{1}{4}\nabla (y_1^2 - y_2^2) = \frac{1}{4}\nabla \big(r^2 \cos 2\theta \big),
\end{align}
\begin{align}
\label{eigenfunctions:id:G2} G_{2, \pm}(y) = \pm G_{2}(y), \quad  G_2(y)  & = -\frac{1}{4\sqrt{2}} \Big(y_1^2 - y_2^2, - 2 y_1 y_2 \Big)  + \frac{1}{4\sqrt{2}}(1,0)\\
 \nonumber &=  -\frac{1}{12\sqrt{2}} \nabla (y_1^3 - 3y_1 y_2^2) + \frac{1}{4\sqrt{2}}\nabla y_1 \\
 \nonumber & = -\frac{1}{12\sqrt{2}}\nabla \big(r^3 \cos 3\theta\big) + \frac{1}{4\sqrt{2}}\nabla(r\cos \theta),
\end{align}	
and 
\begin{align}
\label{eigenfunctions:id:G3}G_{3, \pm}(y)  = G_3(y)&= -\frac{1}{16}\Big( 3 y_1 y_2^2 - y_1^3, 3 y_1^2 y_2 - y_2^3 \Big) +\frac{1}{16}(-3y_1,y_2) \\
\nonumber & = -\frac{1}{64} \nabla \big( 6 y_1^2 y_2^2 - y_1^4 - y_2^4 \big) + \frac{1}{32} \nabla (-3 y_1^2 +y_2^2) \\
\nonumber &= \frac{1}{64} \nabla \big(r^4\cos 4\theta \big)-\frac{1}{32} \nabla \big(r^2 (1 +2 \cos 2\theta)\big).
\end{align}
\end{lemma}
\begin{proof} Let $F(b) = \nabla \Phi_U(b y \pm 2a/\nu)$ and use the Taylor expansion for the function $F(b)$ at $b = 0$, then, let $b = 1$ to get the desired result.  See Appendix \ref{ap:3} for a computation which can be simply verified by Matlab symbolic. 
\end{proof}
Using Lemma \ref{lem:taylorpoissonfield}, and anticipating that $|a-a_\infty|\lesssim \nu$ so that $|\alpha-\beta|\lesssim \nu$, we decompose the error $R^\inn_\pm$ as
 \begin{align*}
R^\inn_{i, \pm} &=\tilde R^\inn_{i, \pm} +\bar R^\inn_{i, \pm} ,
\end{align*}
where $\tilde R^\inn_{i, \pm} $ contains the leading order terms defined by the following renormalized inner approximate eigenfunction system
 \begin{equation} \label{eigenfunction:id:tildeRinn1}
\left\{ \begin{array}{r l} \tilde R^\inn_{i, \pm}  =&  \Ls_0  (\phi^\inn_{i, \pm},V^\inn_{i,\pm}) - \nu^2 \alpha \left( \big( \Lambda   + \frac{\lambda_i}{\beta}\big) \phi_{i, \pm}^\inn +  \nabla.\Big( \phi_{i, \pm}^\inn G_{1,\pm}\Big) \right)\\
&-\alpha^{\frac 32}\nu^3 \nabla.\Big( \phi_{i, \pm}^\inn G_{2,\pm}\Big)-\alpha^2\nu^4 \nabla.\Big( \phi_{i, \pm}^\inn G_{3,\pm}\Big),\\
-\Delta V^\inn_{i,\pm}=& \phi^\inn_{i,\pm}
\end{array} \right.
\end{equation}
and $\bar R^\inn_{i, \pm}$ contains the subleading remainding terms defined by
\begin{align} 
\bar R^\inn_{i, \pm} &= \pm \nu (\alpha-\beta)a.\nabla \phi_{i, \pm}^\inn +\nu^2(\alpha-\beta) \left( \Lambda   + \frac{\lambda_i}{\beta}\right) \ \phi_{i,\pm}^\inn \nonumber\\
& \qquad  - \nabla. \Big( U(y_\pm \pm 2a/\nu) \nabla V^\inn_{i,\pm}\Big)- \nabla . \Big( \phi_{i, \pm}^\inn \tilde G_{\pm} \Big). \label{eigenfunction:id:barRinn1}
\end{align}

\subsubsection{First truncated expansion for the leading system}

Our aim is to construct a first version of the inner eigenfunction and Poisson field $(\phi^\inn_{i,\pm},V^\inn_{i,\pm})$ so that the error $\tilde R^\inn_{i, \pm}$ in the leading system \eqref{eigenfunction:id:tildeRinn1} is small in some suitable function space. There holds $\lambda_i= 2\beta(1 - i) +   2\beta \tilde \lambda_i$. Since we expect $|\tilde \lambda_i|\ll 1$ and $|a-a_\infty|\lesssim \nu$ so that $|\alpha-\beta|\lesssim \nu$, we write \eqref{eigenfunction:id:tildeRinn} as 
 \begin{equation} \label{eigenfunction:id:tildeRinn}
\left\{ \begin{array}{r l} \tilde R^\inn_{i, \pm}(\phi^\inn_{i, \pm},V^\inn_{i,\pm}) =&\Ls_0 (\phi^\inn_{i, \pm},V^\inn_{i,\pm}) - \alpha  \nu^2\left( \Lambda_i  \phi_{i, \pm}^\inn + \nabla.\Big( \phi_{i, \pm}^\inn G_{1,\pm}\Big) \right)\\
&-\alpha^{\frac 32}\nu^3 \nabla.\Big( \phi_{i, \pm}^\inn G_{2,\pm}\Big)-\alpha^2 \nu^4 \nabla.\Big( \phi_{i, \pm}^\inn G_{3,\pm}\Big) - 2 \alpha \nu^2 \tilde \lambda_i \phi_{i, \pm}^\inn,\\
-\Delta V^\inn_{i,\pm}=& \phi^\inn_{i,\pm},
\end{array} \right.
\end{equation}
where $\Lambda_i$ is defined by 
\begin{equation*}
\Lambda_i = \Lambda + 2(1 - i) = z.\nabla+  2(2-i). 
\end{equation*}
We look as a first attempt for an approximate inner eigenfunction of the form of a truncated power series in the small parameters $\nu$ and $\tilde \lambda_i$,
\begin{align}
\label{eigenfunctions:id:tildephiinn}\tilde \phi^\inn_{i,\pm} &= \Lambda U + \alpha \nu^2 T_{2}^{(i)} \pm \alpha^{\frac 32}\nu^3 T_{3} +\alpha^2\nu^4  T_{4}^{(i)}  + \alpha \tilde \lambda_i \nu^2 \hat T_2  +\tilde \lambda_i \alpha^2 \nu^4 \hat T_{4}^{(i)},\\
\label{eigenfunctions:id:tildeViinn} \tilde V^\inn_{i,\pm}&= \Phi_{\Lambda U} + \alpha \nu^2 V_2^{(i)} \pm \alpha^{\frac 32}\nu^3 V_{3} +\alpha^2 \nu^4  V_4^{(i)}  + \alpha \tilde \lambda_i \nu^2 \hat V_2 + \tilde \lambda_i\alpha^2 \nu^4 \hat V_4^{(i)},
\end{align}
where we require
\begin{align} \label{eigenfunction:id:T2} 
&\Ls_0( T_{2}^{(i)},V_2^{(i)}) =  \Lambda_i \Lambda U+\nabla. (\Lambda U G_1),\\
\label{eigenfunction:id:hatT2}
& \Ls_0( \hat T_{2} ,\hat V_2)=2\Lambda U,\\
 \label{eigenfunction:id:T3} 
& \Ls_0  (T_3,V_3) = \nabla. \big(G_{2} \Lambda U\big), \\
\label{eigenfunction:id:T4} 
& \Ls_0 (T_{4}^{(i)},V_4^{(i)}) = \Lambda_i T_2^{(i)} + \nabla. (G_3 \Lambda U) + \nabla. \big(T_2^{(i)} G_1\big),\\
 \label{eigenfunction:id:hatT4} 
& \Ls_0 (\hat T_4^{(i)},\hat V_4^{(i)}) = 2T_2^{(i)}  +  \nabla.(G_1 \hat T_{2}) + \Lambda_i \hat T_{2}, \\
 \label{eigenfunction:id:Poisson-fields}
&-\Delta V_2^{(i)}=T_2^{(i)}, \;\;-\Delta V_4^{(i)}=T_4^{(i)}, \;\; -\Delta V_3=T_3,\;\; -\Delta \hat V_2=\hat T_2 ,\quad -\Delta \hat V_4^{(i)}=\hat T_4^{(i)}.
\end{align}

Plugging \eqref{eigenfunction:id:T2}-\eqref{eigenfunction:id:hatT4} in \eqref{eigenfunction:id:tildeRinn} and recalling $\Ls_0 (\Lambda U, \Phi_{\Lambda U} )=0$, we obtain that $\tilde R^\inn_{i,\pm}(\tilde \phi^\inn_{i, \pm},\tilde V^\inn_{i,\pm})$ consists of the terms of  magnitude order $\Oc(\tilde \lambda_i^2 \nu^4 + \nu^5)$ or higher,
 \begin{align}
\nonumber \tilde R^\inn_{i,\pm}(\tilde \phi^\inn_{i, \pm},\tilde V^\inn_{i,\pm}) =&- 2\alpha^2 \tilde \lambda_i^2 \nu^4 \hat T_2 \\
\nonumber &+\alpha^{\frac 52}\nu^5\left(\mp  \Lambda_i T_3 \mp \nabla.(T_3G_1)\mp  \nabla.(T_2^{(i)}G_2) \right) +\alpha^{\frac 52}\tilde \lambda_i \nu^5\left( \mp \nabla.(\hat T_2 G_2)\mp 2T_3 \right) \\
\nonumber &+\alpha^3\nu^6\left(-  \Lambda_i T_4^{(i)} - \nabla.(T_4^{(i)}G_1)- \nabla.(T_3G_2)- \nabla.(T_2^{(i)}G_3) \right) \\
\nonumber & +\alpha^3\tilde \lambda_i \nu^6\left(-  \Lambda_i \hat T_4^{(i)} - \nabla.(\hat T_4^{(i)}G_1)- \nabla.(\hat T_2G_3) -2T_4^{(i)}\right) - 2\alpha^3\tilde \lambda_i^2 \nu^6 \hat T_4^{(i)} \\
\nonumber &+\alpha^{\frac 72}\nu^7 \left( \mp \nabla.(T_4^{(i)} G_{2})\mp \nabla.(T_3 G_{3}) \right) \mp \alpha^{\frac 72}\tilde \lambda_i \nu^7 \nabla . (\hat T_4^{(i)} G_2)\\
\label{eigenfunctions:id:tildeRinn-decomposition} &-\alpha^4 \nu^8\nabla.(T_4^{(i)} G_{3}) -\alpha^4 \tilde \lambda_i \nu^8\nabla.(\hat T_4^{(i)} G_{3}).
\end{align}
We then compute the remainder associated to the subleading terms by plugging \eqref{eigenfunction:id:T2}-\eqref{eigenfunction:id:hatT4} in \eqref{eigenfunction:id:barRinn1} and using $U\Big(\frac{2a}{\nu}\Big)=\frac{\nu^4}{2|a|^4}+O(\nu^6)$:
\begin{align*}
\bar R^\inn_{i,\pm}(\tilde \phi^\inn_{i, \pm},\tilde V^\inn_{i,\pm}) & =\pm \nu (\alpha-\beta)a.\nabla \Lambda U+\nu^2 (\alpha-\beta)\left(\Lambda+\frac{\lambda_i}{\beta}\right) \Lambda U\pm \nu^3 (\alpha-\beta)\alpha a.\nabla ( T_2^{(i)}+\tilde \lambda_i\hat T_2) \\
& \quad +\frac{\nu^4}{2|a|^4} \Lambda U  \pm \nu (\alpha-\beta)a.\nabla (\tilde \phi_{i,\pm}^\inn-\Lambda U-\alpha \nu^2 T^{(i)}_2-\alpha \nu^2 \tilde \lambda_i \hat T_2)\\
& \quad +\nu^2 (\alpha-\beta)\left(\Lambda+\frac{\lambda_i}{\beta}\right)(\tilde \phi_{i,\pm}^\inn-\Lambda U) +\left(U\Big(y\pm \frac{2a}{\nu}\Big)-\frac{\nu^4}{2|a|^4}\right)\Lambda U\\
& \quad +U\Big(y\pm \frac{2a}{\nu}\Big)(\tilde \phi_{i,\pm}^\inn-\Lambda U) -\nabla U\Big(y\pm \frac{2a}{\nu}\Big).\nabla \tilde V^\inn_{i,\pm}-\nabla .(\tilde \phi^\inn_{i,\pm}\tilde G_\pm).
\end{align*}
Above, the first line gathers terms of order $\nu |a-a_\infty|$, $\nu^2 |a-a_\infty|$, $\nu^4$ and $\nu^3 |a-a_\infty|$ which are not small enough for $|y|\lesssim 1$, but which are subleading for $|y|\gg 1$ in comparison with $\tilde R^\inn_{i,\pm}$ due to a stronger decay. The other lines gather terms of order $\nu^5$ or $\nu^4 |a-a_\infty|$ and higher which are small enough. In order to avoid the terms on the first line, we will add corrections in a second truncated expansion.

\subsubsection{Second expansion with corrections for subleading remainder and Poisson field}

We now look for an expansion that manages simultaneously the main order error terms in $\tilde R^\inn_{i,\pm}$ and $\bar R^\inn_{i,\pm}$. We also want to include corrections that allow the Poisson field $V^\inn_{i,\pm}$ to behave like $y_1$ or $(y_1^2-y_2^2)$ as $|y|\to \infty$. This is because the Poisson field for the full solution will actually depend from the outer part of the solution, and that the main part of this nonlocal effect in the Poisson equation will be located on these first two harmonic polynomials. We look for an expansion of the form
\begin{align}
\label{eigenfunctions:id:decomposition-phiinn} \phi^\inn_{i,\pm} &=\tilde \phi^\inn_{i,\pm}+\bar \phi^\inn_{i,\pm} \\
\nonumber \bar \phi^\inn_{i,\pm}  &=\pm \alpha^{\frac 12} \nu L_i^\inn \partial_{x_1}U\pm  \nu (\alpha-\beta)\mathcal S_{1,1}+ \nu^2 (\alpha-\beta)\mathcal S_{2,1}+\nu^2(\alpha-\beta)^2\mathcal S_{2,2} \pm \alpha^{\frac 32}\nu^3L_i^\inn \mathcal S_{3,0}\\
\label{eigenfunctions:id:barphiinn} & \quad  \pm \nu^3 (\alpha-\beta)\mathcal S_{3,1} +\nu^4\mathcal S_{4,0} +\alpha^2 \nu^4M_i^\inn \mathcal S_{4,0}',\\
\label{eigenfunctions:id:decomposition-Vinn} V^\inn_{i,\pm}&= \tilde V^\inn_{i,\pm}+\bar V_{i,\pm}\\
\nonumber \bar V_{i,\pm} & =\pm \alpha^{\frac 12} \nu L_i^\inn \partial_{x_1}\Phi_{U} \pm  \nu (\alpha-\beta)\mathcal W_{1,1}+ \nu^2 (\alpha-\beta)\mathcal W_{2,1}+\nu^2(\alpha-\beta)^2\mathcal W_{2,2}  \\
\label{eigenfunctions:id:barVinn} & \quad \pm \alpha^{\frac 32} \nu^3 L_i^\inn \mathcal W_{3,0} \pm \nu^3 (\alpha-\beta)\mathcal W_{3,1} +\nu^4\mathcal W_{4,0} +\alpha^2 \nu^4M_i^\inn \mathcal W_{4,0}',
\end{align}
The real constants $L_i^\inn,M_i^\inn$ are for the moment free parameters. They will be adjusted later on to match with the outer Poisson field. All profiles will actually depend in an affine way on $L_i^\inn$, but we chose to hide it in the notation except fo $\partial_{x_1}U$ and $\mathcal S_{3,0}$ which will give the key contribution involving $L_i^\inn$. Also, the profiles in $\bar \phi^\inn_{i,\pm}$ and $\bar V_{i,\pm}^\inn$ above depend on $i$ but we do not include it in the notation for sake of clarity. We require
\begin{align} 
\label{eigenfunction:id:S11}
& \Ls_0(  S_{1,1},\mathcal W_{1,1})=-a.\nabla \Lambda U,\\
\label{eigenfunction:id:S21}
& \Ls_0(  S_{2,1},\mathcal W_{2,1})=-(\Lambda_i+2\tilde \lambda_i)\Lambda U-\alpha^{\frac 12}L_{i}^\inn a.\nabla \partial_{x_1}U,\\
 \label{eigenfunction:id:S2} 
& \Ls_0  (\mathcal S_{2,2},\mathcal W_{2,2})= -a.\nabla \mathcal S_{1,1},\\
\label{eigenfunction:id:S30}
& \Ls_0(  S_{3,0},\mathcal W_{3,0})= (\Lambda_i +2\tilde \lambda_i)\partial_{x_1}U+\nabla . (\partial_{x_1}UG_{1}),\\
 \label{eigenfunction:id:S31} 
& \Ls_0  (\mathcal S_{3,1},\mathcal W_{3,1}) =  - \alpha a .\nabla ( T_2^{(i)} + \tilde \lambda_i  \hat T_2)+ \alpha ((\Lambda_i+2\tilde \lambda_i) \mathcal S_{1,1}+\nabla . (\mathcal S_{1,1}G_{1})) \nonumber \\
& \qquad \qquad \qquad -\alpha^{\frac 32}L_i^\inn (\Lambda_i+2\tilde \lambda_i)\partial_{x_1}U, \\
 \label{eigenfunction:id:S40} 
& \Ls_0  (\mathcal S_{4,0},\mathcal W_{4,0})= -\frac{1}{2|a|^4}\Lambda U+\alpha^2L_i^\inn \nabla.(\partial_{x_1}U G_2),\\
 \label{eigenfunction:id:S40'} 
& \Ls_0  (\mathcal S_{4,0}',\mathcal W_{4,0}')=0,\qquad  -\Delta \mathcal W_{4,0}'=\mathcal S_{4,0}',\\
 \label{eigenfunction:id:Poisson-fields-SW}
&-\Delta \mathcal W_{m,n}=\mathcal S_{m,n} \quad \mbox{for }(m,n)\in \{(1,1),(2,1),(2,2),(3,0),(3,1),(4,0)\}.
\end{align}
We recall that $\Ls_0(\partial_{x_1}U,\Phi_{\partial_{x_1}}U)=0$. Plugging \eqref{eigenfunction:id:T2}-\eqref{eigenfunction:id:hatT4} and \eqref{eigenfunction:id:S11}-\eqref{eigenfunction:id:Poisson-fields-SW} in \eqref{eigenfunction:id:barRinn1} and \eqref{eigenfunction:id:tildeRinn} we obtain:
\begin{align}
\nonumber R^\inn_{i, \pm} & =  \tilde R^\inn_{i,\pm}(\tilde \phi^\inn_{i, \pm},\tilde V^\inn_{i,\pm})   \\
\nonumber  & \qquad  - \alpha  \nu^2\Bigg( (\Lambda_i+2\tilde \lambda_i)\left( \bar  \phi_{i, \pm}^\inn \mp \nu (\alpha-\beta)\mathcal S_{1,1}\mp \alpha^{\frac 12}\nu L_i^\inn \partial_{x_1}U\right)\\
\nonumber & \qquad \qquad \qquad +\nabla .\left(\left(\bar  \phi_{i, \pm}^\inn   \mp \nu (\alpha-\beta)\mathcal S_{1,1}\mp \alpha^{\frac 12}\nu L_i^\inn \partial_{x_1}U\right) G_{1,\pm}  \right)\Bigg)\\
\nonumber  & \qquad - \nabla.\left( \Big(\bar  \phi_{i, \pm}^\inn\mp \alpha^{\frac 12}\nu L_i^\inn \partial_{x_1}U\Big) \alpha^{\frac 32}\nu^3 G_{2,\pm}\right) - \nabla.\left(\bar  \phi_{i, \pm}^\inn  \alpha^2 \nu^4G_{3,\pm} \right)\\
\nonumber  &\qquad \pm \nu (\alpha-\beta)a.\nabla ( \phi_{i,\pm}^\inn-\Lambda U-\alpha \nu^2 T^{(i)}_2-\alpha \nu^2 \tilde \lambda_i \hat T_2\mp \nu (\alpha-\beta)\mathcal S_{1,1}\mp \alpha^{\frac 12}\nu L_i^\inn \partial_{x_1}U)\\
\nonumber & \qquad+\nu^2 (\alpha-\beta)\left(\Lambda+\frac{\lambda_i}{\beta}\right)(\phi_{i,\pm}^\inn-\Lambda U\mp \alpha^{\frac 12}\nu L_i^\inn \partial_{x_1}U)\\
\nonumber  &\qquad +\left(U\Big(y\pm \frac{2a}{\nu}\Big)-\frac{\nu^4}{2|a|^4}\right)\Lambda U+U\Big(y\pm \frac{2a}{\nu}\Big)( \phi_{i,\pm}^\inn-\Lambda U)\\
\label{eigenfunctions:id:Rinn-decomposition}&\qquad -\nabla U\Big(y\pm \frac{2a}{\nu}\Big).\nabla V^\inn_{i,\pm}-\nabla .( \phi^\inn_{i,\pm}\tilde G_\pm).
\end{align}

\subsubsection{Elliptic equations in spherical harmonics} To ease the notations, we write
$$r = r_\pm, \quad \theta = \theta_\pm,$$
wherever there is no need to precise the sign. Using the change of coordinates $\pa_{y_1} = \cos \theta \pa_r - r^{-1} \sin \theta \pa_\theta$ and $\pa_{y_2} = \sin \theta \pa_r + r^{-1} \cos \theta \pa_\theta$ yields the elementary identity
\begin{align}
\nabla. \Big[ f(r) \cos (m \theta) \nabla g(r) \cos (k \theta) \Big] & = \Big[\frac{1}{r} \partial_r (r f \partial_r g)  - \frac{k^2 f g}{r^2} \Big] \cos (m\theta) \cos(k \theta) \nonumber\\
 & \quad + (mk) \frac{fg}{r^2} \sin(m\theta) \sin(k \theta). \label{id:coskm}
\end{align}

\paragraph{Elliptic equations for the leading terms.}

Since the source terms in \eqref{eigenfunction:id:T2} - \eqref{eigenfunction:id:hatT4} depends on spherical harmonics, we decompose the profiles $T_2^{(i)}, T_3, T_4^{(i)}, \hat T_2,$ and $\hat T_4^{(i)}$ accordingly,
\begin{align}
 \label{eigenfunction:id:decomposition-T2i} & T_2^{(i)} = T_{2,0}^{(i)}(r) + \mathsf T_{2,2}(r)\cos (2\theta), \\
 \label{eigenfunction:id:decomposition-V2i} & V_2^{(i)} = V_{2,0}^{(i)}(r) + \mathsf V_{2,2}(r)\cos (2\theta), \\
 \label{eigenfunction:id:decomposition-T3}&T_3=  \mathsf T_{3,1}(r) \cos (\theta) + \mathsf T_{3,3}(r) \cos (3\theta),\\ 
 \label{eigenfunction:id:decomposition-V3}& V_3 =  \mathsf V_{3,1}(r)\cos (\theta) +\mathsf V_{3,3}(r)\cos (3\theta),\\
 \label{eigenfunction:id:decomposition-T4i}&T_4^{(i)} = T_{4,0} ^{(i)}(r) + \mathsf T_{4,2}^{(i)}(r)\cos (2\theta) + \mathsf T_{4,4}(r)\cos (4\theta)\\ 
 \label{eigenfunction:id:decomposition-V4i}&V_4^{(i)} = V_{4,0}^{(i)}(r) + \mathsf V_{4,2}^{(i)}(r)\cos (2\theta) + \mathsf V_{4,4}(r)\cos (4\theta),\\
 \label{eigenfunction:id:decomposition-T2}& \hat T_2(y) = \hat T_{2}(r) , \qquad \quad \hat T_4^{(i)} =\hat T_{4,0}^{(i)}(r)  + \hat{\mathsf T}_{4,2}(r)\cos (2\theta),\\ 
 \label{eigenfunction:id:decomposition-V2}& \hat V_2(y) = \hat V_{2}(r), \qquad \quad  \hat V_4^{(i)} = \hat V_{4,0}^{(i)}(r)+ \hat{\mathsf V}_{4,2}(r) \cos (2\theta),
\end{align}
Note that we slightly abuse notations for radial functions by identifying $f(y) \equiv f(r)$. The notation $f_{m,k}(r)$ means that the term of magnitude order $\nu^m$ and angular dependence $\cos(k \theta)$. We also note in the above decomposition of the profiles that the pairs $(\mathsf T_{k,k}, \mathsf V_{k,k})_{k = 2,3,4}$ and $(\hat{\mathsf T}_{4,2}, \hat{\mathsf V}_{4,2})$ are independent of $i$. 

For radial modes $T(y)=T(r)$ with suitable decay at infinity, $\Phi_T$ is well defined. In particular, we have the explicit formula
\begin{equation} \label{eigenfunction:id:extensionnablaPhiradial}
\pa_r \Phi_T=-r^{-1}\int_0^r T(\tilde r)\tilde r d\tilde r.
\end{equation}
We extend $T\mapsto \nabla \Phi_T$ by the above formula to all radial $T\in L^1_{loc}(\mathbb R^2)$, noting that it is well defined (and abusing notations since $\Phi_T=-\frac{1}{2\pi}\ln r*T$ might not be well defined). We extend accordingly
$$
\Ls_0 (T)=\Delta T-\nabla .(T\nabla \Phi_U)-\nabla .(U\nabla \Phi_T),
$$
in the distributional sense to all radial functions $T\in L^1_{loc}(\mathbb R^2)$ by defining $\nabla \Phi_T$ via \eqref{eigenfunction:id:extensionnablaPhiradial}. Whatever the choice of smooth radial functions $T(y)=T(r)$ and $V(y)=V(r)$ with $-\Delta V=T$ we have then the identity:
\begin{align*}
\Ls_0 (T,V)= \Ls_0 (T).
\end{align*}
For higher order spherical harmonics, if $-\Delta (\omega \cos (k\theta))= \psi \cos(k \theta)$, we have
$$
\Ls_0 (\psi \cos(k \theta), \omega \cos (k\theta))= \cos(k\theta) \Ls_{0,k} (\psi , \omega ),
$$
and
$$
\Delta (\omega \cos(k\theta))=  \cos(k\theta) \Delta_{k} \omega,
$$
where
\begin{equation}\label{def:Ls0k}
\Ls_{0,k}(\psi , \omega )= \Delta_k \psi + \dfrac{4r }{1 + r^2}\pr \psi + \dfrac{32r }{(1+r^2)^3}\pr \omega + \dfrac{16\psi}{(1+r^2)^2},
\end{equation}
and 
$$
-\Delta_k \omega = \psi \quad \textup{with}\quad  \Delta_k= \pr^2+ \dfrac{\pr}{r} - \dfrac{k^2}{r^2},
$$
We use here the convention $\Ls_{0,0} \equiv \Ls_0$ and $\Delta_0 \equiv \Delta$. 
From Lemma \ref{lem:taylorpoissonfield} and \eqref{id:coskm}, we decompose the source terms in \eqref{eigenfunction:id:T2} - \eqref{eigenfunction:id:hatT4} into the spherical harmonics,
\begin{equation}\label{def:Sigma22}
\nabla. (\Lambda U G_1)= \Sigma_{2,2}(r) \cos(2\theta), \qquad \Sigma_{2,2}=\frac{1}{2}r\pa_r \Lambda U,
\end{equation}
and
$$
\nabla. (\Lambda U G_2)= \Sigma_{3,1}(r) \cos \theta+ \Sigma_{3,3}(r) \cos(3\theta ) 
$$
where
\begin{equation}\label{def:Sigma3k}
\Sigma_{3,1}(r) = \frac{1}{4\sqrt{2}} \pa_r \Lambda U, \quad \Sigma_{3,3}(r) = -\frac{1}{4\sqrt{2}} r^2 \pa_r \Lambda U.
\end{equation}
We have
\begin{align*}
 \Lambda_i T_2^{(i)} + \nabla.(G_1 T_2^{(i)}) + \nabla.(G_3 \Lambda U)  = \Sigma_{4,0}^{(i)}(r)+ \Sigma_{4,2}^{(i)}(r)\cos(2\theta) + \Sigma_{4,4}(r)\cos(4 \theta),
\end{align*}
where  
\begin{align}
\label{def:Sigma40}\Sigma_{4,0}^{(i)} &=\Lambda_i T_{2,0} ^{(i)} + \frac{1}{4 r}\partial_r (r^2\mathsf T_{2,2}) - \frac{1}{16 r}  \partial_r(r^2 \Lambda U),\\
\label{def:Sigma42}\Sigma_{4,2}^{(i)} &=  \Lambda_i \mathsf T_{2,2} + \frac{1}{2} r\pa_r T_{2,0}^{(i)} - \frac{r}{16}\pa_r \Lambda U  ,\\
\label{def:Sigma44}\Sigma_{4,4} &= \frac{1}{4} r \pa_r \mathsf T_{2,2} - \frac{1}{2}\mathsf T_{2,2}+\frac{1}{16}r^3 \pa_r \Lambda U,
\end{align} 
and
\begin{align*}
2 T_2^{(i)}  +  \nabla.(G_1 \hat T_{2}) +  \Lambda_i \hat T_{2}=\hat \Sigma_{4,0}^{(i)}(r)+\hat \Sigma_{4,2}(r)\cos (2\theta),
\end{align*}
where
\begin{equation}\label{def:Sigma4khat}
 \hat \Sigma_{4,0}^{(i)}=2 T_{2,0}^{(i)}+ \Lambda_i \hat T_{2}, \qquad \hat \Sigma_{4,2}= 2 \mathsf T_{2,2}+\frac{1}{2}r\pa_r \hat T_2.
\end{equation}
Now, the equations \eqref{eigenfunction:id:T2} -\eqref{eigenfunction:id:Poisson-fields}  are transformed into the following systems of radial functions:
\begin{equation}\label{sys:T20T20hat}
\left\{\begin{array}{l l} & \Ls_0(T_{2,0}^{(i)}) = \Lambda_i \Lambda U, \\
& -\Delta V_{2,0}^{(i)}=T_{2,0}^{(i)},
\end{array} \right. \qquad \mbox{and}\qquad
\left\{\begin{array}{l l} & \Ls_0(\hat T_{2}) = 2 \Lambda U, \\
& -\Delta \hat V_{2}^{(i)}=\hat T_{2,0}^{(i)}.
\end{array}\right.
\end{equation}

\begin{equation}\label{sys:T40T40hat}
\Ls_0(T_{4,0}^{(i)}) = \Sigma_{4,0}^{(i)}, \quad \Ls_0(\hat T_{4,0}^{(i)}) = \hat \Sigma_{4,0}^{(i)},
\end{equation}
\begin{align}
\label{sys:Tkk31} 
k = 2,3,4,\quad &\left\{ \begin{array}{l l} \Ls_{0,k} (\mathsf T_{k,k},\mathsf V_{k,k})=\Sigma_{k,k}, \\ -\Delta_k \mathsf V_{k,k}=\mathsf T_{k,k},  \end{array} \right. \qquad \left\{ \begin{array}{l l} 
 \Ls_{0,1} (\mathsf T_{3,1},\mathsf V_{3,1}) =\Sigma_{3,1},\\
 -\Delta_1 \mathsf V_{3,1}=\mathsf T_{3,1},
  \end{array} \right.
  \end{align}
  where $(\Sigma_{k,k})_{k = 2,3,4}$ and $\Sigma_{3,1}$ are defined in \eqref{def:Sigma22}, \eqref{def:Sigma3k}, \eqref{def:Sigma44}, and 
  \begin{align}
 \label{sys:T42T42hat}
&\left\{ \begin{array}{l l} 
 \Ls_{0,2} (\mathsf T_{4,2}^{(i)},\mathsf V_{4,2}^{(i)}) =\Sigma_{4,2}^{(i)},\\
 -\Delta_2 \mathsf V_{4,2}^{(i)}=\mathsf T_{4,2}^{(i)},
  \end{array} \right. \qquad \left\{ \begin{array}{l l} 
 \Ls_{0,2} (\hat{\mathsf T}_{4,2},\hat{\mathsf V}_{4,2}) =\hat \Sigma_{4,2},\\
 -\Delta_2 \hat{\mathsf V}_{4,2}=\hat{\mathsf T}_{4,2},
  \end{array} \right. 
\end{align}
where $\Sigma_{4,2}^{(i)}$ and $\hat \Sigma_{4,2}$ are defined in \eqref{def:Sigma42} and \eqref{def:Sigma4khat}.\\

\paragraph{Elliptic equations for the subleading terms.} We next decompose the profiles $\mathcal S_{1,1}$, $\mathcal S_{2,1}$, $\mathcal S_{2,2}$, $\mathcal S_{3,0}$, $\mathcal S_{3,1}$, $\mathcal S_{4,0}$ and $\mathcal S_{4,0}'$ and their Poisson fields in spherical harmonics
\begin{align}
 \label{eigenfunction:id:decomposition-mS11} & \mathcal S_{1,1} =  \mathsf S_{1,1,1}(r)\cos (\theta), \qquad \mathcal W_{1,1} =  \mathsf W_{1,1,1}(r)\cos (\theta)\\
 \label{eigenfunction:id:decomposition-mS21}  & \mathcal S_{2,1} =  \mathsf S_{2,1,0}(r)+  \mathsf S_{2,1,2}(r)\cos (2\theta), \qquad \mathcal W_{2,1} =  \mathsf W_{2,1,0}(r)+\mathsf W_{2,1,2}(r)\cos(2\theta)\\
     \label{eigenfunction:id:decomposition-mS22}  & \mathcal S_{2,2} =  \mathsf S_{2,2,0}(r)+ \mathsf S_{2,2,2}(r)\cos (2 \theta), \qquad \mathcal W_{2,2} =  \mathsf W_{2,2,0}(r)+ \mathsf W_{2,2,2}(r)\cos (2\theta)\\
   \label{eigenfunction:id:decomposition-mS30}  & \mathcal S_{3,0} =  \mathsf S_{3,0,1}(r)\cos \theta+ \mathsf S_{3,0,3}(r)\cos (3 \theta), \qquad \mathcal W_{3,1} =  \mathsf W_{3,0,1}(r)\cos \theta+ \mathsf W_{3,0,3}(r)\cos (3\theta)\\
  \label{eigenfunction:id:decomposition-mS31}  & \mathcal S_{3,1} =  \mathsf S_{3,1,1}(r)\cos \theta+ \mathsf S_{3,1,3}(r)\cos (3 \theta), \qquad \mathcal W_{3,1} =  \mathsf W_{3,1,1}(r)\cos \theta+ \mathsf W_{3,1,3}(r)\cos (3\theta)\\
  \label{eigenfunction:id:decomposition-mS40}  & \mathcal S_{4,0} = \sum_{j=0,2,4} \mathsf S_{4,0,j}(r)\cos(j\theta), \qquad \mathcal W_{4,0} = \sum_{j=0,2,4} \mathsf W_{4,0,j}(r)\cos(j\theta),\\
    \label{eigenfunction:id:decomposition-mS40'}  & \mathcal S_{4,0}' =  \mathsf S_{4,0,2}'(r)\cos(2\theta), \qquad \mathcal W_{4,0,2} = \mathsf W_{4,0,2}(r)\cos(2\theta).
\end{align}
We decompose the source terms in \eqref{eigenfunction:id:S11}-\eqref{eigenfunction:id:Poisson-fields-SW} as
\begin{align*}
&  -a.\nabla \Lambda U =-|a|\partial_r \Lambda U \cos(\theta), \\
&  -(\Lambda_i+2\tilde \lambda_i)\Lambda U-\alpha^{\frac 12}L_i^\inn a.\nabla \partial_{x_1}U =-(\Lambda_i+2\tilde \lambda_i)\Lambda U(r)-\frac{\alpha^{\frac 12}}{2}L_i^\inn |a| \left(\partial_r^2 U+\frac{\partial_r U}{r}\right) , \\
& \qquad \qquad \qquad \qquad \qquad \qquad \qquad \qquad -\frac{\alpha^{\frac 12}}{2}L_i^\inn |a| \left(\partial_r^2 U-\frac{\partial_r U}{r}\right)\cos(2\theta)\\
& -a.\nabla \mathcal S_{1,1}  =-\frac{|a|}{2}\left(\partial_r \mathsf S_{1,1,1}+\frac{\mathsf S_{1,1,1}}{r}\right)+\frac{|a|}{2}\left(\frac{\mathsf S_{1,1,1}}{r}-\partial_r \mathsf S_{1,1,1}\right)\cos(2\theta).\\
& (\Lambda_i+2\tilde \lambda_i)\partial_{x_1}U+\nabla.(\partial_{x_1}UG_1) = \left(\frac 54 r \partial_r^2 U+\frac{17-8i-8\tilde \lambda_i}{4}\partial_r U\right)\cos \theta  +\left(\frac r4 \partial_r^2 U-\frac 14 \partial_r U\right) \cos(3\theta),
\end{align*}

\begin{align*}
& - \alpha a .\nabla (T_2^{(i)}+\tilde \lambda_i \hat T_2)+ \alpha ((\Lambda_i+\tilde \lambda_i) \mathcal S_{1,1}+\nabla . (\mathcal S_{1,1}G_{1}))-\alpha^{\frac 32}L_i^\inn (\Lambda_i+2\tilde \lambda_i)\partial_{x_1}U  \\
& \qquad =  \alpha\Big(  2r\partial_r \mathsf S_{1,1,1}+(2(2-i)+2\tilde \lambda_i)\mathsf S_{1,1,1}- \partial_r \mathsf T_{2,0}^{(i)}-\tilde \lambda_i \partial_r \hat T_{2} \\
& \qquad \qquad \qquad \qquad -\frac{\partial_r \mathsf T_{2,2}}{2}-\frac{\mathsf T_{2,2}}{r}-\alpha^{\frac 12}L_i^\inn (\Lambda_i+2\tilde \lambda_i)\partial_r U \Big)\cos \theta\\
&\qquad \qquad \qquad \qquad +\alpha (  r\partial_r \mathsf S_{1,1,1}-\mathsf{S}_{1,1,1}+|a|(\frac{\mathsf T_{2,2}}{r})-\frac{\partial_r \mathsf T_{2,2}}{2})\cos(3\theta),
\end{align*}
and 
\begin{align*}
&-\frac{1}{2|a|^4}\Lambda U+\alpha^2 L_i^\inn \nabla .(\partial_{x_1}UG_2)  =-\frac{1}{2|a|^4}\Lambda U+\frac{\alpha^2 L_i^\inn}{8\sqrt{2}}\left(\partial_r^2 U+\frac{1}{r}\partial_r U\right)\\
& \qquad +\left(-\frac{3}{24\sqrt{2}}(r^2\partial_r^2 U+r\partial_r U)+\frac 12 \partial_r^2 U-\frac{1}{2r}\partial_r U\right)\cos (2\theta) -\frac{3}{24\sqrt{2}}(r^2\partial_r^2 U-r\partial_r U)\cos(4\theta)
\end{align*}
The equations \eqref{eigenfunction:id:Poisson-fields}-\eqref{eigenfunction:id:S11} are transformed into
\begin{align}\label{sys:S111}
& \left\{\begin{array}{l l} & \Ls_{0,1}(\mathsf S_{1,1,1},\mathsf W_{1,1,1}) =-|a|\partial_r \Lambda U, \\
& -\Delta_1 \mathsf W_{1,1,1} =\mathsf S_{1,1,1},
\end{array} \right. \\
\label{sys:S210}
& \left\{\begin{array}{l l} & \Ls_{0}(\mathsf S_{2,1,0}) = \Sigma_{2,1,0}, \\
& -\Delta \mathsf W_{2,1,0} =\mathsf S_{2,1,0},
\end{array} \right. \quad \Sigma_{2,1,0}=-(\Lambda_i+2\tilde \lambda_i)\Lambda U(r)-\frac{|a|}{2}\alpha^{\frac 12} L_i^\inn (\partial_r^2 U+\frac{\partial_r U}{r})\\
\label{sys:S212}
& \left\{\begin{array}{l l} & \Ls_{0}(\mathsf S_{2,1,2}) = \Sigma_{2,1,2}, \\
& -\Delta_2 \mathsf W_{2,1,2} =\mathsf S_{2,1,2},
\end{array} \right. \quad \Sigma_{2,1,0}=-\frac{\alpha^{\frac 12}}{2}L_i^\inn |a| (\partial_r^2 U-\frac{\partial_r U}{r})\\
\label{sys:S220}
& \left\{\begin{array}{l l} & \Ls_{0}(\mathsf S_{2,2,0}) = \Sigma_{2,2,0}, \\
& -\Delta \mathsf W_{2,2,0} =\mathsf S_{2,2,0},
\end{array} \right. \qquad \Sigma_{2,2,0}=-\frac{|a|}{2}(\partial_r \mathsf S_{1,1,1}+\frac{\mathsf S_{1,1,1}}{r})\\
\label{sys:S222}
& \left\{\begin{array}{l l} & \Ls_{0,2}(\mathsf S_{2,2,2}) = \Sigma_{2,2,2}, \\
& -\Delta_2 \mathsf W_{2,2,2} =\mathsf S_{2,2,2},
\end{array} \right. \qquad \Sigma_{2,2,2}=\frac{|a|}{2}(\frac{\mathsf S_{1,1,1}}{r}-\partial_r \mathsf S_{1,1,1})\\
\label{sys:S301}
& \left\{\begin{array}{l l} & \Ls_{0,1}(\mathsf S_{3,0,1},\mathsf W_{3,0,1}) = \Sigma_{3,0,1}, \\
& -\Delta_1 \mathsf W_{3,0,1} =\mathsf S_{3,0,1},
\end{array} \right. \qquad \Sigma_{3,0,1}=\frac 54 r \partial_r^2 U+\frac{17-8i-8\tilde \lambda_i}{4}\partial_r U, \\
\label{sys:S303}
& \left\{\begin{array}{l l} & \Ls_{0,1}(\mathsf S_{3,0,3},\mathsf W_{3,0,3}) = \Sigma_{3,0,3}, \\
& -\Delta_3 \mathsf W_{3,0,3} =\mathsf S_{3,0,3},
\end{array} \right. \qquad \Sigma_{3,0,3}=\frac r4 \partial_r^2 U-\frac 14 \partial_r U, \\
\label{sys:S311}
& \left\{\begin{array}{l l} & \Ls_{0,1}(\mathsf S_{3,1,1},\mathsf W_{3,1,1}) = \Sigma_{3,1,1}, \\
& -\Delta_1 \mathsf W_{3,1,1} =\mathsf S_{3,1,1},
\end{array} \right.  \\
\label{sys:S313}
& \left\{\begin{array}{l l} & \Ls_{0,3}(\mathsf S_{3,1,3},\mathsf W_{3,1,3}) = \Sigma_{3,1,3}, \\
& -\Delta_3 \mathsf W_{3,1,3} =\mathsf S_{3,1,3},
\end{array} \right. \quad \Sigma_{3,1,3}= \alpha (r\partial_r \mathsf S_{1,1,1}-\mathsf{S}_{1,1,1}+|a|(\frac{\mathsf T_{2,2}}{r}-\frac{\partial_r \mathsf T_{2,2}}{2})) \\
\label{sys:S400}
& \left\{\begin{array}{l l} & \Ls_{0}(\mathsf S_{4,0,0}) =\Sigma_{4,0,0}, \\
& -\Delta \mathsf W_{4,0,0} =\mathsf S_{4,0,0},
\end{array} \right. \qquad \Sigma_{4,0,0}=-\frac{1}{2|a|^4}\Lambda U+\frac{\alpha^2 L_i^\inn}{8\sqrt{2}}\left(\partial_r^2 U+\frac{1}{r}\partial_r U\right)\\
\label{sys:S402}
& \left\{\begin{array}{l l} & \Ls_{0}(\mathsf S_{4,0,2}) =\Sigma_{4,0,2}, \\
& -\Delta_2 \mathsf W_{4,0,2} =\mathsf S_{4,0,2},
\end{array} \right. \qquad \Sigma_{4,0,2}=-\frac{3}{24\sqrt{2}}(r^2\partial_r^2 U+r\partial_r U)+\frac 12 \partial_r^2 U-\frac{1}{2r}\partial_r U\\
\label{sys:S404}
& \left\{\begin{array}{l l} & \Ls_{0}(\mathsf S_{4,0,4}) =\Sigma_{4,0,4}, \\
& -\Delta_4 \mathsf W_{4,0,4} =\mathsf S_{4,0,4},
\end{array} \right. \qquad \Sigma_{4,0,4}=-\frac{3}{24\sqrt{2}}(r^2\partial_r^2 U-r\partial_r U)\\
\label{sys:S402'}
& \left\{\begin{array}{l l} & \Ls_{0}(\mathsf S_{4,0,2}') =0, \\
& -\Delta_2 \mathsf W_{4,0,2}' =\mathsf S_{4,0,2}',
\end{array} \right.
\end{align}
where
$$
\Sigma_{3,1,1}= \alpha\left(  2r\partial_r \mathsf S_{1,1,1}+2(2-i)\mathsf S_{1,1,1}- \partial_r \mathsf T_{2,0}^{(i)}-\tilde \lambda_i \partial_r \hat T_{2}-\frac{\partial_r \mathsf T_{2,2}}{2}-\frac{\mathsf T_{2,2}}{r}-\alpha^{\frac 12}L_i^\inn (\Lambda_i+2\tilde \lambda_i)\partial_r U \right).
$$
 
\subsubsection{The inner approximate eigenfunction}

\noindent The following proposition provides an approximate expansion of the eigenfunction $\phi_i$ in the inner zone $|y_\pm| \ll \nu^{-1}$. 

\begin{proposition}[Inner expansion of $\phi_i$]\label{prop:phi1inn} There exists $\epsilon>0$ such that the following holds true. Let $i = 0,1$ and pick any constants $\tilde \lambda_i ,L_i^\inn,M_i^\inn\in \mathbb R$ with $|\tilde \lambda_i|\lesssim |\ln \nu|^{-1}$ and $|\tilde L_i^\inn|+|\tilde M_i^\inn|\lesssim |\ln \nu|$. There exists a smooth solution $(\phi^\inn_{i,\pm},V^\inn_{i,\pm},R^\inn_{i,\pm})$ of \eqref{eigen:id:Riinn1} with the following properties.

\smallskip

\noindent (i) \emph{Decomposition}. There holds
\begin{align} \label{def:phii_inn}
\phi^\inn_{i,\pm} &= \Lambda U + \alpha \nu^2 T_{2}^{(i)} \pm \alpha^{\frac 32}\nu^3 T_{3} +\alpha^2\nu^4  T_{4}^{(i)}  + \alpha \tilde \lambda_i \nu^2 \hat T_2  +\tilde \lambda_i \alpha^2 \nu^4 \hat T_{4}^{(i)}\\
\nonumber &\qquad \pm \alpha^{\frac 12} \nu L_i^\inn \partial_{x_1}U\pm  \nu (\alpha-\beta)\mathcal S_{1,1}+ \nu^2 (\alpha-\beta)\mathcal S_{2,1}+\nu^2(\alpha-\beta)^2\mathcal S_{2,2} \pm \alpha^{\frac 32}\nu^3L_i^\inn \mathcal S_{3,0}\\
\nonumber & \qquad  \pm \nu^3 (\alpha-\beta)\mathcal S_{3,1} +\nu^4\mathcal S_{4,0} +\alpha^2 \nu^4M_i^\inn \mathcal S_{4,0}',\\
\label{def:Vi_inn} V^\inn_{i,\pm}&= \Phi_{\Lambda U} + \alpha \nu^2 V_2^{(i)} \pm \alpha^{\frac 32}\nu^3 V_{3} +\alpha^2 \nu^4  V_4^{(i)}  + \alpha \tilde \lambda_i \nu^2 \hat V_2 + \tilde \lambda_i\alpha^2 \nu^4 \hat V_4^{(i)}\\
\nonumber &\qquad \pm \alpha^{\frac 12} \nu L_i^\inn \partial_{x_1}\Phi_{U}\pm  \nu (\alpha-\beta)\mathcal W_{1,1}+ \nu^2 (\alpha-\beta)\mathcal W_{2,1}+\nu^2(\alpha-\beta)^2\mathcal W_{2,2} \pm \alpha^{\frac 32}\nu^3L_i^\inn \mathcal W_{3,0}\\
\nonumber & \qquad  \pm \nu^3 (\alpha-\beta)\mathcal W_{3,1} +\nu^4\mathcal W_{4,0} +\alpha^2 \nu^4M_i^\inn \mathcal W_{4,0}',
\end{align}
where the functions above are solutions of \eqref{eigenfunction:id:T2}-\eqref{eigenfunction:id:Poisson-fields} and \eqref{eigenfunction:id:S21}-\eqref{eigenfunction:id:Poisson-fields-SW} given by the formulas \eqref{eigenfunction:id:decomposition-T2i}-\eqref{eigenfunction:id:decomposition-T2} and \eqref{eigenfunction:id:decomposition-mS11}-\eqref{eigenfunction:id:decomposition-mS40'} and where the functions in these formulas are given by Lemmas \ref{lem:T20hatT2}, \ref{lem:T22}, \ref{lem:T3k},  \ref{lem:T3k},  \ref{lem:T40hatT40}, \ref{lem:T42hatT42} and \ref{lem:T44} and Proposition \ref{pr:solution-system-SW}.

\smallskip

\noindent (ii) \emph{Expansion}. We have,
\begin{align} \label{expansion:phii_inn}
- \frac{1}{16\nu^4} \phi^\inn_{i,\pm}(y_\pm) & = \frac{1}{\zeta_\pm^4} +\frac{ \alpha}{\zeta_\pm^2} \left(\frac{i }{2}+\frac{1}{4}\cos 2\theta_\pm \right) \mp \alpha^{\frac 32} \frac{1}{12\sqrt{2}\zeta_\pm}\cos(3\theta_\pm)\\
\nonumber &  +\alpha^2 \left(\frac{i}{4}-\frac{7}{64}+\frac{2i-1}{8}\cos(2\theta_\pm)+\frac{1}{32}\cos(4\theta_\pm)\right) \\
\nonumber & +  \tilde \lambda_i  \left(- \frac{\alpha}{2\zeta_\pm^2} +\alpha^2 \left( (\frac i2-\frac 14)( \ln \zeta_\pm-\ln \nu )+\frac{1-10i}{16}-\frac 14\cos(2\theta_\pm)\right)\right)+\phi_{i,\pm,\sharp}^{\inn ,z }(z_\pm) ,
\end{align} 
where for $\nu\lesssim |z_{\pm}|=\zeta_\pm\lesssim 1$:
\begin{equation}  \label{expansion:phii_inn-estimate}
|\nabla^k \phi_{i,\pm,\sharp}^{\inn ,z }(z_\pm) |\lesssim \nu^{\epsilon} (\nu+\zeta_\pm)^{-4-k-\epsilon}
\end{equation}
and
\begin{align}
& \nonumber \frac{-1}{16\nu^2} V^\inn_{i,\pm}(y_\pm)\\
\nonumber &=  \frac{-1}{4\zeta^2_\pm}\pm \frac{\alpha^{\frac 12}L_i^\inn}{4}\frac{1}{\zeta_\pm}\cos(\theta)\\
\nonumber &  +\alpha \Bigg( \frac{\tilde \lambda_i-i}{4} \ln^2 \zeta_\pm +\left(\frac{i-\tilde \lambda_i}{2}\ln \nu+\frac{2i+1-2\tilde \lambda_i)}{4}\right)\ln \zeta_\pm  +\frac{\tilde \lambda_i-i}{4}\ln^2\nu+\frac{2\tilde \lambda_i-1-2i}{4}\ln \nu+\frac{1}{16}\cos (2\theta)\Bigg)\\
\nonumber &\mp  \alpha^{\frac 32}\zeta_\pm \left( L_i^\inn \frac{1-4i+4\tilde \lambda_i}{32}\cos \theta+ \frac{1}{96\sqrt{2}}\cos(3\theta_\pm) \right)\\
\nonumber &+\alpha^2\zeta^2_\pm \Bigg( \frac{1-2i}{16}\tilde \lambda_i \ln \zeta_\pm +\frac{2i-1}{16} \tilde \lambda_i\ln \nu+\frac{2\tilde \lambda_i+1-2i}{32} \ln \zeta_\pm \cos(2\theta_\pm)+\frac{2i-1- \tilde \lambda_i}{32} \ln \nu \cos(2\theta) \\
\nonumber & \qquad \qquad  +\frac{1}{16}\left(\frac{7}{16}-i+\tilde \lambda_i(\frac92 i-\frac54)\right)+\frac{5-10i+10\tilde \lambda_i-8M_i^\inn}{128}\cos(2\theta_\pm)+\frac{1}{384}\cos(4\theta_\pm) \Bigg)\\
\label{expansion:Vi_inn} & \qquad +V^{\inn,z}_{i,\pm,\sharp}(z_\pm)
\end{align}
where for $\nu\lesssim |z_{\pm}|=\zeta_\pm\lesssim 1$ for $i=0,1$:
\begin{equation}  \label{expansion:Vi_inn-estimate}
|\nabla^k V^{\inn,z}_{1,\pm,\sharp}(z_\pm) |\lesssim \nu^{\epsilon} (\nu+\zeta_\pm)^{-2-k-\epsilon}
\end{equation}
Moreover, assuming $|\nu \partial_\nu \tilde \lambda_i|+|\nu \partial_\nu L_i^\inn|+|\nu \partial_\nu M_i^\inn|\lesssim 1$, we have
\begin{align} \label{expansion:phii_inn-estimate-partialnu-partiala-1}
&|\nabla^k \nu \partial_\nu \phi^\inn_{i,\pm}|\lesssim \langle y \rangle^{-4-k}, \qquad |\nabla^k \partial_a \phi^\inn_{i,\pm}|\lesssim \nu \langle y \rangle^{-2-k},\\
\label{expansion:phii_inn-estimate-partialnu-partiala-2}&|\nabla^k \nu \partial_\nu (\phi^\inn_{i,\pm}-\Lambda U\mp \alpha^{\frac 12} \nu L_i^\inn \partial_{x_1}U)|\lesssim \nu \langle y \rangle^{-2-k}, \quad |\nabla^k \partial_a (\phi^\inn_{i,\pm}-\Lambda U\mp \alpha^{\frac 12} \nu L_i^\inn \partial_{x_1}U)|\lesssim \langle y \rangle^{-2-k}.
\end{align}
\smallskip

\noindent (iii) \emph{Estimate for the error}. We have for $ r_\pm  \ll \nu^{-1}$,
\begin{align} \label{estimate:Ri_inn-pointwise}
|\nabla^k R^\inn_{i,\pm}(y)|\lesssim \frac{\nu^4}{|\ln \nu|^2}\langle y \rangle^{-2}+\nu^5\langle y \rangle^{-1-k}
\end{align}

\end{proposition}

\begin{proof}

We take $ \phi^\inn_{i,\pm}$ and $V^\inn_{i,\pm}$ as given by \eqref{def:phii_inn} and \eqref{def:Vi_inn} with the functions involved in the formulas given by Lemmas \ref{lem:T20hatT2}, \ref{lem:T22}, \ref{lem:T3k},  \ref{lem:T3k},  \ref{lem:T40hatT40}, \ref{lem:T42hatT42} and \ref{lem:T44} and Proposition \ref{pr:solution-system-SW}. We then let $R_{i,\pm}^\inn$ be given by the first equation in \eqref{eigen:id:Riinn1}. Then $(\phi^\inn_{i,\pm},V^\inn_{i,\pm},R^\inn_{i,\pm})$ is indeed a smooth solution of \eqref{eigen:id:Riinn1}.

\medskip

\begin{align} 
\label{eigenfunction:id:asymptotichatV40infty}
 \hat V_{4,0}^{(i)} &= (2i-1) r^2\ln r + \Big(\frac 54-\frac{9}{2}i\Big)r^2+ \Oc_{r\to \infty}(\ln^3 r),
\end{align}

\noindent \textbf{Step 1}. \emph{Proof of the expansion \eqref{expansion:phii_inn} for $\phi_{i,\pm}^\inn$.} Let $1<r= |y| \lesssim \nu^{-1}$. We define 
$$
\epsilon = \min\{1,\sqrt{8}-2,\sqrt{13}-3,\sqrt{5}-2,\sqrt{20}-4\}>0.
$$
We estimate all terms in \eqref{def:phii_inn}. For the first one,
\begin{equation} \label{eigenfunctions:id:expansion-LambdaU}
\Lambda U(y)=-\frac{16}{r^4}+\Oc(r^{-4-\epsilon}).
\end{equation}
For the second, by \eqref{eigenfunction:id:decomposition-T2i} and then \eqref{eigenfunction:id:asymptoticT20infty0}, \eqref{eigenfunction:id:asymptoticT20infty} and \eqref{eigenfunction:id:asymptoticT22infty},
\begin{equation} \label{eigenfunctions:id:expansion-T2i}
 T_2^{(i)}(y) = T_{2,0}^{(i)}(r) + \mathsf T_{2,2}(r)\cos (2\theta)=-\frac{8i }{r^2}-\frac{4}{r^2}\cos 2\theta +\Oc(r^{-2-\epsilon}).
\end{equation}
 For the third, by \eqref{eigenfunction:id:decomposition-T3} and then \eqref{eigenfunction:id:asymptoticT310}  and \eqref{eigenfunction:id:asymptoticT33infty},
\begin{equation} \label{eigenfunctions:id:expansion-T3}
T_3(y)=  \mathsf T_{3,1}(r) \cos (\theta) + \mathsf T_{3,3}(r) \cos (3\theta)=\frac{4}{3\sqrt{2}r}\cos(3\theta)+\Oc(r^{-1-\epsilon}).
\end{equation}
For the fourth, by \eqref{eigenfunction:id:decomposition-T4i} and then and \eqref{eigenfunction:id:asymptoticT40infty}, \eqref{eigenfunction:id:asymptoticT42iinfty} and \eqref{eigenfunction:id:asymptoticT44infty},
\begin{equation} \label{eigenfunctions:id:expansion-T4i}
T_4^{(i)}(y) = T_{4,0} ^{(i)}(r) + \mathsf T_{4,2}^{(i)}(r)\cos (2\theta) + \mathsf T_{4,4}(r)\cos (4\theta) =\frac74-4i+(2-4i)\cos(2\theta) -\frac 12\cos(4\theta)+\Oc(r^{-\epsilon})
\end{equation}
For the fifth, by \eqref{eigenfunction:id:asymptotichatT2infty}:
\begin{equation} \label{eigenfunctions:id:expansion-hatT2}
\hat T_{2}(y)=\frac{8}{r^2}+\Oc(r^{-2-\epsilon})
\end{equation}
For the sixth, by \eqref{eigenfunction:id:decomposition-T2} and then \eqref{eigenfunction:id:asymptotichatT40infty} and \eqref{eigenfunction:id:asymptotichatT42iinfty}:
\begin{equation} \label{eigenfunctions:id:expansion-hatT4}
\hat T_4^{(i)} =\hat T_{4,0}^{(i)}(r)  + \hat{\mathsf T}_{4,2}(r)\cos (2\theta)= (4-8i) \ln r +10i-1+4\cos(2\theta)+\Oc(r^{-\epsilon})
\end{equation}
Injecting \eqref{eigenfunctions:id:expansion-LambdaU}-\eqref{eigenfunctions:id:expansion-hatT4} in \eqref{eigenfunctions:id:tildephiinn} we obtain
\begin{align} \label{eigenfunctions:id:expansion-tildephiiinn}
\tilde \phi^\inn_{i,\pm}(y) & =-\frac{16}{r^4} + \alpha \nu^2\left(-\frac{8i }{r^2}-\frac{4}{r^2}\cos 2\theta \right) \pm \alpha^{\frac 32}\nu^3 \frac{4}{3\sqrt{2}r}\cos(3\theta)\\
\nonumber & \qquad  +\alpha^2\nu^4  \left(\frac74-4i+(2-4i)\cos(2\theta) -\frac 12 \cos(4\theta)\right) + \alpha \tilde \lambda_i \nu^2\frac{8}{r^2}\\
\nonumber &\qquad +\tilde \lambda_i \alpha^2 \nu^4 \left( (4-8i)\ln r +10i-1+4\cos(2\theta)\right)+\Oc(r^{-4-\epsilon})
\end{align}
where we used that for $1<r\lesssim \nu^{-1}$ we have $\nu^k r^{k-4-\epsilon}\lesssim r^{-4-\epsilon}$ for $k=0,1,2,3,4$.

Next, injecting \eqref{eigenfunction:id:decomposition-mS11}-\eqref{eigenfunction:id:decomposition-mS40'}, and \eqref{bd:estimate-S301-W301}-\eqref{bd:estimate-S404-W404} in \eqref{eigenfunctions:id:barphiinn} and using $|\tilde L_i^\inn|,|\tilde M_i^\inn|\lesssim |\ln \nu|$ we get
\begin{align} 
\nonumber \bar \phi^\inn_{i,\pm} (y) &=\pm  \nu \Oc(|\ln \nu|r^{-5})\pm  \nu (\alpha-\beta)\Oc(r^{-3})+ \nu^2 (\alpha-\beta)\Oc(|\ln \nu|r^{-2})+ \nu^2 (\alpha-\beta)^2\Oc(r^{-2})\\
&\qquad \pm \nu^3 \Oc(|\ln \nu|r^{-3})\pm \nu^3 (\alpha-\beta)\Oc(|\ln \nu|r^{-1})+\nu^4\Oc(|\ln \nu|r^{-2})+\nu^4\Oc(|\ln \nu|r^{-2})\\
\label{eigenfunctions:id:expansion-barphiinn} &=\Oc(\nu |\ln \nu|r^{-4})
\end{align}
for $1<r\lesssim \nu^{-1}$, where we used the assumption $|a-a_\infty|\lesssim \nu$ so that $|\alpha-\beta|\lesssim \nu$. Injecting \eqref{eigenfunctions:id:expansion-tildephiiinn} and \eqref{eigenfunctions:id:expansion-barphiinn} in \eqref{eigenfunctions:id:decomposition-phiinn}, and using $r_\pm=|y_\pm|=\nu^{-1}\zeta_\pm$ we get:
\begin{align*}
\nu^{-4} \phi^\inn_{i,\pm}(y_\pm) & =-\frac{16}{\zeta_\pm^4} + \alpha \left(-\frac{8i }{\zeta_\pm^2}-\frac{4}{\zeta_\pm^2}\cos 2\theta_\pm \right) \pm \alpha^{\frac 32} \frac{4}{3\sqrt{2}\zeta_\pm}\cos(3\theta_\pm)\\
\nonumber & \qquad  +\alpha^2 \left(\frac74-4i+(2-4i)\cos(2\theta_\pm)-\frac 12\cos(4\theta_\pm)\right) + \alpha \tilde \lambda_i \frac{8}{\zeta_\pm^2}\\
\nonumber &\qquad +\tilde \lambda_i \alpha^2  \left((4-8i) \ln \zeta_\pm-\ln \nu +10i-1+4\cos(2\theta_\pm)\right)+\Oc(\nu^\epsilon \zeta_\pm^{-4-\epsilon}).
\end{align*}
The above identity implies the desired identity \eqref{expansion:phii_inn} with the estimate \eqref{expansion:phii_inn-estimate} for $k=0$. This estimate propagates for higher order derivatives since it does for all the $\Oc()$'s in the computations above. The estimates \eqref{expansion:phii_inn-estimate-partialnu-partiala-1} and \eqref{expansion:phii_inn-estimate-partialnu-partiala-1} are proved similarly, suffice it to use the explicit polynomial dependance in $\alpha$ and $\nu$ in the decomposition \eqref{def:phii_inn}.

\medskip

\noindent \textbf{Step 2}. \emph{Proof of the expansion \eqref{expansion:Vi_inn} for $V_{i,\pm}^\inn$.} We estimate similarly to Step 1 all terms in \eqref{def:Vi_inn} for $1<|y|=r\lesssim \nu^{-1}$. For the first one,
\begin{equation} \label{eigenfunctions:id:expansion-PhiLambdaU}
\Phi_{\Lambda U}(y)=\frac{4}{r^2}+\Oc(r^{-2-\epsilon}).
\end{equation}
For the second, by \eqref{eigenfunction:id:decomposition-V2i} and then \eqref{eigenfunction:id:asymptoticT22infty}, \eqref{eigenfunction:id:asymptoticT20infty} and \eqref{eigenfunction:id:asymptotichatT2infty}:
\begin{equation} \label{eigenfunctions:id:expansion-V2i}
 V_2^{(i)} (y) =V_{2,0}^{(i)}(r) + \mathsf V_{2,2}(r)\cos (2\theta)= 4i \ln^2 r -(4+8i) \ln r-\cos 2\theta +  \Oc_{r \to \infty}.
\end{equation}
For the third, by \eqref{eigenfunction:id:decomposition-V3} and then \eqref{eigenfunction:id:asymptoticT31infty} and \eqref{eigenfunction:id:asymptoticT33infty}:
\begin{align} 
V_3(y) =  \mathsf V_{3,1}(r) \cos (\theta) + \mathsf V_{3,3}(r) \cos (3\theta)& = \Oc(\frac{\ln r}{r})+\frac{1}{6\sqrt{2}}r\cos(3\theta)+\Oc(r^{4-\sqrt{13}})\nonumber \\
& =\frac{1}{6\sqrt{2}}r\cos(3\theta)+\Oc(r^{1-\epsilon}).\label{eigenfunctions:id:expansion-V3}
\end{align}
For the fourth, by \eqref{eigenfunction:id:decomposition-V4i} and then \eqref{eigenfunction:id:asymptoticT40infty}, \eqref{eigenfunction:id:asymptoticT42iinfty} and \eqref{eigenfunction:id:asymptoticV44infty}
\begin{align} \label{eigenfunctions:id:expansion-V4i}
V_4^{(i)}(y) & = V_{4,0}^{(i)}(r) + \mathsf V_{4,2}^{(i)}(r)\cos (2\theta) + \mathsf V_{4,4}(r)\cos (4\theta) \\
\nonumber &=  \left(i-\frac{7}{16} \right)r^2+\left( (i-\frac 12)r^2\ln r +(\frac{5i}{4}-\frac 58)r^2\right)\cos (2\theta)-\frac{r^2}{24}\cos(4\theta)+\Oc_{r\to \infty}(r^{2-\epsilon})
\end{align}
For the fifth, by \eqref{eigenfunction:id:asymptotichatT2infty}
\begin{equation} \label{eigenfunctions:id:expansion-hatV2}
\hat V_{2}(y)= -4 \ln^2 r+ 8 \ln r+ \Oc_{r \to \infty}\Big( \frac{\ln^2 r }{r^2} \Big)
\end{equation}
For the sixth, by \eqref{eigenfunction:id:decomposition-V2} and then \eqref{eigenfunction:id:asymptotichatV40infty} and \eqref{eigenfunction:id:asymptotichatT42iinfty}
\begin{align} \label{eigenfunctions:id:expansion-hatV4}
 \hat V_4^{(i)} & = \hat V_{4,0}^{(i)}(r)+ \hat{\mathsf V}_{4,2}(r) \cos (2\theta) \\
\nonumber & =  (2i-1)r^2\ln r+\left(\frac{5}{4}-\frac{9}{2}i\right)r^2-(r^2\ln r+\frac 54 r^2)\cos(2\theta)+\Oc_{r\to \infty}(r^{2-\epsilon})
\end{align}
Injecting \eqref{eigenfunctions:id:expansion-PhiLambdaU}-\eqref{eigenfunctions:id:expansion-hatV4} in \eqref{eigenfunctions:id:tildeViinn} we obtain for $1<|y|=r\lesssim \nu^{-1}$:
\begin{align} \nonumber
\nonumber \tilde V^\inn_{i,\pm}(y)&=\frac{4}{r^2} +\alpha\nu^2\left(4i \ln^2 r -(4+8i) \ln r-\cos 2\theta\right) \pm \frac{\alpha^{\frac 32}\nu^3}{6\sqrt{2}} r\cos(3\theta)\\
\nonumber &+\alpha^2 \nu^4 \left(   \left(i-\frac{7}{16} \right)r^2+\left( (i-\frac 12)r^2\ln r +(\frac{5i}{4}-\frac 58)r^2\right)\cos (2\theta)-\frac{r^2}{24}\cos(4\theta)\right)\\
\nonumber &+\alpha\nu^2 \tilde \lambda_i \left( -4 \ln^2 r+ 8 \ln r\right)+\alpha^2\nu^4 \tilde \lambda_i\left( (2i-1)r^2\ln r+\left(\frac{5}{4}-\frac{9}{2}i\right)r^2-(r^2\ln r+\frac 54 r^2)\cos(2\theta)\right)\\
\nonumber  & +\Oc_{r\to \infty}(r^{-2-\epsilon})\\
\nonumber &= \frac{4}{r^2}+\alpha \nu^2\left( 4(i-\tilde \lambda_i) \ln^2 r-4(2i+1-2\tilde \lambda_i)\ln r -\cos(2\theta) \right) \pm \frac{\alpha^{\frac 32}\nu^3}{6\sqrt{2}}r\cos(3\theta)\\
\nonumber  &\qquad+\alpha^2\nu^4r^2 \Bigg( (2i-1)\tilde \lambda_i \ln r +(1-\frac 12-\tilde \lambda_i) \ln r \cos(2\theta)+i-\frac{7}{16}+\tilde \lambda_i(\frac 54-\frac 92 i)\\
 \label{eigenfunctions:id:expansion-tildeViinn-tech}  &\qquad \qquad \qquad +(\frac{5i}{4}-\frac 58-\frac{5}{4}\tilde \lambda_i)\cos(2\theta)-\frac{1}{24}\cos(4\theta) \Bigg)+\Oc_{r\to \infty}(r^{-2-\epsilon})
\end{align}
Next, using $\partial_{x_1}\Phi_U(y)=-\frac{4}{r}\cos \theta$, injecting \eqref{eigenfunction:id:decomposition-mS11}-\eqref{eigenfunction:id:decomposition-mS40'}, and \eqref{bd:estimate-S301-W301}-\eqref{bd:estimate-S402'-W402'} in \eqref{eigenfunctions:id:barVinn} we obtain for $1<r\lesssim \nu^{-1}$,
\begin{align}\nonumber 
\bar V_{i,\pm}&=\mp 4 \alpha^{\frac 12}\nu L_i^\inn (r^{-1}\cos \theta+\Oc(r^{-3})) \pm  \nu (\alpha-\beta)\Oc(r^{-1}\ln r)+ \nu^2 (\alpha-\beta)\Oc((|\ln \nu| + |\alpha - \beta|)|\ln r|^2)\\
\nonumber &\qquad  \pm \alpha^{\frac 32}\nu^3 L_i^\inn (\frac 12 -2i+2\tilde \lambda_i)(r\cos \theta+\Oc(r^{-1}\ln r))\pm \nu^3 (\alpha-\beta)\Oc(|\ln \nu|r \ln r)\\
\nonumber &\qquad+\nu^4 \Oc(|\ln \nu|\ln^2 r)+\alpha^2\nu^4M_i^\inn ( r^2 \cos(2\theta)+\Oc(1))\\
\label{eigenfunctions:id:expansion-barVinn} &=\mp 4 \alpha^{\frac 12}\nu L_i^\inn r^{-1}\cos \theta   \pm \alpha^{\frac 32}\nu^3 L_i^\inn (\frac 12 -2i+2\tilde \lambda_i)r\cos \theta +\alpha^2\nu^4M_i^\inn  r^2 \cos(2\theta) +\Oc(\nu|\ln \nu| r^{-2}\ln^2 r)
\end{align}
where we used that $|\alpha-\beta|\lesssim \nu$. Injecting \eqref{eigenfunctions:id:expansion-tildeViinn-tech} and \eqref{eigenfunctions:id:expansion-barVinn} in \eqref{eigenfunctions:id:decomposition-Vinn}, and using $r=|y|=\nu^{-1}\zeta$ (dropping the $\pm$ subscript for simplicity) we get:
\begin{align*}
& \nu^{-2} V^\inn_{i,\pm}(y) \\
&=\frac{4}{\zeta^2}\mp\frac{4\alpha^{\frac 12}L_i^\inn}{\zeta}\cos(\theta)\\
& \qquad +\alpha  \Bigg( 4(i-\tilde \lambda_i) \ln^2 \zeta- \left(8(i-\tilde \lambda_i)\ln \nu+4(2i+1-2\tilde \lambda_i)\right)\ln \zeta \\
&\qquad  +4(i-\tilde \lambda_i)\ln^2\nu+4(2i+1-2\tilde \lambda_i)\ln \nu-\cos (2\theta)\Bigg)\pm \alpha^{\frac 32}\zeta \left(\frac{1}{6\sqrt{2}}\cos(3\theta)+L_i^\inn (\frac 12 -2i+2\tilde \lambda_i)\cos \theta \right)\\
 &\quad+\alpha^2 \zeta^2 \Bigg( (2i-1)\tilde \lambda_i \ln \zeta -(2i-1)\tilde \lambda_i \ln \nu+(i-\frac 12 -\tilde \lambda_i) \ln \zeta \cos(2\theta)-(i-\frac 12 -\tilde \lambda_i)  \ln \nu \cos(2\theta)  \\
 &\qquad \qquad +i-\frac{7}{16}+\tilde \lambda_i(\frac 54-\frac 92 i) +(M_i^\inn+\frac{5i}{4}-\frac 58-\frac{5}{4}\tilde \lambda_i)\cos(2\theta)-\frac{1}{24}\cos(4\theta) \Bigg) +\Oc(\nu^\epsilon \zeta_\pm^{-2-\epsilon})
\end{align*}
The above identity implies the desired identity \eqref{expansion:Vi_inn} with the estimate \eqref{expansion:Vi_inn-estimate}. These estimates propagate for higher order derivatives since it does for all the $\Oc()$'s in the computations above.

\medskip

\noindent \textbf{Step 3}. \emph{Proof of the estimate \eqref{estimate:Ri_inn-pointwise} for $R_{i,\pm}^\inn$}. We estimate all terms in \eqref{eigenfunctions:id:Rinn-decomposition} for $|y|\ll \nu^{-1}$.

\smallskip

\noindent \underline{Estimate for $\tilde R^\inn_{i,\pm}(\tilde \phi^\inn_{i, \pm},\tilde V^\inn_{i,\pm})$}. We estimate all terms in \eqref{eigenfunctions:id:tildeRinn-decomposition}. For the first one, using \eqref{eigenfunctions:id:expansion-hatT2} and $|\tilde \lambda_i|\lesssim |\ln \nu|^{-1}$,
\begin{align} \label{eigenfunctions:bd:estimatetildeRiin-1}
 \tilde \lambda_i^2 \nu^4 \hat T_2  =\Oc \left(\frac{\nu^4}{|\log \nu|^2}\langle y \rangle^{-2}\right)
\end{align}

For the second, by  \eqref{eigenfunctions:id:expansion-hatT2}, \eqref{eigenfunctions:id:expansion-T3}, \eqref{eigenfunctions:id:expansion-T2i}, \eqref{eigenfunctions:id:G1} and \eqref{eigenfunctions:id:G2}, using $|\tilde \lambda_i|\lesssim |\ln \nu|^{-1}$,
\begin{align} \label{eigenfunctions:bd:estimatetildeRiin-2}
&\nu^5\left(\mp  \Lambda_i T_3 \mp \nabla.(T_3G_1)\mp  \nabla.(T_2^{(i)}G_2) \right) +\tilde \lambda_i \nu^5\left( \mp \nabla.(\hat T_2 G_2)\mp 2T_3 \right)\\
\nonumber &=  \nu^5 (\Oc (\langle y \rangle^{-1})+\Oc (\langle y \rangle^{-1})+\Oc (\langle y \rangle^{-1}))+\nu^5\Oc(|\ln \nu|^{-1}) (\Oc (\langle y \rangle^{-1})+\Oc (\langle y \rangle^{-1}))  = \Oc(\nu^5 \langle y \rangle^{-1})
\end{align}

For the third, using \eqref{eigenfunctions:id:expansion-T2i}, \eqref{eigenfunctions:id:expansion-T3}, \eqref{eigenfunctions:id:expansion-T4i}, \eqref{eigenfunctions:id:G1}, \eqref{eigenfunctions:id:G2} and \eqref{eigenfunctions:id:G3}:
\begin{align}
\nonumber \nu^6\left( \Lambda_i T_4^{(i)} +\nabla.(T_4^{(i)}G_1)+ \nabla.(T_3G_2)+ \nabla.(T_2^{(i)}G_3) \right)& =  \nu^6 (\Oc (1)+\Oc (1)+\Oc (1)+\Oc (1))\\
\label{eigenfunctions:bd:estimatetildeRiin-3}& =\Oc(\nu^6)
\end{align}

For the fourth, by \eqref{eigenfunctions:id:expansion-hatT2}, \eqref{eigenfunctions:id:expansion-T4i}, \eqref{eigenfunctions:id:expansion-hatT4}, \eqref{eigenfunctions:id:G1}, \eqref{eigenfunctions:id:G3}, and using $|\tilde \lambda_i|\lesssim |\ln \nu|^{-1}$,
\begin{align} \label{eigenfunctions:bd:estimatetildeRiin-4}
& \tilde \lambda_i \nu^6\left(  \Lambda_i \hat T_4^{(i)}+ \nabla.(\hat T_4^{(i)}G_1)+ \nabla.(\hat T_2G_3) +2T_4^{(i)}\right) +2 \tilde \lambda_i^2 \nu^6 \hat T_4^{(i)}\\
\nonumber &=\nu^6 \Oc(|\ln \nu|^{-1})(\Oc(\langle \ln \langle y \rangle\rangle)+\Oc(\langle \ln \langle y \rangle\rangle)+\Oc(1)+\Oc(1))+\nu^6 \Oc(|\ln \nu|^{-2})\Oc(\langle \ln \langle y \rangle\rangle) =\Oc (\nu^6)
\end{align}
where we used $\ln \langle y \rangle \lesssim |\ln \nu|$ for $|y|\lesssim \nu^{-1}$. We use \eqref{eigenfunctions:id:expansion-T3}, \eqref{eigenfunctions:id:expansion-T4i}, \eqref{eigenfunctions:id:expansion-hatT4}, \eqref{eigenfunctions:id:G2}, \eqref{eigenfunctions:id:G3} and $|\tilde \lambda_i|\lesssim |\ln \nu|^{-1}$ to estimate the fifth term,
\begin{align} \label{eigenfunctions:bd:estimatetildeRiin-5}
& \nu^7 \left( \nabla.(T_4^{(i)} G_{2})+ \nabla.(T_3 G_{3}) \right) +\tilde \lambda_i \nu^7 \nabla . (\hat T_4^{(i)} G_2)\\
\nonumber &= \nu^7 (\Oc(\langle y \rangle)+\Oc(\langle y \rangle))+\nu^7 \Oc(|\ln \nu|^{-1})\Oc(\langle y \rangle \langle \ln \langle y \rangle \rangle)=\Oc(\nu^7 \langle y \rangle)
\end{align}
where we used $\ln \langle y \rangle \lesssim |\ln \nu|$. Finally, for the sixth one, by \eqref{eigenfunctions:id:expansion-T4i}, \eqref{eigenfunctions:id:expansion-hatT4}, \eqref{eigenfunctions:id:G3} and $|\tilde \lambda_i|\lesssim |\ln \nu|^{-1}$:
\begin{align} \label{eigenfunctions:bd:estimatetildeRiin-6}
& -\nu^8\nabla.(T_4^{(i)} G_{3}) -\alpha^4 \tilde \lambda_i \nu^8\nabla.(\hat T_4^{(i)} G_{3})\\
\nonumber &= \nu^8\Oc(\langle y \rangle^2)+\nu^8 \Oc(|\ln \nu|^{-1})\Oc(\langle y \rangle^2 \langle \ln \langle y\rangle \rangle) = \Oc(\nu^8\langle y \rangle^{-2}).
\end{align}
Injecting \eqref{eigenfunctions:bd:estimatetildeRiin-1}-\eqref{eigenfunctions:bd:estimatetildeRiin-6} in \eqref{eigenfunctions:id:Rinn-decomposition}, and using that $\nu^{5+k}\langle y \rangle^{k-1}\lesssim \nu^5 \langle y \rangle^{-1}$ for $|y|\lesssim \nu^{-1}$ we obtain
\begin{align}
\nonumber \tilde R^\inn_{i,\pm}(\tilde \phi^\inn_{i, \pm},\tilde V^\inn_{i,\pm}) & = \Oc\left(\frac{\nu^4}{|\log \nu|^2}\langle y \rangle^{-2}\right)+\Oc(\nu^5 \langle y \rangle^{-1})++\Oc(\nu^6)++\Oc(\nu^7 \langle y \rangle)++\Oc(\nu^8 \langle y \rangle^{2})\\
\label{eigenfunctions:bd:estimatetildeRiin}&= \Oc\left(\frac{\nu^4}{|\log \nu|^2}\langle y \rangle^{-2}\right)+\Oc(\nu^5 \langle y \rangle^{-1}).
\end{align}

\noindent \underline{Estimate for the subleading terms}. We estimate all remaining terms in \eqref{eigenfunctions:id:Rinn-decomposition}. We recall that $|\alpha-\beta|\lesssim \nu$. For the second term, we inject  \eqref{bd:estimate-S301-W301}-\eqref{bd:estimate-S404-W404} in \eqref{eigenfunctions:id:barphiinn} and get
\begin{align}
\nonumber &  \bar  \phi_{i, \pm}^\inn \mp \nu (\alpha-\beta)\mathcal S_{1,1}\mp \alpha^{\frac12}\nu L_i^\inn \partial_{x_1}U \\
\nonumber & =  \nu^2 (\alpha-\beta)\mathcal S_{2,1}+\nu^2(\alpha-\beta)^2\mathcal S_{2,2}\pm\alpha^{\frac 32}\nu^3 L_i^\inn \mathcal S_{3,0}\pm \nu^3 (\alpha-\beta)\mathcal S_{3,1}+\nu^4\mathcal S_{4,0} +\alpha^2\nu^4 M_i^\inn \mathcal S_{4,0}'\\
\nonumber &= \Oc(\nu^3)\Oc(\langle y \rangle^{-2})+\Oc(\nu^4)\Oc(\langle y \rangle^{-2})+\nu^3\Oc(\langle y \rangle^{-3})+\Oc(\nu^4)\Oc(\langle y \rangle^{-1})+\nu^4\Oc( \langle y \rangle^{-2})+\nu^4\Oc( \langle y \rangle^{-2})\\
\label{eigenfunctions:bd:estimate-barphiinn-tech1} &= \Oc(\nu^3 \langle y \rangle^{-2}).
 \end{align}
 By \eqref{eigenfunctions:bd:estimate-barphiinn-tech1} and \eqref{eigenfunctions:id:G1} we have
\begin{align} \nonumber
&  - \alpha  \nu^2\left( \Lambda_i\left( \bar  \phi_{i, \pm}^\inn \mp \nu (\alpha-\beta)\mathcal S_{1,1}\mp \alpha^{\frac 12}\nu L^\inn_i\partial_{x_1}U\right)+\nabla .\left(\left(\bar  \phi_{i, \pm}^\inn   \mp \nu (\alpha-\beta)\mathcal S_{1,1}\mp \alpha^{\frac 12}\nu L^\inn_i\partial_{x_1}U\right) G_{1,\pm}  \right)\right)\\
\label{eigenfunctions:bd:estimatebarRiin-1} &\qquad =\Oc(\nu^2) \Oc(\nu^3 \langle y \rangle^{-2})+\Oc(\nu^2) \Oc(\nu^3 \langle y \rangle^{-2})= \Oc(\nu^5 \langle y \rangle^{-2})
 \end{align}
For the third term, by \eqref{eigenfunctions:bd:estimate-barphiinn-tech1}, \eqref{bd:estimate-S111-W111}, \eqref{eigenfunctions:id:expansion-barphiinn}, \eqref{eigenfunctions:id:G2} and \eqref{eigenfunctions:id:G3},
 \begin{align}\nonumber
& - \alpha^{\frac 32}\nu^3 \nabla.\left( (\bar  \phi_{i, \pm}^\inn\mp \alpha^{\frac 12}\nu L_i^\inn \partial_{x_1}U)  G_{2,\pm}\right)- \alpha^2 \nu^4 \nabla.\left(\bar  \phi_{i, \pm}^\inn G_{3,\pm}\right)\\
 \label{eigenfunctions:bd:estimatebarRiin-2} &\qquad =\nu^3  \nabla. \left(\Oc(\nu^2\langle y \rangle^{-3}) \Oc( \langle y \rangle^{2})\right)+\nu^4 \nabla. \left(\Oc(\nu \langle y \rangle^{-4})\Oc(\langle y \rangle^{3})\right) \ =\Oc(\nu^5\langle y \rangle^{-2})
 \end{align}
For the fourth term, we decompose by \eqref{eigenfunctions:id:barphiinn} and \eqref{eigenfunctions:id:decomposition-phiinn} and then use \eqref{eigenfunctions:id:expansion-T3}, \eqref{eigenfunctions:id:expansion-T4i}, \eqref{eigenfunctions:id:expansion-hatT4} and \eqref{eigenfunctions:bd:estimate-barphiinn-tech1}, and \eqref{eigenfunctions:id:expansion-T2i}, \eqref{eigenfunctions:id:expansion-hatT2} and \eqref{bd:estimate-S111-W111} to bound
\begin{align}
\nonumber& \phi_{i,\pm}^\inn-\Lambda U-\alpha \nu^2 T^{(i)}_2-\alpha \nu^2 \tilde \lambda_i \hat T_2\mp \nu (\alpha-\beta)\mathcal S_{1,1}\mp \alpha^{\frac 12}\nu L_i^\inn \partial_{x_1}U \\
\nonumber&\qquad =   \pm \alpha^{\frac 32}\nu^3 T_{3} +\alpha^2\nu^4  T_{4}^{(i)}   +\tilde \lambda_i \alpha^2 \nu^4 \hat T_{4}^{(i)}  + \bar  \phi_{i, \pm}^\inn \mp \nu (\alpha-\beta)\mathcal S_{1,1} \mp \alpha^{\frac 12}\nu L_i^\inn \partial_{x_1}U \\
\label{eigenfunctions:bd:estimate-barphiinn-tech2} &\qquad =\Oc(\nu^3\langle y \rangle^{-1})+\Oc(\nu^4)+\Oc(\nu^4)+\Oc(\nu^3\langle y \rangle^{-2})= \Oc(\nu^3 \langle y \rangle^{-1}),\\
\label{eigenfunctions:bd:estimate-barphiinn-tech3}& \phi_{i,\pm}^\inn-\Lambda U\mp \alpha^{\frac 12}\nu L_i^\inn \partial_{x_1}U = \Oc(\frac{\nu^2}{|\ln \nu|^2}\langle y \rangle^{-2})+\Oc(\nu^3 \langle y \rangle^{-1})+\Oc(\nu^2 \langle y \rangle^{-2}) = \Oc(\nu^2 \langle y \rangle^{-2}),
\end{align}
so that
  \begin{align}
\nonumber & \pm \nu (\alpha-\beta)a.\nabla ( \phi_{i,\pm}^\inn-\Lambda U-\alpha \nu^2 T^{(i)}_2-\alpha \nu^2 \tilde \lambda_i \hat T_2\mp \nu (\alpha-\beta)\mathcal S_{1,1}) \\
& \qquad +\nu^2 (\alpha-\beta)\left(\Lambda+\frac{\lambda_i}{\beta}\right)(\phi_{i,\pm}^\inn-\Lambda U)= \Oc(\nu^5 \langle y \rangle^{-2}). \label{eigenfunctions:bd:estimatebarRiin-3}
 \end{align}
For the fifth, we have that $U\Big(y\pm \frac{2a}{\nu}\Big)=\Oc(\nu^4)$ and $U\Big(y\pm \frac{2a}{\nu}\Big)-\frac{\nu^4}{2|a|^2}=\Oc(\nu^5\langle y \rangle)$ for $|y|\ll \nu^{-1}$. Hence by \eqref{eigenfunctions:bd:estimate-barphiinn-tech3}:
\begin{align} \label{eigenfunctions:bd:estimatebarRiin-4}
& \left(U\Big(y\pm \frac{2a}{\nu}\Big)-\frac{\nu^4}{2|a|^4}\right)\Lambda U+U\Big(y\pm \frac{2a}{\nu}\Big)( \phi_{i,\pm}^\inn-\Lambda U)\\
&\qquad = \Oc(\nu^5\langle y \rangle)\Oc(\langle y \rangle^{-4})+\Oc(\nu^4)\Oc(\nu^2 \langle y \rangle^{-2}) =\Oc(\nu^5 \langle y \rangle^{-3}).
\end{align}
For the sixth, by \eqref{expansion:phii_inn} and \eqref{expansion:phii_inn-estimate}, and \eqref{expansion:Vi_inn} and \eqref{expansion:Vi_inn-estimate}, and \eqref{lem:taylorpoissonfield} and $\nabla U(y\pm \frac{2a}{\nu})=\Oc(\nu^5)$ for $|y|\ll \nu^{-1}$:
  \begin{align} \label{eigenfunctions:bd:estimatebarRiin-5}
& -\nabla U\Big(y\pm \frac{2a}{\nu}\Big).\nabla V^\inn_{i,\pm}-\nabla .( \phi^\inn_{i,\pm}\tilde G_\pm) \\
&\qquad =\Oc(\nu^5)(\Oc(\langle y\rangle^{-3})+\Oc(\nu^2\langle y \rangle^{-1} \langle \ln \langle y \rangle \rangle)+\nabla . (\Oc(\langle y \rangle^{-4})\Oc(\nu^5 \langle y \rangle^{4})) =\Oc(\nu^5 \langle y \rangle^{-1})
 \end{align}
 Injecting \eqref{eigenfunctions:bd:estimatetildeRiin}, \eqref{eigenfunctions:bd:estimatebarRiin-1}, \eqref{eigenfunctions:bd:estimatebarRiin-2}, \eqref{eigenfunctions:bd:estimatebarRiin-3}, \eqref{eigenfunctions:bd:estimatebarRiin-4} and \eqref{eigenfunctions:bd:estimatebarRiin-5} in \eqref{eigenfunctions:id:Rinn-decomposition} shows the desired estimate \eqref{estimate:Ri_inn-pointwise}, as all estimates in the $\Oc()$'s above propagate to higher order derivatives.

\end{proof}

\subsubsection{Computation of the leading profiles}

In what follows, to lighten notation, the notation $\Oc_{r\to \infty}()$ will implicitly include estimates for derivatives. Namely, for $f\in C^\infty((0,\infty),\mathbb R)$ and $g:(0,\infty)\rightarrow (0,\infty)$ we will write
$$
f(r)=\Oc_{r\to \infty}(g(r))
$$
if for all $k\in \mathbb N$, we have for all $r\geq 1$:
$$
|\partial_r^k f(r)|\lesssim r^{-k}g(r).
$$
Throughout the analysis, we will treat with special care two natural obstacles for solving the systems \eqref{sys:T20T20hat}, \eqref{sys:T40T40hat}, \eqref{sys:Tkk31} and \eqref{sys:T42T42hat}. Indeed, the formal adjoint of $\Ls_0$ in $L^2$ is $\Ls_0^* \varepsilon =\Ms_0 \nabla . (U\nabla \varepsilon)$, and $\Ls_0^*(1)=\Ls_0^*(y_1)=0$ (from \cite{RSma14}, Lemma 2.4), which comes from the conservation of mass and momentum of the Keller-Segel system \eqref{eq:KS2d}. Therefore, if the source terms in these elliptic systems have zero mass and momentum, we will find decaying solutions, while if not we will find solutions which grow as $|y|\to \infty$.

\smallskip

\paragraph{Computing $(T_{2,0}^{(i)}, V_{2,0}^{(i)})$ and $(\hat T_{2}, \hat V_{2})$:} We claim the following.

\begin{lemma} \label{lem:T20hatT2} Let $i = 0, 1$ and for any constants $c_{2,0}^{(0)}, c_{2,0}^{(1)}, \hat c_{2}\in \mathbb R$, the equations \eqref{sys:T20T20hat}  admit smooth solutions $(T_{2,0}^{(i)},c_{2,0}^{(i)}+V_{2,0}^{(i)})$ and $(\hat T_2,\hat c_{2}+\hat V_2)$ such that
\begin{equation}
\label{eigenfunction:id:asymptoticT20infty0}
T_{2,0}^{(0)} = \Oc_{r \to \infty}\Big( \frac{\ln r }{r^4} \Big), \quad V_{2,0}^{(0)}= -4\ln r +  \Oc_{r \to \infty}\Big( \frac{\ln r }{r^2} \Big),
\end{equation}
\begin{equation} \label{eigenfunction:id:asymptoticT20infty}
T_{2,0}^{(1)} = -\frac{8}{r^2} + \Oc_{r \to \infty}\Big( \frac{\ln^2 r }{r^4} \Big), \quad V_{2,0}^{(1)}= 4 \ln^2 r - 12 \ln r+ \Oc_{r \to \infty}\Big( \frac{\ln^2 r }{r^2} \Big),
\end{equation}
\begin{equation}  \label{eigenfunction:id:asymptotichatT2infty}
\hat T_{2} =\frac{8}{r^2}+ \Oc_{r \to \infty}(r^{-4}|\ln r|^2), \quad \hat V_{2,0}=-4 \ln^2 r+ 8 \ln r+ \Oc_{r \to \infty}\Big( \frac{\ln^2 r }{r^2} \Big), 
\end{equation}
and
\begin{equation*} \label{eigenfunction:id:asymptoticT20hatT20}
|r^2 T_{2,0}^{(i)}| + |r \partial_r V_{2,0}^{(i)}| =  \Oc_{r \to 0}(r^2), \qquad |r^2 \hat T_{2,0}| + |r \partial_r \hat V_{2,0}| =  \Oc_{r \to 0}(r^2).
\end{equation*}
Moreover, these estimates propagate for higher order derivatives.

\end{lemma}

\begin{proof} Equation \eqref{sys:T20T20hat} can be solved by introducing the partial mass
$$
 m_{2,0}^{(i)} (r) = \int_0^r T_{2,0}^{(i)}(s) s ds, \quad T_{2,0}^{(i)}(r) = r^{-1} \partial_{r}  m_{2,0}^{(i)}, \quad \partial_rV_{2,0}^{(i)}(r) = - r^{-1}  m_{2,0}^{(i)}.
$$
Here, $  m_{2,0}^{(i)}$ solves the equation
$$
\As_0  m_{2,0}^{(i)} = g^{(i)}_{2,0}\quad \textup{with}\;\; g^{(i)}_{2,0}= \int_0^r \Lambda_i \Lambda U(\zeta) \zeta d\zeta = r^{2 i - 1} \pa_r \big( r^{4 - 2i} U \big). 
$$
where $\As_0$ is the second order operator
\begin{equation} \label{def:As0}
\As_0  = \partial_r^2 - \frac{\partial_r}{r} + \frac{Q}{r} \partial_r + \frac{\partial_r Q}{r} \quad \textup{with} \quad Q = m_U = \frac{4r^2}{1 + r^2}. 
\end{equation}
We compute by using Lemma \ref{lemm:invA0} with $g = g^{(1)}_{2,0}$ to get
\begin{align}
 m_{2,0}^{(1)}(r) &=  8 (r^{-2}+O(r^{-4}))\int_1^r (\zeta+O(\zeta^{-1}\ln \zeta))d\zeta \nonumber \\
 & \qquad \quad + \frac{1}{2} (1+O(r^{-2}\ln r))\int_0^r \zeta^2 \pa_\zeta (\zeta^2 U) d\zeta, \nonumber \\
&= 4 + \Oc\Big( \frac{\ln^2 r}{r^2} \Big) + \frac{1}{2}r^4 U(r) - 8 \ln r + 4  + \Oc\Big( \frac{\ln^2 r}{r^2} \Big) \nonumber \\
\label{eigenfunction:id:m20}
&=  -8 \ln r + 12 + \Oc\Big( \frac{\ln^2 r }{r^2} \Big) \quad \textup{as}\;\; r \to \infty,
\end{align}
where we used $\int_0^r \zeta^3 U d\zeta = 4 \big( \log(1 + r^2) + (1 + r^2)^{-1} - 1 \big)$. 
By noting that $g_{2,0}^{(0)}$ has a better decay than $g_{2,0}^{(1)}$ at infinity, which is $g_{2,0}^{(0)} = \Oc_{r \to \infty}(r^{-4})$, we then compute from Lemma \ref{lemm:invA0}, 
\begin{align}
 m_{2,0}^{(0)}(r) &=  8 (r^{-2}+O(r^{-4}))\int_1^r O(\zeta^{-1})d\zeta + \frac{1}{2} (1+O(r^{-2}\ln r))\int_0^r \pa_\zeta (\zeta^4 U) d\zeta, \nonumber \\
&= \Oc\Big( \frac{\ln r}{r^2} \Big) + \frac{1}{2}r^4 U(r) =  4 + \Oc\Big( \frac{\ln r }{r^2} \Big) \quad \textup{as}\;\; r \to \infty. \label{est:m20_0}
\end{align}
\noindent Similarly, we solve \eqref{sys:T20T20hat}  by applying Lemma \ref{lemm:invA0} with $g =  2 m_{\Lambda U} = 2 r^2 U$ to obtain
\begin{align}
 \hat m_{2} & = - 8 (r^{-2}+O(r^{-4}))\int_1^r (\zeta+O(\zeta^{-1}\ln \zeta))d\zeta+ (1+O(r^{-2}\ln r))\int_0^r \zeta^3 U d\zeta, \nonumber\\
 &=  8\ln r - 8 +O(r^{-2}|\ln r|^2),\quad \mbox{as } r\to \infty. \label{est:m2hat}
\end{align}
The asymptotic behavior of $m_{2,0}^{(0)}$, $m_{2,0}^{(1)}$ and $\hat m_{2}$ near zero is a direct consequence of Lemma \ref{lemm:invA0} and we omit the details. All estimates propagate for higher order derivative from a direct check at the identities \eqref{eigenfunction:id:m20} and \eqref{est:m20_0} using the formulas \eqref{systeme-odes:fundamental-solution-radial} and \eqref{systeme-odes:formula-radial}. This concludes the proof of Lemma \ref{lem:T20hatT2}.
\end{proof}

\paragraph{Computing $\mathsf T_{2,2}$ and $\mathsf V_{2,2}$:} We are going to apply Lemma \ref{lemm:SolHom} and \ref{lemm:sol_inhol} to solve the system \eqref{sys:Tkk31}. 

\begin{lemma} \label{lem:T22}
There exists a smooth solution $(\mathsf T_{2,2},\mathsf V_{2,2})$ of \eqref{sys:Tkk31} for $k = 2$ such that 
\begin{align}  \label{eigenfunction:id:asymptoticT22infty}
& \mathsf T_{2,2}(r)=-4 r^{-2} + \Oc_{r \to \infty}(r^{-\sqrt{8}}), \qquad \mathsf V_{2,2}(r)= -1+\Oc_{r\to \infty}(r^{2-\sqrt{8}}),\\
 \label{eigenfunction:id:asymptoticT220}
&\mathsf T_{2,2}(r) = \Oc_{r \to 0}(r^2), \qquad \qquad \qquad \qquad \quad \mathsf V_{2,2}(r) = \Oc_{r \to 0}(r^2).
\end{align}
\end{lemma}

\begin{proof} We apply Lemma \ref{lemm:sol_inhol} with $\ell = 2$ and get that the solution is of the form
\begin{equation}  \label{eigenfunction:id:asymptoticT220-tech}
\mathsf T_{2,2}(r) = \sum_{k = 1}^4 h_{2,k}(r) \gamma_{2,k}(r), \quad \mathsf V_{2,2}(r)=\sum_{k = 1}^4 g_{2,k}(r) \gamma_{2,k}(r), 
\end{equation}
where $(h_{2,k}, g_{2,k})_{1\leq k \leq 4}$ are given by Lemma \ref{lemm:SolHom} and 
\begin{align*}
 \gamma_{2,k}(r) &= \bar \gamma_{2,k} +(-1)^{k+1}  \int_0^r \frac{\Sigma_{2,2}(s) W_{2, k}(s)}{W_2} ds \\
 & = \bar \gamma_{2,k} +\frac{(-1)^{k} }{3\cdot 2^{12}}\int_0^r \Sigma_{2,2}(s) W_{2,k}(s) s^2 (1 + s^2)^2 ds,
 \end{align*}
for arbitrary constants $(\bar \gamma_{2,k})_{1\leq k\leq 4} $. We recall that from \eqref{eq:asyDlkr0} and \eqref{eq:asyDlkrinf}, $(W_{2,k})_{1\leq k \leq 4}$ admit the following asymptotics, for non-zero constants $(c_{2,k,0})_{1\leq k \leq 4}$:
\begin{align}
\label{eigenfunctions:id:expansion-W21-infty} W_{2,1} &= \left\{ \begin{array}{l l} c_{2,1,0} r^{-3}(1+o_{r\to 0}(1))  ,\\  -3\cdot 2^8 r^{-5}(1+\Oc_{r\to \infty}(r^{2-\sqrt{8}}))  , \end{array} \right.   \quad W_{2,3} = \left\{ \begin{array}{l l} c_{2,3,0} r^{-3} (1+o_{r\to 0}(1))   ,\\ -3\cdot 2^6 C_2 r^{-2\sqrt 2 - 3}(1+\Oc_{r\to \infty}(r^{2-\sqrt{8}}))  , \end{array} \right.    \\
\label{eigenfunctions:id:expansion-W22-infty}W_{2,2} &= \left\{ \begin{array}{l l} c_{2,2,0}  r (1+o_{r\to 0}(1))   ,\\   2^8 r^{-1} (1+\Oc_{r \to \infty}(r^{2-\sqrt{8}}))  ,\end{array} \right.   \quad W_{2,4} = \left\{ \begin{array}{l l} c_{2,4,0}  r(1+o_{r\to 0}(1))   ,\\ -3\cdot 2^6 K_2 r^{2\sqrt 2-3}(1+\Oc_{r\to \infty}(r^{2-\sqrt{8}}))   .\end{array} \right.  
\end{align}

Since $h_{2,2}$ and $h_{2,4}$ are singular at the origin, we choose $\bar \gamma_{2,2} = \bar \gamma_{2,4} = 0$. We further choose for $k=1,3$,
$$
\bar \gamma_{2,k} = \frac{1}{3\cdot 2^{12}}\int_0^\infty \Sigma_{2,2}(s) W_{2,k}(s) s^2 (1 + s^2)^2 ds < +\infty,
$$
to avoid growth at infinity. We have by \eqref{def:Sigma22},
$$
\Sigma_{2,2}(r)  = 32 r^{-4} +\Oc_{r \to \infty}(r^{-6}).
$$
We then obtain
\begin{align*}
\gamma_{2,1}(r) , \gamma_{2,3}(r) &= \Oc_{r \to 0}(1), \quad \gamma_{2,2}(r) , \gamma_{2,4}(r) = \Oc_{r \to 0}(r^6),
\end{align*}
and as $r \to \infty$, 
\begin{align*}
\gamma_{2,1}(r) &= -\frac{1}{3\cdot 2^{12}}\int_r^\infty 2^{13}\cdot 3 s^{-3}(1+\Oc_{s\to \infty}(s^{2-\sqrt{8}}))ds=-r^{-2}(1+\Oc_{r\to \infty}(r^{2-\sqrt{8}})),\\
\gamma_{2,2}(r) &= -\frac{1}{3\cdot 2^{12}}\int_0^r 2^{13}s(1+\Oc_{s\to \infty}(s^{2-\sqrt{8}}))ds=-\frac 13 r^{2}(1+\Oc_{r\to \infty}(r^{2-\sqrt{8}})),\\
\gamma_{2,3}(r) &= -\frac{1}{3\cdot 2^{12}}\int_r^\infty 2^{11}\cdot 3C_2 s^{-\sqrt{8}-1}(1+\Oc_{s\to \infty}(s^{2-\sqrt{8}}))ds=-\frac{C_2}{4\sqrt{2}}r^{-\sqrt{8}}(1+\Oc_{r\to \infty}(r^{2-\sqrt{8}})),\\
\gamma_{2,4}(r) &= -\frac{1}{3\cdot 2^{12}}\int_r^\infty 2^{11}\cdot 3K_2 s^{\sqrt{8}-1}(1+\Oc_{s\to \infty}(s^{2-\sqrt{8}}))ds=-\frac{K_2}{4\sqrt{2}}r^{\sqrt{8}}(1+\Oc_{r\to \infty}(r^{2-\sqrt{8}})).
\end{align*}
Using the asymptotic behaviors of $(h_{2,k}, g_{2,k})$ given in Lemma \ref{lemm:SolHom}, we directly get \eqref{eigenfunction:id:asymptoticT220} and obtain, as $r \to \infty$,
\begin{align*}
& \mathsf T_{2,2}(r)  = \sum_{k=1}^4\gamma_{2,k} h_{2,k}\\
&= -r^{-2}(1+O(r^{2-\sqrt{8}}))8r^{-2}(1+\Oc(r^{-2})) -\frac 13 r^{2}(1+O(r^{2-\sqrt{8}}))24 r^{-6}(1+\Oc(r^{-2}))\\
&-\frac{C_2}{4\sqrt{2}}r^{-\sqrt{8}}(1+O(r^{2-\sqrt{8}}))16 K_2 r^{\sqrt{8}-2}(1+\Oc(r^{-2})) -\frac{K_2}{4\sqrt{2}}r^{\sqrt{8}}(1+O(r^{2-\sqrt{8}}))16 C_2 r^{-\sqrt{8}-2}(1+\Oc(r^{\sqrt{8}-4})) \\
&= -4r^{-2}+\Oc(r^{-\sqrt{8}})
\end{align*}
where we used in the last line $C_2K_2 = \frac{1}{\sqrt{2}}$,  and 
\begin{align*}
& \mathsf V_{2,2}(r)  = \sum_{k=1}^4\gamma_{2,k} g_{2,k}\\
&= -r^{-2}(1+O(r^{2-\sqrt{8}})) r^{2}(1+\Oc(r^{-2})) -\frac 13 r^{2}(1+O(r^{2-\sqrt{8}}))3r^{-2}(1+\Oc(r^{-2}))\\
&-\frac{C_2}{4\sqrt{2}}r^{-\sqrt{8}}(1+O(r^{2-\sqrt{8}}))(-4 K_2 r^{\sqrt{8}}(1+\Oc(r^{-2}))) -\frac{K_2}{4\sqrt{2}}r^{\sqrt{8}}(1+O(r^{2-\sqrt{8}}))(-4 C_2 r^{-\sqrt{8}})(1+\Oc(r^{\sqrt{8}-4})) \\
&= -1+\Oc(r^{2-\sqrt{8}}).
\end{align*}
We notice that the estimates indeed propagate for higher order derivatives by a direct check at the formulas \eqref{eigenfunction:id:asymptoticT220-tech} because of the asymptotic behaviour of the functions in Lemmas \ref{lemm:sol_inhol} and \ref{lemm:SolHom}. This concludes the proof of Lemma \ref{lem:T22}. 
\end{proof}

\paragraph{Computing $T_{3}$ and $V_3$:} We use Lemmas \ref{lemm:SolHom} and \ref{lemm:sol_inhol} to solve \eqref{sys:Tkk31} for $k=3$. 

\begin{lemma} \label{lem:T3k}
There exist smooth solutions $(\mathsf T_{3,k},V_{3,k})$ to \eqref{sys:Tkk31}  for $k = 1,3$ such that 
\begin{align}   \label{eigenfunction:id:asymptoticT310} 
&\mathsf T_{3,1}(r) = \Oc_{r \to 0}(r), \quad \mathsf V_{3,1}(r) = \Oc_{r \to 0}(r), \\
\label{eigenfunction:id:asymptoticT31infty} 
& \mathsf T_{3,1}(r)  =  \Oc_{r \to \infty}\Big( r^{-3}\Big), \quad \mathsf V_{3,1}(r) =  \Oc_{r \to \infty}(\frac{\ln r}{r}).
\end{align}
and
\begin{equation} \label{eigenfunction:id:asymptoticT330}
\mathsf T_{3,3}(r) = \Oc_{r \to 0}(r^3), \quad \mathsf V_{3,3}(r) = \Oc_{r \to 0}(r^3), 
\end{equation}
\begin{equation}  \label{eigenfunction:id:asymptoticT33infty}
 \mathsf T_{3,3}(r)  = \frac{4}{3\sqrt{2}}  r^{-1}+ \Oc_{r \to \infty}\Big( r^{2-\sqrt{13} }\Big), \quad \mathsf V_{3,3}(r) = \frac{1}{6\sqrt{2}}r+\Oc_{r\to \infty}(r^{4-\sqrt{13}}).
\end{equation}
The estimates propagate for higher order derivatives.
\end{lemma}
\begin{proof} We recall from \eqref{sys:Tkk31},
\begin{align} \label{sys:T3kproof} 
k =1,3, \quad \left\{ \begin{array}{l l} 
 \Ls_{0,k} (\mathsf T_{3,k},\mathsf V_{3,k}) =\Sigma_{3,k},\\
 -\Delta_k \mathsf V_{3,k}=\mathsf T_{3,k}.
  \end{array} \right. 
\end{align}
Since solving for $\big(\mathsf T_{3,k}, \mathsf V_{3,k}\big)$ for $k = 1,3$ is similar as solving for $\big(\mathsf T_{2,2}, \mathsf V_{2,2}\big)$ that we have detailed in the proof of Lemma \ref{lem:T22}, we will only detail the main new points.

\medskip

\noindent \textbf{Step 1}. \emph{Computing $\mathsf{T}_{3,1}$}. There is a surprising cancellation for the source term, as it has no horizontal momentum:
$$
\int_{\mathbb R^2} y_1 \nabla . (\Lambda U G_2)dy=-\int_{\mathbb R^2}\Lambda U (1,0).G_2 dy=\frac{1}{4\sqrt{2}}\int_{\mathbb R^2}\Lambda U (y_1^2-y_2^2-1)dy=0
$$
where we used that $\Lambda U$ is spherically symmetric with $\int \Lambda U=0$, which implies in polar coordinates
\begin{equation} \label{id:sigma31-no-momentum}
\int_0^\infty \Sigma_{3,1}(r)r^2dr=0.
\end{equation}
This cancellation permits to find a smooth and decaying solution $(\mathsf T_{3,1}\mathsf V_{3,1})$ of \eqref{sys:Tkk31}. Namely, by applying Lemma \ref{lemm:sol_inhol} with $k=1$ we have a solution of the form
\begin{align} \label{sys:T31-formula-tech}
\mathsf T_{3,1} = \sum_{k = 1}^4\gamma_{1,k} h_{1,k}, \quad \mathsf V_{3,1} = \sum_{k = 1}^4\gamma_{1,k} g_{1,k},
\end{align}
where, using $\Sigma_{3,1}=O_{r\to 0}(r)$ and $\Sigma_{3,1}=O_{r\to \infty}(r^{-5})$,
\begin{align*}
\gamma_{1,1} &= \int_{0}^r \frac{\Sigma_{3,1} W_{1,1}}{W_1} ds = \left\{ \begin{array}{l l}  \Oc_{r\to 0}(\int_0^\infty s ds)=  \Oc_{r\to 0}(r^2) , \\
\Oc_{r\to \infty} (\int_r^\infty s\langle s \rangle^{-2}ds)  =   \Oc_{r \to \infty}(\ln r), \end{array} \right.\\
\gamma_{1,2} &= \int_{0}^r \frac{\Sigma_{3,1} W_{1,2}}{W_1} ds =   \left\{ \begin{array}{l l}  \Oc_{r\to 0}( \int_0^r s^3 ds) = \Oc_{r\to 0}(r^4) ,\\
 \Oc_{r\to \infty}( \int_r^\infty s^{-3} ds) = \Oc_{r \to \infty}(r^{-2}),
 \end{array} \right. \\
\gamma_{1,3} &= -\int_{r}^\infty \frac{\Sigma_{3,1} W_{1,3}}{W_1} ds =   \left\{ \begin{array}{l l}  \Oc_{r\to 0}(  \int_0^\infty s\langle s \rangle^{-\sqrt{5}-3} ds) = \Oc_{r\to 0}(1) ,\\
 \Oc_{r\to \infty}( \int_r^\infty s^{-\sqrt{5}-2} ds) = \Oc_{r \to \infty}(r^{-\sqrt{5}-1}),
 \end{array} \right. \\
\gamma_{1,4} &= \int_0^r \frac{\Sigma_{3,1} W_{1,4}}{W_1} ds =    \left\{ \begin{array}{l l}  \Oc_{r\to 0}(  \int_0^r s^3 ds ) = \Oc_{r\to 0}(r^4), \\
 \Oc_{r\to \infty} ( \int_0^r r^3\langle r\rangle^{\sqrt{5}-5})=  \Oc_{r \to \infty}(r^{\sqrt{5}-1}).
 \end{array} \right. 
\end{align*}
In the second identity, we used that $\int_0^r \frac{\Sigma_{3,1}W_{1,2}}{W_1}ds=-\int_r^\infty \frac{\Sigma_{3,1}W_{1,2}}{W_1}ds$ by \eqref{id:sigma31-no-momentum} and the fact that $\frac{W_{1,2}}{W_1}=\frac{2\sqrt{5}+1}{24}r^2$ from \eqref{id:W12-W1}. Using the asymptotic behaviors of $(h_{1,k}, g_{1,k})$ given in Lemma \ref{lemm:SolHom} and the above asymptotic estimates yield the desired estimates  \eqref{eigenfunction:id:asymptoticT310} and \eqref{eigenfunction:id:asymptoticT31infty}. 

\medskip

\noindent \textbf{Step 2}. \emph{Computing $\mathsf{T}_{3,3}$}. We apply Lemma \ref{lemm:sol_inhol} with $\ell = 3$ and $f = \Sigma_{3,3}$, where we compute from \eqref{def:Sigma3k},
$$\Sigma_{3,3} =O_{r\to 0}(r^3), \qquad \Sigma_{3,3}(r)= -\frac{16}{\sqrt 2} r^{-3} + \Oc_{r \to \infty}(r^{-5}).$$
We then obtain a solution of the form
\begin{align} \label{sys:T33-formula-tech}
\mathsf T_{3,3} = \sum_{k = 1}^4\gamma_{3,k} h_{3,k}, \quad \mathsf V_{3,3} = \sum_{k = 1}^4\gamma_{3,k} g_{3,k}.
\end{align}
We choose the integration constants so that
\begin{align*}
\gamma_{3,1} &=   -\int_{r}^\infty \frac{\Sigma_{3,3} W_{3,1}}{W_3} ds, \qquad  \gamma_{3,2} = - \int_0^r \frac{\Sigma_{3,3} W_{3,2}}{W_3} ds\\
\gamma_{3,3} &= -\int_{r}^\infty \frac{\Sigma_{3,3} W_{3,3}}{W_3} ds , \qquad  \gamma_{3,4} = - \int_{0}^r \frac{\Sigma_{3,3} W_{3,4}}{W_3} ds .
\end{align*}
The regularity of $(\mathsf T_{3,3},\mathsf V_{3,3})$ as $r \to 0$ can be checked the same way we did for $(\mathsf T_{3,1},\mathsf V_{3,1})$. To obtain the expansion as $r\to \infty$ we compute from \eqref{eq:asyDlkrinf}:
\begin{align*}
& W_{3,1}(r)=-3\cdot 2^9 r^{-6}(1+\Oc_{r\to \infty}(r^{3-\sqrt{13}})), \qquad W_{3,2}(r)=-3\cdot 2^8(1+\Oc_{r\to \infty}(r^{3-\sqrt{13}})),\\
& W_{3,3}(r)=-3\cdot 2^8C_\ell r^{-\sqrt{13}-3}(1+\Oc_{r\to \infty}(r^{3-\sqrt{13}})), \qquad W_{3,4}(r)=-3\cdot 2^8K_\ell r^{\sqrt{13}-3}(1+\Oc_{r\to \infty}(r^{3-\sqrt{13}}).
\end{align*}
Recalling that $W_3(r)=-9\cdot2^{13}r^{-6}+\Oc_{r\to \infty}(r^{-8})$ we obtain
\begin{align*}
& \gamma_{3,1} = \int_r^\infty \frac{1}{3\sqrt{2}}s^{-3}(1+\Oc(s^{3-\sqrt{13}}))ds=\frac{1}{6\sqrt{2}}r^{-2}(1+\Oc(r^{3-\sqrt{13}})),\\
& \gamma_{3,2}  = \int_0^r \frac{1}{6\sqrt{2}}s^{3}(1+\Oc(s^{3-\sqrt{13}}))ds=\frac{1}{24\sqrt{2}}r^{4}(1+\Oc(r^{3-\sqrt{13}})),\\
& \gamma_{3,3} = \int_r^\infty \frac{C_\ell}{6\sqrt{2}}s^{-\sqrt{13}}(1+\Oc(s^{3-\sqrt{13}}))ds=\frac{C_\ell}{6\sqrt{2}(\sqrt{13}-1)}r^{1-\sqrt{13}}(1+\Oc(r^{3-\sqrt{13}})),\\
& \gamma_{3,4} = \int_0^r \frac{K_\ell}{6\sqrt{2}}s^{\sqrt{13}}(1+\Oc(s^{3-\sqrt{13}}))ds=\frac{K_\ell}{6\sqrt{2}(\sqrt{13}+1)}r^{1+\sqrt{13}}(1+\Oc(r^{3-\sqrt{13}})),
\end{align*}
from what we deduce, using Lemma \ref{lemm:SolHom},
\begin{align*}
\mathsf T_{3,3} & =\sum_{k=1}^4 \gamma_{3,k}h_{3,k}\\
&= \Oc(r^{-2})\Oc(r^{-1})+\Oc(r^4)\Oc(r^{-7})+\frac{C_\ell r^{1-\sqrt{13}}}{6\sqrt{2}(\sqrt{13}-1)}(1+\Oc(r^{3-\sqrt{13}})) 16 K_\ell r^{\sqrt{13}-2}(1+\Oc(r^{-2}))\\
&+\frac{K_\ell r^{1+\sqrt{13}}}{6\sqrt{2}(\sqrt{13}+1)}(1+\Oc(r^{3-\sqrt{13}})) 16 C_\ell r^{-\sqrt{13}-2}(1+\Oc(r^{\sqrt{13}-5}))\\
&=\frac{4}{3\sqrt{2}}r^{-1}(1+\Oc_{r\to \infty}(r^{3-\sqrt{13}}))
\end{align*}
where we used that $K_3C_3 = \frac{3}{\sqrt{13}}$, and
\begin{align*}
\mathsf V_{3,3} & =\sum_{k=1}^4 \gamma_{3,k}g_{3,k}\\
&= \frac{1}{6\sqrt{2}}r^{-2}(1+\Oc(r^{3-\sqrt{13}}))2r^3(1+\Oc(r^{-2}))+\frac{1}{24\sqrt{2}}r^{4}(1+\Oc(r^{3-\sqrt{13}}))4r^{-3}(1+\Oc(r^{-2}))\\
& +\frac{C_\ell r^{1-\sqrt{13}}}{6\sqrt{2}(\sqrt{13}-1)}(1+\Oc(r^{3-\sqrt{13}}))(-4 K_\ell r^{\sqrt{13}-2})(1+\Oc(r^{-2}))\\
&+\frac{K_\ell r^{1+\sqrt{13}}}{6\sqrt{2}(\sqrt{13}+1)}(1+\Oc(r^{3-\sqrt{13}}))(-4 C_\ell r^{-\sqrt{13}-2})(1+\Oc(r^{\sqrt{13}-5}))\\
&=\frac{r}{6\sqrt{2}}(1+\Oc_{r\to \infty}(r^{3-\sqrt{13}}))
\end{align*}
This concludes the proof of \eqref{eigenfunction:id:asymptoticT330} and \eqref{eigenfunction:id:asymptoticT33infty}. The estimates \eqref{eigenfunction:id:asymptoticT310}, \eqref{eigenfunction:id:asymptoticT31infty}, \eqref{eigenfunction:id:asymptoticT330} and \eqref{eigenfunction:id:asymptoticT33infty} propagate for higher order derivative from a direct check at the identities \eqref{sys:T31-formula-tech} and \eqref{sys:T33-formula-tech} using the asymptotic behaviours of the functions in Lemmas \ref{lemm:sol_inhol} and \ref{lemm:SolHom}.
\end{proof}

\paragraph{Computing $T_{4}$ and $\hat T_{4}$:} We now compute the profiles $T_4$ and $\hat T_4$. 

\begin{lemma} \label{lem:T40hatT40}
Let $i = 0,1$ and $T_{2,0}^{(0)}$, $T_{2,0}^{(1)}$, $\hat T_{2}$ and $\mathsf T_{2,2}$ be described as in Lemmas \ref{lem:T20hatT2} and \ref{lem:T22}. For any constants $c_{4,0}^{(0)}, c_{4,0}^{(1)}$, $\hat c_{4,0}^{(0)}, \hat c_{4,0}^{(1)}$, the systems \eqref{sys:T40T40hat} admit smooth solutions $(T_{4,0}^{(i)},c_{4,0}^{(i)}+V_{4,0}^{(i)})$ and $(\hat T_{4,0}^{(i)}, \hat c_{4,0}^{(i)}+\hat V_{4,0}^{(i)})$ such that for $r \to \infty$,
\begin{equation} \label{eigenfunction:id:asymptoticT40infty}
T_{4,0}^{(i)} = \frac74-4i + \Oc(r^{2-\sqrt{8}}), \quad V_{4,0}^{(i)} = \Big( i-\frac{7}{16} \Big)  r^2 + \Oc(r^{4-\sqrt{8}}),
\end{equation}
\begin{align}  \label{eigenfunction:id:asymptotichatT40infty}
\hat T_{4,0}^{(i)} &= (4-8i) \ln r +10i-1+ \Oc_{r \to \infty}(r^{-2} \ln^2 r), \\
\label{eigenfunction:id:asymptotichatV40infty}
 \hat V_{4,0}^{(i)} &= (2i-1) r^2\ln r + \Big(\frac 54-\frac{9}{2}i\Big)r^2+ \Oc_{r\to \infty}(\ln^3 r),
\end{align}
where $B_2$ is the free parameter introduced in Lemma \ref{lem:T22}. The estimates propagate to higher order derivatives.
\end{lemma}

\begin{proof}
Note that there is a tail cancellation in $\Lambda T_2$ (since $\Lambda (r^{-2}) = 0$), from which we introduce two different profiles $T_4$ and $\hat T_4$.  Roughly speaking, $T_4$ and $\hat T_4$ gather all terms of the order $\Oc(\nu^4)$ and $\Oc(\tilde{\lambda}_1 \nu^4)$ respectively, so that the generated error term will be of the order $\Oc(\tilde{\lambda}_i^2 \nu^4) \sim \Oc( \nu^4 |\ln \nu|^{-2})$ in the blowup variables,  which is the typical size to successfully control the remainder. 

Since $\Sigma_{4,0}^{(i)}$ is radial, we use the partial mass to solve for $T_{4,0}^{(i)}$. We recall from \eqref{def:Sigma40},
$$
 \Sigma_{4,0}^{(i)}=\Lambda_i T_{2,0}^{(i)} + \frac{1}{4}\Lambda \mathsf T_{2,2} - \frac{1}{16 } \Lambda^2 U, \quad \Lambda_i = \Lambda + 2(1 - i).
$$
From \eqref{eigenfunction:id:asymptoticT20infty}, \eqref{eigenfunction:id:asymptoticT20infty0}, \eqref{eigenfunction:id:asymptoticT22infty},  \eqref{est:m20_0} and the identities $\Lambda_i f(r) = r^{2i - 3}\pa_r(r^{4 - 2i} f)$, we get
\begin{align*}
g_{4,0}^{(i)}(r) = \int_0^r \Sigma_{4,0}^{(i)}(\zeta) \zeta d\zeta & = r^2T_{2,0}^{(i)} + 2(1-i)m_{2,0}^{(i)} + \frac{1}{4}r^2 \mathsf T_{2,2} - \frac{1}{16}  r^2 \Lambda U\\
& = 7-16i+ \Oc_{r \to \infty}\Big( r^{2-\sqrt{8}}\Big),
\end{align*}
Applying Lemma \ref{lemm:invA0} to solve $\As_0 m^{(i)}_{4,0}  = g_{4,0}^{(i)}$ yields
$$m_{4,0}^{(i)} = \Big(\frac{7}{8}-2i\Big)  r^2 + \Oc_{r \to \infty}(r^{4-\sqrt{8}}),$$ 
hence,
\begin{align} \label{sys:id:formulaT40i}
T_{4,0}^{(i)}&= \frac{\pa_r m_{4,0}^{(i)}}{r} = \frac 74-4i+ \Oc_{r\to \infty}(r^{2-\sqrt{8}}), \\
 \label{sys:id:formulaV40i} \pa_r V_{4,0}^{(i)} &= -\frac{m_{4,0}^{(i)}}{r} =  \Big(2i- \frac{7}{8}\Big)  r + \Oc_{r \to \infty}(r^{3-\sqrt{8}}).
\end{align}
Similarly, solving $\hat T_{4,0}^{(i)}$ is reduced to solving $\As_0 \hat m_{4,0}^{(i)} = \hat g_{4,0}^{(i)}$, where we compute asymptotically from \eqref{def:Sigma4khat} and Lemma \ref{lem:T20hatT2},
\begin{align*}
\hat g_{4,0}^{(i)} &= m_{2 T_{2,0}^{(i)}+\Lambda_i \hat T_2} = 
 2 m_{2,0}^{(i)}+m_{\Lambda_i \hat T_{2}} \\
 & = 2m_{2,0}^{(i)}+ r \pa_r \hat m_2 + 2(1 - i) \hat m_2= (16-32i)\ln r+32 i + \Oc_{r \to \infty}(r^{-2}\ln^2 r),
\end{align*}
where we used \eqref{eigenfunction:id:m20} \eqref{est:m20_0} and \eqref{est:m2hat}. Applying Lemma \ref{lemm:invA0}  with $g = g_{4,0}^{(i)}$ yields
\begin{align*}
\hat m_{4,0}^{(i)} &= \frac{1}{2 r^2} \int_r^1 \zeta^3 \big[(16-32i)\ln \zeta+32 i \big] d\zeta +\frac{1}{2}\int_0^r \zeta \big[(16-32i)\ln \zeta+32 i \big] d\zeta  + \Oc (\ln^2r)\\
&= (2-4i) r^2 \ln r + \big(7i-\frac 32 \big) r^2 +  \Oc_{r \to \infty} (\ln^2r). 
\end{align*}
hence, 
\begin{align} \label{sys:id:formulahatT40}
\hat T_{4,0} &=  \frac{\pa_r \hat m_{4,0}^{(i)}}{r} =  (4-8i) \ln r +10i-1+ \Oc_{r \to \infty}(r^{-2} \ln^2 r), \\
 \label{sys:id:formulahatV40} \pa_r \hat V_{4,0} &= - \frac{\hat m_{4,0}^{(i)}}{r} = (4i-2) r\ln r +  \big(\frac 32-7i \big) r+ \Oc_{r\to \infty}(r^{-1} \ln^2 r).
\end{align}
This ends the proof of Lemma \ref{lem:T40hatT40}. The estimates \eqref{eigenfunction:id:asymptotichatV40infty}, \eqref{eigenfunction:id:asymptotichatT40infty} and \eqref{eigenfunction:id:asymptoticT40infty} propagate to higher order derivatives by a direct check at the identities \eqref{sys:id:formulaT40i}, \eqref{sys:id:formulaV40i}, \eqref{sys:id:formulahatT40} and \eqref{sys:id:formulahatV40} because of the asymptotic behaviour of the functions in Lemmas \ref{lemm:sol_inhol} and \ref{lemm:SolHom}.
\end{proof}

\begin{lemma} \label{lem:T42hatT42}
Let  $i = 0, 1$ and $T_{2,0}^{(i)}$, $\mathsf T_{2,2}$ and $\hat T_{2}$ be given by Lemmas \ref{lem:T20hatT2} and \ref{lem:T22}. The systems \eqref{sys:T42T42hat} admit smooth solutions $(\mathsf T_{4,2}^{(i)}, \mathsf V_{4,2}^{(i)})$ and $(\hat {\mathsf T}_{4,2},\hat {\mathsf V}_{4,2})$ depending  such that for $r \to \infty$:
\begin{align}
\label{eigenfunction:id:asymptoticT42iinfty}
 {\mathsf T}_{4,2}^{(i)} &= 2- 4 i + \Oc(r^{2-\sqrt{8}}), \quad {\mathsf V}_{4,2}^{(i)} = (i-\frac 12)r^2\ln r+(\frac 54 i-\frac 58) r^2 + \Oc_{r\to \infty}( r^{4-2\sqrt{2}})),\\
\label{eigenfunction:id:asymptotichatT42iinfty} \hat {\mathsf T}_{4,2} &= 4+ \Oc(r^{2-\sqrt{8}}), \quad  \hat {\mathsf V}_{4,2} = - r^2\ln r-\frac{5}{4} r^2 + \Oc_{r\to \infty}( r^{4-2\sqrt{2}})).
\end{align}
The estimates propagate to higher order derivatives.
\end{lemma}
\begin{proof} From Lemma \ref{lemm:sol_inhol} with $\ell = 2$, we choose the solution of the form
$$ \mathsf T_{4,2}^{(i)}(r) = \sum_{k = 1}^4 h_{2,k}(r) \gamma_{4,k}(r), \quad \mathsf V_{4,2}^{(i)}(r)=\sum_{k = 1}^4 g_{2,k}(r) \gamma_{4,k}(r), $$ 
where $(h_{2,k}, g_{2,k})_{1\leq k \leq 4}$ are given in Lemma \ref{lemm:SolHom}, $W_2=- 3. 2^{12}r^{-2}(1+r^2)^{-2}$ and with the following choice of integration constants
\begin{align*}
 \gamma_{4,1}(r) &= \frac{-1}{3\cdot 2^{12}}\int_0^r \Sigma_{4,2}^{(i)}(s) W_{2,1}(s) s^2 (1 + s^2)^2 ds, \quad  \gamma_{4,1}(r) &= \frac{1}{3\cdot 2^{12}}\int_0^r \Sigma_{4,2}^{(i)}(s) W_{2,2}(s) s^2 (1 + s^2)^2 ds,\\
  \gamma_{4,3}(r) &= \frac{1}{3\cdot 2^{12}}\int_r^\infty \Sigma_{4,2}^{(i)}(s) W_{2,1}(s) s^2 (1 + s^2)^2 ds, \quad  \gamma_{4,1}(r) &= \frac{1}{3\cdot 2^{12}}\int_0^r \Sigma_{4,2}^{(i)}(s) W_{2,2}(s) s^2 (1 + s^2)^2 ds.
 \end{align*}
Injecting \eqref{eigenfunction:id:asymptoticT20infty0}, \eqref{eigenfunction:id:asymptoticT20infty} and \eqref{eigenfunction:id:asymptoticT22infty} in \eqref{def:Sigma42}, we have 
\begin{equation} \label{eigenfunctions:id:Sigma42-expansion}
\Sigma_{4,2}^{(i)}(r) = \Lambda_i \mathsf T_{2,2}+\frac 12 r\partial_r \mathsf T_{2,0}^{(i)}-\frac{r}{16} \partial_r \Lambda U = \frac{(-8+16i) }{r^2} + \Oc_{r \to \infty}(r^{-\sqrt{8}}).
\end{equation}
The smoothness of $(\mathsf T_{4,2}^{(i)},\mathsf V_{4,2}^{(i)})$ as $r\to 0$ can be checked as for that of  $(\mathsf T_{2,2},\mathsf V_{2,2})$ in the proof of Lemma \ref{lem:T22}. To obtain the expansion as $r\to \infty$ we use \eqref{eigenfunctions:id:expansion-W21-infty} and \eqref{eigenfunctions:id:expansion-W22-infty} and compute
\begin{align*}
& \gamma_{4,1} =-\frac{1}{3\cdot 2^{12}}\int_0^r -3 \cdot 2^{11} (-1+2i) \langle s\rangle^{-1}(1+\Oc_{s\to \infty}(s^{2-\sqrt{8}}))ds=\frac{-1+2i}{2}\ln r +\Oc_{r\to \infty}(r^{2-\sqrt{8}}),\\
& \gamma_{4,2} = \frac{1}{3\cdot 2^{12}}\int_0^r   2^{11} (-1+2i) s^3 (1+\Oc_{s\to \infty}(s^{2-\sqrt{8}}))ds=\frac{-1+2i}{24}r^4(1+\Oc(r^{2-\sqrt{8}}))\\
& \gamma_{4,3} = \frac{1}{3\cdot 2^{12}}\int_r^\infty - 3\cdot 2^{9} C_2 (-1+2i)s^{1-2\sqrt{2}} (1+\Oc_{s\to \infty}(s^{2-\sqrt{8}}))ds=-\frac{C_2 (-1+2i)}{16(\sqrt{2}-1)}r^{2-2\sqrt{2}} (1+\Oc_{r\to \infty}(r^{2-\sqrt{8}}))\\
& \gamma_{4,4} = \frac{1}{3\cdot 2^{12}}\int_0^r - 3 \cdot 2^{9} K_2 (-1+2i) s^{1+2\sqrt{2}} (1+\Oc_{s\to \infty}(s^{2-\sqrt{8}}))ds=-\frac{K_2 (-1+2i)}{16(\sqrt{2}+1)}r^{2-2\sqrt{2}} (1+\Oc_{r\to \infty}(r^{2-\sqrt{8}})).
\end{align*}
From the asymptotic behaviors of $(h_{2,k}, g_{2,k})_{1 \leq k \leq 4}$ given in Lemma \ref{lemm:SolHom} and the relation $C_2K_2 = \frac{1}{\sqrt 2}$, we end up with 
\begin{align*}
& \mathsf T_{4,2}^{(i)}(r)  = \sum_{k=1}^4\gamma_{4,k} h_{2,k}\\
&= \Oc(\ln r)\Oc(r^{-2})+\Oc(r^4)\Oc(r^{-6})-\frac{C_2 (-1+2i)}{16(\sqrt{2}-1)}r^{2-2\sqrt{2}} (1+\Oc_{r\to \infty}(r^{2-\sqrt{8}})) 16 K_2 r^{\sqrt{8}-2}(1+\Oc(r^{-2}))\\
& -\frac{K_2 (-1+2i)}{16(\sqrt{2}+1)}r^{2-2\sqrt{2}} (1+\Oc_{r\to \infty}(r^{2-\sqrt{8}}))16 C_2 r^{-\sqrt{8}-2}(1+\Oc(r^{\sqrt{8}-4})) \\
&= 2-4i+\Oc(r^{2-\sqrt{8}})
\end{align*}
and 
\begin{align*}
& \mathsf V_{4,2}^{(i)}(r)  = \sum_{k=1}^4\gamma_{2,k} g_{2,k}\\
&= \frac{2i-1}{2}\ln r +\Oc_{r\to \infty}(r^{2-\sqrt{8}}) r^{2}(1+\Oc(r^{-2})) +\frac{2i-1}{24}r^4(1+\Oc(r^{2-\sqrt{8}})) 3r^{-2}(1+\Oc(r^{-2}))\\
&-\frac{C_2 (2i-1)}{16(\sqrt{2}-1)}r^{2-2\sqrt{2}} (1+\Oc_{r\to \infty}(r^{2-\sqrt{8}}))(-4 K_2 r^{\sqrt{8}}(1+\Oc(r^{-2})))\\
& -\frac{K_2 (2i-1)}{16(\sqrt{2}+1)}r^{2-2\sqrt{2}} (1+\Oc_{r\to \infty}(r^{2-\sqrt{8}}))(-4 C_2 r^{-\sqrt{8}})(1+\Oc(r^{\sqrt{8}-4})) \\
&= (i-\frac 12)r^2\ln r+ (\frac 54i-\frac 58)r^2+\Oc(r^{4-\sqrt{8}}).
\end{align*}
This shows the desired identities \eqref{eigenfunction:id:asymptoticT42iinfty}.

The same computation also works for $(\hat {\mathsf T}_{4,2},\hat {\mathsf V}_{4,2})$. Indeed, by \eqref{def:Sigma4khat}, \eqref{eigenfunction:id:asymptotichatT2infty} and \eqref{eigenfunction:id:asymptoticT22infty} we have
$$\hat \Sigma_{4,2} = \frac{-16}{r^2} + \Oc_{r \to \infty}\Big(r^{-2\sqrt{2}}\Big).$$
Notice that this is exactly as $ \Sigma_{4,2}^{(i)}$ given by \eqref{eigenfunctions:id:Sigma42-expansion}, up to a constant factor. This concludes the proof of \ref{lem:T42hatT42}.
\end{proof}

\begin{lemma} \label{lem:T44}
Let  $\hat T_{2}$ and $\mathsf T_{2,2}$ be described in Lemmas \ref{lem:T20hatT2} and \ref{lem:T22}. The system \eqref{sys:Tkk31} with $k= 4$ admits a smooth solution $(\mathsf T_{4,4}, \mathsf V_{4,4})$ such that 
\begin{align}
\label{eigenfunction:id:asymptoticT44infty}
\mathsf T_{4,4}(r) &=   -\frac12  + \Oc_{r \to \infty}(r^{4-\sqrt{20}}), \\
\label{eigenfunction:id:asymptoticV44infty} \mathsf V_{4,4}(r)&= -\frac{1}{24}r^2+ \Oc_{r \to \infty}(r^{6-\sqrt{20}}).
\end{align}
The estimates propagate to higher order derivatives.
\end{lemma}

\begin{proof} The proof is the same as for Lemma \ref{lem:T42hatT42} by applying Lemma \ref{lemm:sol_inhol} with $\ell = 4$ and $f = \Sigma_{4,4}$, where we compute from \eqref{def:Sigma44} and Lemma \ref{lem:T22},
\begin{equation} \label{bd:Sigma44-T44}
\Sigma_{4,4}(r) = \frac 14 r\partial_r \mathsf T_{2,2}-\frac 12 \mathsf T_{2,2} +\frac{1}{16} r^3 \pa_r \Lambda U = \frac{8}{r^2} + \Oc_{r \to \infty}(r^{-\sqrt{8}}). 
\end{equation}
From Lemma \ref{lemm:sol_inhol}, we choose the following solution to \eqref{sys:Tkk31},
\begin{align*}
(\mathsf T_{4,4}, \mathsf V_{4,4})=\sum_{k=1}^4 \gamma_{4,k}(h_{4,k},g_{4,k}),
\end{align*}
where we choose the following integration constants
\begin{align*}
\gamma_{4,1} =\frac{1}{15\cdot 2^{14}}\int_r^\infty \Sigma_{4,4}(s) W_{4,1}(s) s^2(1 + s^2)^2 ds, \quad \gamma_{4,2} =\frac{1}{15\cdot 2^{14}}\int_0^r \Sigma_{4,4}(s) W_{4,2}(s) s^2(1 + s^2)^2 ds,\\
\gamma_{4,3} =\frac{1}{15\cdot 2^{14}}\int_r^\infty \Sigma_{4,4}(s) W_{4,3}(s) s^2(1 + s^2)^2 ds, \quad \gamma_{4,4} =\frac{1}{15\cdot 2^{14}}\int_0^r \Sigma_{4,4}(s) W_{4,4}(s) s^2(1 + s^2)^2 ds.
\end{align*}
To obtain the expansion as $r\to \infty$, we recall that from \eqref{eq:asyDlkrinf},
\begin{align}
\label{bd:W41-T44}W_{4,1} &= -2^9 5 r^{-7} \big(1 + \Oc(r^{4-\sqrt{20}}\big), \quad W_{4,2} = -3\cdot 2^9  r \Big( 1  +\Oc(r^{4-\sqrt{20}}) \Big),\\
\label{bd:W43-T44} W_{4,3}&= - 15\cdot 2^7 C_4 r^{-2\sqrt{5} - 3}\Big(1+ \Oc(r^{-4-\sqrt{20}})\big), \quad W_{4,4} = -15\cdot 2^7 K_4r^{2\sqrt{5} - 3} \Big(  1 + \Oc(r^{4-\sqrt{20}})\big).
\end{align}
We compute thanks to \eqref{bd:Sigma44-T44}, \eqref{bd:W41-T44} and \eqref{bd:W43-T44}:
\begin{align*}
&\gamma_{4,1} =\frac{1}{15\cdot 2^{14}}\int_r^\infty -5\cdot 2^{12} s^{-3}(1+\Oc(s^{4-\sqrt{20}}))ds=-\frac{1}{24r^2}(1+\Oc(r^{4-\sqrt{20}})),\\
& \gamma_{4,2} =\frac{1}{15\cdot 2^{14}}\int_0^r -3\cdot 2^{12} s^{5}(1+\Oc(s^{4-\sqrt{20}}))ds=-\frac{r^6}{120}(1+\Oc(r^{4-\sqrt{20}})),\\
& \gamma_{4,3} =\frac{1}{15\cdot 2^{14}}\int_r^\infty -15\cdot 2^{10}C_4 s^{-\sqrt{20}+1}(1+\Oc(s^{4-\sqrt{20}}))ds=-\frac{C_4}{32(\sqrt{5}-1)}r^{2-\sqrt{20}}(1+\Oc(r^{4-\sqrt{20}})),\\
& \gamma_{4,4} =\frac{1}{15\cdot 2^{14}}\int_0^r -15\cdot 2^{10}K_4 s^{\sqrt{20}+1}(1+\Oc(s^{4-\sqrt{20}}))ds=-\frac{K_4}{32(\sqrt{5}+1)}r^{2+\sqrt{20}}(1+\Oc(r^{4-\sqrt{20}})),\\
\end{align*}

From the definition of $(h_{4,k}, g_{4,k})$ given in Lemma \ref{lemm:SolHom}, we obtain 
\begin{align*}
\mathsf T_{4,4}  & = \sum_{k = 1}^4 \gamma_{4,k} h_{4,k}  \\
& = \Oc(r^{-2})\Oc(1)+\Oc(r^6)\Oc(r^{-8})-\frac{C_4}{32(\sqrt{5}-1)}r^{2-\sqrt{20}}(1+\Oc(r^{4-\sqrt{20}})) 16K_4 r^{\sqrt{20}-2}(1+\Oc(r^{-2})),\\
& -\frac{K_4}{32(\sqrt{5}+1)}r^{2+\sqrt{20}}(1+\Oc(r^{4-\sqrt{20}})) 16C_4 r^{-\sqrt{20}-2}(1+\Oc(r^{4-\sqrt{20}})),\\
& =-\frac 12
\end{align*}
where we used $C_4 K_4 = \frac{2}{\sqrt{}5}$. We get similarly that
\begin{align*}
\mathsf V_{4,4} & = \sum_{k = 1}^4 \gamma_{4,k} g_{4,k} \\
& = -\frac{1}{24r^2}(1+\Oc(r^{4-\sqrt{20}})) (3r^4+\Oc(r^{2}))-\frac{r^6}{120}(1+\Oc(r^{4-\sqrt{20}}))(\frac{5}{r^4}+\Oc(r^{-6}))-\frac{C_4}{32(\sqrt{5}-1)}r^{2-\sqrt{20}}(1+\Oc(r^{4-\sqrt{20}}))(- 4K_4 r^{\sqrt{20}})(1+\Oc(r^{4-2\sqrt{20}})),\\
& -\frac{K_4}{32(\sqrt{5}+1)}r^{2+\sqrt{20}}(1+\Oc(r^{4-\sqrt{20}})) (-4 C_4 r^{-\sqrt{20}})(1+\Oc(r^{\sqrt{20}-6})),\\
& =-\frac{1}{24}r^2+\Oc(r^{6-\sqrt{20}}).
\end{align*}
This concludes the proof of Lemma \ref{lem:T44}. 
\end{proof}

\subsubsection{Computation of the subleading profiles}

\begin{proposition} \label{pr:solution-system-SW}

There exist smooth solutions to \eqref{sys:S111}-\eqref{sys:S402'} such that as $r\to \infty$:
\begin{align}
\label{bd:estimate-S301-W301}& |\partial_r^k \mathsf S_{3,0,1}(r)|\lesssim r^{-3}, \qquad  \mathsf W_{3,0,1}(r)=(\frac 12 -2i+2\tilde \lambda_i)r+\Oc_{r\to \infty}(r^{-1}\ln r)\\
\label{bd:estimate-S402'-W402'}& |\partial_r^k \mathsf S_{4,0,2}'(r)|\lesssim r^{-2-k}, \qquad  \mathsf W_{4,0,2}(r)=r^2+\Oc_{r\to \infty}(1)
\end{align}
and
\begin{align}
\label{bd:estimate-S111-W111}& |\partial_r^k \mathsf S_{1,1,1}(r)|\lesssim r^{-3-k}, \qquad \qquad \qquad |\partial_r^k \mathsf W_{1,1,1}(r)|\lesssim r^{-1-k}\ln r \\
\label{bd:estimate-S210-W210}& |\partial_r^k \mathsf S_{2,1,0}(r)|\lesssim (1+|L_i|)r^{-2-k}, \qquad |\mathsf W_{2,1,0}(r)|\lesssim (1+|L_i|)|\ln r|^{2-\delta_{k\geq 1}}r^{-k},\\
\label{bd:estimate-S212-W212}& |\partial_r^k \mathsf S_{2,1,2}(r)|\lesssim (1+|L_i|) r^{-2-k}, \qquad  |\partial_r^k \mathsf W_{2,1,2}(r)|\lesssim  (1+|L_i|)r^{-k},\\
\label{bd:estimate-S220-W220}& |\partial_r^k \mathsf S_{2,2,0}(r)|\lesssim r^{-2-k}, \qquad \qquad\qquad | \mathsf W_{2,2,0}(r)|\lesssim |\ln r|^{2-\delta_{k\geq 1}}r^{-k}, \\
\label{bd:estimate-S222-W222}& |\partial_r^k \mathsf S_{2,2,2}(r)|\lesssim r^{-2-k}, \qquad \qquad \qquad|\partial_r^k \mathsf W_{2,2,2}(r)|\lesssim r^{-k} \\
\label{bd:estimate-S303-W303}& |\partial_r^k \mathsf S_{3,0,3}(r)|\lesssim r^{-3-k}, \qquad \qquad \qquad |\partial_r^k \mathsf W_{3,0,3}(r)|\lesssim r^{-1-k}, \\
\label{bd:estimate-S311-W311}& |\partial_r^k \mathsf S_{3,1,1}(r)|\lesssim (1+|L_i|) r^{-1-k} , \qquad |\partial_r^k \mathsf W_{3,1,1}(r)|\lesssim  (1+|L_i|)r^{1-k}\ln r, \\
\label{bd:estimate-S313-W313}& |\partial_r^k \mathsf S_{3,1,3}(r)|\lesssim (1+|L_i|) r^{-1-k}, \qquad |\partial_r^k \mathsf W_{3,1,3}(r)|\lesssim (1+|L_i|) r^{1-k}, \\
\label{bd:estimate-S400-W400}& |\partial_r^k \mathsf S_{4,0,0}(r)|\lesssim (1+|L_i|) r^{-2-k}, \qquad | \mathsf W_{4,0,0}(r)|\lesssim (1+|L_i|) r^{-k}\ln^{2-\delta_{k\geq 1}} r,\\
\label{bd:estimate-S402-W402}& |\partial_r^k \mathsf S_{4,0,2}(r)|\lesssim(1+|L_i|)  r^{-2-k}, \qquad  |\partial_r^k \mathsf W_{4,0,2}(r)|\lesssim (1+|L_i|) r^{-k},\\
\label{bd:estimate-S404-W404}& |\partial_r^k \mathsf S_{4,0,4}(r)|\lesssim r^{-2-k}, \qquad \qquad\qquad  |\partial_r^k \mathsf W_{4,0,4}(r)|\lesssim r^{-k}.
\end{align}

\end{proposition}

\begin{proof}

To ease the presentation, we prove all bounds below under the assumption that $|L_i^\inn|\lesssim 1$. The dependance on $L_i^\inn$ in the final identities \eqref{bd:estimate-S111-W111}-\eqref{bd:estimate-S404-W404} then follows directly from the affine dependance on $L_i^\inn$ in the equations \eqref{sys:S111}-\eqref{sys:S402'}.

To lighten notation, the notation $\Oc_{r\to \infty}()$ will again implicitly include estimates for derivatives. Namely, for $f\in C^\infty((0,\infty),\mathbb R)$ and $g:(0,\infty)\rightarrow (0,\infty)$ we will write
$$
f(r)=\Oc_{r\to \infty}(g(r))
$$
if for all $k\in \mathbb N$, we have for all $r\geq 1$:
$$
|\partial_r^k f(r)|\lesssim r^{-k}g(r).
$$

\medskip

\noindent \textbf{Step 1}. \emph{Computation of the first spherical harmonics profiles}.

\smallskip

\noindent \underline{Determination of $(\mathsf S_{1,1,1}, \mathsf W_{1,1,1})$}. We notice that by \eqref{def:Sigma3k} we have $-|a|\partial_r \Lambda U=-4\sqrt{2}|a|\Sigma_{3,1}$. By \eqref{sys:S111}, \eqref{sys:Tkk31} and Lemma \ref{lem:T3k} we thus choose $(\mathsf S_{1,1,1}, \mathsf S_{1,1,1})=(\mathsf T_{3,1},\mathsf W_{3,1})$. The desired estimates \eqref{bd:estimate-S111-W111} then follow from \eqref{eigenfunction:id:asymptoticT31infty}.

\smallskip 

\noindent \underline{Determination of $(\mathsf S_{3,0,1}, \mathsf W_{3,0,1})$}. We first notice that the horizontal momentum of the source term is non-zero by \eqref{eigenfunctions:id:G1} and by integrating by parts:
$$
\int_{\mathbb R^2} y_1 \left((\Lambda_i+2\tilde \lambda_i)\partial_{x_1}U+\nabla . (\partial_{x_1}UG_1) \right)dy=4(4i-1-4\tilde \lambda_i)\pi.
$$
In polar coordinates, this corresponds to
\begin{equation} \label{eigenfunctions:id:Sigma301-momentum}
\int_0^\infty \Sigma_{3,0,1}(r)r^2 dr=4(4i-1-4\tilde \lambda_i).
\end{equation}
We next estimate by \eqref{eigenfunctions:id:G1} 
\begin{equation} \label{eigenfunctions:bd:Sigma301}
\Sigma_{3,0,1}(r)=\Oc_{r\to 0}(r), \qquad \Sigma_{3,0,1}(r)=\Oc_{r\to \infty}(r^{-6}).
\end{equation}
We deduce by \eqref{eigenfunctions:id:Sigma301-momentum} and \eqref{id:W12-W1} that
\begin{equation} \label{eigenfunctions:id:Sigma301-momentum-2}
\int_0^\infty \frac{\Sigma_{3,0,1}W_{1,2}}{W_1} dr= 2i-\frac 12-2\tilde \lambda_i.
\end{equation}
We apply Lemma \ref{lemm:sol_inhol} with $\ell = 1$ and $f = \Sigma_{3,1,1}$ to solve \eqref{sys:S311} and choose the following solution
\begin{equation} \label{eigenfunctions:id:S301-expression}
(\mathsf S_{3,1,1}, \mathsf W_{3,1,1})=\sum_{k=1}^4 \gamma_{3,1,1,k}(h_{1,k},g_{1,k})
\end{equation}
where, using \eqref{eigenfunctions:bd:Sigma301} and \eqref{eigenfunctions:id:Sigma301-momentum-2}:
\begin{align*}
& \gamma_{3,0,1,1}= - \int_r^\infty  \frac{\Sigma_{3,0,1}W_{1,1}}{W_1}ds=\left\{
\begin{array}{l l}
= \Oc_{r\to 0}( \int_0^\infty \langle s\rangle^{-2} ds) = \Oc_{r\to 0}(r^2),\\
= \Oc_{r\to \infty}( \int_0^r s^{-2} ds) = \Oc_{r\to \infty}(r^{-1}),
\end{array} \right.\\
&  \gamma_{3,0,1,2}=-\int_0^r \frac{\Sigma_{3,0,1}W_{1,2}}{W_1}ds =\left\{
\begin{array}{l l}
= \Oc_{r\to 0}( \int_0^r s^3 ds) = \Oc_{r\to 0}(r^4),\\
= \frac 12-2i+2\tilde \lambda_i+ \Oc_{r\to \infty}( \int_r^\infty s^{-4} ds) = \frac 12-2i+2\tilde \lambda_i+\Oc_{r\to \infty}(r^{-3}),
\end{array} \right.\\
& \gamma_{3,0,1,3}=-\int_r^\infty \frac{\Sigma_{3,0,1}W_{1,3}}{W_1}ds=\left\{
\begin{array}{l l}
= \Oc_{r\to 0}( \int_0^\infty s\langle s \rangle^{-1-\sqrt{5}-3} ds) = \Oc_{r\to 0}(1),\\
= \Oc_{r\to \infty}( \int_r^\infty s^{-\sqrt{5}-3} ds) = \Oc_{r\to \infty}(r^{-2-\sqrt{5}}),
\end{array} \right.\\
& \gamma_{3,0,1,4}=-\int_0^r \frac{\Sigma_{3,1,1}W_{1,4}}{W_1}ds=\left\{
\begin{array}{l l}
= \Oc_{r\to 0}( \int_0^r s^3 ds) = \Oc_{r\to 0}(r^4),\\
= \Oc_{r\to \infty}( \int_0^r s^{\sqrt{5}-3} ds) = \Oc_{r\to \infty}(r^{\sqrt{5}-2}),
\end{array} \right.
\end{align*}
Injecting the asymptotic expansions of Lemma \ref{lemm:sol_inhol} in \eqref{eigenfunctions:id:S301-expression} we obtain as $r\to \infty$
\begin{align*}
&\mathsf S_{3,0,1}=  \Oc(r^{-1}) \Oc(r^{-5})+ \Oc(1) \Oc(r^{-3})+ \Oc(r^{-2-\sqrt{5}}) \Oc(r^{\sqrt{5}-2})+ \Oc(r^{\sqrt{5}-2}) \Oc(r^{-\sqrt{5}-2})=\Oc_{r\to \infty}(r^{-3})\\
&\mathsf W_{3,0,1}  =  \Oc(r^{-1}) \Oc(r^{-1})+( \frac 12-2i+2\tilde \lambda_i)(r+ \Oc(r^{-1}\ln r))+ \Oc(r^{-2-\sqrt{5}}) \Oc(r^{\sqrt{5}})+ \Oc(r^{\sqrt{5}-2}) \Oc(r^{-\sqrt{5}})\\
&= ( \frac 12-2i+2\tilde \lambda_i)r+\Oc(r^{-1}\ln r).
\end{align*}
This is \eqref{bd:estimate-S301-W301}.

\smallskip

\noindent \underline{Estimate for $(\mathsf S_{3,1,1}, \mathsf W_{3,1,1})$}. We first compute by \eqref{sys:S311} and  \eqref{bd:estimate-S111-W111}, \eqref{eigenfunction:id:asymptoticT20infty0}, \eqref{eigenfunction:id:asymptoticT20infty}, \eqref{eigenfunction:id:asymptotichatT2infty}, \eqref{eigenfunction:id:asymptoticT220}, \eqref{eigenfunction:id:asymptoticT22infty} and $|\tilde \lambda_i|\lesssim |\ln \nu|^{-1}$ that
\begin{equation} \label{bd:Sigma311-asymptotics}
|\Sigma_{3,1,1}(r)|=\Oc_{r\to 0}(r), \qquad  |\Sigma_{3,1,1}(r)|=\Oc_{r\to \infty}( r^{-3})
\end{equation}
We apply Lemma \ref{lemm:sol_inhol} with $\ell = 1$ and $f = \Sigma_{3,1,1}$ to solve \eqref{sys:S311} and choose the following solution
$$
(\mathsf S_{3,1,1}, \mathsf W_{3,1,1})=\sum_{k=1}^4 \gamma_{3,1,1,k}(h_{1,k},g_{1,k})
$$
where
\begin{align*}
& \gamma_{3,1,1,1}=\int_0^r \frac{\Sigma_{3,1,1}W_{1,1}}{W_1}ds=\left\{
\begin{array}{l l}
= \Oc_{r\to 0}( \int_0^r s ds) = \Oc_{r\to 0}(r^2),\\
= \Oc_{r\to \infty}( \int_0^r s ds) = \Oc_{r\to \infty}(r^2),
\end{array} \right.\\
&  \gamma_{3,1,1,2}=-\int_0^r \frac{\Sigma_{3,1,1}W_{1,2}}{W_1}ds =\left\{
\begin{array}{l l}
= \Oc_{r\to 0}( \int_0^r s^3 ds) = \Oc_{r\to 0}(r^4),\\
= \Oc_{r\to \infty}( \int_0^r s^3 \langle s \rangle^{-4} ds) = \Oc_{r\to \infty}(\ln r),
\end{array} \right.\\
& \gamma_{3,1,1,3}=-\int_r^\infty \frac{\Sigma_{3,1,1}W_{1,3}}{W_1}ds=\left\{
\begin{array}{l l}
= \Oc_{r\to 0}( \int_0^\infty s\langle s \rangle^{-1-\sqrt{5}} ds) = \Oc_{r\to 0}(1),\\
= \Oc_{r\to \infty}( \int_r^\infty s^{-\sqrt{5}} ds) = \Oc_{r\to \infty}(r^{1-\sqrt{5}}),
\end{array} \right.\\
& \gamma_{3,1,1,4}=-\int_0^r \frac{\Sigma_{3,1,1}W_{1,4}}{W_1}ds=\left\{
\begin{array}{l l}
= \Oc_{r\to 0}( \int_0^r s^3 ds) = \Oc_{r\to 0}(r^4),\\
= \Oc_{r\to \infty}( \int_0^r s^{\sqrt{5}} ds) = \Oc_{r\to \infty}(r^{1+\sqrt{5}}),
\end{array} \right.
\end{align*}
where we used \eqref{bd:Sigma311-asymptotics} and the bounds of Lemma \ref{lemm:sol_inhol} to obtain the above estimates. The desired estimate \eqref{bd:estimate-S311-W311} for $k=0$ then follows from the above estimates and that of Lemma \ref{lemm:SolHom}. They are propagated to higher order derivatives $k\geq 1$ by the same method.

\medskip

\noindent \textbf{Step 2}. \emph{Computation of the radial profiles}. We first notice the mass cancellation for all source terms with radial components
\begin{align*}
& 0=\int_{\mathbb R^2} -(\Lambda_i+2\tilde \lambda_i)\Lambda U  -\alpha^{\frac 12}L_i^\inn a.\nabla \partial_{x_1}U\\
&\qquad =\int_{\mathbb R^2} -a.\nabla \mathcal S_{1,1}=\int_{\mathbb R^2} -\frac{1}{2|a|^4}\Lambda U+\alpha^2L_i^\inn \nabla .(\partial_{x_1}UG_2)
\end{align*}
which in polar coordinates corresponds to
\begin{equation} \label{id-mass-cancellation-Sigma220}
0=\int_0^\infty r \Sigma_{2,1,0}(r)dr =\int_0^\infty r \Sigma_{2,2,0}(r)dr= \int_0^\infty r \Sigma_{4,0,0}(r)dr .
\end{equation}
By \eqref{sys:S210}, \eqref{sys:S220} and \eqref{sys:S400}, and \eqref{bd:estimate-S111-W111} we have for $k\in \mathbb N$ and $r\geq 1$,
\begin{equation} \label{bd-Sigma220}
|\partial_r^k \Sigma_{2,1,0}(r)|+|\partial_r^k \Sigma_{2,2,0}(r)|+|\partial_r^k \Sigma_{4,0,0}(r)|\lesssim r^{-4-k}.
\end{equation}
In view of \eqref{id-mass-cancellation-Sigma220} and \eqref{bd-Sigma220}, we will only compute $(\mathsf S_{2,2,0},\mathsf W_{2,2,0})$, and the proof will apply simultaneously to $(\mathsf S_{2,1,0},\mathsf W_{2,1,0})$ and $(\mathsf S_{4,0,0},\mathsf W_{4,0,0})$. Since the system \eqref{sys:S220} corresponds to the radial case, in can be solved by appealing to Lemma \ref{lemm:invA0}. Introducing $m_{2,2,0}(r)=\int_0^r \mathcal S_{2,2,0}(r')r'dr'$ and $\sigma_{2,2,0}(r)=\int_0^r \Sigma_{2,2,0}(r')r'dr'$ we choose
\begin{equation} \label{id-m220}
m_{2,2,0}=\frac 12 \psi_0 \int_r^1 \frac{\zeta^4+4\zeta^2\ln \zeta-1}{\zeta}\sigma_{2,2,0}(\zeta)d\zeta+\frac 12 \tilde \psi_0 \int_0^r \zeta \sigma_{2,2,0}(\zeta)d\zeta.
\end{equation}
By \eqref{systeme-odes:fundamental-solution-radial}, the function $\mathsf S_{2,2,0}(r)=r^{-1}\partial_r m_{2,2,0}(r)$ is then smooth at the origin. We have by \eqref{id-mass-cancellation-Sigma220} and \eqref{bd-Sigma220} that for $k\in \mathbb N$, for any $r>1$,
\begin{equation} \label{bd-sigma220}
|\partial_r^k \sigma_{2,2,0}(r)|=|\partial_r^k \int_r^\infty \Sigma_{2,2,0}(r')dr' |\lesssim r^{-2-k}.
\end{equation}
We have by \eqref{id-m220}
$$
|\mathsf S_{2,2,0}(r)|=\left|\frac{1}{2r}\partial_r \psi_0 \int_r^1 \frac{\zeta^4+4\zeta^2\ln \zeta-1}{\zeta}\sigma_{2,2,0}(\zeta)d\zeta+\frac{1}{2r}\partial_r \tilde \psi_0 \int_0^r \zeta \sigma_{2,2,0}(\zeta)d\zeta \right|\lesssim r^{-2}.
$$
This extends by the same computation to the case $k\geq 1$, which shows the first inequality in \eqref{bd:estimate-S220-W220}. Using $\mathsf W_{2,2,0}(r)=C-\int_0^r \frac{m_{2,2,0}(r')}{r'}dr'$ for some $C\in \mathbb R$ and $|\partial_r^k m_{2,2,0}(r)|\lesssim r^{-k}\ln r$ from \eqref{id-m220} and \eqref{bd-sigma220}, we obtain the second inequality in \eqref{bd:estimate-S220-W220}. As we mentioned above, this proof shows simultaneously \eqref{bd:estimate-S210-W210} and \eqref{bd:estimate-S400-W400}.

\medskip

\noindent \textbf{Step 3}. \emph{Computation of second spherical harmonics profiles}. First, from Lemma \ref{lemm:SolHom} we simply define $(\mathcal S_{4,0,2}',\mathcal W_{4,0,2}')=(h_{2,1},g_{2,1})$ which immediately implies the desired behaviour \eqref{bd:estimate-S402'-W402'}.

We now compute $(\mathcal S_{2,1,2},\mathcal W_{2,1,2})$, $(\mathcal S_{2,2,2},\mathcal W_{2,2,2})$ and $(\mathcal S_{4,0,2},\mathcal W_{4,0,2})$. We have by \eqref{sys:S212}, \eqref{sys:S222}, \eqref{sys:S402} and \eqref{bd:estimate-S111-W111} 
\begin{equation} \label{bd:Sigma222-asymptotics}
\Sigma_{2,1,2}(r),\Sigma_{2,2,2}(r),\Sigma_{4,0,2}(r)=\Oc_{r\to 0}(r^2),
\end{equation}
where we used for $\Sigma_{2,2,2}$ that near the origin $\mathsf S_{1,1}(r)=cr+O(r^3)$ for some constants $c\in \mathbb R$ as $\mathcal S_{1,1}$ is smooth near the origin. We also have
\begin{equation} \label{bd:Sigma222-asymptotics2}
\Sigma_{2,1,2}(r)=\Oc_{r\to \infty}(r^{-6})=\Oc_{r\to \infty}(r^{-4}), \qquad \Sigma_{2,2,2}(r)=\Oc_{r\to \infty}(r^{-4}), \qquad \Sigma_{4,0,2}(r)=\Oc_{r\to \infty}(r^{-4}).
\end{equation}
Because of \eqref{bd:Sigma222-asymptotics} and \eqref{bd:Sigma222-asymptotics2} we will only compute $(\mathsf S_{2,2,2},\mathsf W_{2,2,2})$, and the proof will apply simultaneously to $(\mathsf S_{2,1,2},\mathsf W_{2,1,0})$ and $(\mathsf S_{4,0,2},\mathsf W_{4,0,2})$. We could show more decay as $r\to \infty$ for $(\mathsf S_{2,1,2},\mathsf W_{2,1,0})$ thanks to the extra $r^{-2}$ decay in $\Sigma_{2,1,2}$ but the decay we prove is sufficient.

We solve \eqref{sys:S222} thanks to Lemma \ref{lemm:sol_inhol} with $\ell =2$ and $f = \Sigma_{2,2,2}$ and get a solution of the form
$$
(\mathsf S_{2,2,2}, \mathsf W_{2,2,2})=\sum_{k=1}^4 \gamma_{2,2,2,k}(h_{2,k},g_{2,k})
$$
where we choose the integration constants so that
\begin{align*}
& \gamma_{2,2,2,1}=- \int_r^\infty \frac{\Sigma_{2,2,2}W_{2,1}}{W_2}ds=\left\{
\begin{array}{l l}
= \Oc_{r\to 0}( \int_0^\infty \langle s\rangle^{-3} ds) = \Oc_{r\to 0}(1),\\
= \Oc_{r\to \infty}( \int_r^\infty s^{-3} ds) = \Oc_{r\to \infty}(r^{-2}),
\end{array} \right.\\
&  \gamma_{2,2,2,2}=-\int_0^r \frac{\Sigma_{2,2,2}W_{2,2}}{W_2}ds =\left\{
\begin{array}{l l}
= \Oc_{r\to 0}( \int_0^r s^5 ds) = \Oc_{r\to 0}(r^6),\\
= \Oc_{r\to \infty}( \int_0^r s \langle s \rangle^{-4} ds) = \Oc_{r\to \infty}(r^2),
\end{array} \right.\\
& \gamma_{2,2,2,3}=-\int_r^\infty \frac{\Sigma_{2,2,2}W_{2,3}}{W_2}ds=\left\{
\begin{array}{l l}
= \Oc_{r\to 0}( \int_0^\infty \langle s \rangle^{-1-\sqrt{8}} ds) = \Oc_{r\to 0}(1),\\
= \Oc_{r\to \infty}( \int_r^\infty s^{-1-\sqrt{8}} ds) = \Oc_{r\to \infty}(r^{-\sqrt{8}}),
\end{array} \right.\\
& \gamma_{2,2,2,4}=-\int_0^r \frac{\Sigma_{2,2,2}W_{2,4}}{W_2}ds=\left\{
\begin{array}{l l}
= \Oc_{r\to 0}( \int_0^r s^5 ds) = \Oc_{r\to 0}(r^6),\\
= \Oc_{r\to \infty}( \int_0^r s^{\sqrt{8}-1} ds) = \Oc_{r\to \infty}(r^{\sqrt{8}}),
\end{array} \right.
\end{align*}
and where the estimates follow from \eqref{bd:Sigma313-asymptotics} and the bounds of Lemma \ref{lemm:sol_inhol}. The estimates \eqref{bd:estimate-S313-W313} for $k=0$ then follow from the above estimates and that of Lemma \ref{lemm:SolHom}, and are propagated for $k\geq 1$ by the same computations. As we explained, this proves also simultaneously \eqref{bd:estimate-S212-W212} and \eqref{bd:estimate-S402-W402}.

\medskip

\noindent \textbf{Step 4}. \emph{Computation of third spherical harmonics profiles}. We estimate by \eqref{sys:S303}, \eqref{sys:S313}, \eqref{eigenfunction:id:asymptoticT22infty} and \eqref{bd:estimate-S111-W111}
\begin{equation} \label{bd:Sigma313-asymptotics}
\Sigma_{3,0,3}(r),\Sigma_{3,1,3}(r)=\Oc_{r\to 0}(r^3), \qquad  \Sigma_{3,0,3}(r)=\Oc_{r\to \infty}(r^{-5}), \qquad  \Sigma_{3,1,3}(r)=\Oc_{r\to \infty}(r^{-3})
\end{equation}
where we used that since $\mathcal S_{1,1}$ and $T_{2,2}$ are smooth near the origin then in polar coordinates $\mathsf S_{1,1}(r)=cr+O_{r\to 0}(r^3)$ and $\mathsf S_{2,2}(r)=c'r^2+O_{r\to 0}(r^4)$ for some constants $c,c'\in \mathbb R$ and that these estimates propagate for higher order derivative. We now solve \eqref{sys:S303} and \eqref{sys:S313} simultaneously. Let $j\in \{0,1\}$. Then by \eqref{bd:Sigma313-asymptotics} we have $\Sigma_{3,j,3}=\Oc_{r\to \infty}(r^{2j-5})$. Applying Lemma \ref{lemm:sol_inhol} with $\ell =3$ and $f_j = \Sigma_{3,j,3}$ to solve \eqref{sys:S303} and \eqref{sys:S313} we get the following solutions
$$
(\mathsf S_{3,j,3}, \mathsf W_{3,j,3})=\sum_{k=1}^4 \gamma_{3,j,3,k} (h_{3,k},g_{3,k})
$$
where we choose
\begin{align*}
& \gamma_{3,j,3,1}=- \int_r^\infty \frac{\Sigma_{3,j,1}W_{3,1}}{W_3}ds=\left\{
\begin{array}{l l}
= \Oc_{r\to 0}( \int_0^\infty \langle s\rangle^{2j-5} ds) = \Oc_{r\to 0}(1),\\
= \Oc_{r\to \infty}( \int_r^\infty s^{2j-5} ds) = \Oc_{r\to \infty}(r^{2j-4}),
\end{array} \right.\\
&  \gamma_{3,j,3,2}=-\int_0^r \frac{\Sigma_{3,j,1}W_{3,2}}{W_3}ds =\left\{
\begin{array}{l l}
= \Oc_{r\to 0}( \int_0^r s^7 ds) = \Oc_{r\to 0}(r^8),\\
= \Oc_{r\to \infty}( \int_0^r s^{2j+1}  ds) = \Oc_{r\to \infty}(r^{2j+2}),
\end{array} \right.\\
& \gamma_{3,1,3,3}=-\int_r^\infty \frac{\Sigma_{3,1,1}W_{3,3}}{W_3}ds=\left\{
\begin{array}{l l}
= \Oc_{r\to 0}( \int_0^\infty \langle s \rangle^{-\sqrt{13}+2j-2} ds) = \Oc_{r\to 0}(1),\\
= \Oc_{r\to \infty}( \int_r^\infty s^{-\sqrt{13}+2j-2} ds) = \Oc_{r\to \infty}(r^{2j-1-\sqrt{13}}),
\end{array} \right.\\
& \gamma_{3,1,3,4}=-\int_0^r \frac{\Sigma_{3,1,1}W_{3,4}}{W_3}ds=\left\{
\begin{array}{l l}
= \Oc_{r\to 0}( \int_0^r s^7 ds) = \Oc_{r\to 0}(r^8),\\
= \Oc_{r\to \infty}( \int_0^r s^{\sqrt{13}+2j-2} ds) = \Oc_{r\to \infty}(r^{2j-1+\sqrt{13}}),
\end{array} \right.
\end{align*}
and where the estimates follow from \eqref{bd:Sigma313-asymptotics} and the bounds of Lemma \ref{lemm:sol_inhol}. The estimates \eqref{bd:estimate-S303-W303} and \eqref{bd:estimate-S313-W313} for $k=0$ then follow from the above estimates and that of Lemma \ref{lemm:SolHom}, and are propagated for $k\geq 1$ by the same computations.

\medskip

\noindent \textbf{Step 5}. \emph{Computation of the fourth spherical harmonics profile}. We estimate by \eqref{sys:S404}:
\begin{equation} \label{bd:Sigma404-asymptotics}
\Sigma_{4,0,4}(r)=\Oc_{r\to 0}(r^4), \qquad  \Sigma_{4,0,4}(r)=\Oc_{r\to \infty}(r^{-4}).
\end{equation}
Applying Lemma \ref{lemm:sol_inhol} with $\ell =4$ and $f= \Sigma_{4,0,4}$ to solve \eqref{sys:S404} we get the following solution
$$
(\mathsf S_{4,0,4}, \mathsf W_{4,0,4})=\sum_{k=1}^4 \gamma_{4,0,4,k} (h_{4,k},g_{4,k})
$$
where we choose
\begin{align*}
& \gamma_{4,0,4,1}=- \int_r^\infty \frac{\Sigma_{4,0,4}W_{4,1}}{W_4}ds=\left\{
\begin{array}{l l}
= \Oc_{r\to 0}( \int_0^\infty  \langle s\rangle^{-5} ds) = \Oc_{r\to 0}(1),\\
= \Oc_{r\to \infty}( \int_r^\infty s^{-5} ds) = \Oc_{r\to \infty}(r^{-4}),
\end{array} \right.\\
&  \gamma_{4,0,4,2}=-\int_0^r \frac{\Sigma_{3,0,4}W_{4,2}}{W_4}ds =\left\{
\begin{array}{l l}
= \Oc_{r\to 0}( \int_0^r s^9 ds) = \Oc_{r\to 0}(r^{10}),\\
= \Oc_{r\to \infty}( \int_0^r s^{3}  ds) = \Oc_{r\to \infty}(r^{4}),
\end{array} \right.\\
& \gamma_{4,0,4,3}=-\int_r^\infty \frac{\Sigma_{4,0,4}W_{4,3}}{W_4}ds=\left\{
\begin{array}{l l}
= \Oc_{r\to 0}( \int_0^\infty \langle s \rangle^{-\sqrt{20}-1} ds) = \Oc_{r\to 0}(1),\\
= \Oc_{r\to \infty}( \int_r^\infty s^{-\sqrt{20}-1} ds) = \Oc_{r\to \infty}(r^{\sqrt{20}}),
\end{array} \right.\\
& \gamma_{4,0,4,4}=-\int_0^r \frac{\Sigma_{4,0,4}W_{4,4}}{W_4}ds=\left\{
\begin{array}{l l}
= \Oc_{r\to 0}( \int_0^r s^9 ds) = \Oc_{r\to 0}(r^{10}),\\
= \Oc_{r\to \infty}( \int_0^r s^{\sqrt{20}-1} ds) = \Oc_{r\to \infty}(r^{\sqrt{20}}),
\end{array} \right.
\end{align*}
and where the estimates follow from \eqref{bd:Sigma404-asymptotics} and the bounds of Lemma \ref{lemm:sol_inhol}. The estimates \eqref{bd:estimate-S404-W404} and \eqref{bd:estimate-S404-W404} for $k=0$ then follow from the above estimates and that of Lemma \ref{lemm:SolHom}, and are propagated for $k\geq 1$ by the same computations.

\end{proof}

\subsection{Outer expansion of the eigenfunctions} \label{sec:outer-eigenfunction}

In this subsection, we compute an approximate eigenfunction $\phi_i^\out$ in the exterior zone $\cap_{\pm}|z\pm a|\gg \nu$. To do so, we replace in \eqref{eigen:id:Ri} $\Phi_{\chi^* \phi_i^\out}$ by a more general solution $V_i^\out$ of the Poisson equation $-\Delta V_i^\out=\chi^* \phi_i^\out$, so that the approximate exterior eigenfunction system is
\begin{equation} \label{eigenfunctions-exterior:id:eigenfunction-system}
\left\{\begin{array}{l l l}
& R_i^\out = \tilde{\Ls}^z(\phi_i^\out,V_i^\out )-\lambda_i \phi_i^\out, \\
&-\Delta V_i^\out =\chi^* \phi_i^\out.
\end{array}
\right.
\end{equation}
In order to find a solution of \eqref{eigenfunctions-exterior:id:eigenfunction-system} such that $R_i^\out$ is suitably small, we identify the leading order part. We approximate the derivative of the Poisson field generated by $U_{1+2,\nu}$ by $-\sum_\pm 4(z\pm a_\infty)/|z\pm a_\infty|^2$ so that the linearized operator can be decomposed as
\begin{equation} \label{eigenfunctions-exterior:id:decomposition-Lsz}
\tilde{\Ls}^z(\phi_i^\out,V_i^\out )=\Hs_\beta \phi_i^\out+\nabla \phi_i^\out .\left(-\nabla \Phi_{U_{1+2,\nu}}-4\sum_\pm \frac{z\pm a_\infty}{|z\pm a_\infty|^2}\right)-\nabla U_{1+2,\nu}.\nabla V_i^\out+2U_{1+2,\nu}\phi_i^\out 
\end{equation}
where the leading order part is
\begin{equation} \label{def:Hs0}
\Hs_\beta := \Delta  + \left(\frac{4(z - a_\infty)}{|z-a_\infty|^2} + \frac{4(z + a_\infty)}{|z+a_\infty|^2} \right).\nabla - \beta \Lambda.
\end{equation}
Using \eqref{eigenfunctions-exterior:id:decomposition-Lsz} we can rewrite \eqref{eigenfunctions-exterior:id:eigenfunction-system} with the following decomposition
\begin{equation} \label{eigenfunctions-exterior:id:eigenfunction-system-2}
\left\{\begin{array}{l l l l}
R_i^\out &=  \Hs_\beta \phi_i^\out-\lambda_i \phi_i^\out & \}:=\tilde R_i^\out , \\
&\quad +\nabla \phi_i^\out .(-\nabla \Phi_{U_{1+2,\nu}}-4\sum_\pm \frac{z \pm a_\infty}{|z\pm a_\infty|^2})-\nabla U_{1+2,\nu}.\nabla V_i^\out+2U_{1+2,\nu}\phi_i^\out & \}:=\bar R_i^\out \\
-\Delta V_i^\out & =\chi^* \phi_i^\out.
\end{array}
\right.
\end{equation}
where $\tilde R_i^\out$ denotes the remainder in the leading part of the system and $\bar R_i^\out$ gathers negligible terms. We will first solve the leading system obtained by discarding $\bar R_i^\out$ in \eqref{eigenfunctions-exterior:id:eigenfunction-system-2}, which will provide the outer solution $\phi_i^\out$. Then we will estimate the negligible terms in $\bar R_i^\out$. This will produce a solution with the following properties.

\begin{proposition}[Outer expansion of $\phi_i$] \label{pr:phi1out} Let $ i =0,1$ and $\lambda_i=2\beta (i-1)+2\beta \tilde \lambda_i$, and pick any constants $A^\out_i,B^\out_i,C^{\out}_i\in \mathbb R$. There exists a smooth solution $(\phi_i^\out,V_i^\out,\tilde R_i^\out)$ of the exterior eigenfunction system \eqref{eigenfunctions-exterior:id:eigenfunction-system} on $\mathbb R^2\backslash\{-a_\infty,a_\infty\}$ that has the following properties.

\smallskip

\noindent \emph{(i) Asymptotics as $\zeta_\pm=|z_\pm| = |z \mp a_\infty| \to 0$}. Let $\theta_\pm=\frac{z_\pm}{|z_\pm|}$. We have
\begin{align} \label{eigenfunctions-exterior:id:expansion-phiiout-1}
& \phi_i^\out(z) =\frac{1}{\zeta_\pm^4}+\frac{\beta}{\zeta_\pm^2}\left(\frac{1}{2}i+\frac 1 4 \cos(2\theta_\pm)-\frac 1 2 \tilde \lambda_i\right)\mp \frac{\beta^{\frac 32}}{12\sqrt{2}\zeta_\pm} \cos(3\theta_\pm)+\beta^2\tilde \lambda_i \left(\frac i2 -\frac 14\right) \log \zeta_\pm \\
& \nonumber \qquad +\frac{\beta^2}{32} \left(3i+(8i-4)\cos(2\theta_\pm)+\cos (4\theta_\pm)\right) +4\beta^2 A_i+2\beta^2 \tilde \lambda_i \left(-\frac 1 8 \cos(2\theta_\pm)+B_i\right)+\phi_{i,high}^\out(z),\\
 \label{eigenfunctions-exterior:id:expansion-phiiout-1-bd} &|\nabla^k \phi_{i,high}^\out(z)|+|\nabla^k \nu \partial_\nu \phi_{i,high}^\out(z)|\lesssim \zeta^{\sqrt{5}-2-k}_\pm \qquad \mbox{for }k\in \mathbb N,
\end{align}
where $A_i,B_i\in \mathbb R$ with in particular
\begin{equation} \label{eigenfunctions-exterior:id:def-A}
-1\leq A_1\leq-0.9,
\end{equation}
and
\begin{align} 
\nonumber &V_i^\out(z) = -\frac{1}{4\zeta_\pm^2}-\left(\frac{\beta}{4}i-\frac{\beta}{4}\tilde \lambda_i\right)(\log \zeta_\pm)^2+\frac{\beta}{16}\cos(2\theta_\pm)\mp \frac{\beta^{3/2}}{96\sqrt{2}}\zeta_\pm \cos(3\theta_\pm)+\beta^2 \tilde \lambda_i \frac{1-2i}{16}\zeta_\pm^2\log \zeta_\pm \\
\nonumber &\qquad  \qquad +\beta^2 \zeta_\pm^2 \left(-\frac{3i}{128}-A_i+\left(\frac{2i-1}{16}-\frac{B_i}{2}\right)\tilde \lambda_i+\frac{1-2i+2\tilde \lambda_i}{32}\log \zeta_\pm \cos(2\theta_\pm)+\frac{1}{384}\cos(4\theta_\pm) \right)\\
\label{eigenfunctions-exterior:id:expansion-Viiout-1} &\qquad \qquad \pm A_i^\out \frac{1}{\zeta_\pm}\cos(\theta_\pm) +B_i^\out \log \zeta_\pm +C_i^\out \pm D_i^\out \zeta_\pm \cos(\theta_\pm)+E_i^\out \zeta_\pm^2\cos(2\theta_\pm)+V_{i,high}^\out (z) \\
\label{eigenfunctions-exterior:id:expansion-Viiout-2}  &|\nabla^k V_{i,high}^\out (z) |\lesssim \zeta_\pm^{\sqrt{5}-k} \qquad \mbox{for }k\in \mathbb N,
\end{align}
where $D_i^\out =D_{i,0}^\out +D_{i,1}^\out A_i^\out +D_{i,2}^\out B_i^\out $ and $E_i^\out =E_{i,0}^\out+E_{i,1}^\out A_i^\out+E_{i,2}^\out B_i^\out$ for some constants $D_{i,1}^\out=-\frac{\beta}{8}$ and $D_{i,0}^\out ,D_{i,1}^\out ,D_{i,2}^\out ,E_{i,0}^\out,E_{i,1}^\out,E_{i,2}^\out\in \mathbb R  $.

\smallskip

\noindent \emph{(ii) Asymptotics as $|z|\to \infty$}. Assuming $|\partial_\nu \tilde \lambda_i| \lesssim 1$, we have
\begin{align} \label{eigenfunctions-exterior:bd:expansion-phiiout-infty}
|\nabla^k \phi_i^\out(z)|+|\nabla^k\nu\partial_\nu \phi_i^\out(z)|\lesssim |z|^{C_k},  \qquad \partial_a \phi_i^\out=0,
\end{align}
and
\begin{align} \label{eigenfunctions-exterior:bd:expansion-phiiout-Viout-infty}
 |V_i^\out(z)|\lesssim (1+|A_i^\out|+|B_i^\out|+|C_i^\out|) \langle \log \langle z\rangle \rangle^{\delta_{k=0}}\langle z\rangle^{-k} .
\end{align}

\smallskip

\noindent \emph{(iii) Estimate for the remainder}. There exists $\tilde \epsilon_0>0$ such that the error $R_i^\out$ satisfies if $|\tilde \lambda_i|\lesssim |\ln \nu|^{-1}$ and $|a-a_\infty|\lesssim \nu$ and $|A_i^\out|+|B_i^\out|+|C_i^\out|\lesssim |\ln \nu|$:
\begin{equation} \label{eigenfunctions-exterior:bd:Ri}
|\nabla^k R_i^\out(z)|\lesssim \frac{1}{|\log \nu|^2}\sum_\pm \frac{\langle z\rangle^{C_k}}{|z\pm a|^{2+k}} \qquad \mbox{for }z\in \cap_\pm \{|z\pm a|\gtrsim \nu^{\tilde \epsilon_0}\}.
\end{equation}

\end{proposition}

\begin{proof}

\noindent \textbf{Step 1}. \emph{The ansatz for the leading system}. We first solve the leading part of the system  \eqref{eigenfunctions-exterior:id:eigenfunction-system-2}:
\begin{equation} \label{eigenfunctions-exterior:id:eigenfunction-system-leading}
\left\{\begin{array}{l l l l}
\tilde R_i^\out &=  \Hs_\beta \phi_i^\out-\lambda_i \phi_i^\out , \\
-\Delta V_i^\out & =\chi^* \phi_i^\out.
\end{array}
\right.
\end{equation}
As $\lambda_i=2\beta (1-i)+2\beta \tilde \lambda_i$ with $|\tilde \lambda_i|\ll 1$ we look for a solution with the following decomposition
\begin{align} \label{eigenfunctions-exterior:id:decomposition-phiiout}
& \phi_i^\out(z) = \Omega_i(z) + 2\beta \tilde \lambda_i Z_i(z),  \\
& \label{eigenfunctions-exterior:id:decomposition-Viout} V_i^\out (z)= V_{i,\Omega}^\out(z) + 2\beta \tilde \lambda_i V_{i,Z}^\out (z), 
\end{align}
where $\Omega_i$, $Z_i$, $V_{i,\Omega}^\out$ and $V_{i,Z}^\out$ solve
\begin{align} \label{eigenfunctions-exterior:id:equations-Omega-Z}
& \Big(\Hs_\beta - 2\beta(1-i) \Big) \Omega_i = 0,\qquad \Big(\Hs_\beta - 2\beta(1-i) \Big)  Z_i = \Omega_i ,\\
& \label{eigenfunctions-exterior:id:equations-VOmega-VZ} -\Delta  V_{i,\Omega}^\out= \chi^* \Omega_i, \qquad \qquad  \qquad -\Delta  V_{i,Z}^\out= \chi^* Z_i.
\end{align}
Injecting \eqref{eigenfunctions-exterior:id:decomposition-phiiout} and \eqref{eigenfunctions-exterior:id:equations-Omega-Z} in \eqref{eigenfunctions-exterior:id:eigenfunction-system-leading} we find
\begin{equation} \label{eigenfunctions-exterior:id:decomposition-tildeRi}
\tilde R_i^\out= -4\beta^2\tilde \lambda_i^2 Z_i.
\end{equation}

\noindent \textbf{Step 2}. \emph{Solving \eqref{eigenfunctions-exterior:id:equations-Omega-Z} and proof of \eqref{eigenfunctions-exterior:id:expansion-phiiout-1} and \eqref{eigenfunctions-exterior:bd:expansion-phiiout-infty} for $i=1$}. We make the change of variables for $i=0,1$,
\begin{equation} \label{eigenfunctions-exterior:id:change-variables-Omega-Z}
\Omega_i(z) = 4\beta^2 \tilde \Omega_i(\tilde z), \qquad Z_i(z)=\beta \tilde Z_i(\tilde z), \qquad \tilde z = z \sqrt{2 \beta},
\end{equation}
that transforms $\Hs_\beta = 2\beta \Hs_{\frac 12}$. Hence, \eqref{eigenfunctions-exterior:id:equations-Omega-Z} becomes
\begin{equation} \label{eigenfunctions-exterior:id:system-tildeOmega-tildeZ}
 \Big(\Hs_{\frac 12} - (1-i) \Big) \tilde \Omega_i = 0, \quad \Big(\Hs_{\frac 12} - (1-i) \Big)  \tilde Z_i = 2\tilde \Omega_i .
\end{equation}
The existence and asymptotic behaviors of $\tilde{\Omega}_1$ and $\tilde{Z}_1$ have been rigorously proved in \cite{SSVnon13} (see Lemmas 4.1 and 4.6). There it is showed that there exists a solution with as $\tilde z_\pm=\tilde z\mp (2,0)\to 0$,
\begin{equation} \label{eigenfunctions-exterior:id:decomposition-tildeOmega1}
\tilde \Omega_1(\tilde z_\pm)=\frac{1}{\zeta_\pm^4}+\tilde \Omega_{1,1}(\tilde z_\pm)\pm \tilde \Omega_{1,2}(\tilde z_\pm)+\tilde \Omega_{1,3}(\tilde z_\pm)+A+\tilde \Omega_{1,high}(\tilde z_\pm)
\end{equation}
with $A<0$ computed numerically and
\begin{align}
\label{eigenfunctions-exterior:id:decomposition-tildeOmega11} & \tilde \Omega_{1,1}(\tilde z_{\pm})=\frac{1}{8\tilde \zeta_\pm}[1+2\cos^2\tilde \theta_\pm]=\frac{1}{8\tilde \zeta_\pm^2}[2+\cos(2\tilde \theta_\pm)] ,\\
\label{eigenfunctions-exterior:id:decomposition-tildeOmega12}& \tilde \Omega_{1,2}(\tilde z_{\pm})=\frac{1}{48}\frac{\cos(\tilde \theta_\pm)}{\tilde \zeta_\pm}[3-4\cos^2 \tilde{\theta_\pm}]=-\frac{1}{48\tilde \zeta_\pm}\cos (3\tilde \theta_\pm),\\
\label{eigenfunctions-exterior:id:decomposition-tildeOmega13}&\tilde \Omega_{1,3}(\tilde z_{\pm})=\frac{1}{16}\cos^4 \tilde \theta_\pm=\frac{1}{128}(3+4\cos(2\tilde \theta_\pm)+\cos(4\tilde \theta_\pm)),\\
\label{eigenfunctions-exterior:id:decomposition-tildeOmega1high}&|\nabla^k \tilde \Omega_{1,high}(\tilde z_\pm)|\lesssim\tilde \zeta_\pm^{\sqrt{5}-2-k}, \qquad \mbox{for }k\in \mathbb N.
\end{align}
Actually, \eqref{eigenfunctions-exterior:id:decomposition-tildeOmega1high} is only proved for $k=0$ in \cite{SSVnon13}, but the estimate for $k\geq 1$ follows directly from the Fourier series representation formula in polar coordinates obtained on page 339 in \cite{SSVnon13}. It is also showed that
\begin{equation} \label{eigenfunctions-exterior:id:decomposition-tildeZ1}
\tilde Z_1(\tilde z_\pm)=-\frac{1}{2\tilde \zeta^2_\pm}+\frac 18 \log \tilde \zeta_\pm -\frac 18 \cos (2\tilde \theta_\pm)+Cte+\tilde Z_{1,high}(\tilde z_\pm).
\end{equation}
where
\begin{equation} \label{eigenfunctions-exterior:id:decomposition-tildeZ1high}
|\nabla^k \tilde Z_{1,high}(\tilde z_\pm)|\lesssim \tilde \zeta_\pm^{\sqrt{5}-2-k}, \qquad \mbox{for }k\in \mathbb N.
\end{equation}
Again, the proof of \eqref{eigenfunctions-exterior:id:decomposition-tildeZ1high} follows from a similar Fourier series expansion as the aforementioned one.

The desired expansion \eqref{eigenfunctions-exterior:id:expansion-phiiout-1} and bound \eqref{eigenfunctions-exterior:id:expansion-phiiout-1-bd} then follow from injecting \eqref{eigenfunctions-exterior:id:decomposition-tildeOmega1}-\eqref{eigenfunctions-exterior:id:decomposition-tildeZ1high} in \eqref{eigenfunctions-exterior:id:change-variables-Omega-Z} and \eqref{eigenfunctions-exterior:id:decomposition-phiiout}. The bound $|\nabla^k \tilde \Omega_1|+|\nabla^k \tilde Z_1|\lesssim |\tilde z|^{C_k}$ is proved in \cite{SSVnon13} for $k=0$ using a comparison principle argument, and the same argument can be used to obtain this bound for $k\geq 1$. The bound \eqref{eigenfunctions-exterior:bd:expansion-phiiout-infty} follows from standard elliptic regularity estimates.

\medskip

\noindent \textbf{Step 3}. \emph{Solving \eqref{eigenfunctions-exterior:id:equations-Omega-Z} and proof of \eqref{eigenfunctions-exterior:id:expansion-phiiout-1} and \eqref{eigenfunctions-exterior:bd:expansion-phiiout-infty} for $i=0$}. The construction of $\tilde \Omega_0$ that solves the first equation in \eqref{eigenfunctions-exterior:id:system-tildeOmega-tildeZ} is exactly as in the proof of Lemma 4.1 in \cite{SSVnon13} for $\tilde \Omega_1$. Hence we only show the key estimate for the approximate solution, as the construction of the full solution follows from the exact same arguments. We look for a solution $\tilde \Omega_0$ of the first equation in \eqref{eigenfunctions-exterior:id:equations-Omega-Z} of the form
\begin{equation} \label{eigenfunctions-exterior:id:decomposition-tilde-Omega0}
\tilde \Omega_0(\tilde z)=\sum_\pm (\frac{1}{\tilde \zeta_\pm^4}+\tilde \Omega_{0,1}(\tilde z_\pm)\pm \tilde \Omega_{0,2}(\tilde z_\pm)+\tilde \Omega_{0,3}(\tilde z_\pm))\chi(\tilde z\mp a)+\tilde \omega_{0,high}(\tilde z).
\end{equation}
In what follows, many calculations are simply verified by Matlab symbolic. We begin with a Taylor expansion using $\tilde{a}_\infty = (2,0)$, 
\begin{equation}  \label{eigenfunctions-exterior:id:tilde-Omega-0-tech5}
F(\tilde z_\pm) = \frac{(\tilde z_\pm \pm 2\tilde{a}_\infty)}{|\tilde z_\pm \pm 2\tilde{a}_\infty|^2} \mp \frac{\tilde{a}_\infty}{2} = \sum_{k = 1}^3 (\pm 1)^{k+1}F_k(\tilde z_\pm) + \Oc_{|\tilde z_\pm| \to 0}(|\tilde z_\pm|^4),
\end{equation}
where $F_k(\tilde z)$ is homogeneous of order $k$ with
\begin{align*}
F_1(\tilde z) &= \frac{1}{4}(-\tilde z_{1}, \tilde z_{2}), \quad F_2(\tilde z) = \frac{1}{16}\big(\tilde z_{1}^2 - \tilde z_{2}^2, -2 \tilde z_{1} \tilde z_{2} \big),\\
\quad F_3(\tilde z) & =  \frac{1}{64}\big(-\tilde z_1 (\tilde z_1^2 - 3\tilde z_2), \tilde z_2 (3\tilde z_1^2 - \tilde{z}_2^2) \big).
\end{align*}
We then rewrite 
\begin{equation*}
\Hs_\frac 12 - 1 = \Big(\Delta + \frac{4\tilde z_\pm}{|\tilde z_\pm|^2}.\nabla\Big)  - \Big(\frac{1}{2}\tilde z_\pm.\nabla + 2\Big) + \sum_{k = 1}^3 (\pm 1)^{k+1} F_k(\tilde z_\pm). \nabla + \tilde{F}(\tilde z_\pm).\nabla,
\end{equation*}
where $\tilde{F}(\tilde z_\pm) = \Oc(|\tilde z_\pm|^4)$ as $|\tilde z_\pm| \to 0$. We note the cancellation 
\begin{equation} \label{eigenfunctions-exterior:id:tilde-Omega-0-tech1}
\Big(\Delta + \frac{4\tilde z_\pm}{|\tilde z_\pm|^2}.\nabla\Big) |\tilde z_\pm|^{-4}  - \Big(\frac{1}{2}\tilde z_\pm.\nabla + 2\Big) |\tilde z_\pm|^{-4} = 0,
\end{equation}
and then search for $(\tilde \Omega_{0,k})_{k = 1,2,3}$ satisfying 
\begin{align}
\label{eigenfunctions-exterior:id:tilde-Omega-0-tech2}& \Big(\Delta + \frac{4 \tilde z_\pm}{|\tilde z_\pm|^2}.\nabla\Big) \tilde \Omega_{0,1}(\tilde z_\pm) = -F_1.\nabla (|\tilde z_\pm|^{-4}) := -|\tilde z_\pm|^{-4}\cos(2\tilde \theta_\pm),\\
\label{eigenfunctions-exterior:id:tilde-Omega-0-tech3}& \Big(\Delta + \frac{4\tilde z_\pm}{|\tilde z_\pm|^2}.\nabla\Big) \tilde \Omega_{0,2}(\tilde z_\pm) = \mp F_2.\nabla (|\tilde z_\pm|^{-4}):= c_3 |\tilde z_\pm|^{-3}\cos(3\tilde \theta_\pm),\\
\label{eigenfunctions-exterior:id:tilde-Omega-0-tech4}& \Big(\Delta + \frac{4\tilde z_\pm}{|\tilde z_\pm|^2}.\nabla\Big) \tilde \Omega_{0,3}(\tilde z_\pm) = -F_3.\nabla (|\tilde z_\pm|^{-4}) + \Big(\frac{1}{2}\tilde z_\pm.\nabla + 2\Big)\Omega_{0,1} - F_1 .\nabla \Omega_{0,1}.
\end{align}
From the identity $\big(\Delta + \frac{4 \tilde z}{|\tilde z|^2}. \nabla \big)(f(\tilde \zeta) \cos (k\tilde \theta)) = \big(\pa_{\tilde \zeta}^2f(\tilde \zeta) + \frac{5}{\tilde \zeta} \pa_{\tilde \zeta }f(\tilde \zeta) - \frac{k^2}{\tilde \zeta^2} f(\tilde \zeta) \big) \cos(k\tilde \theta)$, we look for $\tilde \Omega_{0,1}(\tilde z_\pm) = \frac{c_{0,1}}{ |\tilde z_\pm|^{2}} \cos(2\theta_\pm)$ and $\tilde \Omega_{0,2}(\tilde z_\pm) = \frac{c_{0,2}}{ |\tilde z_\pm|} \cos(3\theta_\pm)$. A direct check yields 
\begin{equation}
\label{eigenfunctions-exterior:id:decomposition-tildeOmega012}
\tilde \Omega_{0,1}(\tilde z_\pm) = \frac{\cos(2\tilde \theta_\pm)}{8 |\tilde z_\pm|^2}, \quad \tilde \Omega_{0,2}(\tilde z_\pm) = \frac{1}{48} \frac{\cos(3\tilde \theta_\pm)}{|\tilde z_\pm|}.
\end{equation}
As for $\tilde \Omega_{0,3}$, we expand the source term in the polar coordinate to get after a simplification
\begin{align*}
-F_3(\tilde z_\pm).\nabla (|\tilde z_\pm|^{-4}) + \Big(\frac{1}{2}\tilde z_\pm.\nabla + 2\Big)\Omega_{0,1} - F_1(\tilde z_\pm ) .\nabla \Omega_{0,1} = \frac{\cos(2 \tilde \theta_\pm) - \cos (4\tilde \theta_\pm)}{8|\tilde z_\pm|^2}. 
\end{align*}
A direct check yields 
\begin{equation}\label{eigenfunctions-exterior:id:decomposition-tildeOmega03}
\tilde \Omega_{0,3}(\tilde z_\pm) = -\frac{1}{32} \cos(2\tilde \theta_\pm) + \frac{1}{128} \cos(4\tilde \theta_\pm),
\end{equation}
that satisfies the equation $(\Delta + \frac{4 \tilde{z}_\pm}{|\tilde z_\pm|^2}) \tilde \Omega_{0,3}(\tilde z_\pm) = \frac{\cos(2\tilde \theta_\pm) - \cos (4\tilde \theta_\pm)}{8|\tilde z_\pm|^2}$. 

Using first \eqref{eigenfunctions-exterior:id:tilde-Omega-0-tech5}-\eqref{eigenfunctions-exterior:id:tilde-Omega-0-tech4}, and then \eqref{eigenfunctions-exterior:id:decomposition-tildeOmega012}-\eqref{eigenfunctions-exterior:id:decomposition-tildeOmega03} we get for the error generated by the approximate solution
\begin{align}
\nonumber &(\mathcal H_{\frac 12 }-1)\left(\frac{1}{\tilde \zeta_\pm^4}+\tilde \Omega_{0,1}(\tilde z_\pm)\pm \tilde \Omega_{0,2}(\tilde z_\pm)+\tilde \Omega_{0,3}(\tilde z_\pm) \right)\\
\nonumber & \qquad \qquad = F_1(\tilde z_\pm).\nabla (\tilde \Omega_{0,2}+\tilde \Omega_{0,3})\pm F_2(\tilde z_\pm).\nabla \sum_{j=1}^3 \tilde \Omega_{0,j}+F_3(\tilde z_\pm) .\nabla \sum_{j=1}^3 \tilde \Omega_{0,j}\\
\label{eigenfunctions-exterior:id:tilde-Omega-0-tech100}&\qquad \qquad \qquad +F(\tilde z_\pm).\nabla \left(\frac{1}{\zeta_\pm^4}+\tilde \Omega_{0,1}(z_\pm)\pm \tilde \Omega_{0,2}(\tilde z_\pm)+\tilde \Omega_{0,3}(\tilde z_\pm)  \right) \ = O(|\tilde z_\pm|^{-1}).
\end{align}
Introducing $Q_{0,high}=-(\mathcal H_{\frac 12 }-1)\sum_\pm (\frac{1}{\tilde \zeta_\pm^4}+\tilde \Omega_{0,1}(\tilde z_\pm)\pm \tilde \Omega_{0,2}(\tilde z_\pm)+\tilde \Omega_{0,3}(\tilde z_\pm))\chi(\tilde z\mp a)$ and using \eqref{eigenfunctions-exterior:id:tilde-Omega-0-tech100}, we find that the remainder in \eqref{eigenfunctions-exterior:id:decomposition-tilde-Omega0} solves
\begin{equation} \label{eigenfunctions-exterior:id:tilde-Omega-0-tech10}
(\mathcal H_{\frac 12}-1)\tilde \omega_{0,high}=Q_{0,high}, \qquad |\nabla^k Q_{0,high}|\lesssim \sum_\pm \frac{1}{|\tilde z_\pm|}.
\end{equation}
With the key estimate $|\nabla^k Q_{0,high}|\lesssim \sum_\pm |\tilde z_\pm|^{-1}$ at hand, the solution of \eqref{eigenfunctions-exterior:id:tilde-Omega-0-tech10} can now be found exactly as in the proof of Lemma 4.1 in \cite{SSVnon13}, with for some constant $A_0\in \mathbb R$,
\begin{equation} \label{eigenfunctions-exterior:id:tilde-Omega-0-tech11}
 \tilde \omega_{0,high}(\tilde z)=A_0+\tilde \Omega_{0,high}(\tilde z), \qquad |\nabla^k \tilde \Omega_{0,high}(\tilde z)|\lesssim |\tilde z_\pm|^{\sqrt{5}-2-k} \quad \mbox{for }k\in \mathbb N,
\end{equation}
as $z\to \pm a$. The computation of $\tilde Z_0$ that solves the second equation in \eqref{eigenfunctions-exterior:id:system-tildeOmega-tildeZ} can be done similarly, yielding for some constant $\tilde B_0\in \mathbb R$,
\begin{align} \label{eigenfunctions-exterior:id:decomposition-tildeZ0}
& \tilde Z_0(\tilde z_\pm)=-\frac{1}{2\tilde \zeta^2_\pm}-\frac 18 \log \tilde \zeta_\pm -\frac 18 \cos (2\tilde \theta_\pm)+\tilde B_0+\tilde Z_{0,high}(\tilde z_\pm),\\
& |\nabla^k \tilde Z_{0,high}(\tilde z_\pm)|\lesssim \tilde \zeta_\pm^{\sqrt{5}-2-k}, \qquad \mbox{for }k\in \mathbb N.
\end{align}
The proof is exactly as that of Lemma 4.6 in \cite{SSVnon13}. Injecting \eqref{eigenfunctions-exterior:id:decomposition-tilde-Omega0}, \eqref{eigenfunctions-exterior:id:decomposition-tildeOmega012}, \eqref{eigenfunctions-exterior:id:decomposition-tildeOmega03}, \eqref{eigenfunctions-exterior:id:tilde-Omega-0-tech11} and \eqref{eigenfunctions-exterior:id:decomposition-tildeZ0} in \eqref{eigenfunctions-exterior:id:change-variables-Omega-Z} shows the first desired expansion \eqref{eigenfunctions-exterior:id:expansion-phiiout-1} and estimate \eqref{eigenfunctions-exterior:id:expansion-phiiout-1-bd}. The second desired estimate \eqref{eigenfunctions-exterior:bd:expansion-phiiout-infty} for $i=0$ follows from the same arguments as in Step 2 for $i=1$.

\medskip

\noindent \textbf{Step 4}. \emph{Computation of the Poisson field and proof of \eqref{eigenfunctions-exterior:id:expansion-Viiout-1} and \eqref{eigenfunctions-exterior:bd:expansion-phiiout-Viout-infty}.} We recall the identities
\begin{align}
\label{eigenfunctions-exterior:id:Poisson-tech-1} & - \Delta (-\frac 1{4r^2})=\frac{1}{r^4},\qquad  -\Delta ( -\frac 1{2}(\log r)^2)=\frac{1}{r^2}, \qquad - \Delta ( \frac 1{4}\cos (2\theta))=\frac{\cos(2\theta)}{r^2}\\
\label{eigenfunctions-exterior:id:Poisson-tech-2} & - \Delta ( \frac 1{8}r\cos (3\theta))=\frac{\cos(3\theta)}{r}, \qquad - \Delta(- \frac 14 r^2\log r+\frac 14 r^2)=\log r, \qquad -\Delta (-\frac{r^2}{4})=1,\\
\label{eigenfunctions-exterior:id:Poisson-tech-3} & - \Delta (-\frac{r^2\log r}{4}\cos (2\theta))=\cos (2\theta) \qquad - \Delta( \frac{1}{12}r^2\cos (4\theta))=\cos 4\theta.
\end{align}
In view of \eqref{eigenfunctions-exterior:id:expansion-phiiout-1}, we look for a solution of the form
\begin{align} \label{eigenfunctions-exterior:id:Poisson-tech-3}
&V_{i}^\out(z)=\sum_\pm \bar V_{i,\pm}^\out(z_\pm) \chi(z\mp a)+V_{i,high}^\out(z),\\
\nonumber & \bar V_{i,\pm}^\out(z_\pm)=-\frac{1}{4\zeta_\pm^2}-\left(\frac{\beta}{4}i-\frac{\beta}{4}\tilde \lambda_i\right)(\log r)^2+\frac{\beta}{16}\cos (2\theta) \mp \frac{\beta^{3/2}}{96\sqrt{2}}\zeta_\pm \cos(3\theta_\pm)+(-1)^i\frac{\beta^2}{16}(\zeta_\pm^2\log \zeta_\pm-\zeta_\pm^2),\\
&\qquad \qquad +\beta^2 \zeta_\pm^2 \left(-\frac{3i}{128}-A_i+\left(\frac{2i-1}{16}-\frac{B_i}{2}\right)\tilde \lambda_i+\frac{1-2i+2\tilde \lambda_i}{32}\log \zeta_\pm \cos(2\theta_\pm)+\frac{1}{384}\cos(4\theta_\pm) \right)
\end{align}
Injecting the decomposition \eqref{eigenfunctions-exterior:id:Poisson-tech-3} in the second equation in \eqref{eigenfunctions-exterior:id:eigenfunction-system}, using \eqref{eigenfunctions-exterior:id:Poisson-tech-1} and \eqref{eigenfunctions-exterior:id:Poisson-tech-3} we find that $V_{i,high}^\out$ solves
\begin{align} \label{eigenfunctions-exterior:id:Poisson-tech-4}
&-\Delta V_{i,high}^\out= f_{i,high} \\
&f_{i,high}=\sum_\pm \phi_{i,high}^\out \chi(z\mp a) \\
\nonumber &\qquad \qquad +(\chi^*-\sum_\pm \chi(z\mp a))\phi_i^\out+2\sum_\pm \nabla \bar V_{i,\pm}^\out(z_\pm) .\nabla \chi(z\mp a)+\sum_\pm  \bar V_{i,\pm}^\out(z_\pm) \Delta \chi(z\mp a)
\end{align}
where $f_{i,high}$ is compactly supported on $\mathbb R^2$ and satisfies using \eqref{eigenfunctions-exterior:id:expansion-phiiout-1},
\begin{equation}\label{eigenfunctions-exterior:id:Poisson-tech-10}
|\nabla^k f_{i,high}(z)|\lesssim \sum_\pm |z\pm a|^{\sqrt{5}-2-k}.
\end{equation}
We then choose as a solution to \eqref{eigenfunctions-exterior:id:Poisson-tech-4} $V_{i,high}^\out=\Phi_{ f_{i,high}}$ and claim that as $z_\pm \to 0$,
\begin{align}  
&\nonumber V_{i,high}^\out=C_i\pm (D_i,0).(z\mp a_\infty))+E_i(z_{\pm,1}^2-z_{\pm,2}^2)+\tilde V_{i,high}^\out \\
& \qquad \qquad =C_i\pm D_i\cos(\theta_\pm)+E_i\cos(2\theta_\pm)+\tilde V_{i,high}^\out,\nonumber\\
& \label{eigenfunctions-exterior:id:Poisson-tech-5} |\nabla^k \tilde V_{i,high}^\out(z)|\lesssim \zeta_\pm^{\sqrt{5}-k}\mbox{ for }k\in \mathbb N,
\end{align}
for some constants $C_i,D_i,E_i$ that corresponds to the Taylor expansion of $V_{i,high}$ at $a_{\infty}$. Note that the specific form of this Taylor expansion comes from the fact that $V_{i,high}$ is invariant by the symmetry $z_2\mapsto -z_2$ and satisfies $\Delta V_{i,high}(a_{\infty})=f_{i,high}(a_\infty)=0$. We are going to prove the estimate in \eqref{eigenfunctions-exterior:id:Poisson-tech-5} first for $k=2$. We have for $j=1,2$,
\begin{align*}
 \nabla \partial_{x_j} \tilde V_{i,high}^\out(a_\infty+z_+) & =\nabla \partial_{x_j}  \tilde V_{i,high}^\out(a_\infty+z_+)-\nabla \partial_{x_j}  \tilde V_{i,high}^\out(a_\infty) \\
& = c \int_{\mathbb R^2}\left(\frac{z-z'}{|z-z'|^2}+ \frac{z'}{|z'|^2}\right) \partial_{x_j}  f_{i,high}(a_\infty+z')dz' \\
&= c \int_{2|z|<|z'|} ... dz'+ c \int_{2|z|>|z'|}... dz'  = I+II
\end{align*}
Using \eqref{eigenfunctions-exterior:id:Poisson-tech-10} and $|\frac{z-z'}{|z-z'|^2}+ \frac{z'}{|z'|^2}|\lesssim |z||z'|^{-2}$ if $2|z|<|z'|$ we bound $|I|\lesssim \int_{|z|<2|z'|\lesssim 1}\frac{|z|}{|z'|^2}|z'|^{\sqrt{5}-3}dz'\lesssim |z|^{\sqrt{5}-2}$. We bound $|II|\lesssim \int_{|z'|<2|z|} (|z-z'|^{-1}+|z'|^{-1})|z'|^{\sqrt{5}-3}dz'\lesssim |z|^{\sqrt{5}-2}$. This shows the estimate in \eqref{eigenfunctions-exterior:id:Poisson-tech-5} for $k=2$. Next, since $\partial^{\alpha}\tilde V_{i,high}^\out(a_\infty)=0$ for all $\alpha\in \mathbb N^2$ with $\alpha_1+\alpha_2\leq 2$, and $|\nabla^2 \tilde V_{i,high}^\out(z)|\lesssim |z\pm a_\infty|^{\sqrt{5}-2}$, the estimate in \eqref{eigenfunctions-exterior:id:Poisson-tech-5} for $k=0$ and $k=1$ is a direct consequence of the fundamental Theorem of Calculus. We then show the estimate in \eqref{eigenfunctions-exterior:id:Poisson-tech-5} for $k\geq 3$ by induction. Assume it is true for some $k\geq 2$. Pick then $\alpha \in \mathbb N^2$ with $|\alpha|=k$ and, for $0<R\leq 1$, let $\tilde V_{i,R}^\alpha (y)=\partial^\alpha \tilde V_{i,high}^\out (a_\infty+Ry)$. Then it solves on $\{1/4<|y|<1\}$
$$
\tilde V_{i,R}^\alpha (y)=R^2 \partial^\alpha f_{i,high}(a_\infty+Ry)=O(R^{\sqrt{5}-k})
$$
using \eqref{eigenfunctions-exterior:id:Poisson-tech-10}, where it also satisfies $|\tilde V_{i,R}^\alpha(y)|\lesssim R^{\sqrt{5}-k-\epsilon}$. Hence by elliptic regularity we have $\nabla \tilde V_{i,R}^\alpha=O(R^{\sqrt{5}-k})$ for $\{1/3<|y|<1/2\}$. This means $|\nabla \partial^\alpha \tilde V_{i,high}^\out (a_\infty+z_+)|\lesssim R^{\sqrt{5}-1-k}$ for $R/3<|z_+|<R/2$. Hence \eqref{eigenfunctions-exterior:id:Poisson-tech-5} for $k+1$ as $R$ is arbitrary.

The first desired expansion \eqref{eigenfunctions-exterior:id:expansion-Viiout-1} then follows from the decomposition \eqref{eigenfunctions-exterior:id:Poisson-tech-3}, the estimate \eqref{eigenfunctions-exterior:id:Poisson-tech-5}, and adding any linear combination of the functions $\sum_\pm \frac{z_1\mp a_{\infty,1}}{|z\mp a_\infty|^2}$ and $1$ and $\sum_\pm \log |z\pm a_\infty|$ which solve the Laplace equation. The second desired estimate \eqref{eigenfunctions-exterior:bd:expansion-phiiout-Viout-infty} follows from the fact that $V_{i,high}^\out=\Phi_{ f_{i,high}}$ with $f_{i,high}$ being compactly supported. The value $D_{i,2}^\out$ comes from the Taylor expansion of $z\mapsto \frac{z_1+a_{\infty,1}}{|z+a_{\infty}|^2}$ at $a_\infty$.

\medskip

\noindent \textbf{Step 5}. \emph{Proof of \eqref{eigenfunctions-exterior:bd:Ri}}. We recall that $|a-a_\infty|\lesssim \nu$ and $|\tilde \lambda_i|\lesssim |\ln \nu|^{-1}$. In what follows we take $z\in \cap_\pm \{|z\pm a|\gg \nu\}$, and hence $|z\pm a|\approx |z\pm a_\infty|$. Injecting \eqref{eigenfunctions-exterior:id:decomposition-tildeZ1} and \eqref{eigenfunctions-exterior:id:decomposition-tildeZ0} in \eqref{eigenfunctions-exterior:id:decomposition-tildeRi} shows
\begin{equation} \label{eigenfunctions-exterior:bd:tilde-Ri}
|\tilde R_i^\out(z)|\lesssim \frac{1}{|\ln \nu|^2}|Z_i(z)|\lesssim \frac{1}{|\log \nu|^2}\sum_\pm \frac{\langle z\rangle^C}{|z\pm a_\infty|^2} \lesssim \frac{1}{|\log \nu|^2}\sum_\pm \frac{\langle z\rangle^C}{|z\pm a|^2}  .
\end{equation}
We next estimate $\bar R_i^\out$ in \eqref{eigenfunctions-exterior:id:eigenfunction-system-2}. For the first term, we have since $|a-a_\infty|\lesssim \nu$ that
\begin{align*}
\left|\nabla \Phi_{U_{1+2,\nu}}+4\sum_\pm \frac{z\pm a_\infty}{|z\pm a_\infty|^2} \right|&=4\left|\sum_\pm \frac{|z\pm a_\infty|^2(z\pm a)-(\nu^2+|z\pm a|^2)(z\pm a_\infty)}{|z\pm a_\infty|^2(\nu^2+|z\pm a|^2)} \right|\\
& \qquad \lesssim \sum_\pm \frac{\nu^2+|a-a_\infty|}{|z\pm a|^2}
\end{align*}
so that since $|\nabla \phi_i^\out|\lesssim \sum_\pm |z\pm a_\infty|^{-5}\langle z \rangle^C$ from \eqref{eigenfunctions-exterior:id:expansion-phiiout-1} we find
\begin{equation} \label{eigenfunctions-exterior:bd:bar-Ri1}
\left|\nabla \phi_i^\out .\left(-\nabla \Phi_{U_{1+2,\nu}}-4\sum_\pm \frac{z\pm a_\infty}{|z\pm a_\infty|^2}\right) \right|\lesssim \sum_\pm \frac{\nu^2+|a-a_\infty|}{|z\pm a|^7}\langle z \rangle^C.
\end{equation}
Next, using $|\nabla^k U_{1+2,\nu}|\lesssim \nu^2 \sum_\pm |z\pm a|^{-5}$ for $k=0,1$ and $|\nabla V_i^\out|\lesssim |\ln \nu|\sum_\pm |z\pm a_\infty|^{-3}\langle z \rangle^2$ from \eqref{eigenfunctions-exterior:id:expansion-Viiout-1} and the assumption $|A_i^\out|,|B_i^\out|\lesssim |\ln \nu|$ and $|\phi_i^\out|\lesssim \sum_\pm |z\pm a_\infty|^{-4}\langle z \rangle^C$ from \eqref{eigenfunctions-exterior:id:expansion-phiiout-1} we find for the second and third terms
\begin{equation} \label{eigenfunctions-exterior:bd:bar-Ri2-3}
\left|-\nabla U_{1+2,\nu}.\nabla V_i^\out+2U_{1+2,\nu}\phi_i^\out \right|\lesssim |\ln \nu| \sum_\pm \frac{\nu^2}{|z\pm a|^8}\langle z \rangle^C.
\end{equation}
Injecting \eqref{eigenfunctions-exterior:bd:bar-Ri1} and \eqref{eigenfunctions-exterior:bd:bar-Ri2-3} in the definition of $\bar R_i^\out$ in \eqref{eigenfunctions-exterior:id:eigenfunction-system-2} we find
\begin{equation} \label{eigenfunctions-exterior:bd:bar-Ri}
|\bar R_i|\lesssim  \sum_\pm \frac{\nu^2}{|z\pm a|^8}\langle z \rangle^C+\sum_\pm \frac{|a-a_\infty|}{|z\pm a|^7}\langle z \rangle^C.
\end{equation}
As $R_i^\out=\tilde R_i^\out+\bar R_i^\out$, combining \eqref{eigenfunctions-exterior:bd:tilde-Ri} and \eqref{eigenfunctions-exterior:bd:bar-Ri} using $|a-a_\infty|\lesssim \nu$ shows the desired estimate \eqref{eigenfunctions-exterior:bd:Ri}.

\end{proof}

\subsection{Matching and computation of the eigenvalues}\label{sec:matched-asymptotic}

\subsubsection{Matched asymptotic expansions}

\begin{lemma} \label{eigenfunctions-matching}
Let $i=0,1$ and $\epsilon>0$ be given by Proposition \ref{prop:phi1inn}. Pick
\begin{equation} \label{id:value-tilde-lambdai-tech}
\tilde \lambda_i =\frac{7-10i+256 A_i}{16(2i-1)}\frac{1}{|\ln \nu|} \ \frac{1}{1+\frac{1-10 i-32B_i}{(8i-4)|\ln \nu|}}
\end{equation}
and choose the following values of the constants
\begin{equation} \label{id:value-parameters-tech}
(A_i^{\out},B_i^{\out},C_i^{\out},L_i^{\inn},M_i^{\inn})=\left(A_i^{\out*}+\frac{\alpha^{1/2}}{4}\tilde L_i^\inn,B_i^{\out*},C_i^{\out*},L_i^{\inn*}+\tilde L_i^\inn,M_i^{\inn*}+\tilde M_i^\inn\right)
\end{equation}
where $\tilde L_i^\inn,\tilde M_i^\inn \in \mathbb R$ are free parameters and
\begin{align} \label{id:value-tilde-Aiout-tech}
& A_i^{\out*}=\alpha^{\frac 12}\frac{L_i^{\inn*}}{4},\\
\label{id:value-tilde-Biout-tech}
& B_i^{\out*}=\alpha\left(\frac{i-\tilde \lambda_i}{2}\ln \nu+\frac{2i+1-2\tilde \lambda_i }{4}\right),\\
\label{id:value-tilde-Ciout-tech}& C_i^{\out*} =\alpha\left(\frac{\tilde \lambda_i-i}{4}\ln^2 \nu+\frac{\tilde \lambda_i-1-2i}{4}\ln \nu\right),\\
\label{id:value-tilde-Liout-tech} & L_i^{\inn*} =\frac{32\alpha^{-\frac 32}D_i^\out(A_i^{\out*},B_i^{\out*})}{1-4i+4\tilde \lambda_i},\\
\label{id:value-tilde-Miout-tech} & M_i^{\inn*} = (2i-1- \tilde \lambda_i ) \ln \nu  +\frac{5-10i+10\tilde \lambda_i}{4}-\frac{E_i^\out (A_i^{\out*},B_i^{\out*})}{32\alpha^2},
\end{align}
Then we have
\begin{align} \label{eigenfunctions:id:def-phi-bo}
- \frac{1}{16\nu^4} \phi^\inn_{i,\pm}\left(\frac{z\mp a}{\nu}\right)- \phi_i^\out(z) = \phi_{i,\pm}^\bou (z),\\
\label{eigenfunctions:id:def-V-bo}- \frac{1}{16\nu^2} V^\inn_{i,\pm}\left(\frac{z\mp a}{\nu}\right)- V_i^\out(z) = V_{i,\pm}^\bou (z),
\end{align}
where for $\nu$ small enough, for $|a-a_{\infty}|\lesssim \nu$ and $|\tilde L_i^\inn|,|\tilde M_i^\inn|\lesssim 1$, the boundary terms satisfy for all $k\in \mathbb N$,
\begin{align} \label{eigenfunctions:bd:phi-bo}
&|\nabla^k \phi_{i,\pm}^\bou (z)|\lesssim (\nu^{\epsilon}  |z\mp a|^{-4-\epsilon} +|z\mp a|^{\sqrt{5}-2})|z\mp a|^{-k},\\
& \label{eigenfunctions:bd:V-bo} |\nabla^k V_{i,\pm}^\bou (z)|\lesssim  (\nu^{\epsilon}  |z-a|^{-2-\epsilon} +|z\mp a|^{\sqrt{5}})|z\mp a|^{-k} +|\tilde L_i^\inn||z\mp a|^{1-k}+|\tilde M_i^\inn||z\mp a|^{2-k} ,
\end{align}
for all $\nu\ll |z-a|\ll 1$.

\end{lemma}

\begin{remark}

Note that by Proposition \ref{pr:phi1out}, the equations \eqref{id:value-tilde-Aiout-tech}-\eqref{id:value-tilde-Miout-tech} corresponds to the following linear system for $A_i^{\out*},B_i^{\out*},C_i^{\out*},L_i^{\inn*},M_i^{\inn*}$:
$$
\mathfrak M(A_i^{\out*},B_i^{\out*},C_i^{\out*},L_i^{\inn*},M_i^{\inn*})^\top=(x_1,x_2,x_3,x_4,x_5)^\top
$$
where
$$
\mathfrak M= \begin{pmatrix}
1 & 0 & 0 & -\frac{\alpha^{\frac 12}}{4} & 0 \\
0 & 1 & 0 & 0 & 0\\
0 & 0 & 1 & 0 & 0\\
-\frac{32\alpha^{-\frac 32}D_{i,1}^\out}{1-4i+4\tilde \lambda_i} & -\frac{32\alpha^{-\frac 32}D_{i,2}^\out}{1-4i+4\tilde \lambda_i} & 0 & 1 & 0\\
\frac{\alpha^{-2}E_{i,1}^\out}{32} & \frac{\alpha^{-2}E_{i,2}^\out}{32} & 0 & 0 & 1
\end{pmatrix}
$$
and
\begin{align*}
& x_1 =0,\\
& x_2= \alpha\left(\frac{i-\tilde \lambda_i}{2}\ln \nu+\frac{2i+1-2\tilde \lambda_i)}{4}\right)\\
& x_3= \alpha\left(\frac{\tilde \lambda_i-i}{4}\ln^2 \nu+\frac{\tilde \lambda_i-1-2i}{4}\ln \nu\right),\\
&x_4=\frac{32\alpha^{-\frac 32}D_{i,0}^\out}{1-4i+4\tilde \lambda_i},\\
&x_5=(2i-1- \tilde \lambda_i ) \ln \nu  +\frac{5-10i+10\tilde \lambda_i}{4}-\frac{E_{i,0}^\out }{32\alpha^2}.
\end{align*}
We compute using $\alpha^{-1}\beta=1+O(\nu)$ that $\textup{Det} ( \mathfrak M)=1-\frac{8\alpha^{-1}D_{i,1}^\out}{1-4i+4\tilde \lambda_i}=1+\frac{1}{1-4i+4\tilde \lambda_i}+O(\nu)\neq 0$ so that the system has a unique solution.

Moreover, we have from \eqref{id:value-tilde-Aiout-tech}-\eqref{id:value-tilde-Miout-tech} that
\begin{align*}
&|A_0^{\out *}|+|B_0^{\out *}|+|L_0^{\inn *}|\lesssim 1 \quad \mbox{and} \quad |A_1^{\out *}|+|B_1^{\out *}|+|L_1^{\inn *}|\lesssim |\ln \nu|,\\
&|C_i^\out|+|M_i^\out|\lesssim |\ln \nu| \qquad \mbox{for }i=0,1.
\end{align*}

\end{remark}

\begin{proof}

\noindent \underline{Proof of \eqref{eigenfunctions:bd:phi-bo}}. Recall that $z_\pm=z\mp a$ in the statement of Proposition \ref{prop:phi1inn} and $z_\pm=z\mp a_\infty$ in the statement of Proposition \ref{pr:phi1out}. In the following proof, we take $z_\pm = z\mp a$, introduce $\bar z=a-a_\infty$ and decompose $\phi_{i}^\out(z)=\phi_{i}^\out(z\mp \bar z)+\phi_{i}^\out(z)-\phi_{i}^\out(z\mp \bar z)$. We have by these propositions
\begin{align}
\nonumber &  \phi_{i,\pm}^\bou (z) \\
\nonumber  &=  \frac{1}{\zeta_\pm^4} +\frac{ \alpha}{\zeta_\pm^2} \left(\frac{i }{2}+\frac{1}{4}\cos 2\theta_\pm \right) \mp \alpha^{\frac 32} \frac{1}{12\sqrt{2}\zeta_\pm}\cos(3\theta_\pm)\\
\nonumber & \qquad  +\alpha^2 \left(\frac{i}{4}-\frac{7}{64}+\frac{2i-1}{8}\cos(2\theta_\pm)+\frac{1}{32}\cos(4\theta_\pm)\right) \\
\nonumber &\qquad +  \tilde \lambda_i  \left(- \frac{\alpha}{2\zeta_\pm^2} +\alpha^2 \left( (\frac i2-\frac 14) (\ln \zeta_\pm-\ln \nu )+\frac{1-10i}{16}-\frac 14\cos(2\theta_\pm)\right)\right)+\phi_{i,\pm,\sharp}^{\inn ,z }(z_\pm) \\
\nonumber  &\qquad -\frac{1}{\zeta_\pm^4}-\frac{\beta}{\zeta_\pm^2}\left(\frac{1}{2}i+\frac 1 4 \cos(2\theta_\pm)-\frac 1 2 \tilde \lambda_i\right)\pm \frac{\beta^{\frac 32}}{12\sqrt{2}\zeta_\pm} \cos(3\theta_\pm)-\beta^2\tilde \lambda_i \left(\frac i2 -\frac 14\right) \log \zeta_\pm \\
& \nonumber \qquad -\frac{\beta^2}{32} \left(3i+(8i-4)\cos(2\theta_\pm)+\cos (4\theta_\pm)\right) -4\beta^2 A_i-2\beta^2 \tilde \lambda_i \left(-\frac 1 8 \cos(2\theta_\pm)+B_i\right)\\
\nonumber & \qquad -\phi_{i,high}^\out(z\mp \bar z)+\phi_{i}^\out(z\mp \bar z)-\phi_{i}^\out(z)\\
\label{id:phi-bo-tech} &=\frac{\alpha-\beta}{\zeta_\pm^2}\left(\frac{1}{2}i+\frac 1 4 \cos(2\theta_\pm)-\frac 1 2 \tilde \lambda_i\right)\mp  \frac{\alpha^{\frac 32}-\beta^{\frac 32}}{12\sqrt{2}\zeta_\pm}\cos(3\theta_\pm) \\
\nonumber & \quad +(\alpha^2-\beta^2) \Bigg(\frac{2i-1}{8}\cos(2\theta_\pm)+\frac{1}{32}\cos(4\theta_\pm)+\tilde \lambda_i (\frac i2-\frac 14) \ln \zeta_\pm-\frac{\tilde \lambda_i}4\cos(2\theta_\pm)+\frac{i}{4}-\frac{7}{64}\\
\nonumber &\qquad \qquad \qquad \qquad +\tilde \lambda_i\left((\frac 14 -\frac i2)\ln \nu +\frac{1-10i}{16}\right) \Bigg) \\
\nonumber  &\quad +\beta^2\underbrace{\left(\frac{10i-7}{64}-4A_i +\tilde \lambda_i\left( \frac{2i-1}{4} |\ln \nu|+\frac{1-10i}{16}-2B_i\right)\right)}_{=0} +\phi_{i,\pm,\sharp}^{\inn ,z }(z_\pm)-\phi_{i,high}^\out(z).
\end{align}
In order to match the constant terms above, we picked 
$$\tilde \lambda_i  = -\frac{\frac{10i-7}{64}-4A_i }{ \frac{2i-1}{4} |\ln \nu|+\frac{1-10i}{16}-2B_i} $$
which corresponds to \eqref{id:value-tilde-lambdai-tech}. We have $|a-a_\infty|\lesssim \nu$ so that $|\alpha-\beta|\lesssim \nu$ and $\phi_{i}^\out(z\mp \bar z)-\phi_{i}^\out(z)=\Oc(\nu \zeta_\pm^{-3})=\Oc(\nu^\epsilon \zeta_\pm^{-2-\epsilon}) $ for $\nu\ll \zeta_\pm \ll 1$. Using this and $|\tilde \lambda_i|\lesssim |\ln \nu|^{-1}$ we deduce from \eqref{id:phi-bo-tech} using \eqref{expansion:phii_inn-estimate} and \eqref{eigenfunctions-exterior:id:expansion-phiiout-1-bd} that for $\nu \ll |z-a|\ll 1$,
$$
 \phi_{i,\pm}^\bou (z) =\Oc(\frac{\nu}{|z-a|^2})+\Oc(\frac{\nu}{|z-a|})+\Oc(\nu)+\Oc(\nu^{\epsilon}  |z-a|^{-4-\epsilon} )+\Oc(|z-a|^{\sqrt{5}-2}),
$$
which shows \eqref{eigenfunctions:bd:phi-bo} for $k=1$. The estimate propagates for higher order derivative by \eqref{expansion:phii_inn-estimate} and \eqref{eigenfunctions-exterior:id:expansion-phiiout-1-bd}.

\smallskip

\noindent \underline{Proof of \eqref{eigenfunctions:bd:V-bo}}. By \eqref{expansion:Vi_inn}  and \eqref{eigenfunctions-exterior:id:expansion-Viiout-1} we compute

\begin{align}
\nonumber & V_{i,\pm}^\bou (z)\\
\nonumber  &= \frac{-1}{4\zeta^2_\pm}\pm \frac{\alpha^{\frac 12}L_i^\inn}{4} \frac{1}{\zeta_\pm}\cos(\theta_\pm)+\alpha \Bigg( \frac{\tilde \lambda_i-i}{4} \ln^2 \zeta_\pm+\left(\frac{i-\tilde \lambda_i}{2}\ln \nu+\frac{2i+1-2\tilde \lambda_i)}{4}\right)\ln \zeta_\pm \\
\nonumber  &\qquad \qquad \qquad \qquad \qquad \qquad \quad +\frac{\tilde \lambda_i-i}{4}\ln^2\nu+\frac{2\tilde \lambda_i-1-2i}{4}\ln \nu+\frac{1}{16}\cos (2\theta_\pm)\Bigg)\\
\nonumber  & \quad \mp  \alpha^{\frac 32}\zeta_\pm \left( L_i^\inn \frac{1-4i+4\tilde \lambda_i}{32}\cos \theta+ \frac{1}{96\sqrt{2}}\cos(3\theta_\pm) \right)\\
\nonumber  &  \quad+\alpha^2\zeta^2_\pm \Bigg( \frac{1-2i}{16}\tilde \lambda_i \ln \zeta_\pm +\frac{2i-1}{16} \tilde \lambda_i\ln \nu+\frac{2\tilde \lambda_i+1-2i}{32} \ln \zeta_\pm \cos(2\theta_\pm) +\frac{2i-1- \tilde \lambda_i}{32} \ln \nu \cos(2\theta_\pm)\\
\nonumber  & \qquad \qquad \qquad +\frac{1}{16}\left(\frac{7}{16}-i+\tilde \lambda_i(\frac92 i-\frac54)\right)+\frac{5-10i+10\tilde \lambda_i-8M_i^\inn}{128}\cos(2\theta_\pm)+\frac{1}{384}\cos(4\theta_\pm) \Bigg)\\
\nonumber  & \qquad +V^{\inn,z}_{i,\pm,\sharp}(z_\pm) \\
\nonumber  & +\frac{1}{4\zeta_\pm^2}+\left(\frac{\beta}{4}i-\frac{\beta}{4}\tilde \lambda_i\right)(\log \zeta_\pm)^2-\frac{\beta}{16}\cos(2\theta_\pm)\pm \frac{\beta^{3/2}}{96\sqrt{2}}\zeta_\pm \cos(3\theta_\pm)-\beta^2 \tilde \lambda_i \frac{1-2i}{16}\zeta_\pm^2\log \zeta_\pm \\
\nonumber  &\qquad  \qquad -\beta^2 \zeta_\pm^2 \left(-\frac{3i}{128}-A_i+\left(\frac{2i-1}{16}-\frac{B_i}{2}\right)\tilde \lambda_i+\frac{1-2i+2\tilde \lambda_i}{32}\log \zeta_\pm \cos(2\theta_\pm)+\frac{1}{384}\cos(4\theta_\pm) \right)\\
\nonumber  &\qquad \qquad \mp \frac{A_i^\out}{\zeta_\pm}\cos(\theta_\pm) -B_i^\out \log \zeta_\pm -C_i^\out \mp D_i^\out \zeta_\pm \cos(\theta_\pm)-E_i^\out \zeta_\pm^2\cos(2\theta_\pm)\\
\nonumber  & \qquad -V_{i,high}^\out (z-\bar z)+V_{i}^\out(z\mp \bar z)-V_{i}^\out(z) \\
\nonumber  &=(\alpha-\beta)\left(\frac{\tilde \lambda_i-i}{4} \ln^2 \zeta_\pm +\frac{1}{16}\cos(2\theta_\pm) \right) \mp  \frac{\alpha^{\frac 32}-\beta^{\frac 32}}{96\sqrt{2}} \zeta_{\pm} \cos(3\theta_{\pm})  \\
\nonumber  &\quad+(\alpha^2-\beta^2)\zeta^2_\pm \Bigg( \frac{1-2i}{16}\tilde \lambda_i \ln \zeta_\pm + \frac{2i-1}{16} \tilde \lambda_i\ln \nu+\frac{1}{16}\left(\frac{7}{16}-i+\tilde \lambda_i(\frac92 i-\frac54)\right) \\
\nonumber  & \qquad \qquad \qquad \qquad\qquad +\frac{2\tilde \lambda_i+1-2i}{32} \ln \zeta_\pm \cos(2\theta_\pm) +\frac{1}{384}\cos(4\theta_\pm) \Bigg)\\
\nonumber  & \qquad \pm \underbrace{\left( \frac{\alpha^{\frac 12}L_i^\inn}{4}-A_i^\out\right)}_{=0} \frac{1}{\zeta_\pm}\cos(\theta_\pm)\mp\underbrace{\left(\alpha^{\frac 32}L_i^\inn\frac{1-4i+4\tilde \lambda_i}{32} -D_i^\out\right)}_{=2^{-5}[\alpha^{3/2}(1-4i+4\tilde \lambda_i)-8\alpha^{1/2}D_{i,1}^\out]\tilde L_i^\inn} \zeta_\pm \cos(\theta_\pm) \\
\nonumber  & \quad +\underbrace{\left(\alpha\left(\frac{i-\tilde \lambda_i}{2}\ln \nu+\frac{2i+1-2\tilde \lambda_i)}{4}\right)-B_i^\out \right)}_{=0}\ln \zeta_\pm+\underbrace{\alpha\left(\frac{\tilde \lambda_i-i}{4}\ln^2 \nu+\frac{\tilde \lambda_i-1-2i}{4}\ln \nu\right)-C_i^\out }_{=0}\\
\nonumber  & \quad +\beta^2\zeta_\pm^2 \underbrace{\left( \frac{2i-1}{16} \tilde \lambda_i\ln \nu+\frac{1}{16}\left(\frac{7}{16}-i+\tilde \lambda_i(\frac92 i-\frac54)\right) +\frac{3i}{128}+A_i+\left(\frac{1-2i}{16}+\frac{B_i}{2}\right)\tilde \lambda_i\right)}_{=I=0}\\
\nonumber  & + \zeta^2_\pm \underbrace{\left(\alpha^2\left(\frac{2i-1- \tilde \lambda_i}{32} \ln \nu  +\frac{5-10i+10\tilde \lambda_i-8M_i^\inn}{128}\right)-E_i^\out \right)}_{=-\frac{\alpha^2}{16}\tilde M_i^\inn}\cos(2\theta_\pm)\\
\label{longest-equation-that-I-have-ever-written}&+\Oc(\nu^{\epsilon} \zeta_\pm^{-2-k-\epsilon}) +\Oc(\zeta_\pm^{\sqrt{5}})
\end{align}
where we used \eqref{expansion:Vi_inn-estimate} and \eqref{eigenfunctions-exterior:id:expansion-Viiout-2} to estimate $V^{\inn,z}_{1,\pm,\sharp}$ and $V_{i,high}^\out (z)$, as well as $V_{i}^\out(z\mp \bar z)-V_{i}^\out(z) =\Oc(\nu \zeta_\pm^{-3})=\Oc(\nu^\epsilon \zeta_\pm^{-2-\epsilon})$ since $|\bar z|\lesssim \nu$.

In order to match the $\ln \zeta_\pm$ term above, we picked $B_i^\out=\alpha\left(\frac{i-\tilde \lambda_i}{2}\ln \nu+\frac{2i+1-2\tilde \lambda_i)}{4}\right)$ which corresponds to \eqref{id:value-tilde-Biout-tech}. In order to match the constant term, we picked $C_i^\out =\alpha\left(\frac{\tilde \lambda_i-i}{4}\ln^2 \nu+\frac{\tilde \lambda_i-1-2i}{4}\ln \nu\right)$ which corresponds to \eqref{id:value-tilde-Ciout-tech}. We compute that
\begin{align*}
I& =\frac{1}{4}\left( \frac{2i-1}{4} \tilde \lambda_i\ln \nu+\frac{1}{4}\left(\frac{7}{16}-i+\tilde \lambda_i(\frac92 i-\frac54)\right) +\frac{3i}{32}+4A_i+\left(\frac{1-2i}{4}+2B_i \right)\tilde \lambda_i\right)\\
& =\frac{1}{4}\left( \tilde \lambda_i\left(\frac{2i-1}{4} \ln \nu +\frac{10i-1}{16}+2B_i\right) +\frac{7-10i}{64}+4A_i\right) \ =0 
\end{align*}
by the choice \eqref{id:value-tilde-lambdai-tech} of $\tilde \lambda_i$. In order to match the $\zeta_\pm^2 \cos(2\theta_\pm)$ term, we picked $M_i^\inn= (2i-1- \tilde \lambda_i ) \ln \nu  +\frac{5-10i+10\tilde \lambda_i}{4}-\frac{E_i^\out }{32\alpha^2}$, which corresponds to \eqref{id:value-tilde-Miout-tech}.

\end{proof}

\subsubsection{The full approximate eigenfunctions and proof of Proposition \ref{prop:eigen}}

We eventually define $\lambda_i=2\beta (1-i+\tilde \lambda_i)$ where $\tilde \lambda_i$ is given by \eqref{eigenfunctions-matching}, and
\begin{equation} \label{eigenfunctions:id:decomposition-phii-final}
\phi_i(z) = -\frac{1}{16\nu^{4}}\sum_\pm \phi_{i,\pm}^\inn \Big(\frac{z\mp a}{\nu}\Big)  \chi(\frac{z\mp a}{\eta}) + \phi_{i}^\out(z) \big(1 - \sum_\pm \chi(\frac{z\mp a}{\eta})\big).
\end{equation}
The parameters $A_i^\out,B_i^\out,C_i^\out,M_i^\inn,L_i^\inn$ are given by \eqref{id:value-parameters-tech}, and the correction parameters $\tilde L_i^\inn$ and $\tilde M_i^\inn$ are not fixed yet but will be fixed in the first Step of the proof of Proposition \ref{prop:eigen} below in \eqref{eigenfunctions:id:estimatetildeLtildeM}. The function $\chi$ is a smooth radial cut-off function with $\chi(x)=1$ for $|x|<1$ and $\chi(x)=0$ for $|x|>2$. We choose $\eta=\nu^{\tilde \epsilon}$ with $\tilde \epsilon$ small enough so that we can simplify $\Oc(\nu^\epsilon |\zeta_\pm|^{-4-\epsilon})=\Oc(|\zeta_\pm|^{\sqrt{5}-2})$ and $\Oc(\nu^\epsilon |\zeta_\pm|^{-4-\epsilon})=\Oc(|\zeta_\pm|^{\sqrt{5}-2})$ for $|\zeta_\pm|\approx \eta$ in \eqref{eigenfunctions:bd:phi-bo} and \eqref{eigenfunctions:bd:V-bo}. We can now end the proof of Proposition \ref{prop:eigen}.

\begin{proof}[Proof of Proposition \ref{prop:eigen}]

\noindent \textbf{Step 1}. \emph{Matching the parameters to cancel the Poisson field}. Let $V_i (z)=\Phi_{\chi^*\phi_i}(z) $. We decompose
\begin{equation} \label{eigenfunctions:id:decomposition-Vi-tildeVi}
V_i (z) = \sum_\pm -\frac{1}{16\nu^{2}}V_{i,\pm}^\inn \Big(\frac{z\mp a}{\nu}\Big)  \chi(\frac{z\mp a}{\eta}) + V_{i}^\out(z) \big(1 - \sum_\pm \chi(\frac{z\mp a}{\eta})\big)+\tilde V_i.
\end{equation}
We recall that $-\Delta V_{i,\pm}^\inn= \phi_{i,\pm}^\inn$ and $-\Delta V_{i,\pm}^\out=\chi^* \phi_{i,\pm}^\out$, so that the equation $-\Delta V_i=\chi^* \phi_i$ gives
\begin{equation} \label{eigenfunctions:id:decomposition-Vi-tildeVi-elliptic}
-\Delta \tilde V_i= F_{i,+}^\bou+F_{i,-}^\bou, \qquad F_{i,\pm}^\bou = \frac{2}{\eta}\nabla \chi(\frac{z\mp a}{\eta}).\nabla V_{i,\pm}^\bou+ \frac{1}{\eta2}\Delta  \chi(\frac{z\mp a}{\eta}) V_{i,\pm}^\bou .
\end{equation}
As $V^\out$ grows sublinearly as $|z|\to \infty$ from \eqref{eigenfunctions-exterior:bd:expansion-phiiout-Viout-infty}, by a uniqueness argument for the Poisson equation we have 
\begin{equation} \label{eigenfunctions:id:expression-tilde-Vi}
\tilde V_i=\tilde V_{i,+}+\tilde V_{i,-} \qquad \mbox{where}\qquad \tilde V_{i,\pm}= \Phi_{F_{i,\pm}^\bou}.
\end{equation}
We want to find values for $\tilde L_i^\inn$, $\tilde M_i^\inn$ such that
\begin{align}
\label{eigenfunctions:id:vanishing-remainding-Poisson}
\partial_{x_1}\tilde V_i(a)=\partial_{x_1}^2\tilde V_i(a)=0.
\end{align}
In order to achieve \eqref{eigenfunctions:id:vanishing-remainding-Poisson}, notice that by linearity,
\begin{align} \label{eigenfunctions:id:systeme-tildeLi-tildeMi}
\begin{pmatrix} \partial_{x_1}\tilde V_i(a) \\ \partial_{x_1}^2\tilde V_i(a) \end{pmatrix} = \mathfrak M \begin{pmatrix} \tilde L_i^\inn \\  \tilde M_i^\inn \end{pmatrix} +\begin{pmatrix} \partial_{x_1}\tilde V_i[A_i^{\out*},B_i^{\out*},C_i^\out,L_i^{\out*},M_i^{\inn*}](a) \\ \partial_{x_1}^2\tilde V_i[A_i^{\out*},B_i^{\out*},C_i^{\out*},L_i^{\out*}M_i^{\inn*}](a) \end{pmatrix} 
\end{align}
where
$$
 \mathfrak M= \begin{pmatrix} \partial_{\tilde L_i^\inn}\partial_{x_1}\tilde V_i (a) &  \partial_{\tilde M_i^\inn}\partial_{x_1}\tilde V_i (a)\\
 \partial_{\tilde L_i^\inn}\partial_{x_1}^2\tilde V_i (a) &  \partial_{\tilde M_i^\inn}\partial_{x_1}^2\tilde V_i (a) \end{pmatrix}.
$$
We compute using the representation formula $\Phi_f =-(2\pi)^{-1}\log |\cdot|*f$, and then integrating by parts and using polar coordinates,
\begin{align}
\nonumber \partial_{x_1}\tilde V_{i,+}(a) & =\frac{1}{2\pi} \int_{\mathbb R^2}\frac{z_{+,1}}{|z_+|^2}F_{i,+}^\bou(a+z_+)dz_+ \\
\label{eigenfunctions-id-partial-tildex1-tildeVi+}& =\frac{1}{2\pi \eta^2} \int_{0}^\infty \int_0^{2\pi}\left(-\partial_{rr}\chi(\frac{\zeta_+}{\eta})+\frac{\eta}{\zeta_+}\partial_r \chi(\frac{\zeta_+}{\eta})\right) \cos(\theta_+)V_{i,+}^\bou(a+z_+)d\zeta_+d\theta_+ .
\end{align}
A direct estimate of the above integral with \eqref{eigenfunctions:bd:V-bo} yields
\begin{equation} \label{eigenfunctions:id:partialx1tildeVi-1}
 \partial_{x_1}\tilde V_{i,+}[A_i^{\out*},...,M_i^{\inn*}](a) =\Oc(\eta^{\sqrt{5}-1}).
\end{equation}
Next, by \eqref{longest-equation-that-I-have-ever-written} and \eqref{id:value-parameters-tech} we have
\begin{align} \label{eigenfunctions-id-partial-tildeLi-Vibou}
&\partial_{\tilde L_{i}^\inn}\tilde V_{i,+}^\bou=2^{-5}[\alpha^{3/2}(1-4i+4\tilde \lambda_i)-8\alpha^{1/2}D_{i,1}^\out] \zeta_+\cos(\theta_+)+O(\zeta_+^{\sqrt{5}}),\\
& \label{eigenfunctions-id-partial-tildeMi-Vibou} \partial_{\tilde M_{i}^\inn}\tilde V_{i,+}^\bou =-\frac{\alpha^2}{16}\zeta_+^2 \cos(2\theta_+)+O(\zeta_+^{\sqrt{5}}).
\end{align}
Injecting \eqref{eigenfunctions-id-partial-tildeLi-Vibou} and \eqref{eigenfunctions-id-partial-tildeMi-Vibou} in \eqref{eigenfunctions-id-partial-tildex1-tildeVi+}, using $\int_0^\infty (-\partial_{rr}\chi(\frac{\zeta_+}{\eta})+\frac{\eta}{\zeta_+}\partial_r \chi(\frac{\zeta_+}{\eta}))\zeta_+d\zeta_+=-2\eta^2$ and that $\int_0^{2\pi}\cos(\theta_+)\cos(k\theta_+)d\theta_+=\pi \delta_{k=1}$ we find
\begin{align}
\label{eigenfunctions:id:partialx1tildeVi-2} &\partial_{\tilde L_{i}^\inn}  \partial_{x_1}\tilde V_{i,+}(a)  = 2^{-5}[\alpha^{3/2}(1-4i+4\tilde \lambda_i)-8\alpha^{1/2}D_{i,1}^\out]\tilde L_i^\inn+\Oc(\eta^{\sqrt{5}-1}),\\
\label{eigenfunctions:id:partialx1tildeVi-3}&\partial_{\tilde M_{i}^\inn}  \partial_{x_1}\tilde V_{i,+}(a)  = \Oc(\eta^{\sqrt{5}-1}).
\end{align}
Another computation, similar to the one leading to \eqref{eigenfunctions-id-partial-tildex1-tildeVi+}, shows
\begin{align}
\nonumber \partial_{x_1}^2\tilde V_{i,+}(a) & =\frac{1}{2\pi} \int_{\mathbb R^2}\frac{z_{+,1}^2-z_{+,2}^2}{|z_+|^4}F_{i,+}^\bou(a+z_+)dz_+ \\
\label{eigenfunctions-id-partial-tildex12-tildeVi+}& =\frac{1}{2\pi \eta^2} \int_{0}^\infty \int_0^{2\pi}\left(-\partial_{rr}\chi(\frac{\zeta_+}{\eta})+\frac{\eta}{\zeta_+}\partial_r \chi(\frac{\zeta_+}{\eta})\right) \frac{\cos(2\theta_+)}{\zeta_+}V_{i,+}^\bou(a+z_+)d\zeta_+d\theta_+ .
\end{align}
Using \eqref{eigenfunctions:bd:V-bo}, \eqref{eigenfunctions-id-partial-tildeLi-Vibou} and \eqref{eigenfunctions-id-partial-tildeMi-Vibou} and $\int_0^\infty (-\partial_{rr}\chi(\frac{\zeta_+}{\eta})+3\frac{\eta}{\zeta_+}\partial_r \chi(\frac{\zeta_+}{\eta}))\zeta_+d\zeta_+=-4\eta^2$ we find
\begin{align}
 \label{eigenfunctions:id:partialx12tildeVi-1}
& \partial_{x_1}^2\tilde V_{i,+}[A_i^{\out*},...,M_i^{\inn*}](a) =\Oc(\eta^{\sqrt{5}-2}),\\
\label{eigenfunctions:id:partialx12tildeVi-2} &\partial_{\tilde L_{i}^\inn}  \partial_{x_1}^2\tilde V_{i,+}(a)  = \Oc(\eta^{\sqrt{5}-2}),\\
\label{eigenfunctions:id:partialx12tildeVi-3}&\partial_{\tilde M_{i}^\inn}  \partial_{x_1}^2\tilde V_{i,+}(a)  =\frac{\alpha^2}{8}+ \Oc(\eta^{\sqrt{5}-2}).
\end{align}
We have for $\eta $ small enough, using that $F_{i,-}^\bou$ is supported in $\{\eta<|z+a|<2\eta\}$ and performing a Taylor expansion,
\begin{align*}
\partial_{x_1}\tilde V_{i,-}(a) & =\int_{\mathbb R^2} \frac{z_--2a}{|z_--2a|^2}F_{i,-}^\bou(-a+z_-)dz_- \\
&= \frac{-1}{4|a|^2} \int_{\mathbb R^2} F_{i,-}^\bou(z)dz+\int_{\mathbb R^2} \Oc(|z_-|)F_{i,-}^\bou(-a+z_-)dz_-\\
& = \frac{-1}{4|a|^2} \int_{\mathbb R^2} F_{i,-}^\bou(z)dz+\Oc(\| |z+a|F_{i,-}^\bou\|_{L^1})\\
\end{align*}
Direct estimates using \eqref{eigenfunctions:bd:V-bo}, the analogue of \eqref{eigenfunctions-id-partial-tildeLi-Vibou} and \eqref{eigenfunctions-id-partial-tildeMi-Vibou} for $V_{i,-}^\bou$ for $k=1,2$ show that $\| F_{i,-}^\bou[A_i^{\out*},...,M_i^{\inn*}]\|_{L^1}=\Oc(\eta^{\sqrt{5}})$, $\| |z+a| \partial_{\tilde L_i^\inn} F_{i,-}^\bou\|_{L^1}=\Oc(\eta^2)$ and $\| |z+a| \partial_{\tilde M_i^\inn} F_{i,-}^\bou\|_{L^1}=\Oc(\eta^3)$. Using these estimates and $\int_0^{2\pi}\cos(k\theta_-)d\theta_-=0$ for $k=1,2$ we have
\begin{align}
 \label{eigenfunctions:id:partialx1tildeVi-1-}
& \partial_{x_1} \tilde V_{i,-}[A_i^{\out*},...,M_i^{\inn*}](a) =\Oc(\eta^{\sqrt{5}}),\\
\label{eigenfunctions:id:partialx1tildeVi-2-} &\partial_{\tilde L_{i}^\inn}  \partial_{x_1}\tilde V_{i,+}(a)  = \Oc(\eta^{2}),\\
\label{eigenfunctions:id:partialx1tildeVi-3-}&\partial_{\tilde M_{i}^\inn}  \partial_{x_1}\tilde V_{i,+}(a)  = \Oc(\eta^{\sqrt{5}}).
\end{align}
An analogue computation shows 
\begin{align}
 \label{eigenfunctions:id:partialx12tildeVi-1-}
& \partial_{x_1}^2\tilde V_{i,-}[A_i^{\out*},...,M_i^{\inn*}](a) =\Oc(\eta^{\sqrt{5}}),\\
\label{eigenfunctions:id:partialx12tildeVi-2-} &\partial_{\tilde L_{i}^\inn}  \partial_{x_1}^2\tilde V_{i,+}(a)  = \Oc(\eta^{2}),\\
\label{eigenfunctions:id:partialx12tildeVi-3-}&\partial_{\tilde M_{i}^\inn}  \partial_{x_1}^2\tilde V_{i,+}(a)  = \Oc(\eta^{\sqrt{5}}).
\end{align}
Combining \eqref{eigenfunctions:id:partialx1tildeVi-1}, \eqref{eigenfunctions:id:partialx1tildeVi-2}, \eqref{eigenfunctions:id:partialx1tildeVi-3}, \eqref{eigenfunctions:id:partialx12tildeVi-1}, \eqref{eigenfunctions:id:partialx12tildeVi-2}, \eqref{eigenfunctions:id:partialx12tildeVi-3} and \eqref{eigenfunctions:id:partialx1tildeVi-1-}-\eqref{eigenfunctions:id:partialx12tildeVi-3-} and using $D_{i,1}^\out=-\frac{\beta}{8}$ from Proposition \ref{pr:phi1out} with $|\alpha-\beta|\lesssim \nu$ shows that in the system \eqref{eigenfunctions:id:systeme-tildeLi-tildeMi}
$$
\mathfrak M=\begin{pmatrix} 2^{-5}\alpha^{3/2} [ 2-4i+4\tilde \lambda_i]+\Oc(\eta^{\sqrt{5}-1}) &   \Oc(\eta^{\sqrt{5}-1}) \\
 \Oc(\eta^{\sqrt{5}-2}) & \frac{\alpha^2}{8}+ \Oc(\eta^{\sqrt{5}-2})
\end{pmatrix}
$$
and
$$
\begin{pmatrix} \partial_{x_1}\tilde V_i[A_i^{\out*},B_i^{\out*},...,M_i^{\inn*}](a) \\ \partial_{x_1}^2\tilde V_i[A_i^{\out*},B_i^{\out*},...,M_i^{\inn*}](a) \end{pmatrix}=\begin{pmatrix} \Oc(\eta^{\sqrt{5}-1}) \\ \Oc(\eta^{\sqrt{5}-2})  \end{pmatrix} .
$$
In order to ensure the vanishing condition \eqref{eigenfunctions:id:vanishing-remainding-Poisson} for $\tilde V_i$, we choose
\begin{equation} \label{eigenfunctions:id:estimatetildeLtildeM}
\begin{pmatrix}\tilde L_i^\inn\\ \tilde M_i^\inn\end{pmatrix}=- \mathfrak M^{-1}\begin{pmatrix} \partial_{x_1}\tilde V_i[A_i^{\out*},B_i^{\out*},...,M_i^{\inn*}](a) \\ \partial_{x_1}^2\tilde V_i[A_i^{\out*},B_i^{\out*},...,M_i^{\inn*}](a) \end{pmatrix}=\begin{pmatrix}\Oc(\eta^{\sqrt{5}-1})\\ \Oc(\eta^{\sqrt{5}-2})\end{pmatrix}.
\end{equation}

\medskip

\noindent \textbf{Step 2}. \emph{Estimates for the eigenfunctions}. 

\smallskip

\noindent \underline{Proof of \eqref{est:pointwise_phii}.} We estimate all terms in \eqref{eigenfunctions:id:decomposition-phii-final}. First, by \eqref{def:phii_inn}, \eqref{expansion:phii_inn} and \eqref{expansion:phii_inn-estimate} we have
\begin{equation} \label{eigenfunctions:bd:pointwise-phii-inter1}
\left| \nabla^k \left(\frac{1}{\nu^4}\Phi_{i,\pm}^\inn(\frac{z\mp a}{\nu}) \right)\right|\lesssim \frac{1}{(\nu+|z\mp a|)^{4+k}}
\end{equation}
for $|z\mp a|\leq 1$. Second, by \eqref{eigenfunctions-exterior:id:expansion-phiiout-1} and \eqref{eigenfunctions-exterior:id:expansion-phiiout-1-bd} we have
\begin{equation} \label{eigenfunctions:bd:pointwise-phii-inter2}
\left|\nabla^k \phi_i^\out (z) \right|\lesssim \sum_\pm \frac{\langle z \rangle^{C_k}}{|z\mp a|^{4+k}}
\end{equation}
for $z\in \mathbb R^2 \backslash \{-a,a\}$. Third, by a direct computation,
\begin{equation} \label{eigenfunctions:bd:pointwise-phii-inter3}
\left| \nabla^k \left(\chi (\frac{z\mp a}{\nu^{\tilde \epsilon}}) \right)\right|\lesssim \frac{1}{|z\mp a|^{k}}\mathbbm 1 (\nu^{\tilde \epsilon}\leq |z\mp a|\leq 2 \nu^{\tilde \epsilon})
\end{equation}
for all $z\in \mathbb R^2$. Injecting \eqref{eigenfunctions:bd:pointwise-phii-inter1}, \eqref{eigenfunctions:bd:pointwise-phii-inter2} and \eqref{eigenfunctions:bd:pointwise-phii-inter3} in \eqref{eigenfunctions:id:decomposition-phii-final} shows $|\nabla^k \phi_i(z)|\lesssim \sum_\pm (\nu+|z\mp a|)^{-4-k}\langle z \rangle^{C_k}$. This proves the first inequality in \eqref{est:pointwise_phii}. The second inequality in \eqref{est:pointwise_phii} follows from applying $\nu\partial_\nu $ to \eqref{eigenfunctions:id:decomposition-phii-final} and performing analogue estimates.

\smallskip

\noindent \underline{Proof of the first inequality in \eqref{est:pointwise_phiitil}.} We have by \eqref{def:phii-2}, \eqref{def:phii} and then using $16 L_i=-\alpha^{1/2}L_i^\inn$ that
\begin{align}
\nonumber \tilde \phi_i(z) &= -\frac{1}{16\nu^4} \sum_\pm \phi_{i,\pm}^\inn (\frac{z\mp a}{\nu})\chi_{\pm a,\nu^{\tilde \epsilon}}(z)+\phi_i^\out (1-\sum_\pm\chi_{\pm a,\nu^{\tilde \epsilon}}(z))\\
\nonumber  & - \sum_\pm \left(\frac{-1}{16\nu^4}\Lambda U\pm \frac{L_i}{\nu^3}\partial_{x_1}U\right)(\frac{z\mp a}{\nu}))\chi_{\pm a,\zeta_*}(z)\\
\label{eigenfunctions:bd:estimate-tildephii-inter1}  &= -\frac{1}{16\nu^4}  \sum_\pm (\phi_{i,\pm}^\inn -\Lambda U\mp \alpha^{1/2}L_i^\inn\partial_{x_1}U) (\frac{z\mp a}{\nu})\chi_{\pm a,\nu^{\tilde \epsilon}}(z)\\
\nonumber  & + \sum_\pm \left( \phi_i^\out (z)-|z\mp a|^{-4}\right) (\chi_{\pm a,\zeta_*}(z)-\chi_{\pm a,\nu^{\tilde{\epsilon}}}(z))\\
\nonumber  & -\sum_\pm  \left(\frac{-1}{16\nu^4}(\Lambda U+16|\cdot|^{-4})\pm \frac{L_i}{\nu^3}\partial_{x_1}U\right)(\frac{z\mp a}{\nu}))(\chi_{\pm a,\zeta_*}(z)-\chi_{\pm a,\nu^{\tilde{\epsilon}}}(z)\\
\nonumber  &+\phi_i^\out (1-\sum_\pm\chi_{\pm a,\zeta_*}(z)).
\end{align}
We now estimate all terms in \eqref{eigenfunctions:bd:estimate-tildephii-inter1}. Recall that $|\nabla^k (\phi_{i,\pm}^\inn-\Lambda U\mp \alpha^{1/2}L_i^\inn\nu\partial_{x_1}U)(y)|\lesssim \nu^2 \langle y \rangle^{-2-k}$ by \eqref{eigenfunctions:bd:estimate-barphiinn-tech3}. Hence by \eqref{eigenfunctions:bd:pointwise-phii-inter3} we have 
\begin{equation} \label{eigenfunctions:bd:estimate-tildephii-inter2}
\left|\frac{1}{\nu^4} \nabla^k\left( \sum_\pm (\phi_{i,\pm}^\inn -\Lambda U\mp \alpha^{1/2}L_i^\inn\partial_{x_1}U) (\frac{z\mp a}{\nu})\chi_{\pm a,\nu^{\tilde \epsilon}}(z)\right)\right|\lesssim \sum_\pm \frac{1}{(\nu+|z\mp a|)^{2+k}}.
\end{equation}
We have $|\nabla^k (\phi_i^\out(z)-\sum_\pm |z\mp a|^{-4})|\lesssim \sum_\pm |z\mp a|^{-2-k}$ by \eqref{eigenfunctions-exterior:id:expansion-phiiout-1} and \eqref{eigenfunctions-exterior:id:expansion-phiiout-1-bd}. By this and \eqref{eigenfunctions:bd:pointwise-phii-inter3} we get 
\begin{equation} \label{eigenfunctions:bd:estimate-tildephii-inter3}
\left|\nabla^k \left(\sum_\pm \left( \phi_i^\out (z)-|z\mp a|^{-4}\right) (\chi_{\pm a,\zeta_*}(z)-\chi_{\pm a,\nu^{\tilde{\epsilon}}}(z))\right)\right|\lesssim \sum_\pm \frac{1}{(\nu+|z\mp a|)^{2+k}}.
\end{equation}
We have the tail cancellation $|\nabla^k (\Lambda U+16|y|^{-4})|\lesssim |y|^{-6-k}$, and $|\nabla^k \partial_{x_1}U(y)|\lesssim |y|^{-5-k}$. Hence
$$
\left|   \left(\frac{- \nabla^k}{16\nu^4}(\Lambda U+16|y|^{-4})\pm \frac{L_i}{\nu^3} \nabla^k \partial_{x_1}U\right)(y)\right|\lesssim \nu^{-4}|y|^{-6-k}+\nu^{-3}|y|^{-5-k}\lesssim \nu^{-3\tilde \epsilon}|y|^{-2-k}
$$
for $\nu^{-1+\tilde \epsilon}\lesssim |y|\lesssim \nu^{-1}$. Using that $\chi_{\pm a,\zeta_*}-\chi_{\pm a,\nu^{\tilde \epsilon}}$ has support in $\{\nu^{\tilde \epsilon}\leq |z\mp a|\leq 2 \zeta_*\}$ and \eqref{eigenfunctions:bd:pointwise-phii-inter3} this shows
\begin{equation} \label{eigenfunctions:bd:estimate-tildephii-inter5}
\left| \nabla^k\left( \sum_\pm  \left(\frac{-1}{16\nu^4}(\Lambda U+16|\cdot|^{-4})\pm \frac{L_i}{\nu^3}\partial_{x_1}U\right)(\frac{z\mp a}{\nu}))(\chi_{\pm a,\zeta_*}(z)-\chi_{\pm a,\nu^{\tilde{\epsilon}}}(z)\right)\right|\lesssim \sum_\pm \frac{\nu^{2-3\tilde \epsilon}}{(\nu+|z\mp a|)^{2+k}}.
\end{equation}
Finally, as $1-\sum_\pm \chi_{\pm a, \zeta_*}$ is supported in $\cap_\pm \{|z\mp a|\geq \zeta_*\}$ we have
\begin{equation} \label{eigenfunctions:bd:estimate-tildephii-inter6}
\left| \nabla^k \left(\phi_i^\out (1-\sum_\pm\chi_{\pm a,\zeta_*}(z))\right)\right|\lesssim \langle z \rangle^{C_k}.
\end{equation}
Injecting \eqref{eigenfunctions:bd:estimate-tildephii-inter2}, \eqref{eigenfunctions:bd:estimate-tildephii-inter3}, \eqref{eigenfunctions:bd:estimate-tildephii-inter5} and \eqref{eigenfunctions:bd:estimate-tildephii-inter6} in \eqref{eigenfunctions:bd:estimate-tildephii-inter1} shows $|\nabla^k  \tilde \phi_i(z) |\lesssim \sum_\pm (\nu+|z\mp a|)^{-2-k}\langle z \rangle^{C_k}$. This is the first inequality in \eqref{est:pointwise_phiitil}.

\smallskip

\noindent \underline{Proof of the second inequality in \eqref{est:pointwise_phiitil}.} We compute using \eqref{def:phii-2}, \eqref{def:phii}, $16 L_i=-\alpha^{1/2}L_i^\inn$ and \eqref{eigenfunctions:id:def-phi-bo} that
\begin{align}
  \label{eigenfunctions:bd:estimate-nupartialnu-tildephii-inter1} \nu \partial_\nu \tilde \phi_i(z)  &= -\frac{1}{16}  \sum_\pm \nu\partial_\nu \left(\frac{1}{\nu^4}(\phi_{i,\pm}^\inn -\Lambda U\mp \alpha^{1/2}L_i^\inn\partial_{x_1}U) (\frac{z\mp a}{\nu})\right)\chi_{\pm a,\nu^{\tilde \epsilon}}(z)\\
\nonumber  &  -\sum_\pm \nu\partial_\nu \left(\frac{-1}{16\nu^4}\Lambda U\pm \frac{L_i}{\nu^3}\partial_{x_1}U (\frac{z\mp a}{\nu})\right)(\chi_{\pm a,\zeta_*}(z)-\chi_{\pm a,\nu^{\tilde{\epsilon}}}(z)\\
\nonumber  &+\nu\partial_\nu \phi_i^\out (1-\sum_\pm\chi_{\pm a,\zeta_*}(z))+ \sum_\pm \phi_{i,\pm}^\bou(z)\nu \partial_\nu (\chi_{\pm a,\nu^{\tilde \epsilon}}(z))
\end{align}
We now estimate all terms in \eqref{eigenfunctions:bd:estimate-nupartialnu-tildephii-inter1}. We decompose using \eqref{eigenfunctions:id:decomposition-phiinn}:
\begin{align}
  \label{eigenfunctions:bd:estimate-nupartialnu-tildephii-inter2}& \nu \partial_\nu \left( \frac{1}{\nu^4}\left(\phi_{i,\pm}^\inn -\Lambda U\mp \alpha^{1/2}L_i^\inn\partial_{x_1}U\right) (\frac{z\mp a}{\nu})\right) \\
\nonumber =&\nu \partial_\nu \left( \frac{1}{\nu^4}\left(\tilde \phi_{i,\pm}^\inn -\Lambda U\right) (\frac{z\mp a}{\nu})\right)+\nu \partial_\nu \left( \frac{1}{\nu^4}\left(\bar \phi_{i,\pm}^\inn\mp \alpha^{1/2}L_i^\inn\partial_{x_1}U\right) (\frac{z\mp a}{\nu})\right).
\end{align}
For the first term in \eqref{eigenfunctions:bd:estimate-nupartialnu-tildephii-inter2} we have by \eqref{eigenfunctions:id:tildephiinn}
\begin{align*}
& \nu \partial_\nu \left( \frac{1}{\nu^4}\left(\tilde \phi_{i,\pm}^\inn -\Lambda U\right) (\frac{z\mp a}{\nu})\right)\\
&=  \nu \partial_\nu \left( \left( \frac{\alpha}{\nu^2}(T_2^{(i)}+\tilde \lambda_i\hat T_2)\pm \frac{\alpha^{\frac 32}}{\nu}T_3+\alpha^2 (T_4^{(i)}+\tilde \lambda_i \hat T_4^{(i)})\right) (\frac{z\mp a}{\nu})\right)\\
& =  \left( - \frac{\alpha}{\nu^2}(y.\nabla +2)(T_2^{(i)}+\tilde \lambda_i\hat T_2)\mp \frac{\alpha^{\frac 32}}{\nu}(y.\nabla+1)T_3+\alpha^2y.\nabla (T_4^{(i)}+\tilde \lambda_i \hat T_4^{(i)})\right) (\frac{z\mp a}{\nu})\\
& \qquad  +\frac{\alpha}{\nu^2} (\nu \partial_\nu \tilde \lambda_i)\hat T_2+\alpha^2(\nu \partial_\nu \tilde \lambda_i)T_4^{(i)}.
\end{align*}
We have the following tail cancellations by Lemmas \ref{lem:T20hatT2}, \ref{lem:T22}, \ref{lem:T3k}, \ref{lem:T40hatT40}, \ref{lem:T42hatT42} and \ref{lem:T44}, and using that $\tilde \lambda_i=\Oc(|\ln \nu|^{-1})$ and $\nu \partial_\nu \tilde \lambda_i=\Oc(|\ln \nu|^{-2})$ by \eqref{id:value-tilde-lambdai-tech}:
\begin{align*}
&(y.\nabla+2)(T_2^{(i)}+\tilde \lambda_i)(y)=\Oc(\langle y \rangle^{-\sqrt{8}}),\\
&(y.\nabla+1)T_3(y)=\Oc(\langle y \rangle^{2-\sqrt{13}}),\\
&y.\nabla (T_4^{(i)}+\tilde \lambda_i \hat T_4^{(i)})(y)=\Oc(\langle y \rangle^{4-\sqrt{20}})+\Oc(|\ln \nu|^{-1}),\\
&(\nu \partial_\nu \tilde \lambda_i)\hat T_2=\Oc(|\ln \nu|^{-2}\langle y \rangle^{-2}),\\
&(\nu \partial_\nu \tilde \lambda_i)\hat T_4^{(i)}=\Oc(|\ln \nu|^{-1}),
\end{align*}
for $|y|\lesssim \nu^{-1}$ and these estimates propagate to higher order derivatives. Letting $c=\min (\sqrt{8}-2,\sqrt{13}-3,\sqrt{20}-4)$ we thus obtain
\begin{align}
\nonumber \left|\nabla^k\left( \nu \partial_\nu \left( \frac{1}{\nu^4}\left(\tilde \phi_{i,\pm}^\inn -\Lambda U\right) (\frac{z\mp a}{\nu})\right)\right)\right| & \lesssim \sum_{i=0}^2 \frac{\nu^c}{(\nu+|z\mp a|)^{2+c-i+k}}+\frac{|\ln \nu|^{-2}}{(\nu+|z\mp a|)^{2+k}}+\frac{|\ln \nu|^{-1}}{(\nu+|z\mp a|)^{k}}\\
 \label{eigenfunctions:bd:estimate-nupartialnu-tildephii-inter3}& \lesssim \frac{1}{(\nu+|z\mp a|)^{2+k}}\left(\frac{\nu^c}{(\nu+|z\mp a|)^c}+|\ln \nu|^{-2}\right)
\end{align}
for $|z\mp a|\lesssim \nu^{\tilde \epsilon}$. We estimate similarly the second term in \eqref{eigenfunctions:bd:estimate-nupartialnu-tildephii-inter2} using \eqref{eigenfunctions:id:barphiinn} and Proposition \ref{pr:solution-system-SW}. This is simpler as we do not need to track tail cancellations for this term, and it gives
\begin{equation}  \label{eigenfunctions:bd:estimate-nupartialnu-tildephii-inter4}
\left| \nabla^k \left(\nu \partial_\nu \left( \frac{1}{\nu^4}\left(\bar \phi_{i,\pm}^\inn\mp \alpha^{1/2}L_i^\inn\partial_{x_1}U\right) (\frac{z\mp a}{\nu})\right)\right)\right|\lesssim  \frac{\nu^{c'}}{(\nu+|z\mp a|)^{2+k+c'}}.
\end{equation}
for some constant $c'>0$. Injecting \eqref{eigenfunctions:bd:estimate-nupartialnu-tildephii-inter3} and \eqref{eigenfunctions:bd:estimate-nupartialnu-tildephii-inter4} in \eqref{eigenfunctions:bd:estimate-nupartialnu-tildephii-inter2} we obtain for the first term in \eqref{eigenfunctions:bd:estimate-nupartialnu-tildephii-inter1}
\begin{equation}  \label{eigenfunctions:bd:estimate-nupartialnu-tildephii-inter5}
\left|\nabla^k \left( \nu \partial_\nu \left( \frac{1}{\nu^4}\left(\phi_{i,\pm}^\inn -\Lambda U\mp \alpha^{1/2}L_i^\inn\partial_{x_1}U\right) (\frac{z\mp a}{\nu})\right)\right)\right|\lesssim  \frac{1}{(\nu+|z\mp a|)^{2+k}}\left(\frac{\nu^c}{(\nu+|z\mp a|)^c}+|\ln \nu|^{-2}\right)
\end{equation}
up to changing the constant $c>0$. Next, for the second term in \eqref{eigenfunctions:bd:estimate-nupartialnu-tildephii-inter1} we compute
$$
 \nu\partial_\nu \left(\frac{-1}{16\nu^4}\Lambda U\pm \frac{L_i}{\nu^3}\partial_{x_1}U (\frac{z\mp a}{\nu})\right)= \left(\frac{1}{16\nu^4}(y.\nabla+4)\Lambda U\mp \frac{L_i}{\nu^3}(y.\nabla+3)\partial_{x_1}U (\frac{z\mp a}{\nu})\right)
$$
Using the tail cancellation $(y.\nabla+4)\Lambda U=\Oc(|y|^{-6})$ and $(y.\nabla+3)\partial_{x_1}U=\Oc(|y|^{-5})$, and repeating the same computations as that leading to \eqref{eigenfunctions:bd:estimate-tildephii-inter5} shows
\begin{equation}  \label{eigenfunctions:bd:estimate-nupartialnu-tildephii-inter6}
\left|\nabla^k \left( \nu\partial_\nu \left(\frac{-1}{16\nu^4}\Lambda U\pm \frac{L_i}{\nu^3}\partial_{x_1}U (\frac{z\mp a}{\nu})\right)\right)\right|\lesssim \sum_\pm\frac{\nu^{2-3\tilde \epsilon}}{(\nu+|z\mp a|)^{2+k}}.
\end{equation}
For the third term in \eqref{eigenfunctions:bd:estimate-nupartialnu-tildephii-inter1} we have by \eqref{eigenfunctions-exterior:id:decomposition-phiiout} and \eqref{id:value-tilde-lambdai-tech} that $\nu\partial_\nu \phi_i^\out=\nu \partial_\nu \tilde \lambda_i Z_i=\Oc(|\ln \nu|^{-2})Z_i$. Hence by \eqref{eigenfunctions-exterior:id:decomposition-tildeZ0}, \eqref{eigenfunctions-exterior:id:decomposition-tildeZ1}and \eqref{eigenfunctions:bd:pointwise-phii-inter3} we infer:
\begin{equation}  \label{eigenfunctions:bd:estimate-nupartialnu-tildephii-inter7}
\left| \nabla^k\left(\nu \partial_\nu \phi_i^\out (1-\sum_\pm \chi_{\pm a,\nu^{\tilde \epsilon}})\right)\right|=\left| \nabla^k\left((\nu \partial_\nu \tilde \lambda_i )Z_i(1-\sum_\pm \chi_{\pm a,\nu^{\tilde \epsilon}})\right)\right|\lesssim \frac{1}{|\ln \nu|^2}\sum_\pm \frac{1}{(\nu+|z\mp a|)^{2+k}}.
\end{equation}
Finally, for the fourth term in \eqref{eigenfunctions:bd:estimate-nupartialnu-tildephii-inter1}, using \eqref{eigenfunctions:bd:phi-bo}, \eqref{eigenfunctions:bd:pointwise-phii-inter3} and the fact that $\nu \partial_\nu \chi_{\pm a,\nu^{\tilde \epsilon}}$ has support inside $\{\nu^{\tilde \epsilon}\leq |z\mp a|\leq 2\nu^{\tilde \epsilon}\}$ we have
\begin{equation}  \label{eigenfunctions:bd:estimate-nupartialnu-tildephii-inter8}
\left|\nabla^k\left( \phi_{i,\pm}^\bou(z)\nu \partial_\nu (\chi_{\pm a,\nu^{\tilde \epsilon}}(z))\right)\right|\lesssim \frac{\nu^{c''}}{(\nu+|z\mp a|)^k}
\end{equation}
for some $c''>0$ depending on $\tilde \epsilon$. Injecting \eqref{eigenfunctions:bd:estimate-nupartialnu-tildephii-inter5}, \eqref{eigenfunctions:bd:estimate-nupartialnu-tildephii-inter6}, \eqref{eigenfunctions:bd:estimate-nupartialnu-tildephii-inter7} and \eqref{eigenfunctions:bd:estimate-nupartialnu-tildephii-inter8} in  \eqref{eigenfunctions:bd:estimate-nupartialnu-tildephii-inter1} shows the second inequality in \eqref{est:pointwise_phiitil} as desired.

\smallskip

\noindent \underline{Proof of \eqref{est:pointwise_Phi_phii}.} Let $k\geq 1$. We first estimate $\tilde V_i$. The two vanishing conditions \eqref{eigenfunctions:id:vanishing-remainding-Poisson} imply that
\begin{equation} \label{eigenfunctions:id:vanishing-remainding-Poisson-tech2}
\nabla \tilde V_i(\pm a)=0 \quad \mbox{and} \quad \nabla^2 \tilde V_i(\pm a)=0
\end{equation}
for both signs $\pm$ by symmetry. We first consider the zone $\cup_\pm\{|z\mp a|<\frac12\eta\}$. We notice by \eqref{eigenfunctions:id:expression-tilde-Vi} and \eqref{eigenfunctions:bd:phi-bo-tech2} that there since $F_{i,\pm}^\bou$ is supported in $\{\eta<|z\mp a|<2\eta\}$ using $\nabla^{k}\tilde V_i=-(2\pi)^{-1}\sum_\pm \nabla^{k}( \log|\cdot|)*F_{i,\pm}^\bou$ there holds $|\nabla^{2+k} \tilde V_i|\lesssim \eta^{-2-k}\sum_\pm \| F_i^\bou\|_{L^1}\lesssim \eta^{\sqrt{5}-4-k}$. Therefore, by the vanishing \eqref{eigenfunctions:id:vanishing-remainding-Poisson-tech2} and a Taylor expansion we have $|\nabla^{k} \tilde V_i|\lesssim \nu^{\tilde \epsilon (\sqrt{5}-2)}|z\mp a|^{2}(\nu+|z\mp a|)^{-k}$.

We next consider the zone $\cap_\pm\{3\eta<|z-a|\}$ where we estimate similarly using that $F_{i,\pm}^\bou$ is supported in $\{\eta<|z\mp a|<2\eta\} $ that $|\nabla^k \tilde V_i|\lesssim \sum_{\pm}\frac{1}{|z\mp a|^{-k}}\sum_\pm \| F_{i,\pm}^\bou\|_{L^1}\lesssim \eta^{\sqrt 5}|z\mp a|^{-k}\lesssim \nu^{\tilde \epsilon(\sqrt{5}-2)}(\nu+|z\mp a|)^{2-k}$.

We finally consider the zone $\cup_\pm\{\frac 12 \eta<|z-a|<3\eta\}$. At the boundaries $\cup_\pm\{|z-a|=\frac 12\eta\}$ and $\cup_\pm\{|z-a|=3\eta\}$ we have that $\nabla^k \tilde V_i=\Oc(\nu^{\tilde \epsilon(\sqrt{5}-2)} \eta^{2-k})$. Since $\tilde V_i$ solves the Poisson equation \eqref{eigenfunctions:id:decomposition-Vi-tildeVi-elliptic}, with $\nabla^k F_{i,\pm}^\bou=\Oc(\eta^{\sqrt{5}-2-k})=\Oc( \nu^{\tilde \epsilon(\sqrt{5}-2)}\eta^{-k})$, we deduce by standard elliptic regularity estimates and a scaling argument that $\nabla^k \tilde V_i=\Oc( \nu^{\tilde \epsilon(\sqrt{5}-2)}\eta^{2-k})$.

Combining the analysis made in the three zones, we have shown that for $k\geq 1$,
\begin{equation} \label{eigenfunctions:bd:estimate-tildeVi}
|\nabla^k \tilde V_i|\lesssim \nu^{\tilde \epsilon(\sqrt{5}-2)}(\nu+|z\mp a|)^{2-k}.
\end{equation}
Th estimate \eqref{est:pointwise_Phi_phii} then follows directly from the decomposition \eqref{eigenfunctions:id:decomposition-Vi-tildeVi} and of the estimates \eqref{expansion:Vi_inn}, \eqref{eigenfunctions-exterior:id:expansion-Viiout-1} and \eqref{eigenfunctions:bd:estimate-tildeVi}.

\smallskip

\noindent \underline{Proof of \eqref{est:pointwise_Phiphiitil}}. Using successively \eqref{def:phii}, \eqref{eigenfunctions:id:decomposition-phii-final} with $L_i=-\frac{\alpha^{1/2}L_i^\inn}{16}$ and $A_i^\out=\frac{\alpha^{1/2}L_i^\inn}{4}$ by \eqref{id:value-parameters-tech} and \eqref{id:value-tilde-Aiout-tech}, and then \eqref{eigenfunctions:id:def-phi-bo} we find
\begin{align}
 \label{eigenfunctions:bd:Phitildephi:inter-1}\nabla \Phi_{\chi \tilde \phi_i} & =  - \frac{1}{16\nu^2 }  \sum_\pm  \chi_{\pm a,\nu^{\tilde \epsilon}} \nabla\left((V_{i,\pm}^\inn-\Phi_{\Lambda U}\mp \alpha^{1/2}\nu L_i^\inn \Phi_{\partial_{x_1}U}) (\frac{\cdot-a}{\nu}\right)\\
\nonumber&\quad +\left(1-\sum_\pm \chi_{\pm a,\nu^{\tilde \epsilon}}\right)\left( \nabla V_{i}^\out -\sum_\pm \left(\frac{1}{2|z\mp a|^2}\pm \nabla(\frac{A_i^\out}{|z\mp a|}\cos(\theta_\pm))\right)   \right)\\
\nonumber&\quad  +\frac{1}{\nu^3}\left(1-\sum_\pm \chi_{\pm a,\nu^{\tilde \epsilon}}\right)\sum_\pm \left(\frac{1}{16}\nabla \Phi_{\Lambda U}+\frac{1}{2|\cdot|^3}\mp L_i\nu (\nabla \partial_{x_1} \Phi_U+4\nabla(|\cdot|^{-1}\cos(\frac{\cdot}{|\cdot|}))\right)   (\frac{\cdot-a}{\nu}) \\
\nonumber&\quad  +\sum_\pm V_{i,\pm}^\bou \nabla \chi_{\pm a,\nu^{\tilde \epsilon}} +\nabla \sum_\pm \Phi_{(-\frac{1}{16\nu^4}\pm \frac{L_i}{\nu^2}\partial_{x_1}U)(\frac{\cdot\pm a}{\nu})(1-\chi_{\pm a,\zeta_*})}+\nabla \tilde V_i .
\end{align}
We now estimate all terms in the right-hand side of \eqref{eigenfunctions:bd:Phitildephi:inter-1}. We have by \eqref{eigenfunctions:id:decomposition-Vinn}:
\begin{equation}  \label{eigenfunctions:bd:Phitildephi:inter-2}
V_{i,\pm}^\inn-\Phi_{\Lambda U}\mp \alpha^{1/2}\nu L_i^\inn \Phi_{\partial_{x_1}U}=\tilde V_{i,\pm}^\inn-\Phi_{\Lambda U}+\bar V_{i,\pm}\mp \alpha^{1/2}\nu L_i^\inn \Phi_{\partial_{x_1}U}.
\end{equation}
For $k\geq 1$, using \eqref{eigenfunctions:id:tildephiinn} and then Lemmas \ref{lem:T20hatT2}, \ref{lem:T22}, \ref{lem:T3k}, \ref{lem:T40hatT40}, \ref{lem:T42hatT42} and \ref{lem:T44}, and using that $\tilde \lambda_i=\Oc(|\ln \nu|^{-1})$ we have for all $|y|\lesssim \nu^{-1+\tilde \epsilon}$:
\begin{align}
\nonumber |\nabla^{k}(\tilde V_{i,\pm}^\inn-\Phi_{\Lambda U})|&=|\nabla^k(\alpha\nu^2(V_2^{(i)}+\tilde \lambda_i \hat V_2)\pm \alpha^{\frac 32}\nu^3V_3+\alpha^2 \nu^4 (V_4^{(i)}+\tilde \lambda_i\hat V_4))|\\
\nonumber & \lesssim \nu^2 \langle y \rangle^{-k} \langle \log \langle y \rangle\rangle^i+ \nu^3 \langle y \rangle^{-1-k} \langle \log \langle y \rangle\rangle+ \nu^4 \langle y \rangle^{2-k} \langle \log \langle y \rangle\rangle\\
 \label{eigenfunctions:bd:Phitildephi:inter-3}&\lesssim \nu^2 \langle y \rangle^{-k} | \log\nu|^i.
\end{align}
We have similarly by \eqref{eigenfunctions:id:barVinn} and Proposition \ref{pr:solution-system-SW}:
\begin{equation} \label{eigenfunctions:bd:Phitildephi:inter-4}
|\nabla^k(\bar V_{i,\pm}\mp \alpha^{1/2}\nu L_i^\inn \Phi_{\partial_{x_1}U})(y)|\lesssim \nu^2 \langle y \rangle^{-1-k}  \langle \log \langle y\rangle \rangle
\end{equation}
for all $k\geq 1$ and $|y|\lesssim \nu^{-1+\tilde \epsilon}$. Injecting \eqref{eigenfunctions:bd:Phitildephi:inter-3} and \eqref{eigenfunctions:bd:Phitildephi:inter-4} in \eqref{eigenfunctions:bd:Phitildephi:inter-2} we obtain $|\nabla^k (V_{i,\pm}^\inn-\Phi_{\Lambda U}\mp \alpha^{1/2}\nu L_i^\inn \Phi_{\partial_{x_1}U})(y)|\lesssim \nu^2 \langle y \rangle^{-k} | \log\nu|^i$ for $k\geq 1$ and $|y|\lesssim \nu^{-1+\tilde \epsilon}$. In turn, combining this inequality and \eqref{eigenfunctions:bd:pointwise-phii-inter3} we get
\begin{equation}\label{eigenfunctions:bd:Phitildephi:inter-5}
\left|\nabla^{ k}\left( \frac{1}{\nu^2 }   \chi_{\pm a,\nu^{\tilde \epsilon}} \nabla\left((V_{i,\pm}^\inn-\Phi_{\Lambda U}\mp \alpha^{1/2}\nu L_i^\inn \Phi_{\partial_{x_1}U}) (\frac{\cdot-a}{\nu})\right)\right)(z)\right|\lesssim \frac{|\ln \nu|^i}{(\nu+|z\mp a|)^{k+1}}.
 \end{equation}
 We have by \eqref{eigenfunctions-exterior:id:expansion-Viiout-1}, \eqref{eigenfunctions-exterior:id:expansion-Viiout-2}, \eqref{eigenfunctions:bd:pointwise-phii-inter3}, and $|A^\out_i|,|B^\out_i|\lesssim |\ln \nu|^i$ (by \eqref{id:value-parameters-tech}, \eqref{id:value-tilde-Aiout-tech} and \eqref{id:value-tilde-Biout-tech}) that
\begin{equation}\label{eigenfunctions:bd:Phitildephi:inter-5}
\left|\nabla^{ k}\left( \left(1-\sum_\pm \chi_{\pm a,\nu^{\tilde \epsilon}}\right)\left( \nabla V_{i}^\out -\sum_\pm \left(\frac{1}{2|\cdot\mp a|^2}\pm \nabla(\frac{A_i^\out}{|z\mp a|}\cos(\theta_\pm))\right)   \right)\right)(z)\right|\lesssim \sum_\pm \frac{|\ln \nu|^i}{(\nu+|z\mp a|)^{k+1}}
\end{equation}
We have using the cancellations $\nabla \Phi_{\Lambda U}(y)-\frac{-8}{|y|^3}=\Oc(|y|^{-5})$ and $\partial_{x_1}\Phi_U(y)+\frac{4}{|y|}\cos(\frac{y}{|y|})=\Oc(|y|^{-3})$ as $|y|\to \infty$, and that $1-\sum_\pm \chi_{\pm a ,\nu^{\tilde \epsilon}}$ is supported in $\cap_\pm \{|z\mp a|\geq \nu^{\tilde \epsilon}\}$ with \eqref{eigenfunctions:bd:pointwise-phii-inter3} and $|L_i|\lesssim |\ln \nu|$ that:
\begin{align}
& \nonumber \left|\nabla^{ k}\left( \frac{1}{\nu^3}\left(1-\sum_\pm \chi_{\pm a,\nu^{\tilde \epsilon}}\right)\sum_\pm (\frac{1}{16}\nabla \Phi_{\Lambda U}+\frac{1}{2|\cdot|^3}\mp L_i\nu (\nabla \partial_{x_1} \Phi_U+4\nabla(|\cdot|^{-1}\cos(\frac{\cdot}{|\cdot|})))  (\frac{\cdot-a}{\nu})\right)(z)\right|\\
& \label{eigenfunctions:bd:Phitildephi:inter-6} \lesssim \sum_\pm \frac{\nu^{2-4\tilde{\epsilon}}}{(\nu+|z\mp a|)^{k+1}}.
\end{align}
Next, we have by \eqref{eigenfunctions:bd:V-bo} with $|\tilde L^\inn_i|,|\tilde M^\inn_i|\lesssim 1$, using that$\nabla \chi_{\pm a,\nu^{\tilde \epsilon}}$ is supported in $\cup_\pm \{\nu^{\tilde \epsilon}\leq |z\mp a |\leq 2 \nu^{\tilde \epsilon}\}$ with \eqref{eigenfunctions:bd:pointwise-phii-inter3}:
\begin{equation}\label{eigenfunctions:bd:Phitildephi:inter-7}
 \left|\nabla^{ k}\left( \sum_\pm V_{i,\pm}^\bou \nabla \chi_{\pm a,\nu^{\tilde \epsilon}}\right)\right|\lesssim \sum_\pm \frac{1}{(\nu+|z\mp a|)^k}.
\end{equation}
Using the asymptotic behaviour $|\nabla^k U(y)|\lesssim \langle y \rangle^{-4-k}$ as $|y|\to \infty$ and that $1-\chi_{\pm a,\zeta_*}$ is supported inside $\cap_{\pm}\{ |z\mp a|\geq \zeta_*\}$ we infer that $|\nabla^k((-\frac{1}{16\nu^4}\pm \frac{L_i}{\nu^2}\partial_{x_1}U)(\frac{\cdot\pm a}{\nu})(1-\chi_{\pm a,\zeta_*}))|\lesssim \langle z \rangle^{-4-k}$. Using $\Phi_f=-\frac{1}{2\pi}\log |\cdot|*f $ this implies
\begin{equation}\label{eigenfunctions:bd:Phitildephi:inter-8}
\left| \nabla^k \sum_\pm \Phi_{(-\frac{1}{16\nu^4}\pm \frac{L_i}{\nu^2}\partial_{x_1}U)(\frac{\cdot\pm a}{\nu})(1-\chi_{\pm a,\zeta_*})} \right|\lesssim \langle z \rangle^{-k}
\end{equation}
for $k\geq 1$. Injecting \eqref{eigenfunctions:bd:Phitildephi:inter-5}, \eqref{eigenfunctions:bd:Phitildephi:inter-6}, \eqref{eigenfunctions:bd:Phitildephi:inter-7}, \eqref{eigenfunctions:bd:Phitildephi:inter-8} and \eqref{eigenfunctions:bd:estimate-tildeVi} in \eqref{eigenfunctions:bd:Phitildephi:inter-1} shows the desired bound \eqref{est:pointwise_Phiphiitil}.

\smallskip

\noindent \underline{Proof of \eqref{est:pointwise_phi1m0}.} By \eqref{def:phii} we have
$$
\phi_1-\phi_0=\tilde \phi_1-\tilde \phi_0+\sum_\pm \pm \frac{L_1-L_0}{\nu^3}\partial_{x_1} U\left(\frac{z\mp a}{\nu}\right)\chi_{\pm a,\zeta_*}
$$
so that the first inequality in \eqref{est:pointwise_phi1m0} follows from \eqref{est:pointwise_phiitil} and the estimates $|L_0|,|L_1|\lesssim| \ln \nu|$ and $|\nabla^k \partial_{x_1}U(y)=\Oc(\langle y \rangle^{-5-k})$. 

We have by \eqref{eigenfunctions:id:decomposition-Vi-tildeVi}
$$
\Phi_{\phi_1-\phi_0}=  -\frac{1}{16\nu^{4}}\sum_\pm (V_{1,\pm}^\inn-V_{0,\pm}^\inn) \Big(\frac{z\mp a}{\nu}\Big)  \chi(\frac{z\mp a}{\eta}) + (V_{1}^\out(z)-V_{0}^\out(z) ) \big(1 - \sum_\pm \chi(\frac{z\mp a}{\eta})\big)+\tilde V_1-\tilde V_0.
$$
We decompose using \eqref{def:Vi_inn}:
\begin{align} \label{eigenfunction:id:decomposition-V1-V0}
V_{1,\pm}^\inn-V_{0,\pm}^\inn= \sum_{i=0,1} \left(\alpha \nu^2 V_2^{(i)}+\alpha^2 \nu^4 V_4^{(i)}+\tilde \lambda_1\nu^2 \hat V_2+\tilde \lambda_i \alpha^2 \nu^4 \hat V_4^{(i)}+\bar V_{i,\pm}^\inn\right).
\end{align}
We estimate using \eqref{eigenfunctions:id:expansion-V2i}, \eqref{eigenfunctions:id:expansion-V4i}, \eqref{eigenfunctions:id:expansion-hatV2} and \eqref{eigenfunctions:id:expansion-hatV4}
\begin{align*}
& \alpha \nu^2 V_2^{(i)}+\alpha^2 \nu^4 V_4^{(i)}+\tilde \lambda_1\nu^2 \hat V_2+\tilde \lambda_i \alpha^2 \nu^4 \hat V_4^{(i)} \\
&=\Oc(\nu^2\langle \ln^2 \langle y \rangle\rangle)+\Oc(\nu^4 \langle y \rangle^2 \langle \ln \langle y \rangle \rangle)+\Oc(\frac{\nu^2}{|\ln \nu|} \langle \ln \langle y \rangle \rangle)+\Oc(\frac{\nu^4}{|\ln \nu|}\langle y \rangle^2 \langle \ln \langle y \rangle \rangle)=\Oc(\nu \langle y \rangle^{-1})
\end{align*}
for $|y|\lesssim \nu^{-1+\tilde \epsilon}$. We have using \eqref{eigenfunctions:id:expansion-barVinn} and $|L_i^\inn|,|M_i^\inn|\lesssim |\ln \nu|$:
$$
\bar V_{i,\pm}^\inn (y)=\Oc(\nu |\ln \nu|\langle y \rangle)^{-1}+\Oc(\nu^3 |\ln \nu|\langle y \rangle)+\Oc(\nu^4 |\ln \nu|\langle y \rangle^2)+\Oc(\nu \langle y \rangle^{-2}\langle \ln \langle y \rangle \rangle^2)=\Oc(\nu |\ln \nu|\langle y \rangle)
$$
for $|y|\lesssim \nu^{-1+\tilde \epsilon}$. Combining, we get $V_{1,\pm}^\inn-V_{0,\pm}^\inn=\Oc(\nu |\ln \nu|\langle y \rangle)$ for $|y|\lesssim \nu^{-1+\tilde \epsilon}$, so that
\begin{equation} \label{eigenfunction:id:decomposition-V1-V0-tech1}
\left| \nabla^k \left(  -\frac{1}{16\nu^{4}}\sum_\pm (V_{1,\pm}^\inn-V_{0,\pm}^\inn) \Big(\frac{z\mp a}{\nu}\Big)  \chi(\frac{z\mp a}{\eta}) \right)\right|\lesssim \sum_\pm \frac{|\ln \nu|}{(\nu+|z\mp a|)^{1+kl}}.
\end{equation}
The decomposition \eqref{eigenfunctions-exterior:id:expansion-Viiout-1} shows a cancellation for the leading term in $V_1^{\out}-V_0^\out$, so that since $|A_i^\out|,|B_i^\out|,|C_i^\out|,|L_i^\inn|,|M_i^\inn|\lesssim |\ln\nu|$, we have $V_1^{\out}-V_0^\out=\Oc(|\ln \nu| \sum_{\pm} \frac{\langle z\rangle \langle \ln \langle z \rangle \rangle}{|z\mp a|})$ and hence
\begin{equation} \label{eigenfunction:id:decomposition-V1-V0-tech2}
\left| \nabla^k \left((V_{1}^\out(z)-V_{0}^\out(z) ) \big(1 - \sum_\pm \chi(\frac{z\mp a}{\eta})\big)\right)\right|\lesssim|\ln \nu| \sum_{\pm} \frac{\langle z\rangle}{|z\mp a|^{1+k}}.
\end{equation}
Combining \eqref{eigenfunctions:bd:estimate-tildeVi} and the fact that $\nabla^k(\tilde V_1-\tilde V_2)=\Oc(\langle z \rangle^{-k})$ for $|z|\gg 1$ we have $\nabla^k(\tilde V_1-\tilde V_2)=\Oc((\nu+|z- a|)^2(\nu+|z+ a|)^2\langle z \rangle^{4-k})$. Injecting this inequality, \eqref{eigenfunction:id:decomposition-V1-V0-tech1} and \eqref{eigenfunction:id:decomposition-V1-V0-tech2} shows the second inequality in \eqref{est:pointwise_phi1m0}.

\smallskip

\noindent \underline{Proof of \eqref{est:pointwise_phiitil_da} and \eqref{est:pointwise_phi1m0_dnu}.} They follow from \eqref{eigenfunctions-exterior:id:expansion-phiiout-1-bd}, \eqref{eigenfunctions-exterior:bd:expansion-phiiout-infty}, \eqref{expansion:phii_inn-estimate}, \eqref{expansion:phii_inn-estimate-partialnu-partiala-1}
and \eqref{expansion:phii_inn-estimate-partialnu-partiala-2} via direct computations.

\medskip

\noindent \textbf{Step 3}. \emph{Estimate for the remainder in the approximate eigenfunction equation}. Injecting the estimate \eqref{eigenfunctions:id:estimatetildeLtildeM} in \eqref{eigenfunctions:bd:V-bo}, we obtain that
\begin{align} \label{eigenfunctions:bd:phi-bo-tech2}
&|\nabla^k \phi_{i,\pm}^\bou (z)|\lesssim \eta^{\sqrt{5}-2-k},\\
& \label{eigenfunctions:bd:V-bo-tech2} |\nabla^k V_{i,\pm}^\bou (z)|\lesssim   \eta^{\sqrt{5}-k} ,
\end{align}
for all $z$ such that $|z\mp a|\approx \eta$.

We inject the two decompositions \eqref{eigenfunctions:id:decomposition-phii-final} and \eqref{eigenfunctions:id:decomposition-Vi-tildeVi} in the approximate eigenfunction equation \eqref{eq:eigenfunctions_phi10}, using \eqref{eigenfunctions-interior:id:eigenfunction-system} and \eqref{eigenfunctions-exterior:id:eigenfunction-system} and get
\begin{align*}
R_i & = -\frac{1}{16}\sum_{\pm}\chi(\frac{z\mp a}{\eta})\frac{1}{\nu^6}R_i^\inn (\frac{z\mp a}{\nu})+(1-\sum_\pm \chi(\frac{z\mp a}{\eta}))R_i^\out\\
& +\sum_\pm \left(\frac 2 \eta \nabla \chi(\frac{z\mp a}{\eta}) . \nabla \phi_{i,\pm}^\bou+\frac{1}{\eta^2} \Delta \chi(\frac{z\mp a}{\eta})  \phi_{i,\pm}^\bou\right)-\left(\beta z+\nabla \Phi_{U_{1+2,\nu}}\right).\frac{1}{\eta}\nabla  \chi(\frac{z\mp a}{\eta})  \phi_{i,\pm}^\bou \\
& - \nabla \tilde V_i.\left(\sum_\pm \frac{1}{\nu^3}\nabla U\left(\frac{z\mp a}{\nu}\right)\right)
\end{align*}
We estimate the first line using \eqref{estimate:Ri_inn-pointwise} and \eqref{eigenfunctions-exterior:bd:Ri}:
\begin{align*}
&\left| \nabla^k\left( -\frac{1}{16}\sum_{\pm}\chi(\frac{z\mp a}{\eta})\frac{1}{\nu^6}R_i^\inn (\frac{z\mp a}{\nu})+(1-\sum_\pm \chi(\frac{z\mp a}{\eta}))R_i^\out\right)\right| \\
&\lesssim \sum_\pm \left(\frac{1}{|\log \nu|^2}\frac{1}{(\nu+|z\mp a|)^{2+k}}+\frac{1}{(\nu+|z\mp a|)^{1+k}} \right)\mathbbm 1(|z\mp a|<2 \nu^{\tilde \epsilon}) \\
& \qquad \qquad +\frac{1}{|\log \nu|^2}\sum_\pm \frac{\langle z\rangle^{C_k}}{|z\mp a|^{2+k}}\mathbbm 1(|z\mp a|> \nu^{\tilde \epsilon})\\
&\lesssim \frac{1}{|\log \nu|^2}\sum_\pm \frac{\langle z\rangle^{C_k}}{(\nu+|z\mp a|)^{2+k}}.
\end{align*}
We estimate the second line using \eqref{eigenfunctions:bd:phi-bo-tech2} and $\nabla \Phi_{U_{1+2,\nu}}=\Oc(\sum_\pm (\nu+|z\mp a|)^{-1})$:
\begin{align*}
& \left| \nabla^k \left(\sum_\pm \left(\frac 2 \eta \nabla \chi(\frac{z\mp a}{\eta}) . \nabla \phi_{i,\pm}^\bou+\frac{1}{\eta^2} \Delta \chi(\frac{z\mp a}{\eta})  \phi_{i,\pm}^\bou\right)-\left(\beta z+\nabla \Phi_{U_{1+2,\nu}}\right).\frac{1}{\eta}\nabla  \chi(\frac{z\mp a}{\eta})  \phi_{i,\pm}^\bou\right)\right|\\
&\lesssim \eta^{-\sqrt{5}-4-k} \sum_\pm \mathbbm 1(\eta<|z\mp a|<2\eta) \ \lesssim \nu^{\tilde \epsilon (\sqrt{5}-2)}\sum_{\pm}\frac{1}{(\nu+|z\mp a|)^{2+k}}.
\end{align*}
We estimate the third term using $\nabla^k (U(\frac{z\mp a}{\nu}))\lesssim \nu^4(\nu+|z\mp a|)^{-4-k}$ and \eqref{eigenfunctions:bd:estimate-tildeVi} so that
$$
\left|\nabla^k\left(- \nabla \tilde V_i.\left(\sum_\pm \frac{1}{\nu^3}\nabla U\left(\frac{z\mp a}{\nu}\right)\right)\right)\right|\lesssim \nu^{\tilde \epsilon(\sqrt{5}-2)+2}(\nu+|z\mp a|)^{1-5-k}\lesssim \nu^{\tilde \epsilon(\sqrt{5}-2)}(\nu+|z\mp a|)^{-2-k}
$$
Gathering the above estimates shows \eqref{est:pointwise_Ri}.
\end{proof}

\section{Coercivity of the linearized operator} \label{sec:coercivity}

In self-similar variables \eqref{def:wztau_vars-1} and around $U_{1+2,\nu}$ given by \eqref{def:fnuU12}, the linearized operator is
$$
\Ls^z u =\Delta u- \nabla. \Big( u  \nabla \Phi_{U_{1 + 2, \nu}}+ U_{1 + 2, \nu}\nabla \Phi_u \Big) - \frac 12 \Lambda u.
$$
In this section we construct a matched scalar product for which $\tilde{\Ls}^z$, which is the operator $\Ls^z$ with a slight truncation of the Poisson field defined by \eqref{def:LszNottilde}, is coercive on a finite co-dimensional set.

\subsection{Matched scalar products}  \label{subsec:matched-scalar-product}

\subsubsection{The scalar product for the interior zone}

We present in this subsection a scalar product for which the linearized operator $\Ls^z$ is to leading order symmetric, for functions that are located around $\pm a$. Namely, consider the bilinear functional
\begin{equation} \label{def:scalar-interior}
\langle u, v \rangle_\flat= \int_{\Rb^2} u\Ms^z v \, dz 
\end{equation}
where we introduce the operator 
\begin{equation}\label{def:M12z}
\Ms^z u  = \gamma_{\nu} u  - \nu^2 \Phi_u, \qquad \gamma_\nu(z)= \nu^2 / U_{1+2, \nu}(z)
\end{equation}
We note that the weight $\gamma_\nu$ satisfies
\begin{equation} \label{bd:equivalence-weight-U1+2nu-1}
\gamma_\nu(z) \approx (\nu+|z-a|)^4(\nu+|z+a|)^4\langle z\rangle^{-4}.
\end{equation}
The motivation of introducing \eqref{def:scalar-interior} comes from the fact that there is the free energy functional for positive solutions of the problem in original variables introduced in \cite{NSYfe97} (see also \cite{BDPjde06}, \cite{Bamsa98}, \cite{GZmn98})
\begin{equation}
\Fc[u] = \int_{\Rb^2} u \Big(\ln u - \frac{1}{2}\Phi_u \Big) dx.
\end{equation}  
It decreases along the flow for positive solutions with enough regularity and integrability as
$$
\frac{d}{dt}\Fc[u] = -\int_{\Rb^2} u \Big|\nabla \ln u - \nabla \Phi_u \Big|^2dx.
$$
The quadratic form \eqref{def:scalar-interior} is that associated to the Hessian of this free energy functional at $U_{1+2,\nu}$. 

We first show that $\langle \cdot, \cdot \rangle_\flat$ is well defined in $L^2(\gamma_\nu dz)$, where $L^2(\gamma_\nu dz)$ is the weighted $L^2$ space equipped with the scalar product 
\begin{equation}\label{def:L2gamma}
\langle u, v \rangle_{\gamma_\nu} = \int_{\Rb^2} u v \gamma_\nu dz, \quad \|u\|_{\gamma_\nu} = \sqrt{\langle u, u \rangle_{\gamma_\nu}}
\end{equation}
and it is equivalent to $\langle \cdot,\cdot \rangle_{\gamma_\nu}$ on a suitable subspace.

\begin{proposition}[Scalar product for the interior zone]  \label{pr:scalar-product-interior1}

There exists $C>0$ such that the following hold for $u,v\in L^2( \gamma_\nu dz)$. The functional $\langle \cdot,\cdot \rangle_\flat$ is symmetric
\begin{equation} \label{scalarproduct1:symmetry}
\langle u,v\rangle_\flat=\langle v ,u\rangle_\flat,
\end{equation}
and it is continuous for $|a|>0$ and $\nu$ small enough,
\begin{equation} \label{scalarproduct1:continuity}
|\langle u,v\rangle_\flat| \leq C \|u\|_{\gamma_\nu} \|v\|_{\gamma_\nu}.
\end{equation}
Moreover, let $\eta>0$ small enough and $u$ satisfy the orthogonality conditions
\begin{equation} \label{scalarproduct1:orthogonality-condition}
 \langle u, \mathbbm{1}(|z+\iota a|<\eta) \rangle =  \langle u, (\Lambda U)_{i, \nu}\mathbbm{1}(|z+\iota_i a|<\eta) \rangle =\langle u, \pa_j U_{i, \nu}  \mathbbm{1}(|z+\iota_i a|<\eta) \rangle =0,
\end{equation}
for all $\iota\in \{\pm1\}$, $i=1,2$, $j=1,2$, where $\iota_i=(-1)^i$. Then, the quadratic form $\langle u , u \rangle_\flat $ is positive
\begin{equation} \label{scalarproduct1:positivity}
\|u\|^2_{\gamma_\nu} \leq C \langle u,u\rangle_\flat.
\end{equation}
\end{proposition}

\medskip 

\noindent We then show that $\Ls^z$ can be decomposed as the sum of two operators, one which is symmetric and coercive for $\langle \cdot,\cdot \rangle_\flat$, and the other which is negligible for functions that are localized near $\pm a$.

\begin{lemma}[Coercivity and almost symmetry in the interior zone] \label{lem:almost-symmetry-Lsz-interior}
The operator $\Ls^z$ can be written
\begin{equation} \label{def:L12z_M12}
\Ls^z q  = \nabla. \Big[\gamma_\nu^{-1} \nabla \Ms^z  q \Big]+\nabla.\big[ q \nabla V \big]
\end{equation}
where
\begin{equation} \label{def:V:L12z}
V =  \ln \frac{U_{1+2,\nu}}{U_{12,\nu}} - \frac 14 z^2.
\end{equation}
For Schwartz functions $u$ and $v$, we have
\begin{equation} \label{id:symmetry-Lsz-interior}
\langle \Ls^z u,v\rangle_\flat = - \int_{\Rb^2} \gamma_\nu^{-1} \nabla \Ms^z  u \cdot \nabla \Ms^z  v  dz+\int \nabla .(u\nabla V) \Ms^z  v dz.
\end{equation}
Moreover, for $|a|\approx 1$, there exists $\delta>0$ such that for $\eta',\eta>0$ small enough, if $u$ has support inside $\{|z-a|< \eta'\}\cup \{|z+a|<\eta'\}$ and satisfies the orthogonality conditions 
\begin{equation} \label{scalarproduct1:orthogonality-condition-coercivity-operator}
  \langle u, (\Lambda U)_{i, \nu}\mathbbm{1}(|z+\iota_i a|<\eta) \rangle  =\langle u, \pa_j U_{i, \nu}  \mathbbm{1}(|z+\iota_i a|<\eta) \rangle =0,
\end{equation}
for all $\iota\in \{\pm1\}$, $i=1,2$, $j=1,2$, where $\iota_i=(-1)^i$, there holds
\begin{equation} \label{bd:coercivity-Lsz-interior}
-\langle \Ls^z u,u\rangle_\flat \geq \delta \int |\nabla u|^2  \gamma_\nu dz.
\end{equation}
\end{lemma}

\begin{remark}If the stationary states are in their equilibrium position $a=(2,0)$, then a direct computation shows that $\nabla V(\pm a)=O(\nu^2)$. Hence, by \eqref{def:L12z_M12} we have that $\Ls^z u$ is well approximated by the operator $\nabla. \Big[\gamma_\nu^{-1} \nabla \Ms^z  u \Big]$, for functions $u$ that are localized near $\pm a$. We remark that this operator is symmetric for the scalar product $\langle \cdot,\cdot \rangle_\flat$ and that its associated quadratic form is the first term in the right-hand side of \eqref{id:symmetry-Lsz-interior}. This is what motivated the introduction of the scalar product \eqref{def:scalar-interior} to study the coercivity of $\Ls^z$. Note that the coercivity \eqref{bd:coercivity-Lsz-interior} holds even if $a$ is not in the equilibrium position, as the term $\nabla.[ q \nabla V ]$ is negligible with respect to the dissipation effect of the first term in \eqref{def:L12z_M12}.
\end{remark}

\medskip

We now proceed with the proof of Proposition \ref{pr:scalar-product-interior1} and Lemma \ref{lem:almost-symmetry-Lsz-interior}. In our analysis, we shall rely numerous times on the following inequality.

\begin{lemma}[Hardy inequality in $L^2(\gamma_\nu dz)$] \label{lem:hardy-interior-product}
For all $|a|\sim 1$, there exists $C>0$ such that for all $u\in H^1(\gamma_\nu  dz)$ we have
\begin{equation} \label{bd:hardy-interior-product}
\int_{\Rb^2} \left(\frac{1}{(\nu+|z-a|)^2}+\frac{1}{(\nu+|z+a|)^2}\right)u^2 \gamma_\nu dz \leq C\int_{\Rb^2} |\nabla u|^2 \gamma_\nu dz.
\end{equation}

\end{lemma}

\begin{proof}[Proof of Lemma \ref{lem:hardy-interior-product}]

Let $\chi$ be a smooth cut-off function with $\chi(z)=1$ for $|z|\leq 1$ and $\chi(z)=0$ for $|z|\geq 2$. For $\eta>0$ we decompose
\begin{align*}
u(z) & =\chi(\frac{z-a}{\eta})u+\chi(\frac{z+a}{\eta})u+(1-\chi(\frac{z-a}{\eta})-\chi(\frac{z-a}{\eta}))u \\
&=u_1+u_2+u_3
\end{align*}

\noindent \underline{Hardy inequality for $u_1$ and $u_2$}. We claim that
\begin{equation} \label{bd:hardy-interior-inter1}
\int \left(\frac{u_1^2}{(\nu+|z-a|)^2}+\frac{u_2^2}{(\nu+|z+a|)^2}\right)\gamma_\nu dz\lesssim \int  (|\nabla u_1|^2+|\nabla u_2|^2 ) \gamma_\nu dz.
\end{equation}
We now prove \eqref{bd:hardy-interior-inter1}. We estimate for $|z-a|\leq 2\eta $,
\begin{equation} \label{innerproduct:bd:I-second-term}
\left| \gamma_\nu(z) -\frac{\nu^2}{U_{1,\nu}(z)} \right|=\gamma_\nu(z) \frac{U_{2,\nu}(z)}{U_{1,\nu}(z)}\lesssim \frac{\nu^2}{U_{1,\nu}} \eta^{4}
\end{equation}
since $|U_{2,\nu}(z)|\lesssim \nu^2$ and $U_{1+2,\nu}(z)\geq U_{1,\nu}(z)\approx \nu^{-2}\langle \frac{z-a}{\nu}\rangle^{-4}$.  As $u_1$ is supported for $|z-a|\leq 2\eta$, we have by using \eqref{innerproduct:bd:I-second-term},
\begin{equation} \label{bd:hardy-interior-inter2}
\int \frac{u_1^2}{(\nu+|z-a|)^2}\gamma_\nu dz \approx \int \frac{u_1^2}{(\nu+|z-a|)^2}\frac{\nu^2dz}{U_{1,\nu}}  \quad \mbox{and} \quad   \int  |\nabla u_1|^2\gamma_\nu dz \approx \int  |\nabla u_1|^2 \frac{\nu^2dz}{U_{1,\nu}}.
\end{equation}
We now change variables and let $z=a+\nu y$ and $u_1(z)=v_1(y)$ so that
\begin{equation} \label{bd:hardy-interior-inter3}
 \int \frac{u_1^2}{(\nu+|z-a|)^2}\frac{\nu^2dz}{U_{1,\nu}} = \nu^4 \int \frac{v_1^2}{(1+|y|)^2}\frac{dy}{U} \quad \mbox{and} \quad   \int  |\nabla u_1|^2 \frac{\nu^2dz}{U_{1,\nu}}= \nu^4 \int  |\nabla v_1|^2 \frac{dy}{U}.
\end{equation}
We recall the Hardy-type inequality $\int v_1^2 \langle y \rangle^{-2}\frac{dy}{U}\lesssim \int |\nabla v_1|^2\frac{dy}{U}$, see \eqref{bd:hardy-L2U-1}. Combining this inequality with \eqref{bd:hardy-interior-inter2} and \eqref{bd:hardy-interior-inter3} yields
$$
\int \frac{u_1^2}{(\nu+|z-a|)^2} \gamma_\nu dz\lesssim \int  |\nabla u_1|^2 \gamma_\nu dz.
$$
The analogue inequality holds for $u_2$ by symmetry, which proves \eqref{bd:hardy-interior-inter1}.

\smallskip

\noindent \underline{Hardy inequality for $u_3$}. We claim that 
\begin{equation} \label{bd:hardy-interior-inter4}
\int \left(\frac{1}{(\nu+|z-a|)^2}+\frac{1}{(\nu+|z+a|)^2}\right)u^2_3 \gamma_\nu  dz\lesssim \int |\nabla u_3|^2\gamma_\nu dz.
\end{equation}
Indeed, we have $\gamma_\nu \approx \langle z\rangle^4$ for $z$ in the support of $u_3$ by \eqref{bd:equivalence-weight-U1+2nu-1}. The inequality \eqref{bd:hardy-interior-inter4} then follows from this equivalence and the Hardy-type inequality \eqref{bd:hardy-L2U-1}.

\smallskip

\noindent \underline{Hardy inequality for $u$}. Combining \eqref{bd:hardy-interior-inter1} and \eqref{bd:hardy-interior-inter4} yields
\begin{equation} \label{bd:hardy-interior-inter5}
\int \left(\frac{1}{(\nu+|z-a|)^2}+\frac{1}{(\nu+|z+a|)^2}\right)u^2 \gamma_\nu dz\leq C \int (|\nabla u_1|^2+|\nabla u_2|^2+|\nabla u_3|^2)\gamma_\nu dz.
\end{equation}
We  decompose
$$
 \int|\nabla u|^2\gamma_\nu dz = \int (|\nabla u_1|^2+|\nabla u_2|^2+|\nabla u_3|^2)\gamma_\nu dz +2 \int ( \nabla u_1+\nabla u_2).\nabla u_3\gamma_\nu dz .
$$
Injecting \eqref{bd:hardy-interior-inter5} in the above identity, we have that for any $\kappa>0$,
\begin{align}
\nonumber \int|\nabla u|^2\gamma_\nu dz & = (1-\kappa)\int|\nabla u|^2\gamma_\nu dz + \kappa \int|\nabla u|^2\gamma_\nu dz\\
\nonumber & \geq (1-\kappa)\int|\nabla u|^2\gamma_\nu dz +\frac{\kappa }{C}  \int \left(\frac{1}{(\nu+|z-a|)^2}+\frac{1}{(\nu+|z+a|)^2}\right)u^2\gamma_\nu dz \\
 \label{bd:hardy-interior-inter6} &\qquad +2\kappa  \int ( \nabla u_1+\nabla u_2).\nabla u_3\gamma_\nu dz.
 \end{align}
We have that $|\nabla u_i|\lesssim |\nabla u|+|u|$ for $i=1,2,3$. Hence, using the inequality $ab\leq \sqrt{\kappa}\frac{a^2}{2}+\frac{b^2}{2\sqrt{\kappa}}$ we have 
\begin{equation} \label{bd:hardy-interior-inter7}
2\kappa  \int ( \nabla u_1+\nabla u_2).\nabla u_3 \gamma_\nu dz \geq -C \kappa^{\frac 12}\int |\nabla u|^2 \gamma_\nu dz-C\kappa^{\frac 32}\int_{\{|z\pm a|\leq 2\eta\}} | u|^2 \gamma_\nu dz.
\end{equation}
Injecting \eqref{bd:hardy-interior-inter7} in \eqref{bd:hardy-interior-inter6} and taking $\kappa$ small enough shows the desired inequality \eqref{bd:hardy-interior-product}.
\end{proof}

The proof of Proposition \ref{pr:scalar-product-interior} is intimately related to the following spectral and coercivity properties for the operator $\Ms_0$ around the stationary state $U$. The following originates from \cite{RSma14} (see also \cite{CGMNapde22}).

\begin{lemma}[Kernel and coercivity of $\Ms_0$] \label{lem:coercivity-one-bubble-H1-orthogonality} We have
\begin{equation} \label{id:kernel-Ms0}
\Ms_0 \Lambda U=-2 \qquad \mbox{and} \qquad \Ms_0 \nabla U=0.
\end{equation}
Let $u$ be  a Schwartz function with $\int u = 0$. Then, we have
\begin{equation}\label{est:positivityofMs0}
\langle u, \Ms_0 u \rangle	 \geq 0.
\end{equation}
Moreover, there exists a constant $\delta > 0$ such that for any $R \in [1, \infty]$, if $u$ satisfies in addition
\begin{equation} \label{bd:coercivity-one-bubble-H1-orthogonality}
\int_{|y|\leq R} u\Lambda U dy=\int_{|y|\leq R} u \partial_{y_1}U dy=\int_{|y|\leq R} u \partial_{y_2}U dy=0,
\end{equation}
where $\{|y|\leq R\}=\mathbb R^2$ if $R=\infty$, then
\begin{equation} \label{bd:coercivity-one-bubble-H1}
\langle u, \Ms_0 u \rangle \geq \delta \int \frac{|u|^2}{U}dy \quad \mbox{ and }\quad  \int U |\nabla \Ms_0 u|^2 dy \geq \delta \int \frac{|\nabla u|^2}U dy.
\end{equation}
\end{lemma}

\begin{proof} The identity \eqref{id:kernel-Ms0}, the inequality \eqref{est:positivityofMs0} and the first inequality in \eqref{bd:coercivity-one-bubble-H1} for $R=\infty$ are proved in \cite{RSma14} (see Proposition 2.3). We only prove the second inequality in \eqref{bd:coercivity-one-bubble-H1}, as the proof of the first inequality for $R\in [1,\infty)$ is the same. We recall the subcoercivity estimate from \cite{CGMNapde22} (estimate (3.8)), 
\begin{equation} \label{bd:coercivity-one-bubble-H1-inter1}
\int U |\nabla \Ms_0 u|^2 dy \geq\frac 12 \int |\nabla u|^2\frac{dy}{U}-C \int u^2.
\end{equation}
We now assume by contradiction that \eqref{bd:coercivity-one-bubble-H1} is false. This means that there exist $R_n\in [1,\infty]$ and $u_n$ a sequence of Schwartz functions satisfying \eqref{bd:coercivity-one-bubble-H1-orthogonality} with $R=R_n$, such that 
\begin{equation}\label{bd:coercivity-one-bubble-H1-inter1}
\int |\nabla u_n|^2\frac{dy}{U}=1 \quad \mbox{for all } n\in \mathbb N \qquad \mbox{and}\quad  \int U |\nabla \Ms_0 u_n|^2 dy\to 0.
\end{equation}
We introduce the space $\dot H^1_{U^{-1}}$ which is the completion of $C^\infty_c(\mathbb R^2)$ for the norm $\sqrt{\int |\nabla u|^2U^{-1}dy}$. We may assume, up to a subsequence, that $R_n\to R_\infty \in [1,\infty]$ and that there exists $u_\infty\in \dot H^1_{U^{-1}}$ such that $u_n\rightarrow u_\infty$ in $L^2_{loc}$ and $\nabla u_n \rightharpoonup \nabla u_\infty$ in $L^2(U^{-1}dy)$. We claim that $u_\infty$ satisfies the following three properties:
\begin{equation}\label{bd:coercivity-one-bubble-H1-inter2}
u_\infty \neq 0, \qquad \int U |\nabla \Ms_0 u_\infty|^2dy =0 \qquad \mbox{and }\eqref{bd:coercivity-one-bubble-H1-orthogonality}\mbox{ with }R=R_\infty.
\end{equation}
We now prove \eqref{bd:coercivity-one-bubble-H1-inter2}. By the Hardy inequality \eqref{bd:hardy-L2U-1}, we have for $\epsilon>0$ small that $\int_{|y|\geq \epsilon^{-1}}u^2_n \lesssim \epsilon^2$ for all $n$. By \eqref{bd:coercivity-one-bubble-H1-inter1} and \eqref{bd:coercivity-one-bubble-H1-inter1} we have that $ \int u_n^2\geq c$ for all $n$ for some $c>0$. Combining, there exists $\epsilon_0>0$ such that $\int_{|y|\leq \epsilon_0^{-1}} u_n^2 \geq c/2$ for all $n$. By the compactness of the Sobolev embedding from $H^1(|y|\leq \epsilon_0^{-1})$ into $L^2$, $u_n$ converges strongly in $L^2(|y|\leq \epsilon_0^{-1})$ towards $u_\infty$. Thus, $ \int u_\infty^2\geq c/2$ which implies the first property in \eqref{bd:coercivity-one-bubble-H1-inter2}.

Next, by decomposing $-2\pi \nabla \Phi_u = -[\chi (y)\log |y|]*\nabla u+[\nabla( (1-\chi (y))\log |y|)]* u$, using Cauchy-Schwarz and the Hardy inequality \eqref{bd:hardy-L2U-1} we bound
\begin{equation} \label{scalarproduct1:temp3000}
|\nabla \Phi_u(y)|^2 \lesssim \langle y\rangle^{-2}\int_{\mathbb R^2} |\nabla u|^2 \frac{1}{U(\tilde y)}d\tilde y.
\end{equation}
Using $\nabla \Ms_0 u=\nabla (u/U)-\nabla \Phi_{u}$, by \eqref{scalarproduct1:temp3000} and the Hardy inequality \eqref{bd:hardy-L2U-1}, the map $u\mapsto \nabla \Ms_0 u$ is continuous from $\dot H^1_{U^{-1}}$ into $L^2(Udy)$. As $u_n \rightharpoonup  u_\infty$ in $\dot H^1_{U^{-1}}$, this implies $\nabla \Ms_0 u_n\rightharpoonup \nabla \Ms_0 u_\infty$ in $L^2(Udy)$. The second property in \eqref{bd:coercivity-one-bubble-H1-inter2} thus follows from \eqref{bd:coercivity-one-bubble-H1-inter1} and semi-continuity.

Then, as  $\nabla u_n \rightharpoonup  \nabla u_\infty$ in $L^2_{U^{-1}}$, we have that $ u_n$ converges weakly towards $ u_\infty$ in $L^2$ by the Hardy inequality \eqref{bd:hardy-L2U-1}. As $\Lambda U,\nabla U\in L^2$, we can pass to the limit in the orthogonality \eqref{bd:coercivity-one-bubble-H1-orthogonality} for $u_n$, implying the third property in \eqref{bd:coercivity-one-bubble-H1-inter1}. This concludes the proof of \eqref{bd:coercivity-one-bubble-H1-inter2}.

To conclude, we have $u_\infty \in \dot H^1(U^{-1}dy)  $ and $\int U |\nabla \Ms_0 u_\infty|^2 dy=0$ by \eqref{bd:coercivity-one-bubble-H1-inter2}. Hence $\nabla (u_{\infty}/U)=\nabla \Phi_{u_{\infty}}$, which by \eqref{scalarproduct1:temp3000} implies $|\nabla (u_{\infty}/U)(y)|\lesssim \langle y \rangle^{-1}$. Integrating this inequality and using $|U(y)|\lesssim \langle y \rangle^{-4}$ shows $|u_{\infty}(y)|\lesssim \langle y \rangle^{-4}\langle \ln \langle y \rangle \rangle$. This implies that $| \Phi_{u_{\infty}}(y)|\lesssim |\log |y||* |u_{\infty}|\lesssim \langle  \ln \langle y \rangle \rangle$. By Lemma 2.1 of \cite{RSma14},  there holds $u_\infty\in \mbox{Span}( \Lambda U,\frac{\partial}{\partial x_1}U,\frac{\partial}{\partial x_2}U)$. By the third property in \eqref{bd:coercivity-one-bubble-H1-inter2}, this implies $u_\infty=0$. But this contradicts the first property in \eqref{bd:coercivity-one-bubble-H1-inter2}.

\end{proof}

We now prove Proposition \ref{pr:scalar-product-interior1} and Lemma \ref{lem:almost-symmetry-Lsz-interior}.

\begin{proof}[Proof of Proposition \ref{pr:scalar-product-interior1}]

\textbf{Step 1}. \emph{Proof of the symmetry identity \eqref{scalarproduct1:symmetry}}. By \eqref{def:M12z} we have
\begin{equation} \label{scalarproduct1:expression}
\langle u,v\rangle =  \int uv \gamma_\nu \, dz  - \nu^2 \int  u \Phi_v \, dz .
\end{equation}
Since $\Delta \Phi_f=-f$, the identity \eqref{scalarproduct1:symmetry} then follows from \eqref{scalarproduct1:expression} and an integration by parts.

\medskip

\noindent \textbf{Step 2}. \emph{Proof of the continuity estimate \eqref{scalarproduct1:continuity}}. We claim that
\begin{equation} \label{scalarproduct1:temp-1}
 \int \nu^2 |\Phi_u|^2 U_{1+2,\nu} dz \lesssim  \int u^2 \gamma_\nu dz .
\end{equation}
The continuity estimate \eqref{scalarproduct1:continuity} will then follows from applying the Cauchy-Schwarz inequality and \eqref{scalarproduct1:temp-1} to \eqref{scalarproduct1:expression}. Hence, it remains to show \eqref{scalarproduct1:temp-1}. For that we introduce a small parameter $\eta>0$ (we name it $\eta$ since our computations will also be used in the proof of \eqref{scalarproduct1:positivity}) and decompose:
\begin{align}
\nonumber u & =\mathbbm{1}(|z-a|<\eta)u+\mathbbm{1}(|z+a|<\eta) u +\mathbbm{1}(|z-a|\geq \eta \mbox{ and } |z+a|\geq \eta)u \\
\label{scalarproduct1:decomposition-ui}&=u_1+u_2+u_3.
\end{align}
The estimate \eqref{scalarproduct1:temp-1} will then follow from the estimates \eqref{scalarproduct1:estimate-u3}, \eqref{scalarproduct1:estimate-u1} and \eqref{scalarproduct1:estimate-u2} that we prove below for $u_3$, $u_1$ and $u_2$ respectively.

\smallskip

\noindent \underline{Estimate for the exterior part $u_3$}. We notice that for $\nu $ small enough $U_{1+2,\nu}\approx \frac{\nu^2}{\langle z\rangle^4}$ if simultaneously $|z-a|\geq \eta$ and $|z+a|\geq \eta$, while $ U_{1+2,\nu}\lesssim \frac{1}{\nu^2}$ if either $|z-a|\leq \eta$ or $|z+a|\leq \eta$. Hence, the corresponding left-hand and right-hand sides of \eqref{scalarproduct1:temp-1} satisfy
$$
 \int \nu^2 U_{1+2,\nu}|\Phi_{u_3}|^2 dz  \lesssim \int \langle z \rangle^{-4} |\Phi_{u_3}|^2 dz   \quad \mbox{and} \quad \int u_3^2 \gamma_\nu dz \approx \int \langle z \rangle^4 u_3^2 dz .
$$
Next, we recall the inequality (see the Appendix of \cite{CGMNapde22})
\begin{equation} \label{scalarproduct1:temp}
\int \langle y \rangle^{-4} |\Phi_v|^2 dy \lesssim \int \langle y \rangle^4 v^2dy.
\end{equation}
Combining the above identity with \eqref{scalarproduct1:temp} yields
\begin{equation} \label{scalarproduct1:estimate-u3}
 \int \nu^2 U_{1+2,\nu}|\Phi_{u_3}|^2 dz \lesssim \int u_3^2 \gamma_\nu dz.
\end{equation}

\noindent \underline{Estimate for the inner parts $u_1$ and $u_2$}. We first decompose
\begin{equation} \label{scalarproduct1:estimate-u1-expression}
\int  \nu^2 U_{1+2,\nu}|\Phi_{u_1}|^2 dz \leq \int_{|z+a|>  \eta} \nu^2 U_{1+2,\nu}|\Phi_{u_1}|^2+\int_{|z+a|<  \eta} \nu^2 U_{1+2,\nu}|\Phi_{u_1}|^2.
\end{equation}
To estimate the first term, we change variables and set $z=a+y\nu$ and $v_1(y)=u_1(z)$. Using $\nu^2 U_{1+2,\nu}(z)\approx \nu^2U_{1,\nu}(z)\approx \langle y \rangle^{-4}$ if $|z+a|>\eta$ and $\Phi_{v_1}(y)=\nu^{-2}\Phi_{u_1}(z)$ (because $\int u_1=0$), we obtain
$$
\int_{|z+a|\geq \eta} \nu^2 U_{1+2,\nu}|\Phi_{u_1}|^2 dz \approx \nu^6 \int_{|y+\frac{a}{\nu}|\geq\frac{\eta}{\nu}} |\Phi_{v_1}|^2\langle y \rangle^{-4} dy  \quad \mbox{and} \quad \int u_1^2 \gamma_\nu dz \approx \nu^6 \int  v_1^2\langle y \rangle^4dy .
$$
Hence,  using \eqref{scalarproduct1:temp} we get
\begin{equation} \label{scalarproduct1:estimate-u1-1}
\int_{|z+a|\geq \eta} \nu^2 U_{1+2,\nu}|\Phi_{u_1}|^2 \lesssim \int u_1^2 \gamma_\nu dz.
\end{equation}
To estimate the second term in \eqref{scalarproduct1:estimate-u1-expression} we bound $|\Phi_{u_1}(z)|\lesssim \int |\log (z-z')||u_1|(z')dz'$. As the set $\{|z+a|\leq \eta\}$ is for $|\eta|\leq |a|$ at distance greater than $|a|$ from the support of $u_1$ this implies $\| \Phi_{u_1}\|_{L^\infty(|z+a|\leq \eta)}\lesssim \| u_1\|_{L^1}$. By H\"older we then get
\begin{align*}
\int_{|z+a|\leq \eta} \nu^2 U_{1+2,\nu}|\Phi_{u_1}|^2  dz & \lesssim   \| \Phi_{u_1}\|_{L^\infty(|z-a|\leq \eta)}^2  \int  \nu^2 U_{1+2,\nu} dz  \lesssim  \| u_1\|_{L^1}^2  \nu^2 .
\end{align*}
By Cauchy-Schwarz, we have the estimate for $f\in L^2(\gamma_\nu dz)$,
\begin{equation} \label{scalarproduct1:estimate-u1-L1}
\| f\|_{L^1}^2 \leq  \int \nu^{-2}U_{1+2,\nu} dz   \int f^2 \gamma_\nu dz  \lesssim \nu^{-2}    \int f^2 \gamma_\nu dz.
\end{equation}
Combining the two inequalities above we get
\begin{equation} \label{scalarproduct1:estimate-u1-2}
\int_{|z+a|\leq \eta} \nu^2 U_{1+2,\nu}|\Phi_{u_1}|^2 dz  \lesssim \int u_1^2 \gamma_\nu dz.
\end{equation}
Injecting \eqref{scalarproduct1:estimate-u1-1} and \eqref{scalarproduct1:estimate-u1-2} in \eqref{scalarproduct1:estimate-u1-expression} shows
\begin{equation} \label{scalarproduct1:estimate-u1}
\int  \nu^2 U_{1+2,\nu}|\Phi_{u_1}|^2 dz \lesssim \int u_1^2 \gamma_\nu dz.
\end{equation}
By symmetry, we get the analogue estimate for $u_2$:
\begin{equation} \label{scalarproduct1:estimate-u2}
\int  \nu^2 U_{1+2,\nu}|\Phi_{u_2}|^2 dz \lesssim  \int u_2^2 \gamma_\nu dz.
\end{equation}

\noindent \textbf{Step 3}. \emph{Proof of the positivity \eqref{scalarproduct1:positivity}}. Pick $u\in L^2(\gamma_\nu dz)$ satisfying \eqref{scalarproduct1:orthogonality-condition}. We decompose again $u=u_1+u_2+u_3$ as in \eqref{scalarproduct1:decomposition-ui}. Injecting this decomposition in \eqref{scalarproduct1:expression}, and integrating by parts using $-\Delta \Phi_f=f$ we get
\begin{align}
\nonumber \langle u,u\rangle_\flat & =  \int  u_1^2 \gamma_\nu \, dz -\nu^2 \int u_1\Phi_{u_1} dz\\
\nonumber&+ \int   u_2^2 \gamma_\nu \, dz -\nu^2 \int u_2\Phi_{u_2}d z \\
\nonumber&+\int  u^2_3 \gamma_\nu dz -2\nu^2 \int u_1\Phi_{u_2} dz-2\nu^2 \int \Phi_{u_3}(u_1+u_2)d z  -\nu^2 \int u_3 \Phi_{u_3} d z \\
\nonumber &= I+II+\int  u^2_3 \gamma_\nu dz+IV.
\end{align}
The desired estimate \eqref{scalarproduct1:positivity} is a consequence of the below estimates \eqref{innerproduct:bd:positivity-I-II} and \eqref{innerproduct:bd:positivity-IV}. 

\smallskip

\noindent \underline{Lower bounds for the first terms $I$ and $II$.} We claim that
\begin{equation} \label{innerproduct:bd:positivity-I-II} 
I \geq \delta   \int  u_1^2 \gamma_\nu \, dz \qquad \mbox{and} \qquad II \geq \delta  \int   u_2^2 \gamma_\nu \, dz 
\end{equation}
for some universal constant $\delta>0$. To show the first inequality in \eqref{innerproduct:bd:positivity-I-II} we decompose
\begin{align} \label{innerproduct:id:decomposition-I} 
&I = \int u_1^2 \frac{\nu^2}{U_{1,\nu}} \, dz-\nu^2 \int u_1\Phi_{u_1} \, dz+ \int u_1^2(\gamma_\nu -\frac{\nu^2}{U_{1,\nu}})\, dz.
\end{align}
We change variables and set $z=a+y\nu$ and $v_1(y)=u_1(z)$. We have $\Phi_{v_1}(y)=\nu^{-2}\Phi_{u_1}(z)$. The first term in \eqref{innerproduct:id:decomposition-I} is
\begin{equation} \label{innerproduct:id:inter-change-of-variables-u1}
 \int u_1^2 \frac{\nu^2}{U_{1,\nu}} \, dz-\nu^2 \int u_1\Phi_{u_1} \, dz= \nu^6 \left(  \int v_1^2 \frac{1}{U} \, dy- \int v_1\Phi_{v_1} \, dy\right)
\end{equation}
and the orthogonality conditions \eqref{scalarproduct1:orthogonality-condition} give
$$
\int v_1 \, dy= \int \Lambda U v_1  \, dy = \int \partial_{y_1} U v_1  \, dy = \int \partial_{y_2} U v_1  \, dy =0.
$$
Applying the coercivity result of \cite{RSma14}, we obtain that for some $\delta>0$:
$$
 \int v_1^2 \frac{1}{U} \, dy- \int v_1\Phi_{v_1} \, dy \geq \delta  \int v_1^2 \frac{1}{U} \, dy .
$$
Injecting this inequality in \eqref{innerproduct:id:inter-change-of-variables-u1}, using that $\int v_1^2 \frac{1}{U} \, dy =  \nu^{-4} \int u_1^2 \frac{1}{U_{1,\nu}} \, dz$, shows:
\begin{equation} \label{innerproduct:bd:I-first-term}
 \int u_1^2 \frac{\nu^2}{U_{1,\nu}} \, dz-\nu^2 \int u_1\Phi_{u_1} \, dz \geq \delta  \int u_1^2 \frac{\nu^2}{U_{1,\nu}} \, dz.
\end{equation}
Injecting \eqref{innerproduct:bd:I-first-term} and \eqref{innerproduct:bd:I-second-term} in \eqref{innerproduct:id:decomposition-I} shows the first inequality in \eqref{innerproduct:bd:positivity-I-II} for $\eta$ small enough. The proof of the second inequality is exactly the same by symmetry.

\smallskip

\noindent \underline{Upper bound for the fourth term $IV$.} We claim that for $\eta$ small enough and then $\nu$ small enough,
\begin{equation} \label{innerproduct:bd:positivity-IV} 
|IV| \lesssim \eta  \int  u^2 \gamma_\nu \, dz 
\end{equation}
We now prove this bound. For the first term in $IV$, by \eqref{scalarproduct1:orthogonality-condition} we have $\int u_1=0 $, hence:
$$
\nu^2 \int u_1\Phi_{u_2} dz = \nu^2 \int u_1 (\Phi_{u_2}(z)-\Phi_{u_2}(a)) dz .
$$
Using $|\nabla \Phi_2(z)|\lesssim \int |z-z'|^{-1}|u_2(z')|dz'$, and that $u_2$ is supported for $|z+a|<\eta $ we obtain for $\eta$ small enough that $|\nabla \Phi_{u_2}(z)|\lesssim \| u_2\|_{L^1}$ for all $|z-a|\leq \eta$. Hence by the mean value Theorem
$$
\nu^2 \left| \int u_1\Phi_{u_2} dz \right| \lesssim \nu^2 \| u_1\|_{L^1} \eta  \| u_2\|_{L^1} .
$$
Using the estimate \eqref{scalarproduct1:estimate-u1-L1} yields
\begin{equation} \label{innerproduct:id:positivity-IV-term1}
\nu^2 \left| \int u_1\Phi_{u_2} dz \right| \lesssim \eta \int u^2 \gamma_\nu  dz.
\end{equation}
For the second term in $IV$, we estimate that for $z\in \mathbb R^2$,
\begin{align}
\nonumber |\Phi_{u_3}(z)| & \lesssim \int |\ln |z-z'|||u_3(z')|\lesssim \left(\int |\ln |z-z'||^2 \langle z'\rangle^{-4}dz'\right)^{\frac 12}\left(\int u_3^2 \langle z'\rangle^4dz'\right)^{\frac 12} \\
\label{innerproduct:id:IV:term2:inter1} &\leq C_{\eta}\left(\int u^2 \gamma_\nu(z')dz'\right)^{\frac 12}
\end{align}
where we used Cauchy-Schwarz and that $\gamma_\nu(z')\geq c_{\eta} \langle z'\rangle^4$ for some $c_{\eta}>0$ for $z'$ in the support of $u_3$ from \eqref{bd:equivalence-weight-U1+2nu-1}. Using \eqref{innerproduct:id:IV:term2:inter1} and \eqref{scalarproduct1:estimate-u1-L1} we obtain
\begin{equation} \label{innerproduct:id:IV:term2}
\nu^2 \left|  \int \Phi_{u_3}(u_1+u_2)d z
 \right|  \leq \nu^2 C_{\eta}\left(\int u^2\gamma_\nu dz\right)^{\frac 12} \| u_1+u_2\|_{L^1} \leq C_{\eta}\nu \int u^2 \gamma_\nu dz
\end{equation}
For the third term in $IV$, applying Cauchy-Schwarz, \eqref{scalarproduct1:temp} and $\gamma_\nu^{-1}\approx \langle z \rangle^{-4}$ on the support of $u_3$ yields
\begin{equation} \label{innerproduct:bd-positivity-IV-3}
| \nu^2 \int \Phi_{u_3} u_3| \lesssim \nu^2 \left( \int \Phi_{u_3}^2 \langle z \rangle^{-4} \right)^{\frac 12}\left( \int u_3^2 \langle z \rangle^{4} \right)^{\frac 12}\lesssim  \nu^2 \int u^2_3 \gamma_\nu dz.
\end{equation}
Combining \eqref{innerproduct:id:positivity-IV-term1}, \eqref{innerproduct:id:IV:term2} and \eqref{innerproduct:bd-positivity-IV-3} shows \eqref{innerproduct:bd:positivity-IV} as desired.
\end{proof}

\begin{proof}[Proof of Lemma \ref{lem:almost-symmetry-Lsz-interior}]

\noindent \textbf{Step 1}. \emph{Proof of the identities \eqref{def:L12z_M12} and \eqref{id:symmetry-Lsz-interior}.} We recall the identity $\nabla U=U\nabla \Phi_U$. Hence $\nabla \Phi_{U_{i,\nu}}=\nabla \ln U_{i,\nu}$ for $i=1,2$. We thus rewrite
\begin{align*}
\nabla \Phi_{U_{1+2,\nu}} &=\nabla \Phi_{U_{1,\nu}} +\nabla \Phi_{U_{2,\nu}}  =\nabla \ln U_{1,\nu}+\nabla \ln U_{2,\nu}= \nabla \ln U_{12,\nu}\\
&\quad = \nabla \ln U_{1+2,\nu} +\nabla \ln \frac{U_{12,\nu}}{U_{1+2,\nu}}.
\end{align*}
Using the above identity, we rewrite
\begin{align*}
\Delta u+\nabla. (u\nabla \Phi_{U_{1+2,\nu}}) &=  \nabla . \left( \nabla u -u \nabla \ln U_{1+2,\nu} +u \nabla \ln \frac{U_{1+2,\nu}}{U_{12,\nu}}\right)  \\
&=  \nabla . \left( U_{1+2,\nu}\nabla \left( \frac{u}{U_{1+2,\nu}} \right) +u \nabla \ln \frac{U_{1+2,\nu}}{U_{12,\nu}}\right) .
\end{align*}
The first identity \eqref{def:L12z_M12} of the lemma follows from injecting the above identity and $\Lambda u =\nabla . (u\nabla (z^2/2))$ in \eqref{def:LszNot}. The second identity \eqref{id:symmetry-Lsz-interior} then follows from \eqref{def:L12z_M12} by integrating by parts.

\medskip

\noindent \textbf{Step 2}. \emph{Proof of the coercivity \eqref{bd:coercivity-Lsz-interior}.} By \eqref{id:symmetry-Lsz-interior} we have:
\begin{equation} \label{id:symmetry-Lsz-interior-inter1}
- \langle \Ls^z u,u\rangle_\flat = \int \gamma_\nu^{-1}|\nabla \Ms^z  u|^2 dz-\int \nabla .(u\nabla V) \Ms^z  u dz.
\end{equation}
The coercivity estimate \eqref{bd:coercivity-Lsz-interior} then follows from the two estimates \eqref{id:symmetry-Lsz-interior-inter15} and \eqref{id:symmetry-Lsz-interior-inter20} for the first and second terms in \eqref{id:symmetry-Lsz-interior-inter1} respectively, and by taking $\eta'$ small enough.

\smallskip

\noindent \underline{Estimate for the first term in \eqref{id:symmetry-Lsz-interior-inter1}.} We claim that
\begin{equation} \label{id:symmetry-Lsz-interior-inter15}
\int \gamma_\nu^{-1}|\nabla \Ms^z  u|^2 dz\gtrsim \int |\nabla u|^2 \gamma_\nu dz.
\end{equation}
We first split 
\begin{equation} \label{id:symmetry-Lsz-interior-inter12}
\int\gamma_\nu^{-1}|\nabla \Ms^z  u|^2 dz \geq \int_{|z-a|\leq \eta'} \nu^{-2}U_{1, \nu}|\nabla \Ms^z  u|^2 dz+\int_{|z+a|\leq \eta'} \nu^{-2}U_{2, \nu}|\nabla \Ms^z  u|^2 dz.
\end{equation}
We study the first term in \eqref{id:symmetry-Lsz-interior-inter12}. We decompose
\begin{align*}
u&= \mathbbm{1} (|z-a|\leq \eta')u+\mathbbm{1} (|z+a|\leq \eta')u \\
&= u_1+u_2.
\end{align*}
Using $\Phi_u=\Phi_{u_1}+\Phi_{u_2}$ and $U_{1+2,\nu}^{-1}=U_{1,\nu}^{-1}-U_{2,\nu}U_{1+2,\nu}^{-1} U_{1,\nu}^{-1}$, we decompose for $|z-a|\leq \eta'$:
\begin{align*}
\Ms^zu =\Ms^a_0 u_1- \frac{\nu^2U_{2,\nu}}{U_{1+2,\nu} U_{1,\nu}} u-\nu^2\Phi_{u_2},
\end{align*}
where 
$$\Ms^a_0 u_1 = \nu^2\frac{u_1}{U_{1,\nu}}-\Phi_{u_1}.$$
Therefore, using $(a+b+c)^2\geq a^2-2b^2-2c^2$ we have
\begin{align}
\label{id:symmetry-Lsz-interior-inter10} \int_{|z-a|\leq \eta'} \gamma_\nu^{-1}|\nabla \Ms^z  u|^2 dz & \geq \int \nu^{-2}U_{1, \nu}\left( |\nabla \Ms_{0}^a u_1|^2-2\left|\nabla \Phi_{u_2} \right|^2 -2 \left|\nabla \left(  \frac{U_{2,\nu}}{U_{1+2,\nu} U_{1,\nu}} u\right)\right|^2\right)dz  .
\end{align}
We study the first term in \eqref{id:symmetry-Lsz-interior-inter10}. We change variables $z=a+\nu y$ and $v_1(y)=u_1(z)$. We have $\Ms^a_0 u_1(z)=\nu^4 \Ms_0 v_1(y)$, so that
$$
 \int \nu^{-2}U_{1, \nu}|\nabla \Ms_{0}^a u_1|^2 dz= \nu^4 \int U |\nabla \Ms_0 v_1|^2 dy.
$$
Using the orthogonality conditions \eqref{scalarproduct1:orthogonality-condition-coercivity-operator}, applying Lemma \ref{lem:coercivity-one-bubble-H1-orthogonality} shows
$$
 \int U |\nabla \Ms_0 v_1|^2 dy \gtrsim  \int  |\nabla v_1|^2\frac{dy}{U(y)}.
$$
Combining the two inequalities above, and unwinding the change of variables, we get
\begin{equation} \label{id:symmetry-Lsz-interior-inter14}
 \int \nu^{-2}U_{1, \nu}|\nabla \Ms_{0}^z u_1|^2 dz \gtrsim \int |\nabla u_1|^2\frac{\nu^2}{U_{1,\nu}}dz.
\end{equation}
We turn to the second term in \eqref{id:symmetry-Lsz-interior-inter10}. Using that $|\nabla \Phi_{u_2}(z)|\lesssim (|\cdot|^{-1}*|u_2|)(z) $ and that $u_2$ is supported in $\{|z+a|\leq \eta'\}$ we have that $|\nabla \Phi_{u_2}(z)|\lesssim \| u_2\|_{L^1}$ for $|z-a|\leq \eta'$. As $\int U_{1,\nu}\lesssim 1$ we bound
\begin{equation} \label{id:symmetry-Lsz-interior-inter13}
\int_{|z-a|\leq \eta'} \nu^{2}U_{1, \nu} \left|\nabla \Phi_{u_2} \right|^2dz\lesssim \nu^2 \| u_2\|_{L^1}^2\lesssim \int u^2_2 \gamma_\nu  dz \lesssim \eta^{'2} \int |\nabla u|^2\gamma_\nu dz
\end{equation} 
where we used successively \eqref{scalarproduct1:estimate-u1-L1}, and \eqref{bd:hardy-interior-product} with $\min (\nu+|z+a|)\leq \eta'$ on the support of $u_2$.

We now estimate the third term in \eqref{id:symmetry-Lsz-interior-inter10}. For $|z-a|\leq \eta'$ one has
$$
 \nu^{-4}(U_{1, \nu}(z))^2  \left( \frac{\nu^2U_{2,\nu}}{U_{1+2,\nu} U_{1,\nu}} \right)^2 = \frac{(U_{2,\nu}(z))^2}{(U_{1+2,\nu}(z))^2}\lesssim \eta^{'8},
$$
where we used that $U_{2,\nu}(z)\approx \nu^2$ and $U_{1+2,\nu}(z)\approx \nu^2 (\nu+|z-a|)^{-4}$. Similarly,
$$
 \nu^{-4}(U_{1, \nu}(z))^2  \left|\nabla \left( \frac{\nu^2U_{2,\nu}}{U_{1+2,\nu} U_{1,\nu}}\right)\right|^2 \lesssim \eta^{'6}.
$$
Using the two inequalities above, we have
$$
\nu^{-2}U_{1, \nu} \left|\nabla \left(u  \frac{\nu^2U_{2,\nu}}{U_{1+2,\nu} U_{1,\nu}}  \right)\right|^2\lesssim \eta^{'6}\frac{\nu^2}{U_{1,\nu}}(u^2+|\nabla u|^2).
$$
Therefore,
\begin{align} \label{id:symmetry-Lsz-interior-inter11}
\int_{|z-a|\leq \eta'} \nu^{-2}U_{1, \nu} \left|\nabla \left(u \frac{\nu^2U_{2,\nu}}{U_{1+2,\nu} U_{1,\nu}} \right)\right|^2dz & \lesssim \eta^{'6} \int_{|z-a|\leq \eta'} (u^2+|\nabla u|^2)\frac{\nu^2dz}{U_{1,\nu}} \lesssim \eta^{'6} \int |\nabla u|^2 \frac{\nu^2}{U_{1,\nu}}dz
\end{align}
where we used the Hardy inequality \eqref{bd:hardy-interior-product} and \eqref{innerproduct:bd:I-second-term}.

Injecting \eqref{id:symmetry-Lsz-interior-inter14}, \eqref{id:symmetry-Lsz-interior-inter13} and \eqref{id:symmetry-Lsz-interior-inter11} in \eqref{id:symmetry-Lsz-interior-inter10}, we get 
$$
\int_{|z-a|\leq \eta'} \nu^{-2}U_{1, \nu}|\nabla \Ms^z  u|^2 dz  \geq C^{-1} \int |\nabla u_1|^2\frac{\nu^2}{U_{1,\nu}}dz-C \eta^{'2} \int |\nabla u|^2 \frac{\nu^2}{U_{1,\nu}}dz.
$$
Since the second term in \eqref{id:symmetry-Lsz-interior-inter12} plays a symmetric role, we have similarly
$$
\int_{|z+a|\leq \eta'} \nu^{-2}U_{2, \nu}|\nabla \Ms^z  u|^2 dz  \geq C^{-1} \int |\nabla u_2|^2\frac{\nu^2}{U_{1,\nu}}dz-C \eta^{'2} \int |\nabla u|^2 \frac{\nu^2}{U_{1,\nu}}dz.
$$
Combining the two estimates above shows the desired identity \eqref{id:symmetry-Lsz-interior-inter15}.

\smallskip

\noindent \underline{Estimate for the second term in \eqref{id:symmetry-Lsz-interior-inter1}.} We claim that
\begin{equation} \label{id:symmetry-Lsz-interior-inter20}
\left| \int \nabla .(u\nabla V) \Ms^z  u dz\right|\lesssim \eta' \int |\nabla u|^2 \gamma_\nu dz.
\end{equation}
Using the continuity estimate \eqref{scalarproduct1:continuity}
$$
\left|\int \nabla .(u\nabla V) \Ms^z  u dz\right| \lesssim \left( \int | \nabla .(u\nabla V)|^2 \gamma_\nu dz \right)^{\frac 12}\left( \int u^2 \gamma_\nu dz \right)^{\frac 12} .
$$
From \eqref{def:V:L12z}, we have that $V(z)=\ln ((\nu^2+|z-a|^2)^2+(\nu^2+|z+a|^2)^2)-\ln (8\nu^2)-\frac{\beta}{2}|z|^2$, and  $(\nu^2+|z-a|^2)^2+(\nu^2+|z+a|^2)^2\geq c(a)$ for all $z\in \mathbb R^2$ for some $c(a)>0$. Hence, $\nabla V$ and $\Delta V$ are bounded on compact sets. This implies $|\nabla .(u\nabla V)|\lesssim |u|+|\nabla u|$ so that
$$
\int | \nabla .(u\nabla V)|^2\gamma_\nu dz  \lesssim \int (u^2+|\nabla u|^2)\gamma_\nu.
$$
Using again the support of $u$, and then the Hardy inequality \eqref{bd:hardy-interior-product}, we have 
$$
 \int u^2\gamma_\nu dz \lesssim \eta^{'2}\int \frac{u^2}{\min \left((\nu+|z-a|)^2,(\nu+|z+a|)^2 \right)}\gamma_\nu dz \lesssim \eta^{'2}\int |\nabla u|^2 \gamma_\nu dz.
$$
Combining the three above inequalities yields the desired estimate \eqref{id:symmetry-Lsz-interior-inter20}. This ends the proof of Lemma \ref{lem:almost-symmetry-Lsz-interior}.
\end{proof}

\subsubsection{The scalar product for the exterior zone}

We now present another scalar product, for which the linearized operator is almost symmetric for functions which are localized away from $\pm a$. Namely, consider the weighted space $L^2(\omega_\nu dz)$ with the weight function
\begin{equation} \label{def:omega12}
\omega_{\nu}(z) = \frac{\nu^4}{U_{12,\nu}(z)}e^{-\beta |z|^2/2}
\end{equation}
and the scalar product $\langle \cdot , \cdot \rangle_{\omega_\nu}$. Note that we have normalized the weight to have
\begin{equation} \label{bd:equivalence-weight-omeganu}
\omega_\nu (z) \approx  (\nu+|z-a|)^4(\nu+|z+a|)^4 e^{-\beta |z|^2/2} .
\end{equation}
Functions $u\in L^2_{\omega_\nu}$ may diverge at infinity in which case their Poisson field $\Phi_u$ is not well defined. Since in our proof of Theorem \ref{thm:main} the contribution from the zone $|z|\gtrsim 10$ to the dynamics will be showed to be negligible (up to modulation terms that we will single out), we circumvent this problem by localizing the Poisson field. Let $\chi^*(z)$ be the cut-off defined by \eqref{def:chi-star} and define
\begin{align}\label{def:L12ztilde}
\Lst^z u &=\Delta u- \nabla. \Big( u  \nabla \Phi_{U_{1 + 2, \nu}}\Big)-\nabla U_{1 + 2, \nu}.\nabla \Phi_{\chi^*u}+U_{1+2,\nu}u  - \beta \Lambda u\\
&= \nabla \cdot \Big( \nabla u -  u  \nabla \Phi_{U_{1 + 2, \nu}} - U_{1 + 2, \nu}.\nabla \Phi_{\chi^*u}+U_{1+2,\nu}u - \beta z u\Big),\nonumber
\end{align}
so that by \eqref{def:L12ztilde} we have
$$
\Ls^zu=\Lst^zu-\nabla U_{1+2,\nu}.\nabla \Phi_{(1 - \chi^*)u}.
$$

\begin{lemma}[Coercivity and almost symmetry in the exterior zone] \label{lem:almost-symmetry-Lsz-exterior}
We have
\begin{equation} \label{id:decomposition-Lsz-exterior}
\Lst^z u = \frac{1}{\omega_\nu } \nabla. (\omega_\nu \nabla u ) + 2(U_{1+2, \nu} - \beta) u-\nabla U_{1+2,\nu}.\nabla \Phi_{\chi^* u},
\end{equation}
and, for Schwartz functions $u$ and $v$,
\begin{equation} \label{id:symmetry-Lsz-exterior}
\langle \Lst^z u,v \rangle_{L^2_{\omega_\nu}}= -\int \nabla u . \nabla v \; \omega_\nu dz+\int 2 (U_{1+2,\nu}-\beta)uv \omega_\nu dz -\int \nabla U_{1+2,\nu}.\nabla \Phi_{\chi^*u} v \omega_\nu dz.
\end{equation}
Moreover, there exists $\delta>0$ such that for $|a|\leq 5$ and $\eta>0$ close to $0$, we have for $\nu>0$ small enough that for all $u\in H^1_{\omega_{\nu}}$ that is supported in $\{|z-a|\geq \eta \mbox{ and }|z+a|\geq \eta \}$ that
\begin{equation} \label{bd:coercivity-exterior}
-\langle \Lst^z u,u \rangle_{L^2_{\omega_\nu}}\geq \delta \int (\langle z \rangle^2 u^2+|\nabla u|^2)\omega_{\nu}dz
\end{equation}
\end{lemma}

For $z$ away from $\pm a$, we have $\nabla U_{1+2,\nu}=O(\nu^2)$. Hence, by \eqref{id:decomposition-Lsz-exterior}, $\Lst^z u$ is well approximated by the operator $\omega_\nu^{-1} \nabla. (\omega_\nu \nabla u ) + 2(U_{1+2, \nu} - \beta) u$ for functions $u$ that are localized away from $\pm a$. This operator is self-adjoint in $L^2(\omega_\nu dz)$ and its associated quadratic form is the sum of the first two terms in the right-hand side of \eqref{id:symmetry-Lsz-exterior}. This is what motivated the introduction of the scalar product with the weight function \eqref{def:omega12}.

We now proceed with the proof of Lemma \ref{lem:almost-symmetry-Lsz-exterior}. We first claim the following inequality which shall be used numerous times in the analysis. 
\begin{lemma}[Hardy-Poincar\'e inequality in $H^1_{\omega_\nu}$] There exists $C>0$ such that for all $u\in H^1_{\omega_\nu}$, we have for $\nu$ small enough and $|a|\leq 5$,
\begin{equation} \label{bd:hardy-poincare-H1-omeganu}
\int \left(\frac{1}{(\nu+|z-a|)^2}+\frac{1}{(\nu+|z+a|)^2}+|z|^2 \right)u^2 \omega_\nu dz \leq C \int (u^2+|\nabla u|^2)\omega_\nu dz.
\end{equation}
\end{lemma}

\begin{proof}

We decompose $u=u_1+u_2$ where $u_1=\chi^* u$ and $u_2=(1-\chi^*)u$. For $z$ in the support of $\chi^*$, we have $|z|\leq 20$. By \eqref{bd:equivalence-weight-U1+2nu-1} and \eqref{bd:equivalence-weight-omeganu} we have
\begin{equation} \label{id:matched-scalar-product-equivalence-weights-z<20}
\omega_\nu(z)\approx \gamma_\nu(z) \qquad \mbox{for }|z|\leq 20.
\end{equation}
By the Hardy inequality \eqref{bd:hardy-interior-product} and \eqref{id:matched-scalar-product-equivalence-weights-z<20}, we have
\begin{equation} \label{bd:hardy-poincare-H1-omeganu-u1}
\int \left(\frac{1}{(\nu+|z-a|)^2}+\frac{1}{(\nu+|z+a|)^2} \right)u^2_1 \omega_\nu dz \leq C \int |\nabla u_1|^2 \omega_\nu dz.
\end{equation}
For $|z|\geq 10$ we have $\omega_\nu(z)\approx \langle z \rangle^4 e^{-|z|^2/4}$ by \eqref{bd:equivalence-weight-omeganu}. Hence, by the Poincar\'e inequality \eqref{bd:poincare-gaussian},
\begin{equation} \label{bd:hardy-poincare-H1-omeganu-u2}
\int \langle z \rangle^2 u^2_2 \omega_\nu dz \leq C \int (u_2^2+|\nabla u_2|^2) \omega_\nu dz.
\end{equation}
Combining \eqref{bd:hardy-poincare-H1-omeganu-u1} and \eqref{bd:hardy-poincare-H1-omeganu-u2} yields the desired inequality \eqref{bd:hardy-poincare-H1-omeganu}, since
$$
((\nu+|z-a|)^{-2}+(\nu+|z+a|)^{-2}+|z|^2 )u^2\approx ((\nu+|z-a|)^{-2}+(\nu+|z+a|)^{-2})u_1^2+|z|^2u_2^2
$$
and $|\nabla u_i|\lesssim |u|+|\nabla u|$ for $i=1,2$.
\end{proof}

\begin{proof}[Proof of Lemma \ref{lem:almost-symmetry-Lsz-exterior}]

\noindent \textbf{Step 1}. \emph{Proof of the identities \eqref{id:decomposition-Lsz-exterior} and \eqref{id:symmetry-Lsz-exterior}.} Using $-\Delta \Phi_f=f$ we expand \eqref{def:L12ztilde} as
$$
\Lst^z u =\Delta u -\nabla u .\left( \nabla \Phi_{U_{1+2,\nu}}+\beta z\right)+2(U_{1+2,\nu}-\beta)u-\nabla U_{1+2,\nu}.\nabla \Phi_{\chi^* u}.
$$
Hence, the first identity \eqref{id:decomposition-Lsz-exterior} of the Lemma holds true provided that the weight $\omega_\nu$ satisfies
$$
\nabla \ln \omega_\nu =-\nabla \Phi_{U_{1+2,\nu}}-\beta z.
$$
Since $\nabla \Phi_U=\nabla \ln U$, this equation can be solved explicitly, yielding to the formula \eqref{def:omega12}. Hence \eqref{id:decomposition-Lsz-exterior} holds true. The second identity \eqref{id:symmetry-Lsz-exterior} is then obtained from \eqref{id:decomposition-Lsz-exterior} by integration by parts.

\medskip

\noindent \textbf{Step 2}. \emph{Proof of the coercivity \eqref{bd:coercivity-exterior}.} By \eqref{id:symmetry-Lsz-exterior} we have
\begin{equation} \label{temp1}
-\langle \Lst u ,u \rangle_{\omega_\nu}= \int (|\nabla u|^2+(1-2U_{1+2,\nu})u^2)\omega_\nu dz +\int \nabla U_{1+2,\nu}.\nabla \Phi_{\chi^*u} u \omega_\nu dz.
\end{equation}
We study the first term in \eqref{temp1}. Using $U_{i,\nu}(z)=\nu^28(1+|z\pm a|^2)^{-2}$ we have that $|U_{1+2,\nu}|\lesssim \nu^2 $ on the support of $u$. Using this and the Poincar\'e inequality \eqref{bd:hardy-poincare-H1-omeganu} we find
\begin{equation} \label{temp2}
 \int (|\nabla u|^2+(1-2U_{1+2,\nu})u^2)\omega_\nu dz \geq \delta \int (\langle z \rangle^2 u^2+|\nabla u|^2) \omega_\nu dz.
 \end{equation}
 We now turn to the second term in \eqref{temp1}. We make a general estimate for some other Schwartz function $v$ that is supported on $\cap_{\pm}\{|z\pm a|\geq \eta\}$. Integrating by parts, we have
\begin{equation} \label{temp3}
 \int \nabla U_{1+2,\nu}.\nabla \Phi_{\chi^*u} v \omega_\nu dz=-\int \Phi_{\chi^*u} \left[ (\Delta U_{1+2,\nu}+\nabla U_{1+2,\nu}.\nabla \ln \omega_\nu)v+\nabla U_{1+2,\nu}.\nabla v \right] \omega_\nu dz.
\end{equation}
 Using again that $U_{i,\nu}(z)=\nu^28(1+|z\pm a|^2)^{-2}$ and that $u$ and $v$ are supported on $\cap_{\pm}\{|z\pm a|\geq \eta\}$ we have $\left|(\Delta U_{1+2,\nu}+\nabla U_{1+2,\nu}.\ln \omega_\nu)v+\nabla U_{1+2,\nu}.\nabla v \right|\lesssim \nu^2(v+|\nabla v|)$ for $\nu$ small enough. By Cauchy-Schwarz and this inequality, \eqref{temp3} gives
\begin{equation} \label{temp4}
\left| \int \nabla U_{1+2,\nu}.\nabla \Phi_{\chi^*u}v \omega_\nu dz\right| \lesssim \nu^2 \left( \int_{\min_\pm (|z\pm a|\geq \eta} |\Phi_{\chi^*u}|^2\omega_\nu dz\right)^{\frac 12} \left( \int (v^2+|\nabla v|^2)\omega_\nu dz\right)^{\frac 12}
\end{equation}
where we used that the integrand is supported on $\cap_{\pm}\{|z\pm a|\geq \eta\}$. Using that on the support of $\chi^*$ we have $\omega_\nu \lesssim \langle z \rangle^{-4}$ and $\omega_\nu \approx_\eta 1$ for $\nu$ small enough, we have by \eqref{scalarproduct1:temp} that
\begin{equation} \label{temp6}
\int_{\min_\pm (|z\pm a|\geq \eta} |\Phi_{\chi^*u}|^2\omega_\nu dz\lesssim \int \langle z \rangle^{-4} |\Phi_{\chi^*u}|^2 dz \lesssim \int \langle z \rangle^{4} \chi^{*2} u^2\lesssim \int u^2 \omega_\nu dz.
\end{equation}
Injecting the above inequality in \eqref{temp4} shows
\begin{equation} \label{temp5}
\left| \int \nabla U_{1+2,\nu}.\nabla \Phi_{\chi^*u}v \omega_\nu dz\right| \lesssim \nu^2 \left( \int u^2 \omega_\nu dz\right)^{\frac 12} \left( \int (v^2+|\nabla v|^2)\omega_\nu dz\right)^{\frac 12}
\end{equation}
Injecting \eqref{temp2} and \eqref{temp5} with $u=v$ in \eqref{temp1} shows the desired identity \eqref{bd:coercivity-exterior} for $\nu$ small enough.
\end{proof}

\subsubsection{Matched scalar product}

Recall that we fix $\beta =\frac 12$ for sake of simplicity. We want to match the scalar product for the exterior zone
$$
\langle u,v\rangle_{\omega_\nu}= \int uv \omega_\nu dz
$$
which is convenient for functions located away from $a$ and $-a$, with the scalar product for the interior zone
\begin{equation} \label{id:expression-inter-interior-product}
\langle u, v \rangle_\flat= \int   u v \gamma_\nu \, dz -  \nu^2 \int\Phi_u vdz,
\end{equation}
which is suitable for functions located near $\pm a$. 

Firstly, we compare the weights. Using $U(x)= 8 (1+|x|^2)^{-2}$, a Taylor expansion yields
$$
\omega_\nu (z)= \gamma_\nu(z) \left(2|a|^4 e^{-|a|^2/4} +o(1) \right)
$$
as $\nu\to 0$ and $z\to a$ or $z\to-a$ (see \eqref{id:matched-product-expansion-weights} for the details). Thus, the two weights match up to a factor $c_*=2|a|^4e^{-|a|^2/4}=  32 e^{-1}+O(|a-a_\infty|)$.

Secondly, we consider the second term in \eqref{id:expression-inter-interior-product}. This term does not need to be matched to anything and can be kept as is. The reason is that the contribution from the exterior zone in this term is negligible. Indeed, firstly, the weight function $\gamma_\nu$ in the first term in \eqref{id:expression-inter-interior-product} is of order one away from $a$ and $-a$, and the weight in the second term in \eqref{id:expression-inter-interior-product} is of size $\nu^2$. Secondly, in the second term in \eqref{id:expression-inter-interior-product}, the interaction between the interior zone and the exterior has a cancellation if $v$ satisfies the zero mass condition in the interior zone $\int_{|z\pm a|\leq \eta} v=0$. Indeed, in this case $\int_{|z\pm a|\leq \eta} \nu^2 \Phi_u vdz=\int_{|z\pm a|\leq \eta} \nu^2 (\Phi_u-\Phi_u(\pm a)) vdz$, so that the contribution from the Poisson field of $u$ is lowered.

As the contribution from the exterior zone in the second term in \eqref{id:expression-inter-interior-product} is negligible, we further localise it. The reason is that the Poisson field $\Phi_u$ is not well-defined in the space $L^2_{\omega_\nu}$ due to possible growth at infinity. For that, we will use the cut-off function $\chi^*$.

We thus introduce the "matched scalar product"
\begin{equation}\label{def:scalar12}
\langle u,v\rangle_* = \int u v \omega_\nu dz - c_*\nu^2 \int \chi^* u \, \Phi_{\chi^*v} dz, \qquad c_*=2 |a|^4 e^{-|a|^2/4}.
\end{equation}
We first show that $\langle \cdot, \cdot\rangle_*$ is well defined on $L^2_{\omega_\nu}$, and it defines a scalar product that is equivalent to $\langle \cdot,\cdot \rangle_{\omega_\nu}$ on a suitable subspace. We shall again abuse language and refer to $\langle \cdot,\cdot \rangle_{*}$ as a scalar product.

\begin{proposition}[Matched scalar product]  \label{pr:scalar-product-interior}

There exists $C>0$ such that the following hold for $a$ close to $a_\infty$, for all $u,v\in L^2_{\omega_\nu}$. The functional $\langle \cdot,\cdot \rangle_*$ is symmetric
\begin{equation} \label{matchedscalarproduct1:symmetry}
\langle u,v\rangle_*=\langle v ,u\rangle_*
\end{equation}
and for $\nu$ small enough, it is continuous
\begin{equation} \label{matchedscalarproduct1:continuity}
|\langle u,v\rangle_*|^2 \leq C \int u^2 \omega_\nu dz \int v^2 \omega_\nu dz.
\end{equation}
Moreover, for any $\eta>0$ small enough, then for all $\nu$ small enough, it is positive for functions satisfying the orthogonality conditions
\begin{equation} \label{matchedscalarproduct1:orthogonality-condition}
 \langle u, \mathbbm{1}(|z\pm a|<\eta) \rangle =  \langle u, \Lambda U_{i, \nu}\mathbbm{1}(|z\pm a|<\eta) \rangle  =\langle u, \pa_j U_{i, \nu}  \mathbbm{1}(|z\pm a|<\eta) \rangle =0
\end{equation}
for all $i=1,2$, $j=1,2$ and where $\pm$ designates $-$ or $+$ if $i=1$ or $i=2$, in the sense that
\begin{equation} \label{matchedscalarproduct1:positivity}
  \int u^2\omega_\nu dz \leq C \langle u,u\rangle_* .
\end{equation}
\end{proposition}

\begin{proof}

\textbf{Step 1}. \emph{Proof of the symmetry \eqref{matchedscalarproduct1:symmetry}}. We rewrite \eqref{def:scalar12} as
\begin{equation} \label{id:matched-scalar-product-intermediaire}
\langle u,v\rangle_*= \langle u,v \rangle_{\omega_{\nu}}+ c_* \langle \chi^*u,\chi^* v\rangle_\flat - c_* \langle \chi^* u,\chi^* v\rangle_{\gamma_\nu},
\end{equation}
where $\langle \cdot, \cdot \rangle_\flat$ is introduced in \eqref{def:scalar-interior}. We then obtain \eqref{matchedscalarproduct1:symmetry} by using \eqref{scalarproduct1:symmetry}.

\medskip

\noindent \textbf{Step 2}. \emph{Proof of the continuity \eqref{matchedscalarproduct1:continuity}}. From the identity \eqref{id:matched-scalar-product-intermediaire}, we have by using \eqref{scalarproduct1:continuity} and Cauchy-Schwarz,
\begin{equation} \label{id:matched-scalar-product-intermediaire2}
|\langle u,v\rangle_*|^2\lesssim  \int (\chi^*)^2u^2 \gamma_\nu dz \int (\chi^*)^2v^2 \gamma_\nu dz + \int \omega_\nu u^2dz \int \omega_\nu v^2dz.
\end{equation}
Combining \eqref{id:matched-scalar-product-intermediaire2} and the equivalence $\omega_\nu(z)\approx \gamma_\nu(z)$ for $|z|\leq 20$, one gets \eqref{matchedscalarproduct1:continuity}.

\medskip

\noindent \textbf{Step 3}. \emph{Proof of the positivity \eqref{matchedscalarproduct1:positivity}}. We pick $u\in L^2_{\omega_\nu}$ satisfying \eqref{matchedscalarproduct1:orthogonality-condition} and decompose
\begin{align*}
u& =\mathbbm{1}(|z-a|<\eta)u+\mathbbm{1}(|z+a|<\eta) u +\mathbbm{1}(|z-a|\geq \eta \mbox{ and } |z+a|\geq \eta)u \\
&=u_1+u_2+u_3,
\end{align*}
so that
\begin{align}
\nonumber \langle u,u\rangle_* & =  \int \omega_\nu u_1^2 \, dz -c_* \nu^2 \int u_1\Phi_{u_1} dz\\
\nonumber&+  \int \omega_\nu u_2^2 \, dz -c_* \nu^2 \int u_2\Phi_{u_2} dz \ +\int \omega_\nu u^2_3 dz  \\
\nonumber &-c_* \left(2\nu^2 \int u_1\Phi_{u_2} dz-2\nu^2 \int \Phi_{\chi^*u_3}(u_1+u_2)d z  -\nu^2 \int \chi^*u_3 \Phi_{\chi_{ 10}u_3} d z \right) \\
\nonumber &= I+II+\int \omega_\nu u^2_3 dz+IV.
\end{align}
The desired inequality \eqref{matchedscalarproduct1:positivity} will then follow from the estimates \eqref{id:matched-scalar-product-intermediaire4}, \eqref{id:matched-scalar-product-intermediaire5} and \eqref{matchedscalarproduct1:estimate-IV} that we prove below for $I$, $II$ and $IV$ respectively, by taking $a$ close enough to $(2,0)$, $\eta $ to $0$ and $\nu$ small enough.

\smallskip

\noindent \underline{Lower bounds for $I$ and $II$}. We compute that for $z$ close to $a$:
\begin{equation} \label{bd:estimate-U2nu-near-a}
U_{2,\nu}(z)=\frac{8\nu^2}{(\nu^2+|z+a|^2)^2}=\frac{\nu^2}{2|a|^4}+O(\nu^4+\nu^2|z-a|),
\end{equation}
and
\begin{equation} \label{bd:estimate-U1+2nu/U1nu-near-a}
\frac{U_{1+2,\nu}(z)}{U_{1,\nu}(z)}=1+\frac{U_{2,\nu}(z)}{U_{1,\nu}(z)}=1+\frac{(\nu^2+|z-a|^2)^2}{(\nu^2+|z+a|^2)^2}=1+O((\nu+|z-a|)^4).
\end{equation}
Combining, we get that for $a$ close to $(2,0)$:
\begin{align} \label{id:matched-product-expansion-weights}
\frac{\omega_\nu (z)}{\gamma_\nu (z)} & =\frac{U_{1+2,\nu}(z)}{U_{1,\nu}(z)}\frac{\nu^2}{U_{2,\nu}(z)}e^{-|z|^2/4}=c_*+O(\nu^2+|z-a|).
\end{align}
Using \eqref{bd:estimate-U2nu-near-a} and \eqref{bd:estimate-U1+2nu/U1nu-near-a} this estimate propagates to higher order derivatives:
\begin{equation} \label{id:matched-product-expansion-weights-2}
\left| \nabla \left( \frac{\omega_\nu }{\gamma_\nu } \right) \right|\lesssim 1.
\end{equation}

As $u_1$ is supported inside $\{|z-a|\leq \eta\}$ we infer from \eqref{id:matched-product-expansion-weights} that
\begin{align*}
I & =   \int \left(c_* \gamma_\nu+O\left(\omega_\nu(\nu^2+|z-a|)\right) \right) u_1^2 \, dz -c_*\nu^2 \int u_1\Phi_{u_1} dz \\
&=c_* \langle u_1 ,u_1\rangle_\flat +O\left((\nu^2+\eta)\| u\|_{{\omega_\nu}}^2\right)
\end{align*}
As $u$ satisfies \eqref{matchedscalarproduct1:orthogonality-condition}, we infer that $u_1$ satisfies the orthogonality \eqref{scalarproduct1:orthogonality-condition}. Applying \eqref{scalarproduct1:positivity} from Proposition \ref{pr:scalar-product-interior1} and \eqref{id:matched-product-expansion-weights}, we get
\begin{equation} \label{id:matched-scalar-product-intermediaire4}
I\geq \delta \int \omega_\nu u_1^2 dz  +O\left((\nu^2+\eta)\| u\|_{{\omega_\nu}}^2\right)
\end{equation}
for some $\delta>0$. As $u_1$ and $u_2$ play a symmetric role, we have similarly
\begin{equation} \label{id:matched-scalar-product-intermediaire5}
II\geq \delta \int \omega_\nu u_2^2 dz  +O\left((\nu^2+\eta)\| u\|_{{\omega_\nu}}^2\right).
\end{equation}
\underline{Upper bound for $IV$}. We recall that this term has already been estimated in \eqref{innerproduct:bd:positivity-IV} (up to replacing $u$ by $\chi^* u$ in \eqref{innerproduct:bd:positivity-IV}) so that
\begin{equation} \label{matchedscalarproduct1:estimate-IV}
|IV| \lesssim \eta  \int  (\chi^*)^2u^2 \gamma_\nu \, dz  \lesssim \eta \int \omega_\nu u^2 dz,
\end{equation}
where we used \eqref{id:matched-scalar-product-intermediaire} in the last inequality.
\end{proof}

\subsection{Coercivity of the linearized operator for the matched scalar product} \label{subsec:coercivity}

In this subsection, we first show that the bilinear form associated to the linearized operator with localized Poisson field \eqref{def:L12ztilde} for the matched scalar product \eqref{def:scalar12} is almost symmetric. Then we show it is coercive under suitable orthogonality conditions.

\begin{proposition} \label{prop:almost-symmetry-Lz}

There exists $C>0$ such that for all $a$ close to $a_{\infty}$ and $\nu>0$ small enough the following holds true. For $u$ and $v$ Schwartz functions we have
\begin{equation} \label{id:bilinear-form-linearized-1}
\langle \tilde{\Ls}^z u,v\rangle_*= \mathcal Q(u,v)+\tilde{\mathcal Q}(u,v),
\end{equation}
where $\mathcal Q$ is the following symmetric bilinear form:
\begin{align} \nonumber
\mathcal Q(u,v)&= -\int \nabla u.\nabla v \omega_\nu dz +2\int  uv (U_{1+2,\nu}-\beta) \omega_\nu dz -\int (v \nabla \Phi_{\chi^* u} +u \nabla \Phi_{\chi^* v} ).\nabla U_{1+2,\nu}\omega_\nu dz \\
\label{id:bilinear-form-linearized-2}&+c_* \nu^2 \left(\int \chi^{*}  uv dz- \int \chi^* \nabla \Phi_{\chi^* u}.\nabla \Phi_{\chi^* v}U_{1+2,\nu}dz \right)
\end{align}
and where $\tilde{\mathcal Q}$ is a lower order bilinear form (given by \eqref{bilinear-form:id:tilde-mathcal-Q}) which satisfies:
\begin{align} \label{bd:symmetry-Lz-mathcal Q}
|\tilde{\mathcal Q}(u,v)| & \leq C (\nu+|a-a_\infty|)\sqrt{ |\ln \nu|}  \min \Big( \nu \sqrt{|\ln \nu|}\| u\|_{H^1_{\omega_\nu}}\| v\|_{H^1_{\omega_\nu}} \ , \ \| u\|_{L^2_{\omega_\nu}}\| v\|_{H^1_{\omega_\nu}} \ ,  \ \| u\|_{H^1_{\omega_\nu}}\| v\|_{L^2_{\omega_\nu}} \Big)
\end{align}

\end{proposition}

\begin{remark}

Thanks to \eqref{id:bilinear-form-linearized-1}, we have the almost symmetry identity
\begin{equation} \label{id:bilinear-form-linearized-symmetry}
\langle \tilde{\Ls}^z u,v\rangle_*= \langle  u,\tilde{\Ls}^z v\rangle_* +\tilde{\mathcal Q}(u,v)-\tilde{\mathcal Q}(v,u).
\end{equation}

The estimate \eqref{bd:symmetry-Lz-mathcal Q} shows that the lower order part $\tilde{\mathcal Q}$ of the bilinear form $(u,v)\mapsto \langle \tilde{\Ls}^z u,v\rangle_*$ gains one derivative and one factor $\nu+|a-a_\infty|$ in comparison with the leading order symmetric part $\mathcal Q$. This corresponds to two gains, one gain stems from the fact that we have matched at leading order the scalar products of the interior and exterior zones in the definition of $\langle \cdot, \cdot \rangle_*$, and the other from the fact that there is an extra cancellation if the two stationary states are in their equilibrium positions $a_\infty$ and $-a_\infty$.

\end{remark}

\begin{proof}

\noindent \textbf{Step 1}. \emph{Proof of \eqref{id:bilinear-form-linearized-1} and computation of $\tilde{\mathcal Q}$}. We first decompose by \eqref{def:scalar12}:
\begin{equation} \label{bilinear-form:id:decomposition-1}
\langle \tilde{\Ls}^z u,v\rangle_* = \int  \tilde{\Ls}^z u v\omega_\nu dz -c_* \nu^2 \int \chi^* \tilde{\Ls}^z u \Phi_{\chi^*v} dz.
\end{equation}
To compute the first term in \eqref{bilinear-form:id:decomposition-1}, we use the decomposition $\Lst^z u = \frac{1}{\omega_\nu } \nabla. (\omega_\nu \nabla u ) + 2(U_{1+2, \nu} - \beta) u-\nabla U_{1+2,\nu}.\nabla \Phi_{\chi^* u}$ from \eqref{id:decomposition-Lsz-exterior}. Therefore, this term is after integrating by parts
\begin{equation} \label{bilinear-form:id:decomposition-2}
 \int  \tilde{\Ls}^z u v\omega_\nu dz =  -\int \nabla u.\nabla v \omega_\nu dz +2\int  uv (U_{1+2,\nu}-\beta) \omega_\nu dz -\int \nabla U_{1+2,\nu}.\nabla \Phi_{\chi^* u} v\omega_\nu dz.
\end{equation}
To compute the second term in \eqref{bilinear-form:id:decomposition-1}, we use the decomposition
$$
\Lst^z u =\Delta u- \nabla. \Big( u ( \nabla \Phi_{U_{1 + 2, \nu}}+\beta z)\Big)-\nabla .( U_{1 + 2, \nu}.\nabla \Phi_{\chi^*u})+(1-\chi^*)U_{1+2,\nu}u 
$$
obtained from \eqref{def:L12ztilde} using $-\Delta \Phi_{\chi^*u}=\chi^* u$. We find
\begin{equation} \label{bilinear-form:id:decomposition-3}
\int \chi^* \tilde{\Ls}^z u \Phi_{\chi^*v}  =  \int \chi^* \left( \Delta u- \nabla. \Big( u  (\nabla \Phi_{U_{1 + 2, \nu}}+\beta z)\Big)-\nabla .( U_{1 + 2, \nu}.\nabla \Phi_{\chi^*u})+(1-\chi^*)U_{1+2,\nu}u \right)\Phi_{\chi^* v} .
\end{equation}
Using $\chi^* \Delta u =\Delta (\chi^* u)-2\nabla \chi^*.\nabla u-\Delta \chi^* u$, integrating by parts and using $-\Delta \Phi_{\chi^* v}=\chi^* v$, the first term in the right-hand side of \eqref{bilinear-form:id:decomposition-3} is
\begin{equation} \label{bilinear-form:id:decomposition-4}
 \int \chi^* \Delta u  \Phi_{\chi^* v} =-\int \chi^* u v +\int \chi^* (1-\chi^*)uv-\int (2\nabla \chi^* +\Delta \chi^*)u \Phi_{\chi^* v}.
\end{equation}
Integrating by parts, the second term in the right-hand side of \eqref{bilinear-form:id:decomposition-3} is
\begin{equation} \label{bilinear-form:id:decomposition-5}
- \int \chi^* \nabla. \Big( u  \nabla \Phi_{U_{1 + 2, \nu}}+\beta z \Big) \Phi_{\chi^* v} = \int \chi^* u ( \nabla \Phi_{U_{1 + 2, \nu}}+\beta z).\nabla \Phi_{\chi^* v} +\int \nabla \chi^*.( \nabla \Phi_{U_{1 + 2, \nu}} +\beta z)u  \Phi_{\chi^* v} .
\end{equation}
Integrating by parts, the third term in the right-hand side of \eqref{bilinear-form:id:decomposition-3} is
\begin{equation} \label{bilinear-form:id:decomposition-6}
- \int \chi^* \nabla .( U_{1 + 2, \nu}.\nabla \Phi_{\chi^*u}) \Phi_{\chi^* v} = \int \chi^*  U_{1 + 2, \nu}.\nabla \Phi_{\chi^*u}.\nabla \Phi_{\chi^* v} +\int U_{1 + 2, \nu} \nabla \chi^* .\nabla \Phi_{\chi^*u}  \Phi_{\chi^* v}.
\end{equation}
Injecting \eqref{bilinear-form:id:decomposition-4}, \eqref{bilinear-form:id:decomposition-5} and \eqref{bilinear-form:id:decomposition-6} in \eqref{bilinear-form:id:decomposition-3} we get that the second term in \eqref{bilinear-form:id:decomposition-1} is
\begin{align} 
\nonumber \int \chi^* \tilde{\Ls}^z u \Phi_{\chi^*v} dz  & =  -\int \chi^* u v + \int \chi^* u ( \nabla \Phi_{U_{1 + 2, \nu}}+\beta z).\nabla \Phi_{\chi^* v} + \int \chi^*  U_{1 + 2, \nu}.\nabla \Phi_{\chi^*u}.\nabla \Phi_{\chi^* v}   \\
\nonumber &+ \int \chi^* (1-\chi^*)uv-\int (2\nabla \chi^* +\Delta \chi^*)u\Phi_{\chi^* v} +\int \nabla \chi^*. (\nabla \Phi_{U_{1 + 2, \nu}} +\beta z)u  \Phi_{\chi^* v}\\
\label{bilinear-form:id:decomposition-7} &+\int U_{1 + 2, \nu} \nabla \chi^* .\nabla \Phi_{\chi^*u}  \Phi_{\chi^* v} + \int \chi^* \left((1-\chi^*)U_{1+2,\nu}u\right)\Phi_{\chi^* v} dz.
\end{align}
Injecting \eqref{bilinear-form:id:decomposition-2} and \eqref{bilinear-form:id:decomposition-7} in \eqref{bilinear-form:id:decomposition-1} we obtain
\begin{align} \label{bilinear-form:id:decomposition-7}
\langle \tilde{\Ls}^z u,v\rangle_* & =   -\int \nabla u.\nabla v \omega_\nu dz +2\int  uv (U_{1+2,\nu}-\beta) \omega_\nu dz -\int \nabla U_{1+2,\nu}.\nabla \Phi_{\chi^* u} v\omega_\nu dz \\
\nonumber & +c_* \nu^2 \left(\int \chi^* u v - \int \chi^* u ( \nabla \Phi_{U_{1 + 2, \nu}}+\beta z).\nabla \Phi_{\chi^* v} - \int \chi^*  U_{1 + 2, \nu}.\nabla \Phi_{\chi^*u}.\nabla \Phi_{\chi^* v} \right)  \\
\nonumber &- c_*\nu^2 \Bigg( \int \chi^* (1-\chi^*)uv-\int (2\nabla \chi^* +\Delta \chi^*)u\Phi_{\chi^* v} +\int \nabla \chi^*.( \nabla \Phi_{U_{1 + 2, \nu}} +\beta z)u  \Phi_{\chi^* v}\\
\nonumber &\qquad\qquad +\int U_{1 + 2, \nu} \nabla \chi^* .\nabla \Phi_{\chi^*u}  \Phi_{\chi^* v} + \int \chi^* (1-\chi^*)U_{1+2,\nu}u \Phi_{\chi^* v} dz\Bigg)
\end{align}
Combining \eqref{bilinear-form:id:decomposition-7} and \eqref{id:bilinear-form-linearized-2}, we obtain the desired decomposition \eqref{id:bilinear-form-linearized-1} with $\tilde{\mathcal Q}$ given by
\begin{align} \label{bilinear-form:id:tilde-mathcal-Q}
\tilde{\mathcal Q}(u,v) & = \int \chi^* u \nabla \Phi_{\chi^* v}.(\nabla U_{1+2,\nu} \omega_\nu -c_* \nu^2( \nabla \Phi_{U_{1 + 2, \nu}}+\beta z))  \\
\nonumber & - c_*\nu^2 \Bigg( \int \chi^* (1-\chi^*)uv-\int (2\nabla \chi^* +\Delta \chi^*)u\Phi_{\chi^* v} +\int \nabla \chi^*.( \nabla \Phi_{U_{1 + 2, \nu}} +\beta z)u  \Phi_{\chi^* v}\\
\nonumber &\qquad\qquad +\int U_{1 + 2, \nu} \nabla \chi^* .\nabla \Phi_{\chi^*u}  \Phi_{\chi^* v} + \int \chi^* (1-\chi^*)U_{1+2,\nu}u \Phi_{\chi^* v} dz\Bigg)
\end{align}

\noindent \textbf{Step 2.} \emph{Proof of \eqref{bd:symmetry-Lz-mathcal Q}}. We decompose \eqref{bilinear-form:id:tilde-mathcal-Q} into
\begin{equation} \label{bilinear-form:id:tilde-mathcal-Q-decomposition}
\tilde{\mathcal Q}(u,v)  =\tilde{\mathcal Q}_1(u,v) +\tilde{\mathcal Q}_2(u,v) +\tilde{\mathcal Q}_3(u,v) 
\end{equation}
where
\begin{align*}
& \tilde{\mathcal Q}_1(u,v) = \int \chi^* u \nabla \Phi_{\chi^* v}.(\nabla U_{1+2,\nu} \omega_\nu -c_* \nu^2( \nabla \Phi_{U_{1 + 2, \nu}}+\beta z)) , \\
& \tilde{\mathcal Q}_2(u,v)=- c_*\nu^2 \Bigg( \int \chi^* (1-\chi^*)uv-\int (2\nabla \chi^* +\Delta \chi^*)u\Phi_{\chi^* v} +\int \nabla \chi^*.( \nabla \Phi_{U_{1 + 2, \nu}} +\beta z)u  \Phi_{\chi^* v}\\
 &\qquad\qquad \qquad \qquad +\int U_{1 + 2, \nu} \nabla \chi^* .\nabla \Phi_{\chi^*u}  \Phi_{\chi^* v} + \int \chi^* (1-\chi^*)U_{1+2,\nu}u \Phi_{\chi^* v} dz\Bigg), \\
 & \tilde{\mathcal Q}_3(u,v)=- c_*\nu^2 \int U_{1 + 2, \nu} \nabla \chi^* .\nabla \Phi_{\chi^*u}  \Phi_{\chi^* v} .
\end{align*}
The desired estimate \eqref{bd:symmetry-Lz-mathcal Q} is then a consequence of the estimate \eqref{bilinear-form:bd:tilde-mathcal-Q-1}, \eqref{bilinear-form:bd:tilde-mathcal-Q-2} and \eqref{bilinear-form:bd:tilde-mathcal-Q-3} we prove below for $\tilde{\mathcal Q}_1$, $\tilde{\mathcal Q}_2$ and $\tilde{\mathcal Q}_3$ respectively.

Let $\chi$ be the cut-off function defined by \eqref{def:chiK}. We decompose for $\eta>0$ small,
\begin{align}
\nonumber  u &=  u \chi(\frac{z-a}{\eta})+f\chi(\frac{z+a}{\eta})+u \left(1-\chi(\frac{z-a}{\eta})-\chi(\frac{z+a}{\eta})\right) \\
\label{symmetry-Lztilde-matched-product-def-vin-vex} &= u_{in,a}+u_{in,-a}+u_{ex}\\
\nonumber &= u_{in}+u_{ex}
\end{align}
and
\begin{align}
\nonumber v &= v \chi(\frac{z-a}{4\eta})+u\chi(\frac{z+a}{4\eta})+v \left(1-\chi(\frac{z-a}{4\eta})-\chi(\frac{z+a}{4\eta})\right) \\
\label{symmetry-Lztilde-matched-product-def-uin-uex} &= v_{in,a}+v_{in,-a}+v_{ex}\\
\nonumber &= v_{in}+v_{ex}.
\end{align}
Without mentioning it explicitly, we shall use numerous times that for $f=u,v$,
$$
|f_{ex}|+|f_{in,a}|+|f_{in,-a}|+|f_{in}|\lesssim |f| \quad \mbox{and} \quad|\nabla f_{ex}|+|\nabla f_{in,a}|+|\nabla f_{in,-a}|+|\nabla f_{in}|\lesssim |f|+|\nabla f| 
$$
and that $\omega_\nu \approx_\eta 1$ on the set $\{z\leq 20\}\cap_{\pm}\{|z\pm a|\geq \eta\}$ which contains the supports of $\chi^* u_{ex}$ and $\chi^* v_{ex}$.

\smallskip 

\noindent \underline{Estimate for $\tilde{\mathcal Q}_1(u,v)$}. We decompose:
\begin{align} 
\nonumber \tilde{\mathcal Q}_1(u,v) & = \int  (u_{in,a} \nabla \Phi_{ v_{in,a}}+u_{in,-a} \nabla \Phi_{ v_{in,-a}}+u_{in,a} \nabla \Phi_{ v_{in,-a}+\chi^* v_{ex}}+u_{in,-a} \nabla \Phi_{ v_{in,a}+\chi^* v_{ex}}+u_{ex} \nabla \Phi_{\chi^* v}) \\
\label{eva-air-1} &\qquad \qquad \qquad  .(\nabla U_{1+2,\nu} \omega_\nu -c_* \nu^2\chi^*( \nabla \Phi_{U_{1 + 2, \nu}}+\beta z))dz
\end{align}
We compute that in the above integral
\begin{align}
\nonumber \nabla U_{1+2,\nu} \omega_\nu-c_* \nu^2( \nabla \Phi_{U_{1 + 2, \nu}}+\beta z) &= \nabla U_{1,\nu}\frac{\nu^4 e^{-\beta |z|^2/2}}{U_{1,\nu}U_{2,\nu}}+\nabla U_{2,\nu}  \omega_\nu-c_* \nu^2( \nabla \Phi_{U_{1, \nu}}+\nabla \Phi_{U_{2,\nu}}+\beta z) \\
\label{eva-air-100}&= \nu^2 \nabla \Phi_{U_{1,\nu}} \left(\frac{\nu^2}{U_{2,\nu}}e^{-\beta |z|^2/2}-c_*\right)-c_* \nu^2( \nabla \Phi_{U_{2,\nu}}+\beta z).
\end{align}
Let now $z$ be a point near $a$. By definition of $c_*$, and by definition of $a_\infty$, we have
\begin{equation} \label{eva-air-101}
\frac{\nu^2}{U_{2,\nu}(a)}e^{-\beta |a|^2/2}-c_*=O(\nu^2) \quad \mbox{and} \quad \nabla \Phi_{U_{2,\nu}}(a_\infty)+\beta z(a_\infty)=O(\nu^2).
\end{equation}
Since
$$
\nabla \left(\frac{\nu^2}{U_{2,\nu}}e^{-\beta |z|^2/2}\right)=-\left( \frac{\nabla U_{2,\nu}}{U_{2,\nu}}+\beta z \right)\frac{\nu^2}{U_{2,\nu}}e^{-\beta |z|^2/2}=-\left(\nabla \Phi_{U_{2,\nu}}+\beta z \right)\frac{\nu^2}{U_{2,\nu}}e^{-\beta |z|^2/2},
$$
this implies
$$
\nabla \left(\frac{\nu^2}{U_{2,\nu}}e^{-\beta |z|^2/2}\right) (a_\infty)=O(\nu^2).
$$
Combining, we perform Taylor expansions and get
\begin{align}
\nonumber & \frac{\nu^2}{U_{2,\nu}(z)}e^{-\beta |z|^2/2}-c_*  \ =\frac{\nu^2}{U_{2,\nu}(a)}e^{-\beta |a|^2/2}-c_*+\frac{\nu^2}{U_{2,\nu}(z)}e^{-\beta |z|^2/2}-\frac{\nu^2}{U_{2,\nu}(a)}e^{-\beta |a|^2/2} \\
\nonumber &= O(\nu^2) +(z-a).\nabla \left(\frac{\nu^2}{U_{2,\nu}(\cdot)}e^{-\beta |\cdot |^2/2}\right)(a)+O(|z-a|^2)\\
\nonumber &= O(\nu^2) +(z-a).\left(\nabla \left(\frac{\nu^2}{U_{2,\nu}(\cdot)}e^{-\beta |\cdot |^2/2}\right)(a)-\nabla \left(\frac{\nu^2}{U_{2,\nu}(\cdot)}e^{-\beta |\cdot |^2/2}\right)(a_\infty)\right)+O(|z-a|^2)\\
\label{eva-air-102}&= O(\nu^2+|z-a||a-a_\infty|+ |z-a|^2)
\end{align}
Injectin \eqref{eva-air-100} and \eqref{eva-air-101}, using that $\nabla \Phi_{U_{1,\nu}}=O((\nu+|z-a|)^{-1})$ and a Taylor expansion, we obtain
\begin{align}
\label{eva-air-2} \nabla U_{1+2,\nu}(z) \omega_\nu(z)-c_* \nu^2\chi^*(z)( \nabla \Phi_{U_{1 + 2, \nu}}(z)+\beta z)   &= O((\nu^2+\nu|a-a_\infty|)(\nu+ |z-a|)).
\end{align}
For the first term in \eqref{eva-air-1}, we use \eqref{eva-air-2} and then the estimate \eqref{annexe:bd:poisson-field-divergence-2} with variable $y=z-a$ and $\lambda=\nu$ and get
\begin{align*}
& \left|\int  u_{in,a} \nabla \Phi_{ v_{in,a}}.(\nabla U_{1+2,\nu} \omega_\nu -c_* \nu^2\chi^* ( \nabla \Phi_{U_{1 + 2, \nu}}+\beta z)\right| \\
& = \left| \int O((\nu^2+\nu|a-a_\infty|)(\nu+ |z-a|)) |u_{in,a}|) \nabla \Phi_{ v_{in,a}} \right|\\
& \lesssim (\nu^2+\nu |a-a_\infty|) \left( \int (\nu+|z-a|)^2 u_{in,a}^2 dz \right)^{\frac 12} \left( \int (\nu+|z-a|)^2 v_{in,a}^2 dz \right)^{\frac 12}\\
&+ (\nu^2+\nu |a-a_\infty|) \int  |u_{in,a}| dz  \int |v_{in,a}| dz .
\end{align*}
Using the Hardy-Poincar\'e inequality \eqref{bd:hardy-poincare-H1-omeganu} and \eqref{bd:equivalence-weight-omeganu}, we bound for $f=u,v$,
$$
 \left( \int (\nu+|z-a|)^2 f_{in,a}^2 dz \right)^{\frac 12} \lesssim \min \left( \| f\|_{H^1_{\omega_\nu}} \ , \nu^{-1}\| f\|_{L^2_{\omega_\nu}} \right)  ,
$$
and, using Cauchy-Schwarz with $\| (\nu+|z-a|)^{-1}\mathbbm 1(|z-a|<\eta)\|_{L^2}\approx \sqrt{|\ln \nu|}$ or  $\| (\nu+|z-a|)^{-2}\mathbbm 1(|z-a|<\eta)\|_{L^2}\approx \nu^{-1}$ and \eqref{bd:equivalence-weight-omeganu}, then \eqref{bd:hardy-poincare-H1-omeganu}, we bound
\begin{align} 
\nonumber \| f_{in,a}\|_{L^1} & \lesssim  \min \left( \sqrt{|\ln \nu|} \| (\nu+|z-a|)^{-1} f\|_{L^2_{\omega_\nu}} \ , \nu^{-1}\| f\|_{L^2_{\omega_\nu}} \right)\\
\label{eva-air-1000} & \lesssim  \min \left( \sqrt{|\ln \nu|}\| f\|_{H^1_{\omega_\nu}} \ , \nu^{-1}\| f\|_{L^2_{\omega_\nu}} \right) .
\end{align}
Combining the three inequalities above shows
\begin{align*}
& \left|\int u_{in,a} \nabla \Phi_{ v_{in,a}}.(\nabla U_{1+2,\nu} \omega_\nu -c_* \nu^2\chi^*  ( \nabla \Phi_{U_{1 + 2, \nu}}+\beta z)\right| \\
&\lesssim (\nu+|a-a_{\infty}|)|\ln \nu|^{1/2} \min \left(\nu \sqrt{|\ln \nu|}  \| u\|_{H^1_{\omega_\nu}}\| v\|_{H^1_{\omega_\nu}} \ , \  \| u\|_{L^2_{\omega_\nu}}\| v\|_{H^1_{\omega_\nu}} \ , \ \| u\|_{H^1_{\omega_\nu}}\| v\|_{L^2_{\omega_\nu}}\right).
\end{align*}
The second term in \eqref{eva-air-1} at $-a$ can be bounded by the exact same estimates, so that we obtain
\begin{align}
\nonumber & \left|\int  (u_{in,a} \nabla \Phi_{ v_{in,a}}+u_{in,-a} \nabla \Phi_{ v_{in,-a}}).(\nabla U_{1+2,\nu} \omega_\nu -c_* \nu^2\chi^* ( \nabla \Phi_{U_{1 + 2, \nu}}+\beta z))\right| \\
\label{br086}&\lesssim  \min \Big( \nu (\nu+|a-a_{\infty}|)|\ln \nu| \| u\|_{H^1_{\omega_\nu}}\| v\|_{H^1_{\omega_\nu}} \ , \  (\nu+|a-a_{\infty}|)|\ln \nu|^{1/2} \| u\|_{L^2_{\omega_\nu}}\| v\|_{H^1_{\omega_\nu}}\ ,\\
& \qquad \qquad \qquad  \ (\nu+|a-a_{\infty}|)|\ln \nu|^{1/2} \| u\|_{H^1_{\omega_\nu}}\| v\|_{L^2_{\omega_\nu}} \Big)
\end{align}
For the third term in \eqref{eva-air-1} we bound using \eqref{eva-air-2},
\begin{align}
\nonumber & \left| \int u_{in,a} \nabla \Phi_{ v_{in,-a}+\chi^* v_{ex}} (\nabla U_{1+2,\nu} \omega_\nu -c_* \nu^2\chi^* ( \nabla \Phi_{U_{1 + 2, \nu}}+\beta z))\right| \\
\nonumber &\lesssim (\nu^2+\nu|a-a_\infty|) \| (\nu+ |z-a|)) u_{in,a}\|_{L^1} \| \nabla \Phi_{ v_{in,-a}+\chi^* v_{ex}}\|_{L^\infty (|z-a|<2\eta)} \\
\label{eva-air-200} &\lesssim (\nu^2+\nu|a-a_\infty|) \| (\nu+ |z-a|)) u_{in,a}\|_{L^1} \| { v_{in,-a}+\chi^* v_{ex}}\|_{L^1}.
\end{align}
Above, we used that for $z$ in the support of $u_{in,a}$ and $z'$ in the support of $v_{in,-a}+\chi^* v_{ex}$ we have $|z-z'|\geq 2\eta$ so that $|\nabla \Phi_{v_{in,-a}+\chi^* v_{ex}}(z)|\lesssim \int |\ln |z-z'|| |v_{in,-a}+\chi^* v_{ex}|(z')dz'\lesssim \| v_{in,-a}+\chi^* v_{ex}\|_{L^1}$. We have by Cauchy-Schwarz, \eqref{bd:equivalence-weight-omeganu} and \eqref{bd:hardy-poincare-H1-omeganu}:
\begin{align}
\nonumber \| (\nu+ |z-a|)) u_{in,a}\|_{L^1} & \lesssim  \| (\nu+ |z-a|)) u_{in,a}\|_{L^2}\lesssim \| (\nu+ |z-a|))^{-1} u_{in,a}\|_{L^2_{\omega_\nu}}\\
\label{eva-air-201}& \quad \lesssim \min\left(\nu^{-1}\|  u\|_{L^2_{\omega_\nu}} \ , \ \| u\|_{H^1_{\omega_\nu}} \right).
\end{align}
Also, decomposing $v=v_{in}+v_{ex}$, using \eqref{eva-air-1000} and Cauchy-Schwarz with the fact that $\omega_\nu\approx 1$ on the support of $v_{ex}$,
\begin{align} \label{eva-air-202}
 \|  \chi^* v \|_{L^1} \lesssim  \| v_{in} \|_{L^1}+ \| \chi^* v_{ex} \|_{L^1} \lesssim \min\left(\sqrt{|\ln \nu|}\|  v\|_{H^1_{\omega_\nu}} \ , \ \nu^{-1} \| v\|_{L^2_{\omega_\nu}} \right)
\end{align}
Injecting \eqref{eva-air-201} and \eqref{eva-air-202} into \eqref{eva-air-200}, noticing that the fourth term in \eqref{eva-air-1} can be estimated by a symmetric reasoning, we get
\begin{align}
\nonumber & \left| \int (u_{in,a} \nabla \Phi_{ v_{in,-a}}+u_{in,-a} \nabla \Phi_{ v_{in,a}}) (\nabla U_{1+2,\nu} \omega_\nu -c_* \nu^2\chi^* ( \nabla \Phi_{U_{1 + 2, \nu}}+\beta z))\right| \\
\label{br087} &\lesssim  (\nu+|a-a_\infty|)|\ln \nu|^{1/2} \min \left(\nu \| u\|_{H^1_{\omega_\nu}}\| v\|_{H^1_{\omega_\nu}} \ , \  \| u\|_{H^1_{\omega_\nu}}\| v\|_{L^2_{\omega_\nu}}\ , \| u\|_{L^2_{\omega_\nu}}\| v\|_{H^1_{\omega_\nu}}\right).
\end{align}
For the fifth and last term in \eqref{eva-air-1} we decompose $\chi^* v=\tilde v_{in}+\tilde v_{ex}$ where $\tilde v_{in,a}= v \chi(\frac{4(z-a)}{\eta})+v\chi(\frac{4(z+a)}{\eta})$ and $\tilde v_{ex}=\chi^* v-\tilde v_{in}$. Using that $u_{ex}$ has support inside $\cap_\pm \{|z\pm a|\geq \eta\}$ where $|U_{1+2,\nu}|\lesssim \nu^2$ and $|\nabla \Phi_{U_{1 + 2, \nu}}|\lesssim 1$, we bound
\begin{align}
\nonumber & \left| \int u_{ex} \nabla \Phi_{ v} (\nabla U_{1+2,\nu} \omega_\nu -c_* \nu^2\chi^* ( \nabla \Phi_{U_{1 + 2, \nu}}+\beta z))\right|\\
\nonumber &\lesssim \nu^2 \| u_{ex}\omega_{\nu}\|_{L^1} \| \nabla \Phi_{ \tilde v_{in}}\|_{L^{\infty}(\cap_{\pm}|z\pm a|>\eta)}+ \nu^2 \| u_{ex}\|_{L^2_{\omega_\nu}} \| \nabla \Phi_{ \tilde v_{ex}}\|_{L^{2}_{\omega_\nu}}\\
\label{br088}&\lesssim \min \left( \nu \| u \|_{L^2_{\omega_\nu}} \|  v \|_{L^{2}_{\omega_\nu}} \ , \ \nu^2\sqrt{|\ln \nu|} \| u \|_{L^2_{\omega_\nu}} \|  v \|_{H^{1}_{\omega_\nu}}\right)
\end{align}
where we used that for $z$ in the support of $u_{ex}$ and $z'$ in the support of $\tilde v_{in}$ we have $|z-z'|\geq \eta/2$ so that $|\nabla \Phi_{\tilde v_{in}}(z)|\lesssim \int |z-z'|^{-1}|\tilde v_{in}(z')|dz'\lesssim \| \tilde v_{in}\|_{L^1}\lesssim \min(\sqrt{|\ln \nu|}\| v\|_{H^1_{\omega_\nu}}, \nu^{-1}\| v\|_{L^2_{\omega_\nu}})$ by \eqref{eva-air-202} and that  $\| \tilde v_{ex}\|_{L^2}\lesssim \| \mathbbm 1(|z|<10)|z|^{-1}\|_{L^1}\| \tilde v_{ex}\|_{L^2}\lesssim \| v\|_{L^2_{\omega_\nu}}$ using that $\tilde v_{ex}$ is supported inside $\{|z|<20\}\cap \{|z\pm a|>\eta/4\}$ and Young's inequality for convolution.

Injecting \eqref{br086}, \eqref{br087} and \eqref{br088} in \eqref{eva-air-1} shows:
\begin{equation}  \label{bilinear-form:bd:tilde-mathcal-Q-1}
|\tilde{\mathcal Q}_1(u,v)|\lesssim (\nu+|a-a_{\infty}|)\sqrt{|\ln \nu|} \min \left(\nu \sqrt{|\ln \nu|}  \| u\|_{H^1_{\omega_\nu}}\| v\|_{H^1_{\omega_\nu}} \ , \  \| u\|_{L^2_{\omega_\nu}}\| v\|_{H^1_{\omega_\nu}} \ , \ \| u\|_{H^1_{\omega_\nu}}\| v\|_{L^2_{\omega_\nu}}\right)
\end{equation}

\smallskip 

\noindent \underline{Estimate for $\tilde{\mathcal Q}_2(u,v)$}. Using that the functions $\chi^* (1-\chi^*)$, $\nabla \chi^*$ and $\Delta \chi^*$ all have their support included in the zone $\{10\leq |z|\leq 20\}$ where $\omega_\nu \approx 1$, $|\nabla \Phi_{U_{1+2,\nu}}|\lesssim 1$ and $|U_{1+2,\nu}|\lesssim \nu^2$, we bound
\begin{align} \nonumber
& \Bigg| \int \chi^* (1-\chi^*)uv+u\Phi_{\chi^* v} \left( -(2\nabla \chi^* +\Delta \chi^*) +\nabla \chi^*. (\nabla \Phi_{U_{1 + 2, \nu}} +\beta z)+ \chi^* (1-\chi^*)U_{1+2,\nu}\right) dz\Bigg|\\
 \label{bilinear-form:bd:tilde-mathcal-Q-1} &\qquad \lesssim \int_{10<|z|<20} |u||v|\omega_\nu +\int_{10<|z|<20} |u| |\Phi_{\chi^* v}|\omega_\nu \lesssim \| u\|_{L^2_{\omega_\nu}}\left( \| v\|_{L^2_{\omega_\nu}}+ \| \Phi_{\chi^* v}\|_{L^2(10<|z|<20)}\right).
\end{align}
Let $z\in \mathbb R^2$ with $10< |z|< 20$. We have for $z'$ in the support of $f_{in}$, since this set is supported in $\cup_{\pm}\{|z-a|\leq 2\eta\}\subset \{|z|\leq 6\}$ as $|a|<5$ and $ \eta<1$, that $|z-z'|>1$. Hence, using \eqref{eva-air-1000},
\begin{equation}\label{symmetry-Lz-mathcal Q-mixed-tech-1}
|\nabla^k \Phi_{f_{in}}(z)|\lesssim \int |\nabla^k\ln |z-z'|||f_{in}|dz'\lesssim \| f_{in}\|_{L^1}\lesssim \min \left( \sqrt{|\ln \nu|}\| f\|_{H^1_{\omega_\nu}} \ , \nu^{-1}\| f\|_{L^2_{\omega_\nu}} \right)
\end{equation}
for $k=0,1$. We have $\omega_\nu \lesssim \langle z \rangle^{-4}$ on $\mathbb R^2$, and $\omega_\nu \approx_\eta 1$ for $\nu$ small enough on the support of $f_{ex}$. Hence by \eqref{scalarproduct1:temp}:
\begin{equation} \label{temp666}
\| \Phi_{\chi^*f_{ex}}\|_{L^2_{\omega_\nu}}^2 \lesssim \int \langle z \rangle^{-4} |\Phi_{\chi^*f_{ex}}|^2 dz \lesssim \int \langle z \rangle^{4} \chi^{*2} f_{ex}^2\lesssim \| f\|_{L^2_{\omega_\nu}}^2
\end{equation}
We decompose $\Phi_{\chi^* v}=\Phi_{v_{in}}+\Phi_{\chi^* v_{ex}}$  in \eqref{bilinear-form:bd:tilde-mathcal-Q-1}, and then use \eqref{symmetry-Lz-mathcal Q-mixed-tech-1} and \eqref{temp666} to obtain
\begin{align*} 
& \left| \int \chi^* (1-\chi^*)uv+u\Phi_{\chi^* v} \left( -(2\nabla \chi^* +\Delta \chi^*) +\nabla \chi^*. (\nabla \Phi_{U_{1 + 2, \nu}}+\beta z) + \chi^* (1-\chi^*)U_{1+2,\nu}\right) dz \right| \\ 
&\qquad\qquad \lesssim  \| u\|_{L^2_{\omega_\nu}}\left( \| v\|_{L^2_{\omega_\nu}}+ \| \Phi_{v_{in}}\|_{L^2(10<|z|<20)}+ \| \Phi_{v_{ex}}\|_{L^2(10<|z|<20)}\right)\\
&\qquad \qquad \lesssim  \| u\|_{L^2_{\omega_\nu}}  \min \left( \sqrt{|\ln \nu|}\| v\|_{H^1_{\omega_\nu}} \ , \nu^{-1}\| v\|_{L^2_{\omega_\nu}} \right)
\end{align*}
This shows
\begin{equation}  \label{bilinear-form:bd:tilde-mathcal-Q-2}
|\tilde{\mathcal Q}_2(u,v)|\lesssim  \nu \| u\|_{L^2_{\omega_\nu}}   \min \left(\nu \sqrt{|\ln \nu|}\| v\|_{H^1_{\omega_\nu}} \ , \| v\|_{L^2_{\omega_\nu}} \right)
\end{equation}

\smallskip

\noindent \underline{Estimate for $\tilde{\mathcal Q}_3(u,v)$}. We use that $|U_{1+2,\nu}|\lesssim \nu^2$ and $\omega_\nu\approx 1$ on the support of $\nabla \chi^*$ which is contained in the zone $\{10<|z|<20\}$, and the decompositions $\nabla \Phi_{\chi^*u} =\nabla \Phi_{u_{in}}+\nabla \Phi_{\chi^* u_{ex}}$ and $\Phi_{\chi^* v}=\Phi_{v_{in}}+\Phi_{v_{ex}}$ so that by H\"older
\begin{align}
\label{bilinear-form:bd:tilde-mathcal-Q-3}  & \left| \int U_{1 + 2, \nu} \nabla \chi^* .\nabla \Phi_{\chi^*u}  \Phi_{\chi^* v} \right| \ \lesssim \nu^2 \int_{10<|z|<20} (|\nabla \Phi_{u_{in}}|+|\nabla \Phi_{\chi^* u_{ex}}|) (|\Phi_{v_{in}}|+|\Phi_{\chi^* v_{ex}}|)\omega_\nu\\
\nonumber & \lesssim  \nu^2 (\| \nabla \Phi_{u_{in}}\|_{L^1(10<|z|<20)}+ \| \nabla \Phi_{\chi^* u_{ex}}\|_{L^1(10<|z|<20)}) ( \| \Phi_{v_{in}}\|_{L^\infty(10<|z|<20)}+\| \Phi_{\chi^* v_{ex}}\|_{L^\infty(10<|z|<20)} ).
\end{align}
By \eqref{symmetry-Lz-mathcal Q-mixed-tech-1} we have
$$
\| \nabla \Phi_{u_{in}}\|_{L^1(10<|z|<20)}\lesssim \nu^{-1}\| u \|_{L^2_{\omega_\nu}}  \quad \mbox{and}\quad \| \Phi_{v_{in}}\|_{L^\infty(10<|z|<20)}\lesssim \nu^{-1}\| v \|_{L^2_{\omega_\nu}}
$$
We have using successively H\"older, Sobolev, Lebesgue estimates for the Riesz kernel, the identity $\Delta \Phi_{\chi^* u_{ex}}=-\chi^* u_{ex}$ and H\"older again:
$$
\| \nabla \Phi_{\chi^* u_{ex}}\|_{L^1(10<|z|<20)})\lesssim \| \nabla \Phi_{\chi^* u_{ex}}\|_{L^6} \lesssim \| \nabla^2 \Phi_{\chi^* u_{ex}}\|_{L^{3/2}}\lesssim \| \Delta \Phi_{\chi^* u_{ex}}\|_{L^{3/2}}\lesssim \| \chi^* u_{ex}\|_{L^{3/2}}\lesssim \| u\|_{L^2_{\omega_\nu}}.
$$
Using Young's inequality for convolution with that fact that $\chi^* v_{ex}$ has support inside $\{|z|<20\}$, we have
$$
\| \Phi_{\chi^*v_{ex}}\|_{L^\infty(10<|z|<20)}  \lesssim \| \ln |\cdot| * \chi^*v_{ex}\|_{L^\infty(10<|z|<20)} \lesssim \| \chi^* v_{ex}\|_{L^2}\lesssim \| v\|_{L^2_{\omega_\nu}}.
$$
Injecting the three inequalities above in \eqref{bilinear-form:bd:tilde-mathcal-Q-3} shows
\begin{equation}  \label{bilinear-form:bd:tilde-mathcal-Q-2}
|\tilde{\mathcal Q}_3(u,v)|\lesssim  \nu^2 \| u\|_{L^2_{\omega_\nu}}\| v\|_{L^2_{\omega_\nu}}.
\end{equation}

\end{proof}

\begin{proposition}[Coercivity of the operator for the matched product] \label{Prop:coercive-Lz-global}

There exists $\delta>0$ such that for $\eta>0$ small enough, for $a$ close to $a_\infty$ and $\nu$ close to $0$, if $u$ is a Schwartz function satisfying the orthogonality conditions \eqref{scalarproduct1:orthogonality-condition-coercivity-operator}, then
\begin{equation} \label{bd:coercivity-Lztilde-matched-product}
-\langle \tilde{\Ls}^z u,u\rangle_*\geq \delta \int |\nabla u|^2\omega_{\nu}dz+\delta \int \left(\frac{1}{(\nu+|z-a|)^2})+\frac{1}{(\nu+|z+a|)^2})+\langle z \rangle^2\right)u^2 \omega_\nu dz.
\end{equation}

\end{proposition}

\begin{proof}

Let $\chi$ be the cut-off function defined by \eqref{def:chiK}. We decompose for $\eta'>\eta$, 
\begin{align}
\nonumber u &= u\left(\chi(\frac{z-a}{\eta'})+\chi(\frac{z+a}{\eta'})\right)+u\left(1-\chi(\frac{z-a}{\eta'})-\chi(\frac{z+a}{\eta'})\right) \\
\label{bd:coercivity-Lztilde-matched-product-def-uin-uex} &= u_{in}+u_{ex}.
\end{align}
As $\tilde{\Ls}^z u_{in}=\Ls^z u_{in}$ we further write
\begin{equation} \label{bd:coercivity-Lztilde-matched-product-decomposition}
\langle \tilde{\Ls}^z u,u \rangle_*=\langle \Ls^z u_{in},u_{in} \rangle_*+\langle \tilde{\Ls}^z u_{ex},u_{ex} \rangle_*+\langle \Ls^z u_{in},u_{ex} \rangle_*+\langle \tilde{\Ls}^z u_{ex},u_{in} \rangle_*
\end{equation}
The main features of the proof are the following. The first two terms in \eqref{bd:coercivity-Lztilde-matched-product-decomposition} will be shown to be coercive thanks to Lemmas \ref{lem:almost-symmetry-Lsz-exterior} and \ref{lem:almost-symmetry-Lsz-interior}, see \eqref{bd:coercivity-Lztilde-matched-product-interior} and \eqref{bd:coercivity-Lztilde-matched-product-exterior} below. Then there are two main contributions in the remaining mixed terms in \eqref{bd:coercivity-Lztilde-matched-product-decomposition}. The first one consists of the term \eqref{bd:coercivity-Lztilde-matched-product-mixed1-1}, which involve the boundary terms due to the cut-off in the decomposition $u=u_{in}+u_{ex}$. It will be showed to be negligible in Step 5 by choosing a cut-off location $\eta'=\eta'(u)$ that depends on the function $u$, for which it is small in the zone $\cup_{\pm} \{\eta'<|z\pm a|<2\eta'\}$. The second one consists of the term \eqref{bd:coercivity-Lztilde-matched-product-mixed1-3} due to the nonlocal effects of the exterior zone in the interior zone. We crucially use the cancellation that $\Ls^z u_{in}$ has zero mean as it can be put in divergence form (due to the mass criticality of the problem at hand) to estimate that contribution in Step 3.

\medskip

\noindent \textbf{Step 1}. \emph{Estimate for the interior term}. We claim that for $a$ close $(0,2)$, $\nu$ and $\eta'$ small enough,
\begin{equation} \label{bd:coercivity-Lztilde-matched-product-interior}
\langle \Ls^z u_{in},u_{in} \rangle_* \geq  \delta \int |\nabla u_{in}|^2\omega_\nu dz+\delta \int u_{in}^2 \sum_{\pm} \frac{1}{(\nu+|z\pm a|)^2} \omega_\nu dz,
\end{equation}
and we now prove this estimate. We decompose using \eqref{def:scalar12} and \eqref{def:scalar-interior}:
\begin{equation} \label{bd:coercivity-Lztilde-matched-product-interior-decomposition}
-\langle \Ls^z u_{in},u_{in}\rangle_*= -c_* \langle \Ls^z u_{in},u_{in}\rangle+\int (\Ls^z u_{in})u_{in}\widetilde{\omega_\nu}dz, \qquad \widetilde{\omega_\nu}= c_* \gamma_\nu-\omega_\nu.
\end{equation}
For the first term in \eqref{bd:coercivity-Lztilde-matched-product-interior-decomposition}, we have that $u_{in}$ is supported in $\cup_\pm \{|z\pm a|\leq 2 \eta'\}$ and that it satisfies the orthogonality conditions \eqref{scalarproduct1:orthogonality-condition-coercivity-operator} as $\eta'>\eta$. Hence by Lemma \ref{lem:almost-symmetry-Lsz-interior}, the Hardy-type inequality \eqref{bd:hardy-interior-product} and the equivalence of the weights \eqref{id:matched-scalar-product-equivalence-weights-z<20}:
\begin{equation}  \label{bd:coercivity-Lztilde-matched-product-inter6}
- \langle \Ls^z u_{in},u_{in}\rangle\geq \delta'  \int |\nabla u_{in}|^2\omega_\nu dz+\delta' \int u_{in}^2 \sum_{\pm} \frac{1}{(\nu+|z\pm a|)^2}  \omega_\nu dz.
\end{equation}
The second term in \eqref{bd:coercivity-Lztilde-matched-product-interior-decomposition} is
\begin{align}  \label{bd:coercivity-Lztilde-matched-product-inter5}
\int (\Ls^z u_{in})u_{in}\widetilde{\omega_\nu} dz & =\int \Delta u_{in} u_{in} \widetilde{\omega_\nu} dz\\
\nonumber &-\int \left(\nabla.(u_{in}\nabla \Phi_{U_{1+2,\nu}}+U_{1+2,\nu}\nabla \Phi_{u_{in}})+\frac 12 \Lambda u_{in}\right)u_{in}\widetilde{\omega_\nu} dz.
\end{align}
To estimate the first term in \eqref{bd:coercivity-Lztilde-matched-product-inter5}, we integrate by parts
$$
\int \Delta u_{in} u_{in}  \widetilde{\omega_\nu} dz =-\int |\nabla u_{in}|^2\widetilde{\omega_\nu} dz-\int u_{in}\nabla u_{in}.\nabla \widetilde{\omega_\nu} dz,
$$
By \eqref{id:matched-product-expansion-weights} and \eqref{id:matched-product-expansion-weights-2} we have that for $i=0,1,$ for $|z\pm a|\leq 2 \eta'$,
\begin{equation} \label{bd:widetildeomeganu}
\left| \nabla^i\widetilde{\omega_\nu} \right|=\left| \nabla^i\left(\frac{\nu^2}{U_{1+2,\nu}} \left(c_*-\frac{\omega_\nu}{\gamma_\nu}\right)\right)\right|\lesssim \frac{\nu^2}{U_{1+2,\nu}}\frac{\nu^2+\eta'}{(\nu+|z\pm a|)^i}.
\end{equation}
Combining the identity and inequality above, and using $|xy|\leq x^2+y^2$ and \eqref{id:matched-scalar-product-equivalence-weights-z<20}, we bound
\begin{equation}  \label{bd:coercivity-Lztilde-matched-product-inter4}
\left| \int \Delta u_{in} u_{in} \widetilde{\omega_\nu} dz\right|\lesssim (\nu^2+\eta')\int \left(|\nabla u_{in}|^2+u_{in}^2 \sum_{\pm} \frac{1}{(\nu+|z\pm a|)^2} \right)\omega_\nu dz.
\end{equation}
We next turn to the second term in \eqref{bd:coercivity-Lztilde-matched-product-inter5}. Using $-\Delta \Phi_{f}=f$ and $\nabla \Phi_f=\Phi_{\nabla f}$, that $|U_{1,\nu}|\lesssim \nu^2(\nu+|z-a|)^{-4}$, $|\nabla U_{1,\nu}|\lesssim \nu^2(\nu+|z-a|)^{-5}$ and $|\nabla \Phi_{U_{1,\nu}}|\lesssim (\nu+|z-a|)^{-1}$ and $\nu(\nu+|z-a|)^{-1}\leq 1$, and the analogue estimates for $U_{2,\nu}$, we have
$$
\left|\nabla.(u_{in}\nabla \Phi_{U_{1+2,\nu}}+U_{1+2,\nu}\nabla \Phi_{u_{in}})+\frac 12 \Lambda u_{in}\right|\lesssim \sum_{\pm} \frac{|\nabla u_{in}|}{\nu+|z\pm a|}+\frac{|u_{in}|}{(\nu+|z\pm a|)^2}+\frac{\nu^2|\Phi_{\nabla u_{in}}|}{(\nu+|z\pm a|)^5}.
$$
Using this inequality, \eqref{bd:widetildeomeganu}, and Cauchy-Schwarz,
\begin{align} \label{bd:coercivity-Lztilde-matched-product-inter2}
& \left| \int \left(\nabla.(u_{in}\nabla \Phi_{U_{1+2,\nu}}+U_{1+2,\nu}\nabla \Phi_{u_{in}})+\frac 12 \Lambda u_{in}\right)u_{in}\widetilde{\omega_\nu} dz\right| \\ 
\nonumber  & \lesssim  (\nu^2+\eta')\int \left(|\nabla u_{in}|^2+u_{in}^2 \sum_{\pm} \frac{1}{(\nu+|z\pm a|)^2}+|\Phi_{\nabla u_{in}}|^2  \sum_{\pm} \frac{\nu^4}{(\nu+|z\pm a|)^8}\right)\omega_\nu dz. 
\end{align}
Using successively \eqref{bd:equivalence-weight-omeganu} and $U_{1+2,\nu}\approx \nu^2\sum_{\pm}(\nu+|z\pm a|)^{-4}$ we have
$$
\nu^4  \sum_{\pm} \frac{\nu^4}{(\nu+|z\pm a|)^8} \omega_\nu \lesssim \nu^4 \sum_{\pm}\frac{1}{(\nu+|z\pm a|)^4}\approx \nu^2 U_{1+2,\nu}.
$$
Using this inequality, then \eqref{scalarproduct1:temp-1} and the equivalence of weights \eqref{id:matched-scalar-product-equivalence-weights-z<20} we have
$$
\int  |\Phi_{\nabla u_{in}}|^2  \sum_{\pm} \frac{\nu^4}{(\nu+|z\pm a|)^8} \omega_\nu dz \lesssim \int \nu^2 |\Phi_{\nabla u_{in}}|^2 U_{1+2,\nu}  dz\lesssim \int |\nabla u_{in}|^2\frac{\nu^2}{U_{1+2,\nu}}dz\approx  \int |\nabla u_{in}|^2\omega_\nu dz
$$
Injecting the above inequality in \eqref{bd:coercivity-Lztilde-matched-product-inter2} shows
\begin{align} \label{bd:coercivity-Lztilde-matched-product-inter3}
& \left| \int \left(\nabla.(u_{in}\nabla \Phi_{U_{1+2,\nu}}+U_{1+2,\nu}\nabla \Phi_{u_{in}})+\frac 12 \Lambda u_{in}\right)u_{in}\widetilde{\omega_\nu} dz\right| \\ 
\nonumber  & \lesssim  (\nu^2+\eta')\int \left(|\nabla u_{in}|^2+u_{in}^2 \sum_{\pm} \frac{1}{(\nu+|z\pm a|)^2}\right)\omega_\nu dz. 
\end{align}
Injecting \eqref{bd:coercivity-Lztilde-matched-product-inter4} and \eqref{bd:coercivity-Lztilde-matched-product-inter3} in \eqref{bd:coercivity-Lztilde-matched-product-inter5} shows 
$$
\left|\int (\Ls^z u_{in})u_{in}\widetilde{\omega_\nu} dz \right|\lesssim   (\nu^2+\eta')\int \left(|\nabla u_{in}|^2+u_{in}^2 \sum_{\pm} \frac{1}{(\nu+|z\pm a|)^2}\right)\omega_\nu dz. 
$$
Injecting this inequality and \eqref{bd:coercivity-Lztilde-matched-product-inter6} in \eqref{bd:coercivity-Lztilde-matched-product-interior-decomposition} shows the desired inequality \eqref{bd:coercivity-Lztilde-matched-product-interior}.

\medskip

\noindent \textbf{Step 2}. \emph{Estimate for the exterior term}. We claim that for any $|a|\leq 5$ and $\eta'>0$, if $\nu$ is small enough:
\begin{equation} \label{bd:coercivity-Lztilde-matched-product-exterior}
- \langle \Lst^z u_{ex},u_{ex} \rangle_* \geq  \delta \int |\nabla u_{ex}|^2\omega_\nu dz+\delta \int u_{ex}^2 \left(\frac{1}{|z-a|^2}+\frac{1}{|a+z|^2}+ | z|^2\right) \omega_\nu dz,
\end{equation}
We now show \eqref{bd:coercivity-Lztilde-matched-product-exterior}. We decompose using \eqref{def:scalar12}:
\begin{equation} \label{bd:coercivity-Lztilde-matched-product-exterior-inter1}
 \langle \Lst^z u_{ex},u_{ex} \rangle_* =\langle \Lst^z u_{ex} ,u_{ex}\rangle_{L^2_{\omega_\nu}} -c^*\nu^2 \int \chi^*( \Lst^z u_{ex}) \Phi_{\chi^* u_{ex}} dz .
\end{equation}
For the first term, we have by Lemma \ref{id:decomposition-Lsz-exterior} and the Hardy-Poincar\'e type inequality \eqref{bd:hardy-poincare-H1-omeganu}:
\begin{equation} \label{bd:coercivity-Lztilde-matched-product-exterior-inter2}
- \langle \Lst^z u_{ex} ,u_{ex}\rangle_{L^2_{\omega_\nu}} \geq \delta' \int \left(|\nabla u_{ex}|^2+ u_{ex}^2 \left(\frac{1}{|z-a|^2}+\frac{1}{|z+a|^2}+ | z|^2\right) \right)\omega_\nu dz.
\end{equation}
We rewrite the second term in \eqref{bd:coercivity-Lztilde-matched-product-exterior-inter1} using \eqref{def:L12ztilde} and $\nabla \Phi_f=\Phi_{\nabla f}$:
\begin{align} \label{bd:coercivity-Lztilde-matched-product-exterior-inter3}
&  \int \chi^*( \Lst^z u_{ex}) \Phi_{\chi^* u_{ex}} dz  \\
\nonumber & =    \int \chi^* \left(\Delta u_{ex}-\nabla .(u_{ex}\nabla \Phi_{U_{1+2,\nu}})+U_{1+2,\nu}u_{ex} -\nabla U_{1+2,\nu}\Phi_{\nabla (\chi^*u_{ex})}-\beta \Lambda u_{ex} \right)  \Phi_{\chi^* u_{ex}}  dz.
 \end{align}
 Integrating by parts and using $\nabla \Phi_f=\Phi_{\nabla f}$, and then using $|xy|\leq x^2+y^2$ and that $\chi^*$ is supported in $\{|z|\leq 20\}$, we estimate the first term in \eqref{bd:coercivity-Lztilde-matched-product-exterior-inter3} by
\begin{equation}  \label{bd:coercivity-Lztilde-matched-product-exterior-inter4}
\left| \int \chi^* \Delta u_{ex}  \Phi_{\chi^* u_{ex}}  dz\right|\lesssim  \int_{|z|\leq 20} (|\nabla u_{ex}|^2+  \Phi_{\chi^* u_{ex}}^2+|  \Phi_{\nabla(\chi^* u_{ex})} |^2) dz.
 \end{equation}
Since the integral is performed over $\{|z|\leq 20\}$ and \eqref{scalarproduct1:temp}, we have
\begin{align}
& \nonumber \int_{|z|\leq 20}( \Phi_{\chi^* u_{ex}}^2+|  \Phi_{\nabla(\chi^* u_{ex})} |^2) dz \lesssim  \int \langle z \rangle^4 ( (\chi^* u_{ex})^2+| \nabla(\chi^* u_{ex}) |^2) dz \\
&\qquad \qquad \qquad \qquad \qquad \qquad \qquad  \lesssim  \int_{|z|\leq 20}  (u_{ex}^2+|\nabla u_{ex} |^2) dz\   \leq  \ C_{\eta'} \int (u_{ex}^2+|\nabla u_{ex} |^2)\omega_\nu dz,  
\label{bd:coercivity-Lztilde-matched-product-exterior-inter5}
\end{align}
where we used that 
\begin{equation} \label{bd:coercivity-Lztilde-matched-product-exterior-inter6}
\omega_\nu(z)\geq c_{\eta'}>0 \qquad \mbox{for } |z|\leq 20 \mbox{ and } \min(|z-a|,|z+a|)\geq \eta'
\end{equation}
from \eqref{bd:equivalence-weight-omeganu} for the last inequality. Combining \eqref{bd:coercivity-Lztilde-matched-product-exterior-inter4}, \eqref{bd:coercivity-Lztilde-matched-product-exterior-inter5} and \eqref{bd:coercivity-Lztilde-matched-product-exterior-inter6} shows
\begin{equation} \label{bd:coercivity-Lztilde-matched-product-exterior-inter7}
\left| \int \chi^* \Delta u_{ex}  \Phi_{\chi^* u_{ex}}  dz\right|\leq C_{\eta'} \int (u_{ex}^2+|\nabla u_{ex}|^2)\omega_\nu dz .
\end{equation}
For the remaining terms in \eqref{bd:coercivity-Lztilde-matched-product-exterior-inter3}, using that $\Delta f=-f$, and $|\nabla \Phi_{U_{1+2,\nu}}(z)|\leq C_{\eta'}$ and $|U_{1+2,\nu}(z)|+|\nabla U_{1+2,\nu}(z)|\lesssim \nu^2$ on the support of $u_{ex}$ for $\nu$ small enough, then $|xy|\leq x^2+y^2$ and the support of $\chi^*$, we get
\begin{align} \nonumber
&  \left|  \int \chi^* \left(-\nabla .(u_{ex}\nabla \Phi_{U_{1+2,\nu}})+U_{1+2,\nu}u_{ex} -\nabla U_{1+2,\nu}\Phi_{\nabla (\chi^*u_{ex})}-\beta \Lambda u_{ex} \right)  \Phi_{\chi^* u_{ex}}  d z \right| \\
\nonumber & \leq C_{\eta'} \int_{|z|\leq 20} (|\nabla u_{ex}|^2+u_{ex}^2+(\Phi_{\chi^*u_{ex}})^2+|\Phi_{\nabla(\chi^*u_{ex})}|^2)dz \\
& \leq C_{\eta'} \int (|\nabla u_{ex}|^2+u_{ex}^2)\omega_\nu dz,
 \label{bd:coercivity-Lztilde-matched-product-exterior-inter8}
 \end{align}
where we used \eqref{bd:coercivity-Lztilde-matched-product-exterior-inter5} and \eqref{bd:coercivity-Lztilde-matched-product-exterior-inter6} for the last inequality. Injecting \eqref{bd:coercivity-Lztilde-matched-product-exterior-inter7} and \eqref{bd:coercivity-Lztilde-matched-product-exterior-inter8} in \eqref{bd:coercivity-Lztilde-matched-product-exterior-inter3} shows
\begin{equation} 
\left| \int \chi^*( \Lst^z u_{ex}) \Phi_{\chi^* u_{ex}} dz \right|\leq C_{\eta'} \int (|\nabla u_{ex}|^2+u_{ex}^2)\omega_\nu dz.
 \label{bd:coercivity-Lztilde-matched-product-exterior-inter9}
\end{equation}
In turn, injecting \eqref{bd:coercivity-Lztilde-matched-product-exterior-inter2} and \eqref{bd:coercivity-Lztilde-matched-product-exterior-inter9} in \eqref{bd:coercivity-Lztilde-matched-product-exterior-inter1} and taking $\nu$ small enough shows the desired inequality \eqref{bd:coercivity-Lztilde-matched-product-exterior}.

\medskip

\noindent \textbf{Step 3}. \emph{Estimate for the mixed terms}. We claim that for some universal $C>0$, 
\begin{align} \label{bd:coercivity-Lztilde-matched-product-exterior}
& \left|\langle \Ls^z u_{in},u_{ex}\rangle_* \right|+\left|\langle \Ls^z u_{ex},u_{in}\rangle_* \right| \\
\nonumber & \leq C \int_{\cup_{\pm}\{\eta'<|z\pm a|<2\eta'\}}\left( |\nabla u|^2+\frac{u^2}{\eta^{'2}}\right)\omega_\nu dz+C_{\eta'}\nu |\ln \nu| \int (|\nabla u|^2+u^2)\omega_{\nu}dz
\end{align}
and we now prove this inequality. Using the almost symmetry identity \eqref{id:bilinear-form-linearized-symmetry} and the bound \eqref{bd:symmetry-Lz-mathcal Q} we have 
\begin{equation} \label{bd:coercivity-Lztilde-matched-product-almost-symmetry}
\langle \Ls^z u_{ex},u_{in}\rangle_*= \langle \Ls^z u_{in},u_{ex}\rangle_* +\tilde{\mathcal{Q}}(u_{ex},u_{in}) = \langle \Ls^z u_{in},u_{ex}\rangle_* +O(\nu |\ln\nu|\| u\|_{H^1_{\omega_\nu}}^2).
\end{equation}
so that it suffices to bound $\langle \Ls^z u_{in},u_{ex}\rangle_*$. Using \eqref{def:scalar12}, then \eqref{def:LszNot} and the divergence form of the operator $\Ls^z f=\nabla .(\nabla f-f\nabla \Phi_{U_{1+2,\nu}}-U_{1+2,\nu}\nabla \Phi_f-\beta z f)$ and $-\Delta \Phi_f=f$, we expand
\begin{align} \nonumber
\langle \Ls^z u_{in},u_{ex}\rangle_* & = \int (\Ls^z u_{in})u_{ex} \omega_\nu dz -\nu^2 c_* \int \chi^*(\Ls^z u_{in})\Phi_{\chi^*u_{ex}} \\
\label{bd:coercivity-Lztilde-matched-product-mixed1-1} &= \int (\Delta u_{in}-\nabla u_{in}.\nabla \Phi_{U_{1+2,\nu}}+2U_{1+2,\nu}u_{in}-\beta \Lambda u_{in})u_{ex} \omega_\nu dz \\
\label{bd:coercivity-Lztilde-matched-product-mixed1-2}  &-\int \nabla U_{1+2,\nu}.\nabla \Phi_{u_{in}}u_{ex}\omega_\nu dz\\
&\label{bd:coercivity-Lztilde-matched-product-mixed1-3}  -\nu^2 c_* \int \nabla .(\nabla u_{in}-u_{in}\nabla \Phi_{U_{1+2,\nu}}-U_{1+2,\nu}\nabla \Phi_{u_{in}}-\beta z u_{in})\chi_{ 10} \Phi_{\chi^*u_{ex}}.
\end{align}
The estimates \eqref{bd:coercivity-Lztilde-matched-product-mixed1-9}, \eqref{bd:coercivity-Lztilde-matched-product-mixed1-10} and \eqref{bd:coercivity-Lztilde-matched-product-mixed1-17} we show below for \eqref{bd:coercivity-Lztilde-matched-product-mixed1-1}, \eqref{bd:coercivity-Lztilde-matched-product-mixed1-2} and \eqref{bd:coercivity-Lztilde-matched-product-mixed1-3} respectively, combined with \eqref{bd:coercivity-Lztilde-matched-product-almost-symmetry}, imply the desired inequality \eqref{bd:coercivity-Lztilde-matched-product-exterior}.

\smallskip

\noindent \underline{Estimate for \eqref{bd:coercivity-Lztilde-matched-product-mixed1-1}}. For the first term in \eqref{bd:coercivity-Lztilde-matched-product-mixed1-1}, we integrate by parts and find
$$
\int  \Delta u_{in}u_{ex} \omega_\nu dz =-\int \nabla u_{in}.\nabla \left(u_{ex}\omega_\nu \right).
$$
We have $|\nabla \chi(\frac{z\pm a}{\eta'})|\lesssim \eta^{'-1}\mathbbm 1(\eta'<|z\pm a |<2\eta')$. By \eqref{bd:coercivity-Lztilde-matched-product-def-uin-uex} and the definition of $\chi$ we have
\begin{align}
\label{bd:coercivity-Lztilde-matched-product-mixed1-6} & |u_{in}|\leq |u|\mathbbm 1(\cup_{\pm}\{|z\pm a|<2\eta'\}), \quad |\nabla u_{in}|\lesssim (|\nabla u|+\eta^{'-1}|u|)\mathbbm 1(\cup_{\pm}\{|z\pm a|<2\eta'\}) ,\\
\label{bd:coercivity-Lztilde-matched-product-mixed1-7} & |u_{ex}|\leq |u|\mathbbm 1(\cap_{\pm}\{|z\pm a|>\eta'\}), \quad |\nabla u_{ex}|\lesssim (|\nabla u|+\eta^{'-1}|u|)\mathbbm 1(\cap_{\pm}\{|z\pm a|>\eta'\}) .
\end{align}
The integral is located in $\cup_{\pm}\{\eta'< |z\pm|<2\eta'\}$, where $|\nabla \omega_\nu|\lesssim \eta^{'-1}\omega_\nu  $ from \eqref{def:omega12}. Using this, \eqref{bd:coercivity-Lztilde-matched-product-mixed1-6} and \eqref{bd:coercivity-Lztilde-matched-product-mixed1-7} and $|xy|\leq x^2+y^2$,
\begin{equation}\label{bd:coercivity-Lztilde-matched-product-mixed1-5}
\left| \int  \Delta u_{in}u_{ex} \omega_\nu dz \right|\lesssim \int_{\cup_{\pm}\{\eta'<|z\pm a|<2\eta'\}}\left( |\nabla u|^2+\frac{u^2}{\eta^{'2}}\right)\omega_\nu dz. 
\end{equation}
The remaining terms in \eqref{bd:coercivity-Lztilde-matched-product-mixed1-1} are also located in $\cup_{\pm}\{\eta'<|z\pm a|<2\eta'\}$. Since $|\nabla \Phi_{U_{1+2,\nu}}|\lesssim \sum_\pm (\nu+|z\pm a|)^{-1}$ and $|U_{1+2,\nu}|\lesssim \sum_\pm \nu^2(\nu+|z\pm a|)^{-4}$, we have $|\nabla \Phi_{U_{1+2,\nu}}|\lesssim \eta^{'-1}$ and $|U_{1+2,\nu}|\leq C_{\eta'}\nu^2$ on this set. Using this, $|\Lambda u_{in}|\lesssim |\nabla u_{in}|+ |u_{in}|$, \eqref{bd:coercivity-Lztilde-matched-product-mixed1-6}, \eqref{bd:coercivity-Lztilde-matched-product-mixed1-7} and $|xy|\leq x^2+y^2$,
\begin{equation} \label{bd:coercivity-Lztilde-matched-product-mixed1-8}
\left| \int (-\nabla u_{in}.\nabla U_{1+2,\nu}+2U_{1+2,\nu}u_{in}-\beta \Lambda u_{in})u_{ex} \omega_\nu dz \right|\lesssim \int_{\cup_{\pm}\{\eta'<|z\pm a|<2\eta'\}}\left( |\nabla u|^2+\frac{u^2}{\eta^{'2}}\right)\omega_\nu dz. 
\end{equation}
Combining \eqref{bd:coercivity-Lztilde-matched-product-mixed1-5} and \eqref{bd:coercivity-Lztilde-matched-product-mixed1-8} shows
\begin{equation} \label{bd:coercivity-Lztilde-matched-product-mixed1-9}
\left| \int (\Delta u_{in}-\nabla u_{in}.\nabla \Phi_{U_{1+2,\nu}}+2U_{1+2,\nu}u_{in}-\beta \Lambda u_{in})u_{ex} \omega_\nu dz\right|\lesssim \int_{\cup_{\pm}\{\eta'<|z\pm a|<2\eta'\}}\left( |\nabla u|^2+\frac{u^2}{\eta^{'2}}\right)\omega_\nu dz.
\end{equation}

\smallskip

\noindent \underline{Estimate for \eqref{bd:coercivity-Lztilde-matched-product-mixed1-2}}. We estimate that for $|z|\leq 20$, by Cauchy-Schwarz and as the support of $u_{in}$ is included in $\{|z|\leq 20\}$,
\begin{align}\nonumber
|\Phi_{u_{in}}(z)| & \lesssim \int |\ln |z-z'|| |u_{in}(z')|dz' \lesssim \left(\int_{|z'|\leq20} |\ln|z-z'||^2\frac{dz'}{\omega_\nu (z')} \right)^{\frac 12} \left(\int u_{in}^2(z') \omega_{\nu}(z')dz' \right)^{\frac 12} \\
 \label{bd:coercivity-Lztilde-matched-product-mixed1-200}  &\lesssim \nu^{-1}|\ln \nu| \left(\int u^2 \omega_\nu dz' \right)^{\frac 12}
\end{align}
where we used that $\omega_\nu (z')\approx \frac{\nu^2}{U_{1+2,\nu}(z')}$ for $|z|\leq 20$ from \eqref{id:matched-scalar-product-equivalence-weights-z<20} and $\int |\ln |z-z'||U_{1+2,\nu}(z')dz'\leq |\ln \nu|$ and \eqref{bd:coercivity-Lztilde-matched-product-mixed1-6} for the last inequality. Similarly, using also \eqref{bd:coercivity-Lztilde-matched-product-mixed1-6} and \eqref{bd:hardy-poincare-H1-omeganu},
\begin{equation} \label{bd:coercivity-Lztilde-matched-product-mixed1-100} 
|\nabla \Phi_{u_{in}}(z)|=|\Phi_{\nabla u_{in}}(z)|\lesssim \nu^{-1} |\ln \nu| \left(\int (|\nabla u|^2+u^2)d\omega_\nu dz'\right)^{\frac 12}.
\end{equation}
Using that $|\nabla U_{1+2,\nu}|\leq C_{\eta'} \nu^2$ on the support of $u_{ex}$ and \eqref{bd:coercivity-Lztilde-matched-product-mixed1-100}  we find
\begin{align} \nonumber
\left| \int \nabla U_{1+2,\nu}.\nabla \Phi_{u_{in}}u_{ex}\omega_\nu dz\right| & \leq C_{\eta'} \nu |\ln \nu| \left(\int (|\nabla u|^2+u^2)d\omega_\nu dz'\right)^{\frac 12} \int |u_{ex}|\omega_{\nu} \\
\label{bd:coercivity-Lztilde-matched-product-mixed1-10}  & \leq C_{\eta'}\nu |\ln \nu| \int (|\nabla u|^2+u^2)\omega_\nu dz,
\end{align}
where we used Cauchy-Schwarz, \eqref{bd:coercivity-Lztilde-matched-product-mixed1-7}  and $\int \omega_\nu dz \lesssim 1$ for the last inequality. 

\smallskip

\noindent \underline{Estimate for \eqref{bd:coercivity-Lztilde-matched-product-mixed1-3}}. Integrating by parts,
\begin{align*}
& \int \nabla .(\nabla u_{in}-u_{in}\nabla \Phi_{U_{1+2,\nu}}-U_{1+2,\nu}\nabla \Phi_{u_{in}}-\beta z u_{in})\chi^* \Phi_{\chi^*u_{ex}}dz \\
& =\int  (\nabla u_{in}-u_{in}\nabla \Phi_{U_{1+2,\nu}}-U_{1+2,\nu}\nabla \Phi_{u_{in}}-\beta z u_{in}). \nabla ( \chi^* \Phi_{\chi^*u_{ex}})dz .
\end{align*}
We then estimate by using $|\nabla \Phi_{U_{1+2,\nu}}|\lesssim \sum_\pm (\nu+|z\pm a|)^{-1}$, the support of $u_{in}$ and the H\"older inequality,
\begin{align} \label{bd:coercivity-Lztilde-matched-product-mixed1-14}
&\left| \int \nabla .(\nabla u_{in}-u_{in}\nabla \Phi_{U_{1+2,\nu}}-U_{1+2,\nu}\nabla \Phi_{u_{in}}-\beta z u_{in})\chi^* \Phi_{\chi^*u_{ex}}dz \right| \\
\nonumber &\lesssim \left( \| \nabla u_{in}\|_{L^1}+\|u_{in}\sum_\pm \frac{1}{\nu+|z\pm a|}\|_{L^1}+\| \nabla \Phi_{u_{in}}\|_{L^\infty}\int U_{1+2,\nu}dz\right) \| \nabla ( \chi^* \Phi_{\chi^*u_{ex}})\|_{L^\infty}.
\end{align}
We have by \eqref{scalarproduct1:estimate-u1-L1} and  $\omega_{\nu}\approx \frac{\nu^2}{U_{1+2,\nu}}$ on the support of $u_{in}$,
\begin{align} \nonumber
 \| \nabla u_{in}\|_{L^1}^2+\|u_{in}\sum_\pm \frac{1}{\nu+|z\pm a|}\|_{L^1}^2 & \lesssim \nu^{-2}\int (|\nabla u_{in}|^2+u_{in}^2 \sum_\pm \frac{1}{(\nu+|z\pm a|)^2} )\omega_\nu dz \\
\label{bd:coercivity-Lztilde-matched-product-mixed1-15} &\lesssim \nu^{-2}\int (|\nabla u|^2+u^2)\omega_{\nu}dz,
\end{align}
where we used \eqref{bd:coercivity-Lztilde-matched-product-mixed1-6} and \eqref{bd:hardy-poincare-H1-omeganu} in the last inequality. We next estimate that for $|z|\leq 20$, using the supports of $\chi^*$ and Cauchy-Schwarz,
\begin{align}\nonumber
|\Phi_{\chi^*u_{ex}}(z)| & \lesssim \int_{|z'|\leq 20} |\ln |z-z'|| |u_{ex}(z')|dz' \lesssim \left(\int_{|z'|\leq 20} |\ln|z-z'||^2 dz' \right)^{\frac 12} \left(\int_{|z'|\leq 20} u_{ex}^2(z') dz' \right)^{\frac 12} \\
\label{bd:coercivity-Lztilde-matched-product-mixed1-16}&\leq C_{\eta'} \left(\int u^2 \omega_\nu dz' \right)^{\frac 12},
\end{align}
where we used \eqref{bd:coercivity-Lztilde-matched-product-mixed1-7} and that $\omega_{\nu}\geq c_{\eta'}>0$ for $z$ in the support of $u_{ex}$ with $|z|\leq 20$. We obtain similarly using $\nabla \Phi_f=\Phi_{\nabla f}$ and \eqref{bd:coercivity-Lztilde-matched-product-mixed1-7} that 
\begin{equation} \label{bd:coercivity-Lztilde-matched-product-mixed1-18}
|\nabla \Phi_{\chi^*u_{ex}}(z)| \leq C_{\eta'} \left(\int |\nabla u|^2+u^2) \omega_\nu dz' \right)^{\frac 12}
\end{equation}
Injecting \eqref{bd:coercivity-Lztilde-matched-product-mixed1-15}, \eqref{bd:coercivity-Lztilde-matched-product-mixed1-100}, $\int U_{1+2,\nu}dz\lesssim 1$ and \eqref{bd:coercivity-Lztilde-matched-product-mixed1-16} in \eqref{bd:coercivity-Lztilde-matched-product-mixed1-14} shows
\begin{align}
&\left| \int \nabla .(\nabla u_{in}-u_{in}\nabla \Phi_{U_{1+2,\nu}}-U_{1+2,\nu}\nabla \Phi_{u_{in}}-\beta z u_{in})\chi^* \Phi_{\chi^*u_{ex}}dz\right| \nonumber \\
& \qquad \qquad  \qquad \qquad \qquad \qquad  \leq C_{\eta'}|\ln \nu| \nu^{-1} \int ( |\nabla u|^2+u^2)\omega_\nu dz.  \label{bd:coercivity-Lztilde-matched-product-mixed1-17}
\end{align}

\medskip

\noindent \textbf{Step 4}. \emph{End of the proof}. Injecting \eqref{bd:coercivity-Lztilde-matched-product-exterior} and \eqref{bd:coercivity-Lztilde-matched-product-interior} in \eqref{bd:coercivity-Lztilde-matched-product-decomposition}, we obtain that for some universal constants $C,\delta>0$,
\begin{align} 
- \langle \tilde{\Ls}^z u,u\rangle_* & \geq \delta \int \Big[|\nabla u|^2+u^2\big(\sum_{\pm}\frac{1}{(\nu+|z\pm a|)^2}+|z|^2\big)\Big]\omega_{\nu}dz \nonumber \\
& \qquad \qquad -C \int_{\cup_{\pm}\{\eta'<|z\pm a|<2\eta'\}}\left( |\nabla u|^2+\frac{u^2}{\eta^{'2}}\right)\omega_\nu dz, \label{bd:coercivity-Lztilde-matched-product-end-1}
\end{align}
if $a$ is close enough to $(2,0)$, $\eta'>\eta$ are small enough, and $\nu$ is small enough, up to possibly reducing the value of $\delta$. Let $N\in \mathbb N$ be such that $2^N\eta\ll 1$. Note that by taking $\eta$ small enough, we can ensure that $N$ is as large as we want. Performing a dyadic decomposition,
\begin{align*}
\int \left( |\nabla u|^2+u^2 \sum_{\pm}\frac{1}{(\nu+|z\pm a|)^2}\right)\omega_\nu dz \gtrsim \sum_{n=0}^{N-1} \int_{\cup_{\pm}\{2^n \eta<|z\pm a|<2^{n+1} \eta\}} \left( |\nabla u|^2+\frac{u^2}{2^{2n}\eta^{2}}\right)\omega_\nu dz \\
\geq N \min_{0\leq n \leq N-1} \int_{\cup_{\pm}\{2^n \eta<|z\pm a|<2^{n+1} \eta\}} \left( |\nabla u|^2+\frac{u^2}{2^{2n}\eta^{'2}}\right)\omega_\nu dz.
\end{align*}
Let $0\leq n_0\leq N-1$ be the integer at which the minimum above is attained, and set $\eta'=2^{n_0}\eta$. Then
\begin{equation} \label{bd:coercivity-Lztilde-matched-product-end-2}
 \int_{\cup_{\pm}\{ \eta'<|z\pm a|< 2\eta'\}} \left( |\nabla u|^2+\frac{u^2}{\eta^{'2}}\right)\omega_\nu dz \leq \frac{C}{N}\int \left( |\nabla u|^2+u^2 \sum_{\pm}\frac{1}{(\nu+|z\pm a|)^2}\right)\omega_\nu dz .
\end{equation}
Injecting \eqref{bd:coercivity-Lztilde-matched-product-end-2} in \eqref{bd:coercivity-Lztilde-matched-product-end-1} and taking $N$ large enough shows the desired inequality \eqref{bd:coercivity-Lztilde-matched-product-decomposition}. \end{proof}

\subsection{Coercivity for a second decomposition} \label{subsec:coercivity-second-decomposition}


We introduce the decomposition
\begin{equation} \label{id:coercivity-definition-projection}
u = \hat u +a_{0} \phi_{0}, 
\end{equation}
where the projection $\hat u=\hat u [\nu,a](u)$ and the parameter $a_{0}=a_0[\nu,a](u)$ are chosen so that
\begin{equation} \label{id:coercivity-orthogonalite-projection}
\langle \hat u , \phi_{0}\rangle_{*}=0.
\end{equation}
This projection satisfies the following properties.

\begin{proposition} \label{prop:decomp_vephat} 

For all $a$ close to $a_\infty$ and $\nu$ small enough, for all $u \in H^1_{\omega_\nu}$ there exits a unique decomposition \eqref{id:coercivity-definition-projection} with the property \eqref{id:coercivity-orthogonalite-projection}. Furthermore, there exists $C>0$ such that if $u$ satisfies the orthogonality conditions $\int_{|z\pm a|< \eta} u dz =0$, we have the estimates
\begin{equation} \label{bd:coercivity-control-adapted-norm-hatu}
\langle u, u\rangle_*\leq \langle \hat u, \hat u\rangle_*
\end{equation}
and
\begin{equation} \label{bd:coercivity-control-adapted-norm-hatu-2}
|\langle \hat u, \hat u\rangle_*| \leq C \| u\|_{L^2_{\omega_\nu}}^2
\end{equation}
and
\begin{equation} \label{bd:coercivity-partial-nu-adapted-norm-hat-u}
\left|\nu \partial_\nu (\langle \hat u, \hat u\rangle_*) \right|+\left| \nu  \partial_a (\langle \hat u, \hat u\rangle_*) \right|  \leq C \| u \|_{L^2_{\omega_\nu}}^2,
\end{equation}
as well as
\begin{equation} \label{bd:coercivity-projection-a0}
|a_0| \leq C \min \left( \frac{1}{\sqrt{|\ln \nu|}} \|  u \|_{L^2_{\omega_\nu}} ,\frac{1}{|\ln \nu|} \| \nabla u \|_{L^2_{\omega_\nu}} \right).
\end{equation}
If $u$ satisfies the orthogonality conditions \eqref{matchedscalarproduct1:orthogonality-condition} then one has $\langle \hat u, \hat u\rangle_*\geq 0$ and the equivalence
\begin{equation} \label{bd:coercivity-control-adapted-norm-hatu-2}
\| u\|_{L^2_{\omega_\nu}}^2 \approx \langle u,u\rangle_* \approx \langle \hat u, \hat u\rangle_*.
\end{equation}
Above, the implicit constants in the $\lesssim$ are independent of $a$ and $\nu$. 


\end{proposition}

Before we prove Proposition \ref{prop:decomp_vephat}, we show the following intermediate result.

\begin{lemma}

There exists $C_0(a)$ that is positive for $a$ close to $a_\infty$, such that for all $a$ close to $a_\infty$ and $\nu$ small enough:
\begin{align} 
\label{id:projection-norm-phi0-matched-product} & \langle \phi_0,\phi_0\rangle_*=-C_0|\ln \nu|+O(1),\\
\label{id:projection-norm-phi0-matched-product-2} & |\nu \partial_\nu \langle \phi_0,\phi_0\rangle_*|+|\ln \nu|^{-1}|\partial_a \langle \phi_0,\phi_0\rangle_*| \lesssim 1.
\end{align}

\end{lemma}

\begin{proof}

\textbf{Step 1}. \emph{A preliminary identity}. Pick $\eta>0$ small. We claim that
\begin{equation} \label{id:projection-matched-scalar-LambdaU1nu}
\langle \frac{1}{\nu^4} \Lambda U (\frac{z-a}{\nu}),f\rangle_* = -2 c_* \int_{|z-a|<\eta} f  +O\left(\min \Big( \sqrt{|\ln \nu|}\| f\|_{L^2_{\omega_\nu}}
\ , \  \|\frac{f}{\sqrt{\nu+|z-a|}}\|_{L^2_{\omega_\nu}}\Big)\right).
\end{equation}
To prove it, we decompose using \eqref{def:scalar12}
\begin{align}
\nonumber & \langle \frac{1}{\nu^4} \Lambda U (\frac{z-a}{\nu}),f\rangle_*  =\int \frac{1}{\nu^4} \Lambda U (\frac{z-a}{\nu})f\omega_\nu dz-c_*\nu^2 \int \chi^* f \Phi_{\chi^*  \frac{1}{\nu^4} \Lambda U (\frac{\cdot-a}{\nu})} dz \\
\nonumber &= c_* \int \chi_{a,\eta} f \left(  \frac{1}{\nu^4} \Lambda U (\frac{z-a}{\nu}) \frac{\nu^2}{U_{1,\nu}}-\nu^2 \Phi_{\frac{1}{\nu^4} \Lambda U (\frac{\cdot-a}{\nu})}\right)+c_*\nu^2\int (\chi_{a,\eta}-\chi^*) f\Phi_{ \frac{1}{\nu^4} \Lambda U (\frac{\cdot-a}{\nu})} \\
\nonumber & \qquad + \int f  \frac{1}{\nu^4} \Lambda U (\frac{z-a}{\nu}) \chi_{a,\eta} \left(\omega_\nu-\frac{c_*\nu^2}{U_{1,\nu}} \right)+c_*\nu^2\int \chi^* f\Phi_{(1-\chi^*)  \frac{1}{\nu^4} \Lambda U (\frac{\cdot-a}{\nu})} \\
\label{id:projection-matched-scalar-LambdaU1nu-1} &\qquad + \int f  \frac{1}{\nu^4} \Lambda U (\frac{z-a}{\nu}) \omega_\nu (1-\chi_{a,\eta})
\end{align}
We have $ \frac{1}{\nu^4} \Lambda U (\frac{\cdot-a}{\nu}) \frac{\nu^2}{U_{1,\nu}}-\nu^2 \Phi_{\frac{1}{\nu^4} \Lambda U (\frac{\cdot-a}{\nu})}=\Ms_0(\Lambda U)(\frac{\cdot-a}{\nu})=-2$ by \eqref{id:kernel-Ms0}, hence
\begin{equation} \label{id:projection-matched-scalar-LambdaU1nu-2}
 \int \chi_{a,\eta} f \left(  \frac{1}{\nu^4} \Lambda U (\frac{z-a}{\nu}) \frac{\nu^2}{U_{1,\nu}}-\nu^2 \Phi_{\frac{1}{\nu^4} \Lambda U (\frac{\cdot-a}{\nu})}\right) =-2\int \chi_{a,\eta}f
\end{equation}
We have that $\nu^2 \Phi_{ \frac{1}{\nu^4} \Lambda U (\frac{\cdot-a}{\nu})} = (z-a).\nabla \Phi_{\frac{1}{\nu^2}U(\frac{\cdot-a}{\nu})}+\frac{1}{2\pi}\int U=O(\nu^2)$ for $\zeta_*<|z-a|\lesssim 1$ so that
\begin{equation} \label{id:projection-matched-scalar-LambdaU1nu-3}
\nu^{2}\int (\chi_{a,\eta}-\chi^*) f\Phi_{ \frac{1}{\nu^4} \Lambda U (\frac{\cdot-a}{\nu})} = O(\nu^2  \| \chi^*f \|_{L^1})= O(  \nu \| f \|_{L^2_{\omega_{\nu}}})
\end{equation}
where we used \eqref{scalarproduct1:estimate-u1-L1}, \eqref{bd:equivalence-weight-U1+2nu-1} and \eqref{bd:equivalence-weight-omeganu}. We have for $a$ and $z$ close to $a_\infty$ by \eqref{def:scalar12} that
\begin{equation} \label{id:projection-matched-scalar-LambdaU1nu-4}
\omega_\nu-\frac{c_*\nu^2}{U_{1,\nu}}=\frac{\nu^2}{U_{1,\nu}}\left(\frac{\nu^2}{U_{2,\nu}}e^{-|z|^4/4} -c_*\right)=\frac{\nu^2}{U_{1,\nu}}O(\nu^2+|z-a|).
\end{equation}
Hence by Cauchy-Schwarz, using $\| \chi_{a,\eta}\frac{1}{\nu^4} \Lambda U (\frac{z-a}{\nu}) (\nu^2+|z-a|) \|_{L^2_{\frac{\nu^2}{U_{1,\nu}}}}\lesssim \sqrt{|\ln \nu|}$:
\begin{equation} \label{id:projection-matched-scalar-LambdaU1nu-4000}
\left|\int f  \frac{1}{\nu^4} \Lambda U (\frac{z-a}{\nu}) \chi_{a,\eta} \left(\omega_\nu-\frac{c_*\nu^2}{U_{1,\nu}} \right)dz \right|\lesssim \sqrt{|\ln \nu|} \| f\|_{L^2_{\omega_\nu}} .
\end{equation}
By Cauchy-Schwarz again $\int f  \frac{1}{\nu^4} \Lambda U (\frac{z-a}{\nu}) \chi_{a,\eta} \left(\omega_\nu-\frac{c_*\nu^2}{U_{1,\nu}} \right)dz =O(\| (\nu+|z-a|)^{-1/2}f\|_{L^2_{\omega_\nu}})$ using $\| \sqrt{\nu+|z-a|} \chi_{a,\eta}\frac{1}{\nu^4} \Lambda U (\frac{z-a}{\nu}) (\nu^2+|z-a|) \|_{L^2_{\frac{\nu^2}{U_{1,\nu}}}}\lesssim 1$. Combining this inequality with \eqref{id:projection-matched-scalar-LambdaU1nu-4000} shows
\begin{equation} \label{id:projection-matched-scalar-LambdaU1nu-5}
\left|\int f  \frac{1}{\nu^4} \Lambda U (\frac{z-a}{\nu}) \chi_{a,\eta} \left(\omega_\nu-\frac{c_*\nu^2}{U_{1,\nu}} \right)dz \right|\lesssim \min \left(\sqrt{|\ln \nu|}\| f\|_{L^2_{\omega_\nu}}
\ , \ \|\frac{f}{\sqrt{\nu+|z-a|}}\|_{L^2_{\omega_\nu}}\right).
\end{equation}
We have that $\Phi_{(1-\chi_*)  \frac{1}{\nu^4} \Lambda U (\frac{\cdot-a}{\nu})}$ is bounded as $|(1-\chi_*)  \frac{1}{\nu^4} \Lambda U (\frac{\cdot-a}{\nu})|\lesssim \langle z \rangle^{-4}$, so that
\begin{equation} \label{id:projection-matched-scalar-LambdaU1nu-6}
|\nu^2 \int \chi^* f\Phi_{(1-\chi^*)  \frac{1}{\nu^4} \Lambda U (\frac{\cdot-a}{\nu})}|\lesssim \nu^2 \| \chi^* f\|_{L^1}\lesssim \nu \| f\|_{L^2_{\omega_\nu}}
\end{equation}
by \eqref{scalarproduct1:estimate-u1-L1}, \eqref{bd:equivalence-weight-U1+2nu-1} and \eqref{bd:equivalence-weight-omeganu}. Using that $|\frac{1}{\nu^4}\Lambda U(\frac{z-a}{\nu})|\lesssim |z-a|^{-4}\lesssim 1$ for $|z-a|\gtrsim \eta$ we have by Cauchy-Schwarz
\begin{equation} \label{id:projection-matched-scalar-LambdaU1nu-600000}
 \int f  \frac{1}{\nu^4} \Lambda U (\frac{z-a}{\nu}) \omega_\nu (1-\chi_{a,\eta}) =\Oc(\| f\|_{L^2_{\omega_\nu}}).
\end{equation}
Injecting \eqref{id:projection-matched-scalar-LambdaU1nu-2} and \eqref{id:projection-matched-scalar-LambdaU1nu-3} using $\omega_\nu \approx 1$ for $|z\pm a|\approx \eta$, and \eqref{id:projection-matched-scalar-LambdaU1nu-4}, \eqref{id:projection-matched-scalar-LambdaU1nu-5}, \eqref{id:projection-matched-scalar-LambdaU1nu-6} and \eqref{id:projection-matched-scalar-LambdaU1nu-600000} in \eqref{id:projection-matched-scalar-LambdaU1nu-1} shows \eqref{id:projection-matched-scalar-LambdaU1nu}.

\medskip

\noindent \textbf{Step 2}. \emph{A secondary identity}. We claim that
\begin{equation} \label{id:projection-matched-scalar-partialxU1nu}
\langle \frac{1}{\nu^3} \partial_{x_1} U (\frac{z-a}{\nu}),f\rangle_* = \nu \| f\|_{L^2_{\omega_\nu}}
\end{equation}
The proof is similar to Step 1. We decompose using \eqref{def:scalar12}
\begin{align}
\nonumber & \langle \frac{1}{\nu^3} \partial_{x_1} U (\frac{z-a}{\nu}),f\rangle_* \\
\nonumber &= c_* \int \chi_{a,\eta} f \left(  \frac{1}{\nu^3} \partial_{x_1} U (\frac{z-a}{\nu}) \frac{\nu^2}{U_{1,\nu}}-\nu^2 \Phi_{\frac{1}{\nu^3} \partial_{x_1} U (\frac{\cdot-a}{\nu})}\right)+c_*\nu^2\int (\chi_{a,\eta}-\chi^*) f\Phi_{ \frac{1}{\nu^3} \partial_{x_1} U (\frac{\cdot-a}{\nu})} \\
\nonumber & \qquad + \int f  \frac{1}{\nu^3} \partial_{x_1} U (\frac{z-a}{\nu}) \chi_{a,\eta} \left(\omega_\nu-\frac{c_*\nu^2}{U_{1,\nu}} \right)+c_*\nu^2\int \chi^* f\Phi_{(1-\chi^*)  \frac{1}{\nu^3} \partial_{x_1} U (\frac{\cdot-a}{\nu})} \\
\label{id:projection-matched-scalar-partialx1U1nu-1} &\qquad + \int f  \frac{1}{\nu^3} \partial_{x_1} U (\frac{z-a}{\nu}) \omega_\nu (1-\chi_{a,\eta})
\end{align}
We have $ \frac{1}{\nu^3} \partial_{x_1} U (\frac{\cdot-a}{\nu}) \frac{\nu^2}{U_{1,\nu}}-\nu^2 \Phi_{\frac{1}{\nu^3} \partial_{x_1} U (\frac{\cdot-a}{\nu})}=\nu\Ms_0(\partial_{x_1} U)(\frac{\cdot-a}{\nu})=0$ by \eqref{id:kernel-Ms0}, hence
\begin{equation} \label{id:projection-matched-scalar-partialx1U1nu-2}
\int \chi_{a,\eta} f \left(  \frac{1}{\nu^3} \partial_{x_1} U (\frac{z-a}{\nu}) \frac{\nu^2}{U_{1,\nu}}-\nu^2 \Phi_{\frac{1}{\nu^3} \partial_{x_1} U (\frac{\cdot-a}{\nu})}\right)=0
\end{equation}
The remaining terms in \eqref{id:projection-matched-scalar-partialx1U1nu-1} are estimated exactly as \eqref{id:projection-matched-scalar-LambdaU1nu-3} using $\nu^2 \Phi_{ \frac{1}{\nu^3} \partial_{x_1} U (\frac{\cdot-a}{\nu})} =\nu^2 \partial_{x_1}\Phi_U (\frac{\cdot-a}{\nu})=O(\nu^2)$ for $\zeta_*<|z-a|\lesssim 1$, as \eqref{id:projection-matched-scalar-LambdaU1nu-600000} using $\| \chi_{a,\eta}\frac{1}{\nu^3} \partial_{x_1} U (\frac{z-a}{\nu}) (\nu+|z-a|) \|_{L^2_{\frac{\nu^2}{U_{1,\nu}}}}\lesssim \nu$, as \eqref{id:projection-matched-scalar-LambdaU1nu-6} using $\Phi_{(1-\chi_*)  \frac{1}{\nu^3} \partial_{x_1} U (\frac{\cdot-a}{\nu})}=O(\nu)$ for $|z|\lesssim 1$ and as \eqref{id:projection-matched-scalar-LambdaU1nu-4000} using $|\frac{1}{\nu^3}\partial_{x_1}U(\frac{z-a}{\nu})|\lesssim \nu^2$ for $|z-a|>\eta$, yielding
\begin{align}\label{id:projection-matched-scalar-partialx1U1nu-3}
|\nu^{2}\int (\chi_{a,\eta}-\chi^*) f\Phi_{ \frac{1}{\nu^3} \partial_{x_1} U (\frac{\cdot-a}{\nu})}| \lesssim \nu \| f \|_{L^2_{\omega_\nu}}\\
 \label{id:projection-matched-scalar-partialx1U1nu-4000}
\left|\int f  \frac{1}{\nu^3} \partial_{x_1} U (\frac{z-a}{\nu}) \chi_{a,\eta} \left(\omega_\nu-\frac{c_*\nu^2}{U_{1,\nu}} \right)dz \right|\lesssim \nu \| f\|_{L^2_{\omega_\nu}} ,\\
\label{id:projection-matched-scalar-partialx1U1nu-6}
|\nu^2 \int \chi^* f\Phi_{(1-\chi^*)  \frac{1}{\nu^3} \partial_{x_1} U (\frac{\cdot-a}{\nu})}|\lesssim \nu^2 \| f\|_{L^2_{\omega_\nu}},\\
\label{id:projection-matched-scalar-partialx1U1nu-600000}\left| \int f  \frac{1}{\nu^3} \partial_{x_1} U (\frac{z-a}{\nu}) \omega_\nu (1-\chi_{a,\eta}) \right|\lesssim \nu^2 \| f\|_{L^2_{\omega_\nu}}.
\end{align}
Injecting \eqref{id:projection-matched-scalar-partialx1U1nu-2}, \eqref{id:projection-matched-scalar-partialx1U1nu-3}, \eqref{id:projection-matched-scalar-partialx1U1nu-4000}, \eqref{id:projection-matched-scalar-partialx1U1nu-6} and \eqref{id:projection-matched-scalar-partialx1U1nu-600000} in \eqref{id:projection-matched-scalar-partialx1U1nu-1} shows the desired identity \eqref{id:projection-matched-scalar-partialxU1nu}.

\medskip

\noindent \textbf{Step 3}. \emph{Proof of the Lemma}. Injecting the expansions \eqref{def:phii_inn}, \eqref{expansion:phii_inn} and \eqref{eigenfunctions-exterior:id:expansion-phiiout-1} in \eqref{def:phii-2} we decompose
\begin{equation} \label{id:coercivity:decomposition-phi0}
\phi_0(z)= -\frac{1}{16}\sum_{\pm}\frac{1}{\nu^4} \Lambda U(\frac{z\pm a}{\nu})+\sum_{\pm}\frac{\pm L_0}{\nu^3} \partial_{x_1} U(\frac{z\pm a}{\nu})+\check \phi_0(z),
\end{equation}
with $\check \phi_0=\check \phi_0^1+\check \phi_0^2$ where
\begin{equation} \label{id:coercivity:decomposition-phi0check}
 \check \phi_0^1(z)= \sum_{\pm}\frac{\cos (2\theta_\pm)}{8|z\pm a|^2}  \quad \mbox{and}  \quad \check \phi_0^2(z)=\sum_\pm O\left(\frac{1}{|z\pm a|^2}(\frac{1}{|\ln \nu|}+\frac{\nu^2}{(\nu+|z\pm a|)^2})+\frac{\langle z\rangle^c}{|z\pm a|} \right).
\end{equation}
We then decompose:
\begin{align}
\nonumber & \langle \phi_0,\phi_0\rangle_* \\
\nonumber &=\frac{1}{256} \langle \sum_{\pm}\frac{1}{\nu^4} \Lambda U(\frac{\cdot \pm a}{\nu}),\sum_{\pm}\frac{1}{\nu^4} \Lambda U(\frac{\cdot\pm a}{\nu})\rangle_*\\
\nonumber & +\langle -\frac{1}{8} \sum_{\pm}\frac{1}{\nu^4} \Lambda U(\frac{\cdot \pm a}{\nu})+ \sum_{\pm}\frac{\pm L_0}{\nu^3} \partial_{x_1} U(\frac{z\pm a}{\nu}), \sum_{\pm}\frac{\pm L_0}{\nu^3} \partial_{x_1} U(\frac{z\pm a}{\nu})\rangle_*\\
 \label{coercivity:id:LambdaU-matched-product-1}&+2\langle -\frac{1}{16} \sum_{\pm}\frac{1}{\nu^4} \Lambda U(\frac{\cdot \pm a}{\nu})\pm\frac{ L_0}{\nu^3} \partial_{x_1} U(\frac{z\pm a}{\nu}), \check \phi_0\rangle_*+O(1)
\end{align}
where we used $|\langle \check \phi_0,\check \phi_0\rangle_*|\lesssim \| \check \phi_0\|_{L^2_{\omega_\nu}}^2\lesssim 1$ which follows from \eqref{matchedscalarproduct1:continuity} and \eqref{id:coercivity:decomposition-phi0check}. Note that the analogue estimate of \eqref{id:projection-matched-scalar-LambdaU1nu} holds for $\Lambda U(\frac{z+a}{\nu})$ and $\partial_{x_1}U(\frac{z+a}{\nu})$ by symmetry. Using \eqref{id:projection-matched-scalar-LambdaU1nu}, \eqref{id:projection-matched-scalar-partialxU1nu}, \eqref{id:coercivity:decomposition-phi0check} and $|L_0|\lesssim 1$ shows:
\begin{equation}  \label{coercivity:id:LambdaU-matched-product-2}
2\langle -\frac{1}{16} \sum_{\pm}\frac{1}{\nu^4} \Lambda U(\frac{\cdot \pm a}{\nu})\pm\frac{ L_0}{\nu^3} \partial_{x_1} U(\frac{z\pm a}{\nu}), \check \phi_0\rangle_*=O(1).
\end{equation}
Using \eqref{id:projection-matched-scalar-partialxU1nu} and $|L_0|\lesssim 1$ shows
\begin{align}
\nonumber &\left| \langle -\frac{1}{8} \sum_{\pm}\frac{1}{\nu^4} \Lambda U(\frac{\cdot \pm a}{\nu})+ \sum_{\pm}\frac{\pm L_i}{\nu^3} \partial_{x_1} U(\frac{z\pm a}{\nu}), \sum_{\pm}\frac{\pm L_i}{\nu^3} \partial_{x_1} U(\frac{z\pm a}{\nu})\rangle_*\right|\\
 \label{coercivity:id:LambdaU-matched-product-3}&\lesssim \sum_{\pm} \nu \left\|\frac{1}{\nu^4} \Lambda U(\frac{\cdot \pm a}{\nu})\right\|_{L^2_{\omega_\nu}}+ \sum_{\pm}\nu \left\|\frac{1}{\nu^3} \partial_{x_1} U(\frac{z\pm a}{\nu})\right\|_{L^2_{\omega_\nu}} \lesssim 1+\nu.
\end{align}
Next, combining \eqref{id:projection-matched-scalar-LambdaU1nu-1}, \eqref{id:projection-matched-scalar-LambdaU1nu-2}, \eqref{id:projection-matched-scalar-LambdaU1nu-3}, \eqref{id:projection-matched-scalar-LambdaU1nu-4} and \eqref{id:projection-matched-scalar-LambdaU1nu-6}, with the fact that $\int \chi_{a,\eta}\frac{1}{\nu^4} \Lambda U(\frac{\cdot-a }{\nu})dz=O(1)$ and $\int (\chi^*-\chi_{a,\eta})\frac{1}{\nu^4} \Lambda U(\frac{\cdot-a }{\nu})dz=O(1)$ as $\int \Lambda U=0$, and $ \| \frac{1}{\nu^4} \Lambda U(\frac{\cdot-a }{\nu})\|_{L^2_{\omega_\nu}}\lesssim \nu^{-1}$:
\begin{equation}  \label{coercivity:id:LambdaU-matched-product-4}
\langle \sum_{\pm}\frac{1}{\nu^4} \Lambda U(\frac{\cdot \pm a}{\nu}),\sum_{\pm}\frac{1}{\nu^4} \Lambda U(\frac{\cdot\pm a}{\nu})\rangle_* =  \sum_{\pm} \int  \frac{1}{\nu^8} |\Lambda U (\frac{z\mp a}{\nu})|^2 \chi_{\pm a,\eta} \left(\omega_\nu-\frac{c_*\nu^2}{U_{i_\pm,\nu}} \right)+O(1)
\end{equation}
where $i_+=1$ and $i_-=2$. We perform a Taylor expansion and get
\begin{align*}
\omega_\nu-c_* \frac{\nu^2}{U_{1,\nu}} & =\frac{\nu^2}{U_{1,\nu}}\left(\frac{(\nu^2+|z+a|^2)^2}{8}e^{-|z|^2/4}-\frac{|2a|^4}{8}e^{-|2a|^2/4} \right) \\
& =\frac{\nu^2}{U_{1,\nu}} e^{-|a|^2/4} \left(\frac{(|z+a|^2)^2}{8}e^{(|a|^2-|z|^2)/4}-\frac{|2a|^4}{8}+O(\nu^2) \right)  \\
&= \frac{\nu^2}{U_{1,\nu}} \left( c_1 (1,0).(z-a)+c_2 |z-a|^2+c_3 |(1,0).(z-a)|^2+O(|z-a|^3)\right)
\end{align*}
with $e^{|a|^2/4} c_1=32|a|^3-8|a|^5$, $e^{|a|^2/4}c_2=8|a|^2-4|a|^4$ and $e^{|a|^2/4}c_3=16|a|^2-2|a|^6-16|a|^4$. Therefore, after changing variables $z=a+\nu y$:
\begin{equation}  \label{coercivity:id:LambdaU-matched-product-5}
 \int  \frac{1}{\nu^8} |\Lambda U (\frac{z- a}{\nu})|^2 \chi_{ a,\eta} \left(\omega_\nu-\frac{c_*\nu^2}{U_{1,\nu}} \right) = \int \chi(\frac{\nu y}{\eta})\frac{(\Lambda U)^2}{U}(c_2|y|^2+c_3 y_1^2)dy +O(1).
\end{equation}
where we used that both $\Lambda U$ and $U$ are radial so that the first term cancels. Using that $\frac{(\Lambda U)^2}{U^2}=\frac{32}{|y|^4}+O(|y|^{-6})$ as $|y|\to \infty$ we compute 
\begin{equation} \label{coercivity:id:LambdaU-matched-product-6}
\int \chi(\frac{\nu y}{\eta})\frac{(\Lambda U)^2}{U}(c_2|y|^2+c_3 y_1^2)dy=(64\pi c_2+32\pi c_3)|\ln \nu|+O(1).
\end{equation}
Injecting \eqref{coercivity:id:LambdaU-matched-product-5} and \eqref{coercivity:id:LambdaU-matched-product-6} in \eqref{coercivity:id:LambdaU-matched-product-4} shows
\begin{equation} \label{coercivity:id:LambdaU-matched-product-7}
\langle \sum_{\pm}\frac{1}{\nu^4} \Lambda U(\frac{\cdot \pm a}{\nu}),\sum_{\pm}\frac{1}{\nu^4} \Lambda U(\frac{\cdot\pm a}{\nu})\rangle_* =(128\pi c_2+64\pi c_3)|\ln \nu|+O(1).
\end{equation}
Finally, injecting \eqref{coercivity:id:LambdaU-matched-product-2}, \eqref{coercivity:id:LambdaU-matched-product-3} and \eqref{coercivity:id:LambdaU-matched-product-7} in \eqref{coercivity:id:LambdaU-matched-product-1} shows the desired identity \eqref{id:projection-norm-phi0-matched-product}. The sign of the constants follows from the explicit computations $c_2(a_\infty)=-32 e^{-1}$ and $c_3(a_\infty)=-320 e^{-1}$. The proof of \eqref{id:projection-norm-phi0-matched-product-2} follows from the exact same computations, by differentiating all expressions with respect to $\nu$ and $a$ ; we omit it.

\end{proof}

\begin{proof}[Proof of Proposition \ref{prop:decomp_vephat} ]

By the orthogonality \eqref{id:coercivity-orthogonalite-projection} the projection is given by
\begin{equation} \label{id:coercivity-projection-formula}
a_0=\frac{\langle u ,\phi_0\rangle_*}{\langle \phi_0,\phi_0\rangle_*},
\end{equation}
the denominator being nonzero by \eqref{id:projection-norm-phi0-matched-product}.

\smallskip

\noindent \underline{Proof of \eqref{bd:coercivity-control-adapted-norm-hatu}}. We have by the orthogonality \eqref{id:coercivity-orthogonalite-projection}
\begin{equation} \label{id:coercivity-projection-pythagorus}
\langle u,u\rangle_* =\langle \hat u,\hat u\rangle_*+a_0^2 \langle \phi_0,\phi_0\rangle_*.
\end{equation}
This implies \eqref{bd:coercivity-control-adapted-norm-hatu} since $\langle \phi_0,\phi_0\rangle_*<0$ by \eqref{id:projection-norm-phi0-matched-product}.

\smallskip

\noindent \underline{Proof of \eqref{bd:coercivity-control-adapted-norm-hatu-2}}. Injecting \eqref{id:coercivity-projection-formula} in \eqref{id:coercivity-projection-pythagorus} gives
\begin{equation} \label{id:matched-product-hatu-expression}
\langle \hat u,\hat u\rangle_*=\langle u,u\rangle_* -\frac{\langle u ,\phi_0\rangle_*^2}{\langle \phi_0,\phi_0\rangle_*}.
\end{equation}
We already know $|\langle u,u\rangle_*|\lesssim \| u\|_{L^2_{\omega_\nu}}^2$ from \eqref{matchedscalarproduct1:continuity} and $|\langle \phi_0,\phi_0\rangle_*|\approx |\ln \nu|$ from \eqref{id:projection-norm-phi0-matched-product} so it suffices to prove
\begin{equation} \label{bd:matched-product-u-phi0}
|\langle u ,\phi_0\rangle_*| \lesssim \min(\sqrt{|\ln \nu|} \| u\|_{L^2_{\omega_\nu}},\| u\|_{H^1_{\omega_\nu}}).
\end{equation}
Using the decomposition \eqref{id:coercivity:decomposition-phi0} we have $\langle u ,\phi_0\rangle_*=\langle u ,\sum_\pm \frac{1}{\nu^4}\Lambda U (\frac{z\pm a}{\nu})+\check \phi_0\rangle_*$. We have by \eqref{id:projection-matched-scalar-LambdaU1nu}, since $\int_{|z\pm a|<\eta} u dz=0$, that $\langle u , \frac{1}{\nu^4}\Lambda U (\frac{z\pm a}{\nu})\rangle_* =O(\min(\sqrt{|\ln \nu|}\| u\|_{L^2_{\omega_\nu}},\| u\|_{H^1_{\omega_\nu}}))$. We have by \eqref{matchedscalarproduct1:continuity}, since $\| \check \phi_0\|_{L^2_{\omega_\nu}}\lesssim 1$ by \eqref{id:coercivity:decomposition-phi0check}, that $\langle u , \check \phi_0 \rangle_* =O(\| u\|_{L^2_{\omega_\nu}})$. This shows \eqref{bd:matched-product-u-phi0} as desired.

\smallskip

\noindent \underline{Proof of \eqref{bd:coercivity-partial-nu-adapted-norm-hat-u}}. We differentiate the first term in \eqref{id:matched-product-hatu-expression}:
\begin{align*}
\nu \partial_\nu \langle u,u \rangle_*=\int u^2 \nu \partial_\nu \omega_\nu dz -2c_* \nu^2 \int \chi^* u \Phi_{\chi^* u}dz,\\
\nu \partial_a \langle u,u \rangle_*=\int u^2 \nu \partial_a \omega_\nu dz -2 \partial_a (c_*) \nu^3 \int \chi^* u \Phi_{\chi^* u}dz.
\end{align*}
Using $|\nu\partial_\nu \omega_\nu|+|\nu \partial_a \omega_\nu|\lesssim \omega_\nu$ from \eqref{def:omega12}, $\partial_a c_*=O(1)$,\eqref{scalarproduct1:temp-1}, \eqref{bd:equivalence-weight-omeganu} and \eqref{bd:equivalence-weight-U1+2nu-1} we get:
$$
|\nu \partial_\nu \langle u,u \rangle_*|+|\nu \partial_a \langle u,u \rangle_*|\lesssim \| u\|_{L^2_{\omega_\nu}}^2.
$$
The derivative with respect to $a$ and $\nu$ of the second term in \eqref{id:matched-product-hatu-expression} are estimated similarly, injecting \eqref{id:projection-norm-phi0-matched-product-2} and performing the same estimates as in the proof of \eqref{bd:coercivity-control-adapted-norm-hatu-2}.

\smallskip

\noindent \underline{Proof of \eqref{bd:coercivity-projection-a0}}. It follows from injecting \eqref{bd:matched-product-u-phi0} and \eqref{id:projection-norm-phi0-matched-product} in \eqref{id:coercivity-projection-formula}.

\smallskip

\noindent \underline{Proof of \eqref{bd:coercivity-control-adapted-norm-hatu-2}.} It follows from \eqref{bd:coercivity-control-adapted-norm-hatu}, \eqref{bd:coercivity-control-adapted-norm-hatu-2} and the coercivity \eqref{matchedscalarproduct1:positivity}.

\end{proof}

\begin{proposition}[Coercivity of the operator for the second decomposition] \label{Prop:coercive-Lz-global-second}

There exists $\delta>0$ such that for $\eta>0$ small enough, for $\nu$ close to $0$ and $|a-a_{\infty}|\lesssim \nu$, if $u$ is a Schwartz function satisfying the orthogonality conditions \eqref{scalarproduct1:orthogonality-condition-coercivity-operator}, then
\begin{equation} \label{bd:coercivity-Lztilde-matched-product-second}
-\left\langle \widehat{\tilde{\Ls}^z u},\hat u\right\rangle_*\geq \delta \int |\nabla u|^2\omega_{\nu}dz+\delta \int \left(\frac{1}{(\nu+|z-a|)^2})+\frac{1}{(\nu+|z+a|)^2})+\langle z \rangle^2\right)u^2 \omega_\nu dz.
\end{equation}

\end{proposition}

\begin{proof}

Using successively the orthogonality \eqref{id:coercivity-orthogonalite-projection}, the decomposition $u=\hat u+a_0\phi_0$, the almost symmetry identity \eqref{id:bilinear-form-linearized-symmetry} and the almost eigenfunction identity \eqref{eq:eigenfunctions_phi10} we have
\begin{align}
\label{printemps1} \langle \widehat{\tilde{\Ls}^z u},\hat u\rangle_* & = \langle \tilde{\Ls}^z u, \hat u \rangle_* \\
\nonumber& = \langle \tilde{\Ls}^z u, u\rangle_*-a_0 \langle \tilde{\Ls}^z u, \phi_0 \rangle_*  \\
\nonumber& = \langle \tilde{\Ls}^z u, u \rangle_*-a_0( \langle  u,\tilde{\Ls}^z \phi_0 \rangle_*+\tilde{\mathcal Q}(u,\phi_0)-\tilde{\mathcal Q}(\phi_0,u) )  \\
\nonumber & = \langle \tilde{\Ls}^z u, u\rangle_*-a_0\left(  2\beta \Big(1 + \frac{\gamma_0}{\ln \nu} \Big) \langle  u, \phi_0 \rangle_*+ \langle u,R_0\rangle_*+\tilde{\mathcal Q}(u,\phi_0)-\tilde{\mathcal Q}(\phi_0,u) \right)  
\end{align}
By \eqref{bd:coercivity-projection-a0} and \eqref{id:projection-norm-phi0-matched-product} the second term in the right-hand side above is
\begin{equation}\label{printemps100}
| a_0 2\beta \Big(1 + \frac{\gamma_0}{\ln \nu} \Big) \langle  \phi_0, \phi_0 \rangle_*|\lesssim \| \nabla u\|_{L^2_{\omega_\nu}}.
\end{equation}
We have for the third term
\begin{equation}\label{printemps2}
\langle u,R_{0}\rangle_*\lesssim \| u\|_{L^2_{\omega_\nu}}\| R_{0}\|_{L^2_{\omega_\nu}}\lesssim \frac{1}{|\ln \nu|^2}\| u\|_{L^2_{\omega_\nu}}
\end{equation}
by \eqref{matchedscalarproduct1:continuity}, \eqref{est:pointwise_Ri} and a direct estimate. We have by \eqref{bd:symmetry-Lz-mathcal Q} and \eqref{est:pointwise_phii} for the fourth and fifth terms in \eqref{printemps1}
\begin{align}
|\tilde{\mathcal Q}(u,\phi_0)|+|\tilde{\mathcal Q}(\phi_0,u)|& \lesssim (\nu+|a-a_{\infty}|)\sqrt{|\ln \nu|}\| u\|_{H^1_{\omega_\nu}}\| \phi_0\|_{L^2_{\omega_\nu}} \nonumber \\
& \lesssim \sqrt{|\ln \nu|}(1+\nu^{-1}|a-a_{\infty}|)\| u\|_{H^1_{\omega_\nu}}.\label{printemps3}
\end{align}
Injecting \eqref{bd:coercivity-projection-a0}, \eqref{printemps100}, \eqref{printemps2} and \eqref{printemps3} in \eqref{printemps1} shows:
$$
\langle \widehat{\tilde{\Ls}^z u},\hat u\rangle_* =  \langle \tilde{\Ls}^z u, u\rangle_* +O\left(\frac{1+\nu^{-1}|a-a_\infty|}{\sqrt{|\ln \nu|}}\| u\|_{H^1_{\omega_\nu}}^2 \right)
$$
The desired inequality \eqref{bd:coercivity-Lztilde-matched-product-second} then follows from the above identity and Proposition \ref{Prop:coercive-Lz-global}.

\end{proof}

\section{Corrected multisoliton in parabolic variables} \label{sec:corrected-multisoliton}

In this section we consider the evolution term of the Keller-Segel equations in self-similar parabolic variables
\begin{equation} \label{id:stationary-equation-parabolic-variables}
\Delta w -\nabla. (w\nabla \Phi_{w})-\beta \Lambda w.
\end{equation}
By the scaling renormalization $w(z)=2\beta \tilde w(\sqrt{2\beta}z)$ it will suffice to consider the case $\beta=\frac 12$:
\begin{equation} \label{id:stationary-equation-parabolic-variables-normalized}
\Delta w -\nabla. (w\nabla \Phi_{ w})-\frac 12 \Lambda w.
\end{equation}
We apply the terminology "soliton" in analogy with wave equations, and call the sum of two stationary states a multisoliton. We look for a solution that is to leading order a multisoliton. We aim at relating the evolution term \eqref{id:stationary-equation-parabolic-variables} it generates to the eigenfunctions of Proposition \ref{prop:eigen} and to the modulation directions (that is to say, the parameter derivatives) of the multisoliton.

\subsection{Derivation of the corrected stationary equation}

\subsubsection{The equilibrium position $\overline{a}$ of the stationary states}

We first look for an approximate stationary solution, i.e. that makes \eqref{id:stationary-equation-parabolic-variables-normalized} vanish to leading order, of the form
$$
U_{1+2,\nu}=U_{1,\nu}+U_{2,\nu},
$$
where
$$
U_{1,\nu} (z)= \frac{1}{\nu^2}U(\frac{z-a}{\nu}) \quad \mbox{and}\quad U_{2,\nu} (z)=\frac{1}{\nu^2}U(\frac{z+a}{\nu})
$$
for $\nu>0$ small. We compute the evolution term \eqref{id:stationary-equation-parabolic-variables-normalized}
\begin{align*}
\Delta U_{1+2,\nu} -\nabla. (U_{1+2,\nu}\nabla \Phi_{U_{1+2,\nu}})-\frac 12 \Lambda U_{1+2,\nu} & =- \left(\nabla \Phi_{U_{2,\nu}}+\frac z2\right).\nabla U_{1,\nu}-\left(\nabla \Phi_{U_{1,\nu}}+\frac z2\right). \nabla U_{2,\nu} \\
&\qquad -U_{1+2,\nu}+2U_{1,\nu}U_{2,\nu}
\end{align*}
The most singular term being on the first line, the above is minimised when the vector field vanish at the positions of the stationary states, i.e. when
$$
\nabla \Phi_{U_{2,\nu}}(a)+\frac a2=0.
$$
Since $\nabla \Phi_{U_{2,\nu}}(z)=- \frac{4(z+a)}{\nu^2+|z+a|^2}$ the above leads in the limit $\nu\to 0$ to solving for $a_\infty$ in the equation $-\frac{2a_\infty}{|a_\infty|^2}+\frac{a_\infty}{2}=0$. The solution is $a_\infty=(2,0)$ given our symmetry assumptions. For equation \eqref{id:stationary-equation-parabolic-variables} the equilibrium position is thus
\begin{equation} \label{id:definition-bar-a}
a_\infty = \left( \sqrt{\frac{2}{\beta}},0\right).
\end{equation}

\subsubsection{Projection of the error on the eigenfunctions and the modulation directions}

We now aim at a refined decomposition of the evolution term \eqref{id:stationary-equation-parabolic-variables} generated by $U_{1+2,\nu}$ as an approximate solution to \eqref{id:stationary-equation-parabolic-variables}, when the stationnary states are near their equilibrium position \eqref{id:definition-bar-a}. To do so, we expand
$$
\nabla \Phi_{U_{2,\nu}}(z)=- \frac{8a}{\nu^2+4|a|^2}+(\nabla \Phi_{U_{2,\nu}}(z)-\nabla \Phi_{U_{2,\nu}}(a)),
$$
which leads to
\begin{align}
\nonumber & \Delta U_{1+2,\nu} -\nabla. (U_{1+2,\nu}\nabla \Phi_{U_{1+2,\nu}})-\beta \Lambda U_{1+2,\nu}\\
\label{multisoliton:id:error}&\qquad \qquad =- \frac{\beta}{\nu^2}\Lambda U \left(\frac{\cdot-a}{\nu}\right)- \frac{\beta}{\nu^2}\Lambda U \left(\frac{\cdot+a}{\nu}\right)+ \left( \frac{8a}{\nu^2+4|a|^2}-\beta a\right).\left( \nabla U_{1,\nu}- \nabla U_{2,\nu}\right) \\
\nonumber &\qquad \qquad \quad - \left(\nabla \Phi_{U_{2,\nu}}(z)-\nabla \Phi_{U_{2,\nu}}(a) \right). \nabla U_{1,\nu}- \left(\nabla \Phi_{U_{1,\nu}}(z)-\nabla \Phi_{U_{1,\nu}}(-a)\right) .\nabla U_{2,\nu}  +2U_{1,\nu}U_{2,\nu}.
\end{align}
Our aim is to link the right-hand side, which is the time derivative term in the Keller-Segel equation  in parabolic variables generated by the multisoliton, to the eigenfunctions constructed in Proposition \ref{prop:eigen}, and to the modulation directions (directional derivatives) of the multisoliton. We notice from the expansion of the eigenfunctions that $- \frac{\beta}{\nu^2}\Lambda U \left(\frac{\cdot-a}{\nu}\right)- \frac{\beta}{\nu^2}\Lambda U \left(\frac{\cdot+a}{\nu}\right)$ is the leading order term in the expansion of $16\beta \nu^2 \phi_0$ near $\pm  a$. The second term $ \nabla U_{1,\nu}- \nabla U_{2,\nu}$ is the directional derivative of the multisoliton with respect to the position $a$ of the stationary states. The effect of second line will be of lower order.

Therefore, the main problem in this section will be to construct a small enough function $\Xi$ such that the corrected multisoliton
\begin{equation}\label{def:ImprovedSS}
\mathfrak{U} = U_{1+2, \nu} + \Xi ,
\end{equation}
satisfies the corrected stationary equation
\begin{align} \label{id:corrected-stationary-equation}
E_0 = & \Delta \mathfrak{U} -\nabla. (\mathfrak{U}\nabla \Phi_{ U_{1+2, \nu} +\chi^* \Xi})-\beta \Lambda \mathfrak{U}  \\
\nonumber & = 8\lambda_0 \nu^2\phi_0+\mu' .\left( \nabla U_{1,\nu}- \nabla U_{2,\nu}\right) +\Psi_0
\end{align}
with $\mu=\mu[a,\nu]\in \mathbb R$ and $\mu'=\mu'[a,\nu]\in \mathbb R^2$ some parameters to be found and $\Psi_0$ suitably small, where $\chi^*$ is defined in \eqref{def:chiK}. Since $\Xi$ grows at infinity that makes $\Phi_{\Xi}$ is not well-defined, the localization $\Phi_{\chi^*\Xi}$ is technically necessary. This localization will nonetheless be harmless since the contribution of the far away zone $|z| \gg 1$ will be of lower order.

\subsection{The corrected multisoliton} \label{subsec:corrected-multisoliton}

\begin{proposition}[Improved stationary solution in parabolic variables] \label{prop:E0} There exists a smooth function $\Xi[\nu,a]$ defined for all  $\nu $ small enough and $|a - \bar a| \lesssim \frac{\nu}{|\ln \nu|}$ that satisfies the following for
\begin{equation}\label{def:mumuprime}
 \mu' = \frac{8a}{\nu^2+4|a|^2}-\beta a+\tilde \mu',
\end{equation}
where $\tilde \mu'(a)=\Oc (\nu^2)$.
\begin{itemize}
\item[i)] \textup{(Pointwise estimates)} For $z \in \Rb^2$ and $j \in \mathbb{N}$,
\begin{align}
&|\nabla^k \Xi (z)| + |\nabla^j \nu \partial_\nu \Xi(z)|\lesssim \sum_{\pm}\frac{\nu^2}{(\nu+ |z \pm a|)^{2+k}} \langle z \rangle^{C_k},\label{est:pointwise_Xi}\\
& \quad |\nabla^k \partial_a \Xi(z)| \leq \sum_{\pm} \frac{\nu^2}{(\nu + |z \pm a|)^{3+k}} \langle z \rangle^{C_k},\label{est:pointwise_Xi_dk}
\end{align}
and 
\begin{equation}\label{est:Xi_localmass}
\sum_{\pm}\Big|\int_{|z \pm a| \leq \zeta_*} \nu \pa_\nu \Xi(z)dz \Big| \lesssim \nu^2. 
\end{equation}
\item[ii)] \textup{(Estimate of the Poisson field)}
\begin{equation}\label{est:PhiXi}
|\nabla^k \Phi_{\chi^* \Xi}|\lesssim \sum_\pm \frac{\nu^2}{(\nu+|z\pm a|)^{k}} \langle z\rangle^2 \langle \ln \langle z \rangle \rangle \qquad \mbox{for }k\geq 1,
\end{equation}
and in particular
\begin{equation}
|\nabla \Phi_{\chi^*\Xi}(a)|\lesssim \nu. \label{est:nablaPhiXia}
\end{equation}
\item[iii)] \textup{(Generated error)} The function $\mathfrak {U}$ solves the corrected stationary equation \eqref{id:corrected-stationary-equation} with
\begin{align} \label{multisoliton:bd:estimate-Psi0-pointwise}
&|\Psi_0(z)| \lesssim \sum_{\pm} \frac{\langle z \rangle^C}{|\log \nu|^2(\nu + |z \pm a|)^2}, 
\end{align}
so that we have in particular
\begin{align}\label{est:Psi_L2omega}
&\|\Psi_0\|_{L^2_{\omega_\nu }} \lesssim \frac{\nu^{2}}{|\log \nu|^2},\\
&\label{est:Psi0mass} \int_{|z \pm a| \lesssim 1}\Psi_0 \lesssim  \frac{\nu^{2}}{|\log \nu|}.
\end{align}

\item[iv)] \textup{(Parameter derivative)} The function $\mathfrak {U}$ satisfies the parameter derivative identity
\begin{align} \label{multisoliton:id:nupartialnu-mathfrakU-identity}
\nu\partial_\nu \mathfrak U & =16 \phi_0-\nu^2(L_0,0).(\nabla U_{1,\nu}-\nabla U_{2,\nu})+\tilde{\mathfrak U},\\
|\nabla^k \tilde{\mathfrak U}| &  \lesssim \sum_\pm \frac{\nu^2}{(\nu+|z\mp a|)^{2+k}}\left(\frac{\nu^{\delta_0}}{(\nu+|z\mp a|)^{\delta_0}}+\frac{1}{|\ln \nu|^2}\right). \label{est:Ukpointwise}
\end{align}

\end{itemize}
\end{proposition}

\noindent We postpone the proof of Proposition \ref{prop:E0} to the end of this section.

\subsection{The first ansatz for the leading order correction} \label{subsec:first-ansatz-multisoliton}

Injecting the decomposition \eqref{def:ImprovedSS} and \eqref{multisoliton:id:error} in \eqref{id:corrected-stationary-equation} we have
\begin{align} 
\nonumber \Psi_0 & = \Delta \mathfrak{U} -\nabla. (\mathfrak{U}\nabla \Phi_{ U_{1+2, \nu} +\chi^* \Xi})-\beta \Lambda \mathfrak{U}-8\lambda_0 \nu^2\phi_0-\mu'.\left( \nabla U_{1,\nu}- \nabla U_{2,\nu}\right)   \\
\nonumber &= - \frac{\beta}{\nu^2}\Lambda U \left(\frac{\cdot-a}{\nu}\right)- \frac{\beta}{\nu^2}\Lambda U \left(\frac{\cdot+a}{\nu}\right)+ \left( \frac{8a}{\nu^2+4|a|^2}-\beta a\right).\left( \nabla U_{1,\nu}- \nabla U_{2,\nu}\right) \\
\nonumber &\quad - \left(\nabla \Phi_{U_{2,\nu}}(z)-\nabla \Phi_{U_{2,\nu}}(a) \right). \nabla U_{1,\nu}- \left(\nabla \Phi_{U_{1,\nu}}(z)-\nabla \Phi_{U_{1,\nu}}(-a)\right) .\nabla U_{2,\nu}  +2U_{1,\nu}U_{2,\nu}\\
\nonumber & \quad +\tilde \Ls^z(\Xi) -\nabla. (\Xi \nabla \Phi_{ \chi^* \Xi})-\mu \nu^2\phi_0- \mu'.\left( \nabla U_{1,\nu}- \nabla U_{2,\nu}\right) \\
\nonumber &= - \frac{\beta}{\nu^2}\Lambda U \left(\frac{\cdot-a}{\nu}\right)- \frac{\beta}{\nu^2}\Lambda U \left(\frac{\cdot+a}{\nu}\right)-8\lambda_0 \nu^2\phi_0+\left( \frac{8a}{\nu^2+4|a|^2}-\beta a-\mu'\right).\left( \nabla U_{1,\nu}- \nabla U_{2,\nu}\right) \\
\nonumber &\quad - \left(\nabla \Phi_{U_{2,\nu}}(z)-\nabla \Phi_{U_{2,\nu}}(a) \right). \nabla U_{1,\nu}- \left(\nabla \Phi_{U_{1,\nu}}(z)-\nabla \Phi_{U_{1,\nu}}(-a)\right) .\nabla U_{2,\nu}  +2U_{1,\nu}U_{2,\nu}\\
\label{multisoliton:id:expression-Psi0-first} & \quad +\tilde \Ls^z(\Xi) -\nabla. (\Xi \nabla \Phi_{ \chi^* \Xi}) 
\end{align}
Our first ansatz is to consider the decomposition
\begin{align}
\label{multisoliton:id:decomposition-Xi-1}  & \Xi= 8 \left[ \nu^2 \phi_0+(\frac{1}{16} \Lambda U-L_0\nu \partial_{x_1}U)_{1,\nu}+(\frac{1}{16}\Lambda U+L_0\nu \partial_{x_1}U)_{2,\nu} \right]+\tilde \Xi,\\
\label{multisoliton:id:decomposition-PhiXi-1} & \Phi_{\chi^* \Xi}= 8 \left[ \nu^2 \Phi_{\chi^* \phi_0}+\Phi_{(\frac{1}{16} \Lambda U-L_0\nu \partial_{x_1}U)_{1,\nu}}+\Phi_{(\frac{1}{16}\Lambda U+L_0\nu \partial_{x_1}U)_{2,\nu}} \right]+\chi^*\tilde{\mathfrak V}+\bar{\mathfrak V}
\end{align}
where for $f:\mathbb R^2\rightarrow \mathbb R$ we introduce the notation
$$
f_{1,\nu}=\frac{1}{ \nu^2}f \left(\frac{\cdot-a}{\nu}\right) \quad \mbox{and}\quad  f_{2,\nu}=\frac{1}{ \nu^2}f \left(\frac{\cdot+a}{\nu}\right)
$$
and where $-\Delta \tilde{\mathfrak V}=\tilde{\Xi}$. Recalling that $\Phi_f$ is the unique solution modulo constant to $-\Delta \Phi_f=f$ that grows at most logarithmically, we choose in view of the identity \eqref{multisoliton:id:decomposition-PhiXi-1}
\begin{align} \label{multisoliton:id:elliptic-mathfrakV}
& \bar{\mathfrak V}= \bar{\mathfrak V}_1+ \bar{\mathfrak V}_2,\\
\label{multisoliton:id:elliptic-mathfrakV-2}& \bar{\mathfrak V_1}= \Phi_{ \left( (\frac{1}{2}\Lambda U-8L_0\nu \partial_{x_1}U)_{1,\nu}+(\frac{1}{2}\Lambda U+8L_0\nu \partial_{x_1}U)_{2,\nu} \right)(\chi^*-1)}, \qquad \bar{\mathfrak V}_2=2\Phi_{\nabla \tilde{\mathfrak V}.\nabla \chi^*}+\Phi_{\tilde{\mathfrak V} \Delta \chi^*}.
\end{align}
We then have
\begin{align*}
\tilde{\Ls}^z (\Xi) &= 8 \nu^2 \tilde{\Ls}^z \phi_0+\frac{1}{2}\Ls^z ((\Lambda U)_{1,\nu})-\nu 8L_0  \Ls^z ((\partial_{x_1} U)_{1,\nu}) +\frac{1}{2}\Ls^z ((\Lambda U)_{2,\nu})+8\nu L_0 \Ls^z ((\partial_{x_1} U)_{2,\nu})\\
&\quad +\Ls^z(\tilde{\Xi},\tilde{\mathfrak V})- \nabla\bar{\mathfrak V}(a).(\nabla U_{1,\nu}-\nabla U_{2,\nu})\\
&\quad -\left(\nabla\bar{\mathfrak V}-\nabla\bar{\mathfrak V}(a)+\nabla((\chi^*-1)\tilde {\mathcal V})\right).\nabla U_{1,\nu}-\left(\nabla \bar{\mathfrak V}-\nabla\bar{\mathfrak V}(-a)+\nabla((\chi^*-1)\tilde {\mathcal V})\right).\nabla U_{2,\nu}.
\end{align*}
We compute for $j=1,2$, using the cancellations $\Delta  (\Lambda U)_{j,\nu} -\nabla .( (\Lambda U)_{j,\nu}\nabla \Phi_{U_{j,\nu}})-\nabla .(U_{j,\nu}\nabla \Phi_{ (\Lambda U)_{j,\nu} })=0$ and $\Delta  (\partial_{x_1} U)_{j,\nu} -\nabla .( (\partial_{x_1} U)_{j,\nu}\nabla \Phi_{U_{j,\nu}})-\nabla .(U_{j,\nu}\nabla \Phi_{ (\partial_{x_1}U)_{j,\nu} })=0$, that
\begin{align*}
& \Ls^z\left( (\Lambda U)_{1,\nu} \right)= -\nabla U_{2,\nu} \nabla \Phi_{(\Lambda U)_{1,\nu}}+2U_{2,\nu}(\Lambda U)_{1,\nu}-\nabla (\Lambda U)_{1,\nu} \nabla \Phi_{U_{2,\nu}}-\beta (\Lambda^2 U)_{1,\nu}-\beta a.\nabla (\Lambda U)_{1,\nu},\\
& \Ls^z\left( (\partial_{x_1} U)_{1,\nu} \right)= -\nabla U_{2,\nu} \nabla \Phi_{(\partial_{x_1} U)_{1,\nu}}+2U_{2,\nu}(\partial_{x_1} U)_{1,\nu}-\nabla (\partial_{x_1} U)_{1,\nu} \nabla \Phi_{U_{2,\nu}}-\beta (\Lambda \partial_{x_1} U)_{1,\nu}-\beta a.\nabla (\partial_{x_1} U)_{1,\nu},\\
& \Ls^z\left( (\Lambda U)_{2,\nu} \right)= -\nabla U_{1,\nu} \nabla \Phi_{(\Lambda U)_{2,\nu}}+2U_{1,\nu}(\Lambda U)_{2,\nu}-\nabla (\Lambda U)_{2,\nu} \nabla \Phi_{U_{1,\nu}}-\beta (\Lambda^2 U)_{2,\nu}+\beta a.\nabla (\Lambda U)_{2,\nu},\\
& \Ls^z\left( (\partial_{x_1}U)_{2,\nu} \right)= -\nabla U_{1,\nu} \nabla \Phi_{(\partial_{x_1} U)_{2,\nu}}+2U_{1,\nu}(\partial_{x_1} U)_{2,\nu}-\nabla (\partial_{x_1} U)_{2,\nu} \nabla \Phi_{U_{1,\nu}}-\beta (\Lambda \partial_{x_1} U)_{2,\nu}+\beta a.\nabla (\partial_{x_1} U)_{2,\nu},
\end{align*}
Injecting \eqref{eigen:id:Ri} and the above identities in \eqref{multisoliton:id:expression-Psi0-first}, the error becomes
\begin{align}
\nonumber  \Psi_0 & = - \beta (\Lambda U)_{1,\nu} - \beta (\Lambda U)_{2,\nu} - 8\lambda_0 \nu^2\phi_0+\left( \frac{8a}{\nu^2+4|a|^2}-\beta a-\nabla\bar{\mathfrak V}(a)-\mu'\right).\left( \nabla U_{1,\nu}- \nabla U_{2,\nu}\right) \\
\nonumber &\quad - \left(\nabla \Phi_{U_{2,\nu}}(z)-\nabla \Phi_{U_{2,\nu}}(a) \right). \nabla U_{1,\nu}- \left(\nabla \Phi_{U_{1,\nu}}(z)-\nabla \Phi_{U_{1,\nu}}(-a)\right) .\nabla U_{2,\nu}  +2U_{1,\nu}U_{2,\nu}\\
\nonumber  & \quad +\frac{1}{2}\left(-\nabla U_{2,\nu} \nabla \Phi_{(\Lambda U)_{1,\nu}}+2U_{2,\nu}(\Lambda U)_{1,\nu}-\nabla (\Lambda U)_{1,\nu} \nabla \Phi_{U_{2,\nu}}-\beta (\Lambda^2 U)_{1,\nu}-\beta a.\nabla (\Lambda U)_{1,\nu}\right),\\
\nonumber  &  \quad -8 L_0\nu \left(-\nabla U_{2,\nu} \nabla \Phi_{(\partial_{x_1} U)_{1,\nu}}+2U_{2,\nu}(\partial_{x_1} U)_{1,\nu}-\nabla (\partial_{x_1} U)_{1,\nu} \nabla \Phi_{U_{2,\nu}}-\beta (\Lambda \partial_{x_1} U)_{1,\nu}-\beta a.\nabla (\partial_{x_1} U)_{1,\nu}\right)\\
\nonumber  & \quad +\frac{1}{2}\left(-\nabla U_{1,\nu} \nabla \Phi_{(\Lambda U)_{2,\nu}}+2U_{1,\nu}(\Lambda U)_{2,\nu}-\nabla (\Lambda U)_{2,\nu} \nabla \Phi_{U_{1,\nu}}-\beta (\Lambda^2 U)_{2,\nu}+\beta a.\nabla (\Lambda U)_{2,\nu}\right)\\
\nonumber  &\quad +8 \nu L_0 \left(-\nabla U_{1,\nu} \nabla \Phi_{(\partial_{x_1} U)_{2,\nu}}+2U_{1,\nu}(\partial_{x_1} U)_{2,\nu}-\nabla (\partial_{x_1} U)_{2,\nu} \nabla \Phi_{U_{1,\nu}}-\beta (\Lambda \partial_{x_1} U)_{2,\nu}+\beta a.\nabla (\partial_{x_1} U)_{2,\nu}\right)\\
\nonumber  & \quad -\left(\nabla \bar{\mathfrak V}-\nabla \bar{\mathfrak V}(a)+\nabla((\chi^*-1)\tilde {\mathcal V})\right).\nabla U_{1,\nu}-\left(\nabla \bar{\mathfrak V}-\nabla\bar{\mathfrak V}(-a)+\nabla((\chi^*-1)\tilde {\mathcal V})\right).\nabla U_{2,\nu} \\
\nonumber  & \quad +8\lambda_0 \nu^2 \phi_0+8\nu^2 R_0\\
\nonumber  & \quad +\Ls^z(\tilde \Xi,\tilde{\mathfrak V}) -\nabla. (\Xi \nabla \Phi_{ \chi^* \Xi})\\
\label{multisoliton:id:decomposition-Psi0-1}  & = - \beta [ \Lambda U+\frac 12 \Lambda^2 U-8 \nu L_0\Lambda \partial_{x_1}U]_{1,\nu} - \beta [ \Lambda U+\frac 12 \Lambda^2 U+8 L_0\nu \Lambda \partial_{x_1}U]_{2,\nu}\\
\nonumber  &\quad  +\left( \frac{8a}{\nu^2+4|a|^2}-\beta a-\nabla \Phi_{(\frac{1}{2}\Lambda U+8L_0\nu \partial_{x_1}U)_{2,\nu}}(a)-\nabla\bar{\mathfrak V}(a)-\mu'\right).\left( \nabla U_{1,\nu}- \nabla U_{2,\nu}\right) \\
\nonumber &\quad - \left(\nabla \Phi_{U_{2,\nu}}(z)-\nabla \Phi_{U_{2,\nu}}(a) \right). \nabla (U+\frac{1}{2}\Lambda U-8L_0\nu \partial_{x_1}U)_{1,\nu}- \left(\nabla \Phi_{U_{1,\nu}}(z)-\nabla \Phi_{U_{1,\nu}}(-a)\right) .\nabla (U +\frac{1}{2}\Lambda U+8L_0\nu \partial_{x_1}U)_{2,\nu} \\
\nonumber  & \quad  +\left(-\nabla \Phi_{U_{2,\nu}}(a)-\beta a\right).\nabla (\frac{1}{2}\Lambda U-8L_0\nu \partial_{x_1})_{1,\nu}+\left(\nabla \Phi_{U_{2,\nu}}(a)+\beta a\right).\nabla (\frac{1}{2}\Lambda U+8L_0\nu \partial_{x_1}U)_{2,\nu} +2U_{1,\nu}U_{2,\nu} \\
\nonumber  &\quad -\nabla U_{2,\nu} .(\nabla \Phi_{(\frac{1}{2}\Lambda U-\nu 8 L_0\partial_{x_1}U)_{1,\nu}}-\nabla \Phi_{(\frac{1}{2}\Lambda U+8 \nu L_0 \partial_{x_1}U)_{1,\nu}}(-a)) +2U_{2,\nu}(\frac{1}{2}\Lambda U-8\nu L_0\partial_{x_1}U)_{1,\nu} \\
\nonumber  & \quad -\nabla U_{1,\nu}. ( \nabla \Phi_{(\frac{1}{2}\Lambda U+8\nu L_0\partial_{x_1}U)_{2,\nu}}-\nabla \Phi_{(\frac{1}{2}\Lambda U+8 \nu L_0\partial_{x_1}U)_{2,\nu}}(a))+2U_{1,\nu}(\frac{1}{2}\Lambda U+8 \nu L_0 \partial_{x_1}U)_{2,\nu} \\
\nonumber  & \quad -\left(\nabla\bar{\mathfrak V}-\nabla\bar{\mathfrak V}(a)+\nabla((\chi^*-1)\tilde {\mathcal V})\right).\nabla U_{1,\nu}-\left(\nabla\bar{\mathfrak V}-\nabla \bar{\mathfrak V}(-a)+\nabla((\chi^*-1)\tilde {\mathcal V})\right).\nabla U_{2,\nu} \\
\nonumber  & \quad +8 \nu^2 R_0\\
\nonumber  & \quad +\Ls^z(\tilde \Xi,\tilde{\mathfrak V}) -\nabla. (\Xi \nabla \Phi_{ \chi^* \Xi}) 
\end{align}
All terms in the right-hand side above (except at this stage those involving $\tilde \Xi$, $\tilde{\mathfrak V}$ and $\bar{\mathfrak V}$ that have not been yet defined) will be shown to be suitably small away from the stationary states.

\subsection{The second ansatz for higher order corrections} \label{subsec:second-corrections-multisolitons}

\subsubsection{Decomposition of the error terms generated by the first ansatz}

 There remains to introduce further corrections that are localized near the stationary states. To do so we perform Taylor expansions to get the leading order terms in the above expression \eqref{multisoliton:id:decomposition-Psi0-1}. Thanks to Lemma \ref{lem:taylorpoissonfield} we have
\begin{align}
\nonumber & \left(\nabla \Phi_{U_{2,\nu}}(z)-\nabla \Phi_{U_{2,\nu}}(a) \right). \nabla (U+\frac{1}{2}\Lambda U-8L_0\nu \partial_{x_1}U)_{1,\nu} \\
\nonumber & = \sum_{i=1}^3 \alpha^{\frac{i+1}{2}}\nu^{i-3}(G_{i,+}.\nabla (U+\frac 12 \Lambda U-8L_0\nu \partial_{x_1}U))(\frac{\cdot-a}{\nu})-\frac{1}{\nu}(\gamma_1,0).\nabla (U+\frac 12 \Lambda U-8L_0\nu \partial_{x_1}U))(\frac{\cdot-a}{\nu})\\
\label{multisoliton:id:decomposition-evolution-term-1} &\quad -\frac{1}{\nu^3}[\nabla \Phi_{2,\nu}(a)+\alpha a-\nu^2 (\gamma_1,0)].\nabla (U+\frac{1}{2}\Lambda U-8L_0\nu \partial_{x_1}U)(\frac{\cdot-a}{\nu}) +\frac{1}{\nu^4}(\tilde G_+.\nabla (U+\frac{1}{2}\Lambda U-8L_0\nu \partial_{x_1}U)(\frac{\cdot-a}{\nu}).
\end{align}
where $\gamma_1=\frac{\alpha^{\frac 32}}{4\sqrt{2}}$. Next,
\begin{equation} \label{multisoliton:id:decomposition-evolution-term-2} 
\frac{1}{2}(-\nabla \Phi_{U_{2,\nu}}(a)-\beta a).\nabla(\frac 12 \Lambda U-8L_0\nu \partial_{x_1}U)_{1,\nu}=\frac{1}{\nu^2}(\gamma_2,0). \nabla \Lambda U(\frac{\cdot-a}{\nu})-\frac{1}{\nu}(16L_0\gamma_2,0). \nabla\partial_{x_1} U(\frac{\cdot-a}{\nu})
\end{equation}
where $(\gamma_2,0)=-\frac{\beta a+\nabla \Phi_{U_{2,\nu}}(a)}{2\nu}\in \mathbb R^2$ satisfies $|\gamma_2|\lesssim 1$ since $|a-a_\infty|\lesssim \nu$. Then,
\begin{align}
\nonumber & 2U_{1,\nu}U_{2,\nu}+U_{2,\nu}(\Lambda U)_{1,\nu}+U_{1,\nu}(\Lambda U)_{2,\nu}\\
\label{multisoliton:id:decomposition-evolution-term-3}  & = \gamma_3 \Lambda U(\frac{\cdot-a}{\nu})+  \gamma_3 \Lambda U(\frac{\cdot+a}{\nu})\\
\nonumber &\quad  +2U_{1,\nu}U_{2,\nu}+U_{2,\nu}(\Lambda U_{1,\nu}+U_{1,\nu}(\Lambda U)_{2,\nu} -U_{2,\nu}(a)(\Lambda U)_{1,\nu} -U_{1,\nu}(-a)(\Lambda U)_{2,\nu}
\end{align}
where $\gamma_3=\frac{U_{2,\nu}(a)}{\nu^2}$ satisfies $|\gamma_3|\lesssim 1$ and the last line satisfies
\begin{align}
\label{multisoliton:id:decomposition-produits}& 2U_{1,\nu}U_{2,\nu}+U_{2,\nu}(\Lambda U_{1,\nu}+U_{1,\nu}(\Lambda U)_{2,\nu} -U_{2,\nu}(a)(\Lambda U)_{1,\nu} -U_{1,\nu}(-a)(\Lambda U)_{2,\nu}\\
\nonumber & = (2U+\Lambda U)_{2,\nu}(a)U_{1,\nu}+ (2U+\Lambda U)_{1,\nu}(-a)U_{2,\nu}\\
\nonumber & \quad + 2(U_{1,\nu}-U_{1,\nu}(-a))(U_{2,\nu}-U_{2,\nu}(a))+ (U_{2,\nu}-U_{2,\nu}(a))( (\Lambda U)_{1,\nu}-(\Lambda U)_{1,\nu}(-a))\\
\nonumber & \quad + (U_{1,\nu}-U_{1,\nu}(-a))( (\Lambda U)_{2,\nu}-(\Lambda U)_{2,\nu}(a))\\
\nonumber & \quad -2U_{1,\nu}(-a)U_{2,\nu}(a)-U_{2,\nu}(a)(\Lambda U)_{1,\nu}(-a)-U_{1,\nu}(-a)(\Lambda U)_{2,\nu}(a).
\end{align}
All terms in \eqref{multisoliton:id:decomposition-produits} will be shown to be suitably small later on. We expand
\begin{align}
\label{multisoliton:id:decomposition-evolution-term-4}   &- \nabla U_{1,\nu}.(\nabla \Phi_{(\frac 12 \Lambda U+8L_0\nu \partial_{x_1}U)_{2,\nu}}(z)-\nabla \Phi_{(\frac 12\Lambda U+8L_0\nu \partial_{x_1}U)_{2,\nu}}(a))\\
\nonumber & =  \frac{\alpha^2}{16}(\nabla U.( -3y_1,y_2 ))(\frac{z-a}{\nu}) \\
\nonumber &\quad -\nabla U_{1,\nu}.\Bigg( \nabla \Phi_{(\frac 12 \Lambda U+8L_0\nu \partial_{x_1}U)_{2,\nu}}(z)-\frac 12\nabla \Phi_{(\frac 12 \Lambda U+8L_0\nu \partial_{x_1}U)_{2,\nu}}(a) +\frac{\alpha^2}{16}\nu^2 (-3(z_1-|a|),z_2)    \Bigg)
\end{align}
We compute for future use that
$$
\nabla \Phi_{\Lambda U}(y)=\nabla (y.\nabla \Phi_U)(y)=-\frac{8y}{(1+|y|^2)^2}.
$$
This implies
$$
\partial_{y_1}^2  \Phi_{\Lambda U}(y)= \frac{32y_1^2-8-8|y|^2}{(1+|y|^2)^3}, \quad \partial_{y_2}^2  \Phi_{\Lambda U}(y)= \frac{32y_2^2-8-8|y|^2}{(1+|y|^2)^3}, \quad \partial_{y_1}\partial_{y_2} \Phi_{\Lambda U}(y)=\frac{32y_1y_2}{(1+|y|^2)^3}.
$$
Hence
$$
\partial_{z_1}^2 \Phi_{(\Lambda U)_{2,\nu}}(a)=\nu^2\frac{96|a|^2-8\nu^2}{(4|a|^2+\nu^2)^3}, \quad \partial_{z_2}^2 \Phi_{(\Lambda U)_{2,\nu}}(a)=\nu^2\frac{-32 |a|^2-8\nu^2}{(4|a|^2+\nu^2)^3}, \quad \partial_{z_1}\partial_{z_2} \Phi_{(\Lambda U)_{2,\nu}}(a)=0.
$$
Therefore, recalling that $\alpha=\frac{2}{|a|^2}$,
$$
(\nabla^2  \Phi_{(\Lambda U)_{2,\nu}}(a)) (z-a)=\frac{\alpha^2}{8}\nu^2 (3(z_1-|a|),-z_2)+O(\nu^4|z-a|).
$$
Hence by a Taylor expansion, we have the following estimate
$$
\left| \nabla \Phi_{(\Lambda U)_{2,\nu}}(z)-\nabla \Phi_{(\Lambda U)_{2,\nu}}(a) +\frac{\alpha^2}{8}\nu^2 (-3(z_1-|a|),z_2)\right|\lesssim \frac{\nu^2(\nu+|z-a|)^2}{(\nu+|z+a|)^3}\langle z \rangle^2.
$$
Using $\Phi_{(\partial_{x_1}U)_{2,\nu}}(z)=\partial_{x_1}\Phi (\frac{z+a}{\nu})$ and $|\nabla^k \partial_{x_1}\Phi_U(y)|\lesssim \langle y \rangle^{-1-k}$ we have by Taylor expansion
$$
|\nabla \Phi_{(\partial_{x_1}U)_{2,\nu}}(z)-\nabla \Phi_{(\partial_{x_1}U)_{2,\nu}}(a)|\lesssim \frac{\nu^2(\nu+|z-a|)}{(\nu+|z+a|)^3}\langle z \rangle^2.
$$
Combining, we have the following estimate for the remainder term appearing in \eqref{multisoliton:id:decomposition-evolution-term-4},
\begin{align} 
\nonumber  & \left|\nabla \Phi_{(\frac 12\Lambda U+8L_0\nu \partial_{x_1}U)_{2,\nu}}(z)-\nabla \Phi_{(\frac12\Lambda U+8L_0\nu \partial_{x_1}U)_{2,\nu}}(a) +\frac{\alpha^2}{16}\nu^2 (-3(z_1-|a|),z_2) \right|\\
\label{multisoliton:bd:expansion-nabla-Phi-tech-estimate} & \lesssim \frac{\nu^2(\nu+|z-a|)^2}{(\nu+|z+a|)^3}\langle z \rangle^2
\end{align}
which will be used to show that it is suitably small everywhere later on. Next, we have $\nabla^2 \bar V_1(a)(z-a)=\partial_{z_1}^2 \bar V_1(a)(z_1-|a|,-z_2)$ because $\bar V_1$ is even in $z_2$ and because it satisfies $-\Delta \bar V_1(a)=0$. We introduce $\gamma_4=\nu^{-2}\nabla^2 \bar V_1(a)$ and expand accordingly
\begin{equation} \label{multisoliton:id:decomposition-evolution-term-999}
\nabla \bar V_1(z)=\nabla \bar V_1(a)+\gamma_4\nu^2 (z_1-|a|,-z_2)+\left(\nabla \bar V_1(z)-\nabla \bar V_1(a)-\nabla^2 \bar V_1(a)(z-a)\right).
\end{equation}
By the formula \eqref{multisoliton:id:elliptic-mathfrakV-2} and elliptic regularity, using $|\nabla^k U(y)|\lesssim \langle y \rangle^{-4-k}$ and that $\chi^*-1$ is supported for $|z|\geq 10$, we have $|\gamma_4|\lesssim 1$ and
\begin{equation} \label{multisoliton:bd:expansion-nabla-barV1-estimate} 
|\nabla \bar V_1(z)-\nabla \bar V_1(a)-\nabla^2 \bar V_1(a)(z-a)|\lesssim \nu^2 |z-a|^2.
\end{equation}
Injecting \eqref{multisoliton:id:decomposition-evolution-term-1}, \eqref{multisoliton:id:decomposition-evolution-term-2}, \eqref{multisoliton:id:decomposition-evolution-term-3}, \eqref{multisoliton:id:decomposition-evolution-term-4} and \eqref{multisoliton:id:decomposition-evolution-term-999} in \eqref{multisoliton:id:decomposition-Psi0-1} gives the decomposition
\begin{align} \label{multisoliton:id:decomposition-Psi0-2}
\Psi_0 & = \sum_{i=0}^2 \sum_{\pm}\nu^{i-2} F_{i+2,\pm}\left(\frac{z\mp a}{\nu}\right)+\Ls^z(\tilde \Xi,\tilde{\mathfrak V})+\bar \Psi_0^{(0)}\\
\nonumber & \qquad  +\left( \frac{8a}{\nu^2+4|a|^2}-\beta a-\nabla \Phi_{(\frac 12 \Lambda U+8L_0\nu \partial_{x_1}U)_{2,\nu}}(a)-\nabla \bar{\mathfrak V}(a)-\mu'\right)  .\left( \nabla U_{1,\nu}- \nabla U_{2,\nu}\right)-\nabla. (\Xi \nabla \Phi_{ \chi^* \Xi})
\end{align}
where
\begin{align}
\label{multisoliton:id:F2} & F_{2,\pm}=-\beta (\Lambda U+\frac 12\Lambda^2 U)-\alpha G_{1,+}.\nabla (U+\frac 12 \Lambda U)\pm(\gamma_2,0). \nabla \Lambda U ,\\
\label{multisoliton:id:F3} & F_{3,\pm}=\mp (\alpha^{\frac 32}G_{2,+}-(\gamma_1,0)).\nabla (U+\frac 12 \Lambda U)\pm 8 L_0 \beta \Lambda \partial_{x_1}U\pm8L_0\alpha G_{1,+}.\nabla \partial_{x_1}U-(16L_0 \gamma_2,0). \nabla \partial_{x_1} U,\\
\label{multisoliton:id:F4} & F_{4,\pm}=- \alpha^2 G_{3,+}.\nabla (U+\frac 12 \Lambda U)+\gamma_3 \Lambda U+  \frac{\alpha^2}{16} \nabla U.( -3y_1,y_2 )+8L_0\alpha^{\frac 32}G_{2,+}.\nabla \partial_{x_1}U\\
\nonumber &\qquad \qquad -8L_0 (\gamma_1,0).\nabla \partial_{x_1} U-\gamma_4(y_1,-y_2).\nabla U ,
\end{align}
and
\begin{align}
\label{multisoliton:id:def-barPsi0}\bar \Psi_0^{(0)} & = - \frac{1}{\nu^4}(\tilde G_+.\nabla (U+\frac{1}{2}\Lambda U-8L_0\nu \partial_{x_1}U)(\frac{\cdot-a}{\nu}) - \frac{1}{\nu^4}(\tilde G_-.\nabla (U+\frac{1}{2}\Lambda U+8L_0\nu \partial_{x_1}U)(\frac{\cdot+a}{\nu})\\
\nonumber & \quad -\frac{1}{\nu^3}[\nabla \Phi_{U_{2,\nu}}(a)+\alpha a-\nu^2(\gamma_1,0)].\left( \nabla (U+\frac{1}{2}\Lambda U-8L_0\nu \partial_{x_1}U)(\frac{\cdot-a}{\nu})-\nabla (U+\frac{1}{2}\Lambda U+8L_0\nu \partial_{x_1}U)(\frac{\cdot+a}{\nu})\right)\\
\nonumber & \quad + 8L_0\alpha^2 \nu (G_{3,+}.\nabla \partial_{x_1}U)(\frac{\cdot-a}{\nu})- 8L_0\alpha^2 \nu (G_{3,+}.\nabla \partial_{x_1}U)(\frac{\cdot+a}{\nu}) \\
\nonumber &\quad +2U_{1,\nu}U_{2,\nu}+U_{2,\nu}(\Lambda U_{1,\nu}+U_{1,\nu}(\Lambda U)_{2,\nu} -U_{2,\nu}(a)(\Lambda U)_{1,\nu} -U_{1,\nu}(-a)(\Lambda U)_{2,\nu}\\
\nonumber &\quad - \nabla U_{1,\nu}.\left( \nabla \Phi_{(\frac 12 \Lambda U+8L_0\nu \partial_{x_1}U)_{2,\nu}}(z)-\frac 12\nabla \Phi_{(\frac 12 \Lambda U+8L_0\nu \partial_{x_1}U)_{2,\nu}}(a) +\frac{\alpha^2}{16}\nu^2 (-3(z_1-|a|),z_2)\right)\\
\nonumber &\quad - \nabla U_{2,\nu}.\left( \nabla \Phi_{(\frac 12 \Lambda U-8L_0\nu \partial_{x_1}U)_{1,\nu}}(z)-\frac 12\nabla \Phi_{(\frac 12 \Lambda U-8L_0\nu \partial_{x_1}U)_{1,\nu}}(a) +\frac{\alpha^2}{16}\nu^2 (-3(z_1+|a|),z_2)\right)\\
\nonumber & \quad -16 L_0\nu U_{2,\nu}(\partial_{x_1}U)_{1,\nu}+16 L_0\nu U_{1,\nu}(\partial_{x_1}U)_{2,\nu}\\
\nonumber  & \quad -\left(\nabla \bar{\mathfrak V}(z)-\nabla \bar{\mathfrak V}(a)-\nabla^2 \bar{\mathfrak V}_1(a)(z-a)+\nabla((\chi^*-1)\tilde {\mathcal V})\right).\nabla U_{1,\nu}-\left(\nabla \bar{\mathfrak V}(z)-\nabla \bar{\mathfrak V}(-a)-\nabla^2 \bar{\mathfrak V}_1(-a)(z+a)+\nabla((\chi^*-1)\tilde {\mathcal V})\right).\nabla U_{2,\nu} \\
\nonumber &\quad +8 \nu^2 R_0  ,
\end{align}

\subsubsection{The second ansatz}

We introduce $\iota_+=1$ and $\iota_-=2$. We compute for two functions $\mathfrak T$ and $ \mathfrak V$ such that $-\Delta  \mathfrak V=\mathfrak T$,
\begin{align*}
& \Ls^z \left(\mathfrak T(\frac{\cdot \mp a}{\nu}),\nu^2 \mathfrak V(\frac{\cdot \mp a}{\nu})\right)\\
 &= \frac{1}{\nu^2}\Ls_0(\mathfrak T,\mathfrak V)(\frac{\cdot\mp a}{\nu})-\nabla .\left(\mathfrak T(\frac{\cdot\mp a}{\nu}) \nabla \Phi_{U_{\iota_\mp},\nu}\right)-\nu \nabla .\left( U_{\iota_\mp,\nu} \nabla \mathfrak V(\frac{\cdot\mp a}{\nu})\right)-\beta \Lambda (\mathfrak T(\frac{\cdot \mp a}{\nu})) \\
& = \frac{1}{\nu^2}\Ls_0(\mathfrak T,\mathfrak V)(\frac{\cdot\mp a}{\nu})-\frac{1}{\nu^2}\nabla \mathfrak T(\frac{\cdot\mp a}{\nu}).\nabla \Phi_{U}(\frac{\cdot \pm a}{\nu})-\frac{1}{\nu^2} \nabla U(\frac{\cdot \pm a}{\nu}).\left( \nabla \mathfrak V(\frac{\cdot\mp a}{\nu})-\nabla \mathfrak V(\frac{\mp 2a}{\nu})\right)\\
&\quad +\frac{2}{\nu^2}\mathfrak T(\frac{\cdot \mp a}{\nu})U(\frac{\cdot \pm a}{\nu}) -\beta \Lambda \mathfrak T(\frac{\cdot \mp a}{\nu})\mp\frac{\beta a}{\nu}.\nabla  \mathfrak T(\frac{\cdot \mp a}{\nu})- \nu \nabla \mathfrak V(\frac{\mp 2a}{\nu}).\nabla U_{\iota_\mp, \nu}
\end{align*}

We introduce for $0\leq i \leq 3$,
$$
\bar G_{>i,\pm}(z)=\frac{1}{\nu}\left(\sum_{j=i+1}^3 \alpha^{\frac{j+1}{2}}\nu^{j+1}G_{i,\pm}+\tilde G_{\pm}\right)(\frac{z\mp a}{\nu})
$$
so that from Lemma \ref{lem:taylorpoissonfield}  we have $\nabla \Phi_{U}(y\pm \frac{2a}{\nu})=\mp \alpha a \nu+\sum_{j=1}^i\alpha^{\frac{j+1}{2}}\nu^{j+1}G_{j,\pm}+\nu \tilde G_{>i,\pm}(\pm a +\nu y)$ for $0\leq i \leq 3$. This implies that
$$
\frac{1}{\nu}\nabla \Phi_{U}(\frac{z\pm a}{\nu})=\mp \alpha a+\sum_{j=1}^i \alpha^{\frac{j+1}{2}}\nu^{j}G_{j,\pm}(\frac{\cdot \mp a}{\nu})+\bar G_{>i,\pm}(z).
$$
By Lemma \ref{lem:taylorpoissonfield} we have for $|z\mp a|\ll 1$ that $|\bar G_{>i,\pm}(z)|\lesssim (\nu+|z\mp a|)^{i+1}$. Using $|\nabla \Phi_U(y)|\lesssim \langle y \rangle^{-1}$, \eqref{eigenfunctions:id:G1}, \eqref{eigenfunctions:id:G2} and \eqref{eigenfunctions:id:G3} we have for $|z\mp a|\gtrsim 1$ that $|\bar G_{>i,\pm}(z)|\lesssim (\nu+|z\pm a|)^{-1}+(\nu+|z\mp a|)^i$. Combining, we have for $i=0,1,2,3$,
\begin{equation} \label{multisoliton:bd:estimate-bar-G>i}
|\bar G_{>i,\pm}(y)|\lesssim (\nu+|z\mp a|)^{i+1}(\nu+|z\pm a|)^{-1}
\end{equation}
for $z\in \mathbb R^2$, which will be used later on. This leads to the decomposition
\begin{align*}
& \Ls^z \left(\mathfrak T(\frac{\cdot \mp a}{\nu}),\nu^2 \mathfrak V(\frac{\cdot \mp a}{\nu})\right)\\
& = \frac{1}{\nu^2}\Ls_0(\mathfrak T,\mathfrak V)(\frac{\cdot\mp a}{\nu})- \nu \nabla \mathfrak V(\frac{\mp 2a}{\nu}).\nabla U_{\iota_\pm, \nu}\pm \frac{\alpha-\beta}{\nu}a.\nabla \mathfrak T(\frac{\cdot\mp a}{\nu}) -\alpha (G_{1,\pm}.\nabla \mathfrak T)(\frac{\cdot\mp a}{\nu})-\beta \Lambda \mathfrak T (\frac{\cdot\mp a}{\nu})\\
&\quad -\frac{1}{\nu}\nabla \mathfrak T(\frac{\cdot\mp a}{\nu}).\bar G_{>1,\pm}(z) -\frac{1}{\nu^2} \nabla U(\frac{\cdot \pm a}{\nu}).\left( \nabla \mathfrak V(\frac{\cdot\mp a}{\nu})-\nabla \mathfrak V(\frac{\mp 2a}{\nu})\right)+\frac{2}{\nu^2}\mathfrak T(\frac{\cdot \mp a}{\nu})U(\frac{\cdot \pm a}{\nu})
\end{align*}
and for $j=1,2$ to the decomposition
\begin{align*}
& \Ls^z \left(\nu^j\mathfrak T(\frac{\cdot \mp a}{\nu}),\nu^{j+2} \mathfrak V(\frac{\cdot \mp a}{\nu})\right)\\
& =\nu^{j-2}\Ls_0(\mathfrak T,\mathfrak V)(\frac{\cdot\mp a}{\nu})- \nu^{j+1} \nabla \mathfrak V(\frac{\mp 2a}{\nu}).\nabla U_{\iota_\pm, \nu}\pm \nu^j \frac{\alpha-\beta}{\nu}a.\nabla \mathfrak T(\frac{\cdot\mp a}{\nu}) -\nu^j \beta \Lambda \mathfrak T (\frac{\cdot\mp a}{\nu})\\
&\quad -\nu^{j-1}\nabla \mathfrak T(\frac{\cdot\mp a}{\nu}).\bar G_{>0,\pm}(z) -\nu^{j-2}\nabla U(\frac{\cdot \pm a}{\nu}).\left( \nabla \mathfrak V(\frac{\cdot\mp a}{\nu})- \nabla \mathfrak V(\frac{\mp 2a}{\nu})\right)+2\nu^{j-2}\mathfrak T(\frac{\cdot \mp a}{\nu})U(\frac{\cdot \pm a}{\nu}).
\end{align*}
Our second ansatz is the refined decomposition
\begin{equation} \label{multisoliton:id:def-tildeXi}
(\tilde \Xi,\tilde{\mathfrak V})=\sum_\pm \sum_{i=0}^2 (\nu^i \mathfrak T_{i+2,\pm},\nu^{i+2}\mathfrak V_{i+2,\pm})(\frac{z\mp a}{\nu})
\end{equation}
Then, assuming that $\nabla \mathfrak V_{i+2,+}(\frac{-2 a}{\nu})=-\nabla \mathfrak V_{i+2,-}(\frac{2a}{\nu})$ for $i=0,1,2$ which will be ensured by symmetry properties, we get
\begin{align*}
& \Ls^z \left(\sum_{i=0}^2 (\nu^i \mathfrak T_{i+2,\pm},\nu^{i+2}\mathfrak V_{i+2,\pm})(\frac{\cdot \mp a}{\nu})\right)\\
 & = \sum_{i=0}^2 \sum_\pm \nu^{i-2}\Ls_0(\mathfrak T_{i+2,\pm},\mathfrak V_{i+2,\pm})(\frac{\cdot\mp a}{\nu})-\left(\sum_{i=0}^2\nu^{i+1}\nabla \mathfrak V_{i+2,-}(\frac{2a}{\nu})\right).(\nabla U_{1,\nu}-\nabla U_{2,\nu})\\
 & +\sum_{\pm} \left(\pm \frac{\alpha-\beta}{\nu}a.\nabla \mathfrak T_{2,\pm} - \beta \Lambda \mathfrak T_{2,\pm} -\alpha (G_{1,\pm}.\nabla \mathfrak T_{2,\pm}) \right)(\frac{\cdot \mp a}{\nu})\\
 & \quad  -\sum_\pm \frac{1}{\nu}\nabla \mathfrak T_{2,\pm}(\frac{\cdot\mp a}{\nu}).\bar G_{>1,\pm}(z)+\sum_\pm\sum_{i=1,2}\left(\pm \nu^i \frac{\alpha-\beta}{\nu}a.\nabla \mathfrak T_{i+2,\pm}(\frac{\cdot\mp a}{\nu}) -\nu^i \beta \Lambda \mathfrak T_{i+2,\pm} (\frac{\cdot\mp a}{\nu}\right) \\
 & -\sum_{\pm}\sum_{i=1,2}\nu^{i-1}\nabla \mathfrak T_{i+2,\pm}(\frac{\cdot\mp a}{\nu}).\bar G_{>0,\pm}(z)\\
 &+\sum_{\pm}\sum_{i=0}^2\left(-\nu^{i-2}\nabla U(\frac{\cdot \pm a}{\nu}). \left( \nabla \mathfrak V_{i+2,\pm}(\frac{\cdot\mp a}{\nu})-\nabla \mathfrak V_{i+2,\pm}(\frac{\mp 2 a}{\nu})\right)+2\nu^{i-2}\mathfrak T_{i+2}(\frac{\cdot \mp a}{\nu})U(\frac{\cdot \pm a}{\nu})\right).
\end{align*}
We furthermore choose
\begin{equation} \label{multisoliton:id:decomposition-mu'}
\mu'= \frac{8a}{\nu^2+4|a|^2}-\beta a- \nabla \Phi_{(\frac 12 \Lambda U+8L_0\nu \partial_{x_1}U)_{2,\nu}}(a)-\bar{\mathfrak V}(a)-\sum_{i=0}^2\nu^{i+1}\nabla \mathfrak V_{i+2,-}(\frac{2a}{\nu})+\nu^2( \mu'',0).
\end{equation}
The identity \eqref{multisoliton:id:decomposition-Psi0-2} becomes
\begin{align}
\label{multisoliton:id:decomposition-Psi0-3}  \Psi_0 &=  \sum_{\pm}\nu^{-2} \left(F_{2,\pm}+\Ls_0(\mathfrak T_{2,\pm},\mathfrak V_{2,\pm})\right)\left(\frac{z\mp a}{\nu}\right) \\
\nonumber & \quad + \sum_{\pm}\nu^{-1} \left(F_{3,\pm}\mp \mu''\partial_{x_1}U+\Ls_0(\mathfrak T_{3,\pm},\mathfrak V_{3,\pm})\right)\left(\frac{z\mp a}{\nu}\right) \\
\nonumber & \quad +\sum_{\pm}\left(F_{4,\pm}+\Ls_0(\mathfrak T_{4,\pm},\mathfrak V_{4,\pm})\pm \frac{\alpha-\beta}{\nu}a.\nabla \mathfrak T_{2,\pm} - \beta \Lambda \mathfrak T_{2,\pm} -\alpha (G_{1,\pm}.\nabla \mathfrak T_{2,\pm})\right)\left(\frac{z\mp a}{\nu}\right)\\
\nonumber &\quad  +\bar \Psi_0^{(0)}+\bar \Psi_0^{(1)}-\nabla. (\Xi \nabla \Phi_{ \chi^* \Xi})
\end{align}
where
\begin{align}
\nonumber \bar \Psi_0^{(1)} &= -\sum_\pm \frac{1}{\nu}\nabla \mathfrak T_{2,\pm}(\frac{\cdot\mp a}{\nu}).\bar G_{>1,\pm}(z)+\sum_\pm\sum_{i=1,2}\left(\pm \nu^i \frac{\alpha-\beta}{\nu}a.\nabla \mathfrak T_{i+2,\pm}(\frac{\cdot\mp a}{\nu}) -\nu^i \beta \Lambda \mathfrak T_{i+2,\pm} (\frac{\cdot\mp a}{\nu}\right) \\
\label{multisoliton:id:def-barPsi2}  &\quad -\sum_{\pm}\sum_{i=1,2}\nu^{i-1}\nabla \mathfrak T_{i+2,\pm}(\frac{\cdot\mp a}{\nu}).\bar G_{>0,\pm}(z)\\
\nonumber &\quad +\sum_{\pm}\sum_{i=0}^2\left(-\nu^{i-2}\nabla U(\frac{\cdot \pm a}{\nu}).\left( \nabla \mathfrak V_{i+2,\pm}(\frac{\cdot\mp a}{\nu})-\nabla \mathfrak V_{i+2,\pm}(\frac{\mp 2 a}{\nu})\right)+2\nu^{i-2}\mathfrak T_{i+2}(\frac{\cdot \mp a}{\nu})U(\frac{\cdot \pm a}{\nu})\right).
\end{align}

\subsubsection{Decomposition of the nonlinear terms}

We have
\begin{align}
& \Xi= 8\nu^2 \phi_0+\frac{1}{2}(\Lambda U)_{1,\nu}+\frac 12 (\Lambda U)_{2,\nu}-8L_0\nu (\partial_{x_1}U)_{1,\nu}+8L_0\nu (\partial_{x_1}U)_{2,\nu}+\sum_\pm \sum_{i=0}^2 \nu^i \mathfrak T_{i+2,\pm}(\frac{\cdot \mp a}{\nu}), \label{def:Xi_expansion}\\
& \Phi_{\chi^* \Xi}=8 \nu^2 \Phi_{\chi^* \phi_0}+\frac{1}{2}\Phi_{(\Lambda U)_{1,\nu}}+\frac{1}{2}\Phi_{(\Lambda U)_{2,\nu}}-8L_0\nu \Phi_{(\partial_{x_1}U)_{1,\nu}}+8L_0\nu \Phi_{(\partial_{x_1}U)_{2,\nu}} +\chi^*\tilde{\mathfrak V}+\bar{\mathfrak V}. \label{def:PhiXi_expansion}
\end{align}
We introduce
\begin{align*}
&( \tilde T_{0,\pm},\tilde V_{0,\pm})=(\Lambda U,\Phi_{\Lambda U}),\\
& (\tilde T_{1,\pm}, \tilde V_{1,\pm})=\mp 16 L_0 (\partial_{x_1}U, \partial_{x_1}\Phi_U), \\
&(\tilde T_{2,\pm},\tilde V_{2,\pm})= \alpha (T_2^{(0)},V_2^{(0)})+ \alpha \tilde{\lambda_0}(\hat T_2,\hat V_2)\pm \frac{\alpha-\beta}{\nu}(\mathcal S_{1,1},\mathcal W_{1,1}).
\end{align*}
We let
\begin{align*}
& (\tilde{\phi}_{0,>2},\tilde V_{0,>2})=  (\phi_0,\Phi_{\chi^* V_0}) -\frac{1}{16} \sum_{\pm}\sum_{j=0}^2 (\nu^{j-4} \tilde T_{j,\pm},\nu^{j-2} \tilde V_{j,\pm}) (\frac{\cdot\mp a}{\nu}),
\end{align*}
 so that
\begin{align*}
& \phi_0 =-\frac{1}{16}\sum_\pm\sum_{j=0}^2 \nu^{j-4} \tilde T_{j,\pm} (\frac{\cdot\mp a}{\nu})+\tilde{\phi}_{0,>2} ,\\
& \Phi_{\chi^* \phi_0} =-\frac{1}{16}\sum_\pm\sum_{j=0}^2 \nu^{j-4} \tilde V_{j,\pm} (\frac{\cdot\mp a}{\nu})+\tilde{V}_{0,>2},
\end{align*}
where by \textbf{mettre ref}
\begin{align}
\label{multisoliton:bd:tildephii>2} & |\tilde{\phi}_{0,>2}(z)|\lesssim \frac{1}{(\nu+|z\mp a|)(\nu+|z\pm a|)}\langle z \rangle^C,\\
\label{multisoliton:bd:tildeV>2}  &|\nabla \tilde{V}_{0,>2}(z)|\lesssim  |\ln \nu|  .
\end{align}
We define 
\begin{align}
\label{multisoliton:id:decomposition-tildemathfrakT2} &( \tilde{\mathfrak T}_{2,\pm}, \tilde{\mathfrak V}_{2,\pm} ) = ( \mathfrak T_{2,\pm}, \mathfrak V_{2,\pm})-\frac 12 ( \tilde T_{2,\pm},\tilde V_{2,\pm}), \\
\label{multisoliton:id:decomposition-tildeXi>2} & \tilde{\Xi}_{>2}  =8\nu^2 \tilde{\phi}_{0,>2}+ \sum_\pm \sum_{j=3}^4 \nu^{j-2}\mathfrak T_{j,\pm}(\frac{\cdot \mp a}{\nu}), \\
\label{multisoliton:id:decomposition-tildeV>2}  & \tilde{\mathfrak V}_{>2} =8\nu^2 \tilde V_{0,>2}+\chi^* \sum_\pm \sum_{j=3}^4  \nu^{j}\mathfrak V_{j,\pm}(\frac{\cdot \mp a}{\nu})+(\chi^*-1) \sum_\pm \nu^2 \tilde{\mathfrak V}_{2,\pm}(\frac{\cdot \mp a}{\nu}), 
\end{align}
so that
\begin{align}
\label{multisoliton:id:decomposition-Xi->2}& \Xi= \sum_\pm \tilde{\mathfrak T}_{2,\pm}(\frac{\cdot \mp a}{\nu})+ \tilde{\Xi}_{>2} ,\\ 
\nonumber & \Phi_{\chi^* \Xi}= \sum_\pm \nu^2 \tilde{\mathfrak V}_{2,\pm}(\frac{\cdot \mp a}{\nu})+  \tilde{\mathfrak V}_{>2}+\bar{\mathfrak V}.
\end{align}
We decompose the nonlinear term accordingly,
\begin{align} \label{multisoliton:id:decomposition-nonlinear-term} 
\nabla . (\Xi \nabla \Phi_{\chi^* \Xi})& =\sum_\pm \nabla . (\tilde{\mathfrak T}_{2,\pm}\nabla \tilde{\mathfrak V}_{2,\pm})(\frac{\cdot \mp a}{\nu})\\
\nonumber & \quad +\sum_\pm \nabla . (\tilde{\mathfrak T}_{2,\pm}(\frac{\cdot \mp a}{\nu})(\nu \nabla \tilde{\mathfrak V}_{2,\mp}(\frac{\cdot\pm a}{\nu})+\nabla \tilde{\mathfrak V}_{>2}+\nabla \bar{\mathfrak V})) \\
& \quad +\nabla . ((\tilde{\mathfrak T}_{2,\mp}(\frac{\cdot \pm a}{\nu})+\tilde{\Xi}_{>2})(\nu \nabla \tilde{\mathfrak V}_{2,\pm}(\frac{\cdot\mp a}{\nu})) \quad +\nabla.(\tilde{\Xi}_{>2}(\nabla \tilde{\mathfrak V}_{>2}+\nabla \bar{\mathfrak V})). \nonumber
\end{align}
Injecting \eqref{multisoliton:id:decomposition-nonlinear-term} in \eqref{multisoliton:id:decomposition-Psi0-3}, the error can be written as
\begin{align}
\label{multisoliton:id:decomposition-Psi0-4}  \Psi_0 &=  \sum_{\pm}\nu^{-2} \left(F_{2,\pm}+\Ls_0(\mathfrak T_{2,\pm},\mathfrak V_{2,\pm})\right)\left(\frac{z\mp a}{\nu}\right) \\
\nonumber & \quad + \sum_{\pm}\nu^{-1} \left(F_{3,\pm}\mp \mu''\partial_{x_1}U+\Ls_0(\mathfrak T_{3,\pm},\mathfrak V_{3,\pm})\right)\left(\frac{z\mp a}{\nu}\right) \\
\nonumber & \quad +\sum_{\pm}\left(F_{4,\pm}+\Ls_0(\mathfrak T_{4,\pm},\mathfrak V_{4,\pm})\pm \frac{\alpha-\beta}{\nu}a.\nabla \mathfrak T_{2,\pm} - \beta \Lambda \mathfrak T_{2,\pm} -\alpha (G_{1,\pm}.\nabla \mathfrak T_{2,\pm})- \nabla . (\tilde{\mathfrak T}_{2,\pm}\nabla \tilde{\mathfrak V}_{2,\pm})\right)\left(\frac{z\mp a}{\nu}\right)\\
\nonumber &\quad  +\bar \Psi_0^{(0)}+\bar \Psi_0^{(1)}+\bar \Psi_0^{(2)}
\end{align}
where
\begin{align}
\nonumber \bar \Psi_0^{(2)} &= -\sum_\pm \nabla . (\tilde{\mathfrak T}_{2,\pm}(\frac{\cdot \mp a}{\nu})(\nu \nabla \tilde{\mathfrak V}_{2,\mp}(\frac{\cdot\pm a}{\nu})+\nabla \tilde{\mathfrak V}_{>2}+\nabla \bar{\mathfrak V}))+\nabla . ((\tilde{\mathfrak T}_{2,\mp}(\frac{\cdot \pm a}{\nu})+\tilde{\Xi}_{>2})(\nu \nabla \tilde{\mathfrak V}_{2,\pm}(\frac{\cdot\mp a}{\nu}))\\
\label{multisoliton:id:def-barPsi2} & \quad -\nabla.(\tilde{\Xi}_{>2}(\nabla \tilde{\mathfrak V}_{>2}+\nabla \bar{\mathfrak V})),
\end{align}

\subsubsection{Elliptic equations for the higher order corrections}

In view of \eqref{multisoliton:id:decomposition-Psi0-4}, we determine $(\mathfrak T_{2,\pm},\mathfrak V_{2,\pm})$, $(\mathfrak T_{3,\pm},\mathfrak V_{3,\pm})$ and $(\mathfrak T_{4,\pm},\mathfrak V_{4,\pm})$ as solutions to the following systems
\begin{align}
\label{multisoliton:id:systeme-T2} & \left\{ \begin{array}{l l} \Ls_0(\mathfrak T_{2,\pm},\mathfrak V_{2,\pm})=-F_{2,\pm},\\
-\Delta \mathfrak V_{2,\pm}=\mathfrak T_{2,\pm},
\end{array}\right. \\
\label{multisoliton:id:systeme-T3} & \left\{ \begin{array}{l l} \Ls_0(\mathfrak T_{3,\pm},\mathfrak V_{3,\pm})=-F_{3,\pm}\pm \mu''\partial_{x_1}U,\\
-\Delta \mathfrak V_{3,\pm}=\mathfrak T_{3,\pm},
\end{array}\right. \\
\label{multisoliton:id:systeme-T4}& \left\{ \begin{array}{l l} \Ls_0(\mathfrak T_{4,\pm},\mathfrak V_{4,\pm})=-F_{4,\pm}\mp \frac{\alpha-\beta}{\nu}a.\nabla \mathfrak T_{2,\pm} + \beta \Lambda \mathfrak T_{2,\pm} +\alpha (G_{1,\pm}.\nabla \mathfrak T_{2,\pm})+\nabla . (\tilde{\mathfrak T}_{2,\pm}\nabla \tilde{\mathfrak V}_{2,\pm}),\\
-\Delta \mathfrak V_{4,\pm}=\mathfrak T_{4,\pm}.
\end{array}\right. 
\end{align}
This transforms \eqref{multisoliton:id:decomposition-Psi0-4} into
\begin{align}
\label{multisoliton:id:Psi0-decomposition-2nd-ansatz}\Psi_0 &=\bar \Psi_0^{(0)}+\bar \Psi_0^{(1)}+\bar \Psi_0^{(2)}
\end{align}

\subsubsection{Elliptic equations in polar coordinates}

We use polar coordinates to solve \eqref{multisoliton:id:systeme-T2}, \eqref{multisoliton:id:systeme-T3} and \eqref{multisoliton:id:systeme-T4}. We recall the formula
$$
\nabla (f(r)\cos(m\theta)).\nabla (g(r)\cos(k\theta))=\partial_r f\partial_r g\cos(m\theta)\cos(k\theta)+\frac{fg}{r}mk \sin(m\theta)\sin(k\theta)
$$
as well as the trigonometric identities
$$
2 \cos(m\theta)\cos(k\theta)= \cos((m+k)\theta)+\cos((m-k)\theta), \quad 2 \sin(m\theta)\sin(k\theta)=-\cos((m+k)\theta)+\cos((m-k)\theta).
$$
Using \eqref{eigenfunctions:id:G1} we compute
\begin{align*}
& G_{1,+}.\nabla (U+\frac{1}{2}\Lambda U)=\frac 12 r\partial_r (U+\frac{1}{2}\Lambda U)\cos(2\theta),\\
& (\gamma_2,0).\nabla \Lambda U=\gamma_2 \partial_r \Lambda U \cos (\theta),
\end{align*}
which, injected in \eqref{multisoliton:id:F2} gives
\begin{align}
\nonumber F_{2,\pm}&=-\beta (\Lambda U+\frac12\Lambda^2 U)\pm \gamma_2 \partial_r \Lambda U \cos \theta -\frac{\alpha}{2}r\partial_r (U+\frac 12 \Lambda U)\cos(2\theta)\\
\label{multisoliton:id:def-S2k}&=:-\mathfrak S_{2,0}(r)\mp\mathfrak S_{2,1}(r)\cos(\theta)-\mathfrak S_{2,2}(r)\cos(2\theta).
\end{align}
In order to solve \eqref{multisoliton:id:systeme-T2} we thus decompose
\begin{equation} \label{multisoliton:id:expansion-T2}
(\mathfrak T_{2,\pm},\mathfrak V_{2,\pm})= \sum_{k=0}^2 (\pm 1)^k (\mathfrak T_{2,k}(r)\cos(k\theta),\mathfrak V_{2,k}(r)\cos(k\theta))
\end{equation}
and we require that for $k=0,1,2$ the profiles $ (\mathfrak T_{2,k},\mathfrak V_{2,k})$ solve the systems
\begin{equation} \label{multisoliton:id:equation-T2-polar}
\left\{ \begin{array}{l l} \Ls_{0,k}(\mathfrak T_{2,k},\mathfrak V_{2,k})=\mathfrak S_{2,k},\\
-\Delta_k \mathfrak V_{2,k}=\mathfrak T_{2,k}.
\end{array} \right.
\end{equation}
Next, we have using \eqref{eigenfunctions:id:G2} the cancellation
$$
\alpha^{\frac 32}G_{2,+}-(\gamma_1,0) =  \alpha^{\frac 32}(-\frac{1}{12\sqrt{2}}\nabla(r^3\cos(3\theta))+\frac{1}{4\sqrt{2}}(1,0))-(\frac{\alpha^{\frac 32}}{4\sqrt{2}},0)= -\frac{\alpha^{\frac 32}}{12\sqrt{2}}\nabla(r^3\cos(3\theta))\\
$$
so that the first term in the expression \eqref{multisoliton:id:F3} is
$$
\mp (\alpha^{\frac 32}G_{2,+}-(\gamma_1,0)).\nabla (U+\frac 12 \Lambda U)=\pm \frac{\alpha^{\frac 32}}{4\sqrt{2}}r^2\partial_r (U+\frac 12 \Lambda U)\cos(3\theta).
$$
Moreover, we have
\begin{align*}
& \Lambda \partial_{x_1} U =\Lambda \partial_r U \cos \theta,\\
&G_{1,+}.\nabla \partial_{x_1}U = (\frac 14 r\partial_r^2 U+\frac 12 \partial_r U)\cos \theta + (\frac 14 r\partial_r^2 U-\frac 12 \partial_r U)\cos (3 \theta ),\\
&(1,0).\nabla \partial_{x_1}U= \frac{1}{2}\partial_r^2 U+\frac{1}{2r}\partial_r U+ (\frac{1}{2}\partial_r^2 U-\frac{1}{2r}\partial_r U)\cos (2\theta),
\end{align*}
and the identity \eqref{multisoliton:id:F3} then implies
\begin{align} \label{multisoliton:id:def-S33}
F_{3,\pm}\mp \mu''\partial_{x_1}U&=  -8L_0 \gamma_2 (\partial_r^2 U+\frac{1}{r}\partial_r U)\pm \left(8L_0 (\frac{\alpha}{4}+\beta )\Lambda \partial_r U-\mu'' \partial_r U\right)\cos (\theta )\\
\nonumber & -8L_0 \gamma_2 (\partial_r^2 U -\frac 1r \partial_r U)\cos (2\theta)\pm \left(\frac{\alpha^{\frac 32}}{4\sqrt{2}}r^2\partial_r (U+\frac12\Lambda U)+8L_0\alpha(\frac 14 r\partial_r^2 U-\frac 12 \partial_r U)\right)\cos (3\theta)\\
\nonumber &:= -\mathfrak S_{3,0}(r)\mp\mathfrak S_{3,1}(r)\cos(\theta)-\mathfrak S_{3,2}(r)\cos(2\theta)\mp \mathfrak S_{3,3}(r)\cos(3\theta).
\end{align}
To solve the system \eqref{multisoliton:id:systeme-T2} we thus take
\begin{equation} \label{multisoliton:id:expansion-T3}
(\mathfrak T_{3,\pm},\mathfrak V_{3,\pm})=\sum_{k=0}^3 (\pm 1)^k (\mathfrak T_{3,k}(r)\cos(k\theta),\mathfrak V_{3,k}(r)\cos(k\theta))
\end{equation}
where for $k=0,1,2,3$ the profiles $(\mathfrak T_{3,k},\mathfrak V_{3,k})$ solve the system
\begin{equation} \label{multisoliton:id:equation-T3-polar}
\left\{ \begin{array}{l l} \Ls_{0,k}(\mathfrak T_{3,k},\mathfrak V_{3,k})=\mathfrak S_{3,k},\\
-\Delta_k \mathfrak V_{3,k}=\mathfrak T_{3,k}.
\end{array} \right.
\end{equation}
Finally, we compute using \eqref{eigenfunctions:id:G3} the cancellation
\begin{align*}
& - \alpha^2 G_{3,+}.\nabla (U+\frac 12 \Lambda U)+\gamma_3 \Lambda U+  \frac{\alpha^2}{16} \nabla U.( -3y_1,y_2 ) \\
& =-\alpha^2 \left(\frac{1}{64}\nabla(r^4\cos(4\theta))+\frac{1}{16}(-3y_1,y_2)\right).\nabla(U+\frac 12 \Lambda U)+\gamma_3 \Lambda U+\frac{\alpha^2}{16}\nabla U.(-3y_1,y_2) \\
&=\frac{\alpha^2}{32}r\partial_r \Lambda U+\gamma_3 \Lambda U +\frac{\alpha^2}{16}r\partial_r \Lambda U \cos(2\theta)-\frac{\alpha^2}{16}r^3\partial_r (U+\frac{1}{2}\Lambda U)\cos(4\theta),
\end{align*}
and using $\gamma_1=\frac{\alpha^{3/2}}{4\sqrt{2}}$ and \eqref{eigenfunctions:id:G2} the other cancellation
\begin{align*}
&8L_0\alpha^{\frac 32}G_{2,+}.\nabla \partial_{x_1}U-8L_0(\gamma_1,0).\nabla \partial_{x_1}U\\
&=L_0\alpha^{\frac 32}(-\frac{1}{\sqrt{2}}r^2 \partial_r^2 U-\frac{1}{3\sqrt{2}}r\partial_r U)\cos(2\theta)+L_0\alpha^{\frac 32}(-\frac{1}{\sqrt{2}}r^2 \partial_r^2 U+\frac{1}{3\sqrt{2}}r\partial_r U)\cos(4\theta)
\end{align*}
and
$$
-\gamma_4(y_1,-y_2).\nabla U=-\gamma_4 r\partial_r U \cos(2\theta)
$$
so that the identity \eqref{multisoliton:id:F3} becomes
\begin{align*}
F_{4,\pm} &=\frac{\alpha^2}{32}r\partial_r \Lambda U+\gamma_3 \Lambda U +\left( (\frac{\alpha^2}{16}-\gamma_4)r\partial_r \Lambda U +L_0\alpha^{\frac 32}(-\frac{1}{\sqrt{2}}r^2 \partial_r^2 U-\frac{1}{3\sqrt{2}}r\partial_r U)\right) \cos(2\theta)\\
& \quad +\left( -\frac{\alpha^2}{16}r^3\partial_r (U+\frac{1}{2}\Lambda U)+L_0\alpha^{\frac 32}(-\frac{1}{\sqrt{2}}r^2 \partial_r^2 U+\frac{1}{3\sqrt{2}}r\partial_r U) \right) \cos(4\theta).
\end{align*}

We have by \eqref{multisoliton:id:expansion-T2} that
\begin{align*}
& a.\nabla \mathfrak T_{2,\pm}= \pm \frac{|a|}{2}[\partial_r \mathfrak T_{2,1}+\frac{ \mathfrak T_{2,1}}{r}]+|a|(\partial_r  \mathfrak T_{2,0}+\frac{\partial_r  \mathfrak T_{2,2}}{2}+\frac{ \mathfrak T_{2,2}}{2r})\cos(\theta)\\
&\qquad \qquad  \quad \pm \frac{a}{2}[\partial_r \mathfrak T_{2,1}-\frac{ \mathfrak T_{2,1}}{r}]\cos(2\theta)+\frac{|a|}{2}(\partial_r \mathfrak T_{2,2}-\frac{\mathfrak T_{2,2}}{r})\cos(3\theta), \\
& \beta \Lambda \mathfrak T_{2,\pm}= \beta \Lambda  \mathfrak T_{2,0}\pm \beta \Lambda  \mathfrak T_{2,1}\cos(\theta)+\beta \Lambda  \mathfrak T_{2,2}\cos(2\theta), \\
&G_{1,\pm}.\nabla  \mathfrak T_{2,\pm} =\frac{r}{4}\partial_r  \mathfrak T_{2,2}+\frac{1}{2}\mathfrak T_{2,2} \pm (\frac{r}{4}\partial_r  \mathfrak T_{2,1}+\frac{ \mathfrak T_{2,1}}{4})\cos(\theta)\\
& \qquad \qquad \quad \qquad +\frac{r}{2}\partial_r  \mathfrak T_{2,0}\cos(2\theta)\pm (\frac{r}{4}\partial_r  \mathfrak T_{2,1}-\frac{ \mathfrak T_{2,1}}{4})\cos(3\theta)+(\frac{r}{4}\partial_r  \mathfrak T_{2,2}-\frac 12  \mathfrak T_{2,2})\cos(4\theta).
\end{align*}
We introduce
\begin{align}
& \label{multisoliton:id:deftildeT20} (\tilde{T}_{2,0},\tilde{V}_{2,0})=\alpha(T_{2,0}^{(0)},V_{2,0}^{(0)})+\alpha \tilde{\lambda_0}(\hat T_{2},\hat V_{2}),\\
& \nonumber (\tilde{T}_{2,1},\tilde{V}_{2,1})=\frac{\alpha-\beta}{\nu}(\mathcal S_{1,1,1},\mathcal W_{1,1,1}),\\
& \nonumber (\tilde{T}_{2,2},\tilde{V}_{2,2})=\alpha(T_{2,2},V_{2,2})
\end{align}
and, for $k=0,1,2$,
$$
(\tilde{\mathfrak T}_{2,k},\tilde{\mathfrak V}_{2,k})= ( \mathfrak T_{2,k}, \mathfrak V_{2,k})-\frac 12 ( \tilde T_{2,k},\tilde V_{2,k})
$$
so that we have the decomposition
$$
(\tilde{\mathfrak T}_{2,\pm},\tilde{\mathfrak V}_{2,\pm})=\sum_{k=0}^2 (\pm 1)^k (\tilde{\mathfrak T}_{2,k}(r)\cos(k\theta),\tilde{\mathfrak V}_{2,k}(r)\cos (k\theta)).
$$
The nonlinear term is
$$
-\nabla . (\tilde{\mathfrak T}_{2,\pm}\nabla \tilde{\mathfrak V}_{2,\pm})=\tilde{\mathfrak T}_{2,\pm}^2-\nabla . \tilde{\mathfrak T}_{2,\pm}.\nabla \tilde{\mathfrak V}_{2,\pm}.
$$
We compute each term above:
\begin{align*}
\tilde{\mathfrak T}_{2,\pm}^2 & =\tilde{\mathfrak T}_{2,0}^2 +\frac 12 \tilde{\mathfrak T}_{2,1}^2 +\frac 12 \tilde{\mathfrak T}_{2,2}^2 \pm (2 \tilde{\mathfrak T}_{2,0} \tilde{\mathfrak T}_{2,1}+\tilde{\mathfrak T}_{2,1}\tilde{\mathfrak T}_{2,2})\cos(\theta) \\
& +  ( \frac 12 \tilde{\mathfrak T}_{2,1}^2+2 \tilde{\mathfrak T}_{2,0}\tilde{\mathfrak T}_{2,2})\cos(2 \theta) \pm  \tilde{\mathfrak T}_{2,1}\tilde{\mathfrak T}_{2,2})\cos(3 \theta)  +\frac 12 \tilde{\mathfrak T}_{2,2}^2 \cos(4 \theta) 
\end{align*}
and
\begin{align*}
& \nabla . \tilde{\mathfrak T}_{2,\pm}.\nabla \tilde{\mathfrak V}_{2,\pm} \\
&= \partial_r \tilde{\mathfrak T}_{2,0}\partial_r \tilde{\mathfrak V}_{2,0}+\frac 12 \partial_r \tilde{\mathfrak T}_{2,1}\partial_r \tilde{\mathfrak V}_{2,1}+\frac{1}{2r^2}  \tilde{\mathfrak T}_{2,1} \tilde{\mathfrak V}_{2,1}+\frac 12 \partial_r \tilde{\mathfrak T}_{2,2}\partial_r \tilde{\mathfrak V}_{2,2}+\frac{2}{r^2}  \tilde{\mathfrak T}_{2,2} \tilde{\mathfrak V}_{2,2}\\
& \pm \left(\partial_r \tilde{\mathfrak T}_{2,0}\partial_r \tilde{\mathfrak V}_{2,1}+\partial_r \tilde{\mathfrak T}_{2,1}\partial_r \tilde{\mathfrak V}_{2,0}+\frac 12\partial_r \tilde{\mathfrak T}_{2,2}\partial_r \tilde{\mathfrak V}_{2,1}+\frac{1}{2r^2} \tilde{\mathfrak T}_{2,2} \tilde{\mathfrak V}_{2,1}+\frac 12\partial_r \tilde{\mathfrak T}_{2,1}\partial_r \tilde{\mathfrak V}_{2,2}+\frac{1}{2r^2} \tilde{\mathfrak T}_{2,1} \tilde{\mathfrak V}_{2,2}  \right)\cos(\theta)\\
& +\left(\frac 12\partial_r \tilde{\mathfrak T}_{2,1}\partial_r \tilde{\mathfrak V}_{2,1}-\frac{1}{2r^2} \tilde{\mathfrak T}_{2,1} \tilde{\mathfrak V}_{2,1}+\partial_r \tilde{\mathfrak T}_{2,0}\partial_r \tilde{\mathfrak V}_{2,2}+\partial_r \tilde{\mathfrak T}_{2,2}\partial_r \tilde{\mathfrak V}_{2,0}     \right) \cos(2\theta) \\
& \pm \left(\frac 12\partial_r \tilde{\mathfrak T}_{2,2}\partial_r \tilde{\mathfrak V}_{2,1}-\frac{1}{r^2} \tilde{\mathfrak T}_{2,2} \tilde{\mathfrak V}_{2,1}+\frac 12 \partial_r \tilde{\mathfrak T}_{2,1}\partial_r \tilde{\mathfrak V}_{2,2}-\frac{1}{r^2} \tilde{\mathfrak T}_{2,1}\partial_r \tilde{\mathfrak V}_{2,2}     \right) \cos(3\theta)\\
& + \left(\frac 12\partial_r \tilde{\mathfrak T}_{2,2}\partial_r \tilde{\mathfrak V}_{2,2}-\frac{2}{r^2} \tilde{\mathfrak T}_{2,2} \tilde{\mathfrak V}_{2,2}   \right) \cos(4\theta),
\end{align*}
and so the nonlinear term can be decomposed as
\begin{align*}
& \nabla . (\tilde{\mathfrak T}_{2,\pm}\nabla \tilde{\mathfrak V}_{2,\pm})\\
& = -\tilde{\mathfrak T}_{2,0}^2 -\frac 12 \tilde{\mathfrak T}_{2,1}^2 -\frac 12 \tilde{\mathfrak T}_{2,2}^2+ \partial_r \tilde{\mathfrak T}_{2,0}\partial_r \tilde{\mathfrak V}_{2,0}+\frac 12 \partial_r \tilde{\mathfrak T}_{2,1}\partial_r \tilde{\mathfrak V}_{2,1}+\frac{1}{2r^2}  \tilde{\mathfrak T}_{2,1} \tilde{\mathfrak V}_{2,1}+\frac 12 \partial_r \tilde{\mathfrak T}_{2,2}\partial_r \tilde{\mathfrak V}_{2,2}+\frac{2}{r^2}  \tilde{\mathfrak T}_{2,2} \tilde{\mathfrak V}_{2,2}\\
& \pm \Bigg(-2 \tilde{\mathfrak T}_{2,0} \tilde{\mathfrak T}_{2,1}-\tilde{\mathfrak T}_{2,1}\tilde{\mathfrak T}_{2,2}+\partial_r \tilde{\mathfrak T}_{2,0}\partial_r \tilde{\mathfrak V}_{2,1}+\partial_r \tilde{\mathfrak T}_{2,1}\partial_r \tilde{\mathfrak V}_{2,0}+\frac 12\partial_r \tilde{\mathfrak T}_{2,2}\partial_r \tilde{\mathfrak V}_{2,1}+\frac{1}{2r^2} \tilde{\mathfrak T}_{2,2} \tilde{\mathfrak V}_{2,1}\\
&\qquad \qquad+\frac 12\partial_r \tilde{\mathfrak T}_{2,1}\partial_r \tilde{\mathfrak V}_{2,2}+\frac{1}{2r^2} \tilde{\mathfrak T}_{2,1} \tilde{\mathfrak V}_{2,2} \Bigg)\cos(\theta)\\
& +  \left(- \frac 12 \tilde{\mathfrak T}_{2,1}^2-2 \tilde{\mathfrak T}_{2,0}\tilde{\mathfrak T}_{2,2}+\frac 12\partial_r \tilde{\mathfrak T}_{2,1}\partial_r \tilde{\mathfrak V}_{2,1}-\frac{1}{2r^2} \tilde{\mathfrak T}_{2,1} \tilde{\mathfrak V}_{2,1}+\partial_r \tilde{\mathfrak T}_{2,0}\partial_r \tilde{\mathfrak V}_{2,2}+\partial_r \tilde{\mathfrak T}_{2,2}\partial_r \tilde{\mathfrak V}_{2,0}  \right)\cos(2 \theta) \\
& \pm \left(- \tilde{\mathfrak T}_{2,1}\tilde{\mathfrak T}_{2,2}+\frac 12\partial_r \tilde{\mathfrak T}_{2,2}\partial_r \tilde{\mathfrak V}_{2,1}-\frac{1}{r^2} \tilde{\mathfrak T}_{2,2} \tilde{\mathfrak V}_{2,1}+\frac 12 \partial_r \tilde{\mathfrak T}_{2,1}\partial_r \tilde{\mathfrak V}_{2,2}-\frac{1}{r^2} \tilde{\mathfrak T}_{2,1}\partial_r \tilde{\mathfrak V}_{2,2} \right) \cos(3 \theta)  \\
& +\left(- \frac 12 \tilde{\mathfrak T}_{2,2}^2+\frac 12\partial_r \tilde{\mathfrak T}_{2,2}\partial_r \tilde{\mathfrak V}_{2,2}-\frac{2}{r^2} \tilde{\mathfrak T}_{2,2} \tilde{\mathfrak V}_{2,2} \right) \cos(4 \theta)
\end{align*}

Therefore, by \eqref{multisoliton:id:F4},
\begin{align*}
-F_{4,\pm}\mp \frac{\alpha-\beta}{\nu}a.\nabla \mathfrak T_{2,\pm}+\Lambda \mathfrak T_{2,\pm}+\alpha (G_{1,\pm}.\nabla \mathfrak T_{2,\pm})+\nabla . (\tilde{\mathfrak T}_{2,\pm}\nabla \tilde{\mathfrak V}_{2,\pm})=\sum_{k=0}^4(\pm 1)^k \mathfrak S_{4,k}\cos(4\theta)
\end{align*}
with
\begin{align}
\label{multisoliton:id:def-S40}&\mathfrak S_{4,0}= -\frac{\alpha^2}{32}r\partial_r \Lambda U-\gamma_3 \Lambda U -\frac{\alpha-\beta}{\nu}\frac{|a|}{2}(\partial_r \mathfrak T_{2,1}+\frac{\mathfrak T_{2,1}}{r})+\beta \Lambda \mathfrak T_{2,0}+\frac{\alpha}{4}\Lambda \mathfrak T_{2,2}- \tilde{\mathfrak T}_{2,0}^2 -\frac 12 \tilde{\mathfrak T}_{2,1}^2\\
\nonumber & \qquad \quad  -\frac 12 \tilde{\mathfrak T}_{2,2}^2+ \partial_r \tilde{\mathfrak T}_{2,0}\partial_r \tilde{\mathfrak V}_{2,0}+\frac 12 \partial_r \tilde{\mathfrak T}_{2,1}\partial_r \tilde{\mathfrak V}_{2,1}+\frac{1}{2r^2}  \tilde{\mathfrak T}_{2,1} \tilde{\mathfrak V}_{2,1}+\frac 12 \partial_r \tilde{\mathfrak T}_{2,2}\partial_r \tilde{\mathfrak V}_{2,2}+\frac{2}{r^2}  \tilde{\mathfrak T}_{2,2} \tilde{\mathfrak V}_{2,2},\\
\label{multisoliton:id:def-S41}&\mathfrak S_{4,1}= - \frac{\alpha-\beta}{\nu}|a|(\partial_r \mathfrak T_{2,0}+\frac{\partial_r \mathfrak T_{2,2}}{2}+\frac{ \mathfrak T_{2,2}}{2r})+ \beta \Lambda  \mathfrak T_{2,1}+ \alpha (\frac{r}{4}\partial_r  \mathfrak T_{2,1}+\frac{ \mathfrak T_{2,1}}{4})\\
\nonumber &\qquad \quad -2 \tilde{\mathfrak T}_{2,0} \tilde{\mathfrak T}_{2,1}-\tilde{\mathfrak T}_{2,1}\tilde{\mathfrak T}_{2,2}+\partial_r \tilde{\mathfrak T}_{2,0}\partial_r \tilde{\mathfrak V}_{2,1}+\partial_r \tilde{\mathfrak T}_{2,1}\partial_r \tilde{\mathfrak V}_{2,0}+\frac 12\partial_r \tilde{\mathfrak T}_{2,2}\partial_r \tilde{\mathfrak V}_{2,1}\\
\nonumber &\qquad \quad +\frac{1}{2r^2} \tilde{\mathfrak T}_{2,2} \tilde{\mathfrak V}_{2,1} +\frac 12\partial_r \tilde{\mathfrak T}_{2,1}\partial_r \tilde{\mathfrak V}_{2,2}+\frac{1}{2r^2} \tilde{\mathfrak T}_{2,1} \tilde{\mathfrak V}_{2,2} ,\\
\label{multisoliton:id:def-S42}&\mathfrak S_{4,2}= (\gamma_4 -\frac{\alpha^2}{16})r\partial_r \Lambda U-L_0\alpha^{\frac 32}(-\frac{1}{\sqrt{2}}r^2 \partial_r^2 U-\frac{1}{3\sqrt{2}}r\partial_r U) -\frac{\alpha-\beta}{\nu}\frac{|a|}{2}(\partial_r  \mathfrak T_{2,1}-\frac{ \mathfrak T_{2,1}}{r})+\beta \Lambda  \mathfrak T_{2,2}+\frac{\alpha}{2}r\partial_r  \mathfrak T_{2,0},\\
\nonumber & \qquad \quad- \frac 12 \tilde{\mathfrak T}_{2,1}^2-2 \tilde{\mathfrak T}_{2,0}\tilde{\mathfrak T}_{2,2}+\frac 12\partial_r \tilde{\mathfrak T}_{2,1}\partial_r \tilde{\mathfrak V}_{2,1}-\frac{1}{2r^2} \tilde{\mathfrak T}_{2,1} \tilde{\mathfrak V}_{2,1}+\partial_r \tilde{\mathfrak T}_{2,0}\partial_r \tilde{\mathfrak V}_{2,2}+\partial_r \tilde{\mathfrak T}_{2,2}\partial_r \tilde{\mathfrak V}_{2,0} ,\\
\label{multisoliton:id:def-S43}&\mathfrak S_{4,3}=- \frac{\alpha-\beta}{\nu}\frac{|a|}{2}(\partial_r  \mathfrak T_{2,2}-\frac{ \mathfrak T_{2,2}}{r})+ \alpha (\frac{r}{4}\partial_r  \mathfrak T_{2,1}-\frac{ \mathfrak T_{2,1}}{4})\\
\nonumber&\qquad \quad - \tilde{\mathfrak T}_{2,1}\tilde{\mathfrak T}_{2,2}+\frac 12\partial_r \tilde{\mathfrak T}_{2,2}\partial_r \tilde{\mathfrak V}_{2,1}-\frac{1}{r^2} \tilde{\mathfrak T}_{2,2} \tilde{\mathfrak V}_{2,1}+\frac 12 \partial_r \tilde{\mathfrak T}_{2,1}\partial_r \tilde{\mathfrak V}_{2,2}-\frac{1}{r^2} \tilde{\mathfrak T}_{2,1}\partial_r \tilde{\mathfrak V}_{2,2} \\
\label{multisoliton:id:def-S44}&\mathfrak S_{4,4}= \frac{\alpha^2}{16}r^3 \partial_r (U+\frac 12 \Lambda U)-L_0\alpha^{\frac 32}(-\frac{1}{\sqrt{2}}r^2 \partial_r^2 U+\frac{1}{3\sqrt{2}}r\partial_r U)  +\frac{\alpha}{4}r\partial_r  \mathfrak T_{2,2}-\frac{\alpha}{2} \mathfrak T_{2,2}\\
\nonumber & \qquad \quad - \frac 12 \tilde{\mathfrak T}_{2,2}^2+\frac 12\partial_r \tilde{\mathfrak T}_{2,2}\partial_r \tilde{\mathfrak V}_{2,2}-\frac{2}{r^2} \tilde{\mathfrak T}_{2,2} \tilde{\mathfrak V}_{2,2}.
\end{align}
In order to solve the system \eqref{multisoliton:id:systeme-T2} we decompose accordingly
\begin{equation} \label{multisoliton:id:expansion-T4}
(\mathfrak T_{4,\pm},\mathfrak V_{4,\pm})= \sum_{k=0}^4 (-1)^k (\mathfrak T_{2,k}(r)\cos(k\theta),\mathfrak V_{2,k}(r)\cos(k\theta))
\end{equation}
where the profiles $(\mathfrak T_{4,k},\mathfrak V_{4,k})$ for $k=1,...,4$ solve the systems
\begin{equation} \label{multisoliton:id:equation-T4-polar}
\left\{ \begin{array}{l l} \Ls_{0,k}(\mathfrak T_{4,k},\mathfrak V_{4,k})=\mathfrak S_{4,k},\\
-\Delta_k \mathfrak V_{4,k}=\mathfrak T_{4,k}.
\end{array} \right.
\end{equation}

In the lemma below we make an abuse of language and say that $(\mathfrak T_{i,k},\mathfrak V_{i,k})$ is smooth if the function $y\mapsto (\mathfrak T_{i,k}(r)\cos(\theta),\mathfrak V_{i,k}(r)\cos(\theta))$ is smooth on $\mathbb R^2$.

\begin{lemma} \label{lem:profile_mathfrakTk}

Let $\mu''=-8L_0(\frac{\alpha}{4}+\beta)$. There exist smooth solutions to \eqref{multisoliton:id:equation-T2-polar}, \eqref{multisoliton:id:equation-T3-polar} and \eqref{multisoliton:id:equation-T4-polar} such that as $r\to \infty$, for $k\in \mathbb N$:
\begin{align}
&\label{multisoliton:bd:T20} |\partial_r^k \mathfrak T_{2,0}(r)|\lesssim r^{-4-k}\ln r, \qquad  |\partial_r^k \mathfrak V_{2,0}(r)|\lesssim r^{-2-k}\ln r,\\
&\label{multisoliton:bd:T21} |\partial_r^k \mathfrak T_{2,1}(r)|\lesssim r^{-3-k}, \qquad  |\partial_r^k \mathfrak V_{2,1}(r)|\lesssim r^{-1-k}\ln r,\\
&\label{multisoliton:bd:T22} |\partial_r^k \mathfrak T_{2,2}(r)|\lesssim r^{-4-k}, \qquad  |\partial_r^k \mathfrak V_{2,2}(r)|\lesssim r^{-2-k}\ln r,\\
&\label{multisoliton:bd:T30} |\partial_r^k \mathfrak T_{3,0}(r)|\lesssim r^{-4-k}\ln r,, \qquad  |\partial_r^k \mathfrak V_{3,0}(r)|\lesssim  r^{-2-k}\ln r,\\
&\label{multisoliton:bd:T31} |\partial_r^k \mathfrak T_{3,1}(r)|\lesssim r^{-3-k}, \qquad  |\partial_r^k \mathfrak V_{3,1}(r)|\lesssim  r^{-1-k}\ln r,\\
&\label{multisoliton:bd:T32} |\partial_r^k \mathfrak T_{3,2}(r)|\lesssim r^{-4-k}, \qquad  |\partial_r^k \mathfrak V_{3,2}(r)|\lesssim r^{-2-k}\ln r,\\
&\label{multisoliton:bd:T33} |\partial_r^k \mathfrak T_{3,3}(r)|\lesssim r^{-3-k}, \qquad  |\partial_r^k \mathfrak V_{3,3}(r)|\lesssim r^{-1-k},\\
&\label{multisoliton:bd:T40} |\partial_r^k \mathfrak T_{4,0}(r)|\lesssim r^{-2-k}\ln^2 r, \qquad  |\partial_r^k \mathfrak V_{4,0}(r)|\lesssim r^{-k}\ln^3 r,\\
&\label{multisoliton:bd:T41} |\partial_r^k \mathfrak T_{4,1}(r)|\lesssim r^{-1-k}\ln r, \qquad  |\partial_r^k \mathfrak V_{4,1}(r)|\lesssim r^{1-k}\ln r,\\
&\label{multisoliton:bd:T42} |\partial_r^k \mathfrak T_{4,2}(r)|\lesssim r^{-2-k}\ln r, \qquad  |\partial_r^k \mathfrak V_{4,2}(r)|\lesssim r^{-k}\ln r,\\
&\label{multisoliton:bd:T43} |\partial_r^k \mathfrak T_{4,3}(r)|\lesssim r^{-1-k}, \qquad  |\partial_r^k \mathfrak V_{4,3}(r)|\lesssim r^{1-k},\\
&\label{multisoliton:bd:T44} |\partial_r^k \mathfrak T_{4,4}(r)|\lesssim r^{-2-k}, \qquad  |\partial_r^k \mathfrak V_{4,4}(r)|\lesssim r^{-k}.
\end{align}

\end{lemma}

Thanks to the above Lemma, injecting \eqref{multisoliton:bd:T20}-\eqref{multisoliton:bd:T44} in \eqref{multisoliton:id:expansion-T2}, \eqref{multisoliton:id:expansion-T3} and \eqref{multisoliton:id:expansion-T4} shows
\begin{align}
& \label{multisoliton:bd:T2} |\nabla^k \mathfrak T_2(y)|\lesssim \langle y \rangle^{-3-k}, \qquad  |\nabla^k \mathfrak V_2(y)|\lesssim \langle y \rangle^{-1-k}\langle \ln \langle y \rangle\rangle,\\
& \label{multisoliton:bd:T3} |\nabla^k \mathfrak T_3(y)|\lesssim \langle y \rangle^{-3-k}, \qquad  |\nabla^k \mathfrak V_3(y)|\lesssim \langle y \rangle^{-1-k},\\
& \label{multisoliton:bd:T4} |\nabla^k \mathfrak T_4(y)|\lesssim \langle y \rangle^{-1-k}, \qquad  |\nabla^k \mathfrak V_4(y)|\lesssim \langle y \rangle^{1-k}\langle \ln \langle y \rangle\rangle,
\end{align}

\begin{proof}

As we will see, most of the equations in \eqref{multisoliton:id:equation-T2-polar}, \eqref{multisoliton:id:equation-T3-polar} and \eqref{multisoliton:id:equation-T4-polar} involve source terms with the same particular cancellations and estimates as for the equations \eqref{sys:S111}-\eqref{sys:S402'} that were solved in Proposition \ref{pr:solution-system-SW}. Therefore, we will refer to the proof of Proposition \ref{pr:solution-system-SW} when the computations will be the same, and will give details for the equations that are different.

The smoothness of these profiles (smoothness seen as the radial component of the corresponding spherical harmonics) follows from the choice of the integration constants in the formulas below and from the behaviour near $0$ of the solutions to the homogeneous systems in Lemma \ref{lemm:SolHom}. We use the notation $f(r)=\Oc_{r\to \infty}(g(r))$ if $|\partial_r^k f(r)|\lesssim r^{-k}g(r)$ for $r\geq 1$.

\medskip

\noindent \underline{Computation of $(\mathfrak T_{2,0},\mathfrak V_{2,0})$}. We move to the partial mass setting. We first notice that the source term $\mathfrak S_{2,0}$ enjoys three cancellations. Namely we have the mass cancellation $\int  \Lambda U+\frac 12 \Lambda^2 U=0$ so that $ \int_0^\infty r\mathfrak S_{2,0}(r)dr=0$ and hence $m_{\mathfrak S_{2,0}}(r)=\int_0^r \zeta \mathfrak S_{2,0}(\zeta)d\zeta=-\int_r^\infty \zeta \mathfrak S_{2,0}(\zeta)d\zeta$. This, and the tail cancellation $U+\frac 12 \Lambda U=\Oc(\langle y \rangle^{-6})$ that implies $\mathfrak S_{2,0}(r)=\Oc_{r\to \infty}(r^{-6})$, give
\begin{equation} \label{multisoliton:bd:S20-1}
m_{\mathfrak S_{2,0}}(r)=\Oc_{r\to \infty}(r^{-4}).
\end{equation}
Using $\Lambda f=r^{-1}\partial_r (r^2 f)$ we have the identity $m_{\mathfrak S_{2,0}}(r)=r^2 U(r)+\frac 12 r \partial_r (r^2 U)$. This implies the following third cancellation for the partial mass
\begin{equation} \label{multisoliton:bd:S20-2}
\int_0^\infty rm_{\mathfrak S_{2,0}}(r)dr=\int_0^\infty (r^3 U(r)+\frac 12 r^2 \partial_r (r^2 U))dr=0.
\end{equation}
We solve \eqref{multisoliton:id:equation-T2-polar} for $k=0$ using Lemma \eqref{lemm:invA0} and obtain
\begin{align}
\label{multisoliton:id:T20-formula} & \mathfrak T_{2,0}(r)=\frac{\partial_r m_{\mathfrak T_{2,0}}}{r}, \qquad \Phi_{\mathfrak T_{2,0}}=C-\int_1^r \frac{m_{\mathfrak T_{2,0}}(\zeta)}{\zeta}d\zeta,\\
\label{multisoliton:id:mT20-formula}& m_{\mathfrak T_{2,0}}(r)=  \frac{1}{2}\psi_0\int_r^1 \frac{\zeta^4  + 4\zeta^2 \ln \zeta -1}{\zeta} m_{\mathfrak S_{2,0}}(\zeta) d\zeta +\frac{1}{2}\tilde \psi_0\int_0^r \zeta m_{\mathfrak S_{2,0}}(\zeta) d\zeta,
\end{align}
for $C\in \mathbb R$ any integration constant. Injecting \eqref{multisoliton:bd:S20-1}, \eqref{multisoliton:bd:S20-2} and \eqref{systeme-odes:fundamental-solution-radial} in \eqref{multisoliton:id:mT20-formula} shows $m_{\mathfrak T_{2,0}}(r)=\Oc_{r\to \infty}(r^{-2}\ln r)$. In turn, injecting this in \eqref{multisoliton:id:T20-formula} shows the desired estimates \eqref{multisoliton:bd:T20}.

\medskip

\noindent \underline{Computation of $(\mathfrak T_{2,1},\mathfrak V_{2,1})$}. By \eqref{multisoliton:id:def-S2k}, the system \eqref{multisoliton:id:equation-T2-polar} for $k=1$ is the same as \eqref{sys:S111} up to a constant factor. We thus take $(\mathfrak T_{2,1},\mathfrak V_{2,1})=(\mathcal S_{1,1,1},\mathcal W_{1,1,1})$. The estimates \eqref{multisoliton:bd:T21} then follow from \eqref{bd:estimate-S111-W111}.

\medskip

\noindent \underline{Computation of $(\mathfrak T_{2,2},\mathfrak V_{2,2})$}. Injecting the tail cancellation $U+\frac 12 \Lambda U=\Oc(\langle y \rangle^{-6})$ in \eqref{multisoliton:id:def-S2k} gives
\begin{equation} \label{multisoliton:bd:S22}
\mathfrak S_{2,2}(r)=\Oc_{r\to \infty}(r^{-6}).
\end{equation}
We solve the system \eqref{multisoliton:id:equation-T2-polar} by applying Lemma \ref{lemm:sol_inhol} and obtain
$$
(\mathfrak T_{2,2},\mathfrak V_{2,2})=\sum_{k=0}^4 \gamma_{2,2,k}(r)(h_{2,k},g_{2,k})
$$
where we choose the integration constants so that
\begin{align*}
& \gamma_{2,2,1}(r)= -\int_r^\infty \frac{\mathfrak S_{2,2}W_{2,1}}{W_2}ds = \int_r^\infty  \Oc_{s\to \infty}(s^{-5})ds = \Oc_{r\to \infty}(r^{-4}),\\
& \gamma_{2,2,2}(r)= -\int_0^r \frac{\mathfrak S_{2,2}W_{2,2}}{W_2}ds = \int_0^r  \Oc_{s\to \infty}(s^{-1})ds =\Oc_{r\to \infty}(\ln r ),\\
& \gamma_{2,2,3}(r)= -\int_r^\infty \frac{\mathfrak S_{2,2}W_{2,3}}{W_2}ds  = \int_r^\infty  \Oc_{s\to \infty}(s^{-3-\sqrt{8}})ds=\Oc_{r\to \infty}(r^{-2-\sqrt{8}}),\\
& \gamma_{2,2,4}(r)= -\int_0^r \frac{\mathfrak S_{2,2}W_{2,4}}{W_2}ds = \int_0^r  \Oc_{s\to \infty}(s^{-3+\sqrt{8}})ds =\Oc_{r\to \infty}(r^{-2+\sqrt{8}})
\end{align*}
where the estimates of the above integrals follow from \eqref{multisoliton:bd:S22} and the estimates of Lemma \ref{lemm:sol_inhol}. Combining these estimates with the asymptotic behaviours of $(h_{2,k},g_{2,k})$ of Lemma \ref{lemm:SolHom} shows the desired estimates \eqref{multisoliton:bd:T22}.

\medskip

\noindent \underline{Computation of $(\mathfrak T_{3,0},\mathfrak V_{3,0})$}. We have by \eqref{multisoliton:id:def-S33} the cancellation
$$
\int_0^\infty \mathfrak S_{3,0}=8L_0\gamma_2 \int_0^\infty (r\partial_r^2 U+\partial_r U)dr=0
$$
and the asymptotic behaviour $ \mathfrak S_{3,0}(r)=\Oc_{r\to \infty}(r^{-6})$. Hence $\mathfrak S_{3,0}$ satisfies the same cancellation and estimate as the source term $\mathfrak S_{2,0}$ in the previous computation of $(\mathfrak T_{2,0},\mathfrak V_{2,0})$. The same arguments then yield  \eqref{multisoliton:bd:T30}.

\medskip

\noindent \underline{Computation of $(\mathfrak T_{3,1},\mathfrak V_{3,1})$}. The idea is to use the parameter $\mu''$ to enforce an orthogonality condition. Using the commutator relation $\partial_r \Lambda -\Lambda \partial_r =\partial_r$ and the identity $\partial_r \Lambda U=4\sqrt{2}\Sigma_{3,1}$, and choosing $\mu''=-8L_0(\frac{\alpha}{4}+\beta)$ we get by \eqref{multisoliton:id:def-S33} that $\mathfrak S_{3,1}=-32\sqrt{2}L_0(\frac{\alpha}{4}+\beta)\Sigma_{3,1}$. In view of Lemma \ref{lem:T3k}, we thus choose $(\mathfrak T_{3,1},\mathfrak V_{3,1})=-32\sqrt{2}L_0(\frac{\alpha}{4}+\beta)(T_{3,1},V_{3,1})$ and obtain the estimates \eqref{multisoliton:bd:T31}.

\medskip

\noindent \underline{Computation of $(\mathfrak T_{3,2},\mathfrak V_{3,2})$}. We have by \eqref{multisoliton:id:def-S33} that $\mathfrak S_{3,2}(r)=\Oc_{r\to \infty}(r^{-6})$. This is the same estimate as \eqref{multisoliton:bd:S22} for $\mathfrak S_{2,2}$. The same arguments used in the computation of $(\mathfrak T_{2,2},\mathfrak V_{2,2})$ then show \eqref{multisoliton:bd:T32}.

\medskip

\noindent \underline{Computation of $(\mathfrak T_{3,3},\mathfrak V_{3,3})$}. We inject the tail cancellation $U+\frac 12 \Lambda U=\Oc(\langle y \rangle^{-6})$ in \eqref{multisoliton:id:def-S33} and obtain
$$
\mathfrak S_{3,3}=\Oc_{r\to \infty}(r^{-5}).
$$
We notice that this asymptotic behaviour as $r \to \infty$ is exactly as that of $\Sigma_{3,0,3}$ in Step 4 of the proof of Proposition \ref{pr:solution-system-SW}, see \eqref{bd:Sigma313-asymptotics}. By reproducing exactly the same argument as that which was made to compute $(\mathcal S_{3,0,3},\mathcal W_{3,0,3})$, we get the same estimates \eqref{bd:estimate-S303-W303} for $(\mathfrak T_{3,3},\mathfrak V_{3,3})$. This shows \eqref{multisoliton:bd:T33}.

\medskip

\noindent \underline{Computation of $(\mathfrak T_{4,0},\mathfrak V_{4,0})$}. We first notice using \eqref{multisoliton:id:def-S40} the mass cancellation
\begin{align}
\nonumber \int_0^\infty r\mathfrak S_{4,0} dr & =-\frac{\alpha^2}{32}\int_0^\infty r^2\partial_r \Lambda U dr -\gamma_3 \int_0^\infty r \Lambda U dr -\frac{\alpha-\beta}{\nu}\frac{|a|}{2} \int_0^\infty ( r\partial_r \mathfrak T_{2,1}+\mathfrak T_{2,1})dr\\
 \label{multisoliton:bd:S40} & \qquad +\beta \int_0^\infty r\Lambda \mathfrak T_{2,0}dr+\frac{\alpha}{4}\int_0^\infty r\Lambda \mathfrak T_{2,2}dr +\int_{\mathbb R^2}\nabla . (\tilde{\mathfrak{ T}}_{2,+}\nabla \tilde{ \mathfrak V}_{2,+})dy\ =0.
\end{align}
Indeed, all terms in the right-hand side above vanish, because of integration by parts that are legitimate from the asymptotic behaviours \eqref{multisoliton:bd:T20}, \eqref{multisoliton:bd:T21} and \eqref{multisoliton:bd:T22} that have already been obtained previously in the proof. Moreover, these same asymptotic behaviours imply $\mathfrak S_{4,0} (r)=\Oc(r^{-4}\ln r)$. We notice that this estimate and the cancellation \eqref{multisoliton:bd:S40} are exactly as that for $\Sigma_{2,2,0}$ in Step 2 of the proof of Proposition \ref{pr:solution-system-SW}, see \eqref{id-mass-cancellation-Sigma220} and \eqref{bd-Sigma220}, up to a $\ln r$ factor for the estimate. The same argument as that of Step 2 in the proof of Proposition \ref{pr:solution-system-SW} can be applied, yielding for $\mathfrak S_{4,0}$ the same estimates \eqref{bd:estimate-S220-W220} as for $(\mathcal S_{2,2,0},\mathcal W_{2,2,0})$, with an additional $\log r$ factor. This is \eqref{multisoliton:bd:T40}.

\medskip

\noindent \underline{Computation of $(\mathfrak T_{4,1},\mathfrak V_{4,1})$}. Injecting \eqref{multisoliton:bd:T20}, \eqref{multisoliton:bd:T21} and \eqref{multisoliton:bd:T22} in \eqref{multisoliton:id:def-S41} we find
$$
\mathfrak S_{4,1}(r)=\Oc_{r\to \infty}(r^{-3}).
$$ 
This estimate is exactly as the estimate \eqref{bd:Sigma311-asymptotics} for $\Sigma_{3,1,1}$ in Step 1 of the proof of Proposition \ref{pr:solution-system-SW}. The exact same argument yield that $(\mathfrak T_{4,1},\mathfrak V_{4,1})$ satisfiies the same estimates as \eqref{bd:estimate-S311-W311} for $(\mathcal S_{3,1,1},\mathcal W_{3,1,1})$, yielding \eqref{multisoliton:bd:T41}.

\medskip

\noindent \underline{Computation of $(\mathfrak T_{4,2},\mathfrak V_{4,2})$}. By \eqref{multisoliton:bd:T20}, \eqref{multisoliton:bd:T21} and \eqref{multisoliton:bd:T22} and \eqref{multisoliton:id:def-S42} we have
$$
\mathfrak S_{4,2}(r)=\Oc_{r\to \infty}(r^{-4}\ln r).
$$ 
This is the same estimate as for $\Sigma_{2,2,2}$ in Step 3 of the proof of Proposition \ref{pr:solution-system-SW}, up to a $\log r$ factor. Reproducing the same arguments shows that $(\mathfrak T_{4,2},\mathfrak V_{4,2})$ enjoys the same estimates as \eqref{bd:estimate-S222-W222} for $(\mathcal S_{2,2,2},\mathcal W_{2,2,2})$ with an additional $\log r$ factor. This shows \eqref{multisoliton:bd:T42}.

\medskip

\noindent \underline{Computation of $(\mathfrak T_{4,3},\mathfrak V_{4,3})$}. Injecting \eqref{multisoliton:bd:T21} and \eqref{multisoliton:bd:T22} in \eqref{multisoliton:id:def-S43} we get
$$
\mathfrak S_{4,3}(r)=\Oc_{r\to \infty}(r^{-3} ).
$$ 
We recognise the same estimate as that for $\Sigma_{3,1,3}$ in Step 4 of the proof of Proposition \ref{pr:solution-system-SW}. The exact same reasoning shows for $(\mathfrak T_{4,3},\mathfrak V_{4,3})$ the same estimates as \eqref{bd:estimate-S313-W313} for $(\mathcal S_{3,1,3},\mathcal W_{3,1,3})$. This implies \eqref{multisoliton:bd:T43}.

\medskip

\noindent \underline{Computation of $(\mathfrak T_{4,4},\mathfrak V_{4,4})$}. The tail cancellation $U+\frac 12 \Lambda U=\Oc(\langle y \rangle^{-6})$ and \eqref{multisoliton:bd:T22} imply using \eqref{multisoliton:id:def-S44} that
$$
\mathfrak S_{4,4}(r)=\Oc_{r\to \infty}(r^{-4}).
$$
This estimate being the same as the estimate \eqref{bd:Sigma404-asymptotics} for $\Sigma_{4,0,4}$ in Step 5 of the proof of Proposition \ref{pr:solution-system-SW}, we obtain the same estimates \eqref{bd:estimate-S404-W404} for $(\mathfrak T_{4,4},\mathfrak V_{4,4})$. This is precisely \eqref{multisoliton:bd:T44}.

\end{proof}

\subsection{End of the construction of the corrected multisoliton}

\label{subsec:final-corrections-multisolitons}

\begin{proof}[Proof of Proposition \ref{prop:E0}]

\noindent \textbf{Step 1}. \emph{Proof of the pointwise estimates \eqref{est:pointwise_Xi}, \eqref{est:pointwise_Xi_dk} and \eqref{est:PhiXi}}. We estimate the right-hand sides of \eqref{multisoliton:id:decomposition-Xi-1} and \eqref{multisoliton:id:decomposition-PhiXi-1}. We have by \eqref{def:phii} that
$$
 \nu^2 \phi_0+(\frac{1}{16} \Lambda U-L_0\nu \partial_{x_1}U)_{1,\nu}+(\frac{1}{16}\Lambda U+L_0\nu \partial_{x_1}U)_{2,\nu} =\sum_\pm (1-\chi_{\pm a,\zeta_*})\nu^{-2}(\frac{1}{16} \Lambda U-L_0\nu \partial_{x_1}U)(\frac{\cdot \mp a}{\nu})+\nu^2 \tilde \phi_0.
$$
Using that $(1-\chi_{\pm a,\zeta_*})$ is supported for $|z\mp a|\geq \zeta_*$ and $|\nabla^k U(y)|\lesssim \langle y \rangle^{-4-k}$ we have 
$$
|\sum_\pm (1-\chi_{\pm a,\zeta_*})\nu^{-2}(\frac{1}{16} \Lambda U-L_0\nu \partial_{x_1}U)(\frac{\cdot \mp a}{\nu})|\lesssim \nu^2 \langle z\rangle^{C}.
$$
Similarly, since $\Phi_f=-(2\pi)^{-1}\log |\cdot|*f$ we have
$$
|\Phi_{\sum_\pm (\chi^*-\chi_{\pm a,\zeta_*})\nu^{-2}(\frac{1}{16} \Lambda U-L_0\nu \partial_{x_1}U)(\frac{\cdot \mp a}{\nu})}|\lesssim \nu^2 \langle \ln \langle z\rangle \rangle.
$$
Combining with \eqref{est:pointwise_phiitil} and \eqref{est:pointwise_Phiphiitil} this shows
\begin{equation} \label{multisoliton:bd:Xi-1}
\left| \nu^2 \phi_0+(\frac{1}{16} \Lambda U-L_0\nu \partial_{x_1}U)_{1,\nu}+(\frac{1}{16}\Lambda U+L_0\nu \partial_{x_1}U)_{2,\nu}\right|\lesssim  \sum_\pm \frac{\nu^2}{(\nu+|z\pm a|)^2}\langle z \rangle^C
\end{equation}
and
\begin{equation} \label{multisoliton:bd:PhiXi-1}
\left|\nabla \Phi_{\chi^* \nu^2 \phi_0+(\frac{1}{16} \Lambda U-L_0\nu \partial_{x_1}U)_{1,\nu}+(\frac{1}{16}\Lambda U+L_0\nu \partial_{x_1}U)_{2,\nu}}\right|\lesssim  \sum_\pm \frac{\nu^2}{(\nu+|z\pm a|)}\langle z \rangle^C.
\end{equation}
We have by injecting the estimates of Lemma \ref{lem:profile_mathfrakTk} in \eqref{multisoliton:id:def-tildeXi} using $\ln \langle y \rangle\lesssim \langle y \rangle^\epsilon$ that for any $0<\epsilon<1$ arbitrarily small that
\begin{equation} \label{multisoliton:bd:Xi-2}
|\tilde \Xi |\lesssim \sum_\pm \sum_{i=0}^2 \nu^i |\mathfrak T_{i+2,\pm} (\frac{z\mp a}{\nu})|\lesssim \sum_\pm  \sum_{i=0}^2 \frac{\nu^{3-\epsilon}}{(\nu+|z\mp a|)^{3-\epsilon-i}} \lesssim \sum_\pm \frac{\nu^{3-\epsilon}}{(\nu+|z\pm a|)^{3-\epsilon}}\langle z \rangle^C.
\end{equation}
and similarly
\begin{equation} \label{multisoliton:bd:PhiXi-2}
|\tilde{\mathfrak V} |\lesssim \sum_\pm \frac{\nu^{3-\epsilon}}{(\nu+|z\pm a|)^{1-\epsilon}}\langle z \rangle^C.
\end{equation}
Injecting \eqref{multisoliton:bd:PhiXi-2} and $|U|\lesssim \langle y \rangle^{-4}$ in \eqref{multisoliton:id:elliptic-mathfrakV}, then using $\Phi_f=-(2\pi)^{-1}\log |\cdot|*f$ we find
\begin{equation} \label{multisoliton:bd:PhiXi-3}
 \bar{\mathfrak V}= \Phi_{ \left( (\frac{1}{2}\Lambda U-8L_0\nu \partial_{x_1}U)_{1,\nu}+(\frac{1}{2}\Lambda U+8L_0\nu \partial_{x_1}U)_{2,\nu} \right)(\chi^*-1)}+2\Phi_{\nabla \tilde{\mathfrak V}.\nabla \chi^*}+\Phi_{\tilde{\mathfrak V} \Delta \chi^*}=\Oc(\nu^2 \langle z \rangle)
\end{equation}

Injecting \eqref{multisoliton:bd:Xi-1}, \eqref{multisoliton:bd:PhiXi-1}, \eqref{multisoliton:bd:Xi-2}, \eqref{multisoliton:bd:PhiXi-2} and \eqref{multisoliton:bd:PhiXi-3} in \eqref{multisoliton:id:decomposition-Xi-1} and \eqref{multisoliton:id:decomposition-PhiXi-1} shows the first inequality in \eqref{est:pointwise_Xi} and \eqref{est:PhiXi} for $k=0$. The estimates for general $k\geq 1$ follow similarly as all estimates above adapt to higher order derivatives. The $\nu\partial_\nu$ and $\partial_a$ also follow similarly. This shows \eqref{est:pointwise_Xi}, \eqref{est:pointwise_Xi_dk} and \eqref{est:PhiXi}.

\medskip

\noindent \textbf{Step 2}. \emph{Proof of \eqref{def:mumuprime}}. We define $\mu'$ by \eqref{multisoliton:id:decomposition-mu'} with $\mu''=-8L_0(\frac{\alpha}{4}+\beta)$ as in Lemma \ref{lem:profile_mathfrakTk}. We have
$$
\Phi_{(\frac 12 \Lambda U +8L_0 \nu \partial_{x_1}U)_{2,\nu}}(a)=(\nabla (\frac 12 y.\nabla \Phi_U+8L_0\nu \partial_{x_1}\Phi_U)(\frac{2a}{\nu})=\Oc(\nu^2).
$$
We have $\bar{\mathfrak V}(a)=\Oc(\nu^{2})$ by \eqref{multisoliton:bd:PhiXi-3}, and using the estimates of Lemma \ref{lem:profile_mathfrakTk} with the bound $\langle\log \langle y \rangle \rangle =\Oc(\langle y \rangle^{\epsilon})$ for any $0<\epsilon\ll1$ arbitrarily small we have
$$
\sum_{i=0}^2\nu^{i+1}\nabla \mathfrak V_{i+2,-}(\frac{2a}{\nu}) =\Oc(\nu^{3-\epsilon}).
$$
Injecting the above estimates in \eqref{multisoliton:id:decomposition-mu'}, we obtain that $\mu'= \frac{8a}{\nu^2+4|a|^2}-\beta a+\Oc(\nu^2)$ as desired.

\medskip

\noindent \textbf{Step 3}. \emph{Proof of \eqref{est:Xi_localmass}}. We have by the decomposition \eqref{multisoliton:id:decomposition-Xi->2}:
\begin{align}
\label{multisoliton:id:decomposition-local-mass-Xi}\int_{|z \mp a| \leq \zeta_*} \nu \pa_\nu \Xi(z)dz & =  \nu \pa_\nu \int_{|z \mp a| \leq \zeta_*}\tilde{\mathfrak T}_{2,\pm}(\frac{\cdot \mp a}{\nu})+ \nu \pa_\nu \int_{|z \mp a| \leq \zeta_*}\tilde{\mathfrak T}_{2,\mp}(\frac{\cdot \pm a}{\nu})     +\nu \partial_\nu  \int_{|z \mp a| \leq \zeta_*}\tilde \Xi 
\end{align}
Using the decomposition \eqref{multisoliton:id:decomposition-tildemathfrakT2} and decomposing onto spherical harmonics, and then performing an explicit computation by changing variables $y=\frac{z\pm a}{\nu}$ and using \eqref{multisoliton:bd:T20}, \eqref{multisoliton:id:deftildeT20}, \eqref{eigenfunction:id:asymptotichatT2infty}, \eqref{eigenfunction:id:asymptoticT20infty0} and $\tilde \lambda_0=\Oc(|\ln \nu|^{-1})$ we find
$$
\int_{|z \mp a| \leq \zeta_*}\tilde{\mathfrak T}_{2,\pm}(\frac{\cdot \mp a}{\nu})dz=\int_{|z \mp a| \leq \zeta_*} (\mathfrak T_{2,0}-\frac 12 \tilde T_{2,0})(\frac{\cdot \mp a}{\nu})dz =\Oc(\nu^2).
$$
Using the decomposition \eqref{multisoliton:id:decomposition-tildemathfrakT2} we find that
$$
\int_{|z \mp a| \leq \zeta_*}\tilde{\mathfrak T}_{2,\mp}(\frac{\cdot \pm a}{\nu})  =\Oc(\nu^2).
$$
Using \eqref{multisoliton:bd:Xi-2} we have 
$$
 \int_{|z \mp a| \leq \zeta_*}\tilde \Xi =\Oc(\nu^2)
$$
Injecting the above estimates in \eqref{multisoliton:id:decomposition-local-mass-Xi} shows \eqref{est:Xi_localmass} as desired.

\medskip

\noindent \textbf{Step 4}. \emph{Estimate for the error $\bar \Psi_0$ generated by the first terms in the correction}. We estimate all terms in the right-hand side of \eqref{multisoliton:id:def-barPsi0}. For the first line we use that for $|y|\ll \nu^{-1}$ there hold $|\tilde G_+|\lesssim \nu^5 \langle y \rangle^4$ from Lemma \ref{lem:taylorpoissonfield} and $U+\frac 12 \Lambda U-8L_0\nu \partial_{x_1}U=\Oc(\langle y \rangle^{-6})$ to infer that for $z$ close to $a$,
$$
 \frac{1}{\nu^4}(\tilde G_+.\nabla (U+\frac{1}{2}\Lambda U-8L_0\nu \partial_{x_1}U)(\frac{\cdot-a}{\nu})=\Oc( \frac{\nu^4}{(\nu+|z-a|)^3}).
$$
For $|z-a|\gtrsim 1$ we have $ \nabla (U+\frac{1}{2}\Lambda U-8L_0\nu \partial_{x_1}U)(\frac{\cdot-a}{\nu})=\Oc(\nu^7 |z-a|^{-6})=\Oc(\nu^7 \langle z \rangle^{-6})$, and $\tilde G_+(\frac{\cdot-a}{\nu})=\Oc(\frac{\nu}{\nu+|z+a|})+\Oc(\nu |z-a|^3)=\Oc(\frac{\nu}{\nu+|z+a|})+\Oc(\nu \langle z \rangle^3)$ by \eqref{id:expansion:PhiU-other-bubble} and $\nabla \Phi_U(y)=\Oc(\langle y \rangle^{-1})$. Hence 
$$
 \frac{1}{\nu^4}(\tilde G_+.\nabla (U+\frac{1}{2}\Lambda U)(\frac{\cdot-a}{\nu})=\Oc(\frac{\nu^8\langle z \rangle^{-6}}{\nu+|z+a|})+\Oc(\nu^4\langle z \rangle^{-3})=\Oc(\frac{\nu^4\langle z \rangle^{-2}}{\nu+|z+a|}).
$$
Combining, we obtain
\begin{equation} \label{multisoliton:bd:estimate-bar-Psi0-line1} 
 \frac{1}{\nu^4}(\tilde G_+.\nabla (U+\frac{1}{2}\Lambda U-8L_0\nu \partial_{x_1}U)(\frac{\cdot-a}{\nu})=\Oc( \frac{\nu^4\langle z\rangle}{(\nu+|z+a|)(\nu+|z-a|)^3})
\end{equation}
For the second line, we use Lemma \ref{lem:taylorpoissonfield} and estimate
$$
\Phi_{U_{2,\nu}}(a)+\alpha a-\nu^2 (\gamma_1,0)=\frac{1}{\nu}(-\alpha a \nu+\alpha^{\frac 32}\nu^3 (\frac{1}{4\sqrt{2}},0)+\Oc(\nu^5))+\alpha a-\nu^2 (\frac{\alpha^{\frac 32}}{4\sqrt{2}},0)=\Oc(\nu^4).
$$
Using $ \nabla (U+\frac{1}{2}\Lambda U-8L_0\nu \partial_{x_1}U)(y)=\Oc(\langle y \rangle^{-7})+\Oc(\nu \langle y \rangle^{-6})$ this implies
\begin{equation} \label{multisoliton:bd:estimate-bar-Psi0-line2} 
\frac{1}{\nu^3}[\Phi_{U_{2,\nu}}(a)+\alpha a-\nu^2(\gamma_1,0)].\nabla (U+\frac{1}{2}\Lambda U-8L_0\nu \partial_{x_1}U)(\frac{\cdot-a}{\nu}) =\Oc(\frac{\nu^8\langle z \rangle}{\nu^7+|z-a|^7}).
\end{equation}
For the third line, we use $G_{3,+}=\Oc(\langle y \rangle^3)$ from Lemma \ref{lem:taylorpoissonfield} and $\nabla \partial_{x_1}U=\Oc(\langle y \rangle^{-6})$ so that
\begin{equation} \label{multisoliton:bd:estimate-bar-Psi0-line3}
8L_0\alpha^2 \nu G_{3,+}.\nabla \partial_{x_1}U (\frac{\cdot-a}{\nu})=\Oc(\frac{\nu^4}{(\nu+|z-a|)^3}).
\end{equation}
For the fourth line, we recall the decomposition \eqref{multisoliton:id:decomposition-produits} and estimate all terms in the right-hand side of this decomposition. For the first line, using $2U+\Lambda U=\Oc(\langle y \rangle^{-6})$ we have $ (2U+\Lambda U)_{2,\nu}(a)= (2U+\Lambda U)_{1,\nu}(-a)=\Oc(\nu^4)$ so that since $U=\Oc(\langle y \rangle^{-4})$ we have
\begin{equation} \label{multisoliton:bd:decomposition-produits-tech1}
 (2U+\Lambda U)_{2,\nu}(a)U_{1,\nu}+ (2U+\Lambda U)_{1,\nu}(-a)U_{2,\nu}=\Oc(\frac{\nu^6}{\nu^4+|z-a|^4})+\Oc(\frac{\nu^6}{\nu^4+|z+a|^4}).
\end{equation}
For the second and third line, we have for $z$ close to $-a$ that $(U_{1,\nu}-U_{1,\nu}(-a))=\Oc(\nu^2|z+a|)$ by Taylor expansion, and for $z$ away from $-a$ that $(U_{1,\nu}-U_{1,\nu}(-a))=\Oc(\frac{\nu^2}{(\nu+|z-a|)^4}\langle z \rangle^4)$. This implies $(U_{1,\nu}-U_{1,\nu}(-a))=\Oc(\frac{\nu^2|z+a|}{(\nu+|z-a|)^4}\langle z \rangle^3)$ for $z\in \mathbb R^2$. Analogue estimates can be showed for $(U_{2,\nu}-U_{2,\nu}(a))$, $ (\Lambda U)_{1,\nu}-(\Lambda U)_{1,\nu}(-a)$ and $ (\Lambda U)_{2,\nu}-(\Lambda U)_{2,\nu}(a)$, producing:
\begin{align}
\nonumber & 2(U_{1,\nu}-U_{1,\nu}(-a))(U_{2,\nu}-U_{2,\nu}(a))+ (U_{2,\nu}-U_{2,\nu}(a))( (\Lambda U)_{1,\nu}-(\Lambda U)_{1,\nu}(-a))\\
\label{multisoliton:bd:decomposition-produits-tech2} &\qquad \qquad + (U_{1,\nu}-U_{1,\nu}(-a))( (\Lambda U)_{2,\nu}-(\Lambda U)_{2,\nu}(a)) \ = \Oc(\frac{\nu^4 \langle z \rangle^6}{(\nu+|z-a|)^3(\nu+|z+a|)^3}).
\end{align}
Finally, since $U_{1,\nu}(-a),U_{2,\nu}(a) , (\Lambda U)_{1,\nu}(-a),(\Lambda U)_{2,\nu}(a) =\Oc(\nu^2)$ we have
\begin{equation} \label{multisoliton:bd:decomposition-produits-tech3}
-2U_{1,\nu}(-a)U_{2,\nu}(a)-U_{2,\nu}(a)(\Lambda U)_{1,\nu}(-a)-U_{1,\nu}(-a)(\Lambda U)_{2,\nu}(a)= \Oc(\nu^4).
\end{equation}
Injecting \eqref{multisoliton:bd:decomposition-produits-tech1}, \eqref{multisoliton:bd:decomposition-produits-tech2} and \eqref{multisoliton:bd:decomposition-produits-tech3} in \eqref{multisoliton:id:decomposition-produits} shows
\begin{align}
\label{multisoliton:bd:estimate-bar-Psi0-line4} & 2U_{1,\nu}U_{2,\nu}+U_{2,\nu}(\Lambda U_{1,\nu}+U_{1,\nu}(\Lambda U)_{2,\nu} -U_{2,\nu}(a)(\Lambda U)_{1,\nu} -U_{1,\nu}(-a)(\Lambda U)_{2,\nu}\\
\nonumber & = \Oc(\frac{\nu^4 \langle z \rangle^6}{(\nu+|z-a|)^3(\nu+|z+a|)^3}).
\end{align}
For the fifth line, we use \eqref{multisoliton:bd:expansion-nabla-Phi-tech-estimate} and $\nabla U_{1,\nu}=\Oc(\frac{\nu^2}{(\nu+|z-a|)^5})$ to get
\begin{align} \label{multisoliton:bd:estimate-bar-Psi0-line5}
& \nabla U_{1,\nu}.\left(\nabla \Phi_{(\frac 12\Lambda U+8L_0\nu \partial_{x_1}U)_{2,\nu}}(z)-\nabla \Phi_{(\frac 12 \Lambda U+8L_0\nu \partial_{x_1}U)_{2,\nu}}(a)+\frac{\alpha^2}{8}\nu^2 (-3(z_1-|a|),z_2)\right)\\
\nonumber & \qquad \qquad \qquad \qquad =\Oc(\frac{\nu^4}{(\nu+|z+a|)^3(\nu+|z-a|)^3}\langle z \rangle^2).
\end{align}
For the seventh line, using $U_{2,\nu}(z)=\Oc(\frac{\nu^2}{(\nu+|z+a|)^4})$ and $(\partial_{x_1}U)_{1,\nu}(z)=\Oc(\frac{\nu^3}{(\nu+|z-a|)^5})$ we have
\begin{align} \label{multisoliton:bd:estimate-bar-Psi0-line7}
16L_0\nu U_{2,\nu}(\partial_{x_1}U)_{1,\nu}=\Oc(\frac{\nu^6}{(\nu+|z-a|)^5(\nu+|z+a|)^4}).
\end{align}
For the eighth line we have that $|\nabla^k\bar{\mathfrak V}_2(z)|\lesssim \nu^{3-\epsilon}$ for $z$ close to $a$ because of \eqref{multisoliton:id:elliptic-mathfrakV-2}, \eqref{multisoliton:bd:PhiXi-2} and elliptic regularity. Hence $|\nabla \bar{\mathfrak V}_2(z)-\nabla \bar{\mathfrak V}_2(a)|\lesssim \nu^{3-\epsilon}|z-a|$. Combining with \eqref{multisoliton:bd:expansion-nabla-barV1-estimate} and \eqref{multisoliton:bd:PhiXi-2} this shows
$$
\left|\nabla \bar{\mathfrak V}(z)-\nabla \bar{\mathfrak V}(a)-\nabla^2 \bar{V}_1(a)(z-a)+\nabla((\chi^*-1)\tilde {\mathcal V})\right|\lesssim \nu^2 (\nu+|z-a|)^{2-\epsilon}.
$$
Using that $|\nabla U_{1,\nu}|\lesssim \nu^{2}(\nu+|z-a|)^{-5}$ this implies
\begin{align} \label{multisoliton:bd:estimate-bar-Psi0-line8}
\left| \left(\nabla \bar{\mathfrak V}(z)-\nabla \bar{\mathfrak V}(a)-\nabla^2 \bar{V}_1(a)(z-a)+\nabla((\chi^*-1)\tilde {\mathcal V})\right).\nabla U_{1,\nu}\right|\lesssim \frac{\nu^4}{(\nu+|z-a|)^{3-\epsilon}}.
\end{align}
Injecting \eqref{multisoliton:bd:estimate-bar-Psi0-line1}, \eqref{multisoliton:bd:estimate-bar-Psi0-line2}, \eqref{multisoliton:bd:estimate-bar-Psi0-line3}, \eqref{multisoliton:bd:estimate-bar-Psi0-line4}, \eqref{multisoliton:bd:estimate-bar-Psi0-line5} and  \eqref{multisoliton:bd:estimate-bar-Psi0-line7} and their analogue estimates for the symmetric terms, and \eqref{est:pointwise_Ri}, in \eqref{multisoliton:id:def-barPsi0} shows
\begin{equation} \label{multisoliton:bd:estimate-barPsi00-pointwise}
|\bar \Psi_0^{(0)}|\lesssim  \frac{\nu^2}{|\ln \nu|^2} \sum_{\pm} \frac{1  }{(\nu + |z \pm a|)^{2+k}} \langle z \rangle^{C} 
\end{equation}

\noindent \textbf{Step 5}. \emph{Estimate for the error term $\bar \Psi_0^{(1)}$ due to higher order linear terms for higher order corrections}. We estimate the right-hand side of \eqref{multisoliton:id:def-barPsi2}. For the first term, we have by \eqref{multisoliton:bd:T2} and \eqref{multisoliton:bd:estimate-bar-G>i} that
\begin{equation} \label{multisoliton:bd:error-correction-1}
\left| \frac{1}{\nu}\nabla \mathfrak T_{2,\pm}(\frac{\cdot\mp a}{\nu}).\bar G_{>1,\pm}(z) \right|\lesssim \frac{1}{\nu}\langle \frac{z\mp a}{\nu}\rangle^{-4}(\nu+|z\mp a|)^2(\nu+|z\pm a|)^{-1}\lesssim \frac{\nu^3}{(\nu+|z\mp a|)^2(\nu+|z\pm a|)}.
\end{equation}
Next, the estimates \eqref{multisoliton:bd:T2}, \eqref{multisoliton:bd:T3} and \eqref{multisoliton:bd:T4} imply
\begin{equation}  \label{multisoliton:bd:Ti}
 |\nabla^k \mathfrak T_{2+i}(y)|\lesssim \langle y \rangle^{-3+i-k}, \qquad  |\nabla^k \mathfrak V_{2+i}(y)|\lesssim \langle y \rangle^{-1+i-k}\langle \ln \langle y \rangle\rangle
\end{equation}
for $i=0,1,2$. The above estimate for $i=1$ is weaker than \eqref{multisoliton:bd:T3} but this will be harmless. For the second and third terms, we have by \eqref{multisoliton:bd:Ti}, \eqref{multisoliton:bd:estimate-bar-G>i} and $|\frac{\alpha-\beta}{\nu}|\lesssim 1$ that for $i=1,2$,
\begin{align}
\nonumber & \left| \pm \nu^i \frac{\alpha-\beta}{\nu}a.\nabla \mathfrak T_{i+2,\pm}(\frac{\cdot\mp a}{\nu}) -\nu^i \beta \Lambda \mathfrak T_{i+2,\pm} (\frac{\cdot\mp a}{\nu})  -\nu^{i-1}\nabla \mathfrak T_{i+2,\pm}(\frac{\cdot\mp a}{\nu}).\bar G_{>0,\pm}(z)\right|\\
 \label{multisoliton:bd:error-correction-2}&\lesssim  \nu^i \langle \frac{z\mp a}{\nu} \rangle^{-5+i}+\nu^i \langle \frac{z\mp a}{\nu}\rangle^{-4+i}+\nu^{i-1}\langle \frac{z\mp a}{\nu} \rangle^{-5+i}\frac{\nu+|z\mp a|}{\nu+|z\pm a|} \ \lesssim \frac{\nu^4 \langle z \rangle^2}{(\nu+|z\mp a|)^3(\nu+|z\pm a|)}.
\end{align}
For the fourth term, we have using \eqref{multisoliton:bd:Ti} for $i=0,1,2$, and $\langle \ln \langle y \rangle \rangle\lesssim \langle y \rangle^{\epsilon}$ for any $\epsilon>0$, that $|\nabla^{1+k} \mathfrak V_{i+2,\pm}(\frac{z\mp a}{\nu})|\lesssim \frac{\nu^{2-i-k-\epsilon}}{(\nu+|z\mp a|)^{2-i-\epsilon}}$. Hence by a Taylor expansion we deduce that for $z$ close to $\mp a$ we have $|\mathfrak V_{i+2,\pm}(\frac{z\mp a}{\nu})-\mathfrak V_{i+2,\pm}(\frac{\mp 2a}{\nu})|\lesssim \nu^{2-i-\epsilon}|z\mp a|$. Combining, we obtain
$$
|\mathfrak V_{i+2,\pm}(\frac{z\mp a}{\nu})-\mathfrak V_{i+2,\pm}(\frac{\mp 2a}{\nu})|\lesssim \frac{\nu^{2-i-\epsilon}}{(\nu+|z\mp a|)^{2-i-\epsilon}}|z\pm a|\langle z \rangle^{1-i}
$$
for $z\in \mathbb R^2$. Using this inequality, \eqref{multisoliton:bd:Ti} and $|U|\lesssim \langle y \rangle^{-4}$ we infer
\begin{align}
\nonumber  & \left|-\nu^{i-2}\nabla U(\frac{\cdot \pm a}{\nu}).\left( \nabla \mathfrak V_{i+2,\pm}(\frac{\cdot\mp a}{\nu})-\nabla \mathfrak V_{i+2,\pm}(\frac{\mp 2 a}{\nu})\right)+2\nu^{i-2}\mathfrak T_{i+2,\pm}(\frac{\cdot \mp a}{\nu})U(\frac{\cdot \pm a}{\nu})\right| \\
\nonumber & \lesssim  \nu^{i-2}\langle \frac{z\mp a}{\nu}\rangle^{-5}\frac{\nu^{2-i-\epsilon}}{(\nu+|z\mp a|)^{2-i-\epsilon}}|z\pm a|\langle z \rangle^{1-i}+\nu^{i-2}\langle \frac{z\mp a}{\nu}\rangle^{-3+i}\langle \frac{z\pm a}{\nu}\rangle^{-4}\\
\label{multisoliton:bd:error-correction-3} &\lesssim \frac{\nu^{5-\epsilon}}{(\nu+|z\mp a|)^3 (\nu+|z\pm a|)^4}\langle z \rangle^2
\end{align}
Combining \eqref{multisoliton:bd:error-correction-1}, \eqref{multisoliton:bd:error-correction-2} and \eqref{multisoliton:bd:error-correction-3} shows
\begin{align}
\nonumber |\bar \Psi_0^{(1)}|&   \lesssim  \sum_{\pm}\left( \frac{\nu^3}{(\nu+|z\mp a|)^2(\nu+|z\pm a|)}+ \frac{\nu^4 \langle z \rangle}{(\nu+|z\mp a|)^3(\nu+|z\pm a|)}+ \frac{\nu^5\langle z \rangle^2}{(\nu+|z\mp a|)^3(\nu+|z\pm a|^4)} \right)\\
\label{multisoliton:bd:estimate-barPsi01-pointwise} &\lesssim \frac{\nu^{3-\epsilon}\langle  z \rangle}{(\nu+|z-a|)^2(\nu+|z+a|)^2}.
\end{align}

\noindent \textbf{Step 6}. \emph{Estimate for the error term $\bar \Psi_0^{(2)}$ due to nonlinear terms generated by the correction}. We estimate all terms in \eqref{multisoliton:id:def-barPsi2}. By \eqref{multisoliton:id:decomposition-tildemathfrakT2} and Lemmas \ref{lem:T20hatT2}, \ref{lem:T22} and \ref{lem:profile_mathfrakTk} we have
\begin{align}
\label{multisoliton:bd:estimqte-tildemathfrakT2}& |\nabla^k \tilde{\mathfrak T}_{2,\pm}(y)|\lesssim |\nabla^k \mathfrak T_{2,\pm}(y)|+|\nabla^k \tilde{ T}_{2,\pm}(y)|\lesssim \langle y \rangle^{-3-k}+\langle y \rangle^{-2-k}\lesssim \langle y \rangle^{-2-k},\\
\label{multisoliton:bd:estimqte-tildemathfrakV2}& |\nabla^k \tilde{\mathfrak V}_{2,\pm}(y)|\lesssim |\nabla^k \mathfrak V_{2,\pm}(y)|+|\nabla^k \tilde{ V}_{2,\pm}(y)|\lesssim \langle y \rangle^{-3-k}+\langle y \rangle^{-k}\langle \ln \langle y \rangle \rangle\lesssim \langle y \rangle^{\epsilon-k},
\end{align}
for any $0<\epsilon\ll1$ arbitrarily small. By \eqref{multisoliton:id:decomposition-tildeXi>2}, \eqref{multisoliton:bd:tildephii>2} and Lemma \ref{lem:profile_mathfrakTk} using $\langle \ln \langle y \rangle \rangle\lesssim \langle y \rangle^\epsilon$ for any $0<\epsilon\ll1$ we have
\begin{align}
\nonumber |\nabla^k \tilde{\Xi}_{>2}(z)| & \lesssim \nu^2 |\nabla^k \tilde \phi_{0,>2}(z)|+ \sum_\pm \sum_{j=3}^4 \nu^{j-2-k}|(\nabla^k\mathfrak T_{j,\pm})(\frac{z\mp a}{\nu})|\\
\label{multisoliton:bd:estimqte-tildeXi>2}&\qquad \lesssim \sum_\pm \frac{\nu^2\langle z \rangle^{C_k}}{(\nu+|z\mp a|)^{1+k}}+ \sum_\pm \sum_{j=3}^4 \frac{\nu^{3-\epsilon}}{(\nu+|z\mp a|)^{5-\epsilon+k-j}} \lesssim \sum_\pm \frac{\nu^2\langle z \rangle^{C_k}}{(\nu+|z\mp a|)^{1+k}}.
\end{align}
By \eqref{multisoliton:id:decomposition-tildeV>2}, \eqref{multisoliton:bd:tildeV>2}, Lemma \ref{lem:profile_mathfrakTk} with $\langle \ln \langle y \rangle \rangle\lesssim \langle y \rangle^\epsilon$ and \eqref{multisoliton:bd:estimqte-tildemathfrakV2} we have for $k\geq 1$,
\begin{align}
\nonumber |\nabla^k \tilde{\mathfrak V}_{>2}(z)| & \lesssim \nu^2 |\nabla^k \tilde V_{0,>2}|+|\nabla^k(\chi^* \sum_\pm \sum_{j=3}^4  \nu^{j}\mathfrak V_{j,\pm}(\frac{\cdot \mp a}{\nu}))|+|\nabla^k((\chi^*-1) \sum_\pm \nu^2 \tilde{\mathfrak V}_{2,\pm}(\frac{\cdot \mp a}{\nu}))|\\
\label{multisoliton:bd:estimqte-tildemathfrakV>2} & \lesssim \sum_{\pm}\frac{\nu^2 |\ln \nu|}{(\nu+|z\pm a|)^{k-1}}+ \sum_\pm \sum_{j=3}^4 \frac{\nu^{3-\epsilon}}{(\nu+|z\mp a|)^{3-\epsilon+k-j}} +\nu^{2-\epsilon} \ \lesssim  \sum_{\pm}\frac{\nu^{2-\epsilon}}{(\nu+|z\pm a|)^{k-1}}.
\end{align}
Using \eqref{multisoliton:bd:estimqte-tildemathfrakT2}, \eqref{multisoliton:bd:estimqte-tildemathfrakV2}, \eqref{multisoliton:bd:estimqte-tildemathfrakV>2} and \eqref{multisoliton:bd:PhiXi-3} we infer that for the first term in \eqref{multisoliton:id:def-barPsi2}:
\begin{align}
\nonumber & |\frac{1}{\nu}\nabla  \tilde{\mathfrak T}_{2,\pm}(\frac{\cdot \mp a}{\nu}).(\nu \nabla \tilde{\mathfrak V}_{2,\mp}(\frac{\cdot\pm a}{\nu})+\nabla \tilde{\mathfrak V}_{>2}+\nabla \bar{\mathfrak V})|\\
\label{multisoliton:bd:estimqte-nonlinear-particular} &\qquad \lesssim\frac{\nu^2}{(\nu+|z\mp a|)^3}\left(\frac{\nu^{2-\epsilon}}{\nu+|z\pm a|}+\nu^{2-\epsilon}+\nu^2\right)\lesssim \frac{\nu^{4-\epsilon}}{(\nu+|z\mp a|)^3(\nu+|z\pm a|)}
\end{align}
All other terms in \eqref{multisoliton:id:def-barPsi2} can be estimated by similar computations to the one establishing \eqref{multisoliton:bd:estimqte-nonlinear-particular}, using \eqref{multisoliton:bd:estimqte-tildemathfrakT2}, \eqref{multisoliton:bd:estimqte-tildemathfrakV2}, \eqref{multisoliton:bd:estimqte-tildeXi>2}, \eqref{multisoliton:bd:estimqte-tildemathfrakV>2} and \eqref{multisoliton:bd:PhiXi-3}, producing eventually the following estimate:
\begin{align} 
\label{multisoliton:bd:estimate-barPsi02-pointwise} |\bar \Psi_0^{(2)}| &\lesssim \sum_\pm  \frac{\nu^{4-\epsilon}\langle z \rangle^C}{(\nu+|z\mp a|)^3}.
\end{align}

\medskip

\noindent \textbf{Step 7}. \emph{Proof of \eqref{multisoliton:bd:estimate-Psi0-pointwise}, \eqref{est:Psi_L2omega} and \eqref{est:Psi0mass}}. The bound \eqref{multisoliton:bd:estimate-Psi0-pointwise} follows directly from injecting \eqref{multisoliton:bd:estimate-barPsi00-pointwise}, \eqref{multisoliton:bd:estimate-barPsi01-pointwise} and \eqref{multisoliton:bd:estimate-barPsi02-pointwise} in \eqref{multisoliton:id:Psi0-decomposition-2nd-ansatz}. Then, the bounds \eqref{est:Psi_L2omega} and \eqref{est:Psi0mass} follow directly from \eqref{multisoliton:bd:estimate-Psi0-pointwise}.

\medskip

\noindent \textbf{Step 8}. \emph{Proof of \eqref{multisoliton:id:nupartialnu-mathfrakU-identity}}. We have by \eqref{multisoliton:id:decomposition-Xi-1} and \eqref{def:phii}:
$$
\mathfrak U=U_{1+2,\nu}+8\left(\nu^2 \tilde \phi_0(z)+ \sum_\pm \left(\frac{1}{16\nu^2}\Lambda U\mp\frac{L_0}{\nu}\partial_{x_1}U\right)(\frac{z\mp a}{\nu})(1-\chi_{\pm a,\zeta_*}(z)) \right)+\tilde \Xi.
$$
Hence we have the identity
\begin{align*}
\nu\partial_\nu \mathfrak U & =-(\Lambda U)_{1,\nu}-(\Lambda U)_{2,\nu}+16\nu^2 \tilde \phi_0\\
& +\frac{8}{\nu^2}\left( \sum_\pm \left(\frac{-1}{16}\Lambda^2 U\pm L_0\nu (y.\nabla+1)( \partial_{x_1}U)\right)(\frac{z\mp a}{\nu})(1-\chi_{\pm a,\zeta_*}(z)) \right) +8\nu^3\partial_\nu \tilde \phi_0(z)+\nu\partial_\nu \tilde \Xi.
\end{align*}
We have by \eqref{def:phii} that
\begin{align*}
& -(\Lambda U)_{1,\nu}-(\Lambda U)_{2,\nu}+16\nu^2 \tilde \phi_0 \\
&= 16 \phi_0-\nu^2(L_0,0).(\nabla U_{1,\nu}-\nabla U_{2,\nu})+\sum_\pm \left(\frac{1}{\nu^2}\Lambda U \mp \frac{16 L_0}{\nu}\partial_{x_1}U\right)(\frac{z\mp a}{\nu})(\chi_{\pm a,\zeta_*}-1).
\end{align*}
Therefore,
\begin{align}
\label{multisoliton:id:nupartialnumathfrakU-inter1}\nu\partial_\nu \mathfrak U & =16 \phi_0-\nu^2(L_0,0).(\nabla U_{1,\nu}-\nabla U_{2,\nu})+8\nu^3\partial_\nu \tilde \phi_0(z)+\nu\partial_\nu \tilde \Xi\\
\nonumber & \qquad + \sum_\pm \left(\frac{1}{\nu^2}(\Lambda^2 U+\frac 12 \Lambda^2 U)\mp \frac{8L_0}{\nu} (y.\nabla+3)( \partial_{x_1}U)\right)(\frac{z\mp a}{\nu}) (\chi_{\pm a,\zeta_*}(z)-1) .
\end{align}
We have using \eqref{multisoliton:id:def-tildeXi} and Lemma \ref{lem:profile_mathfrakTk} with $\langle \log \langle y \rangle \rangle \lesssim \langle y \rangle^{\epsilon}$ for any $0<\epsilon\ll1$ arbitrarily small that
\begin{align}
\nonumber |\nabla^k \nu \partial_\nu \tilde \Xi| & \lesssim \sum_\pm \sum_{i=0}^2 \nu^i\left|\nabla^k\left( (i\mathfrak T_{i+2,\pm}-y.\nabla \mathfrak T_{i+2,\pm})(\frac{z\mp a}{\nu})\right) \right|\\
\label{multisoliton:id:nupartialnumathfrakU-inter2}& \lesssim \sum_\pm \frac{\nu^{3-\epsilon}}{(\nu+|z\mp a|)^{3-\epsilon-i+k}}\lesssim \sum_\pm \frac{\nu^{3-\epsilon}\langle z \rangle^{C_k}}{(\nu+|z\mp a|)^{3-\epsilon+k}}.
\end{align}
We have using the tail cancellation $\Lambda U(y)+\frac 12 \Lambda U(y)=\Oc(\langle y \rangle^{-6})$, that $\partial_{x_1}U(y)=\Oc(\langle y \rangle^{-5})$ and that $\chi_{\pm a,\zeta_*}-1$ is supported inside $\{|z\mp a|\geq \zeta_*\}$ that
\begin{equation} \label{multisoliton:id:nupartialnumathfrakU-inter3}
\left|\nabla^k \left( \sum_\pm \left(\frac{1}{\nu^2}(\Lambda^2 U+\frac 12 \Lambda^2 U)\mp \frac{8L_0}{\nu} (y.\nabla+3)( \partial_{x_1}U)\right)(\frac{z\mp a}{\nu}) (\chi_{\pm a,\zeta_*}(z)-1) \right)\right|\lesssim \nu^4 \langle z \rangle^{-5-k}.
\end{equation}
Injecting the second inequality in \eqref{est:pointwise_phiitil}, \eqref{multisoliton:id:nupartialnumathfrakU-inter2} and \eqref{multisoliton:id:nupartialnumathfrakU-inter3} in \eqref{multisoliton:id:nupartialnumathfrakU-inter1} shows the desired identity \eqref{multisoliton:id:nupartialnu-mathfrakU-identity}.

\end{proof}

\section{Approximate blowup solution and bootstrap regime} \label{sec:bootstrap}

\subsection{Full approximate blowup solution}\label{subsec:full-approx-sol}
We now define the approximate solution to the self-similar equation \eqref{eq:wztauIntro} in the sense that the generated error term is sufficiently small in $L^2_{\omega_\nu}$ after extracting the leading order terms. It happens that the leading order terms are linked to the approximate eigenfunctions constructed in Proposition \ref{prop:eigen}, from which we can derive appropriate modulation equations driving the law of blowup.  

\begin{proposition}[Approximate blowup solution] \label{prop:E} 
Assume that $(\nu, a, \alpha)$ are $\Cc^1$ maps 
$$(\nu, a, \alpha): [\tau_0, \tau_1) \mapsto (0, \nu^*] \times \big[2 - a^*, 2 + a^*\big] \times (0, \alpha^*], $$
for $0 < \nu^*, a^*, \alpha^* \ll 1$ and the a priori bounds: 
\begin{align} \label{est:aprioribounds_nu_a_alpha}
\Big| \frac{\nu_\tau}{\nu} \Big| \lesssim \frac{1}{|\ln \nu|}, \quad \quad |a_\tau| \lesssim \nu^2,  \quad |\alpha| \lesssim \nu^2, \quad |\alpha_\tau| \lesssim \frac{\nu^2}{|\ln \nu|}.  
\end{align}
We define an approximate solution to \eqref{eq:wztauIntro} by
\begin{equation}\label{def:Wap}
W[\nu, a, \alpha] = \Uk[\nu,a]+  \alpha \Big[ \phi_{1} -   \phi_{0} - \frac{L_1 - L_0}{\nu} \big(\pa_1 U_{1,\nu} \chi_{a,\zeta_*, a} - \pa_1 U_{2,\nu} \chi_{-a,\zeta_*} \big) \Big]  = U_{1+2, \nu} + \tilde{W}, 
\end{equation}
where $\Uk = U_{1+2, \nu} + \Xi$ is the improved stationary solution introduced in \eqref{def:ImprovedSS},  $\phi_{0}, \phi_{1}$ are the two approximate eigenfunctions described in Proposition \ref{prop:eigen} and 
\begin{equation}\label{def:Wtilde}
\tilde{W} = \Xi + \alpha H, \quad H =  \phi_{1} -   \phi_{0} - \frac{L_1 - L_0}{\nu} \big(\pa_1 U_{1,\nu} \chi_{a,\zeta_*} - \pa_1 U_{2,\nu} \chi_{-a,\zeta_*} \big).
\end{equation}
The error generated by $W$ to \eqref{eq:wztauIntro} is defined by 
\begin{align} \label{def:Eapp}
E = - \partial_\tau W  + \Delta W  - \nabla \cdot (W  \nabla \Phi_{U_{1+2,\nu} + \tilde{W}\chi^*}) - \beta \Lambda W ,
\end{align}
that can be decomposed as 
\begin{align}
E(z,\tau)&= \textup{Mod}_0(\tau) \; \phi_0(z) + \textup{Mod}_1(\tau) \; \phi_1(z) +\textup{Mod}_a(\tau) \cdot \big(\nabla U_{1,\nu} - \nabla U_{2,\nu} \big)  +  \Psi(z, \tau), \label{def:EztauMod}
\end{align}
where 
\begin{align}
\textup{Mod}_0 & = -16 \nu^2 \Big(\frac{\nu_\tau}{\nu} - \frac{\lambda_0}{2}\Big) + \alpha_\tau - \lambda_0 \alpha, \label{def:Mod0}\\
\textup{Mod}_1 & = -\alpha_\tau + \lambda_1 \alpha, \qquad \nu \textup{Mod}_a = a_\tau + \frac{8a}{\nu^2 + 4|a|^2} - \beta a + \hat \mu, \label{def:Mod1_Moda}
\end{align}
where $\lambda_0$ and $\lambda_1$ are the two approximate eigenvalues given in \ref{prop:eigen},  $L_0, L_1$ in \eqref{def:phii} and $\hat \mu = \Oc(\nu^2)$, and $\Psi$ is the remaining small term satisfying
\begin{align}\label{est:PsiinL2omega}
\sum \int_{|z \pm a| \leq \zeta_*} \Psi \lesssim \frac{\nu^2}{|\ln \nu|}, \qquad  \| \Psi\|_{\omega_\nu} +  \Big\| \big(\nu + \sum|z \pm a| \big) \nabla\Psi\Big\|_{\omega_\nu} \lesssim \frac{\nu^2}{|\ln \nu|^2}.
\end{align}
\end{proposition}

\begin{proof} We first derive the expression of the generated error $\Psi$ by plugging $W  = U_{1+2, \nu} + \Xi + H = \Uk + H$ into the expression \eqref{def:Eapp} of $E$ and using the definition \eqref{id:corrected-stationary-equation} of $E_0$ and the definition \eqref{def:LszNottilde} of $\tilde \Ls^z$ to write 
\begin{align*}
E &= -\pa_\tau W + \nabla.(\nabla W - W \nabla \Phi_{U_{1 + 2, \nu} + \tilde W \chi^*}) - \beta \Lambda W \\
& =- \pa_\tau \big(\Uk + \alpha H \big) + E_0 +  \alpha \tilde \Ls^z H - \nabla . \Big( \alpha H \nabla \Phi_{(\Xi + \alpha H) \chi^*} +  (\Xi + \alpha H)) \nabla \Phi_{\alpha H \chi^*} \Big).
\end{align*}
We recall from \eqref{multisoliton:id:nupartialnu-mathfrakU-identity} the identity
\begin{align*}
    -\pa_\tau \Uk &= -\frac{\nu_\tau}{\nu} \nu \pa_\nu \Uk - a_\tau \pa_a \Uk\\
    & = -\frac{\nu_\tau}{\nu} \Big[16 \nu^2 \phi_0 - \nu^2 (L_0,0) \big(\nabla U_{1, \nu} - \nabla U_{2,\nu}\big) + \tilde \Uk \Big]  + \frac{a_\tau}{\nu}.\big(\nabla U_{1, \nu} - \nabla U_{2, \nu}\big) - a_\tau \pa_a  \Xi. 
\end{align*}
and from the expansion \eqref{def:phii} of $\phi_i$ the relation $H = \tilde \phi_1 - \tilde \phi_0$ and  the identity
\begin{align*}
    -\pa_\tau (\alpha H) &= -\alpha_\tau H  - \alpha \pa_\tau H \\
    & = -\alpha_\tau \Big[ \phi_1- \phi_0  - \frac{L_1 - L_0}{\nu}\big(\pa_1 U_{1, \nu} \chi_{a, \zeta_*} - \pa_1 U_{2, \nu} \chi_{-a, \zeta_*} \big)\Big] - \alpha \Big( \frac{\nu_\tau}{\nu} \nu \pa_\nu + a_\tau \pa_a\Big) (\tilde \phi_1 - \tilde \phi_0).
\end{align*}
From the expansion \eqref{id:corrected-stationary-equation} of $E_0$ and the approximate eigenproblem \eqref{eigen:id:Ri}, we write 
\begin{align*}
    E_0 + \alpha \tilde \Ls^z H &= 8 \lambda_0 \nu^2 \phi_0 + \mu'. (\nabla U_{1, \nu} - \nabla U_{2, \nu}) + \Psi_0 + \alpha \big( \lambda_1 \phi_1 - \lambda_0 \phi_0 + R_1 - R_0\big) \\
    & \quad - \frac{\alpha}{\nu}(L_1 - L_0) \tilde \Ls^z \big( \pa_1 U_{1, \nu} \chi_{a, \zeta_*} - \pa_1 U_{2, \nu} \chi_{-a, \zeta_*} \big). 
\end{align*}
Plugging all these identities together with $H = \tilde \phi_1 - \tilde \phi_0$, we then obtain 
\begin{align*}
    E &= (\alpha_\tau - 16 \nu_\tau \nu + 8\lambda_0 \nu^2 - \alpha \lambda_0)\phi_0 - (\alpha_\tau - \alpha \lambda_1) \phi_1 \\
    & \quad + \Big( \frac{a_\tau}{\nu} +\nu_\tau \nu (L_0, 0) + \frac{\alpha_\tau}{\nu} (L_1 - L_0, 0) + \mu' \Big).\big(\nabla U_{1, \nu} - \nabla U_{2, \nu} \big)\\
    & \quad  - \frac{\nu_\tau}{\nu} \tilde \Uk - a_\tau \pa_a \Xi  - \alpha \Big( \frac{\nu_\tau}{\nu} \nu \pa_\nu + a_\tau \pa_a\Big) (\tilde \phi_1 - \tilde \phi_0) + \Psi_0 + \alpha (R_1 - R_0) \\
    & \quad - \frac{\alpha_\tau}{\nu}(L_1 - L_0) \big(\pa_1 U_{1, \nu} (1 - \chi_{a, \zeta_*}) - \pa_1 U_{2, \nu} (1 - \chi_{-a, \zeta_*}) \big)\\
    & \quad - \frac{\alpha}{\nu}(L_1 - L_0) \tilde \Ls^z \big( \pa_1 U_{1, \nu} \chi_{a, \zeta_*} - \pa_1 U_{2, \nu} \chi_{-a, \zeta_*} \big)\\
    & \quad - \nabla . \Big( \alpha (\tilde \phi_1 - \tilde \phi_0) \nabla \Phi_{(\Xi + \alpha (\tilde \phi_1 - \tilde \phi_0)) \chi^*} +  (\Xi + \alpha (\tilde \phi_1 - \tilde \phi_0)) \nabla \Phi_{\alpha (\tilde \phi_1 - \tilde \phi_0) \chi^*} \Big)\\
    & = \textup{Mod}_0 \phi_0 + \textup{Mod}_1 \phi_1 + \textup{Mod}_a.\big(\nabla U_{1, \nu} - \nabla U_{2, \nu} \big) + \Psi.
\end{align*}
Here, $\textup{Mod}_0$, $\textup{Mod}_1$ and $\textup{Mod}_a$ are given by \eqref{def:Mod0} and \eqref{def:Mod1_Moda} with the precise definition thanks to \eqref{def:mumuprime} and the estimates $|L_0| \lesssim 1, |L_1| \lesssim |\ln \nu|$ (see Proposition \ref{prop:eigen})
$$\hat \mu = \tilde \mu' + \nu_\tau \nu (L_0, 0) + \frac{\alpha_\tau}{\nu} (L_1 - L_0, 0) = \Oc(\nu^2),$$
and 
\begin{align*}
    \Psi &= - \frac{\nu_\tau}{\nu} \tilde \Uk - a_\tau \pa_a \Xi  - \alpha \Big( \frac{\nu_\tau}{\nu} \nu \pa_\nu + a_\tau \pa_a\Big) (\tilde \phi_1 - \tilde \phi_0) + \Psi_0 + \alpha (R_1 - R_0) \\
    & \quad - \frac{\alpha_\tau}{\nu}(L_1 - L_0) \big(\pa_1 U_{1, \nu} (1 - \chi_{a, \zeta_*}) - \pa_1 U_{2, \nu} (1 - \chi_{-a, \zeta_*}) \big)\\
    & \quad - \frac{\alpha}{\nu}(L_1 - L_0) \tilde \Ls^z \big( \pa_1 U_{1, \nu} \chi_{a, \zeta_*} - \pa_1 U_{2, \nu} \chi_{-a, \zeta_*} \big)\\
    & \quad - \nabla . \Big( \alpha (\tilde \phi_1 - \tilde \phi_0) \nabla \Phi_{(\Xi + \alpha (\tilde \phi_1 - \tilde \phi_0)) \chi^*} +  (\Xi + \alpha (\tilde \phi_1 - \tilde \phi_0)) \nabla \Phi_{\alpha (\tilde \phi_1 - \tilde \phi_0) \chi^*} \Big).
\end{align*}
We now estimate all terms involved in the expression of $\Psi$. From the pointwise estimate \eqref{est:Ukpointwise} and the a prior estimate $|\nu_\tau/\nu| \lesssim 1/|\ln \nu|$, we have 
\begin{align*}
    \Big| \frac{\nu_\tau}{\nu} \tilde \Uk \Big| \lesssim \sum_\pm \frac{\nu^2}{(\nu + |z\pm a|)^2 |\ln \nu|} \Big( \frac{\nu^{\delta_0}}{(\nu + |z \pm a|)^{\delta_0}} + \frac{1}{|\ln \nu|^2} \Big),
\end{align*}
from which we estimate 
\begin{align*}
    \Big|\int_{|z \pm a| \leq \zeta_*} \frac{\nu_\tau}{\nu} \tilde \Uk (z) dz \Big| \lesssim \frac{\nu^2}{|\ln \nu|}\int_{|y| \leq \zeta_*/\nu} \frac{1}{\langle y \rangle^2} \Big( \frac{1}{\langle y \rangle^{\delta_0}} + \frac{1}{|\ln \nu|^2} \Big) dy \lesssim \frac{\nu^2}{|\ln \nu|},
\end{align*}
and 
\begin{align*}
    \Big\| \frac{\nu_\tau}\nu \tilde \Uk\Big\|_{\omega_\nu}^2 &\lesssim \frac{1}{|\ln \nu|^2} \int \sum_\pm \frac{\nu^4}{(\nu + |z \pm a|)^4} \Big( \frac{\nu^{2\delta_0}}{(\nu + |z \pm a|)^{2\delta_0}} + \frac{1}{|\ln \nu|^4} \Big) (\nu + |z \pm a|)^4 e^{-\beta |z|^2} dz\\
    & \lesssim \frac{\nu^6}{|\ln \nu|^2} \int \Big( \frac{1}{\langle y \rangle^{2 \delta_0}} + \frac{1}{|\ln \nu|^4} e^{-\beta \nu^2 |y|^2} \Big) dy \lesssim \frac{\nu^6}{|\ln \nu|^2}\big( \nu^{-2 + 2\delta_0} + \frac{\nu^{-2}}{|\ln \nu|^4} \big) \lesssim \frac{\nu^4}{|\ln \nu|^6}.
\end{align*}
From \eqref{est:pointwise_Xi_dk} and the a prior estimate $|a_\tau| \lesssim \nu^2$, we have 
$$\big|a_\tau \pa_a \Xi \big| \lesssim \sum_\pm \frac{\nu^4}{(\nu + |z \pm a|)^3} \langle z \rangle^{C_0},$$
which gives
\begin{align*}
    \Big|\int_{|z \pm a| \leq \zeta_*} a_\tau \pa_a \Xi dz \Big| \lesssim \nu^3 \int_{|y| \leq \frac{\zeta_*}{\nu}} \frac{1}{\langle y \rangle^3} dy \lesssim \nu^3,
\end{align*}
and 
\begin{align*}
    \Big\| a_\tau \pa_a\Xi \Big\|^2_{\omega_\nu} &\lesssim \nu^4 \int \sum_\pm \frac{\nu^4}{(\nu + |z \pm a|)^6} \langle z \rangle^{2C_0} (\nu + |z \pm a|)^4 e^{-\beta |z|^2} dz \\
    &\lesssim \nu^8 \int \langle y \rangle^{-2} \langle \nu y \rangle^{2 C_0} e^{-\beta \nu^2|y|^2} dy \lesssim \nu^8 |\ln \nu|. 
\end{align*}
From the pointwise estimates \eqref{est:pointwise_phiitil} and \eqref{est:pointwise_phiitil_da}, and the a prior bounds \eqref{est:aprioribounds_nu_a_alpha}, we have 
\begin{align*}
\Big| \alpha \frac{\nu_\tau}{\nu} \nu \pa_\nu (\tilde \phi_1 - \tilde \phi_0) \Big| \lesssim \frac{\nu^2}{|\ln \nu|} \sum_\pm \frac{\langle z \rangle^{C_0}}{(\nu + |z \pm a|)^2}\Big( \frac{\nu^{\delta_0}}{(\nu + |z \pm a|)^{\delta_0}} + \frac{1}{|\ln \nu|^2} \Big),
\end{align*}
which is the same pointwise estimate as for $\nu_\tau/\nu \tilde \Uk$, hence, 
$$\Big| \int_{|z \pm a| \leq \zeta_*} \alpha \frac{\nu_\tau}{\nu} \nu \pa_\nu (\tilde \phi_1 - \tilde \phi_0) dz \Big| \lesssim \frac{\nu^2}{|\ln \nu|}, \quad \big\|\alpha \frac{\nu_\tau}{\nu} \nu \pa_\nu (\tilde \phi_1 - \tilde \phi_0) \big\|^2_{\omega_\nu} \lesssim \frac{\nu^4}{|\ln \nu|^6}.$$
Similarly, we have 
\begin{align*}
\Big| \alpha a_\tau \pa_a (\tilde \phi_1 - \tilde \phi_0) \Big| \lesssim \nu^4 \sum_\pm \frac{\langle z \rangle^{C_0}}{(\nu + |z \pm a|)^3},
\end{align*}
which is the same pointwise estimate as $a_\tau \pa_a \Xi$, hence, 
$$\Big| \int_{|z \pm a| \leq \zeta_*} \alpha a_\tau \pa_a (\tilde \phi_1 - \tilde \phi_0) dz \Big| \lesssim \nu^3, \quad \big\|\alpha a_\tau \pa_a (\tilde \phi_1 - \tilde \phi_0) \big\|^2_{\omega_\nu} \lesssim \nu^8 |\ln \nu|.$$
From \eqref{est:Psi_L2omega} and \eqref{est:Psi0mass}, we have already obtained the estimates 
$$ \Big| \int_{|z \pm a|\leq \zeta_*} \Psi_0 dz \Big| \lesssim \frac{\nu^2}{|\ln \nu|}, \quad  \| \Psi_0\|_{\omega_\nu} \lesssim \frac{\nu^2}{|\ln \nu|^2}.$$
From \eqref{est:pointwise_Ri} and $|\alpha| \lesssim \nu^2$, we have already obtained the partial mass estimate 
$$\Big| \int \alpha (R_1 - R_0) \Big| \lesssim \frac{\nu^2}{|\ln \nu|}.$$
Using the pointwise estimate \eqref{est:pointwise_Ri}, we estimate
$$\| \alpha (R_1 - R_0)\|^2_{\omega_\nu} \lesssim \frac{\nu^4}{|\ln \nu|^4}\int \sum_\pm \frac{\langle z \rangle^{2C_0}}{(\nu + |z \pm a|)^4} (\nu + |z\pm a|)^4 e^{-\beta |z|^2} dz \lesssim \frac{\nu^4}{|\ln \nu|^4}.$$
Using the pointwise estimates \eqref{est:pointwise_Xi}, \eqref{est:pointwise_phiitil}, \eqref{est:pointwise_Phiphiitil}, and the a priori bound $|\alpha| \lesssim \nu^2$, we get 
\begin{align*}
&\|- \nabla . \Big( \alpha (\tilde \phi_1 - \tilde \phi_0) \nabla \Phi_{(\Xi + \alpha (\tilde \phi_1 - \tilde \phi_0)) \chi^*} +  (\Xi + \alpha (\tilde \phi_1 - \tilde \phi_0)) \nabla \Phi_{\alpha (\tilde \phi_1 - \tilde \phi_0) \chi^*} \Big)\|_{L^2_{\omega_\nu}} \\
& \lesssim \nu^4 \left\| \sum \frac{\langle z \rangle^C}{(\nu + |z \pm a|)^4}\right\|_{L^2_{\omega_\nu}} + \nu^3 \left\| \sum  \frac{\langle z \rangle^C}{(\nu + |z \pm a|)^3}\right\|_{L^2_{\omega_\nu}} \lesssim \nu^3 +  \nu^3 \sqrt{|\ln \nu|}.
\end{align*}
Using the divergence structure and the pointwise estimates \eqref{est:pointwise_phiitil}, \eqref{est:pointwise_Phiphiitil}, \eqref{est:pointwise_Xi},
\begin{align*}
&\Big|\int_{|z \pm a| \lesssim \zeta_*} - \nabla . \Big( \alpha (\tilde \phi_1 - \tilde \phi_0) \nabla \Phi_{(\Xi + \alpha (\tilde \phi_1 - \tilde \phi_0)) \chi^*} +  (\Xi + \alpha (\tilde \phi_1 - \tilde \phi_0)) \nabla \Phi_{\alpha (\tilde \phi_1 - \tilde \phi_0) \chi^*} \Big)\Big| \\
& \quad \lesssim \int_{|z \pm a| \sim \zeta_*} \Big| \alpha (\tilde \phi_1 - \tilde \phi_0) \nabla \Phi_{(\Xi + \alpha (\tilde \phi_1 - \tilde \phi_0)) \chi^*} +  (\Xi + \alpha (\tilde \phi_1 - \tilde \phi_0)) \nabla \Phi_{\alpha (\tilde \phi_1 - \tilde \phi_0) \chi^*} \Big| \lesssim \nu^4 |\ln \nu|.
\end{align*}
Collecting all the above estimates yields the desired estimates \eqref{est:PsiinL2omega}. This ends the proof of Proposition \ref{prop:E}. 
\end{proof}

\subsection{Formulation of the problem and bootstrap regime}\label{subsec:bootstrap-description}
We formulate the problem of constructing solutions to \eqref{eq:KS2d} that blow up in finite time by a collision of two single-solitons. In terms of the self-similarity variables \eqref{eq:wztauIntro}, we introduce the first linearization
\begin{align}
w(z, \tau) = U_{1+2, \nu}(z) + \tilde \vep(z,\tau),\label{def:AppSol-2nddecomp}
\end{align}
where $\tilde \vep$ solves 
\begin{equation}\label{eq:veptil}
\pa_\tau \tilde \vep = \Ls^z\tilde \vep - \nabla \cdot(\tilde \vep \nabla \Phi_{\tilde \vep}) + \tilde E,
\end{equation}
with
\begin{equation}\label{def:Etil}
\tilde E = -(\pa_\tau + \beta \Lambda)  U_{1+2, \nu} - \nabla \cdot (U_{1, \nu}\nabla \Phi_{U_{2,\nu}} + U_{2, \nu}\nabla \Phi_{U_{1, \nu}}).  
\end{equation}
We use this equation to estimate the outer part $\tilde{\vep}(1 - \chi^*)$ and the contribution of the Poision field $\nabla \Phi_{\tilde \vep (1 - \chi^*)}$, namely we decompose 
\begin{equation}
\tilde{\vep} = \tilde \vep \chi^* + \vep (1 - \chi^*), \quad  \nabla \Phi_{\tilde{\vep}} = \nabla \Phi_{\tilde{\vep}\chi^*} + \nabla \Phi_{\tilde \vep (1 - \chi^*)}.
\end{equation}
For the inner part, we use the refined approximate solution $W$ introduced in Proposition \ref{prop:E} to introduce the second linearization
\begin{equation}
w(z, \tau) = W[\nu, a, \alpha](z) + \vep(z, \tau) = U_{1+2, \nu}(z) +  \tilde{W}[\nu, a, \alpha](z) + \vep(z, \tau), \label{def:AppSol}
\end{equation} 
where $\nu, a, \alpha$ are parameter functions to be determined, $W$ and $\tilde W$ are defined by \eqref{def:Wap} and \eqref{def:Wtilde}.  We thus have the relation
\begin{equation*}
\tilde{\vep} = \tilde{W} + \vep,
\end{equation*}
where $\vep$ solves the linearized problem
\begin{equation}\label{eq:vep}
\pa_\tau \vep = \tilde \Ls^z\vep  - L(\vep)  - NL(\vep) + E, 
\end{equation}
where $\tilde \Ls^z$ is the linearized operator around $U_{1+2, \nu}$ with a cut-off of the Poison field introduced in \eqref{def:L12ztilde},  $L(\vep)$ is the small linear term
\begin{equation}
L \vep =  \nabla \cdot\Big(\vep \nabla \Phi_{\tilde W \chi^*} + \tilde W  \nabla \Phi_{\vep \chi^*} \Big) + \nabla \cdot \big( \tilde W\nabla \Phi_{\tilde \vep (1 - \chi^*)} \big), \label{def:Lep}
\end{equation}
and $NL(\vep)$ is the nonlinear term
\begin{equation}\label{def:NLvep}
NL(\vep) = \nabla \cdot\big( \vep  \nabla \Phi_{\vep \chi^*} + \vep \nabla \Phi_{\tilde \vep (1 - \chi^*)} \big),
\end{equation}
and $E$ is defined in \eqref{def:Eapp}.

Our aim is to construct a global in time solution $\vep$ to \eqref{eq:vep} though a robust energy estimate technique. Recall from Proposition \ref{Prop:coercive-Lz-global} that the localized linear operator $\tilde \Ls^z$ is coercive under the adapted scalar product $\langle \cdot, \cdot \rangle_\ast$ introduced in  \eqref{def:scalar12_intro}. It's natural to control the remainder through the adapted norm
\begin{equation}
\| \vep\|_\ast^2  = \langle \vep, \vep \rangle_\ast = \int_{\Rb^2} \vep^2 \omega_\nu dz - c_\ast \nu^2 \int_{\Rb^2} \chi^* \vep \Phi_{\chi^* \vep} dz, \quad c_\ast = 32e^{-1}.
\end{equation}
It is proved to be coercive once we impose the orthogonality conditions:
\begin{equation}\label{eq:orthog}
\langle \vep, 1 \rangle_{L^2_{loc, \pm}} = \langle \vep, \Lambda U_{i, \nu} \rangle_{L^2_{loc, \pm}} = \langle \vep, \pa_j U_{i, \nu} \rangle_{L^2_{loc, \pm}} = 0, \quad i,j = 1,2, 
\end{equation}
where we use the notation 
$$\langle f,g \rangle_{L^2_{loc, \pm}} = \int_{|z \mp a| < \eta} f g dz, \quad \textup{for some fixed constant $0 < \eta \ll 1$.}$$
The orthogonality condition \eqref{eq:orthog} ensures the coercivity estimate (see Proposition \ref{pr:scalar-product-interior})
\begin{equation} \label{est:coercivity_vep}
\|\vep \|_\ast^2 = \langle \vep, \vep \rangle_\ast \gtrsim  \|\vep\|^2_{L^2_{\omega_\nu }}.
\end{equation}

\noindent In order to close the modulation equations for the parameters $\nu, a, \alpha$ and to close the nonlinear analysis, we need an estimate of $\nabla \vep$ as well. To do so, we split the control of $\vep$ into three parts: \\

\noindent - the inner zone $B_{\zeta_*}(a) \cup B_{\zeta_*}(-a)$ with a fixed constant $0 < \zeta_* \ll 1$: we introduce 
\begin{equation}\label{def:qpminn}
 q^\inn_\pm(y_\pm, \tau) = \nu^2 \vep(z, \tau) \chi_{\pm a, \zeta_*}(z), \qquad y_\pm = \frac{z \mp a}{\nu},
 \end{equation}
 and the $H^1$ inner norm
\begin{equation}\label{def:H1innnorm}
\|q(\tau)\|^2_{H^1_\inn} = \sum_\pm \Big(\langle - \Ls_0 q^\inn_\pm, \Ms_0 q^\inn_\pm \rangle + \frac{\nu^2}{\zeta_*^2}\langle q^\inn_\pm, \Ms_0 q^\inn_\pm \rangle \Big) \gtrsim \sum_\pm \Big(\int \frac{|\nabla q^\inn_\pm|^2}{U(y_\pm)}  + \frac{\nu^2}{\zeta_*^2}\int \frac{|q^\inn_\pm|^2}{U(y_\pm)} \Big),
\end{equation} 
where the coercivity is archived by the orthogonality condition \eqref{eq:orthog} (see Lemma \ref{lem:coercivity-one-bubble-H1-orthogonality}). We also need a round bound of $L^\infty$, for which we introduce
\begin{equation}
q_{2, \pm}^\inn = \Ls_0 q_\pm^\inn, 
\end{equation}
and further decompose
\begin{equation}\label{def:q2tildecomp}
q_{2, \pm}^\inn(y_\pm, \tau)  = c_1 \Lambda U(y_\pm) + c_{2,1} \pa_1 U(y_\pm) + c_{2,2} \pa_2 U(y_\pm) + \tilde q_{2, \pm}^\inn(y_\pm),
\end{equation}
where $c_1$ and $c_{2,1}$, $c_{2,2}$ are chosen so that 
$$\tilde q_{2, \pm}^\inn \perp_{L^2} \{\Lambda U(y_\pm), \nabla U(y_\pm)\},$$
to ensure the coercivity  (see  Lemma \ref{lem:coercivity-one-bubble-H1-orthogonality})
\begin{equation}\label{est:coercityH2}
\|\tilde q(\tau)\|^2_{H^2_\inn} = \sum_\pm  \langle \tilde{q}_{2, \pm}^\inn,\Ms_0 \tilde{q}_{2, \pm}^\inn \rangle \gtrsim  \sum_\pm \int \frac{|\tilde{q}_{2, \pm}^\inn|^2}{U(y_\pm)} dy_\pm.
\end{equation} 
The coefficients 
\begin{equation}\label{est:c1c2}
|c_1| + |c_{2,1}| + |c_{2,2}| \lesssim \| q(\tau)\|_{H^1_\inn}, 
\end{equation}
together with $H^2_\inn$-norm and Sobolev inequality yield a round $L^\infty$ bound of $q^\inn_\pm$ which is enough for the nonlinear analysis. 

\medskip

\noindent - the intermediate zone $B_{\zeta^*} \setminus \{B_{\zeta_*/2}(a) \cup B_{\zeta_*/2}(-a)\}$ with a fixed  constant $\zeta^* \geq 10$: we use the parabolic regularity to control 
\begin{equation}
\|\vep(\tau)\|_{bd} = \sum_{\pm}\|\vep(\tau)\|_{H^2(B_{\zeta^*} \setminus B_{\zeta_*/2}(\pm a))}. 
\end{equation}

\medskip

\noindent - for the far away zone $\Rb^2 \setminus B_{\zeta^*}$, we simply use the maximum principle to obtain the bound
\begin{equation}
\| \tilde \vep(\tau)\|_\out = \big\|\tilde \vep (1 - \chi_{\zeta^*}) |z|^{\frac{3}{2}}\big\|_{L^\infty(\Rb^2)}. 
\end{equation}

\begin{definition}[Bootstrap regime] \label{def:bootstrap} Let $\tau_0 \gg 1$, $\bar K \gg 1$, $K \gg 1$, $0 < \zeta_*, 1/\zeta^* \ll 1$. For $\tau \geq \tau_0$, we define the set $\Sc(\tau)$ of all functions $\vep$ for which  the following holds: 
\begin{itemize}
\item[i)] \textup{(Modulation parameters)} There exist $\Cc^1$ functions $\nu(\tau), a(\tau)$ and $\alpha(\tau)$ satisfying
\begin{equation} \label{bootstrap:alpha}
|\alpha - 8\nu^2| \leq \bar K \frac{\nu^2}{|\ln \nu|} . 
\end{equation}
\begin{equation}\label{bootstrap:nu}
\left| \frac{\nu_\tau}{\nu} - \frac{\beta \gamma_1}{\ln \nu}\right| \leq \frac{\bar K}{|\ln \nu|^2}, 
\end{equation}
\begin{equation}\label{bootstrap:a}
\Big|a - (2,0) \Big| \leq \bar K \nu,
\end{equation}
\item[ii)] \textup{(Control of the remainder)}
\begin{equation}\label{bootstrap:L2omega_nu}
\| \vep(\tau)\|_{L^2_{\omega_\nu}} \leq K \frac{\nu^2}{|\ln \nu|^2},
\end{equation}
\begin{equation}\label{bootstrap:H2boundary}
\|\vep(\tau)\|_{bd} \leq K^2 \frac{\nu^2}{|\ln \nu|^2}.
\end{equation}
\begin{equation}\label{bootstrap:H1omega_0}
\| q(\tau)\|_{H^1_\inn} \leq K^3 \frac{\nu^2}{|\ln \nu|^2}.
\end{equation}
\begin{equation}\label{bootstrap:H2omega_0}
\|\tilde q(\tau)\|_{H^2_\inn} \leq K^4 \frac{\nu^2}{|\ln \nu|^2}.
\end{equation}
\begin{equation}\label{bootstrap:Linf_outer}
\|\tilde \vep(\tau)\|_\out \leq K^5 \nu^2.
\end{equation}
\end{itemize}
\end{definition}

\subsection{Modulation equations} \label{subsec:modulation-equations}

\begin{lemma}[Modulation equation] \label{lemm:mod} We have 

\begin{equation} \label{eq:Mod0}
\big| \textup{Mod}_0 \big| \lesssim \| q(\tau)\|_{H^1_{\inn}} + \|\vep(\tau)\|_{L^2_{\omega_\nu}} + \frac{\nu^2}{|\ln \nu|^2},
\end{equation}

\begin{equation}\label{eq:Mod1}
\big| \textup{Mod}_1 \big| \lesssim \frac{1}{|\ln \nu|} \Big(\| q(\tau)\|_{H^1_{\inn}} + \|\vep(\tau)\|_{L^2_{\omega_\nu}} + \frac{\nu^2}{|\ln \nu|} \Big),
\end{equation}

\begin{equation}\label{eq:Moda}
\big| \nu \textup{Mod}_a \big| \lesssim  \| q(\tau)\|_{H^1_{\inn}} + \|\vep(\tau)\|_{L^2_{\omega_\nu}} + \frac{\nu^2}{|\ln \nu|^2}.
\end{equation}

\end{lemma}

\begin{proof} \underline{$\textup{Mod}_1$  equation:} We first derive \eqref{eq:Mod1} by writing from the local orthogonality condition \eqref{eq:orthog} and equation \eqref{eq:vep} of $\vep$,
\begin{align*}
0 = \int_{|z \pm a | \lesssim \zeta_*} \pa_\tau\vep dz = \int_{|z \pm a| \lesssim \zeta_*} \big(\tilde{\Ls} \vep -L(\vep) - NL(\vep) + E\big) dz.
\end{align*}
Using the divergent structure of $\tilde{\Ls}(\vep), L(\vep), NL(\vep)$ defined in \eqref{def:L12ztilde}, \eqref{def:Lep}, \eqref{def:NLvep} to get the estimate
\begin{align*}
& \Big|\int_{|z \pm a| \lesssim \zeta_*} \tilde{\Ls} \vep -L(\vep) - NL(\vep)\Big| \lesssim \int_{|z \pm a| \sim \zeta_*} \Big| \nabla \vep -  \vep  \nabla \Phi_{U_{1 + 2, \nu}} - U_{1 + 2, \nu}.\nabla \Phi_{\chi^* \vep}+U_{1+2,\nu} \vep - \beta z \vep \Big|\\
& \quad + \int_{|z \pm a| \sim \zeta_*}\Big| \vep \nabla \Phi_{\tilde W \chi^*} +\tilde W \nabla \Phi_{\vep \chi^*} + \tilde W \nabla \Phi_{\tilde\vep(1 - \chi^*)}\Big|  + \int_{|z \pm a| \sim \zeta_*}\Big|\vep \nabla \Phi_{\vep \chi^*} + \vep \nabla \Phi_{\tilde \vep (1 - \chi^*)}   \Big|\\
& \lesssim \int_{|z \pm a| \sim \zeta_*} |\nabla \vep| + \|U_{1+2, \nu}. \nabla \Phi_{\chi^* \vep} + \tilde W \nabla \Phi_{\vep \chi^*}\|_{L^\infty(|z \pm a| \sim \zeta_*)} \\
& \quad \quad + \Big(\| \nabla \Phi_{U_{1+2, \nu}} + U_{1+2, \nu} + \beta z + \nabla \Phi_{\tilde W \chi^*} + \nabla \Phi_{\vep \chi^*}+ \nabla \Phi_{\tilde \vep (1 - \chi^*)}\|_{L^\infty(|z\pm a| \sim \zeta_*) }\Big) \int_{|z \pm a| \sim \zeta_*}|\vep|.
\end{align*}
Using the rough estimates \eqref{est:Wtilround}, \eqref{est:Phivepchi}, \eqref{est:PhiNabtilvepExt},  $|U_{1+2, \nu}(z)| \lesssim \nu^2$ and $|\nabla \Phi_{U_{1+2, \nu}}| \lesssim 1$ for $|z \pm a| \sim \zeta_*$ and the estimate
$$ |\nabla \Phi_{\chi^* \vep}(z)| \lesssim \int |\ln | z - z' | \; \nabla (\vep \chi^*) | \lesssim \frac{|\ln \nu|}{\nu} \big(\|q(\tau)\|_{H^1_\inn} + \|\vep\|_{L^2_{\omega_\nu}} + \|\vep\|_{bd}\big) \lesssim K^3\nu |\ln \nu|.$$
to get 
\begin{align}
\Big|\int_{|z \pm a| \lesssim \zeta_*} \tilde{\Ls} \vep -L(\vep) - NL(\vep)\Big| &\lesssim \int_{|z \pm a| \sim \zeta_*} \big(|\nabla \vep| + |\vep|\big) + \nu^2 \|\nabla \Phi_{\vep \chi^*}\|_{L^\infty}\nonumber\\
& \lesssim \| q(\tau)\|_{H^1_{\inn}} + \|\vep(\tau)\|_{L^2_{\omega_\nu}} + K^3 \nu^3 |\ln \nu|. \label{est:modtmp11}
\end{align}
For the error term, we use the decomposition \eqref{def:EztauMod} and \eqref{est:PsiinL2omega} to get 
\begin{align*}
\Big| \int_{|z \pm a| \lesssim \zeta_*} E dz \Big| &= \Big|\int_{|z \pm a| \lesssim \zeta_*} \Big(\textup{Mod}_0 \phi_0 + \textup{Mod}_1 \phi_1\Big) dz \Big| \\
& \quad  + \Oc \Big(\big|\textup{Mod}_a\big|\sum_{i=1}^2\int_{|z \pm a| \sim \zeta_* } \big( |U_{i, \nu}| + |\Xi| \big) dz  + \frac{\nu^2}{|\ln \nu|}\Big)\\
& = \Big|\int_{|z \pm a| \lesssim \zeta_*} \Big(\textup{Mod}_0 \phi_0 + \textup{Mod}_1 \phi_1\Big) dz \Big| + \Oc \Big(\nu^2\big|\textup{Mod}_a\big| + \frac{\nu^2}{|\ln \nu|} \Big).
\end{align*}
Using $\int_{\Rb^2} \Lambda U = 0$ and $\int_{\Rb^2}\pa_{1}U = 0$, the decomposition \eqref{id:coercivity:decomposition-phi0} and 
\eqref{id:coercivity:decomposition-phi0check}, we obtain the bound 
\begin{align*}
\Big|\int_{|z \pm a| \lesssim \zeta_*} \phi_0(z) dz\Big| \lesssim \Big| \int_{|z \pm a| \lesssim \zeta_*} \Big[\frac{1}{\nu^2}\Lambda U_{\nu}(z \pm a)  \pm \frac{L_0}{\nu} \pa_1 U_\nu(z \pm a)\Big] dz\Big| + \Big| \int_{|z \pm a| \lesssim \zeta_*} \check{\phi}_0(z) dz \Big| \lesssim 1.  
\end{align*}
From \eqref{def:phii_inn} and the fact that the term $T_2^{(1)}$ consists of a radial component $T_{2,0}^{(1)}$ which behaves like $\frac{1}{|y_\pm|^2}$ for $|y_{\pm}| \gg 1$, combined with the cancellation $\int_{\Rb^2} \Lambda U = \int_{\Rb^2} \pa_1 U =  0$, we end up with the estimate 
\begin{align*}
\Big|\int_{|z \pm a| \lesssim \zeta_*} \phi_1(z) dz\Big| = \Big| \int_{|z \pm a| \lesssim \zeta_*} \frac{1}{\nu^2} T_{2,0}^{(i)}\big( \frac{z \pm a}{\nu} \big) dz \Big| + \Oc(1) \sim |\ln \nu|. 
\end{align*}
Collecting all the estimates yields 
\begin{equation}\label{eq:Mod1tmp}
\big| \textup{Mod}_1\big| \lesssim \frac{1}{|\ln \nu|} \Big(\| q(\tau)\|_{H^1_{\inn}} + \|\vep(\tau)\|_{L^2_{\omega_\nu}} + \big| \textup{Mod}_0 \big| + \nu^2 \big| \textup{Mod}_a\big| + \frac{\nu^2}{|\ln \nu|} \Big),
\end{equation}
which concludes the proof of \eqref{eq:Mod1}. \\

\underline{$\textup{Mod}_0$  equation:} The derivation of \eqref{eq:Mod0} is proceeded similarly by writing from the local orthogonality \eqref{eq:orthog} and equation \eqref{eq:vep} of $\vep$,
\begin{align*}
0 &= \frac{d}{d\tau} \int_{|z \pm a | \lesssim \zeta_*} \vep \Lambda U_\nu (z \pm a) dz = \int_{|z \pm a| \lesssim \zeta_*} \big(\tilde{\Ls} \vep -L(\vep) - NL(\vep) + E\big)\Lambda U_\nu (z \pm a) dz\\
& + \frac{\nu_\tau}{\nu} \int_{|z \pm a| \lesssim \zeta_*} \vep \nu \pa_\nu \Lambda U_\nu(z \pm a) dz \pm a_\tau \int_{|z \pm a| \lesssim \zeta_*} \vep \nabla_a\Lambda U_\nu(z \pm a) dz.
\end{align*}
For the main linear term, small linear and nonlinear terms, we integrate by parts and estimate similarly as in \eqref{est:modtmp11}, then use Cauchy-Schwarz,
\begin{align*}
&\Big| \int_{|z \pm a| \lesssim \zeta_*} \big(\tilde{\Ls} \vep -L(\vep) - NL(\vep)\big)\Lambda U_\nu (z \pm a) dz \Big|\\
&  \qquad \lesssim \int_{|z \pm a| \sim \zeta_*} \big( |\nabla \vep| + |\vep| + \nu^2 \| \nabla \Phi_{\vep \chi^*}\|_{L^\infty} \big)|\Lambda U_\nu (z \pm a)| dz\\
& \quad \qquad  + \int_{|z \pm a| \lesssim \zeta_*} \big( |\nabla \vep| + |\vep| + \nu^2 \| \nabla \Phi_{\vep \chi^*}\|_{L^\infty} \big)|\nabla_a \Lambda U_\nu (z \pm a)| dz\\
& \qquad \lesssim \nu^2 \Big( \| q(\tau)\|_{H^1_{\inn}} + \|\vep(\tau)\|_{L^2_{\omega_\nu}} + K^3 \nu^3 |\ln \nu| \Big) \\
& \quad \qquad + \nu^{-4}\Big( \| q(\tau)\|_{H^1_{\inn}} + \|\vep(\tau)\|_{L^2_{\omega_\nu}} + K^3 \nu^3 |\ln \nu| \Big),
\end{align*}
where we used $\omega_\nu(z) \sim \sum_{\pm} (\nu + |z \pm a|)^4$ for $|z \pm a| \lesssim \zeta_*$ and the estimate 
$$\int_{|z \pm a| \lesssim \zeta_*} | \nabla \Lambda U_\nu(z\pm a)|^2 \omega_\nu^{-1} dz \lesssim \sum_\pm \int_{|z \pm a| \lesssim \zeta_*} \frac{\nu^4}{ (\nu + |z \pm a|)^{10}} \frac{1}{(\nu + |z \pm a|)^4} dz \lesssim \nu^{-8}.$$
The time derivative terms are estimated by Cauchy-Schwarz, 
\begin{align*}
&\Big|\frac{\nu_\tau}{\nu} \int_{|z \pm a| \lesssim \zeta_*} \vep \nu \pa_\nu \Lambda U_\nu(z \pm a) dz \pm a_\tau \int_{|z \pm a| \lesssim \zeta_*} \vep \nabla_z\Lambda U_\nu(z \pm a) dz\Big| \\
& \lesssim \Big|\frac{\nu_\tau}{\nu}\Big| \| \vep\|_{L^2_{\omega_\nu}} \Big(\sum_\pm \int_{|z \pm a| \lesssim \zeta_*} |\Lambda^2 U_\nu|^2 \omega_\nu^{-1} dz  \Big)^\frac{1}{2} + |a_\tau|\| \vep\|_{L^2_{\omega_\nu}} \Big(\sum_\pm \int_{|z \pm a| \lesssim \zeta_*} |\nabla_a \Lambda U_\nu|^2 \omega_\nu^{-1} dz  \Big)^\frac{1}{2}\\
& \lesssim \| \vep\|_{L^2_{\omega_\nu}} \Big(\Big|\frac{\nu_\tau}{\nu}\Big|\nu^{-3}  +|a_\tau|\nu^{-4} \Big). 
\end{align*}
For the error term, we recall from \eqref{def:phii} the estimate for $|z \pm a| \lesssim \zeta_*$,
$$\phi_i(z) \sim \sum \frac{1}{\nu^2} \Lambda U_\nu(z \pm a) + \sum \frac{1}{\nu} \pa_1 U_\nu (z \pm a)  \sim \sum \frac{1}{(\nu + |z \pm a|)^4} + \sum\frac{\nu^2}{(\nu + |z \pm a|)^5}.$$
We then estimate
\begin{align*}
&\Big|\int_{|z \pm a| \lesssim \zeta_*} E \Lambda U_\nu(z \pm a) dz  \Big| = \sum_{i=0}^1\Big| \textup{Mod}_i \int_{|z \pm a| \lesssim \zeta_*} \phi_i \Lambda U_\nu (z \pm a) dz \Big| \\
& \qquad + \Oc \Big( \big|\textup{Mod}_a\big| \sum_{i=0}^1 \sum_\pm \int_{|z \pm a| \lesssim \zeta_*} \big(|U_{i, \nu}| + |\Xi|\big) | \nabla \Lambda U_\nu(z \pm a)| dz + \int_{|z\pm a| \lesssim \zeta_*} |\Psi| |\Lambda U_\nu(z \pm a)| dz \Big)\\
& \sim  \sum_{i=0}^1 \big|\textup{Mod}_i \big| \Big(\sum_\pm \int_{|z \pm a| \lesssim \zeta_*} \frac{\nu^2}{(\nu + |z \pm a|)^8}dz + \sum_\pm \int_{|z \pm a| \lesssim \zeta_*} \frac{\nu^4}{(\nu + |z \pm a|)^9} dz  \Big)\\
& \qquad + \Oc\Big( \big| \textup{Mod}_a\big|   \sum_{\pm}\int_{|z \pm a| \lesssim \zeta_*} \frac{\nu^4}{(\nu + |z \pm a|)^9}  dz  + \|\Psi \|_{L^2_{\omega_\nu}} \Big( \sum_\pm \int_{|z \pm a| \lesssim \zeta_*} \frac{\nu^4}{(\nu + |z \pm a|)^{12}} dz   \Big)^\frac{1}{2} \Big)\\
& \sim \nu^{-4} \sum_{i=0}^1 \big|\textup{Mod}_i \big| +  \Oc\Big( \nu^{-3}\big| \textup{Mod}_a\big|  + \nu^{-3} \|\Psi \|_{L^2_{\omega_\nu}} \Big).
\end{align*}
A collection of the above estimates and using the bounds \eqref{est:PsiinL2omega}, $|\nu_\tau/\nu| \lesssim 1/|\ln \nu|, |a_\tau| \lesssim \nu$ yield 
\begin{equation}\label{eq:Mod0tmp}
\sum_{i=0}^1 \big|\textup{Mod}_i \big| \lesssim \| q(\tau)\|_{H^1_{\inn}} + \|\vep(\tau)\|_{L^2_{\omega_\nu}} + \nu \big| \textup{Mod}_a\big|  +K^3 \nu^3 |\ln \nu|. 
\end{equation}

\underline{$\textup{Mod}_a$  equation:} We now prove the equation \eqref{eq:Moda}. The derivation of \eqref{eq:Moda} is entirely similar as for \eqref{eq:Mod0tmp} and \eqref{eq:Mod1tmp} by using the orthogonality \eqref{eq:orthog} and equation \eqref{eq:vep}, integration by parts and Cauchy-Schwarz. We omit the detail here and only write the outcome estimate
\begin{equation}\label{eq:Modatmp}
\nu \big| \textup{Mod}_a\big| \lesssim  \| q(\tau)\|_{H^1_{\inn}} + \|\vep(\tau)\|_{L^2_{\omega_\nu}} + \sum_{i=0}^1 \big|\textup{Mod}_i \big| + \frac{\nu^2}{|\ln \nu|^2}. 
\end{equation}
Putting together the three estimates \eqref{eq:Mod1tmp}, \eqref{eq:Mod0tmp} and \eqref{eq:Modatmp} yields desired equations and concludes the proof of Lemma \ref{lemm:mod}. 
\end{proof}

%
%

\section{Analysis in the bootstrap regime} \label{sec:energy-estimates}

\subsection{Matched-scalar product based energy estimate}\label{subsec:energy-matched-scalar}
\subsubsection{Further decomposition and energy identity}
The orthogonality condition \eqref{eq:orthog} also allows us to derive a dynamical system driving the law of blowup as in Lemma \ref{lemm:mod} below. However, such a local condition \eqref{eq:orthog} is not enough, at the linear level, to control $\vep$ using an energy estimate due to the modulation terms which do not satisfy a suitable estimate (the Mod terms appearing in \eqref{def:EztauMod}). To overcome this issue, we use the projection \eqref{id:coercivity-definition-projection} and decompose
\begin{equation}\label{def:vephat}
\vep = \hat \vep+a_0 \phi_0,
\end{equation}
where $\hat \vep$ satisfies the orthogonality in the self-similar zone
\begin{equation}\label{eq:orthog_global}
\langle \hat \vep, \phi_0 \rangle_\ast  = 0.
\end{equation}
The uniqueness and the properties of this decomposition can be found in Proposition \ref{prop:decomp_vephat}, and imply the estimates
\begin{equation}\label{est:a0a1}
|a_0|  \lesssim \min \left(\frac{1}{\sqrt{\ln \nu}}  \|  \vep\|_{L^2_{\omega_\nu }}, \frac{1}{|\ln \nu|} \| \nabla \vep\|_{L^2_{\omega_\nu }}\right),
\end{equation}
and the equivalence
\begin{equation} \label{est:equivvep_vephat}
\|\vep \|^2_{\omega_\nu} \approx \langle \vep, \vep \rangle_\ast \approx \langle \hat \vep, \hat \vep \rangle_\ast.
\end{equation}
We compute
$$
\frac{d}{d\tau} \langle \hat \vep ,\hat \vep \rangle_* =\frac{\nu_\tau}{\nu} \nu \partial_\nu (\langle \hat \vep,\hat \vep \rangle_*)+\frac{a_\tau}{\nu} \nu \partial_a (\langle \hat \vep,\hat \vep \rangle_*)+2\langle \widehat{\partial_\tau \vep},\hat \vep\rangle_*,
$$
(we remark that $\widehat{\partial_\tau \vep}$ is the projection of $\partial_\tau \varepsilon$, so $ \widehat{\partial_\tau \vep}\neq \partial_\tau \hat \vep$), and use the estimate \eqref{bd:coercivity-partial-nu-adapted-norm-hat-u} to bound the first term by
$$
\left| \frac{\nu_\tau}{\nu} \nu \partial_\nu (\langle \hat \vep,\hat \vep \rangle_*)+\frac{a_\tau}{\nu} \nu \partial_a (\langle \hat \vep,\hat \vep \rangle_*)\right|\lesssim \frac{|\nu_\tau|+|a_\tau|}{\nu} \| \varepsilon \|_{L^2_{\omega_\nu}}^2,
$$
and the orthogonality \eqref{id:coercivity-orthogonalite-projection} to obtain that $\langle \widehat{\partial_\tau \vep},\hat \vep\rangle_*=\langle \partial_\tau \vep,\hat \vep\rangle_*$. Combining and injecting \eqref{def:NLvep} and \eqref{bd:coercivity-control-adapted-norm-hatu-2}, we arrive at the following energy identity for the adapted norm of the projection:
\begin{equation}\label{eq:hatvep}
\frac 12 \frac{d}{d\tau} \langle \hat \vep ,\hat \vep \rangle_*  +O\left(\frac{|\nu_\tau|+|a_\tau|}{\nu}\langle \hat \vep ,\hat \vep \rangle_*\right) =\langle  \tilde \Ls^z \vep  - L(\vep)  - NL(\vep) + E ,\hat \vep \rangle_*.
\end{equation}
Thanks to the global orthogonality \eqref{eq:orthog_global}, we can eliminate the Mod terms in $E$ once performing a standard energy estimate. Once the energy estimate for $\hat \vep$ is established, we can directly derive a similar estimate for $\vep$ thanks to Proposition \ref{prop:decomp_vephat}.

\subsubsection{$L^2_{\omega_\nu}$ energy estimate}
This section is devoted to the main energy estimate in $L^2_{\omega_\nu}$, where we crucially rely on the equivalence \eqref{est:equivvep_vephat}, the global coercivity of the linearized operator for the adapted scalar product \eqref{bd:coercivity-Lztilde-matched-product} to gain the dissipation, and the global orthogonality \eqref{eq:orthog_global} to cancel out the Mod terms in $E$. We claim the following. 
\begin{lemma}[$L^2_{\omega_\nu}$-estimate]  \label{lemm:L2energy} Let $\vep$ be a solution to \eqref{eq:vep} and satisfy the bootstrap bounds in Definition \ref{def:bootstrap} for $\tau \in [\tau_0, \tau_1]$. Then, we have the following estimate for $\tau \in [\tau_0, \tau_1]$:
\begin{equation}\label{est:energyL2_form}
 \frac{d}{d\tau} \langle \hat \vep, \hat \vep\rangle_\ast \leq -\delta_0 \langle \hat \vep, \hat \vep\rangle_\ast  + C\Big( |\textup{Mod}_1| +  \nu^2 |\textup{Mod}_a| + \frac{\nu^2}{|\ln \nu|^2}\Big)^2.
\end{equation}
for some $\delta_0 > 0$, and $C  = C(\zeta_*, \zeta^*)> 0$. 
\end{lemma}
 \begin{proof} We estimate all the terms in the right hand side of \eqref{eq:hatvep}. As for the main linear term, we recall from Proposition \ref{Prop:coercive-Lz-global-second} and the fact that $\langle \tilde{\Ls}^z \vep, \hat \vep \rangle_\ast = \langle \widehat{ \tilde{\Ls}^z \vep}, \hat \vep \rangle_\ast$ the coercivity estimate
\begin{align}\label{est:spectralgap_hateps}
\langle \tilde{\Ls}^z \vep, \hat \vep \rangle_\ast &  \leq - \delta \left[ \langle \hat \vep, \hat \vep \rangle_\ast +  \int \big(  \sum_\pm \frac{|\vep|^2}{\nu^2 + |z \pm a|^2} +  \langle z \rangle^2 \vep^2  + |\nabla \vep|^2\big) \omega_\nu\right]  \quad \textup{for some $\delta > 0$}.
\end{align}
We now use this coercivity estimate and the equivalence of the norm \eqref{est:equivvep_vephat} to control all the remaining terms in \eqref{eq:hatvep}. \\

\paragraph{\underline{\it The small linear term}:} From the pointwise estimate \eqref{est:pointwise_Xi} of $\Xi$, the improved pointwise estimate \eqref{est:pointwise_phi1m0} and the relation $|\alpha| \sim \nu^2$, we have the rough bound
\begin{equation}\label{est:Wtilround}
k \in \mathbb{N}, \quad \big| \nabla^k \tilde W(z)\big| \lesssim \big|\nabla^k\big(\Xi(z) + \alpha(\phi_1 - \phi_0)\big)\big| \lesssim \sum_\pm\frac{\nu^2}{(\nu + |z \pm a|)^{2 + k}}, \quad |\nabla \Phi_{\tilde W \chi^*}| \lesssim \nu |\ln \nu|.
\end{equation}
Let us write from the decomposition \eqref{def:vephat},
\begin{align*}
\langle L\vep, \hat \vep \rangle_\ast  &=  \langle L\vep, \vep \rangle_\ast + a_0 \langle L\vep, \phi_0 \rangle_\ast,
\end{align*}
where we recall $\langle , \rangle_\ast$ the adapted scalar product defined by \eqref{def:scalar12_intro}, and the definition of $L$ from \eqref{def:Lep} the expression 
\begin{align*}
L(\vep) &= \nabla.\Big[ \vep \nabla \Phi_{\tilde W \chi^*} + \tilde W \nabla \Phi_{\vep \chi^*} + \tilde W \nabla \Phi_{\tilde \vep (1 - \chi^*)} \Big]\\
& = \nabla \vep . \nabla \Phi_{\tilde W \chi^*} -  2 \chi^* \tilde W \vep + \nabla \tilde W . \big( \nabla \Phi_{\vep \chi^*}  + \nabla \Phi_{\tilde \vep (1 - \chi^*)}\big)  - \tilde W \tilde \vep  (1 - \chi^*).
\end{align*}
Since estimates concerning $\tilde W$  already has a gain of $\nu^2$ factor, so the estimate for $\langle L\vep, \vep \rangle_\ast $ is simply absorbed by the dissipation. From the continuity estimate \eqref{matchedscalarproduct1:continuity} and the bound \eqref{est:Wtilround}, we get 
\begin{align*}
|\langle \nabla \vep . \nabla \Phi_{\tilde W \chi^*}, \vep \rangle_\ast| \lesssim \|\nabla \Phi_{\tilde W \chi^*}\|_\infty \|\nabla \vep\|_{\omega_\nu}\|\vep \|_{\omega_\nu} \lesssim \nu |\ln \nu| \; \| \big( \nabla \vep\|_{L^2_{\omega_\nu}}^2 + \|\vep \|^2_{\omega_\nu}\big).
\end{align*}
We estimate from the convolution formula and Cauchy-Schwarz,
\begin{equation} \label{est:Phivepchi}
|\Phi_{\chi^* \vep}(z)| \lesssim \int \big|\ln |z - z'| \vep(z')\big| \chi^* dz' \lesssim \|\vep\|_{L^2_{\omega_\nu}} \Big(\int_{|z'| \leq 2 \zeta^*} \ln^2|z - z'| \omega_\nu^{-1} dz'  \Big)^\frac{1}{2} \lesssim \frac{|\ln \nu|}{\nu} \|\vep\|_{L^2_{\omega_\nu}},
\end{equation}
and 
\begin{equation}\label{est:PhiNabvepchi}
|\nabla \Phi_{\chi^* \vep}(z)| \lesssim \int_{|z'| \leq 2 \zeta^*} \big|\ln |z - z'| \big( \nabla \vep(z') + \vep(z')\nabla \chi^*(z')\big)\big|  dz' \lesssim \frac{|\ln \nu|}{\nu} \big( \|\nabla \vep\|_{L^2_{\omega_\nu}} + \|\vep\|_{L^2_{\omega_\nu}}\big).
\end{equation}
We have from Lemma \ref{lemm:PointwisePoisonField} with $\gamma = 3/2$ and $\sigma = 1/8$ and  
\begin{equation}\label{est:PhiNabtilvepExt}
|\nabla \Phi_{\tilde \vep(1 - \chi^*)}(z)| \lesssim \|\tilde \vep (1 - \chi^*)|z|^\frac{3}{2}\|_{\infty} \Big(|z|^{-\sigma} + \langle z \rangle^{-3/2}|z|\Big) \lesssim \|\vep\|_\out  \Big(|z|^{-1/8} + \langle z \rangle^{-3/2}|z|\Big).
\end{equation}

Using the estimate \eqref{est:Wtilround}, \eqref{est:Phivepchi}, the definition of $\langle, \rangle_\ast$ and Cauchy-Schwarz, we obtain
\begin{align*}
|\langle \chi^* \tilde W \vep, \vep \rangle_\ast | &\lesssim \nu^2 \int \sum_\pm \frac{|\vep|^2}{(\nu + |z \pm a|)^2} \omega_\nu dz + \int \chi^* \sum_\pm \frac{\nu^2 |\vep|}{(\nu + |z \pm a|)^2} |\Phi_{\vep \chi^*}| dz\\
& \lesssim \nu^2 \int \sum_\pm \frac{|\vep|^2}{(\nu + |z \pm a|)^2} \omega_\nu dz \\
& \quad +  \nu |\ln \nu| \|\vep\|_{L^2_{\omega_\nu}}\left(\int \sum_\pm \frac{|\vep|^2}{(\nu + |z \pm a|)^2} \omega_\nu dz \right)^\frac{1}{2} \left(\int_{|z| \leq 2\zeta^*} \frac{dz}{(\nu + |z \pm a|)^2} \right)^\frac{1}{2}\\
& \lesssim \nu^2 \int \sum_\pm \frac{|\vep|^2}{(\nu + |z \pm a|)^2} \omega_\nu dz +  \nu |\ln \nu|^\frac{3}{2} \|\vep\|_{L^2_{\omega_\nu}}\left(\int \sum_\pm \frac{|\vep|^2}{(\nu + |z \pm a|)^2} \omega_\nu dz \right)^\frac{1}{2}\\
&\lesssim  \nu |\ln \nu|^\frac 32\Big( \int \sum_\pm \frac{|\vep|^2}{(\nu + |z \pm a|)^2} \omega_\nu dz + \|\vep\|^2_{L^2_{\omega_\nu}}\Big).
\end{align*}
For the estimate of the Poisson field, we use the continuity \eqref{matchedscalarproduct1:continuity}, \eqref{est:Wtilround}, \eqref{est:PhiNabvepchi}, \eqref{est:PhiNabtilvepExt} to obtain 
\begin{align*}
&\big|\langle \nabla \tilde W (\nabla \Phi_{\vep \chi^*} + \nabla \Phi_{\tilde \vep (1 - \chi^*)}), \vep \rangle_\ast \big| \lesssim \| \nabla \tilde W . (\nabla \Phi_{\vep \chi^*} + \nabla \Phi_{\tilde \vep (1 - \chi^*)})\|_{\omega_\nu} \|\vep\|_{\omega_\nu}\\
& \qquad \lesssim \nu |\ln \nu| \| \vep\|_{\omega_\nu} (\|\nabla \vep\|_{\omega_\nu} + \|\vep\|_{\omega_\nu}) \sum_\pm \left( \int \frac{\langle z \rangle^C}{(\nu + |z \pm a|)^6} \omega_\nu  \right)^\frac{1}{2} \\
& \qquad \quad + \nu^2 \| \vep\|_{\omega_\nu} \|\vep\|_\out \sum_\pm \left(\int \frac{\langle z \rangle^C |z|^{-1/4}}{(\nu + |z \pm a|)^6} \omega_\nu \right)^\frac{1}{2}\\
& \qquad\lesssim \nu |\ln \nu|^ \frac{3}{2} \| \vep\|_{\omega_\nu} (\|\nabla \vep\|_{\omega_\nu} + \|\vep\|_{\omega_\nu}) + \nu^{2 - 1/8} |\ln \nu| \| \vep\|_{\omega_\nu}\|\vep\|_\out\\
& \qquad \lesssim \nu |\ln \nu|^ \frac{3}{2}\big(\|\nabla \vep\|^2_{\omega_\nu} + \|\vep\|^2_{\omega_\nu} \big) + K^4 \frac{\nu^{6 - 1/8}}{|\ln \nu|},
\end{align*}
where we used the bootstrap bounds \eqref{bootstrap:L2omega_nu} and \eqref{bootstrap:Linf_outer} int the last step. \\
The estimate of the small outer term is obtained from the continuity \eqref{matchedscalarproduct1:continuity} and \eqref{est:Wtilround},  
\begin{align*}
\big|\langle \tilde W \tilde{\vep}(1 - \chi^*), \vep \rangle_\ast \big| \lesssim \|\tilde W(1 - \chi^*)\|_\infty \|\vep\|^2_{\omega_\nu} \lesssim \nu^2\|\vep\|^2_{\omega_\nu}.
\end{align*}
A collection of all the above estimates and taking $\tau_0$ large so that $\nu(\tau_0)$ small yield
\begin{equation}\label{est:tmpL2Lvepvep}
|\langle L \vep, \vep \rangle_\ast| \lesssim \nu |\ln \nu|^\frac{3}{2}\Big(\|\nabla \vep\|^2_{\omega_\nu} + \|\vep\|^2_{\omega_\nu} + \sum_\pm \Big\|\frac{\vep}{\nu + |z \pm a|} \Big\|_{\omega_\nu}^2  \Big) + \frac{\nu^4}{|\ln \nu|^4}.
\end{equation}

\noindent We now estimate $\langle L\vep, \phi_0 \rangle_\ast$, we shall use the divergence structure of $L\vep$, the adapted scalar product \eqref{def:scalar12_intro} and the cancellation $ \nabla \Ms_0 \Lambda U = 0$ to gain a smallness for this estimate. We recall from \eqref{def:phii} and  \eqref{est:pointwise_phiitil}, 
\begin{equation}\label{id:decomp_phi0}
\phi_0(z) =\sum_\pm \Big( - \frac{1}{16\nu^2}\Lambda U_\nu (z \pm a) \pm \frac{L_0}{\nu} \pa_1 U_\nu(z \pm a) \Big)\chi_{_{\pm a, \zeta_*}}(z) + \tilde \phi_0, \quad |\nabla ^k \tilde \phi_0 (z)| \lesssim \sum_\pm \frac{\langle z \rangle ^{C_k}}{(\nu + |z \pm a|)^{2 + k}}. 
\end{equation}
We first handle the leading term $\frac{1}{\nu^2}\Lambda U_\nu(z \pm a)\chi_{_{\pm a, \zeta_*}}$ since the estimate concerning $\tilde{\phi}_0$ is straightforward without using the cancellation $ \nabla \Ms_0 \Lambda U = 0$. We remark that a direct estimate by using \eqref{matchedscalarproduct1:continuity} applied to the leading term is not sufficient since $\|\nu^{-2} \Lambda U_\nu(z \pm a)\|_{\omega_\nu} \lesssim \nu^{-1}$, this is the reason we do need to use the algebraic cancellation $\nabla \Ms_0 \Lambda U = 0$ to improve the estimate concerning the leading term of $\phi_0$.  We have from the identity \eqref{id:projection-matched-scalar-LambdaU1nu} and the divergence of $L\vep$, 
\begin{align*}
& \Big|\nu^{-2} \sum_\pm \langle L \vep, \Lambda U_\nu(z \pm a) \chi_{_{\pm a, \zeta_*}} \rangle_\ast  \Big| = \sum_\pm \Big| -2c_*\int_{|z\pm a| \leq \zeta_*} L \vep  \Big| + \Oc\Big( \sqrt{\ln \nu} \| L\vep\|_{\omega_\nu} \Big)\\
& \qquad \qquad  \lesssim \sum_\pm \int_{|z \pm a| \sim \zeta_*} \Big| \vep \nabla \Phi_{\tilde W \chi^*} + \tilde W \nabla \Phi_{\vep \chi^*} + \tilde W \nabla \Phi_{\tilde \vep (1 - \chi^*)} \Big| + \Oc\Big( \sqrt{\ln \nu} \| L\vep\|_{\omega_\nu} \Big). 
\end{align*} 
We use \eqref{est:Wtilround}, \eqref{est:PhiNabvepchi} and \eqref{est:PhiNabtilvepExt} to estimate 
\begin{align*}
&\sum_\pm \int_{|z \pm a| \sim \zeta_*} \Big| \vep \nabla \Phi_{\tilde W \chi^*} + \tilde W \nabla \Phi_{\vep \chi^*} + \tilde W \nabla \Phi_{\tilde \vep (1 - \chi^*)} \Big| \\
& \qquad  \lesssim \sum_\pm \| \nabla \Phi_{\tilde W \chi^*}\|_{L^\infty(|z \pm a| \sim \zeta_*)} \int_{|z \pm a| \sim \zeta_*} |\vep| + \sum_{\pm}\|\tilde W \nabla \Phi_{\vep \chi^*} + \tilde W \nabla \Phi_{\tilde \vep (1 - \chi^*)}\|_{L^\infty(|z \pm a| \sim \zeta_*)}\\
& \qquad  \lesssim \nu |\ln \nu| \|\vep\|_{\omega_\nu} + \nu^2 \frac{|\ln \nu|}{\nu} \Big( \|\nabla \vep\|_{\omega_\nu} + \|\vep\|_{\omega_\nu}\Big) + \nu^2 \|\vep\|_\out,
\end{align*}
and from \eqref{est:Wtilround} the bounds $\|\nabla \Phi_{\tilde W \chi^*}\| \lesssim \nu |\ln \nu|$, $\|\chi^* \tilde W\| \lesssim 1$,  $\| \nabla \tilde W\|_{\omega_\nu} \lesssim \nu^2 \sqrt{|\ln \nu|}$ and $\|\tilde W\|_{\omega_\nu} \lesssim \nu^2$,
\begin{align*}
\|L \vep\|_{\omega_\nu} &= \big\| \nabla \vep . \nabla \Phi_{\tilde W \chi^*} -  2 \chi^* \tilde W \vep + \nabla \tilde W . \big( \nabla \Phi_{\vep \chi^*}  + \nabla \Phi_{\tilde \vep (1 - \chi^*)}\big)  - \tilde W \tilde \vep  (1 - \chi^*) \big\|_{\omega_\nu}\\
& \lesssim \| \nabla \Phi_{\tilde W \chi^*}\|_\infty \|\nabla \vep\|_{\omega_\nu} + \|\chi^* \tilde W \|_\infty \|\vep\|_{\omega_\nu} + \|\nabla \Phi_{\chi^* \vep}\|_\infty \| \nabla \tilde W\|_{\omega_\nu} \\
& \qquad  + \|\vep\|_\out \| \nabla \tilde{W} (|z|^{-1/8} + \langle z \rangle^{-3/2}|z|)\|_{\omega_\nu} + \|\vep\|_\out \|\tilde W\|_{\omega_\nu}\\
& \lesssim \nu |\ln \nu| \|\nabla \vep\|_{\omega_\nu} + \|\vep\|_{\omega_\nu} + \nu |\ln \nu|^\frac{3}{2} \big(\|\nabla \vep\|_{\omega_\nu} + \|\vep\|_{\omega_\nu} \big) + \|\vep\|_\out \nu^{2 - 1/8} + \nu^2 \|\vep\|_{\out}.
\end{align*}
Using $\|\tilde \phi_0\|_{\omega_\nu} \lesssim 1$, the estimate \eqref{id:projection-matched-scalar-partialx1U1nu-1} $|\langle \nu^{-1} \pa_{1}U_\nu(z \pm a), L \vep \rangle_\ast| \lesssim \nu \| L \vep \|_{\omega_\nu}$, the continuity estimate \eqref{matchedscalarproduct1:continuity}, the smallness estimate \eqref{est:a0a1}, the bootstrap bound \eqref{bootstrap:Linf_outer} and gathering all the obtained estimates yield 
\begin{align*}
\big| a_0 \langle L \vep, \phi_0 \rangle_\ast \big| &\lesssim |a_0|\Big\{ \Big|\nu^{-2} \sum_\pm \langle L \vep, \Lambda U_\nu(z \pm a) \chi_{_{\pm a, \zeta_*}} \rangle_\ast  \Big| + \nu \|L\vep\|_{\omega_\nu} +  \| \vep\|_{\omega_\nu} \|\tilde \phi_0\|_{\omega_\nu} \Big\}\nonumber \\
& \lesssim \frac{1}{|\ln \nu|} \|\nabla \vep\|_{\omega_\nu} \Big\{ \nu |\ln \nu|^2 \|\nabla \vep\|_{\omega_\nu} + \| \vep\|_{\omega_\nu} + \nu^{2 - 1/8}\sqrt{|\ln \nu|} \|\vep\|_\out \Big\} \nonumber \\
& \lesssim \nu |\ln \nu| \| \nabla \vep\|_{\omega_\nu}^2 + \frac{1}{|\ln \nu|} \|\vep\|_{\omega_\nu}  \| \nabla \vep\|_{\omega_\nu} + \frac{\nu^4}{|\ln \nu|^4}.
\end{align*}
We combine this estimate and \eqref{est:tmpL2Lvepvep} to complete the control of the small linear term as
\begin{equation}\label{est:LvepL2tmp}
|\langle L \vep, \hat \vep \rangle| \lesssim \frac{1}{|\ln \nu|} \Big(\|\nabla \vep\|^2_{\omega_\nu} + \|\vep\|^2_{\omega_\nu} + \sum_\pm \Big\|\frac{\vep}{\nu + |z \pm a|} \Big\|_{\omega_\nu}^2  \Big) + \frac{\nu^4}{|\ln \nu|^4}. 
\end{equation}

\medskip

\paragraph{\underline{\it The nonlinear term}:} From the definitions \eqref{def:NLvep} and \eqref{def:vephat}, the continuity estimate \eqref{matchedscalarproduct1:continuity}, we write 
\begin{align*}
&\big| \langle NL(\vep), \hat \vep \rangle_\ast \big| = \big| \langle NL(\vep),  \vep \rangle_\ast  + a_0 \langle NL(\vep), \phi_0 \rangle_\ast\big| \lesssim \|NL(\vep)\|_{\omega_\nu} \|\vep\|_{\omega_\nu} + |a_0| |\langle NL(\vep), \phi_0 \rangle_\ast |.
\end{align*}
For the estimate of $\|NL(\vep)\|_{\omega_\nu}$, we expand $NL(\vep)$ and use \eqref{est:PhiNabvepchi}, \eqref{est:PhiNabtilvepExt}, 
\begin{align*}
\|NL(\vep)\|_{\omega_\nu} &= \|  \nabla \vep \cdot  \nabla \Phi_{\vep \chi^*} + \nabla\vep \cdot \nabla \Phi_{\tilde \vep (1 - \chi^*)} \big) - \vep^2 \chi^* - \vep \tilde{\vep}(1 - \chi^*)\|_{\omega_\nu}\\
& \lesssim \|\nabla \vep\|_{\omega_\nu} \|\nabla \Phi_{\vep \chi^*}\|_{\infty} + \|\nabla \vep (|z|^{-1/8} + \langle z \rangle^{-3/2}|z|)\|_{\omega_\nu} \|\vep\|_{\out} + \|\vep\|_{\omega_\nu} \big(\|\vep \chi^*\|_{\infty} + \|\vep\|_\out\big)\\
& \lesssim  \| \nabla \vep\|_{\omega_\nu}\big( \nu |\ln \nu| + K^3 \nu^{2 - 1/8} \big) + \|\vep\|_{\omega_\nu}\big(\frac{K^3}{|\ln \nu|} + K^3 \nu^2\big)\\
& \lesssim \nu |\ln \nu| \|\nabla \vep\|_{\omega_\nu} + \frac{K^3}{|\ln \nu|} \|\vep\|_{\omega_\nu}.
\end{align*}
where we used the bootstrap bound \eqref{bootstrap:Linf_outer} and the $L^\infty$ estimate from Sobolev and the bootstrap bounds \eqref{bootstrap:H2boundary}, \eqref{bootstrap:H1omega_0}, \eqref{bootstrap:H2omega_0} and \eqref{est:c1c2}, 
\begin{align*}
\| \vep \chi^*\|_\infty \leq \sum_\pm \nu^{-2}\|q^\inn_\pm \mathbf{1}_{\{|y_\pm| \lesssim \zeta^*/\nu}\}\|_\infty \lesssim \nu^{-2}\big(\|\tilde q\|_{H^2_\inn} + \|q\|_{H^1_\inn} + \|\vep\|_{bd}  \big) \lesssim \frac{K^3}{|\ln \nu|}. 
\end{align*}
The estimate for $|\langle NL(\vep), \phi_0 \rangle_\ast |$ is proceeded as for the small linear term. We decompose $\phi_0$ as in \eqref{id:decomp_phi0} to extract the leading term of $\phi_0$  for which we use the identity \eqref{id:projection-matched-scalar-LambdaU1nu} and the divergent form of $NL(\vep)$ to get the estimate 
\begin{align*}
|\langle NL(\vep), \phi_0 \rangle_\ast | &\lesssim \left|\sum_{\pm} \langle NL(\vep), \nu^{-2} \Lambda U_\nu(z \pm a) \chi_{\pm a, \zeta_*} \rangle_\ast\right| + \nu\|NL(\vep)\|_{\omega_\nu} + \|NL(\vep)\|_{\omega_\nu}\|\tilde{\phi}_0\|_{\omega_\nu}\\
&\lesssim \sum_{\pm} \left| -2c_* \int_{|z \pm a| \leq \zeta_*} NL(\vep) + 4c_*\int_{|z \pm a| \leq \zeta_*}NL(\vep)  \right| + \|NL(\vep)\|_{\omega_\nu} \big(\sqrt{\ln \nu} +\nu +  \|\tilde{\phi}_0\|_{\omega_\nu} \big)\\
& \lesssim \sum_\pm \int_{|z \pm a| \sim \zeta_*} \big|\vep \nabla \Phi_{\vep \chi^*} + \vep \nabla \Phi_{\tilde \vep (1 - \chi^*)}  \big| + \|NL(\vep)\|_{\omega_\nu} \big(\sqrt{\ln \nu} + 1\big)\\
& \lesssim \|\vep\|_{\omega_\nu}\big( \|\nabla \Phi_{\vep \chi^*} \mathbf{1}_{\{|z \pm a| \sim \zeta_*\}}\|_\infty + \|\nabla \Phi_{\tilde \vep (1 - \chi^*)} \mathbf{1}_{\{|z \pm a| \sim \zeta_*\}}\|_\infty  \big) + \|NL(\vep)\|_{\omega_\nu}\sqrt{\ln \nu} \\
& \lesssim \|\vep\|_{\omega_\nu} \big( \frac{|\ln \nu|}{\nu} (\|\nabla \vep\|_{\omega_\nu} + \|\vep\|_{\omega_\nu}) + \| \vep\|_\out  \big) + \nu |\ln \nu|^{3/2} \|\nabla \vep\|_{\omega_\nu} + \frac{K^3}{\sqrt{|\ln \nu|}}\|\vep\|_{\omega_\nu}\\
& \lesssim \frac{K^5 \nu^2}{|\ln \nu|^{5/2}}.
\end{align*}
A collection of the above estimates and using the estimate \eqref{est:a0a1} yield 
\begin{equation}\label{est:NLvepL2tmp}
\big| \langle NL(\vep), \hat \vep \rangle_\ast \big| \lesssim \Big( \nu |\ln \nu| \|\nabla \vep\|_{\omega_\nu} + \frac{K^3}{|\ln \nu|} \|\vep\|_{\omega_\nu}\Big) \|\vep\|_{\omega_\nu} + \frac{K^5 \nu^2}{|\ln \nu|^{5/2}} \frac{1}{|\ln \nu|} \|\nabla \vep\|_{\omega_\nu} \lesssim \frac{K^5 \nu^4}{|\ln \nu|^{9/2}}. 
\end{equation} 

\medskip

\paragraph{\underline{\it The error term}:} We use the decomposition \eqref{def:EztauMod}, the orthogonality \eqref{eq:orthog_global}  to write 
\begin{align*}
|\langle E, \hat \vep \rangle_\ast| \leq |\textup{Mod}_1| |\langle \phi_1 - \phi_0, \hat \vep \rangle_\ast | + |\textup{Mod}_a| |\langle \nabla U_{1, \nu} + \chi_{-a, \zeta_*}\nabla \Xi - \nabla U_{2, \nu} - \chi_{a, \zeta_*} \nabla \Xi , \hat \vep \rangle_\ast| + |\langle \Psi, \hat \vep\rangle_\ast|.
\end{align*}
We remark that canceling the $\textup{Mod}_0$ term in $E$ is crucial to avoid the loop in the analysis, the $\textup{Mod}_1$ already gains a $\frac{1}{|\ln \nu|}$ and $\textup{Mod}_a$ is small enough. Let recall from \eqref{def:vephat} the decomposition $\hat \vep = \vep + a_0 \phi_0$, and from  \eqref{def:phii} $\phi_1 - \phi_0 = \sum_\pm \frac{L_1 - L_0}{\nu}\pa_1 U_\nu(z \pm a)  + \Oc\Big(\sum_\pm \frac{\langle z \rangle^C}{(\nu + |z\pm a|)^2}\Big)$, then we use the cancellation $\Ms(\pa_1 U) = 0$ and  the identity \eqref{matchedscalarproduct1:continuity} to estimate
\begin{align*}
|\langle \phi_1 - \phi_0, \vep \rangle_\ast |  \lesssim \Big\|\sum_\pm \frac{\langle z \rangle^C}{(\nu + |z\pm a|)^2} \Big\|_{\omega_\nu} \|\vep\|_{\omega_\nu} \lesssim \|\vep\|_{\omega_\nu}. 
\end{align*} 
For the term $a_0 \phi_0$, we proceed as before for the linear small term by using \eqref{id:decomp_phi0}, the identity \eqref{id:projection-matched-scalar-LambdaU1nu}, the estimate \eqref{id:projection-matched-scalar-partialx1U1nu-1} and the continuity estimate \eqref{matchedscalarproduct1:continuity} to obtain
\begin{align*}
|\langle \phi_1 - \phi_0, \phi_0 \rangle_\ast | &\lesssim\Big| \nu^{-2}\sum_\pm \langle \phi_1 - \phi_0, \Lambda U_\nu(z \pm a)\chi_{\pm a, \zeta_*} \rangle_\ast  \Big| \\
& \qquad \qquad \qquad  + \Big|\sum \langle \phi_1 - \phi_0, \nu^{-1}\pa_1 U_\nu(z \pm a) \rangle_\ast \Big| + |\langle \phi_1 - \phi_0, \tilde \phi_0 \rangle_\ast | \\
& \lesssim  \sum_\pm \Big| \int_{|z \pm a| \lesssim \zeta_*} \phi_1 - \phi_0   \Big| +  \sum_{\pm }\Big\| \frac{\phi_1 - \phi_0}{\sqrt{\nu + |z \pm a|}} \Big\|_{\omega_\nu} + \|\phi_1 - \phi_0\|_{\omega_\nu} \big( \nu + \| \tilde \phi_0\|_{\omega_\nu}\big)\\
& \lesssim |\ln \nu|.
\end{align*}
Summing up the above bounds and using \eqref{est:a0a1} yields, 
\begin{equation}\label{est:tmpL2EMod1}
|\textup{Mod}_1| |\langle \phi_1 - \phi_0, \hat \vep \rangle_\ast | \lesssim |\textup{Mod}_1| \big( \|\vep\|_{\omega_\nu} + |a_0| |\ln \nu|\big) \lesssim  |\textup{Mod}_1| \big( \|\vep\|_{\omega_\nu} + \|\nabla \vep\|_{\omega_\nu}\big). 
\end{equation}
For the estimate of the $\textup{Mod}_a$ term, we recall the cancellation $\Ms_0 \nabla U = 0$, from which and the definition of the inner produce $\langle \cdot, \cdot \rangle_\ast$ to write 
\begin{align*}
&|\langle \nabla U_{1, \nu} - \nabla U_{2, \nu}, \hat \vep \rangle_\ast| \lesssim |\langle \nabla U_{1, \nu}(1 - \chi_{a, \zeta_*}) - \nabla U_{2, \nu}(1 - \chi_{-a, \zeta_*}), \hat \vep \rangle_\ast| \equiv \big| \langle G, \hat \vep \rangle_\ast  \big|, 
\end{align*}
where we have a rough pointwise bound from $|\nabla U_\nu(z \pm a)| \lesssim \frac{\nu^2}{(\nu + |z \pm a|)^5}$, 
\begin{align*}
|G| &= | \nabla U_{1, \nu}(1 - \chi_{a, \zeta_*})  - \nabla U_{2, \nu}(1 - \chi_{-a, \zeta_*})| \lesssim \nu^2 \mathbf{1}_{\{|z \pm a| \gtrsim \zeta_*\}}.
\end{align*}
We use this estimate and the decompositions \eqref{def:vephat}, \eqref{id:decomp_phi0}, the continuity estimate \eqref{matchedscalarproduct1:continuity}, the identity \eqref{id:projection-matched-scalar-LambdaU1nu} and \eqref{est:a0a1} to obtain 
\begin{align*}
\big| \langle G, \hat \vep \rangle_\ast  \big| &\lesssim \big| \langle G, \vep \rangle_\ast  \big| + |a_0|\big| \langle G, \nu^{-2}\sum_\pm \Lambda U_\nu(a \pm a)\chi_{\pm a, \zeta_*} \rangle_\ast  \big| + |a_0|\big| \langle G, \tilde \phi_0 \rangle_\ast  \big|\\
& \lesssim \|G\|_{\omega_\nu} \|\vep\|_{\omega_\nu} + |a_0|\left\{ \sum_\pm \int_{|z \pm a| \lesssim \zeta_*} |G| + \Big\| \frac{G}{\sqrt{\nu + |z \pm a|}} \Big\|_{\omega_\nu} +  \|G\|_{\omega_\nu} \|\tilde{\phi}_0\|_{\omega_\nu}\right\}\\
& \lesssim \nu^2  \|\vep\|_{\omega_\nu} + \frac{1}{|\ln \nu|}  \|\nabla \vep\|_{\omega_\nu} \nu^2 \lesssim \nu^2 \big( \|\vep\|_{\omega_\nu} + \|\nabla \vep\|_{\omega_\nu}\big).
\end{align*}
 We then arrive at the estimate 
 \begin{equation}\label{est:tmpL2EModa}
 |\textup{Mod}_a| |\langle \nabla U_{1, \nu}  - \nabla U_{2, \nu}, \hat \vep \rangle_\ast| \lesssim  \nu^2 |\textup{Mod}_a|\big( \|\vep\|_{\omega_\nu} + \|\nabla \vep\|_{\omega_\nu}\big). 
 \end{equation}
It remains to estimate the scalar with $\Psi$ term. Similarly, we use \eqref{def:vephat}, \eqref{id:decomp_phi0},  \eqref{matchedscalarproduct1:continuity},  \eqref{id:projection-matched-scalar-LambdaU1nu}, \eqref{est:Psi_L2omega} and  \eqref{est:a0a1} to write 
\begin{align}
|\langle \Psi, \hat \vep \rangle_\ast| &\leq |\langle \Psi, \vep \rangle_\ast| + |a_0| |\langle \Psi, \phi_0 \rangle_\ast| \nonumber \\
& \leq  |\langle \Psi, \vep \rangle_\ast| + |a_0| \Big( \sum_\pm | \langle \Psi, \nu^{-2} \Lambda U_\nu(z \pm a) \chi_{\pm a, \zeta_*} \rangle_\ast | +   |\langle \Psi, \tilde \phi_0 \rangle_\ast| \Big)\nonumber\\
& \lesssim \|\Psi\|_{\omega_\nu} \|\vep\|_{\omega_\nu} + |a_0| \Big( \Big|\sum_{\pm} \int_{|z\pm a| \lesssim \zeta_*} \Psi \Big|  + \sqrt{|\ln \nu|} \|\Psi\|_{\omega_\nu}  + \|\Psi\|_{\omega_\nu}\|\tilde \phi_0\|_{\omega_\nu}\Big)\nonumber\\
& \lesssim \frac{\nu^2}{|\ln \nu|^2} \| \vep\|_{\omega_\nu} + \frac{1}{|\ln \nu|} \| \nabla  \vep\|_{\omega_\nu} \Big( \frac{\nu^2}{|\ln \nu|} + \frac{\nu^2}{\sqrt{|\ln \nu|}} + \frac{\nu^2}{|\ln \nu|^2} \Big)\nonumber\\
& \lesssim \frac{\nu^2}{|\ln \nu|^2} \big( \| \vep\|_{\omega_\nu} + \| \nabla \vep\|_{\omega_\nu}\big).  \label{est:tmpL2EPsi}
\end{align}
Summing up \eqref{est:tmpL2EMod1}, \eqref{est:tmpL2EModa},  \eqref{est:tmpL2EPsi} yields 
\begin{equation}\label{est:PsivepL2tmp}
|\langle E, \hat \vep\rangle_\ast| \lesssim \big( \| \vep\|_{\omega_\nu} + \| \nabla \vep\|_{\omega_\nu}\big) \Big( |\textup{Mod}_1| +  \nu^2 |\textup{Mod}_a| + \frac{\nu^2}{|\ln \nu|^2}\Big).
\end{equation}

\noindent Plugging all the estimates \eqref{est:LvepL2tmp}, \eqref{est:NLvepL2tmp}, \eqref{est:PsivepL2tmp}, \eqref{est:spectralgap_hateps}  into \eqref{eq:hatvep} and using Cauchy-Schwarz with $\epsilon$ and the equivalence \eqref{est:equivvep_vephat}  yield
\begin{align*}
 \frac{1}{2} \frac{d}{d\tau} \langle \hat \vep, \hat \vep \rangle_\ast &\leq \Big(-\delta  + \Oc\big( \frac{|\nu_\tau| + |a_\tau|}{\nu}\big) \Big)\langle \hat \vep, \hat \vep \rangle_\ast  -\delta \Big( \|\nabla \vep\|_{\omega_\nu}^2 + \| \langle z \rangle \vep\|_{\omega_\nu}^2 + \sum_\pm \big\|  \frac{\vep}{\nu + |z \pm a|}\big\|^2_{\omega_\nu} \Big) \\
& \qquad + \frac{C}{|\ln \nu|} \Big(\|\nabla \vep\|^2_{\omega_\nu} + \|\vep\|^2_{\omega_\nu} + \sum_\pm \Big\|\frac{\vep}{\nu + |z \pm a|} \Big\|_{\omega_\nu}^2  \Big) + C\frac{\nu^4}{|\ln \nu|^4} + \frac{CK^5 \nu^4}{|\ln \nu|^{9/2}}\\
& \qquad + C\big( \| \vep\|_{\omega_\nu} + \| \nabla \vep\|_{\omega_\nu}\big) \Big( |\textup{Mod}_1| +  \nu^2 |\textup{Mod}_a| + \frac{\nu^2}{|\ln \nu|^2}\Big)\\
& \leq \Big(-\frac{\delta}{2}  + \Oc\big( \frac{|\nu_\tau| + |a_\tau|}{\nu}\big) \Big)\langle \hat \vep, \hat \vep \rangle_\ast -\frac{\delta}{2} \Big( \|\nabla \vep\|_{\omega_\nu}^2 + \| \langle z \rangle \vep\|_{\omega_\nu}^2 + \sum_\pm \big\|  \frac{\vep}{\nu + |z \pm a|}\big\|^2_{\omega_\nu} \Big)\\
& \qquad +  C\Big( |\textup{Mod}_1| +  \nu^2 |\textup{Mod}_a| + \frac{\nu^2}{|\ln \nu|^2}\Big)^2\\
& \leq -\frac{\delta}{4}\langle \hat \vep, \hat \vep \rangle_\ast  +  C\Big( |\textup{Mod}_1| +  \nu^2 |\textup{Mod}_a| + \frac{\nu^2}{|\ln \nu|^2}\Big)^2,
\end{align*}
which is the desired estimate \eqref{est:energyL2_form}. This concludes the proof of Lemma \ref{lemm:L2energy}.
\end{proof}

\subsection{$H^1$ estimate in the inner region} \label{subsec:H1-energy-estimate}
We derive in this section a monotonicity formula to control the remainder $\vep(z, \tau)$ in the inner region $|z \pm a|\ll 1$. In particular, we aim at improving the bootstrap bound \eqref{bootstrap:H1omega_0} thanks to the following lemma.
\begin{lemma}[$H^1$-estimate]  \label{lemm:H1energy} Let $\vep$ be a solution to \eqref{eq:vep} and satisfy the bootstrap bounds in Definition \ref{def:bootstrap} for $\tau \in [\tau_0, \tau_1]$. Then, we have the following estimate for $\tau \in [\tau_0, \tau_1]$:
\begin{equation}\label{est:qinH1}
\frac{d}{d\tau} \|q(\tau)\|^2_{H^1_\inn} \leq -\frac{1}{\zeta_*}\|q(\tau)\|^2_{H^1_\inn} + C\Big( \| \vep(\tau)\|^2_{bd} + \|\vep(\tau)\|^2_{L^2_{\omega_\nu}} + \frac{ \nu^4}{|\ln \nu|^4}\Big).
\end{equation}
for some  $C = C(\zeta_*, \zeta^*) > 0$. 
\end{lemma}
\begin{proof} To ease the notation, we drop $\pm$ in the definition \eqref{def:qpminn} of $q^\inn_\pm(y_\pm, s)$ in this proof and write 
$$q(y,s) = \nu^2 \vep (z, \tau), \quad q^\inn(y,s) = q(y,s)\chi_*(y), \quad y = \frac{z - a}{\nu}, \quad \frac{ds}{d\tau} = \frac{1}{\nu^2},$$
where $\chi_*$ is introduced in \eqref{def:chi-star}. The case $y = \frac{z + a}{\nu}$ is handled in the same way. \\
\paragraph{Energy identities:} From \eqref{eq:vep}, we write the equation for $q(y, s)$, 
\begin{equation}\label{eq:qysH1}
\pa_s q = \Ls_0 q + L_1 q - NL(q) +  \nu^4 E,
\end{equation}
where $\Ls_0$ is defined in \eqref{def:Ls0Ms0} and $L_1$ is a small linear term
\begin{align}
L_1(q) & = \nu^2\Big( \frac{\nu_\tau}{\nu} - \beta \Big) \Lambda q + \Big(\nu a_\tau - \beta \nu a - \nabla \Phi_U(y + 2a/\nu)  \Big).\nabla q \nonumber\\
& \quad - \nabla.\big( q \nabla \Phi_{\nu^2 \tilde W \chi^*} + \nu^2\tilde W \nabla \Phi_{q\chi^*} \big) - \nabla \big(  \nu^2 \tilde W \nabla_{\tilde q (1 - \chi^*)} \big), \label{def:L1qys}
\end{align}
and $NL(q)$ is the nonlinear term 
\begin{equation}\label{def:NLq}
NL(q) = \nabla \cdot\big(q \nabla_{q \chi^*} + q \nabla \Phi_{\tilde q (1 - \chi^*)}\big), \quad \tilde q(y,s) = \nu^2 \tilde{\vep}(z, \tau).
\end{equation}
We multiply the equation of $q$ by $\chi_*$, take a scalar with $\frac{\nu^2}{\zeta_*^2} \Ms_0 q^\inn$ and $-\Ms_0 \Ls_0 q^\inn$ respectively, and use the fact that $\Ms_0$ is symmetric to get 
\begin{align}
\frac{1}{2}\frac{d}{ds}& \left\{- \int q_2^\inn\Ms_0 q^\inn + \frac{\nu^2}{\zeta_*^2}\int q^\inn \Ms_0 q^\inn \right\}  = - \int q_2^\inn \Ms_0 q_2^\inn +\frac{\nu^2}{\zeta_*^2} \int q_2^\inn \Ms_0 q^\inn \qquad \big(q_2^\inn = \Ls_0 q^\inn\big) \nonumber \\
& \qquad \quad  - \int q_2^\inn \Ms_0\big[\big( L_1(q) - NL(q) + \nu^4 E \big)\chi_*\big] \nonumber\\
& \qquad \quad + \frac{\nu^2}{\zeta_*^2} \int q^\inn \Ms_0\big[\big( L_1(q) - NL(q) + \nu^4 E \big)\chi_*\big]  - \frac{\nu^4}{\zeta_*^2} \frac{\nu_\tau}{\nu} \int q^\inn \Ms_0 q^\inn. \label{id:H1estimate}
\end{align}
The local orthogonality condition \eqref{eq:orthog} and Lemma \ref{lem:coercivity-one-bubble-H1-orthogonality} yields the coercivity
\begin{equation}
-\int q_2^\inn\Ms_0 (q^\inn) = \int U |\nabla \Ms_0(q^\inn)|^2 dy \gtrsim  \int \frac{|\nabla (q^\inn)|^2}{U} dy,
\end{equation}
and recall from \eqref{bd:coercivity-one-bubble-H1},
\begin{equation}
\int q^\inn \Ms_0 q^\inn dy \gtrsim \int \frac{|q^\inn|^2}{U}dy.
\end{equation}
We have by Hardy inequality and the support of $q^\inn$, 
\begin{equation}
-\int \frac{|\nabla q^\inn|^2}{U} dy \lesssim - \int \frac{|q^\inn|^2}{(1 + |y|^2) U} dy  \lesssim -\frac{\nu^2}{\zeta_*^2} \int \frac{|q^\inn|^2}{U}. 
\end{equation}
We recall the fact that $\langle f, \Ms_0 f\rangle_{L^2} \geq 0$ if $\int f = 0$  (see again Proposition 2.3 in \cite{RSma14}). The divergence form of $\Ls_0$ already yields the orthogonality $\langle q_2^\inn, 1 \rangle_{L^2} = 0$, hence, 
\begin{equation*}
\int q_2^\inn \Ms_0 q_2^\inn dy \geq 0. 
\end{equation*}
Collecting these inequalities and using the a priori bound $|\nu_\tau/\nu| \lesssim 1/|\ln \nu| \ll 1$ for $\nu(\tau_0) \ll 1$, we rewrite the energy identity \eqref{id:H1estimate} as 
\begin{align}
\frac{1}{2}\frac{d}{ds}& \left\{\int U \big| \nabla \Ms_0 q^\inn \big|^2 + \frac{\nu^2}{\zeta_*^2}\int q^\inn \Ms_0 q^\inn \right\} \leq -\delta\Big( \frac{\nu^2}{ \zeta_*^2} \int \frac{|\nabla q^\inn|^2}{U} + \frac{\nu^4}{\zeta_*^2} \int \frac{|q^\inn|^2}{U} \Big) \nonumber\\
& \qquad \quad  + \left|\int q_2^\inn \Ms_0\big[\big( L_1 (q) - NL(q) + \nu^4 E \big)\chi_*\big] \right| \nonumber\\
& \qquad \quad + \frac{\nu^2}{\zeta_*^2} \left|\int \Ms_0q^\inn \big[\big( L_1(q) - NL(q) + \nu^4 E \big)\chi_*\big] \right|. \label{id:H1estimate1}
\end{align}
Thanks to the dissipation, we now estimate all the remaining terms by using integration by parts, Cauchy-Schwarz inequality and the bootstrap bounds in Definition \ref{def:bootstrap}. We emphasize that  the algebraic  identities $ \Ms_0 \Lambda U = -2$ and $\Ms_0 \nabla U = 0$ and the local orthogonality \eqref{eq:orthog} $\langle q^\inn, 1 \rangle_{L^2_{loc}} = 0$ are crucial in getting a better estimate once the generated error term $E \chi_*$ is concerned. \\
\paragraph{The scaling linear term $\Lambda q$:} We first write
\begin{align*}
\Ms_0 (\Lambda q \chi_*) &= \Ms_0(\Lambda q^\inn - y.\nabla \chi_* q) \\&= y.\nabla \Ms_0 q^\inn + 2\Ms_0 q^\inn  + \frac{q^\inn}{U} y.\nabla \Phi_U - y.\nabla \Phi_{q^\inn} - \Phi_{y.\nabla q^\inn} - \Ms_0(y.\nabla \chi_* q).  
\end{align*}
We then estimate the highest order term by using the divergence form $q_2^\inn = \nabla. (U \nabla \Ms_0 q^\inn)$ and integration by parts
\begin{align*}
\left| \int q_2^\inn y .\nabla \Ms_0 q^\inn + 2 \int q_2^\inn \Ms_0 q^\inn \right| &\leq \left|\int U \nabla \Ms_0 q^\inn . \nabla (y . \nabla \Ms_0 q^\inn ) \right| + 2\int U|\nabla \Ms_0 q^\inn|^2 \\
&\lesssim \int U|\nabla \Ms_0 q^\inn|^2 + \int U |\nabla \Ms_0 q^\inn|^2 \Big| \frac{\nabla.(yU)}{U} \Big| \lesssim \int U|\nabla \Ms_0 q^\inn|^2.
\end{align*}
Using the round bound $|y.\nabla \Phi_U| \lesssim 1$, Cauchy-Schwarz and Hardy inequality, we estimate
\begin{align*}
\left|\int q_2^\inn \frac{q^\inn}{U} y.\nabla \Phi_U \right| &=\left| \int U \nabla \Ms_0 q^\inn . \nabla \big[ (\Ms_0 q^\inn + \Phi_{q^\inn}) y.\nabla \Phi_U\big] \right|\\
&\lesssim \int U|\nabla \Ms_0 q^\inn|^2 + \int U|\nabla \Phi_{q^\inn}|^2 + \int U |\Phi_{q^\inn}|^2|\nabla (y.\nabla \Phi_U)|^2\\
& \lesssim \int U|\nabla \Ms_0 q^\inn|^2.
\end{align*}
Similarly, we estimate for lower order terms by Cauchy-Schwarz and Hardy inequality,
\begin{align*}
\left|\int q_2^\inn \big(y.\nabla \Phi_{q^\inn} + \Phi_{y.\nabla q^\inn} \big) \right| & \lesssim  \int U|\nabla \Ms_0 q^\inn|^2. 
\end{align*}
As for the boundary term, we use the $H^2$-bound \eqref{bootstrap:H2boundary} to estimate 
\begin{align*}
\left| \int q_2^\inn \Ms_0 (y .\nabla \chi_* q)  \right| &\lesssim \int U|\nabla \Ms_0 q^\inn|^2 + \int U |\nabla \Ms_0 (y.\nabla \chi_* q)|^2 \\
& \lesssim \int U|\nabla \Ms_0 q^\inn|^2 + \int_{\nu |y| \sim \zeta_*} \frac{|\nabla q|^2}{U} + \frac{|q|^2}{|y|^2U} + U|\nabla \Phi_{y.\nabla \chi_* q}|^2 \\
& \lesssim \int U|\nabla \Ms_0 q^\inn|^2 + \frac{\zeta_*^4}{\nu^4}\|\nabla q\|^2_{L^2(|\nu y| \sim \zeta_*)} +  \frac{\zeta_*^2}{\nu^2}\|q\|^2_{L^2(|\nu y| \sim \zeta_*)} + \frac{\nu^4}{\zeta_*^4}\|\nabla \Phi_{y.\nabla \chi_* q}\|^2_{L^2(\nu|y| \sim \zeta_*)}\\
& \lesssim  \int U|\nabla \Ms_0 q^\inn|^2 + C(\zeta_*)\|\vep(\tau)\|^2_{bd}. 
\end{align*}
A collection of all the estimates and recall that $|\nu_\tau/\nu| \lesssim 1/|\ln \nu|$, we derive 
\begin{align}
\left|\nu^2 \Big( \frac{\nu_\tau}{\nu} - \beta \Big) \int q_2^\inn \Ms_0 (\Lambda q \chi_*) \right| &\lesssim  \nu^2 \int U |\nabla \Ms_0 q^\inn|^2 + C(\zeta_*) \nu^2 \| \vep(\tau)\|^2_{bd}. \label{est:H1scalingterm1}
\end{align}
The estimate for $\int q^\inn \Ms_0(\Lambda q \chi_*)$ is proceeded in the same manner by integration by parts and Cauchy-Schwarz inequality, where we end up with 
\begin{align}
\frac{\nu^2}{\zeta_*^2} \left|\nu^2 \big( \frac{\nu_\tau}{\nu} - \beta \big) \int \Ms_0 q^\inn \Lambda q \chi_* \right| \lesssim \frac{\nu^4}{\zeta^2_*} \left(\int \frac{|q^\inn|^2}{U}  + C(\zeta_*) \nu^{-2} \|\vep(\tau)\|_{bd}^2 \right).\label{est:H1scalingterm2}
\end{align}

\paragraph{The linear gradient term $\nabla q$:} We write
$$\Ms_0(\nabla q \chi_*) = \Ms_0 (\nabla q^\inn - y \nabla \chi_* q) = \nabla \Ms_0 q^\inn + \frac{q^\inn}{U}\nabla \Phi_U + \nabla \Phi_{q^\inn} - \Ms_0(\nabla \chi_* q), $$
and from the definition of $\textup{Mod}_a$ and Taylor expansion of $\nabla \Phi_U(y + 2a/\nu) = \nabla \Phi_U(2a/\nu) + \nu^2 F_1(y)$ with $|F_1(y)| + |y \nabla . F_1(y)| \lesssim 1 + |y|$ for $|y| \leq \frac{2\zeta_*}{\nu}$,
\begin{align*}
\nu a_\tau - \beta a \nu - \nabla \Phi_U(y + 2a/\nu) = \nu \textup{Mod}_a + \nu^2 F_1(y),
\end{align*}
where we recall from \eqref{eq:Moda} and  \eqref{est:nablaPhiXia} the bounds 
$$|\nu \textup{Mod}_a| \lesssim \|q(\tau)\|_\inn + \|\vep\|_{\omega_\nu} + \frac{\nu^2}{|\ln \nu|^2}.$$
We estimate by integration by parts
\begin{align*}
\left| \int q_2^\inn F_1(y) . \nabla \Ms_0 q^\inn \right| &\lesssim  \left| \int \nabla.(U \nabla \Ms_0 q^\inn) F_1(y). \nabla \Ms_0 q^\inn \right|  \lesssim \int U|\nabla \Ms_0 q^\inn|^2 \frac{|\nabla . (F_1 U)|}{U} \lesssim  \int U|\nabla \Ms_0 q^\inn|^2. 
\end{align*}
The lower order terms and the boundary term are estimated by using integration by parts, Hardy inequality and $H^2$-bound \eqref{bootstrap:H2boundary} as for the scaling term, 
\begin{align*}
\left| \int q_2^\inn F_1. \big( \frac{q^\inn}{U}\nabla \Phi_U + \nabla \Phi_{q^\inn} - \Ms_0(\nabla \chi_* q \big) \right| \lesssim \int U|\nabla \Ms_0 q^\inn|^2 + C(\zeta_*)\|\vep(\tau)\|_{bd}^2. 
\end{align*}
Collecting all the bounds yields 
\begin{align}
&\left|\int q_2^\inn \Ms_0 \big[ (\nu a_\tau - \beta a \nu - \nabla \Phi_U(y + 2a/\nu)).\nabla q \chi_*\big]\right| \nonumber \\
& \qquad \qquad \lesssim \big(\|q(\tau)\|_\inn + \|\vep\|_{\omega_\nu} + \frac{\nu^2}{|\ln \nu|^2} \big) \Big(\int U|\nabla \Ms_0 q^\inn|^2 + C(\zeta_*)\|\vep(\tau)\|^2_{bd}\Big). \label{est:H1nablaq1}
\end{align}
Similarly, we have
\begin{align}
&\frac{\nu^2}{\zeta_*^2} \left|\int \Ms_0 q^\inn  \big[ (\nu a_\tau - \beta a \nu - \nabla \Phi_U(y + 2a/\nu)).\nabla q \chi_*\big]\right|\nonumber  \\
& \qquad \qquad \lesssim \frac{\nu^2}{\zeta_*^2} \big(\|q(\tau)\|_\inn + \|\vep\|_{\omega_\nu} + \frac{\nu^2}{|\ln \nu|^2} \big) \Big(\int \frac{|q^\inn|^2}{U} + C(\zeta_*)\nu^{-2}\|\vep(\tau)\|^2_{bd}\Big). \label{est:H1nablaq1}
\end{align}
\paragraph{The small linear term:} From the definition \eqref{def:Wtilde}, the pointwise estimate \eqref{est:pointwise_Xi} of $\Xi$, the improved pointwise estimate \eqref{est:pointwise_phi1m0}, and the bound $|\alpha| \lesssim \nu^2$, we have the round bound
$$\nu^2 \big( (y.\nabla)^2 \tilde W(\nu y) +  |y. \nabla \tilde W(\nu y)| +  |\tilde W(\nu y)|\big) \lesssim \frac{\nu^2}{1 + |y|^2}, \quad \nu^2|\nabla \Phi_{\chi^* \tilde W(\nu y)}| \lesssim \nu^3 |\ln \nu|. $$
We write 
\begin{align*}
\big[\nabla.(q \nabla \Phi_{\tilde W \chi^*}) + \nabla.(\tilde W \nabla \Phi_{q \chi^*})\big]\chi_* = \nabla q^\inn. \nabla \Phi_{\tilde W \chi^*} - q \nabla \chi_*. \nabla \Phi_{\tilde W \chi^*} - 2 q^\inn \tilde W + \nabla \tilde W . \nabla \Phi_{q \chi^*} \chi_{_*}. 
\end{align*}
The highest order term $\nabla q^\inn.\nabla \Phi_{\tilde W \chi^*}$ is simply estimated by integration by parts,
\begin{align*}
\left|\nu^2 \int q_2^\inn \Ms_0 (\nabla q^\inn.\nabla \Phi_{\tilde W \chi^*})  \right| + \frac{\nu^4}{\zeta_*^2} \left|\int \Ms_0 q^\inn \nabla q^\inn.\nabla \Phi_{\tilde W \chi^*}  \right| \lesssim \nu^2 \int U|\nabla \Ms_0 q^\inn|^2 + \frac{\nu^4}{\zeta_*^2}\int \frac{|q^\inn|^2}{U}.
\end{align*} 
The lower and boundary terms are estimated by integration by parts, $H^2$-bound \eqref{bootstrap:H2boundary} and Hardy inequality in the same manner as for the scaling term, 
\begin{align*}
&\left|\nu^2 \int q_2^\inn \Ms_0 \big[- q \nabla \chi_*. \nabla \Phi_{\tilde W \chi^*} - 2 q^\inn \tilde W +\nabla \tilde W . \nabla \Phi_{q \chi^*} \chi_*\big]  \right| \\
& \qquad + \left| \frac{\nu^4}{\zeta_*^2} \int \Ms_0 q^\inn \big[- q \nabla \chi_*. \nabla \Phi_{\tilde W \chi^*} - 2 q^\inn \tilde W +\nabla \tilde W . \nabla \Phi_{q \chi^*} \chi_*\big]  \right|\\
& \qquad \qquad  \lesssim \nu^2 \left(\int U |\nabla \Ms_0 q^\inn|^2 + C(\zeta_*)\|\vep(\tau)\|_{bd}^2\right) + \frac{\nu^4}{\zeta_*^2}\left(\int \frac{|q^\inn|^2}{U} + C(\zeta_*)\nu^{-2}\|\vep(\tau)\|_{bd}^2 \right).
\end{align*} 
For the Poison term in the outer region, we estimate using \eqref{bootstrap:H2boundary}, \eqref{bootstrap:Linf_outer} and Lemma \ref{lemm:PointwisePoisonField} with $\gamma = 3/2$ and $\sigma = 1/4$
\begin{equation} \label{est:Phiqout}
|\chi_{_*} \nabla_y \Phi_{\tilde q(1 - \chi_{_*})}| \lesssim \nu |\chi_{\zeta_*, a} \nabla_z \Phi_{\vep(1 - \chi_{\zeta_*,a})}| \lesssim \nu^\frac{3}{4} \|\vep(\tau)\|_\out (|y|^{-1/4}), 
\end{equation}
From this and Cauchy-Schwarz inequality, we get 
\begin{align*}
\left|\nu^2\int q_2^\inn \Ms_0\big(\nabla \tilde W. \nabla \Phi_{\tilde q(1 - \chi_*)} \chi_* \big)\right| \lesssim \nu^2 \int U|\nabla \Ms_0 q^\inn|^2  +\nu^3\|\vep(\tau)\|_{\out}^2.
\end{align*}
and 
\begin{align*}
\left|\frac{\nu^4}{\zeta_*^2}\int \Ms_0 q^\inn \nabla \tilde W. \nabla \Phi_{\tilde q(1 - \chi_*)} \chi_* \right| \lesssim \frac{\nu^4}{\zeta_*^2} \int \frac{|q^\inn|^2}{U} + C(\zeta_*)\nu^3\|\vep(\tau)\|_{\out}^2.
\end{align*}
Collecting all these estimates yields the estimate for the small linear term
\begin{align}
&\left|\nu^2 \int q_2^\inn \Ms_0 \Big(\big[ (\nabla .(q \nabla \Phi_{\tilde W \chi_*}) + \nabla.(\tilde W \Phi_{q \chi_*}) + \nabla .(\tilde{W}\nabla \Phi_{\tilde q (1 - \chi^*)} ) \big] \chi_* \Big)\right| \nonumber \\
& \qquad + \left|\frac{\nu^4}{\zeta_*^2} \int \Ms_0 q^\inn \big[ (\nabla .(q \nabla \Phi_{\tilde W \chi_*}) + \nabla.(\tilde W \Phi_{q \chi_*}) + \nabla .(\tilde{W}\nabla \Phi_{\tilde q (1 - \chi^*)} )\big] \chi_*\right| \nonumber \\
& \qquad \qquad \lesssim \nu^2 \int U|\nabla \Ms_0 q^\inn|^2 + \frac{\nu^4}{\zeta_*^2}\int \frac{|q^\inn|^2}{U} + C(\zeta_*) \big( \nu^2\|\vep(\tau)\|_{bd} + \nu^3 \|\vep(\tau)\|_{\out}\big). \label{est:H1smalllinearterm}
\end{align}

\paragraph{The generated error term $\nu^4 E$:} By using the algebraic identities $ \Ms_0 \Lambda U = -2, \Ms_0 \nabla U = 0$ and the local orthogonality $\langle q, 1 \rangle_{L^2_{loc}} = 0$, we cancel out the leading order term $\Lambda U$ and $\nabla U$ of $\nu^4 E$ that crucially gains the factor $\nu^2$.  We recall from the definition \eqref{def:phii} of $\phi_0$ and $\phi_1$, and $\chi_*(y) = \chi_{\zeta_*/\nu}(y) =  \chi_{a, \zeta_*}(z)$, 
\begin{equation*}
\nu^4 \phi_i(z) \chi_* = \Big(-\frac{1}{16}\Lambda U(y)  \pm \nu L_i \pa_1 U(y) \Big)\chi_*(y) + \nu^4\tilde{\phi}_i(z) \chi_*(y),
\end{equation*}
where we have by \eqref{est:pointwise_phiitil_da}
\begin{equation}\label{est:phitilE}
k \in \mathbb{N}, \quad   |\nu^4 \nabla^k_y \tilde{\phi}_i(z)| = |\nu^{4 + k} \nabla_z \tilde \phi_i(z)| \lesssim \frac{\nu^2}{1 + |y|^{2+k}}, \quad \quad |y| \leq \frac{2\zeta_*}{\nu}. 
\end{equation}
From the cancellation $\nabla \Ms_0 \Lambda U = 0$ and $\Ms_0 \nabla U= 0$, integration by parts and Cauchy-Schwarz inequality, we estimate
\begin{align*}
\left| \int q_2^\inn \Ms_0 (\nu^4 E \chi_*) \right| &\lesssim \sum_{i=0}^1 |\Mod_i | \int \left|U \nabla \Ms_0 q^\inn \right| \big| \nabla \Ms_0\big(\Lambda U(1 - \chi_*) + \nu \nabla \Ms_0\big(\pa_1U(1 - \chi_*)  + \nu^4 \tilde \phi_i \chi_* \big) \big|\\
& \quad +  |\Mod_a| \int \big|U \nabla \Ms_0 q^\inn \big| \big|\nabla \Ms_0\big(\nu \nabla U(y) (1 - \chi_*) - \nu \nabla U(y + 2a/\nu) \chi_* + \nu^5\nabla \Xi \chi_* \big) \big|\\
& \qquad  \int \big|U \nabla \Ms_0 q^\inn \big| \big|\nabla \Ms_0 (\nu^4\Psi \chi_*) \big|\\
& \lesssim \nu^2 \int U|\nabla \Ms_0 q^\inn|^2 + \sum_{i =0}^1 \nu^{-2} |\Mod_i|^2 \int U\Big|\nabla \Ms_0\big(\Lambda U(1 - \chi_*) + \nu \nabla \Ms_0\big(\pa_1U(1 - \chi_*) + \nu^4 \tilde \phi_i \chi_* \big)\Big|^2 \\
& \qquad + |\Mod_a|\int U \Big|\nabla \Ms_0\big(\nabla U(y) (1 - \chi_*) - \nabla U(y + 2a/\nu) \chi_* + \nu^4\nabla \Xi \chi_* \big) \Big|^2\\
& \qquad + \int U\big|\nabla \Ms_0(\nu^3 \Psi \chi_*)\big|^2\\
&\lesssim \nu^2 \int U|\nabla \Ms_0 q^\inn|^2 + \nu^2 \sum_{i = 0}^1|\Mod_i|^2 + \nu^3 |\Mod_a|^2 + \frac{\nu^6}{|\ln \nu|^4},
\end{align*}
where we used in the last line the decays $|\Lambda U(y)| \lesssim (1 + |y|)^{-4}$ and $|\nabla U(y)|\lesssim (1 + |y|)^{-5}$, the pointwise estimate \eqref{est:pointwise_Xi} $|\nabla_y \Xi(z) \chi_*| = \nu|\nabla_z \Xi(z) \chi_{a, \zeta_*}(z)| \lesssim (1 + |y|)^{-3}$ and \eqref{est:PsiinL2omega}.
From Lemma \ref{lemm:mod}, we thus derive 
\begin{equation}\label{est:errortermH1}
\left| \int q_2^\inn \Ms_0 (\nu^4 E \chi_*) \right| \lesssim \nu^2 \int U|\nabla \Ms_0 q^\inn|^2 + \nu^2\| \vep(\tau)\|^2_{L^2_{\omega_\nu}} + \nu^2\|q(\tau)\|^2_{H^1_\inn} + \frac{\nu^6}{|\ln \nu|^4}.
\end{equation}
The estimate for $\frac{\nu^2}{\zeta_*^2} \int q^\inn \Ms_0 (\nu^4 E)$ is to use the cancellation $\Ms_0 \Lambda U = -2$, $\Ms_0 \nabla U = 0$ together with the local orthogonality $\langle q^\inn, 1 \rangle_{L^2} = 0$ to cancel out the leading order terms in $\nu^4 E$. The remaining terms are estimated similarly and we omit it here. \\

\paragraph{The nonlinear term:} We write from \eqref{def:NLq}  and use $ \chi_* \chi^* = \chi_*$, $ \chi_*(1 - \chi^*) = 0$, 
\begin{align*}
NL(q) \chi_* &= \nabla \cdot \big[q \nabla \big(\Phi_{q \chi^*} + \Phi_{\tilde q(1 - \chi^*)} \big) \big] \chi_* \\
& =  \nabla q^\inn . \nabla \big(\Phi_{q \chi^*} + \Phi_{\tilde q(1 - \chi^*)} \big) - q\nabla \chi_*.\nabla \big(\Phi_{q \chi^*} + \Phi_{\tilde q(1 - \chi^*)} \big) - q^\inn q. 
\end{align*}
We first get the bound from the definition of $\nabla \Phi_f$, Cauchy-Schwarz inequality and the $H^2$-bound \eqref{bootstrap:H2boundary},
\begin{align*}
|\nabla \Phi_{q \chi^*}(y)| &\lesssim |\nabla \Phi_{q \chi_*}| + \nu|\nabla_z \Phi_{\vep (1 - \chi_*)\chi^*}| \\
& \lesssim \int |\ln |x - y| \; |\nabla q^\inn(x)|  dx + \nu\int |\ln |x - z| \; |\nabla_x \vep (1 - \chi_*) \chi^*| dx\\
& \lesssim \left(\int \frac{|\nabla q^\inn|^2}{U}\right)^\frac{1}{2} + C(\zeta_*, \zeta^*)\nu \|\vep(\tau)\|_{bd} \lesssim \|q(\tau)\|_{H^1_\inn} +  C(\zeta_*, \zeta^*, K)\frac{\nu^3}{|\ln \nu|}. 
\end{align*}
Using this bound, the definition of $\Ms_0$, Cauchy-Schwartz inequality, the estimate $\|\Phi_{\nabla q^\inn}\|_\infty^2 \lesssim \int \frac{|\nabla q^\inn|^2}{U}$, the bootstrap bounds \eqref{bootstrap:H1omega_0}, \eqref{bootstrap:H2omega_0} and \eqref{est:c1c2} to obtain
\begin{align*}
\left| \int q_2^\inn \Ms_0(\nabla q^\inn . \nabla \Phi_{q \chi^*}) \right| & \lesssim \left| \int \frac{q_2^\inn \nabla q^\inn.\nabla \Phi_{q\chi^*}}{U} \right| + \left| \int q_2^\inn \Phi_{\nabla q^\inn.\nabla \Phi_{q\chi^*}}\right|\\
&\lesssim \| \nabla \Phi_{q \chi^*}\|_\infty  \left( \int \frac{|q_2^\inn|^2}{U} \right)^{\frac{1}{2}}\left[ \left( \int \frac{|\nabla q^\inn|^2}{U} \right)^\frac{1}{2} + \|\Phi_{|\nabla q^\inn|}\|_\infty\left(\int U  \right)^\frac{1}{2}\right]\\
& \lesssim \| \nabla \Phi_{q \chi^*}\|_\infty \left( \int \frac{|q_2^\inn|^2}{U} \right)^{\frac{1}{2}} \left( \int \frac{|\nabla q^\inn|^2}{U} \right)^\frac{1}{2}\lesssim  \frac{ C(\zeta_*, \zeta^*, K) \nu^6}{|\ln \nu|^{5}}.
\end{align*}
Using the bound \eqref{est:Phiqout}, we similarly estimate
\begin{align*}
\left| \int q_2^\inn \Ms_0(\nabla q^\inn . \nabla \Phi_{\tilde q(1- \chi^*)}) \right| & \lesssim \nu^\frac{3}{4}\| \vep(\tau)\|_\out  \left( \int \frac{|q_2^\inn|^2}{U} \right)^{\frac{1}{2}} \left( \int \frac{|\nabla q^\inn|^2}{U} \right)^\frac{1}{2} \lesssim  \frac{ C(\zeta_*, \zeta^*, K) \nu^{6 + 3/4}}{|\ln \nu|^4}.
\end{align*}
and the cut-off term 
\begin{align*}
\left| \int q_2^\inn \Ms_0 \big (q \nabla \chi_* . (\nabla \Phi_{q \chi^*} + \nabla \Phi_{\tilde q(1  - \chi^*)}  \big)   \right| &\lesssim  \| \nabla \Phi_{q \chi^*} + \nabla \Phi_{\tilde q(1  - \chi^*)}\|_\infty \left( \int \frac{|q_2^\inn|^2}{U} \right)^\frac{1}{2} \times \\
& \qquad \times  \left[ \left( \int \frac{|q y. \nabla \chi_*|^2}{y^2 U} \right)^\frac{1}{2} + \| \Phi_{q \nabla \chi_*}\|_\infty \left( \int U\right)^\frac{1}{2} \right]\\
&  \lesssim  \frac{ C(\zeta_*, \zeta^*, K) \nu^6}{|\ln \nu|^5},
\end{align*}
and the quadratic term by using the $\textup{supp}(\chi_*) = \{|y| \leq \frac{2\zeta_*}{\nu}\}$, the $L^\infty$ bound from Sobolev, \eqref{def:q2tildecomp} and \eqref{est:c1c2}, $\|q(s)\|_{L^\infty(|y| \leq 2\zeta_*/\nu)} \lesssim \|\tilde q(\tau)\|_{H^2_\inn} + \|q(\tau)\|_{H^1_\inn} + \|\vep(\tau)\|_{bd}$,
\begin{align*}
\left| \int q_2^\inn \Ms_0 (q^\inn q) \right| &\lesssim \int U |\nabla \Ms_0 q^\inn| \left( \frac{|\nabla q^\inn q|}{U} + \frac{q^\inn \nabla q}{U} - \frac{|q^\inn q|}{U}|\nabla \Phi_U| +  |\nabla \Phi_{q^\inn q}| \right)\\
& \lesssim \left(\int U|\nabla \Ms_0 q^\inn|^2 \right)^\frac{1}{2} \left[\| q\|_{L^\infty(|y| \leq \frac{2\zeta_*}{\nu})} \left( \int \frac{|\nabla q^\inn|^2}{U} \right)^\frac{1}{2} + \|q^\inn\|_\infty  \left( \int_{|y| \leq \frac{2\zeta_*}{\nu}} \frac{|\nabla q|^2}{U} \right)^\frac{1}{2} \right. \\
& \qquad \qquad  \cdots \left. + \left(\int \frac{|q^\inn|^2}{U}|\nabla \Phi_U|^2\right)^\frac{1}{2} + \| q\|_{L^\infty(|y| \leq \frac{2\zeta_*}{\nu})} \| \nabla \Phi_{q^\inn}\|_\infty \left( \int U \right)^\frac{1}{2}    \right]\\
& \lesssim \|q(\tau)\|_{H^1_\inn}^2\big( \|\tilde q(\tau)\|_{H^2_\inn} + \| q(\tau)\|_{H^1_\inn}+ \|\vep\|_{bd}  \big) \lesssim \frac{ C(\zeta_*, \zeta^*, K) \nu^6}{|\ln \nu|^5}. 
\end{align*}
The estimate for $\frac{\nu^2}{\zeta_*^2} \int \Ms_0 q^\inn NL(q)\chi_*$ is simpler, so we   omit it. Collecting all the estimates to the energy identity \eqref{id:H1estimate1} and taking $\zeta_* \ll 1$, $\tau_0 \gg 1$ such that $\frac{1}{\zeta_*} \gg 1$ and $\frac{1}{|\ln \nu|} \ll 1$ and changing the available $\frac{d\tau}{ds} = \nu^2$ yield the desired estimate \eqref{est:qinH1}. This concludes the proof of Lemma \ref{lemm:H1energy}.
\end{proof}

\subsection{$H^2$ energy estimate in the inner region} \label{subsec:H2-energy-estimate}
We perform a similar energy estimate to the precedent subsection. The extra condition we use here is the coercivity \eqref{est:coercityH2}. Note by scaling that the typical size of $\|q(\tau)\|_{H^2_\inn}$ is of size $\Oc(\nu^3/|\ln \nu|)$, but we don't need to go to this order as an $H^2_\inn$ estimate of size $\Oc(\nu^2/|\ln \nu|)$ is enough to get a round $L^\infty$ bound by Sobolev for $q^\inn$ of size $\Oc(\nu^2/|\ln \nu|)$  in order to  close the nonlinear estimate. We claim the following. 
\begin{lemma}[$H^2$-estimate]  \label{lemm:H2energy} Let $\vep$ be a solution to \eqref{eq:vep} and satisfy the bootstrap bounds in Definition \ref{def:bootstrap} for $\tau \in [\tau_0, \tau_1]$. Then, we have the following estimate for $\tau \in [\tau_0, \tau_1]$:
\begin{equation}\label{est:qinH1}
\frac{d}{d\tau} \|\tilde q(\tau)\|^2_{H^2_\inn} \leq -\delta_2 \|\tilde q(\tau)\|^2_{H^2_\inn} + C\Big( \| q(\tau)\|^2_{H^1_\inn} + \| \vep(\tau)\|^2_{bd} + \|\vep(\tau)\|^2_{L^2_{\omega_\nu}} + \frac{ \nu^4}{|\ln \nu|^4}\Big).
\end{equation}
for some $\delta_2 > 0$ and $C = C(\zeta_*, \zeta^*) > 0$. 
\end{lemma}
\begin{proof} We write from \eqref{eq:qysH1}mthe equation for $q_2^\inn = \Ls_0 q^\inn$ with $q^\inn = q \chi_*$,
\begin{equation}
\pa_s q_2^\inn = \Ls_0 q_2^\inn + \Ls_0 \Big[ [\chi_*, \Ls_0]q + \big(L_1 q -NL(q) + \nu^4 E \big)\chi_* \Big],
\end{equation}
where the commutator $[ \chi_*, \Ls_0] q = \chi_* \Ls_0 q - \Ls_0(q \chi_*)$. We then write the energy identity for  $\tilde{q}_2^\inn$ introduced by \eqref{def:q2tildecomp} by using $\Ms_0 \Lambda U = -2$, $\Ms_0 \nabla U = 0$, $\int q_2^\inn  = 0$, Cauchy-Schwarz inequality and the divergence form of $\Ls_0$,
\begin{align*}
\frac{1}{2}\frac{d}{ds} \int \tilde q_2^\inn \Ms_0 \Ls_0 \tilde q_2^\inn &= -\int U|\nabla \Ms_0 \tilde q^\inn_2|^2 + \int \Ms_0 \tilde q_2^\inn \Ls_0 \Big[ [\chi_*, \Ls_0]q + \big(L_1 q -NL(q) + \nu^4 E \big)\chi_* \Big]\\
& \lesssim -\frac{1}{2}\int U|\nabla \Ms_0 \tilde q_2^\inn|^2 + \int U \Big|\nabla \Ms_0 \Big( [\chi_*, \Ls_0]q + \big(L_1 q -NL(q) + \nu^4 E \big)\chi_* \Big) \Big|^2.
\end{align*}
We recall from Lemma \ref{lem:coercivity-one-bubble-H1-orthogonality} the coercivity estimate $-\int U|\nabla \Ms_0 \tilde q_2^\inn| \leq -\delta \int \frac{|\nabla \tilde q_2^\inn|^2}{U}$. We have by Hardy inequality and the support of $\chi_*$ the estimate $- \int \frac{|\nabla \tilde q_2^\inn|^2}{U} \lesssim -\frac{\nu^2}{\zeta_*^2} \int \frac{|\tilde q_2^\inn|^2}{U}$ which shows the dissipation to control all lower order terms. The estimate of all the terms in $\int U \big|\nabla \Ms_0 \big[ \big(L_1 q -NL(q) + \nu^4 E \big)\chi_*  \big]\big|^2$ follows the same lines as for Lemma \ref{lemm:H1energy}, where we emphasize again the crucial algebraic identities $\nabla \Ms_0 \Lambda U = \nabla \Ms_0 \nabla U = 0$ to gain a factor $\nu^2$ once the generated error term $\nu^4 E$ is concerned. The small linear term is estimated in the same way, except that there is one more derivative in comparison with the proof of Lemma \ref{lemm:H1energy} that we use in addition $\|q(\tau)\|_{H^1_\inn}$ bound. The nonlinear term is estimated by $\frac{C(\zeta_*, \zeta^*, K)\nu^6}{|\ln \nu|^5}$ thanks to the round $L^\infty$ bound $\|q^\inn(\tau)\|_\infty \lesssim C(K) \frac{\nu^2}{|\ln \nu|^2}$, hence, we always gain an extra small factor $\frac{1}{|\ln \nu|}$. The extra commutator term is located in the support of $\nabla \chi_*$ and is bounded by $C(\zeta_*)\nu^2 \|\vep(\tau)\|^2_{bd}$ thanks to the decay of $U(y) \lesssim (1 + |y|^4)^{-1}$. This concludes the proof of Lemma \ref{lemm:H2energy}.
\end{proof}

\subsection{Control in the middle and outer regions} \label{subsec:control-mid-outer}

We recall the two fixed positive constants 
$$\zeta^* \gg 1, \quad \zeta_* \ll 1,$$
and claim the following. 
\begin{lemma}[$H^2$ control of $\vep$ in the middle zone] \label{lemm:H2mid} Let $\vep$ be a solution to \eqref{eq:vep} and satisfy the bootstrap bounds in Definition \ref{def:bootstrap} for $\tau \in [\tau_0, \tau_1]$. Then, we have the following estimate for $\tau \in [\tau_0, \tau_1]$:
\begin{equation}
\|\vep(\tau)\|_{H^2(B_{\zeta^*} \setminus \cup_\pm B_{\zeta_*}(\pm a))} \leq C(\zeta^*, \zeta_*) K \frac{\nu^2}{|\ln \nu|^2}.
\end{equation}
\end{lemma}
\begin{proof} We note from the bound \eqref{bootstrap:L2omega_nu} and the definition of the weight function $\omega_{\nu}$ which is of size $\omega_\nu(z) \sim 1$ for $|z\pm a| \sim 1$, the $L^2$ estimate 
\begin{equation}
\|\vep(\tau)\|_{L^2(Q_4)} \leq C(\zeta_*, \zeta^*) K \frac{\nu^2}{|\ln \nu|^2}, \quad Q_{4} = \big\{ z \in  B_{4\zeta^*} \setminus \cup_\pm B_{\zeta_*/4}(\pm a)\big\}.
\end{equation}
From this $L^2$ bound and a standard parabolic regularity, see for example Lemma 4.7 in \cite{CGMNcpam21}, we obtain the desired estimate. 
\end{proof}

As for the control in the outer region, we use the maximum principle for a sake of simplicity even through other techniques such as standard-energy estimates gives a similar result. We claim the following. 
\begin{lemma}[Weighted $L^\infty$ control of $\tilde\vep$ in the outer zone] \label{lemm:Linfout} Let $\tilde \vep$ be a solution to \eqref{eq:veptil} and satisfy the bootstrap bounds in Definition \ref{def:bootstrap} for $\tau \in [\tau_0, \tau_1]$ and $|\tilde{\vep}(z, \tau_0)| \leq  \nu^2(\tau_0)\langle z \rangle^{-\frac{3}{2}}$ for $z \in \Rb^2$. Then, we have the following estimate for $\tau \in [\tau_0, \tau_1]$ and $|z| \geq \zeta^* \gg 1$:
\begin{equation}
\big|\tilde \vep(z,\tau)\big| \leq \frac 12 K^3 \nu^2 \, |z|^{-\frac{3}{2}}.
\end{equation}
\end{lemma}
\begin{proof} We rewrite equation \eqref{eq:veptil} as 
\begin{equation}\label{eq:vep_ex}
\tilde \vep_\tau = \Delta \tilde \vep - \big(\nabla \Phi_{U_{1+2, \nu}} + \frac{1}{2}z\big). \nabla \tilde \vep - \nabla U_{1+2, \nu}.\nabla \Phi_{\tilde \vep}  - (1 - 2U_{1+2, \nu})\tilde \vep + \tilde E(z,\tau). 
\end{equation}
We note from the definition of $U_{1+2, \nu}(z)$ and the bootstrap bound $|\nu_\tau/\nu| \leq K/|\ln \nu|$, the estimates for $|z| \geq \zeta^* \gg 1$, 
\begin{equation}\label{est:U12_DU12_Ez_out}
|U_{1+2, \nu}(z)| \lesssim \frac{\nu^2}{|z|^4}, \quad |\nabla U_{1+2, \nu}(z)| \lesssim \frac{\nu^2}{|z|^5}, \quad |\nabla \Phi_{U_{1+2, \nu}}(z)| \lesssim \frac{1}{|z|}, \quad |E(z,\tau)| \lesssim K\frac{\nu^2}{|z|^4}. 
\end{equation}
We note that there is $-\tilde \vep$ term in \eqref{eq:vep_ex} that gives a decay estimate, we simply consider the functions
$$\tilde h^*(z,\tau) = \frac{1}{2}K^3\nu^2 |z|^{-\frac{3}{2}}, \quad \tilde h_*(z,\tau) = -\frac{1}{2}K^3\nu^2 |z|^{-\frac{3}{2}},$$
that serves as a sub and super solutions to \eqref{eq:vep_ex}. The information we need to justify this fact is a pointwise estimate of the Poison field $\nabla \Phi_{\tilde h^*}$ and the boundary data $\tilde \vep(z, \tau)$ at $|z| = \zeta^*$ that can be retrieved from the a priori bootstrap bounds given in Definition \ref{def:bootstrap}. As for the estimate of $\nabla \Phi_{\tilde h^*}$, we have from Lemma \ref{lemm:PointwisePoisonField} with $\gamma = \frac{3}{2}$ the estimate 
\begin{equation}\label{est:Poison_out}
\forall |z| \geq \zeta^*, \quad |\nabla \Phi_{\tilde h^*}(z)| \lesssim \big\|\langle z \rangle^\frac{3}{2} \tilde h^*(z) \big\|_{L^\infty(\Rb^2)} |z|^{-\sigma} \lesssim K^3 \nu^2 |z|^{-\sigma}, \quad \sigma \in (0, \frac 12).
\end{equation}
From \eqref{est:U12_DU12_Ez_out} and \eqref{est:Poison_out}, we estimate  for $|z| \geq \zeta^*$,
\begin{align*}
&\frac{2}{\tilde h^*}\Big[\tilde h^*_\tau - \Delta \tilde h^* + \big(\nabla \Phi_{U_{1+2, \nu}} + \frac{1}{2}z\big). \nabla \tilde h^* + \nabla U_{1+2, \nu}.\nabla \Phi_{\tilde h^*}  + (1 - 2U_{1+2, \nu})\tilde h^* - \tilde E(z,\tau)\Big]\\
& \geq -\Big|\frac{2\nu_\tau}{\nu}\Big| - \frac{9}{4|z|^2} - \frac{C}{|z|^2} -  \frac{3}{4} - \frac{C\nu^2}{|z|^{7/2}} \frac{1}{|z|^\sigma} + 1 - \frac{C\nu^2}{|z|^4} - \frac{C}{K^3 |z|^{5/2}}\\
& \geq \frac{1}{4} - \frac{C}{|\ln \nu|} - \frac{C}{(\zeta^*)^2} > 0,
\end{align*}
for $\tau_0$ and $\zeta^*$ large enough. 
We check the boundary $|z| = \zeta^*$ from  the decomposition \eqref{def:AppSol-2nddecomp}, $\tilde \vep = \tilde W  + \vep$, the bootstrap $H^2$-bound \eqref{bootstrap:H2boundary},  and the pointwise estimate on $\tilde W  = \Xi + \alpha (\phi_1 - \phi_0)$ for $|z| = \zeta^*$, 
$$\forall |z| = \zeta^*, \quad  |\tilde \vep(z, \tau)|  = | \tilde W (z, \tau) | + |\vep (z, \tau)| \leq  \frac{C \nu^2}{|z|^2} +  K^3 \frac{\nu^2}{|\ln \nu|} < \frac{K^3 \nu^2}{2|z|^\frac{3}{2}}. $$
Similar estimates hold for $\tilde h_*$, and the desired estimate follows from a comparison principle. \end{proof}
\subsection{Proof of the main theorem} \label{subsec:proof-main-theorem}

The existence proof of the main theorem \ref{thm:main} follows once we show that the residue $\tilde{\vep}$ belongs to the bootstrap regime \ref{def:bootstrap} for all time $\tau \geq \tau_0$ provided that the initial data $\vep(z, \tau_0)$ is well prepared and belongs to the bootstrap regime \ref{def:bootstrap} as well.  We assume by contradiction that  
$$\tau^* = \sup\{\tau_1 \geq \tau_0 \; \textup{such that  the solution belongs to \ref{def:bootstrap} on} \;[\tau_0, \tau_1]\} < +\infty.$$
We will show that this is impossible by solving the modulation equations obtained in Lemma \ref{lemm:mod}, the various monotone energy estimates in Lemmas \ref{lemm:L2energy}, \ref{lemm:H1energy}, \ref{lemm:H2energy}, \ref{lemm:H2mid} and \ref{lemm:Linfout}. \\

\noindent - \textit{Improved estimates for modulation parameter functions:} We write from Lemma \ref{lemm:mod}, the bootstrap definition \eqref{def:bootstrap}, the eigenvalues $\lambda_0 = 1 + \frac{\gamma_0}{\ln \nu} + \Oc(|\ln \nu|^{-2})$ and $\lambda_1 = \frac{\gamma_1}{\ln \nu}  + \Oc(|\ln \nu|^{-2})$ from Proposition \ref{prop:eigen} and the fixed $\beta = \frac{1}{2}$,  the system for the three parameter functions $\nu(\tau), \alpha(\tau)$ and $a(\tau)$, 
\begin{align*}
& -16\nu^2 \Big( \frac{\nu_\tau}{\nu} - \frac{1}{2} - \frac{\gamma_0}{2\ln \nu} \Big) + \alpha_\tau - \Big(1 + \frac{\gamma_0}{\ln \nu})\alpha  = \Oc\Big( \frac{\nu^2}{|\ln \nu|^2}\Big),\\
&\frac{\alpha_\tau}{\alpha} - \frac{\gamma_1}{\ln \nu} = \Oc\Big( \frac{1}{|\ln \nu|^2}\Big), \qquad \qquad a_\tau =  \frac{a}{2} - \frac{8a}{\nu^2 + 4|a|^2} + \Oc\Big( \frac{\nu}{|\ln \nu|^2}\Big). 
\end{align*} 
The first equation can be rearranged as 
$$\Big( 8\nu^2 - \alpha\Big)\Big(1 + \frac{\gamma_0}{\ln \nu} \Big) + (\alpha - 8\nu^2)_\tau = \Oc\Big( \frac{\nu^2}{|\ln \nu|^2} \Big).$$
We note that the choice $\alpha = 8\nu^2$ makes the left-hand side vanish and the relation $\frac{\alpha_\tau}{\alpha} = 2\frac{\nu_\tau}{\nu}$. Thus, the second equation for $\alpha$ gives 
\begin{equation}\label{eq:nulaw}
  \Big|\frac{\nu_\tau}{\nu} -  \frac{\gamma_1}{2\ln \nu}\Big| \leq  \frac{C}{|\ln \nu|^2}.  
\end{equation}
where $C > 0$ is independent of the bootstrap constant $K$. Hence, we take $\bar K$ large for which the bootstrap estimate \eqref{bootstrap:nu} holds with a strict inequality. The solution to \eqref{eq:nulaw} can be explicitly obtained by integrating in time, 
$$\ln^2 \nu(\tau) = \gamma_1 \tau \Big(1 + \Oc\big(\frac{1}{\sqrt{\tau}}\big) \Big), \quad \nu(\tau) = e^{-\sqrt{\gamma_1 \tau} + \Oc(1)}.$$
In terms of the variable $t$, we recall that $\tau = - \ln(T-t)$, we then obtain the final blowup law as 
$$\lambda(t) = \sqrt{T-t}\nu(\tau) = C_0 \sqrt{T - t}e^{- \sqrt{\gamma_1 |\ln (T-t)|}} \quad \textup{as} \;\; t \to T.$$
We then introduce the linearization 
$$ \alpha = 8\nu^2 + \tilde \alpha, \quad a = \bar a + \tilde a, \quad \bar a = (2,0),$$
where $\bar a$ are the exact solutions to 
$$  \frac{2\bar a}{|a|^2} = \frac{\bar a}{2}.$$
The linearized equations for $\tilde \alpha$ and $\tilde{a}$ satisfy 
\begin{align}
&\tilde{\alpha}_\tau = \tilde \alpha \Big(1 + \frac{\gamma_0}{\ln \nu} \Big) + \Oc\Big( \frac{\nu^2}{|\ln \nu|^2} \Big),\label{ode:alphatil}\\
&\tilde{a}_\tau = \frac{\bar a}{2} \bar a . \tilde a + \Oc\Big(|\tilde a|^2 + \frac{\nu}{|\ln \nu|^2}\Big).\label{ode:atil}
\end{align}
These two ODEs do not have a decaying solution for arbitrary initial data, and we shall solve them by a Brouwer fixed point argument later.

\noindent - \textit{Improved energy bounds:} From Lemma \ref{lemm:L2energy}, \ref{lemm:mod} and the bootstrap bounds in \ref{def:bootstrap}, we have 
\begin{align*}
\frac{d}{d\tau}\langle \hat \vep, \hat \vep \rangle_\ast  &\leq - \delta_0 \langle \hat \vep, \hat \vep \rangle_\ast + \frac{C}{|\ln \nu|^2}\Big(\| q(\tau)\|_{H^1_\inn}^2 + \| \vep(\tau)\|_{\omega_\nu}^2 \Big) + \frac{C \nu^4}{|\ln \nu|^4}\\
& \leq  - \delta_0 \langle \hat \vep, \hat \vep \rangle_\ast + \frac{C \nu^4}{|\ln \nu|^4} \leq  - \delta_0 \langle \hat \vep, \hat \vep \rangle_\ast + C K \frac{e^{-4\sqrt{\gamma_1 \tau}}}{\tau^2},
\end{align*}
where we used the bootstrap  $|\nu(\tau)| \leq \bar K e^{-\sqrt{\gamma_1 \tau}}$ and $|\ln \nu(\tau)| \sim \sqrt{\tau}$. For a fixed $\tau_0 \gg 1$, consider $\tau_1 \in [\tau_0, 2 \tau_0]$ and  integrate in time the above inequality, we get 
\begin{align*}
\langle \hat \vep, \hat \vep \rangle_\ast &\leq e^{-\delta_0(\tau_1 - \tau_0)}  + CK e^{-\delta_0 \tau_1} \int_{\tau_0}^{\tau_1} e^{\delta_0 \tau} \frac{e^{-4 \sqrt{\gamma_1 \tau}}}{\tau^2} d\tau\\
&\leq CK \frac{e^{-4\sqrt{\gamma_1 \tau_1}}}{\tau_1^2} \leq \frac{K^2}{4}\frac{\nu(\tau_1)^4}{|\ln \nu(\tau_1)|^4},
\end{align*}
where we used integration by parts and $1/\tau_1 \sim 1/\tau \sim 1/\tau_0$ for $\tau, \tau_1 \in [\tau_0, 2\tau_0]$ and $\tau_0 \gg 1$. We recall from \eqref{bd:coercivity-control-adapted-norm-hatu-2} the equivalence $\hat \vep, \hat \vep \rangle_\ast \sim \|\vep \|^2_{\omega_\nu}$, hence, we get the improved bound 
$$\|\vep(\tau)\|_{\omega_\nu} \leq \frac{K}{2}\frac{\nu^2}{|\ln \nu|^2}, \quad \forall \tau \in [\tau_0, 2\tau_0]. $$
From Lemmas \ref{lemm:H2mid} and \eqref{lemm:Linfout}, we already obtain the improved bounds 
$$\|\vep(\tau)\|_{bd} \leq CK \frac{\nu^2}{|\ln \nu|^2} \leq \frac{K^2}{2}\frac{\nu^2}{|\ln \nu|^2}, \quad \|\vep(\tau)\|_{\out} \leq \frac{K^5}{2} \nu^2. $$
For the improvement of $H^1$ bound, we have from the energy estimate obtained in Lemma \ref{lemm:H1energy} and the bootstrap definition \ref{def:bootstrap}, 
\begin{align*}
\frac{d}{d\tau}\|q(\tau)\|_{H^1_\inn}^2 & \leq -\frac{1}{\zeta_*}\|q(\tau)\|_{H^1_\inn}^2 + CK^4 \frac{\nu^4}{|\ln \nu|^4},
\end{align*}
from which we integrate in time to get the estimate
$$\|q(\tau)\|_{H^1_\inn}^2 \leq CK^4 \frac{\nu^4}{|\ln \nu|^4} \leq \frac{K^6}{4}\frac{\nu^4}{|\ln \nu|^4}. $$
Similarly, we write from the energy estimate obtained in Lemma \ref{lemm:H2energy} and the bootstrap definition \ref{def:bootstrap}, 
\begin{align*}
\frac{d}{d\tau}\|\tilde q(\tau)\|_{H^2_\inn}^2 \leq -\delta_2\|\tilde q(\tau)\|_{H^2_\inn}^2 + CK^6 \frac{\nu^4}{|\ln \nu|^4}, 
\end{align*}
thus, the improved bound 
$$\|\tilde q(\tau)\|_{H^2_\inn}^2 \leq CK^6 \frac{\nu^4}{|\ln \nu|^2} \leq \frac{K^{8}}{4}\frac{\nu^4}{|\ln \nu|^4}. $$

\noindent We have shown that all the bootstrap estimates in Definition \ref{def:bootstrap} are improved, except for the growing modes $\tilde \alpha$ and $\tilde a$. In fact, at the exit time $\tau = \tau^*$, we compute from the two odes satisfied by $\tilde \alpha$ and $\tilde a$,
$$\frac{d}{d\tau} \big(|\tilde \alpha|^2\big) \Big\vert_{\tau = \tau^*} > 0, \quad \frac{d}{d\tau} \big(|\tilde a|^2\big) \Big\vert_{\tau = \tau^*} > 0.$$
By a Brouwer fixed-point argument,  there exists initial data $\tilde \alpha (\tau_0)$ and $\tilde a(\tau_0)$ such that $|\tilde \alpha(\tau)| \leq \frac{\bar K \nu^2}{|\ln \nu|}$ and  $|\tilde a(\tau)| \leq \bar K \nu^2$ for all $\tau \geq \tau_0$.  Hence, the solution $\vep(z,\tau)$ is trapped in the bootstrap regime \ref{def:bootstrap} for all $\tau \geq \tau_0$. \\

\noindent - \textit{Positivity and $16 \pi$ mass concentration at the origin of the constructed solution:} In terms of the self-similar variable \eqref{eq:wztau} and the decomposition \eqref{def:AppSol}, we have constructed initial data $w(\tau_0)$ for \eqref{eq:wztau}  of the form 
\begin{equation} \label{def:initialdatawtau0}
 w(z, \tau_0) = U_{\nu_0}(z - a_0) + U_{\nu_0}(z + a_0) + \tilde W[\nu_0, a_0, \alpha_0](z) + \vep_0(z),   
\end{equation}
where $(\vep_0, \nu_0, a_0, \alpha_0)$ satisfy estimates given in the bootstrap regime \ref{def:bootstrap}, from which the solution $w(z, \tau)$ to \eqref{eq:wztau} is decomposed as 
\begin{equation}\label{decomp:wfinal}
w(z, \tau) = U_{\nu}(z - a) + U_{\nu}(z + a) + \tilde W[\nu, a, \alpha](z) + \vep(z, \tau), \quad \forall \tau \geq \tau_0,
\end{equation}
where $(\vep(\tau), \nu(\tau), a(\tau), \alpha(\tau))$ is trapped in the bootstrap regime \ref{def:bootstrap} for all time $\tau \geq \tau_0$ with a specific choice of initial data $\tilde a(\tau_0)$ and $\tilde \alpha (\tau_0)$. We first claim that the initial data \eqref{def:initialdatawtau0} is nonnegative, so is our constructed solution $w(z,\tau)$. We fix $0< \eta_0 \ll 1$ and split the domain into $|z \pm a_0| \leq \eta_0$ and $|z \pm a_0| \leq \eta_0$. In the region $|z \pm a_0| \geq \eta_0$, by the construction of the approximate eigenfunctions (see Proposition  \ref{prop:eigen}) and the corrected steady state \eqref{def:ImprovedSS} (see Proposition \ref{prop:E0}) and $\alpha_0 \sim \nu_0^2$, we have 
$$|\tilde W[\nu_0, a_0, \alpha_0](z)| = \big|\alpha_0 \big(\phi_1[\nu_0, a_0](z) - \phi_0[\nu_0, a_0](z)\big) + \Xi[\nu_0, a_0](z)\big| \lesssim \frac{\nu_0^2}{\nu_0^2 + |z \pm a_0|^2} \lesssim \frac{\nu_0^2}{\eta_0^2},$$
and since $\vep_0$ satisfies the bootstrap estimates which implies that $|\vep_0(z)| \leq C(K) \nu_0^2$, and by the definition of $U_{\nu_0}(z \pm a_0)$, we have for $|z \pm a_0| \geq \eta_0$,
$$U_{\nu_0}(z \pm a_0) = \frac{8 \nu_0^2}{(\nu_0 + |z \pm a_0|)^4} \geq \frac{8\nu_0^2}{\eta_0^4} \gg \frac{\nu_0^2}{\eta_0^2} + C(K)\nu_0^2,$$
from which we get $w(z, \tau_0) > 0$ for $|z \pm a_0| \geq \eta_0$. In the region $|z \pm a| \leq \eta_0$, we compute from the expansion of the inner approximate eigenfunction (see Proposition \ref{prop:phi1inn}), we write in the variable $ y_\pm = \frac{z \mp a_0}{\nu_0}$,
$$U_{\nu_0}(z\pm a) + \tilde W(\nu_0, a_0, \alpha_0)(z) + \vep_0(z) =  \frac{1}{\nu_0^2}\Big[ U(y_\pm) + \Oc (\sum_{k = 2}^4 \nu_0^k \langle y_\pm \rangle^{k - 4}) + \|q_0 \mathbf{1}_{\{|y_\pm| \leq \frac{\eta_0}{\nu_0}\}}\|_{\infty}\Big],$$
where $q_0(y_\pm) = \nu_0^2 \vep_0(z)$ and we have from the bootstrap definition \ref{def:bootstrap} the bound $\|q_0 \mathbf{1}_{\{|y_\pm| \leq \frac{\eta_0}{\nu_0}\}}\|_{\infty} \lesssim \frac{\nu_0^2}{|\ln \nu_0|^2}$, which implies that $w(z, \tau_0) > 0$ in the region $\langle y_\pm \rangle \leq \frac{\eta_0}{\nu_0}$. This proves the positivity of the initial data $w(z, \tau_0)$, so is $w(z, \tau)$.

We also compute the total mass concentrating at the origin by estimating from the bootstrap definition \ref{def:bootstrap},
$$\|\vep(\tau)\|_{L^1(|z| \lesssim 1)} \lesssim \|\vep(\tau)\|_{L^2_{\omega_\nu}} \lesssim \frac{\nu^2}{|\ln \nu|^2}, \quad \|\tilde W \|_{L^1(|z| \lesssim 1)} \lesssim \nu^2 |\ln \nu|,$$
hence, for any $A \gg 1$ fixed constant,
$$\int_{|x| \leq A\sqrt{T-t}} u(x,t) dx = \int_{|z| \leq A} w(z, \tau) dz = \int_{\Rb^2} U_{1+2, \nu}(z) dz + \Oc(\nu^2 |\ln \nu|) = 16\pi + \Oc(\nu^2 |\ln \nu|). $$
This fact shows that our constructed solution has a $16\pi$ mass concentrating at the origin. This concludes the proof of the existence part of Theorem \ref{thm:main}. \\

\appendix
\section{Inhomogenous linearized system around a stationary state} \label{sec:linear-system}

In this subsection we solve for $\ell\in \mathbb N$ the system
$$ 
\left\{\begin{array}{l l} \Ls_{0}( w,\phi_w) = f(r)\cos(\ell\theta),\\
-\Delta \phi_w = w,
\end{array}  \right.
$$
following to a larger extent \cite{Vsiam02,SSVnon13}.

\subsection{The radial case}

For $\ell = 0 $, we introduce the partial mass variable: 
$$m_w(r) = \int_0^r w(r')r'dr', \quad w = \frac{\partial_r m_w}{r}, \quad \partial_r \Phi_w = -\frac{m_w}{r},$$
where $m_w$ solves the equation 
\begin{equation}
\As_0 m_w = m_f \quad \textup{with} \quad \As_0  = \partial_r^2 - \frac{\partial_r}{r} + \frac{Q}{r} \partial_r + \frac{\partial_r Q}{r}, \quad Q = \frac{4 r^2}{1 + r^2}. 
\end{equation}
\begin{lemma}[Inversion of $\As_0$] \label{lemm:invA0} For any $g \in \Cc(\Rb^+, \Rb)$, a solution to $\As_0 u  =g$ is given by 
\begin{equation} \label{systeme-odes:formula-radial}
\As_0^{-1} g = \frac{1}{2}\psi_0\int_r^1 \frac{\zeta^4  + 4\zeta^2 \ln \zeta -1}{\zeta} g(\zeta) d\zeta +\frac{1}{2}\tilde \psi_0\int_0^r \zeta g(\zeta) d\zeta,
\end{equation}
where $\psi_0$ and $\tilde{\psi}_0$ are the two fundamental solutions to $\As_0 \phi = 0$ given by 
\begin{equation} \label{systeme-odes:fundamental-solution-radial}
\psi_0 = \frac{r^2}{(1  + r^2)^2}, \quad \tilde{\psi}_0 =  \frac{r^4 + 4 r^2 \ln r  -1}{(1 + r^2)^2}. 
\end{equation}
\end{lemma}

\subsection{The non radial case}

\noindent For $\ell \geq 1$, we look for the solution  of the form
$$w = \psi(r)\cos(\ell \theta), \quad \phi_w = \omega(r)\cos(\ell \theta),$$
where $(\psi, \omega)$ satisfies the system
\begin{equation}\label{sys:psiomega}
\arraycolsep=1.4pt\def\arraystretch{2.2}
\left\{\begin{array}{l}
\pr^2\psi + \dfrac{\pr \psi}{r} - \dfrac{\ell^2}{r^2}\psi + \dfrac{4r }{1 + r^2}\pr \psi + \dfrac{32r }{(1+r^2)^3}\pr \omega + \dfrac{16\psi}{(1+r^2)^2} = f(r),\\
\pr^2\omega + \dfrac{\pr \omega}{r} - \dfrac{\ell^2}{r^2}\omega + \psi = 0.
\end{array}\right.
\end{equation}
System \eqref{sys:psiomega} can be solved in the closed form thanks to the following result.

\begin{lemma}[Solution of the homogeneous system \eqref{sys:psiomega}] \label{lemm:SolHom} For $\ell  \geq 1, $ there exist four linearly independent solutions $(\psi, \omega)$ of the homogegenous system \eqref{sys:psiomega} with $f \equiv 0$:
\begin{align*}
(h_{\ell,1}, g_{\ell,1}) & =[(\ell - 1)r^2 + \ell + 1] \left(\frac{8r^\ell}{(r^2 + 1)^3}, \frac{r^\ell}{r^2 + 1}\right),\\
(h_{\ell,2}, g_{\ell,2}) &=\Big(r^4 + 4r^2\ln r - 1\Big)\left(\frac{8}{r(r^2 + 1)^3}, \frac{1}{r(r^2+1)} \right) \quad \textup{for $\ell = 1$},\\
(h_{\ell,2}, g_{\ell,2}) & = [(\ell + 1)r^2 + \ell - 1]\left(\frac{8}{r^\ell(r^2 + 1)^3}, \frac{1}{r^\ell(r^2 + 1)}\right) \quad \textup{for $\ell \geq 2$},\\
(h_{\ell,3}, g_{\ell,3}) &= (8r^\ell+O(r^{\ell+2}), - r^\ell+O(r^{\ell+2})) \qquad \textup{as} \;\; r \to 0,\\
(h_{\ell,3}, g_{\ell,3}) &= K_\ell r^{\sqrt{\ell^2 + 4}}\left(16r^{-2}+O(r^{-4}), - 4+O(r^{\ell-\sqrt{\ell^2+4}})\right) \quad \textup{as} \;\; r \to +\infty,\\
(h_{\ell,4}, g_{\ell,4}) &= (c_{1,4}r^{-1}, c_{1,4}' r^{-1})+O(r) \qquad \text{as} \;\; r \to 0, \quad (c_{1,4},c_{1,4}')\neq (0,0), \quad \textup{for $\ell = 1$},\\
(h_{\ell,4}, g_{\ell,4}) & = \left(r^{-\sqrt{5}-2}+O(r^{-\sqrt{5}-4}),\frac{-1}{4}r^{-\sqrt{5}}+O(r^{-\sqrt{5}-2})\right) \quad \textup{as} \;\; r \to +\infty \quad \textup{for $\ell = 1$},\\
(h_{\ell,4}, g_{\ell,4}) &=(c_{\ell,4}r^{-\ell}+O(r^{-\ell+2}), c_{\ell,4}' r^{-\ell}+O(r^{-\ell+2})) \qquad \text{as} \;\; r \to 0, \quad  (c_{\ell,4},c_{\ell,4}')\neq (0,0) \quad \textup{for $\ell \geq 2$},\\
(h_{\ell,4}, g_{\ell,4}) &= C_\ell \left(16 r^{-\sqrt{\ell^2 + 4} - 2}+O(r^{-\ell-4}), - 4r^{-{\sqrt{\ell^2 + 4}}}+o(r^{-\ell-2})\right) \quad \textup{as} \;\; r \to +\infty \quad \textup{for $\ell \geq 2$},
\end{align*}
for some constants $K_\ell ,C_\ell > 0$ and $ \kappa_\ell \geq 0$. These estimates propagate for higher order derivatives. Moreover, we have for $\ell \geq 2$,
\begin{equation}\label{eq:ClKl}
C_\ell K_\ell = \frac{\ell}{\sqrt{\ell^2 + 4}}.
\end{equation}
and we have for $\ell\geq 1$,
\begin{equation}\label{eq:detA}
W_\ell := \left|\begin{matrix} h_{\ell,1} & h_{\ell,2} & h_{\ell,3} & h_{\ell,4}\\
g_{\ell,1} & g_{\ell,2} & g_{\ell,3} & g_{\ell,4}\\
h_{\ell,1}' & h_{\ell,2}' & h_{\ell,3}' & h_{\ell,4}'\\
g_{\ell,1}' & g_{\ell,2}' & g_{\ell,3}' & g_{\ell,4}'
\end{matrix}\right| = \frac{E_\ell}{r^2(1 + r^2)^2}.
\end{equation}
with $E_1=-128\sqrt{5}K_1$ and $E_\ell = -2^{10}(\ell + 1)(\ell - 1)\ell^2$ for $\ell\geq 2$. 
\end{lemma}

\begin{proof}

\textbf{Step 1}. \emph{The solutions}. The functions $(h_{\ell,k},g_{\ell,k})$ are obtained for $\ell\geq 2$ and $k=1,2,3$ in Theorem 3.2 in \cite{SSVnon13} in Theorem 4.3 in \cite{Vsiam02}, and for $\ell=1$ and $k=1,2,3$ in Theorem 4.2 in \cite{Vsiam02}. The main term in the asymptotic behaviours as $r\to 0$ and $r\to \infty$ for $h_{\ell,3}$ and $g_{\ell,3}$ is obtained in these references. The remainders in the $O()'s$ for $h_{\ell,3}$ and $g_{\ell,3}$ are easily obtained from the proofs of Theorems 4.2 and 4.3 in \cite{Vsiam02}. 

The only difference between these references and the present lemma is the behaviour as $r\to \infty$ of $(h_{\ell,4}, g_{\ell,4})$. First, for $\ell\geq 2$, in \cite{SSVnon13} a solution $(\tilde h_{\ell,4}, \tilde g_{\ell,4})$ is obtained such that 
\begin{align*}
(\tilde h_{\ell,4}, \tilde g_{\ell,4}) &\sim (8r^{-\ell}, - r^{-\ell}) \qquad \text{as} \;\; r \to 0,  \quad \textup{for $\ell \geq 2$},\\
(\tilde h_{\ell,4}, \tilde g_{\ell,4}) &= C_\ell \left(16 r^{-\sqrt{\ell^2 + 4} - 2}+O(r^{-\ell-4}), \kappa_\ell r^{-\ell} - 4r^{-{\sqrt{\ell^2 + 4}}}+o(r^{-\ell-2})\right) \; \textup{as} \; r \to +\infty \;\; \textup{for $\ell \geq 2$}.
\end{align*}
where $\kappa_\ell\in \mathbb R$ and $C_\ell>0$. We define $( h_{\ell,4},g_{\ell,4})=(\tilde h_{\ell,4}, \tilde g_{\ell,4})-\frac{C_\ell \kappa_\ell }{\ell+1}(h_{\ell,2},g_{\ell,2})$ which is a solution that satisfies the desired behaviour of the Lemma.

For $\ell=1$ we actually need to compute $(g_{1,4},h_{1,4})$. We choose one that has an asymptotic behaviour that is different from an analogue solution obtained in Theorem 4.2 in \cite{Vsiam02} because there is an error there in the computation of a constant called $B_1$ on page 1598.

We recall from page 1596 in \cite{Vsiam02} first that for $\ell=1$ a function $(\psi,\omega)$ solves \eqref{sys:psiomega} if and only if it can be decomposed as 
$$
(\psi,\omega)=\left(\frac{8}{r(r^2+1)^3}(F+G),\frac{1}{r(r^2+1)}(F-G)\right)
$$
where $(F,G)$ solves
\begin{align}
\label{sys:ell=1:eq-F}&\partial_{r}^2 F-\left(\frac{1}{r}+\frac{4r}{r^2+1}\right)\partial_r F+\frac{8}{r^2+1}F=\frac{4}{r^2+1}(r\partial_r G-3G),\\
\label{sys:ell=1:eq-G}&\partial_r^2 G-\left(\frac{1}{r}+\frac{8r}{r^2+1}\right)\partial_r G +\left(\frac{20}{r^2+1}-\frac{16}{(r^2+1)^2}\right)G=0,
\end{align}
and second that solutions to \eqref{sys:ell=1:eq-F} with $G\equiv 0$ are given by
\begin{align}
\label{sys:ell=1:id-F1-F2}
F_1=2r^2 \qquad \mbox{and}\qquad F_2=r^4+4r^2\log r-1.
\end{align}

\noindent \underline{Construction of a particular solution to \eqref{sys:ell=1:eq-G}}.
We claim that there exists a solution to \eqref{sys:ell=1:eq-G} of the form
\begin{equation} \label{sys:ell=1:id-Gbeta}
G(r)=r^{5-\sqrt{5}}+cr^{3-\sqrt{5}}+\tilde G, \quad c\in \mathbb R, \quad |\nabla^k \tilde G(r)|\lesssim r^{1-\sqrt{5}-k}\quad \mbox{as } r\to \infty.
\end{equation}
Indeed, we decompose
\begin{align*}
& \partial_r^2 -\left(\frac{1}{r}+\frac{8r}{r^2+1}\right)\partial_r  +\left(\frac{20}{r^2+1}-\frac{16}{(r^2+1)^2}\right) = L+\tilde L,\\
& L=\partial_r^2-\frac 9 r \partial_r +\frac{20}{r^2},\\
&\tilde L=-\frac{8}{r(r^2+1)}\partial_r+\left(\frac{16}{(r^2+1)^2}+\frac{20}{r^2(r^2+1)}\right)= \frac{8}{r^3}\partial_r-\frac{36}{r^4}+O(r^{-5})\partial_r+O(r^{-6}).
\end{align*}
Hence for the choice $c=\frac{8\sqrt{5}-4}{6+4\sqrt{5}}$ we have
$$
H(r)=\left(\partial_r^2 -\left(\frac{1}{r}+\frac{8r}{r^2+1}\right)\partial_r  +\left(\frac{20}{r^2+1}-\frac{16}{(r^2+1)^2}\right)\right)(r^{5-\sqrt{5}}+cr^{3-\sqrt{5}})=O(r^{-1-\sqrt{5}}).
$$
as $r\to \infty$. The kernel of $L$ is spanned by $r^{5-\sqrt{5}}$ and $r^{5+\sqrt{5}}$ so that for Schwartz functions $f$ a particular solution $u=L^{-1}f$ to $Lu=f$ is given by
$$
L^{-1}f (r)= -\frac{r^{5+\sqrt{5}}}{2\sqrt{5}}\int_r^\infty \rho^{-4-\sqrt{5}}f(\rho)d\rho+\frac{r^{5+\sqrt{5}}}{2\sqrt{5}}\int_r^\infty \rho^{-4+\sqrt{5}}f(\rho)d\rho.
$$
We compute, by an integration by parts, that for Schwartz functions $f$,
\begin{align*}
L^{-1}\tilde L f & =\sum_\pm \mp \frac{r^{5\pm\sqrt{5}}}{2\sqrt{5}}\int_r^\infty \left[\rho^{-4\mp \sqrt{5}}\left(\frac{16}{(\rho^2+1)^2}+\frac{20}{\rho^2(\rho^2+1)}\right)+8\partial_\rho \left(\frac{\rho^{-5\mp\sqrt{5}}}{\rho^2+1}\right)\right]d\rho.
\end{align*}
A solution of \eqref{sys:ell=1:eq-G} is of the form \eqref{sys:ell=1:id-Gbeta} if and only if $\tilde G$ solves $(L+\tilde L)\tilde G=+H$. We obtain the solution $\tilde G$ of this latter equation by solving the fixed point equation $L\tilde G=-L^{-1}\tilde L \tilde F+L^{-1}H$, i.e.
\begin{align*}
\tilde G(r) & =\sum_\pm \mp \frac{r^{5\pm\sqrt{5}}}{2\sqrt{5}}\int_r^\infty \rho^{-4\mp \sqrt{5}} H (\rho)d\rho\\
&+\sum_\pm \pm \frac{r^{5\pm\sqrt{5}}}{2\sqrt{5}}\int_r^\infty \left[\rho^{-4\mp \sqrt{5}}\left(\frac{16}{(\rho^2+1)^2}+\frac{20}{\rho^2(\rho^2+1)}\right)+8\partial_\rho \left(\frac{\rho^{-5\mp\sqrt{5}}}{\rho^2+1}\right)\right]G(\rho)d\rho .
\end{align*}
The above defines a contraction in $L^\infty([R,\infty),r^{-1+\sqrt{5}}dr)$ for $R$ large enough, providing the existence of the desired function $\tilde G$ satisfying the inequality in \eqref{sys:ell=1:id-Gbeta} for $k=0$. The estimates of higher order derivatives $k\geq 1$ of $\tilde G$ are obtained by induction from the above formula.

\smallskip

\noindent \underline{Construction of a particular corresponding solution to \eqref{sys:ell=1:eq-F}}. We claim that for $G$ given by \eqref{sys:ell=1:id-Gbeta} there exists a solution to \eqref{sys:ell=1:eq-F} of the form
\begin{equation} \label{sys:ell=1:id-Fbeta}
F(r)=r^{5-\sqrt{5}}+c'r^{3-\sqrt{5}}+\tilde F, \quad c'\in \mathbb R, \quad |\nabla^k \tilde F(r)|\lesssim r^{1-\sqrt{5}-k}\quad \mbox{as } r\to \infty.
\end{equation}
Indeed, a direct computation shows
\begin{align*}
I=& -\left(\partial_{r}^2 -\left(\frac{1}{r}+\frac{4r}{r^2+1}\right)\partial_r +\frac{8}{r^2+1}\right)\left(r^{5-\sqrt{5}}+c'r^{3-\sqrt{5}}\right)+\frac{4}{r^2+1}(r\partial_r G-3G)   \\
& =(-12+4\sqrt{5}+6c'-8+4\sqrt{5}(1-c))r^{1-\sqrt{5}}+O(r^{-1-\sqrt{5}}).
\end{align*}
By the above identity and \eqref{sys:ell=1:id-Gbeta}, choosing $c'$ such that $-12+4\sqrt{5}+6c'-8+4\sqrt{5}(1-c)=0$ yields $|\nabla^k I(r)|=O(r^{-1-\sqrt{5}-k})$ for all $k\in \mathbb N$ as $r\to \infty$. A function of the form \eqref{sys:ell=1:id-Fbeta} then solves \eqref{sys:ell=1:eq-F} if and only if $\tilde F$ solves $\partial_{r}^2 \tilde F-\left(\frac{1}{r}+\frac{4r}{r^2+1}\right)\partial_r \tilde F+\frac{8}{r^2+1}\tilde F=I$. Given the two particular solutions \eqref{sys:ell=1:id-F1-F2}, a solution $\tilde F$ is provided by the formula
$$
\tilde F=\frac{r^2}{2}\int_r^\infty \frac{\rho^4+4\rho^2\log r-1}{\rho (\rho^2+1)^2}I(\rho)d\rho -\frac{r^4+4r^2\log r-1}{2}\int_r^\infty \frac{\rho}{(\rho^2+1)^2}I(\rho)d\rho.
$$
The above identity gives the desired function $\tilde F$ that satisfies the inequality in \eqref{sys:ell=1:id-Fbeta}.

\smallskip

\noindent \underline{Construction of $(h_{1,4},g_{1,4})$}. We let $h_{1,4}=\frac{1}{2r(r^2+1)^3}(F+G) $ and $g_{1,4}=\frac{1}{16 r(r^2+1)}(F-G)$. Then by \eqref{sys:ell=1:id-Gbeta} and \eqref{sys:ell=1:id-Fbeta} the function $h_{1,4}$ satisfies the desired asymptotics $h_{1,4}(r)=r^{-\sqrt{5}-2}+O(r^{-4-\sqrt{5}})$ as $\to \infty$. We also infer $g_{1,4}=c'' r^{-\sqrt{5}}+O(r^{-\sqrt{5}-2})$. Since $g_{1,4}$ solves $\pr^2g_{1,4} + \dfrac{\pr g_{1,4}}{r} - \dfrac{\ell^2}{r^2}g_{1,4} + h_{1,4}=0$ we get $c''=-\frac{1}{4}$ hence the desired result.

\medskip

\noindent \textbf{Step 2}. \emph{Remaining identities}. The identity \eqref{eq:ClKl} is obtained in Proposition 3.5 in \cite{SSVnon13}. This proposition also proves the identity \eqref{eq:detA} for $\ell\geq 2$, but with $(h_{\ell,4},g_{\ell,4})$ replaced by $(\tilde h_{\ell,4},\tilde g_{\ell,4})$ in the last row. Since $( h_{\ell,4},g_{\ell,4})=(\tilde h_{\ell,4}, \tilde g_{\ell,4})-\frac{C_\ell \kappa_\ell }{\ell+1}(h_{\ell,2},g_{\ell,2})$, these determinants with $( h_{\ell,4},g_{\ell,4})$ or  $(\tilde h_{\ell,4}, \tilde g_{\ell,4})$ in the last row are actually equal, which shows \eqref{eq:detA}.

There only remains to prove \eqref{eq:detA} for $\ell=1$. By a direct computation, $W_1$ solves $W_1'=-(1/r^2+4r/(1+r^2))W_1$ so that $W_1=E_1/(r^2(1 + r^2)^2)$ for some $E_1\in \mathbb R$. We compute as $r\to \infty$ using the asymptotics of $(h_{1,k},g_{1,k})$ for $k=1,2,3,4$ that $W_1(r)=-128\sqrt{5}K_1r^{-6}+o(r^{-8})$, and deduce $E_1=-128\sqrt{5}K_1$.
\end{proof}

\begin{lemma}[Variation of parameters formula for \eqref{sys:psiomega}] \label{lemm:sol_inhol} For $\ell\geq 1$ the solution to the inhomogeneous system \eqref{sys:psiomega} is given by 
\begin{equation} \label{id:formula-solution-system}
\Big(\psi(r), \omega(r)\Big) = \sum_{k=1}^4 \Big(h_{\ell,k}(r), g_{\ell, k}(r) \Big) (-1)^{k +1} \int \frac{f(s) W_{\ell, k} (s) }{W_\ell(s)} ds, 
\end{equation}
where 
\begin{equation}\label{def:Wlk}
W_{\ell,k}(r) = \left|\begin{array}{ccc}
h_{\ell,i}(r) & h_{\ell,j}(r) & h_{\ell,m}(r)\\
g_{\ell,i}(r) & g_{\ell,j}(r) & g_{\ell,m}(r)\\
g_{\ell,i}'(r)& g_{\ell,j}'(r)& g_{\ell,m}'(r)
\end{array}\right|, \quad i,j,m \in \{1,2,3,4\}\setminus \{k\}, \;\; i < j< m.
\end{equation}
We have the identity
\begin{equation} \label{id:W12-W1}
W_{1,2}=\frac{1}{8}W_1r^2.
\end{equation}
The function $W_{\ell, k}$ admits the asymptotic behaviors: 
\begin{equation}\label{eq:asyDlkr0-l=1}
\left(\begin{array}{c}
W_{1,1}(r)\\
W_{1,2}(r)\\
W_{1,3}(r)\\
W_{1,4}(r)
\end{array} \right) = \left(\begin{array}{l}
\frac{2(c_{1,4}-8c_{1,4}')}{r^2}+O(1)\\
-64c_{1,4}'+O(r^2)\\
\frac{4(c_{1,4}-8c_{1,4}')}{r^2}+O(1)\\
-64+O(r^2)
\end{array} \right), \quad \textup{as} \quad r \to 0,
\end{equation}
and for $\ell\geq 2$,
\begin{equation}\label{eq:asyDlkr0}
\left(\begin{array}{c}
W_{\ell,1}(r)\\
W_{\ell,2}(r)\\
W_{\ell,3}(r)\\
W_{\ell,4}(r)
\end{array} \right) =  -\frac{32\ell}{r}\left(\begin{array}{l}
(\ell - 1)r^{-\ell}+o(r^{-\ell})\\
(\ell + 1)r^{\ell}+o(r^{\ell})\\
(\ell^2 - 1)r^{-\ell}+o(r^{-\ell})\\
(\ell^2 - 1)r^{\ell}+o(r^\ell)
\end{array} \right), \quad \textup{as} \quad r \to 0,
\end{equation}
and as $r \to +\infty$,
\begin{equation}\label{eq:asyDlkrinf}
\left(\begin{array}{c}
W_{1,1}(r)\\
W_{1,2}(r)\\
W_{1,3}(r)\\
W_{1,4}(r)
\end{array} \right) =  \left(\begin{array}{c}
-16K_1\sqrt{5} r^{-2}\Big[ 1 + \Oc(r^{-(\sqrt{5}-1)})\Big] \\
-16K_1 \sqrt{5}r^{-4}\Big[ 1 + \Oc(r^{1-\sqrt{5}} )\Big]\\
4 r^{-\sqrt{5}-3}\Big[1 + \Oc(r^{-2}\ln r  )\Big] \\
64 K_1 r^{\sqrt{5}-3} \big[1 + \Oc(r^{-4}) \big]
\end{array} \right),
\end{equation}
and for $\ell\geq 2$,
\begin{equation}\label{eq:asyDlkrinf}
\left(\begin{array}{c}
W_{\ell,1}(r)\\
W_{\ell,2}(r)\\
W_{\ell,3}(r)\\
W_{\ell,4}(r)
\end{array} \right) =-\frac{32\ell}{r^3}  \left(\begin{array}{c}
4 (\ell+1)
r^{-\ell}\Big[ 1 + \Oc(r^{\ell-\sqrt{\ell^2 + 4} })\Big] \\
4 ( \ell -1) r^{\ell}\Big[ 1 + \Oc(r^{\ell - \sqrt{\ell^2 + 4}})\Big]\\
(\ell^2-1) C_\ell r^{-\sqrt{\ell^2 + 4}} \Big[1 + \Oc(r^{\ell - \sqrt{\ell^2 + 4}})\Big] \\
(\ell^2-1)K_\ell r^{\sqrt{\ell^2 + 4}} \Big[1 + \Oc(r^{\ell - \sqrt{\ell^2 + 4}})\Big]
\end{array} \right),
\end{equation}
where $C_\ell, K_\ell$ and $\kappa_\ell$ are the constants introduced in Lemma \ref{lemm:SolHom}. These estimates propagate for higher order derivatives.\\
\end{lemma}
\begin{proof} The formula \eqref{id:formula-solution-system} can be found using Cramer's formula to solve an ODE, what is done in the proof of Lemma 3.2. in \cite{SSVnon13}. The asymptotic behavior of $W_{\ell, k}$ is a direct computation using Lemma \ref{lemm:SolHom}, using that $\sqrt{\ell^2+4}-\ell<\ell+2-\sqrt{\ell^2+4}$ for $\ell\geq 2$ in order to estimate the $O()$'s.

There remains to prove the identity \eqref{id:W12-W1}. Let $f$ be a smooth compactly supported function of the form $f(x)=\mathsf f(r)\cos \theta$, such that $\int_0^\infty \frac{W_{1,2}(r)}{W_1(r)}\mathsf f(r)dr=0$. Let then $(u,v)(x)=(\mathsf u(r)\cos \theta,\mathsf v(r)\cos \theta)$ be given by $(\mathsf u(r),\mathsf v(r))= \sum_{k=1}^4 \gamma_{1,k}(r) (h_{1,k}(r),g_{1,k}(r))$ where
\begin{align*}
&\gamma_{1,1}(r)=-\int_r^\infty \frac{W_{1,1}(s)}{W_1(s)}\mathsf f(s) ds, \qquad \gamma_{1,2}(r)=-\int_0^r \frac{W_{1,2}(s)}{W_1(s)}\mathsf f(s) ds, \\
&\gamma_{1,3}(r)=-\int_r^\infty \frac{W_{1,3}(s)}{W_1(s)}\mathsf f(s) ds, \qquad \gamma_{1,k}(r)=-\int_0^r \frac{W_{1,4}(s)}{W_1(s)}\mathsf f(s) ds.
\end{align*}
Then by \eqref{eq:asyDlkrinf} and \eqref{eq:asyDlkr0-l=1} $(u,v)$ is smooth, and by  \eqref{id:formula-solution-system} it is a solution of $\Ls_0(u,v)=f$ and $-\Delta v=u$. Since $f$ is compactly supported and $\int_0^\infty \frac{W_{1,2}(r)}{W_1(r)}\mathsf f(r)dr=0$, we have for all $r$ large enough that $(\mathsf u(r),\mathsf v(r))=\left(-\int_0^\infty \frac{W_{1,4}(s)}{W_1(s)}\mathsf f(s)ds\right) (h_{1,4}(r),g_{1,4}(r))$. We deduce from Lemma \ref{lemm:SolHom} that for all $k\in \mathbb N$ there holds $|\nabla^k u(x)|\lesssim |x|^{-\sqrt{5}-2-k}$ and $|\nabla^k v(x)|\lesssim |x|^{-\sqrt{5}-k}$ for $|x|\geq 1$. By a uniqueness argument for the Laplace equation, this decay implies that $v=\Phi_u$. In particular,  this implies that $u$ solves $\Ls_0 u=f$. We then compute using polar coordinates that on the one hand
$$
\int_{\mathbb R^2} y_1 f dy = \pi \int_0^\infty r^2 \mathsf f(r)dr.
$$
On the other hand, by performing integration by parts, which is permitted by the decay of $u$ and $\Phi_u=v$ as $|x|\to \infty$, 
\begin{align*}
\int_{\mathbb R^2} y_1 f dy & = \int_{\mathbb R^2} y_1 \Ls_0 u dy  = \int_{\mathbb R^2} \Ls_0^*( y_1) u dy 
\end{align*}
where $ \Ls_0^* \varepsilon =\Ms_0 \nabla . (U\nabla \varepsilon)$. We recall that $ \Ls_0^*( y_1) =0$ from \cite{RSma14}, Lemma 2.4., so that 
$$
\int_0^\infty r^2 \mathsf f(r)dr=0.
$$
We have proved that any smooth compactly supported function $\mathsf f$ that satisfies $\int_0^\infty \frac{W_{1,2}(r)}{W_1(r)}\mathsf f(r)dr=0$ must satisfy $\int_0^\infty r^2 \mathsf f(r)dr=0$, which is only possible if $\frac{W_{1,2}}{W_1}=Cr^2$ for some $C\in \mathbb R$. The asymptotic behaviours \eqref{eq:detA} and \eqref{eq:asyDlkrinf} of $W_1$ and $W_{1,2}$ as $r\to \infty$ then imply $C=\frac{1}{8}$. The desired identity \eqref{id:W12-W1} follows.

\end{proof}

\section{Estimates related to the Poisson field} \label{ap:3}
We claim the following pointwise estimate of Poison field
\begin{lemma} \label{lemm:PointwisePoisonField} Let $\gamma \in (1,2)$ and $f \in L^\infty_\gamma(\Rb^2)$, where $$L^\infty_\gamma(\Rb^2) = \{f \in L^\infty(\Rb^2),  \| f\|_{L^\infty_\gamma(\Rb^2)} = \| \langle x \rangle^\gamma f(x)\|_{L^\infty(\Rb^2)} < +\infty\}.$$
Then, it holds 
$$\forall z \in \Rb^2, \quad  |\nabla \Phi_f(z)| \lesssim  \| f\|_{L^\infty_\gamma (\Rb^2)} \big(|z|^{-\sigma} + \langle z \rangle^{-\gamma}|z|\big),$$ 
for $\sigma \in (0, \gamma - 1)$.  
\end{lemma}
\begin{proof} The proof is straightforward from the convolution formula
\begin{align*}
|\nabla \Phi_f(z)| \lesssim \int_{\Rb^2} \frac{|f(x)|}{|x - z|} dx  &\lesssim  \int_{|x - z| \leq \frac{|z|}2}\frac{|f(x)|}{|x - z|} dx + \int_{|x - z| \geq \frac{|z|}2}\frac{|f(x)|}{|x - z|}dx.
\end{align*}
Since $|x- z| \leq \frac{|z|}{2}$ implies $\frac{|z|}{2} \leq |x| \leq \frac{3|z|}{2}$, we get
\begin{align*}
\int_{|x - z| \leq \frac{|z|}2}\frac{|f(x)|}{|x - z|} dx & \lesssim \big\|f(z) \big\|_{L_\gamma^\infty(\Rb^2)}\int_{|x - z| \leq \frac{|z|}{2}}  \frac{\langle x \rangle^{-\gamma}}{ |x-z|} dx \lesssim   \big\| f\big\|_{L_\gamma^\infty(\Rb^2)} \langle z \rangle^{-\gamma}|z|. 
\end{align*}
The second integral is estimated by Holder inequality, 
\begin{align*}
 \int_{|x - z| \geq \frac{|z|}2}\frac{|f(x)|}{|x - z|}dx & \lesssim  \big\| f \big\|_{L_\gamma^\infty(\Rb^2)}  \int_{|x - z| \geq \frac{|z|}{2}} \frac{ \langle x \rangle^{-\gamma}}{|x - z|} dx\\
 & \lesssim \big\| f\big\|_{L_\gamma^\infty(\Rb^2)} \Big(\int_{|x - z| \geq \frac{|z|}{2}} |x - z|^{-p} dx   \Big)^\frac{1}{p}\Big(\int_{|x - z| \geq \frac{|z|}{2}} \langle x \rangle^{-\gamma q} dx\Big)^\frac{1}{q}\\
 & \lesssim  \big\| f\big\|_{L_\gamma^\infty(\Rb^2)} |z|^{-\sigma}, 
\end{align*}
where 
$$ \frac{1}{p} + \frac{1}{q} = 1, \quad \gamma q - 2 > 0, \quad \sigma = 1 - \frac{2}{p} \in (0, \gamma -1).$$
This concludes the proof of Lemma \ref{lemm:PointwisePoisonField}. 
\end{proof}

We Taylor expanse the Poison field from the other bubble in order to refine the expansion of the inner approximate eigenfunction. 
\begin{proof}[Proof of Lemma \ref{lem:taylorpoissonfield}]
Recall that $\nabla \Phi_U(y)=-\frac{4y}{1+|y|^2}$ and let $b=\pm 2a$, we write
$$
\nabla \Phi_U(y\pm \frac{2a}{\nu})=-4\frac{y+\frac{b}{\nu}}{1+|y+\frac{b}{\nu}|^2}= \frac{-4\nu(b+\nu y)}{|b|^2+2\nu y.b+\nu^2 (1+|y|^2)}.
$$
Let $A=2\nu y.b+\nu^2 (1+|y|^2)$, we expand
$$
\frac{1}{|b|^2+A}=\frac{1}{|b|^2}\left(1-\frac{1}{|b|^2} A+\frac{1}{|b|^4}A^2-\frac{1}{|b|^6}A^3+O(\nu^4|y|^4)  \right),
$$
where
\begin{align*}
& A^2=4\nu^2 (y.b)^2+4\nu^3 (y.b)(1+|y|^2)+O(\nu^4|y|^4),\\
&A^3=8 \nu^3(y.b)^3+O(\nu^4|y|^4).
\end{align*}
Thus, 
\begin{align*}
& \frac{1}{|b|^2+2\nu y.b+\nu^2 (1+|y|^2)} \\
&=\frac{1}{|b|^2}\left[1-\nu \frac{2y.b}{|b|^2}+\nu^2\left(4\frac{(y.b)^2}{|b|^4} -\frac{1+|y|^2}{|b|^2}\right)+\nu^3 \left(\frac{4(y.b)(1+|y|^2)}{|b|^4}-\frac{8(y.b)^3}{|b|^6} \right)\right]\\
& \qquad +O(\nu^4|y|^4),
\end{align*}
and hence,
\begin{align*}
\nabla \Phi_U\big(y\pm \frac{2a}{\nu}\big) &=\frac{-4\nu b - 4 \nu^2 y}{|b|^2+2\nu y.b+\nu^2 (1+|y|^2)}  \\
& = -\frac{4\nu b}{|b|^2}+4\frac{\nu^2}{|b|^2} \left(\frac{2y.b}{|b|^2}b-y \right)+4\frac{\nu^3}{|b|^3}\left(\left( 1+|y|^2-4\frac{(y.b)^2}{|b|^2} \right)\frac{b}{|b|}+\frac{2y.b}{|b|}y \right)\\
& +4\frac{\nu^4}{|b|^4} \left( \left(\frac{8(y.b)^3}{|b|^3}-\frac{4(y.b)(1+|y|^2)}{|b|} \right)\frac{b}{|b|}+\left( 1+|y|^2 \right)y-4\frac{(y.b)^2}{|b|^2}y \right).
\end{align*}
Recalling that $b=\pm 2a=(\pm 2|a|,0)$ with $a^2=2/\beta$, we obtain
$$
-4\nu \frac{b}{|b|^2}=\mp \beta \nu a,
$$
and
\begin{align*}
\frac{2y.b}{|b|^2}b-y = 2(y. (\pm 1,0))(\pm 1,0)-y=(2y_1,0)-y=(y_1,-y_2),
\end{align*}
and
\begin{align*}
\left( 1+|y|^2-4\frac{(y.b)^2}{|b|^2} \right)\frac{b}{|b|}+\frac{2y.b}{|b|}y &= \left( 1+|y|^2-4y_1^2 \right)(\pm 1,0)+2(y.(\pm 1,0))y\\
&= \pm \left( (1-3y_1^2+y_2^2,0)+(2y_1^2,2y_1y_2)\right)\\
&=\pm  \left( (1-y_1^2+y_2^2,2y_1y_2)\right),
\end{align*}
and
\begin{align*}
& \left(\frac{8(y.b)^3}{|b|^3}-\frac{4(y.b)(1+|y|^2)}{|b|} \right)\frac{b}{|b|}+\left( 1+|y|^2 \right)y-4\frac{(y.b)^2}{|b|^2}y\\
&= \left(8y_1^3-4y_1(1+|y|^2) \right)(1,0)+\Big((1+|y|^2)y_1,(1+|y|^2)y_2 \Big)-4 \big(y_1^3,y_1^2y_2\big)\\
&= (-3y_1-3y_1y_2^2+y_1^3,y_2-3y_1^2y_2+y_2^3).
\end{align*}
We substitute $b=\pm 2a$ and arrive at
\begin{align*}
\nabla \Phi_U\big(y\pm \frac{2a}{\nu}\big) & = \mp \beta \nu a+\frac{\nu^2}{|a|^2}(y_1,-y_2)\pm \frac{\nu^3}{2|a|^3} \left( (1-y_1^2+y_2^2,2y_1y_2)\right)\\
& +\frac{\nu^4}{4|a|^4}(-3y_1-3y_1y_2^2+y_1^3,y_2-3y_1^2y_2+y_2^3)+O(\nu^5|y|^4),
\end{align*}
which concludes the proof of Lemma \ref{lem:taylorpoissonfield}. 
\end{proof}

\section{Functional analysis} \label{sec:functional-analysis}

\begin{lemma}[Hardy inequality in $H^1(\langle y \rangle^{4}dy)$]

There exists $C>0$ such that for any $u\in H^1(U^{-1}dy)$ one has
\begin{equation} \label{bd:hardy-L2U-1}
\int_{\mathbb R^2} u^2 \langle y \rangle^{2}dy \leq C\int_{\mathbb R^2} |\nabla u|^2 \langle y \rangle^{4}dy.
\end{equation}

\end{lemma}

\begin{proof}

Consider the vector field $\phi(y)=(|y|^3+1)\frac{y}{|y|}$. A direct computation shows that it satisfies
\begin{equation} \label{bd:hardy-L2U-1-inter1}
|\phi|\lesssim \langle y \rangle^3 \quad \mbox{and} \quad \nabla . \phi\leq c \langle y\rangle^2
\end{equation}
for some $c>0$. For $u$ a Schwartz function we have by integration by parts and \eqref{bd:hardy-L2U-1-inter1}:
$$
-2\int u\nabla u.\phi = \int u^2 \nabla . \phi \geq c \int \langle y \rangle^2 u^2 dy.
$$
On the other hand, by Cauchy-Schwarz and \eqref{bd:hardy-L2U-1-inter1}:
\begin{align*}
-2\int u\nabla u.\phi & \leq \left( \int |\nabla u|^2 |\phi|\langle y \rangle dy \right)^{\frac 12}\left( \int |u|^2 |\phi|\langle y \rangle^{-1} dy \right)^{\frac 12} \\
& \leq C \left( \int |\nabla u|^2 \langle y \rangle^4 dy \right)^{\frac 12}\left( \int |u|^2 \langle y \rangle^{2} dy \right)^{\frac 12} .
\end{align*}
Combining, we get $\int_{\mathbb R^2} u^2 \langle y \rangle^{2}dy \leq C^2c^{-2}\int_{\mathbb R^2} |\nabla u|^2 \langle y \rangle^{4}dy$ as desired.
\end{proof}

\begin{lemma}[Poincar\'e inequality in $H^1(\langle y\rangle^8 e^{-|y|^2/4}dy)$]

We have for all $u\in H^1(\langle y\rangle^8 e^{-|y|^2/4}dy)$:
\begin{equation} \label{bd:poincare-gaussian}
\int u^2 \langle y \rangle^{10} e^{-|y|^2/4}dy \lesssim\int (u^2+|\nabla u|^2) \langle y \rangle^8 e^{-|y|^2/4}dy.
\end{equation}

\end{lemma}

\begin{proof}

An integration by parts shows that
$$
\int u^2 |y|^{10} e^{-|y|^2/4}dy =20 \int u^2|y|^8 e^{-|y|^2/4}dy+4\int u \nabla u.y |y|^8 e^{-|y|^2/4}dy.
$$
By Cauchy' inequality we have $4u\nabla u.y|y|^8\leq u^2|y|^{10}/2+8|\nabla u|^2 |y|^8$ so we deduce from the above identity that
$$
\int u^2 |y|^{10} e^{-|y|^2/4}dy \leq  \int (40 u^2+16|\nabla u|^2)|y|^4 e^{-|y|^2/4}dy.
$$
This implies \eqref{bd:poincare-gaussian}.
\end{proof}

\begin{lemma}[Estimates for the Poisson field]

For any $\epsilon>0$, there exists $C_{\epsilon}>0$ such that for any Schwartz function $u$ we have
\begin{equation} \label{annexe:bd:poisson-field}
\int |\Phi_u|^2 \langle y\rangle^{-2-\epsilon} dy \leq C_{\epsilon}\int u^2 \langle y\rangle^{2+\epsilon} dy.
\end{equation}
There exists $C>0$ such that if $u$ and $v$ are Schwartz functions with $u$ taking values in $\mathbb R^2$,
For any $\lambda>0$,
\begin{equation} \label{annexe:bd:poisson-field-divergence-2}
\left| \int \nabla.u\Phi_v dy \right|\leq C\left( \int  |u|^2 dy \right)^{\frac 12} \left( \int (\lambda+| y|)^2 v^2 dy \right)^{\frac 12}+C  \int  \frac{|u|}{\lambda+|y|} dy \int v dy .
\end{equation}

\end{lemma}

\begin{remark}

Note that we need $\epsilon>0$ in both the left-hand side and right-hand side of the inequality \eqref{annexe:bd:poisson-field}. Indeed, since $|\Phi_u|$ can have logarithmic growth at infinity, the integral $\int |\Phi_u|^2 \langle y\rangle^{-2} dy$ may diverge. For similar reasons, the quantity $\int_{|y|<1} |\Phi_u|^2 dy$ cannot be bounded by $(\int u^2 \langle y\rangle^{2}dy)^{1/2} $.

The inequality \eqref{annexe:bd:poisson-field-divergence-2} somehow corresponds to the case $\epsilon=0$ in \eqref{annexe:bd:poisson-field}. This is permitted because of the divergence structure of the first term.

\end{remark}

\begin{proof}

\noindent \underline{Proof of \eqref{annexe:bd:poisson-field}}. We perform a dyadic partition $1=\sum_{k\geq 0} \chi_k(y)$ with $\chi_0$ being supported on $\{|y|<2\}$ and $\chi_k$ being supported on $\{2^{k}<|y|<2^{k+2}\}$ for $k\geq 1$, which are smooth functions such that $\nabla^j \chi_k\lesssim 2^{-jk} $. We decompose $f=\sum_k f_k$ where $f_k=\chi_k f$. We have $\Phi_u=\sum_{k,l} (\Phi_{u_k})_l$. For $y$ in the support of $(\Phi_{u_k})_l$ we have $2^{l}<|y|<2^{l+2}$ and we estimate by Cauchy-Schwarz
\begin{align*}
|(\Phi_{u_k})_l (y)|& \lesssim \int_{2^k<|y'|<2^{k+2}} |\ln |y-y'|| u_k(y') dy' \lesssim  \| \ln |y-y'|\|_{L^2(2^k<|y'|<2^{k+2})} \| u_k\|_{L^2} \\
& \lesssim 2^k \max(l,k) \| u_k\|_{L^2} \lesssim 2^{-k\epsilon} \max(l,k) \| u \langle y \rangle^{1+\epsilon}\|_{L^2} 
\end{align*}
Hence 
$$
\| (\Phi_{u_k})_l \langle y \rangle^{-1-\epsilon} \|_{L^2} \lesssim 2^{-l \epsilon} \| (\Phi_{u_k})_l\|_{L^\infty}\lesssim 2^{-k\epsilon-l\epsilon} \max(l,k) \| u \langle y \rangle^{1+\epsilon}\|_{L^2} .
$$
Therefore,
$$
\| \Phi_{u}\langle y \rangle^{-1-\epsilon} \|_{L^2} \lesssim \sum_{k,l} \| (\Phi_{u_k})_l \langle y \rangle^{-1-\epsilon} \|_{L^2} \lesssim \| u \langle y \rangle^{1+\epsilon}\|_{L^2}  \sum_{k,l}2^{-k\epsilon-l\epsilon} \max(l,k) \leq C_\epsilon\| u \langle y \rangle^{1+\epsilon}\|_{L^2}
$$
as desired.

\smallskip

\noindent \underline{Proof of \eqref{annexe:bd:poisson-field-divergence-2}}. We decompose:
$$
\int \nabla .u \Phi_v= \sum_{k<l-2} \int \nabla . u_k \Phi_{v_l}+\sum_{k>l+2} \int \nabla .u_k \Phi_{v_l} +\sum_{l-2\leq k \leq 2+2} \int \nabla .u_k \Phi_{v_l}.
$$
If $k<l-2$ we integrate by parts and bound by H\"older
\begin{equation} \label{formosa-100}
\left| \int \nabla .u_k \Phi_{v_l} \right|=\left| \int  u_k.\nabla \Phi_{v_l}  \right| \lesssim \| u_k\|_{L^1} \| \nabla \Phi_{v_l} \|_{L^\infty(2^k<|y|<2^{k+2})} \lesssim 2^{k-l}  \| \langle y \rangle^{-1}u_k\|_{L^1} \| v_l\|_{L^1}
\end{equation}
where we used that for $2^k<|y|<2^{k+2}$ there holds
$$
|\nabla \Phi_{v_l}(y)|\lesssim \int_{2^l<|y'|<2^{l+2}} |y-y'|^{-1}|v_l (y')|dy'\lesssim 2^{-l}\| v_l\|_{L^1} .
$$
If $k>l+2$ we integrate by parts and bound similarly
\begin{equation} \label{formosa-101}
\left| \int \nabla .u_k \Phi_{v_l}\right|=\left| \int u_k.\nabla \Phi_{v_l} \right|  \lesssim  \| u_k\|_{L^1} \| \nabla \Phi_{v_l} \|_{L^\infty(2^k<|y|<2^{k+2})} \lesssim  \| \langle y \rangle^{-1} u_k\|_{L^1} \|  v_l\|_{L^1}
\end{equation}
where we used that for $2^k<|y|<2^{k+2}$ there holds
$$
|\nabla \Phi_{v_l}(y)|\lesssim \int_{2^l<|y'|<2^{l+2}} |y-y'|^{-1}|v_l (y')|dy'\lesssim 2^{-k}\| v_l\|_{L^1}
$$
Therefore, after summation of \eqref{formosa-100} and \eqref{formosa-101} we get
\begin{equation} \label{formosa-2}
  \left|\sum_{k<l-2} \int \nabla . u_k \Phi_{v_l}\right|+\left|\sum_{k>l+2} \int \nabla . u_k \Phi_{v_l}\right|\lesssim  \| \langle y \rangle^{-1} u \|_{L^1}  \|  v \|_{L^1}.
\end{equation}
If $k=l+j$ for $-2\leq j \leq 2 $ we rescale variables defining $z=2^{-k}y$ and let $\tilde u_k(z)=u_k(y)$ and $\tilde v_l(z)=v_l(y)$. Then
\begin{align*}
\Phi_{\tilde v_l}(z)&=-\frac{1}{2\pi} \int \ln |z-z'|v_l(2^k z')dz'=-\frac{1}{2\pi} 2^{-2k}\int \ln |2^{-k}y-2^{-k}y'|v_l(y')dy' \\
& =-\frac{1}{2\pi} 2^{-2k}\int (-k\ln 2+\ln |y-y'|)v_l(y')dy' = \frac{k2^{-2k}}{2\pi}\int v_l +2^{-2k}\Phi_{v_l}(y).
\end{align*}
Thus, changing variables and using that $\tilde u_k$ has support in $\{|z|<4\}$,
\begin{align*}
\left|\int \nabla .u_k \Phi_{v_l} dy\right| &= 2^{3k}\left|\int \nabla .\tilde u_k(z) \left(\Phi_{\tilde v_l}(z)-\frac{k}{2\pi}\int v_l \right)dz\right| \ = 2^{3k}\left|\int \nabla .\tilde u_k(z) \Phi_{\tilde v_l}(z) \right|\\
&= 2^{3k}\left|\int \tilde u_k(z) .\nabla \Phi_{\tilde v_l}(z) \right| \ \lesssim 2^{3k}\left|(\int |\tilde u_k(z)|^2 dz\right)^{\frac 12} \left( \int_{|z|<10} |\nabla \Phi_{\tilde v} (z)|^2 dz \right)^{\frac 12}.
\end{align*}
By H\"older, $\int_{|z|<10} |\nabla \Phi_{\tilde v} (z)|^2 dz \lesssim \| \nabla \Phi_{\tilde v_l}\|_{L^6}^2$ and $\| \tilde v_l\|_{L^{3/2}}\lesssim \| \tilde v_l\|_{L^2}$,  where we used that $\tilde v_l$ has support inside $\{|z|<16\}$. We have $\| \nabla \Phi_{\tilde v}\|_{L^6}\lesssim\| \nabla^2 \Phi_{\tilde v_l}\|_{L^{3/2}} $ by Sobolev. Since $-\Delta \Phi_{\tilde v}=\tilde v$, by elliptic regularity estimates we have $\| \nabla^2 \Phi_{\tilde v_l}\|_{L^{3/2}}\lesssim \| \tilde v_l\|_{L^{3/2}}$. Combining, we get $\int_{|z|<10} |\nabla \Phi_{\tilde v} (z)|^2 dz\lesssim \| \tilde v_l\|_{L^2}^2$. Using these inequality and changing back variables $z\mapsto y$ we get
$$
\left|\int \nabla .u_k \Phi_{v_l} dy\right|\lesssim 2^{3k}\| \tilde u_k\|_{L^2}\| \tilde v_l\|_{L^2}=2^{k}\| u_k\|_{L^2}\| v_l\|_{L^2}\lesssim \| u_k\|_{L^2}\| \langle y \rangle v_l\|_{L^2}.
$$
Therefore, by Cauchy-Schwarz,
\begin{equation} \label{formosa}
\left| \sum_{l-2\leq k \leq 2+2} \int \nabla .u_k \Phi_{v_l}\right| \lesssim \|  u \|_{L^2}\| \langle y\rangle v\|_{L^2}.
\end{equation}
Combining \eqref{formosa-2} and \eqref{formosa} shows \eqref{annexe:bd:poisson-field-divergence-2} with parameter $\lambda=1$. Let now $\lambda>0$. We change variables $z=\lambda^{-1} y$ and define $\tilde u(z)=u(y)$ and $\tilde v(z)=v(y)$. Then by \eqref{annexe:bd:poisson-field-divergence-2} with parameter $\lambda=1$,
\begin{align*}
\left| \int \nabla. u\Phi_v dy \right|&=\lambda^{3}\left|\int \nabla.\tilde u(z) \Phi_{\tilde v}(z)dz \right|\\
&\leq \lambda^{3} C\left( \int  |\tilde u|^2 dz \right)^{\frac 12} \left( \int (1+| z|)^2 \tilde v^2 dz \right)^{\frac 12}+C\lambda^{3}  \int  \frac{|\tilde u|}{1+|z|} dz \int |\tilde v|  dz .\\
&= C\left( \int |\tilde u|^2 dy \right)^{\frac 12} \left( \int (\lambda+| y|)^2 \tilde v^2 dy \right)^{\frac 12}+C  \int  \frac{|\tilde u|}{\lambda+|y|} dy \int |\tilde v|  dy .
\end{align*}
This is the desired inequality \eqref{annexe:bd:poisson-field-divergence-2}.
\end{proof}

\bibliographystyle{alpha}

\newcommand{\etalchar}[1]{$^{#1}$}
\def\cprime{$'$}

\end{document}